\documentclass[a4paper,twoside]{report}

\usepackage[cyr]{aeguill}
\usepackage[english]{babel}
\usepackage[utf8]{inputenc}
\usepackage[T1]{fontenc}

\usepackage[a4paper,top=3cm,bottom=2cm,left=3cm,right=3cm,marginparwidth=1.75cm]{geometry}

\usepackage{amsmath}
\usepackage{graphicx}
\usepackage[colorinlistoftodos]{todonotes}
\usepackage[colorlinks=true, allcolors=blue]{hyperref}
\usepackage{amssymb}
\usepackage{amsthm}
\usepackage{bbold}
\usepackage[all]{xy}
\usepackage{dirtytalk}
\usepackage[titletoc,title]{appendix}
\usepackage{upgreek}
\usepackage{staves}
\usepackage{ulem}
\usepackage{stmaryrd}
\usepackage{soul}
\usepackage{centernot}

\newtheorem{theorem}{Theorem}[section]
\newtheorem{corollary}[theorem]{Corollary}
\newtheorem{lemma}[theorem]{Lemma}
\newtheorem{proposition}[theorem]{Proposition}
\theoremstyle{definition}

\newtheorem{definition}[theorem]{Definition}
\newtheorem{example}[theorem]{Example}

\newtheorem{remark}[theorem]{Remark}
\newtheorem{notations}[theorem]{Notations}
\newtheorem{question}[theorem]{Question}

\newcommand\scalemath[2]{\scalebox{#1}{\mbox{\ensuremath{\displaystyle #2}}}}

\makeatletter
\newcommand{\oset}[2]{{\mathpalette\o@set{{#1}{#2}}}}
\newcommand{\o@set}[2]{\o@@set{#1}#2}
\newcommand{\o@@set}[3]{%
  \vbox{\offinterlineskip
    \ialign{\hfil##\hfil\cr
      $\m@th\o@set@demote{#1}#2$\cr
      \noalign{\vskip0.2pt}
      $\m@th#1#3$\cr
    }%
  }%
}
\newcommand{\o@set@demote}[1]{%
  \ifx#1\displaystyle\scriptstyle\else
  \ifx#1\textstyle\scriptstyle\else
  \scriptscriptstyle\fi\fi
}
\makeatother

\newcommand{\owave}[1]{\oset{\rotatebox{180}{\uwave{\hspace{\widthof{\ensuremath{#1}}}}}}{#1}}

\newcommand{\emeta}{\ensuremath{\raisebox{-1.1ex}{\includegraphics[height=2.1ex]{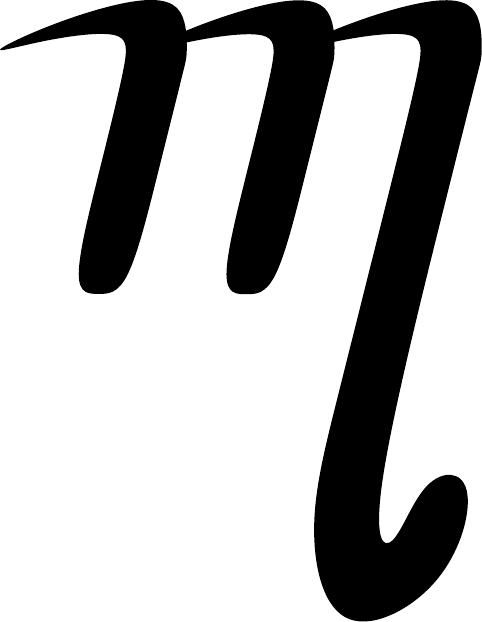}}}} 

\title{\textbf{NONLINEARITY, FRACTALS, FOURIER DECAY} \newline \hfill \break {Harmonic analysis of equilibrium
states for hyperbolic dynamical systems} }
\author{Gaétan Leclerc}

\date{2021-2024}

\begin{document}

\maketitle

\cleardoublepage
\begin{abstract}
This PhD lies at the intersection between fractal geometry and hyperbolic dynamics. Being given a (low dimensional) hyperbolic dynamical system $f$ in some euclidean space, let us consider a \emph{fractal} compact invariant subset $K$, and an invariant probability measure $\mu$ supported on $K$ with good statistical properties, such as the measure of maximal entropy. The question is the following: does the Fourier transform of $\mu$ exhibit power decay ? \\

More generally, being given a smooth phase $\psi:K \rightarrow \mathbb{R}$ and a test function $\chi : 
 K \rightarrow \mathbb{R}$, can we prove a broader statement of decay results for oscillatory integrals, of the form $$ \exists C,\rho >0, \ \forall \xi \in \mathbb{R}^*, \ \left| \int_K e^{i \xi \psi(x)} \chi(x) d\mu(x) \right| \leq C |\xi|^{-\rho} \ ? $$

Our main goal is to give evidence, for several families of hyperbolic dynamical systems, that \emph{nonlinearity} of the dynamics is enough to prove such decay results. These statements will be obtained using a powerful tool from the field of additive combinatorics: the \emph{sum-product phenomenon}.
\end{abstract}

\cleardoublepage

\section*{\hfil Acknowledgement \hfil}

My first thoughts go to my PhD advisor, Frédéric Naud, for his insightful comments, suggestions, and for his assistance at every stage of this PhD. Thank you very much for the numerous helpful discussions, for the support, and for guiding me into the realm of hyperbolic dynamics and additive combinatorics—two domains that I knew nothing about prior to this PhD. \\

I would also like to thank the Doctoral School, \emph{École doctorale de Sciences Mathématiques de Paris Centre}, and \emph{Sorbonne Université} for the very stimulating work environment at the \emph{Institut de Mathématiques de Jussieu - Paris Rive Gauche}, and more specifically in the \emph{Analyse Complexe et Géometrie} research team. Thanks are due to all the involved researchers, and also to the administrative staff, cleaning staff, and security staff. Moreover, I would like to thank the \emph{École Normale Supérieure de Rennes} for funding this PhD. \\

During these first years in the world of mathematical research, I had the pleasure of meeting and chatting with many inspiring people. I would like to kindly thank all the researchers who took an interest in my work and encouraged me to continue with my research. Special thanks go to Thibault Lefeuvre and Nguyen-Viet Dang for accepting to be on my thesis monitoring committee and for friendly chats during these three years; to Sébastien Gouëzel and Tuomas Sahlsten for very stimulating discussions, encouragement throughout the PhD, and for kindly accepting to be examiners (and write a report) for my PhD defense; and to Gabriel Rivière, Romain Dujardin, Yves Coudène, Thibault Lefeuvre, and Masato Tsujii for accepting to be part of the jury committee (sorry, Masato, that the calendar didn’t allow for a reasonable meeting time from Japan!). My warm thanks go to Colin Guillarmou, Frédéric Faure, Mihajlo Cekic, Simon Baker, Osama Khalil, Felix Lequen, Joackim Bernier, Christophe Dupont, Julia Slipantschuk, Mark Pollicott, Semyon Dyatlov, Nalini Anantharaman, Bruno Santiago, Martin Leguil, Sylvain Crovisier, Federico Rodriguez Hertz, William O'Regan, Bram Petri, Andrés Sambarino, Yann Chaubet, Jialun Li, and many more, for giving me the opportunity to discuss my work and/or for some very stimulating exchanges. Many thanks to Tristan Humbert for very friendly discussions around a cup of coffee, and to Jiasheng Lin and Antoine Toussaint for sharing some memorable time with me in our shared office. \\

In my personal relationship with mathematics, my past professors have played a great role, and I would like to take the time to thank some of them profusely. First of all, I would like to express my most sincere gratitude to all the teachers at the \emph{Université de Rennes 1} and the \emph{École Normale Supérieure de Rennes}. I would particularly like to mention Rozenn Texier-Picard, Jean-Christophe Breton, and Christophe Dupont for some very kind discussions that went beyond the sole focus of mathematics. My next thoughts go to my math teachers from when I was a student in \emph{CPGE} in Tours, at the \emph{Lycée Descartes}. I’d like to thank Michel Boutemy and Marie Hézard, who was probably the best math teacher I ever had. Her rigorous yet intuitive and pedagogical mathematical style was a wonderful introduction to undergraduate-level mathematics and inspired me greatly. Finally, I’d like to thank my math teachers from high school and before.  \\

On a more personal level, my friends and family were of central importance to my well-being during these past 10 years. Without friends, no mathematics are possible. To Larouci, Gaelle, and Dylan, thanks for the many unforgettable late video-game nights. To Paul, thanks for the strong friendship and for everything you taught me about algebra and mathematics in general. I owe a lot to you in this regard. To Naël and Alex, thanks for letting me sleep many times at your house in Strasbourg, for the friendly discussions, and the good times. To Léa, thank you for the very strong memories from the time I lived with you in Rennes. I wish you all the best. To my friends from Rennes—Magali, Thomas, Antoine, Arnaud, Florian, Panda, Anne-Lise, Maude, and everyone from those years—thank you for the good times. To all my friends from Tours—Héloïse, Victor, Bastien, Clémence, Julie, Gautier, Paulin, Remy, Zoé, Marie, Simon, Aurélie, Corentin, Lucas, Octave, Émile, Soukaina, and everyone from that period—thank you for the good memories. To my childhood friends Ella and Guillaume, thank you for everything ! \\
\newpage 
Finally, and most importantly, I would like to thank my family for always encouraging me to pursue my passions and for giving me the stability necessary to complete these demanding studies. Thank you, Mom, Dad, Emmanuelle, and Franckie. Thanks to my grandparents, my aunts and uncles, and my beloved cousins Antoine, Élie, Lucas, Léopold, Nicolas, and Julien. Thank you, Louise, for being an incredible sister. Thank you, Samuel Vivien, and your family for the time I spent with you all during the summer. Thank you, Gaelle, Camarade, René-Louis, and Sensei Jildaz for the good times this summer and for teaching me the game of Go in front of the sea. Thanks to William, Christophe, Isabelle, Marie-France, Gérard, and Danielle for the time we shared together since I met Marie.  \\

Marie André, you are by far the most wonderful person I have ever met. I love you. Thank you so much for being a part of my life. This thesis is dedicated to you.

\cleardoublepage
\tableofcontents

\cleardoublepage
\chapter{Introduction}

\section{Fractal geometry, hyperbolic dynamics, and harmonic analysis}

\subsection{Hausdorff dimension and Fourier dimension}

We begin this thesis by recalling some important notions in the study of fractal sets, namely, Hausdorff measures, Hausdorff dimension and their link with Fourier analysis. On this subject, the book of Mattila \cite{Ma15}, the book of Kahane and Salem \cite{KS64}, and the appendix in the book of Zinsmeister \cite{Zi96} are good references.

\begin{definition}[Hausdorff measure]

Let $M$ be a manifold (equipped with some distance). Let $\alpha \in \mathbb{R}_+^*$. The Hausdorff measure of dimension $\alpha$ is defined by:
$$ H_\alpha(E) := \lim_{\varepsilon \rightarrow 0} \ \inf_{(U_i)_{i\in I} \in \mathcal{U}^\varepsilon_E } \sum_{i \in I} (\text{diam} \ U_i) ^\alpha \quad \in [0, \infty] ,$$
where $E \subset M$ is a borel set, and $\mathcal{U}^\varepsilon_E := \left\{ (U_i)_{i \in I} \ | \ \bigcup_{i \in I} U_i \supset E, \  \text{diam}(U_i) \leq \varepsilon \right\} $ is the set of all $\varepsilon$-coverings of $E$. 
\end{definition}

\begin{remark}
The measure $H_\alpha$ is invariant under isometries.
Since, in $\mathbb{R}^d$, $0 < H_d([0,1]^d) < \infty$, we see that $H_d$ is a multiple of the $d$-dimensional lebesgue measure in this context.
\end{remark}

\begin{definition}[Hausdorff dimension]
Let $X \subset M$ be a compact set. Its Hausdorff dimension, denoted $\dim_H(X) \in [0, \dim(M)]$, is defined as:
$$ \dim_H(X) := \sup \{\alpha \in \mathbb{R}^*_+ \ | \ H_\alpha(X)=\infty \} = \inf \{\alpha \in \mathbb{R}^*_+ \ | \ H_\alpha(X)=0 \}  $$
\end{definition}

\begin{remark}
The Hausdorff dimension is invariant under diffeomorphisms. Any countable set has Hausdorff dimension zero. A set with nonempty interior in $M$ has Hausdorff dimension $\text{dim}(M)$. A $k$-dimensional submanifold of $M$ has Hausdorff dimension $k$.
\end{remark}

This notion of fractal dimension appears a lot in dynamics, but as is, isn't very easy to compute on concrete examples. To help, we introduce Frostman's lemma, which yields an alternative definition for the Hausdorff dimension.

\begin{theorem}[Frostman's lemma]
Let $X \subset \mathbb{R}^d$ be a compact set. Denote by $\mathcal{P}(X)$ the set of all borel probability measures supported on $X$. Then:
$$ \dim_H(X) = \sup \Big{\{} \alpha \in [0,d] \Big{|} \ \exists \mu \in \mathcal{P}(X), \ \exists C>0, \forall (x,r) \in X \times \mathbb{R}_+^*, \ \mu(B(x,r)) \leq C r^\alpha \Big{\}} .$$
\end{theorem}

\begin{remark}
This result gives us an interesting tool to estimate/compute the Hausdorff dimension of explicit fractal sets. For example, with the help of Frostman's lemma, one can prove that the \emph{triadic Cantor set}, defined as
$$ C := \Big{\{} \sum_{n=1}^\infty \frac{2 \varepsilon_n}{3^n} \ , \ (\varepsilon_n) \in \{0,1\}^{\mathbb{N}^*} \Big{\}} \subset \mathbb{R},$$
has Hausdorff dimension $\ln(2)/\ln(3) \in (0,1)$.
\end{remark}

We then introduce the notion of energy integral. This fundamental object will make a link between the study of fractals and harmonic analysis.

\begin{definition}
Let $X \subset \mathbb{R}^d$ a compact set, $\beta \in (0,d)$, and $\mu \in \mathcal{P}(X)$. We define the $\beta$-energy integral of $\mu$ as:
$$ I_\beta(\mu) := \iint_{X^2} \frac{d \mu(x) d \mu(y)}{|x-y|^\beta} .$$
Notice that, if $\mu$ admits a regularity bound of the form
$$ \exists C, \forall (x,r) \in X \times \mathbb{R}^+_*, \  \mu(B(x,r)) \leq C r^\alpha ,$$
where $\alpha > \beta$, then $I_\beta(\mu) < \infty$. 
\end{definition}

\begin{theorem}[Frostman's lemma, revisited]
Let $X \subset \mathbb{R}^d$ be a compact set. Then:
$$ \dim_H(X) = \sup \Big{\{} \alpha \in [0,d] \ \Big{|} \ \exists \mu \in \mathcal{P}(X), \ I_\alpha(\mu) < \infty \Big\} .$$
\end{theorem}

\begin{remark}
To make a link with Fourier analysis, we make the following remark. Notice that $$ I_\alpha(\mu) = \langle \mu , k_\alpha * \mu \rangle, $$
where $k_\alpha(t) := 1/|t|^\alpha$ is a Riesz potential. It is well known, when $\alpha \in (0,d)$, that its Fourier transform is $ \widehat{k_\alpha}(\xi) = c_{\alpha,d} |\xi|^{\alpha-d} $. It then follows from Parseval's identity that
$$ I_\alpha(\mu) = c_{\alpha,d} \int_{\mathbb{R}^d} |\widehat{\mu}(\xi)|^2 |\xi|^{\alpha-d} d\xi,$$
where $\widehat{\mu} : \mathbb{R}^d \rightarrow \mathbb{C}$, defined as $$ \widehat{\mu}(\xi) := \int_{\mathbb{R}^d} e^{-2 i \pi x \cdot \xi} d\mu(x), $$
is the \textbf{Fourier transform} of the probability measure $\mu \in \mathcal{P}(X)$. Frostman's formula can then be rewritten as:
$$  \dim_H(X) = \sup \Big{\{} \alpha \in [0,d] \Big{|} \ \exists \mu \in \mathcal{P}(X), \int_{\mathbb{R}^d} |\widehat{\mu}(\xi)|^2 |\xi|^{\alpha-d} d\xi < \infty \Big\} .$$
Notice that this condition means that $\widehat{\mu}(\xi)$ decays at least like $|\xi|^{-\alpha/2}$ \say{on average}. This suggest the following question: can we get rid of this \say{on average}, and ask for a pointwise decay estimate instead ? We formalise this question by introducing the Fourier dimension.
\end{remark}

\begin{definition}[Fourier dimension]
Let $X \subset \mathbb{R}^d$ be a compact set. We define the Fourier dimension of $X$ by the formula:
$$ \dim_F(X) = \sup \Big{\{} \alpha \in [0,d] \Big{|} \ \exists \mu \in \mathcal{P}(X), \exists C>0, \forall \xi \in \mathbb{R}^d \setminus \{0\}, \ |\widehat{\mu}(\xi)| \leq C |\xi|^{-\alpha/2} \Big{\}} .$$
\end{definition}

\begin{remark}
Frostman's lemma ensures the following inequality, for $X \subset \mathbb{R}^d$:
$$ 0 \leq \dim_F(X) \leq \dim_H(X) \leq d .$$
On some very particular examples, we are able to compute the Fourier dimension: for example, if $X$ is countable, then $\dim_F(X) = \dim_H(X) = 0$, and if $X$ has nonempty interior in $\mathbb{R}^d$, $\dim_F(X) = \dim_H(X) = d$.
\end{remark}

\begin{remark}
Notice the following technicality. When $X \subset \mathbb{R}^d$ has Hausdorff dimension strictly less than $d$, then the Fourier transform of any probability measure supported on $X$ can decay at most like $|\xi|^{-\delta/2}$, because of the energy integral argument, which relies on the Fourier transform of the Riesz potential $1/|t|^\delta$. This argument breaks when $\delta=d$, as the function $1/|\xi|^d$ is not a tempered distribution anymore, and so its Fourier transform is not well defined. And indeed, if one chooses $X \subset \mathbb{R}^d$ to be an open set, then any smooth and compactly supported measure in $X$ will enjoy Fourier decay faster than any polynomial.
\end{remark}

We will now compute the Fourier dimension on a fundamental example: the triadic Cantor set.

\begin{lemma}
The triadic Cantor set $C$ has Fourier dimension zero (even though its Hausdorff dimension is positive !).
\end{lemma}
\newpage
\begin{proof}
We proceed by contradiction. Assume that there exists a probability measure $\mu \in \mathcal{P}(C)$ such that $\widehat{\mu}(\xi) \underset{\xi \rightarrow \infty}{\longrightarrow} 0$. (Such a measure doesn't have any atoms \cite{Ly20}). The main idea is to use the fact that $C$ is (up to a zero-measure set) an invariant compact set for some expanding dynamical system, namely, the map $T(x) := 3x \text{ mod } 1$. One can then consider the sequence of measures $(\nu_k)_{k \in \mathbb{N}} \in \mathcal{P}(C)^\mathbb{N}$ defined as $\nu_k := (T_*)^k \mu$. Since $C$ is compact, we can extract a converging subsequence and denote its limit $\nu \in \mathcal{P}(C)$. To find a contradiction, let us compute the Fourier coefficients of $\nu$. First of all, notice that for all $n \in \mathbb{Z}$, the $n$-th Fourier coefficient of $T_* \mu$ can be expressed as
$$ c_n(T_* \mu) = \int_{C} e^{2 i \pi n x} d(T_*\mu)(x) = \int_{C} e^{2 i \pi 3 n x} d\mu(x) = c_{3n}(\mu) .$$
It follows that $c_n(\nu_k) = c_{3^k n}(\mu)$, and hence, since $c_n(\mu) \underset{n \rightarrow 0}{\longrightarrow} 0$ by hypothesis: $ c_n(\nu) = \delta_{0,n} $.
We recognise the Fourier coefficients of the lebesgue measure $\lambda$ on the circle. This implies that $\nu = \lambda$, which is impossible since $\lambda$ isn't supported on the triadic cantor set $C$.
\end{proof}

What happened ? Even though the triadic Cantor set $C$ has positive Hausdorff dimension $\dim_H(C) = \ln(2)/\ln(3) \in (0,1)$, its Fourier dimension vanishes. Thoses responsible for this in the previous argument are:
\begin{itemize}
\item First of all, $C \mod 1$ isn't the full circle
\item Second, $C$ is invariant by a \emph{linear} dynamical system (and its coefficient is an integer).
\end{itemize}

This example is the prototype of what we may expect to happen regarding the possible values of the Fourier dimension. For fractals with empty interior that are invariant by some hyperbolic dynamical system, linearity (and, in this case, some algebraic conditions on the coefficients) seems to sometimes be an obstruction for positivity of the Fourier dimension. The goal of this thesis is to go in the other direction: we seek for deterministic examples where nonlinearity of the underlying dynamical system yields positivity of the Fourier dimension. Overall, one can state our main questions in the following form.

\begin{question}
Let $X \subset \mathbb{R}^d$. When do we have $\dim_H(X) = \dim_F(X)$ ? Is this a generic situation? More generally, when do we have $\dim_F(X) >0 $ ? If $X$ is invariant under some \emph{nonlinear} hyperbolic dynamical system, do we have $\dim_F(X)>0$ ?
\end{question}

In fact, we will be only interested in the last question. Many hyperbolic dynamical systems, with their invariant fractal subsets, are interesting to consider. Most of them are defined naturally on manifolds: for example, the geodesic flow $\phi$ on a convex-cocompact hyperbolic manifold $M$ leaves invariant its non-wandering set $NW(\phi)$, which lives naturally inside the unit tangent bundle $T^1 M$. Can we make sense of the Fourier dimension in this context ? \\

Let us suggest some new definitions, to try to make sense of the Fourier dimension on abstract manifolds. Subsections 1.1.2 to 1.1.5 are adapted from the appendix of the author preprint \cite{Le23b}. Let us mention that another direction for generalization would be to stay in a euclidean setting, but try to interpolate between the Fourier and the Hausdorff dimension. This is done in a work of Fraser under the name of \say{Fourier dimension spectrum} \cite{Fr22}.

\subsection{The upper and lower Fourier dimension of a measure}

Let us define the Fourier dimension of a probability measure $\mu \in \mathcal{P}(E)$, supported on some compact set $E \subset \mathbb{R}^d$, as:
$$ \dim_F(\mu) := \sup \{\alpha \geq 0 \ | \ \exists C \geq 1, \forall \xi \in \mathbb{R}^d \setminus \{0\}, \ |\widehat{\mu}(\xi)| \leq C |\xi|^{-\alpha/2} \}, $$
where the Fourier transform of $\mu$ is defined by $$ \widehat{\mu}(\xi) := \int_{E} e^{-2 i \pi \xi \cdot x} d\mu(x) .$$
The Fourier dimension of a compact set $E \subset \mathbb{R}^d$ can then be written as $$ \dim_F(E) := \sup \{ \min(d,\dim_F \mu) , \ \mu \in \mathcal{P}(E) \} \leq \dim_H E. $$
To define the Fourier dimension of a measure lying in an abstract manifold, a natural idea is to look at our measure into local charts. But this supposes that we have a meaningful way to \say{localize} the usual definition of the Fourier dimension. This is the content of the next well known lemma.

\begin{lemma}
Let $E \subset \mathbb{R}^d$ be a compact set. Let $\mu \in \mathcal{P}(E)$. Let $\varepsilon>0$. Denote by $\text{Bump}(\varepsilon)$ the set of smooth functions with support of diameter at most $\varepsilon$. Then:

$$ \dim_F \mu = \inf \{ \dim_F(\chi d\mu) \ | \ \chi \in \text{Bump}(\varepsilon) \} .$$
\end{lemma}

\begin{proof}
Let $E \subset \mathbb{R}^d$ be a fixed compact set, and let $\mu \in \mathcal{P}(E)$ be a fixed (borel) probability measure supported on $E$. First of all, consider a finite covering of the compact set $E$ by balls $(B_i)_{i \in I}$ of radius $\varepsilon$. Consider an associated partition of unity $(\chi_i)_{i \in I}$. Then, for all $\alpha < \inf_{\chi} \dim_F (\chi d\mu)$, there exists $C \geq 1$ such that:
$$ |\widehat{\mu}(\xi)| = \left|\sum_{i \in I} \widehat{\chi_i d\mu}(\xi)\right| \leq C |\xi|^{-\alpha} .$$
Hence $\dim_F \mu \geq \alpha$. Since this hold for any $\alpha < \inf \{ \dim_F(\chi d\mu) \ | \ \chi \in \text{Bump}(\varepsilon) \} $, this yields $\dim_F \mu \geq \inf \{ \dim_F(\chi d\mu) \ | \ \chi \in \text{Bump}(\varepsilon) \}.$ Now we prove the other inequality. \\

Fix some smooth function with compact support $\chi$. Its Fourier transform $\widehat{\chi}$ is in the Schwartz class: in particular, for all $N \geq d+1$, there exists $C_N$ such that $|\widehat{\chi}(\eta)| \leq C_N |\eta|^{-N}$ for all $\eta \in \mathbb{R}^d \setminus \{0\}$. Let $\alpha < \alpha' < \dim_F \mu$. Then there exists $C \geq 1$ such that $|\widehat{\mu}(\xi)| \leq C |\xi|^{-\alpha'}$ for all $\xi \in \mathbb{R}^d \setminus \{0\}$. Now, notice that:
$$ \widehat{\chi d\mu}(\xi) = \widehat{\chi} * \widehat{\mu}(\xi)  = \int_{\mathbb{R}^d} \widehat{\chi}(\eta) \widehat{\mu}(\xi-\eta) d\eta. $$
We cut the integral in two parts, depending on some radius $r>0$ that we choose to be $r := |\xi|^{1-\varepsilon}$, where $\varepsilon := 1-\alpha/\alpha'$. We suppose that $|\xi| \geq 2$. In this case, a direct computation show that whenever $\eta \in B(0,r)$, we have $|\xi|^{1-\varepsilon} \leq C |\xi-\eta|$. We are finally ready to conclude our computation:
$$ \left|  \widehat{\chi d\mu}(\xi)\right| \leq \left| \int_{B(0,r)} \widehat{\chi}(\eta) \widehat{\mu}(\xi-\eta) d\eta \right| +  \left| \int_{B(0,r)^c} \widehat{\chi}(\eta) \widehat{\mu}(\xi-\eta) d\eta \right| $$
$$   \lesssim_N  \int_{\mathbb{R}^d} |\widehat{\chi}(\eta)| d\eta \cdot \frac{C}{|\xi|^{(1-\varepsilon)\alpha'}} +   \int_{B(0,r)^c} \frac{1}{|\eta|^N} d\eta $$
$$ \lesssim_N \frac{1}{|\xi|^\alpha} + r^{d-N} \int_{B(0,1)^c} \frac{1}{|\zeta|^N} d\zeta \lesssim \frac{1}{|\xi|^{\alpha}} $$
if $N$ is choosen large enough. It follows that $\dim_F(\chi d\mu) \geq \alpha$, and this for any $\alpha<\dim_F \mu$, so $\dim_F(\chi d\mu) \geq \dim_F (\mu)$. Taking the infimum in $\chi$ yields the desired inequality.
\end{proof}

Now we understand how the Fourier dimension of a measure $\mu$ can be computed by looking at the local behavior of $\mu$. But another, much harder problem arise now: the Fourier dimension of a measure depends very much on the embedding of this measure in the ambiant space. In concrete terms, the Fourier dimension is not going to be independent on the choice of local charts. A way to introduce an \say{intrinsic} quantity related to the Fourier dimension of a measure would be to take the supremum or the infimum under all those charts.
We directly give our definition in the context of a manifold.

\begin{definition}
Let $M$ be a smooth manifold of dimension $d$. Let $E\subset M$ be a compact set. Let $\mu \in \mathcal{P}(E)$. Let $k \in \mathbb{N}^*$. Let $\text{Bump}(E)$ denote the set of all smooth functions $\chi : M \rightarrow \mathbb{R}$ such that $\text{supp}(\chi)$ is contained in a local chart. We denote by $\text{Chart}(\chi,C^k)$ the set of all $C^k$ local charts $\varphi : U \rightarrow \mathbb{R}^d$, where $U \supset \text{supp}(\chi)$ is an open set containing the support of $\chi$. Now, define the \underline{lower Fourier dimension} of $\mu$ by $C^k$ charts of $M$ by:
$$ \underline{\dim}_{F,C^k}(\mu) := \inf_{\chi \in \text{Bump}(E)} \inf \{ \dim_F( \varphi_*(\chi d\mu) ) , \ \varphi \in \text{Chart}(\chi,C^k) \}. $$
Similarly, define the \underline{upper Fourier dimension} of $\mu$ by $C^k$ charts of $M$ by:
$$ \overline{\dim}_{F,C^k}(\mu) := \inf_{\chi \in \text{Bump}(E)} \sup\{ \dim_F( \varphi_*(\chi d\mu) ) , \ \varphi \in \text{Chart}(\chi,C^k) \}. $$
\end{definition}

\begin{definition}
Let $M$ be a smooth manifold of dimension $d$. Let $E\subset M$ be a compact set. Let $\mu \in \mathcal{P}(E)$. We define the lower Fourier dimension of $\mu$ by:
$$ \underline{\dim}_F(\mu) = \underline{\dim}_{F,C^\infty}(\mu).$$
\end{definition}

\begin{remark}
The lower Fourier dimension tests if, for any localization $\chi d\mu$ of $\mu$, and for any smooth local chart $\varphi$, one has some decay of the Fourier transform of $\varphi_*(\chi d \mu)$. We then take the infimum of all the best decay exponents. This quantity is \underline{$C^\infty$-intrinsic} in the following sense: if $\Phi : M \rightarrow M$ is a $C^\infty$-diffeomorphism, then $\underline{\dim}_{F}(\Phi_*\mu) = \underline{\dim}_{F}(\mu)$. 
Symetrically, the $C^k$-upper Fourier dimension test if, for any localization $\chi d\mu$ of $\mu$, there exists a $C^k$-chart $\varphi$ such that one has some decay for the Fourier transform of $\varphi_*(\chi d \mu)$. This quantity is also $C^\infty$-intrinsic. Still, beware that the upper and lower Fourier dimensions depends on the dimension of the ambiant manifold. 
\end{remark}

\begin{remark}
Let $E \subset M$ be a compact set lying in a manifold $M$ of dimension $d$. Fix a bump function $\chi$ and a local chart $\varphi \in \text{Chart}(\chi,C^k)$. For $\mu \in \mathcal{P}(E)$ a measure supported in $E \subset M$, we have the following bounds:
$$ 0 \leq \underline{\dim}_{F,C^k} \mu \leq \dim_F \varphi_*(\chi d\mu) \leq \overline{\dim}_{F,C^k} \mu .$$
Moreover, if $\dim_H E < d$, then:
$$ \overline{\dim}_{F,C^k} \mu \leq \dim_H E. $$
\end{remark}

\begin{example}
Let $M$ be a manifold of dimension $d$, and consider any smooth hypersurface $N \subset M$. Let $k \geq 1$. Let $\mu$ be any smooth and compactly supported measure on $N$. Then: $$ \underline{\dim}_{F,C^k} (\mu) = 0, \quad \overline{\dim}_{F,C^k} (\mu) = d-1.$$
The first fact is easily proved by noticing that, locally, $N$ is diffeomorphic to a linear subspace of $\mathbb{R}^d$, which has zero Fourier dimension. The second fact is proved by noticing that, locally, $N$ is diffeomorphic to a half sphere, and any smooth measure supported on the half sphere exhibit power decay of its Fourier transform with exponent $(d-1)/2$.
\end{example}

\begin{remark}
It seems that, for some well behaved measures $\mu \in \mathcal{P}(E)$ supported on compacts $E$ with $\dim_H E < d$, one might expect the quantity $\overline{\dim}_{F,C^k} \mu$ to be comparable to $\dim_H E$. For some measures lying in a 1-dimensional curve, this is the content of Theorem 2 in \cite{Ek16}.
\end{remark}

\subsection{A variation with real valued phases}

For completeness, we suggest two variations for intrinsic notions of Fourier dimension for a measure in an abstract manifold. The first is exposed in this subsection, and the next will be discussed in the next subsection. We may want to look at more general oscillatory integrals involving $\mu$. A possibility is the following.

\begin{definition}
Let $M$ be a smooth manifold of dimension $d$. Let $E\subset M$ be a compact set. Let $\mu \in \mathcal{P}(E)$. Let $k \in \mathbb{N}^*$. Let $\text{Bump}(E)$ denote the set of all smooth functions $\chi : M \rightarrow \mathbb{R}$ such that $\text{supp}(\chi)$ is contained in a local chart. We denote by $\text{Phase}(\chi,C^k)$ the set of all real valued and $C^k$ maps $\psi : U \rightarrow \mathbb{R}$ that satisfies $(d \psi)_x \neq 0$ for all $x \in U$, where $U \supset \text{supp}(\chi)$ is an open set containing the support of $\chi$. Now, define the \underline{lower Fourier dimension} of $\mu$ by $C^k$ phases of $M$ by:
$$ \underline{\dim}_{F,C^k}^{\text{real}}(\mu) := \inf_{\chi \in \text{Bump}(E)} \inf \{ \dim_F( \psi_*(\chi d\mu) ) , \ \psi \in \text{Phase}(\chi,C^k) \}. $$
Similarly, define the \underline{upper Fourier dimension} of $\mu$ by $C^k$ phases of $M$ by:
$$ \overline{\dim}_{F,C^k}^{\text{real}}(\mu) := \inf_{\chi \in \text{Bump}(E)} \sup\{ \dim_F( \psi_*(\chi d\mu) ) , \ \psi \in \text{Phase}(\chi,C^k) \}. $$
As before, we also denote $\underline{\dim}_F^{\text{real}}(\mu) := \underline{\dim}_{F,C^\infty}^{\text{real}}(\mu)$.
\end{definition}

\begin{remark}
First of all, notice that $\psi_*(\chi d\mu)$ is a measure supported in $\mathbb{R}$, so its Fourier transform is a function from $\mathbb{R}$ to $\mathbb{C}$. More precisely:
$$ \forall t \in \mathbb{R}, \ \widehat{\psi_*(\chi d\mu)}(t) := \int_{E} e^{i t \psi(x)} \chi(x) d\mu(x) . $$
Like before, the lower/upper Fourier dimensions with real phases are $C^\infty$-intrinsic in the sense that for any $C^\infty$-diffeomorphism $\Phi : M \rightarrow M$, we have $\underline{\dim}_{F,C^k}^{\text{real}}(\Phi_*\mu) = \underline{\dim}_{F,C^k}^{\text{real}}(\mu) $ and $\overline{\dim}_{F,C^k}^{\text{real}}(\Phi_*\mu) = \overline{\dim}_{F,C^k}^{\text{real}}(\mu)$. 
\end{remark}

\begin{example}
Let $M$ be a smooth manifold, and let $N$ be a smooth submanifold of $M$, with $\text{dim} \ N < \text{dim} \ M$. Let $\mu$ be a smooth and compactly supported probability measure in $N$. Then:
$$ \underline{\dim}_{F,C^k}^{\text{real}}(\mu) = 0 , \quad \overline{\dim}_{F,C^k}^{\text{real}}(\mu) = \infty. $$
These equalities can be proved as follow. Consider some smooth bump function $\chi$ with small enough support. Now, there exists a phase $\psi$, defined on a neighborhood $U$ of $\text{supp}(\chi)$, with nonvanishing differential on $U$ but which is constant on $N$. The associated oscillatory integral $\widehat{\psi_*(\chi d\mu)}$ doesn't decay, hence the computation on the lower Fourier dimension with real phases. There also exists a smooth phase $\psi$ such that $(d\psi)_{|TN}$ doesn't vanish. By the non-stationnary phase lemma, the associated oscillatory integral decay more than $t^{-N}$, for any $N \geq 0$, hence the computation on the upper Fourier dimension with real phases. \\
Notice how, in particular, $\min(\overline{\dim}_{F,C^k}^{real}(\mu),d)$ may be strictly larger than the Hausdorff dimension of the support of $\mu$. This may be a sign that this variation of the upper dimension isn't well behaving as a \say{Fourier dimension}. 
\end{example}

\begin{lemma}
We can compare this Fourier dimension with the previous one. We have:
$$ \underline{\dim}_{F,C^k}(\mu) \leq \underline{\dim}_{F,C^k}^{real}(\mu) , \quad \overline{\dim}_{F,C^k}(\mu) \leq \overline{\dim}_{F,C^k}^{real}(\mu) . $$

\end{lemma}

\begin{proof}
Let $\alpha<\underline{\dim}_{F,C^k}(\mu)$. Then, for any bump function $\chi$ and for any associated local chart $\varphi$, there exists some constant $C$ such that, for all $\xi \in \mathbb{R}^d \setminus \{0\}$, we have $|\widehat{\varphi_*(\chi d\mu)}(\xi)| \leq C |\xi|^{-\alpha/2}$. Now fix $\psi : U \rightarrow \mathbb{R}$ with nonvanishing differential, where $\text{supp}(\chi) \subset U$. By the submersion theorem, there exists a local chart $\varphi:U \rightarrow \mathbb{R}^d$ such that $\varphi(x) = \psi(x) e_1 + \sum_{j=2}^d f_j(x) e_j$ (where $(e_i)_i$ is the canonical basis of $\mathbb{R}^d$, and where $f_j$ are some real valued functions). Hence, one can write:
$$ |\psi_*(\chi d\mu)(t)| = |\varphi_*(\chi d\mu)(t e_1)| \leq C |t|^{-\alpha/2}.$$
Hence $\underline{\dim}_{F,C^k}^{real}(\mu) \geq \alpha$, and this for any $\alpha < \underline{\dim}_{F,C^k}(\mu)$, hence the desired bound. \\

The second bound is proved as follow. Let $\alpha < \overline{\dim}_{F,C^k}(\mu)$. Let $\chi$ be a small bump function. There there exists a local chart $\varphi : U \rightarrow \mathbb{R}^d$, with $U \supset \text{supp}(\chi)$, such that $|\widehat{\varphi_*(\chi d\mu)}(\xi)| \lesssim |\xi|^{-\alpha/2}$. Let $u \in \mathbb{S}^{d-1}$ and consider $\psi(x) := u \cdot \varphi(x)$. It is easy to check that $\psi$ has nonvanishing differential, and since, for any $t \in \mathbb{R} \setminus \{0\}$,
$$ |\widehat{\psi_*(\chi d\mu)}(t)| = |\widehat{\varphi_*(\chi d\mu)}(ut)| \lesssim |t|^{-\alpha/2} ,$$
we get $\overline{\dim}_{F,C^k}^{real}(\mu) \geq \alpha$. The bound follow.
\end{proof}

In concrete cases, we expect the lower Fourier dimension and the lower Fourier dimension with real phases to be equal. Unfortunately, our choices of definitions doesn't clearly make that happen all the time. We have to add a very natural assumption for the equality to hold.

\begin{definition}
Let $\mu \in \mathcal{P}(E)$, where $E \subset M$ is a compact subset of a smooth manifold. We say that $\mu$ admits reasonable constants for $C^k$-phases if, for any $\alpha < \underline{\dim}^{\text{real}}_{F,C^k}(\mu)$, and for any $\chi \in \text{Bump}(E)$, the following hold:
$$ \forall R \geq 1, \ \exists C_R \geq 1, \ \forall \psi \in \text{Phase}(\chi,C^k), $$ $$  \left( \|\psi\|_{C^k} + \sup_{x \in U} \| (d\psi)_x \|^{-1} 
 \leq R \right) \Longrightarrow \left( \forall t \in \mathbb{R}^*, \ |\widehat{\psi_*(\chi d\mu)}(t)| \leq C_R t^{-\alpha/2} \right). $$
\end{definition}

Under this natural assumption, we have equality of the lower Fourier dimensions.

\begin{lemma}
Let $\mu\in \mathcal{P}(E)$, where $E \subset M$ is a compact subset of some smooth manifold $M$. Suppose that $\mu$ admits reasonable constants for $C^k$-phases. Then:
$$ \underline{\dim}_{F,C^k}(\mu) = \underline{\dim}_{F,C^k}^{real}(\mu) $$
\end{lemma}

\begin{proof}
An inequality is already known, we just have to prove the second one. Let $\alpha < \underline{\dim}_{F,C^k}^{\text{real}}(\mu)$. Let $\chi \in \text{Bump}(E)$, and let $\varphi \in \text{Chart}(\chi,C^k)$. Now, fix any $\xi_0 \in \mathbb{R}^d$ of norm one. Then, by definition of the lower Fourier dimension with $C^k$ phases, and by our hypothesis that $\mu$ admits reasonable constants for $C^k$-phases, we know that there exists a constant $C_0$, uniform in the choice of $\xi_0$, such that
$$ \forall t \in \mathbb{R}^*, \Big{|} \int_{E} e^{i t (\xi_0 \cdot \varphi(x)) } \chi(x) d\mu(x) \Big{|} \leq C_0 t^{-\alpha/2} .$$
It follows that $\alpha < \underline{\dim}_{F,C^k}(\mu)$. The other inequality is proved, and hence $\underline{\dim}_{F,C^k}(\mu) = \underline{\dim}_{F,C^k}^{\text{real}}(\mu)$ in this setting.
\end{proof}

\subsection{A directional variation}

A second natural and intrinsic idea would be to fix some (spatial) direction on which to look for Fourier decay. We quickly discuss these notions and then we will move on to discuss some notions of Fourier dimensions for sets.

\begin{definition}
Let $E \subset M$ be a compact set in some smooth manifold. Let $V \subset TM$ be a continuous vector bundle on an open neighborhood $\tilde{E}$ of $E$. Denote by $\text{Bump}^V(E)$ the set of all smooth bump functions with support included in $\tilde{E}$, and included in some local chart. For some $\chi \in \text{Bump}^V(E)$, denote by $\text{Phase}^V(\chi,C^k)$ the set of all $C^k$ maps $\psi:U \rightarrow \mathbb{R}$ such that $(d\psi)_{|V}$ doesn't vanish on $U$, where $\text{supp}(\chi) \subset U \subset \tilde{E}$ is some open set. \\

For $\mu \in \mathcal{P}(E)$, we define its \underline{lower Fourier dimension} in the direction $V$ for $C^k$ phases by:
$$ \underline{\dim}_{F,C^k}^V(\mu) := \inf_{\chi \in \text{Bump}^V(E)} \inf \{ \dim_F( \psi_*(\chi d\mu) ) , \ \psi \in \text{Phase}^V(\chi,C^k) \} .$$
Similarly, define its \underline{upper Fourier dimension} in the direction $V$ for $C^k$ phases by:
$$ \overline{\dim}_{F,C^k}^V(\mu) := \inf_{\chi \in \text{Bump}^V(E)} \sup \{ \dim_F( \psi_*(\chi d\mu) ) , \ \psi \in \text{Phase}^V(\chi,C^k) \} .$$
\end{definition}

\begin{remark}
Again, these notions of Fourier dimensions are $C^\infty$-intrinsic, in the following sense: if $\Phi:M \rightarrow M$ is a $C^k$-diffeomorphism of $M$, then $  \underline{\dim}_{F,C^k}^{\Phi_*V}(\Phi_*\mu) =  \underline{\dim}_{F,C^k}^V(\mu) $, and $ \overline{\dim}_{F,C^k}^{\Phi_*V}(\Phi_*\mu) =  \overline{\dim}_{F,C^k}^V(\mu)$.
\end{remark}

\begin{lemma}
Let $V_1,\dots,V_n \subset TM$ be some continuous vector bundles defined on some open neighborhood $\tilde{E}$ of $E$. Suppose that $(V_1)_p + \dots (V_n)_p = T_p M$ for all $p \in \tilde{E}$. Then:
$$ \min_{j} \underline{\dim}_{F,C^k}^{V_j}(\mu) = \underline{\dim}_{F,C^k}^{\text{real}}(\mu) , \quad \max_{j} \overline{\dim}_{F,C^k}^{V_j}(\mu) \leq \overline{\dim}_{F,C^k}^{\text{real}}(\mu) .$$ \end{lemma}

\begin{proof}
Let $\alpha < \underline{\dim}_{F,C^k}^{\text{real}}(\mu)$. Then, for any bump $\chi$ and associated phase $\psi$, one has $|\widehat{\psi_*(\chi d\mu)}(t)| \lesssim |t|^{-\alpha/2}$. In particular, for any phase $\phi_j \in \text{Phase}^{V_j}(\chi,C^k)$, the previous decay holds, and so $\min_{j} \underline{\dim}_{F,C^k}^{V_j}(\mu) \geq \alpha$. Hence $\min_{j} \underline{\dim}_{F,C^k}^{V_j}(\mu)  \geq \underline{\dim}_{F,C^k}^{\text{real}}(\mu)$. \\

Now let $\alpha < \min_{j} \underline{\dim}_{F,C^k}^{V_j}(\mu)$. Then, for all $j$, for any bump $\chi$, and for any phase $\psi_j \in \text{Phase}^{V_j}(\chi,C^k)$, the previous decay applies. Now, if we fix some $\chi$ and some associated phase $\psi \in \text{Phase}(\chi,C^k)$, we know that at each point $p$, $(d\psi)_p$ is nonzero. In particular, there exists $j(p)$ such that $(d\psi)_{|V^{j(p)}_p} \neq 0$. Using a partition of unity to localize, this implies $\widehat{\psi_*(\chi d\mu)}(t) \lesssim |t|^{-\alpha/2}$. Hence $\underline{\dim}_{F,C^k}^{\text{real}}(\mu) \geq \alpha$, and we have proved equality. \\

For our last bound, let $\alpha < \max_j \overline{\dim}_{F,C^k}^{V_j} \mu$. Then there exists $j$ such that, for all bump $\chi$, there exists an associated phase $\psi_j \in \text{Phase}^{V_j}(\chi,C^k)$ such that $|\widehat{\psi_j(\chi d\mu)}(t)| \lesssim |t|^{-\alpha/2}$. Since $\psi_j \in \text{Phase}(\chi,C^k)$, we get $\overline{\dim}_{F,C^k}^{\text{real}} \mu \geq  \max_j \overline{\dim}_{F,C^k}^{V_j} \mu$.
\end{proof}

\begin{remark}
The reverse bound for the upper dimensions is not clear: if for all bump functions $\chi$, there exists a phase $\psi$ with good fourier decay properties for $\mu$, then nothing allows us to think that $\psi$ is going to have nonvanishing differential in some fixed $V_j$ on all $E$.
\end{remark}
\newpage
\subsection{What about sets ?}

We finally define some intrinsic notions of Fourier dimensions for sets. First of all, recall that the usual definition for some $E \subset \mathbb{R}^d$ is:

$$\dim_F(E) := \sup \{ \min(d,\dim_F \mu) , \ \mu \in \mathcal{P}(E) \} \leq \dim_H E.$$

In particular, in view of the proof of Lemma 1.1.14, we see that any measure $\mu$ with some Fourier decay properties may be localized anywhere on its support to still yield a measure with large Fourier dimension. Hence we find the following \say{localized} formula, for any $\varepsilon>0$:

 $$ \dim_F(E) = \underset{\text{diam}(U) \leq \varepsilon}{\underset{U \text{ open}}{\sup_{U \cap E \neq \emptyset}}} \dim_F (E \cap U).$$ 

Now, we have two main ways to define the up(per) and low(er) Fourier dimension of a compact set in a manifold: directly computing the Fourier dimension of parts of $E$ in local charts, or taking the sup over the previously defined notions of Fourier dimension for measures. 

\begin{definition}
Let $E \subset M$ be a nonempty compact set included in some smooth manifold. We define its lower/upper Fourier dimension with $C^k$-charts by:
$$ \underline{\dim}_{F,C^k}(E) := \sup\{  \underline{\dim}_{F,C^k}(\mu)  \leq d, \ \mu \in \mathcal{P}(E)\}, \quad \overline{\dim}_{F,C^k}(E) := \sup\{  \overline{\dim}_{F,C^k}(\mu) \leq d, \ \mu \in \mathcal{P}(E)\}. $$
We also define the $C^k$-low Fourier dimension and $C^k$-up Fourier dimensions of $E$ by:
$$ \uwave{\dim}{}_{F,C^k}(E) := \underset{U \text{ open chart}}{\sup_{U \cap E \neq \emptyset}} \inf \{\dim_F (\varphi(E \cap U)) \ , \ \varphi:U \rightarrow \mathbb{R}^d \ C^k \text{ local chart} \}, $$
$$ {\owave{\dim}}_{F,C^k}(E) := \underset{U \text{ open chart}}{\sup_{U \cap E \neq \emptyset}} \sup \{\dim_F (\varphi(E \cap U)) \ , \ \varphi:U \rightarrow \mathbb{R}^d \ C^k \text{ local chart} \}. $$
\end{definition}

\begin{remark}
The low and up Fourier dimension are $C^k$-intrinsic in the natural sense. For exemple, if $\Phi : M \rightarrow M$ is a $C^k$-diffeomorphism, then $ \uwave{\dim}{}_{F,C^k}(\Phi(E)) =  \uwave{\dim}{}_{F,C^k}(E) $. The lower and upper Fourier dimension are $C^\infty$-intrinsic.
\end{remark}

\begin{lemma}
Let $E \subset M$ be a compact set in some smooth manifold $M$. Then:
$$ 0 \leq \underline{\dim}_{F,C^k}(E) \leq \uwave{\dim}{}_{F,C^k}(E) \leq {\owave{\dim}}{}_{F,C^k}(E) =  \overline{\dim}_{F,C^k}(E) \leq \dim_H(E) \leq d  .$$
\end{lemma}

\begin{proof}
Let us prove all the inequalities in order, from left to right. $ 0 \leq \underline{\dim}_{F,C^k}(E) $ is trivial. Let us prove the second one. \\

Let $\alpha<\underline{\dim}_{F,C^k}(E)$. By definition, there exists some probability measure $\mu \in \mathcal{P}(E)$ such that $\underline{\dim}_{F,C^k}(\mu) \geq \alpha$.
Now, since the support of $\mu$ is nonempty, there exists $U$ some small open set and a bump function $\chi$ supported in $U$ such that $\chi d\mu$ is a (localized) nonzero measure. Let $\varphi : U \rightarrow \mathbb{R}^d$ a local chart. Then, by hypothesis on $\mu$, $\dim_F \varphi_*(\chi d\mu) \geq \alpha$. In particular, since (up to normalization) $\varphi_*(\chi d\mu) \in \mathcal{P}(\varphi(E \cap U)) $, we have $\dim_F \varphi(E \cap U) \geq \alpha$. This holds for any local chart $\varphi$, and so $\inf_{\varphi} \dim_F(\varphi(E \cap U)) \geq \alpha$. This yields $ \uwave{\dim}{}_{F,C^k}(E) \geq \alpha$. Since this is true for any $\alpha<\underline{\dim}_{F,C^k}(E)$, we get the desired inequality. \\

The inequality $\uwave{\dim}{}_{F,C^k}(E) \leq {\owave{\dim}}{}_{F,C^k}(E) $ is trivial. Let us prove the equality between ${\owave{\dim}}{}_{F,C^k}(E)$ and $\overline{\dim}_{F,C^k}(E)$. Let $\alpha < {\owave{\dim}}{}_{F,C^k}(E)$. Then, there exists some small open set $U$ such that $U \cap E \neq \emptyset$ and a local chart $\varphi:U \rightarrow \mathbb{R}^d$ such that $\dim_F(\varphi(U \cap E)) \geq \alpha$. By definition, it means that there exists some measure $\nu \in \mathcal{P}(\varphi(E \cap U))$ such that $\dim_F \nu \geq \alpha$. Letting $\mu := \varphi^{-1}_*\nu \in \mathcal{P}(E \cap U)$ yields a measure supported in $E$ that satisfies $\overline{\dim}_{F,C^k}(\mu) \geq \alpha$ (in view of the proof of Lemma 1.1.14). Hence, $\overline{\dim}_{F,C^k}(E) \geq \alpha$. This, for any $\alpha < {\owave{\dim}}{}_{F,C^k}(E)$, so that $\overline{\dim}_{F,C^k}(E) \geq {\owave{\dim}}{}_{F,C^k}(E)$. \\

Let us prove the other inequality. Let $\alpha < \overline{\dim}_{F,C^k}(E) $. By definition, there exists $\mu \in \mathcal{P}(E)$ such that $\overline{\dim}_{F,C^k}(\mu) \geq \alpha$. Now let $U$ be some small open set with $\mu_{|U} \neq 0$, and let $\varphi:U \rightarrow \mathbb{R}^d$ be a local chart. Let $\chi$ be some bump function supported in $U$. Then, by hypothesis on $\mu$, we have $\dim_F \varphi_*(\chi d\mu) \geq \alpha$. In particular, $\dim_F( \varphi(E \cap U) ) \geq \alpha.$ Hence $\owave{\dim}_{F,C^k}(E) \geq \alpha$. This proves the other inequality, and hence concludes the proof that ${\owave{\dim}}{}_{F,C^k}(E) =  \overline{\dim}_{F,C^k}(E)$. \\

Finally, the fact that the Hausdorff dimension is invariant under $C^1$-diffeomorphisms implies
$$ \owave{\dim}_{F,C^k}(E) = \sup_{U} \sup_{\varphi} \dim_F(\varphi(U \cap E)) \leq  \sup_{U} \sup_{\varphi} \dim_H( \varphi(E \cap U)) $$ $$ = \sup_{U} \dim_H(E \cap U) = \dim_H(E) \leq d. $$ \end{proof}

\begin{example}
Let $N \subset M$ be a hypersurface in some smooth manifold $M$ of dimension $d$. Then:
$$ \underline{\dim}_{F,C^k}(N) = \uwave{\dim}{}_{F,C^k}(N) = 0 \quad, \owave{\dim}_{F,C^k}(N) = \overline{\dim}_{F,C^k}(N) = \dim_H(N) = d-1 .$$
\end{example}

\begin{example}
Let $L$ be a 1-dimensional manifold, and let $E \subset L$ be \underline{\textbf{any compact subset}}. Then:
$$ \overline{\dim}_{F,C^1} E = \dim_H E. $$
This very striking result is proved in \cite{Ek16}. Also, Ekstrom proves that, for any $k \geq 1$, we have $\overline{\dim}_{F,C^k} E \geq (\dim_H E)/k$. This motivates the following question: do we have, for any compact set $E$ in any manifold $M$, the formula $\overline{\dim}_{F,C^1}(E) = \dim_H(E) $ ? 
\end{example}

\begin{remark}
Other natural questions are the following. Can we find an example of set $E \subset \mathbb{R}^d$ such that $\underline{\dim}_{F,C^k}(E) < \uwave{\dim}{}_{F,C^k}(E)$ ? Or is it always an equality ? Is the lower Fourier dimension $C^k$-intrinsic ?
\end{remark}

For completeness, we conclude by introducing the real variation for the lower Fourier dimension. We will not introduce this variation for the upper Fourier dimension, as we said earlier that these seems to behave quite badly with respect to the Hausdorff dimension. To keep it concise, we will not discuss the directional variations.

\begin{definition}
Let $E \subset M$ be a compact subset of some smooth manifold $M$. Define the lower Fourier dimension with $C^k$-phases by:
$$ \underline{\dim}_{F,C^k}^{\text{real}}(E) := \sup \{ \underline{\dim}_{F,C^k}^{\text{real}} \mu \leq d \ , \ \mu \in \mathcal{P}(E) \}. $$
\end{definition}

\begin{remark}
By Lemma 1.1.25, we see that $\underline{\dim}_{F,C^k}(E) \leq \underline{\dim}_{F,C^k}^{\text{real}}(E)$. Is this an equality, or are we able to produce an exemple were this inequality is strict ? A related question is: if we denote by $\mathcal{P}_{reas,C^k}(E)$ the set of probability measures that admits reasonables constants for $C^k$-phases (see Definition 1.1.26), do we have
$$ \underline{\dim}_{F,C^k}(E) = \sup\{ \underline{\dim}_{F,C^k} \mu \leq d  , \ \mu \in \mathcal{P}_{reas,C^k}(E) \} \quad ? $$
\end{remark}

\subsection{Some sets with positive Fourier dimension.}

As we saw earlier with the example of the triadic cantor set, while it is clear for $E \subset \mathbb{R}^d$ that $0 \leq \dim_F E \leq \dim_H E \leq d$, equality between Hausdorff and Fourier dimension doesn't always happen. Still, for \say{chaotic sets}, we expect equality to happen: in this case, we call $E$ a Salem Set. In particular, we expect sets with positive Fourier dimension to be quite common. \\

The easiest way to find Salem sets (or sets with positive Fourier dimension) is to construct sets via random processes. In fact, Salem \cite{Sa51} constructed the first Salem sets as random Cantor sets, thus proving the existence of a Salem set of Hausdorff dimension $s$ in $\mathbb{R}$ for any $s \in [0,1]$. Then, Kahane showed that Salem sets can be found by considering the image of some sets under random stochastic processes, such as the Brownian motion (see, for example, \cite{Ka66a}, \cite{Ka66b} and \cite{Ka85}), thus constructing (random) Salem sets of any given Hausdorff dimension.
The random setting is still widely studied in our days: authors in this topic include Bluhm \cite{Bl96}, Ekstrom \cite{Ek16}, Łaba and Pramanik \cite{ŁP09}, Chen and Seeger \cite{CS17}, Shmerkin
and Suomala \cite{SS18}, Fraser and Sahlsten \cite{FS18}... \newpage

One of the first \emph{deterministic} set with positive Fourier dimension was discovered by Kaufmann in 1980 \cite{Ka80} (the construction is related to continued fractions). A year later, Kaufmann \cite{Ka81} 
constructed a deterministic example of a (fractal)  Salem set in $\mathbb{R}$. The methods found in the first paper were generalized by Queffelec and Ramare \cite{QR03}, and then by Jordan and Sahlsten \cite{JS16}, producing Fourier decay results for a large family of invariant measures for the Gauss map. The construction of the deterministic Salem set was generalized in $\mathbb{R}^2$ in 2017 \cite{Ha17}, and then in $\mathbb{R}^d$ in 2023 \cite{FH23}: those constructions relies on arithmetic considerations. \\

For sets of positive Fourier dimension, instead of trying to \say{construct} deterministic example, we can try to find \say{natural} examples, in the sense that these are going to appear as invariant subsets for some hyperbolic dynamical system. As we saw in the case of the triadic Cantor set, linearity of the dynamical system and some (difficult to grasp) algebraic conditions on the coefficients seems to sometimes be an obstruction  to the positivity of the Fourier dimension. But if the (linear) coefficients are taken randomly, then one can expect positivity of the Fourier dimension. On this subject, one might refer to papers related to the study of Cantor sets coming from linear IFS, such as \cite{Er40}, \cite{LS19a}, \cite{LS19b}, \cite{So19}, \cite{Br19}, \cite{VY20}, and \cite{Ra21}. Positivity of the Fourier dimension for Cantor sets coming from linear IFS is a very subtle question, involving complex number-theoric considerations. As a matter of example, to the authors knowledge, the following question is open.

\begin{question}
Let $\mu \in \mathcal{P}( [-1,1] )$ be the Bernoulli convolution given as the probability measure invariant under the IFS $\{ x \mapsto \frac{1}{\pi} x + 1, x \mapsto \frac{1}{\pi} x - 1 \}$ with probabilities $\{1/2,1/2\}$. Do we have $\dim_F(\mu) > 0$ ?
\end{question}

In this PhD, we will be only interested in the case where the underlying dynamical system is \emph{nonlinear}. In this setting, the desired result (positivity of some Fourier dimension) has more chance to be true, but the counterpart is that explicit computations are now almost impossible, and we have to rely on more sophisticated dynamical and analytic tools. An important method for us, based on a result coming from additive combinatorics (the \say{sum-product phenomenon}), was developed in 2017 by Bourgain and Dyatlov \cite{BD17}. The context is the following. \\

If we fix some Schottky group $\Gamma < \text{PSL}(2,\mathbb{R})$, then $\Gamma$ acts naturally on $\mathbb{R}$ (using Moebius transformations). One can prove that there exists a Cantor set  $\Lambda_\Gamma \subset \mathbb{R}$, called a limit set, which is invariant by $\Gamma$. On this limit set, there is a family of natural probability measures associated with the dynamics that are called Patterson-Sullivan measures. Using the \say{sum-product phenomenon}, Bourgain and Dyatlov managed to prove power decay for the Fourier transform of those probability measures. In particular, the limit set $\Lambda_\Gamma \subset \mathbb{R}$ has positive Fourier dimension. In fact, if we denote by $\mu \in \mathcal{P}(\Lambda_\Gamma)$ \say{the} Patterson-Sullivan measure, they even prove $\underline{\text{dim}}_{F,C^2}(\mu)>0$. An essential feature of $\Gamma$ was the \emph{nonlinearity} of Moebius transformations.  \\

The method introduced in this paper inspired numerous generalizations, beginning with a paper of Li, Naud and Pan \cite{LNP19} proving power decay for Patterson-Sullivan measures related to (Zariski dense) Kleinian Schottky groups $\Gamma < \text{PSL}(2,\mathbb{C})$. Using the sum-product phenomenon, they obtain positivity of the Fourier dimension of the associated limit set $\Lambda_\Gamma \subset \mathbb{C}$. In fact, as earlier, denoting by $\mu$ the Patterson-Sullivan measure in their setting, they prove $\underline{\text{dim}}_{F,C^2}(\mu)>0$. \\

An important generalization was done by Sahlsten and Steven soon after \cite{SS20}. A central remark is that Patterson-Sullivan measures belong to a general family of well studied measures in the field of hyperbolic dynamical systems: equilibrium states. The result of Sahlsten and Steven, then, goes as follow: for one-dimensional and \say{totally nonlinear} IFS, one can show that any equilibrium state exhibit power Fourier decay. In particular, Sahslten and Steven obtain positive Fourier dimension for a large class of \say{nonlinear} Cantor sets (in fact, denoting by $\mu$ some equilibrium state, they obtain $\underline{\text{dim}}_{F,C^2}(\mu)>0$ under a nonlinearity condition on the underlying expanding dynamics). This paper uses the method introduced by Bourgain-Dyatlov, and is also inspired by previous techniques appearing in \cite{JS16}. See also \cite{ARW20} for some related work on nonlinear IFS and pointwise normality. \\

Another interesting generalization of this method was done by Li in \cite{Li18} and \cite{Li20}, proving decay results for the Fourier transform of stationary measures with exponential moment for random walks on split groups. Past the split setting, further results seem difficult to achieve. 

We are finally ready to explain our contributions. This thesis is divided into 6 chapters, the first being the introduction, where we still have to introduce the Thermodynamical formalism (and the notion of \say{equilibrium state}), and where we will also discuss this \say{sum-product phenomenon}. Chapters 2 to 6 are each adapted from a paper/preprint of the author. \\

Chapter 2 is devoted to explaining the method found in \cite{SS20} in a simple context, where the expanding dynamical system is a perturbation of the doubling map. This chapter contains a new way to check the \say{non-concentration hypothesis} that is needed to apply the \say{sum-product phenomenon}, and is adapted from \cite{Le22}. \\

In Chapter 3, we work in the context of complex dynamics. More specifically, we will give ourselves a rational map acting on the Riemann sphere, and we will consider its Julia set. Under a hyperbolicity hypothesis, and supposing that the Julia set isn't included in a circle, we will prove positivity of the lower Fourier dimension of some equilibrium states in this context. This chapter is adapted from \cite{Le21}. \\

In Chapter 4 and 5, we conduct a study of the Fourier dimension of basic sets for hyperbolic diffeomorphisms. Chapter 4 is adapted from \cite{Le23} and deals with nonlinear, bunched attractors with codimension one unstable direction. In this chapter, we are only able to prove Fourier decay for equilibrium states in the unstable direction. Using our previously defined language, we prove that $\underline{\dim}_{F,C^2}^{E^u}(\mu) > 0$ for some equilibrium states $\mu$, under a technical nonconcentration condition that we study in detail in the next Chapter. \\

Chapter 5 deals with the case where the dynamics is area-preserving on a surface. We are able this time to prove positivity of the lower Fourier dimension of equilibrium states in a nonlinear setting. This was done by adapting ideas from a recent work of Tsujii and Zhang about exponential mixing of 3-dimensional Anosov flows. Unfortunately, this time, we are not able to prove that the nonlinearity condition is generic on the dynamics, even though we have good reasons to think so. This chapter is adapted from some recent work that is not yet online. \\

Finally, Chapter 6 deals with equilibrium states for the geodesic flow on (convex)-cocompact hyperbolic surfaces. While positivity of the lower Fourier dimension of the non-wandering set for the geodesic flow can be seen as a corollary of \cite{BD17}, we prove positivity of the lower Fourier dimension of many equilibrium states, which is new. To do so, we prove that the Patterson-Sullivan densities (that generalizes the one found in \cite{BD17} and \cite{LNP19}) can be seen as stationary measures with exponential moment for good random walks (under a reasonable regularity condition on the measure). This allows us to conclude, using the previously mentionned work of Li \cite{Li20}. This chapter is adapted from \cite{Le23b}. \\

Before heading into a discussion regarding the Thermodynamical formalism, let us quickly discuss some potential applications of Fourier decay for fractal measures.

\subsection{Consequences of positivity for the Fourier dimension}

Establishing Fourier decay for measures supported on fractal sets have several corollaries. Let us quickly mention a couple of them.

\begin{itemize}
    \item A first application of Fourier decay is related to the old question of findings \say{sets of uniqueness}. We say that a subset $E \subset \mathbb{R}$ is a set of multiplicity if there exists a sequence $(c_n)_{n \in \mathbb{Z}}$ such that the following holds:
    $$ \Big( \forall x \notin E, \sum_{|n|\leq N} c_n e^{2 i \pi n x} \underset{N \rightarrow \infty}{\longrightarrow} 0 \Big) \ \text{and} \ \Big( \exists x \in E, \sum_{|n|\leq N} c_n e^{2 i \pi n x} \underset{N \rightarrow \infty}{ 
\centernot \longrightarrow} 0 \Big) .$$
    We say that $E$ is a set of uniqueness if it is not a set of multiplicity. The question of understanding such sets is of fundamental importance in the general theory of Fourier series. It is known (see \cite{KS64}, Chapter 5) that if a compact set $E \subset \mathbb{R}$ admits a probability measure $\mu \in \mathcal{P}(E)$ satisfying $\widehat{\mu}(\xi) {\longrightarrow} 0$ when $|\xi| \rightarrow \infty$, then $E$ is a set of multiplicity. In particular, the nonlinear Cantor sets found in the work of Sahlsten and Steven \cite{SS20} are sets of multiplicity. It is interesting to notice that the triadic Cantor set is known to be a set of uniqueness. In fact, for a \say{linear} Cantor set with constant zoom ratio $\xi \in (0,1/2)$, it is known that the property of uniqueness or multiplicity is linked to the arithmetic properties of $\xi$ (see, for example, \cite{Sa43} or \cite{Kö92}).

    \item One of the main motivations for the work of Bourgain and Dyatlov \cite{BD17} was to prove a \say{Fractal Uncertainty Principle} (FUP) using Fourier decay properties. Let us precise what we mean. Consider a (fractal) set $E \subset \mathbb{R}^d$ equipped with some probability measure $\mu \in \mathcal{P}(E)$. We say that $(E,\mu)$ satisfies (FUP) if the family of linear operators $B(h) : L^2(E,\mu) \rightarrow L^2(E,\mu)$, with $h \in (0,1)$, defined by
    $$ B(h)u(\xi) := \int_E e^{i x \xi/h} u(x) d\mu(x) $$
    satisfies the bound
    $$ \| B(h) \|_{L ^2(E,\mu) \rightarrow L ^2(E,\mu)} \leq C h^{\varepsilon} $$
    for some $C, \varepsilon>0$. In low dimensions, it appears (see, for example, \cite{BD17} and \cite{BS23}) that positivity of the Fourier dimension for $\mu$ is enough to prove such results. These kind of estimates can be useful to obtain essential spectral gaps for some manifolds with negative curvature, see \cite{DZ16}.
    \item Another application of power Fourier decay for a probability measure $\mu \in \mathcal{P}(E)$ is for the restriction problem (see \cite{Ma15}). The question is the following: we know that the classical Fourier transform can be seen, for example, as an operator $L^1(\mathbb{R}^d) \rightarrow C^0_{\rightarrow 0}(\mathbb{R}^d)$, or as an operator $L^2(\mathbb{R}^d) \rightarrow L^2(\mathbb{R}^d)$. In particular, notice how we can meaningfully restrict $\widehat{f}$ to a point if $f \in L^1$, but how we can't do the same for a typical $f \in L^2$. The restriction problem, then, goes as follow: being given a fractal $E$ and some function $f \in L^p$, $p \in (1,2)$, can we meaningfully restrict $\widehat{f}$ to $E$ ? If $\text{dim}_F(E)>0$, it appear that the answer is yes, if $p$ is close enough to one. Let us precise what we mean. \\
    Suppose that $\mu \in \mathcal{P}(E)$ is such that $\text{dim}_F(\mu)>0$. Then $\widehat{\mu} \in L^q$ for $q>q_0$ large enough. Let $p := (1-(2q)^{-1})^{-1} > 1$, and let $p' > 1$ be such that $p^{-1} + (p')^{-1} = 1$. Now, for any $f : \mathbb{R}^d \rightarrow \mathbb{C}$ in the Schwartz class, we can write, using Parseval's formula, Hölder's inequality and Young's inequality:
    $$ \| \widehat{f} \|_{L^2(\mu)}^2 = \langle f , f * \widehat{\mu} \rangle \leq \|f\|_{L^p(\mathbb{R}^d)} \| f * \widehat{\mu}\|_{L^{p'}(\mathbb{R}^d)} \leq \| f\|_{L^p(\mathbb{R}^d)}^2 \|\widehat{\mu} \|_{L^q(\mathbb{R}^d)} .$$
    In particular, we see that we can meaningfully define a (bounded) restriction operator for the Fourier transform $\mathcal{F}_{|E} : L^p(\mathbb{R}^d) \rightarrow L^2(E,\mu) $ if $1 \leq p \leq p_0(\mu)$ is small enough. One can try to find the best $p_0$ that satisfies this property: on this topic, see \cite{ŁW16}.
\end{itemize}
The link between positivity of Fourier dimension and arithmetic properties of a set probably have various other corollaries. But now, let us turn into the next part of this introduction: we will discuss the Thermodynamical formalism, which will allow us to construct some well behaved measures in a dynamical context.

\section{The Thermodynamical formalism}


\subsection{Equilibrium states on shift spaces}

In this subsection, we recall various results on the measures that will mostly interest us in our thesis: equilibrium states. To keep this introduction concise, we will only discuss the construction in the context of shifts spaces. Let us first briefly recall why this is an interesting setting.

\begin{remark}
Consider the dynamical system $(C,T)$, where $C := \{ \sum_{n=0}^\infty 2 \varepsilon_n 3^{-{(n+1)}} \ , \ (\varepsilon_n)_n \in \{0,1\}^{\mathbb{N}} \}$, and where $T(x) := 3x \ \text{mod} \ 1$. Modulo a finite number of points, we easily check that this dynamical system is isomorphic to the one-sided shift $(\Sigma^+,\sigma)$, where $\Sigma^+ := \{0,1\}^{\mathbb{N}}$ and $\sigma : (\varepsilon_n)_{n \geq 0} \in \Sigma_+ \longmapsto (\varepsilon_{n+1})_{n \geq 0} \in \Sigma^+$.
The \say{Cantor law} on the triadic Cantor set $C$ identifies to a Bernouilli measure on $\Sigma^+$.
\end{remark}

This kind of isomorphisms is very common in hyperbolic dynamical systems. Let us mention, without proof (precise statements will be recalled in due time), the following idea: 
\begin{itemize}
\item Expanding dynamical systems are often isomorphic to one-sided subshifts of finite type.
\item Hyperbolic diffeomorphisms are often isomorphic to (two-sided) subshifts of finite type.
\item Hyperbolic flows are often isomorphic to suspensions flows over (two-sided) subshifts of finite type.
\end{itemize}

Now that we motivated the study of those objects, let us define correctly what they are. Suspension flows will be discussed in the next subsection: we begin by discussing one and two sided shifts. Good references for this part are the notes of Bowen \cite{Bo75}, Bowen-Ruelle \cite{BR75} and the book of Parry and Pollicott \cite{PP90}. A good introduction to the Thermodynamical formalism can also be found in $\cite{Zi96}$.

\begin{definition}
Let $d \geq 2$ and let $\mathcal{A} = \{ 1, 2 , \dots, d\}$ be a \underline{finite} set of cardinal $d$, call it an alphabet.
Let $M$ be a $d \times d$-matrix with zeroes and ones. Define:
$$\Sigma_M^+ := \{ \mathbf{a} := (a_n)_{n \geq 0} \in \mathcal{A}^\mathbb{N} \ | \ \forall n \geq 0, \ M_{a_n a_{n+1}} = 1 \}.$$
and $$ \Sigma_M := \{ \mathbf{a} := (a_n)_{n \in \mathbb{Z}} \in \mathcal{A}^\mathbb{Z} \ | \ \forall n \in \mathbb{Z}, \ M_{a_n a_{n+1}} = 1 \}. $$
If $M$ doesn't have any vanishing coefficient, then $\Sigma_M$ and $\Sigma_M^+$ are called full shifts (on the finite alphabet $\mathcal{A}$). In general, $\Sigma_M$ and $\Sigma_M^+$ are called subshifts of finite type. \\
These sets are naturally seen as compact metric spaces, as a closed subset of a countable product of finite sets. One can choose the concrete distance $d(\mathbf{a},\mathbf{b}) := \exp({- \min\{ |n|, \ a_n \neq b_n \}})$.
On these spaces, define the one-sided (resp. two-sided) shift $ \sigma : (a_n)_{n \geq 0} \in \Sigma_M^+ \longmapsto (a_{n+1})_{n \geq 0} \in \Sigma_M^+$ (resp. $ \sigma : (a_n)_{n \in \mathbb{Z}} \in \Sigma_M \longmapsto (a_{n+1})_{n \in \mathbb{Z}} \in \Sigma_M$). This defines a Lipschitz dynamical system. Notice that $\sigma$ is invertible on $\Sigma_M$, but isn't on $\Sigma_M^+$. Finally, notice that there is a natural Lipschitz projection $\pi : \Sigma_M \rightarrow \Sigma_M^+$ that commutes with the dynamics.
\end{definition}

Let us now recall basic facts about an important family of invariant measures for theses dynamical systems: equilibrium states. In the next Theorem, we will denote by $C^{\alpha}(\Sigma_M,X)$ the set of $\alpha$-Hölder regular functions from $\Sigma_M$ to some metric space $X$. Let us also recall that $\sigma$ is said to be topologically mixing if, for any open sets $U,V \subset \Sigma_M$, there exists $N$ such that: $\forall n \geq N, \sigma^n(U) \cap V \neq \emptyset$.

\begin{theorem}[Variational principle]
Suppose that the two-sided shift $(\Sigma_M,\sigma)$ is topologically mixing.
Denote by $\mathcal{P}_\sigma(\Sigma_M)$ the set of all $\sigma$-invariants probability measures $\mu$ on $\Sigma_M$. For such a measure, denote by $h_\mu(\sigma)$ the measure-theoretic entropy of $\sigma$. Finally, choose a Hölder regular function $\phi \in C^{\alpha}(\Sigma_M, \mathbb{R} )$, we call it a potential. We define the $\emph{pressure}$ of $\phi$ as:
$$ P(\phi) := \sup_{\mu \in \mathcal{P}_\sigma(\Sigma_M)} \left( h_\mu(\sigma) + \int_{\Sigma_M} \phi \ d\mu \right).$$
Then, there exists a unique measure $\mu_\phi \in \mathcal{P}_\sigma(\Sigma_M)$ such that
$$ P(\phi) = h_{\mu_\phi}(\sigma) + \int_{\Sigma_M} \phi 
 \ d\mu_{\phi} . $$
This measure is called the equilibrium state associated to $\phi$.
\end{theorem}

\begin{remark}
Notice that, if $\phi$ is a constant function, then
$P(\phi) = P(0) + \phi$. The quantity $P(0) =: h(\sigma)$ is called the topological entropy of the shift $\sigma$. The associated equilibrium state is then a measure $\mu_{max}$ such that
$$ h(\sigma) = \sup_{\mu \in \mathcal{P}_\sigma(\Sigma_M)} h_\mu(\sigma) = h_{\mu_{max}}(\sigma). $$
This is the \emph{measure of maximal entropy}.
\end{remark}

\begin{remark}
Let $\phi_1,\phi_2 \in C^{\alpha}(\Sigma_M,\mathbb{R})$ be such that there exists a map $\theta \in C^\alpha(\Sigma_M,\mathbb{R})$ satisfying $\phi_1 = \phi_2 + \theta \circ \sigma - \theta$. We say that $\phi_1$ and $\phi_2$ are \emph{cohomologous}. We then see easily that $P(\phi_1) = P(\phi_2)$, and uniqueness of equilibrium states then ensure that $\mu_{\phi_1} = \mu_{\phi_2}$. 
\end{remark}

\begin{remark}
Notice that any function $\psi \in C^{\alpha}(\Sigma_M^+,\mathbb{R})$ can be seen as a function $\Sigma_M \rightarrow \mathbb{R}$ satisfying $\psi(\mathbf{a}) = \psi(\mathbf{b})$ when $(a_n)_{n \geq 0} = (b_n)_{n \geq 0}$ (identifying $\psi$ with $\psi \circ \pi$).
Then, being given a potential $\phi \in C^{\alpha}(\Sigma_M,\mathbb{R})$, one can always find a potential $\psi \in C^{\beta}(\Sigma_M^+,\mathbb{R})$ that is cohomologous to $\phi$ (for some $0<\beta<\alpha$). It follows that any equilibrium state can be seen as an equilibrium state associated to a potential in $C^{\alpha}(\Sigma_M^+,\mathbb{R})$. Furthermore, since adding to the potential a constant map doesn't modify the equilibrium state, one can also ask for this potential $\psi$ to have vanishing pressure: $P(\psi)=0$.
\end{remark}

Now, to understand the structure of those measures, we will introduce the notion of transfer operators. These are objects that will allows us to understand the autosimilarity properties of equilibrium states.

\begin{definition}[Transfer operators]

Let $\phi \in C^{\alpha}(\Sigma_M^+,\mathbb{C})$. The associated transfer operator $\mathcal{L}_\phi : C^{\alpha}(\Sigma_M^+,\mathbb{C}) \longrightarrow C^{\alpha}(\Sigma_M^+,\mathbb{C})$ is defined by the following formula:
$$ \forall x \in \Sigma_M^+, \ \mathcal{L}_\phi h(x) := \sum_{y \in \sigma^{-1}(x)} e^{\phi(y)} h(y) .$$
For any $a \in \mathcal{A}$ and $x \in \Sigma_M^+$, denote $a \rightarrow x$ if $M_{a x_0} = 1$. Then, we get the following equivalent formula:
$$ \forall x \in \Sigma_M^+, \ \mathcal{L}_\phi h(x) := \sum_{a \rightarrow x} e^{\phi(ax)} h(ax) ,$$
where $ax$ is the only element of $\Sigma_M^+$ satisfying $(ax)_0 = a$, and $\sigma(ax) = x$.
\end{definition}

\begin{remark}
Iterating $\mathcal{L}_\phi$ yields the following formula:
$$ \forall x \in \Sigma_M^+, \ \mathcal{L}_\phi ^n h(x) = \sum_{y \in \sigma^{-n}(x)} e^{S_n \phi(y)} h(y) ,$$
where $$S_n \phi := \sum_{k=0}^{n-1} \phi \circ \sigma^k$$ is a Birkhoff sum. If we denote by $\Sigma_M^+(n)$ the set of words $\mathbf{a}\in \mathcal{A}^n$ such that $M_{a_k a_{k+1}} = 1$ for all $k$, and if we denote $\mathbf{a} \rightarrow x$ when $M_{a_n x_0} = 1$ for $x \in \Sigma_M^+$, then this formula can be rewritten as:
$$ \forall x \in \Sigma_M^+, \ \mathcal{L}_\phi^n h(x) = \underset{\mathbf{a} \rightarrow x}{\sum_{\mathbf{a} \in \Sigma_M^+(n)}} e^{S_n \phi(\mathbf{a} x)} h(\mathbf{a} x) ,$$
where $\mathbf{a} x \in \Sigma_M^+$ is the concatenation of $\mathbf{a}$ and $x$. 
\end{remark}

\begin{theorem}[Perron-Frobenius-Ruelle]
Suppose that $(\Sigma_M^+,\sigma)$ is topologically mixing. Let $\phi \in C^{\alpha}(\Sigma_M^+,\mathbb{R})$ be such that $P(\phi)=0$. Then, there exists a unique map $h \in C^{\alpha}(\Sigma_M^+,\mathbb{R}^*_+)$, and a unique measure $\nu \in \mathcal{P}(\Sigma_M^+)$, with $\int h d\nu = 1$, such that
$$\mathcal{L}_\phi h = h  \text{ and }  \quad \mathcal{L}_\phi^* \nu = \nu.$$
In this case, there exists $C>0$ and $\rho \in (0,1)$ such that, for all $f \in C^{\alpha}(\Sigma_M^+,\mathbb{C})$, $$ \| \mathcal{L}_\phi^n f - h \cdot \int f \  h d\nu\|_{\infty} \leq C \|f\|_{C^\alpha} \rho^n .$$
Moreover, the measure $h \cdot d\nu$ is the pushforward by $\pi : \Sigma_M \rightarrow \Sigma_M^+$ of the equilibrium state given by $\phi$. Formally: $h d\nu = \pi_* \mu_{\phi \circ \pi}$.
\end{theorem}

\begin{remark}
The expression $\mathcal{L}_\phi^* \nu = \nu$ means that, for all continuous map $f:\Sigma_M^+ \rightarrow \mathbb{R}$, one has $$ \int \mathcal{L}_\phi f d\nu = \int f d\nu. $$
\end{remark}

\begin{remark}
Replacing $\phi$ by $\psi := \phi - \ln h \circ \sigma + \ln h$ yields the same equilibrium state. The resulting potential $\psi$ gives a transfer operator with eigenfunction $1$. In fact, being given an equilibrium state $\mu$, one can always find a potential $\psi \in C^{\alpha}(\Sigma_M^+,\mathbb{R}_-)$ such that $P(\psi)=0$,  such that $\mathcal{L}_\psi 1 = 1$, and such that $\mathcal{L}_\psi^* \mu = \mu$. Such a potential is called $\emph{normalized}$. If the dynamics $(\Sigma_M,\sigma)$ is topologically mixing, then any normalized potential is \say{eventually negative}, in the sense that there exists $N \geq 1$ such that $S_N \varphi := \sum_{k=0}^{N-1} \varphi \circ \sigma^k < 0$.
\end{remark}

\begin{remark}
Let $\mu$ be an equilibrium state given by some normalized potential $\phi \in C^{\alpha}(\Sigma_M^+,\mathbb{R}_-)$. Let $f_1,f_2 : \Sigma_M^+ \rightarrow \mathbb{C}$ be Hölder maps. Then:
$$ \int f_1 \circ \sigma^n \cdot f_2 d\mu = \int f_1 \cdot \mathcal{L}_\phi^n(f_2) d\mu. $$
In particular, Perron-Frobenius-Ruelle Theorem ensure that $(\Sigma_M^+,\sigma,\mu)$ is \emph{exponentially mixing}, in the sense that:
$$ \left|  \int f_1 \circ \sigma^n \cdot f_2 d\mu - \int f_1 d\mu \int f_2 d\mu \right| \leq C \|f_1\|_{C^\alpha} \|f_2 \|_{C^{\alpha}} \rho^n .$$
In fact, $(\Sigma_M,\sigma,\mu)$ is also exponentially mixing. (In particular, it is ergodic, that is, any measurable function $f:\Sigma_M^+ \rightarrow \mathbb{C}$ that is $\sigma$-invariant is essentially constant.)
\end{remark}

Ergodicity/mixing of the shift with respect to equilibrium states allows one to expect statistical results that are typical in chaotic dynamical system, such as $$ \frac{1}{n} \sum_{k=0}^{n-1} f \circ \sigma^k \simeq \int f d\mu .$$
More precisely, one has the following \emph{large deviation} result \cite{Yo90}, \cite{Ki90}:

\begin{theorem}
Let $(\Sigma_M,\sigma,\mu)$ be a topologically mixing subshift of finite type, where $\mu$ is an equilibrium state associated to some Hölder potential $\phi$. Let $f \in C^{\alpha}(\Sigma_M,\mathbb{C})$. Then, for any $\varepsilon>0$, there exists $C,\delta>0$ such that:
$$ \forall n \geq 1, \  \mu\left( x \in \Sigma_M, \ \left| \frac{1}{n} \sum_{k=0}^{n-1} f \circ \sigma^k(x) - \int f d\mu  \right| \geq \varepsilon \right) \leq C e^{- n \delta}. $$
\end{theorem}

For a proof in the context of Julia sets, see Section 3.7. We conclude this section by recalling useful statistical estimates on equilibrium states, called \emph{Gibbs estimates}, and a corollary.

\begin{theorem}[Gibbs estimates]
Let $\mu$ be an equilibrium state associated to some Hölder potential $\phi$. For any word $\mathbf{a} = a_0 \dots a_n$, define the cylinder $C(\mathbf{a}) := \{ x=(x_k)_{k \in \mathbb{Z}}| x_0=a_0 , \dots x_n = a_n \} \subset \Sigma_M $. There exists $C \geq 1$ such that:
$$C^{-1} \exp(S_n \phi(x_{\mathbf{a}}) - n P(\phi) ) \leq \mu(C_{\mathbf{a}}) \leq C \exp(S_n \phi(x_{\mathbf{a}}) - n P(\phi) ) ,$$
where $x_\mathbf{a} \in C_{\mathbf{a}}$ is any point in our cylinder, and where $S_n \phi := \sum_{k=0}^{n-1} \phi \circ \sigma^k$ is a \emph{Birkhoff sum}.
\end{theorem}

\begin{remark}[Regularity estimates in shift space]
Let $\mu$ be an equilibrium state associated to some normalized potential $\phi$. Since $\phi$ is eventually negative on a compact set, there exists $0 < \delta_{reg} < D_{reg} $ and $N \geq 1$ such that $-D_{reg} < S_N \phi < -\delta_{reg}$. It follows from Gibbs estimates that there exists $C \geq 1$ such that, for all $\mathbf{a}=a_0 \dots a_n$:
$$ C^{-1} e^{- D_{reg} n/N} \leq \mu(C(\mathbf{a})) \leq C e^{- \delta_{reg} n/N}. $$
\end{remark}

\subsection{Suspension flows and Dolgopyat's estimates}

We discussed a model for expanding maps and hyperbolic diffeomorphisms: now we discuss a useful model for flows. The idea is the following: consider a flow in a 3 dimensional and compact manifold (say, a solid torus, for example, turning around in circle). By considering a well chosen surface (say, a disk transverse to this flow), one can try to understand the dynamics by understanding the \emph{first return time} and the \emph{first return map} associated to this surface. This suggest the abstract notion of suspension flow.

\begin{definition}[Suspension flows over (two-sided) shifts]
Let $(\Sigma_M,\sigma)$ be a (two-sided) subshift of finite type, as before. Let $\tau : \Sigma_M \rightarrow \mathbb{R}^+_*$ be a Hölder continuous \emph{roof function}. First, define the quotient space
$$ \Sigma_M^\tau := \Big{\{} (x,t) \in \Sigma_M \times \mathbb{R}_+ | \ t \in [0,\tau(x)] \Big{\}} \Big{/} \Big{\{} (x,\tau(x)) \sim (\sigma(x),0) \Big{\}}. $$
On this set, we define the associated \emph{suspension flow} by letting $x$ invariant and moving in the vertical \say{time} direction, jumping to $\sigma(x)$ when necessary. More precisely, for $(x,t) \in \Sigma_M \times \mathbb{R}_+$, define $n(x,t)$ as the only integer such that $$ \{t\}_{x} := t - S_{n(x,t)} \tau(x) \in [0,\tau(\sigma^{n(x,t)}(x))[ .$$
The suspension flow (for positive times) can then be expressed as:
$$ \forall t \geq 0, \ \sigma_\tau^t((x,s)) = (\sigma^{n(x,t + s)}(x), \{t + s\}_{x}). $$
\end{definition}

We have a Thermodynamical formalism and a variational principle in this context.

\begin{theorem}[Variational principle]
Let $(\Sigma_M,\sigma)$ be a topologically mixing subshift of finite type. Let $\tau: \Sigma_M \rightarrow \mathbb{R}_+^*$ be a Hölder roof function. Denote by $(\Sigma_M^\tau,\sigma_\tau)$ the associated suspension flow. Finally, let $\Phi : \Sigma_M^\tau \rightarrow \mathbb{R}$ be a Hölder regular potential. We define its \emph{pressure} as
$$ P(\Phi) := \sup_{m \in \mathcal{P}_{\sigma_\tau}(\Sigma_M^\tau)} \Big{(} h_{m}(\sigma_\tau^1) + \int_{\Sigma_M^\tau} \Phi dm \Big{)} ,$$
where $\mathcal{P}_{\sigma_\tau}(\Sigma_M^\tau)$ is the set of probability measures on $\Sigma_M^\tau$ that are invariant by the suspension flow $(\sigma_\tau^t)_{t \in \mathbb{R}}$. Then, there exists a unique measure $m_\Phi \in \mathcal{P}_{\sigma_\tau}(\Sigma_M^\tau)$, called the equilibrium state associated to $\Phi$, such that
$$ P(\Phi) = h_{m_\Phi}(\sigma_\tau^1) + \int_{\Sigma_M^\tau} \Phi dm_{\Phi} .$$
Furthermore, if we define $\phi(x) := \int_0^{\tau(x)} \Phi(x,t) dt $, then $m_\Phi$ is linked to the equilibrium state on the shift space $\mu_\phi \in \mathcal{P}_{\sigma}(\Sigma_M)$ in the following way:

$$ \forall f \in C^0(\Sigma_M^\tau,\mathbb{C}), \ \int_{\Sigma_M^\tau} f dm_{\Phi} := \frac{1}{ \int_{\Sigma_M} \tau d\mu_\phi } \int_{\Sigma_M}  \int_0^{\tau(x)} f(x,t) dt \ d\mu_\phi(x)  .$$
In other words, \say{locally}, $m_{\Phi} = c_0 \cdot \mu_\phi \otimes dt$ for some constant $c_0$.
\end{theorem}

While it is easy to see that the suspension flow $(\Sigma_M,\sigma_\tau,m_{\Phi})$ is ergodic in this context, mixing is not always true. For example, if $\tau$ is a constant roof function, the suspension flow isn't mixing. Let us dive a bit deeper into the kind of conditions that allows one to prove mixing properties in this context. 

\begin{definition}
Let $(\Sigma_M^\tau,\sigma_\tau,m_{\Phi})$ be a suspension flow over a topologically mixing subshift of finite type, equipped with some equilibrium state. Let $A,B : \Sigma_M^\tau \rightarrow \mathbb{C}$ be Hölder regular \say{observables}. The associated correlation function is:
$$ \rho(t) := \int_{\Sigma_M^\tau} A \circ \sigma_\tau^t \cdot B dm_{\Phi} - \int_{\Sigma_M^\tau} A dm_{\Phi} \int_{\Sigma_M^\tau} B dm_{\Phi}.$$
\end{definition}

To study the decay properties of $\rho$, a common idea is to study the regularity properties of its Fourier transform instead. Since $\rho \in L^\infty(\mathbb{R})$, its Fourier transfrom $\widehat{\rho}$ always exists, at least in the sense of distributions. A Paley-Weiener type of results then allows us to prove decay of correlations if we understand the smoothness properties of $\widehat{\rho}$. For example, to get exponential mixing, we expect $\widehat{\rho}$ to extend analytically on a strip $\{ \xi \in \mathbb{C} \ , \ |\text{Im}(\xi)| < \varepsilon \}$. \\

Let us follow Pollicot \cite{Po85} and compute $\widehat{\rho}$, at least formally. Suppose, for simplicity, that $\int A dm_{\Phi} = \int B dm_\Phi =0$ and suppose also, for further simplification, that $(A(x,t),B(x,t),\tau(x)) = (A(y,t),B(y,t),\tau(y))$ as soon as $(x_i)_{i \geq 0} = (y_i)_{i \geq 0}$. Say that we are only interested in what is happening for positive times, and so let $\rho=0$ for negative times. We then have:

$$ \widehat{\rho}(\xi) =  \int_{0}^\infty e^{- 2 i \pi t \xi} \int_{\Sigma_{M}^\tau} A \circ \sigma_\tau^t \cdot B  dm_\Phi \cdot dt $$
$$ = c_0 \int_{\Sigma_M} \int_{0}^{\tau(x)} B(x,s) \int_0^\infty A(\sigma^{n(x,s+t)}x, \{t + s\}_{x}) e^{- 2 i \pi \xi t} dt \ ds \ d\mu_{\phi}(x) .$$
We compute the integral involving $A$:
$$ \int_{0}^\infty A(\sigma^{n(x,s+t)}x, \{t + s\}_{x}) e^{- 2 i \pi \xi t} dt  =   \sum_{n=0}^\infty \int_{\{ t \in \mathbb{R}_+, n(x,t+s) = n \}}  A(\sigma^{n(x,s+t)}x, \{t + s\}_{x}) e^{- 2 i \pi \xi t} dt $$
$$ = \int_0^{\tau(x)-s} 
A(x,t+s) e^{- 2 i \pi \xi t} dt + \sum_{n=1}^\infty \int_{S_n \tau(x)-s}^{S_{n+1} \tau(x)-s}  A(\sigma^{n}x, t+s-S_n \tau(x)) e^{- 2 i \pi \xi t} dt $$
$$ = - \int_0^s A(x,u) e^{-2 i \pi \xi (u-s)} du  + \sum_{n=0}^{\infty} \int_0^{\tau(\sigma^n(x))} A(\sigma^n(x),u) e^{- 2 i \pi \xi (u - s + S_n \tau(x))} du $$
$$ = \text{(analytic)} + \sum_{n=0}^\infty \widehat{A}(\sigma^n(x),\xi) e^{ 2i \pi \xi (s - S_n \tau(x) )} ,$$
where we defined $\widehat{A}(x,\xi) := \int_0^{\tau(x)} A(x,u) e^{- 2 i \pi \xi u} du$. Define similarly $\widehat{B} : \Sigma_M \times \mathbb{R} \rightarrow \mathbb{C}$. Notice that $\widehat{A},\widehat{B}$ are analytic in $\xi$.  Injecting this in the expression of $\widehat{\rho}$ yields:

$$ \widehat{\rho}(\xi) = \text{(analytic)} + c_0 \int_{\Sigma_M} \int_{0}^{\tau(x)} B(x,s) \sum_{n=0}^\infty \widehat{A}(\sigma^n(x)) e^{ 2i \pi \xi (s - S_n \tau(x) )} \ ds \ d\mu_{\phi}(x)$$
$$ = \text{(analytic)} + c_0 \sum_{n=0}^\infty  \int_{\Sigma_M} \widehat{B}(x,-\xi) \widehat{A}(\sigma^n(x),\xi) e^{ - 2i \pi \xi S_n \tau(x) } \   d\mu_{\phi}(x).$$
Now, since $\widehat{A}(\cdot,\xi)$ and $\widehat{B}(\cdot,\xi) e^{-2 i \pi \xi S_n \tau}$ can be seen as functions $\Sigma_M^+ \rightarrow \mathbb{C}$, we can make a transfer operator appear. Let $\tilde{\phi}$ be a normalized potential cohomologous to $\phi - P(\phi)$. Then:
$$ \widehat{\rho}(\xi) = \text{(analytic)} + c_0 \sum_{n=0}^\infty  \int_{\Sigma_M}  \widehat{A}(x,\xi) \mathcal{L}_{\tilde{\phi}}^n \Big{(}\widehat{B}( \cdot ,-\xi) e^{ - 2i \pi \xi S_n \tau } \Big{)}(x) \   d\mu_{\phi}(x) $$
$$ = \text{(analytic)} + c_0 \sum_{n=0}^\infty  \int_{\Sigma_M}  \widehat{A}(x,\xi) \mathcal{L}_{\tilde{\phi} - 2 i \pi \xi \tau}^n \Big{(}\widehat{B}( \cdot ,-\xi)  \Big{)}(x) \   d\mu_{\phi}(x) .$$
This will be our final expression. We see that everything boils down to understanding a family of \say{twisted transfer operator} $\mathcal{L}_{\tilde{\phi}+ i \xi \tau} : C^{\alpha}(\Sigma_M,\mathbb{C}) \rightarrow C^{\alpha}(\Sigma_M,\mathbb{C})$, defined as $$\mathcal{L}_{\tilde{\phi} + i \xi \tau} h(x) := \sum_{y \in \sigma^{-1}(x)} e^{\tilde{\phi}(y) + i \xi \tau(y) } h(y). $$
Every term in the sum appearing in the expression of $\widehat{\rho}$ yields an analytic function: to ensure that $\widehat{\rho}$ is analytic itself, it follows that one, at least, has to check convergence of this series. As we saw before, this can not always be true, for example if $\tau$ is a constant map. Notice that, if this series converges, we can then write (at least formally again)
$$ \widehat{\rho}(\xi) = \text{(analytic)} + c_0 \int_{\Sigma_M} \widehat{A}(\cdot,\xi) \cdot  \Big{(}I - \mathcal{L}_{\tilde{\phi}- 2 i \pi \xi \tau}\Big{)}^{-1}\Big{(}\widehat{B}(\cdot, - \xi)\Big{)} d\mu_\phi .$$

For small $\xi \in \mathbb{R}$, analyticity of this expression comes from the fact that, as seen in the previous subsection, $\mathcal{L}_{\tilde{\phi}}$ contracts exponentially quickly functions with zero mean, such as $\widehat{B}(\cdot,\xi)$. This property should be true for $\mathcal{L}_{\tilde{\phi}+i \xi \tau}$ by a perturbation argument. For large $\xi \in \mathbb{R}$, a more sophiticated argument is needed. \\

Under some \say{nonconstantness} condition (called, in our symbolic setting, \say{strong non-integrability}) for $\tau$, Dolgopyat proved the following estimates (first in a geometrical setting in \cite{Do98}, and then in our symbolic setting in \cite{Do00}).

\begin{theorem}{\cite{Do00}}
Under a \say{strong non-integrability condition} on $\tau : \Sigma_M^+ \rightarrow \mathbb{R}_+^*$, for any $\alpha >0$, there exists $C,p \geq 1$ and $\rho \in (0,1)$ such that, for any $\xi \in \mathbb{C}$ with $|\text{Re}(\xi)| \geq C $, $|\text{Im}(\xi)| \leq C^{-1}$, we have:
$$ \forall h \in C^{\alpha}(\Sigma_M,\mathbb{C}), \ \| \mathcal{L}_{\tilde{\phi}+ i \xi \tau}^n h \|_\infty \leq C |\text{Im}(\xi)|^p \rho^n \|h\|_{C^\alpha(\Sigma_M)} .$$
\end{theorem}

These estimates are enough to obtain exponential mixing for the suspension flow $(\Sigma_M^\tau,\sigma_\tau,m_\Phi)$. Such results were quickly generalized, for different kind of dynamical systems, and under different names for the \say{nonconstantness} condition made on $\tau$: \say{Strong non-integrability} \cite{Do00}, \say{Uniform Non-Integrability}
 (UNI) (\cite{Ch98}, \cite{Do98}, \cite{AGY06}, \cite{BV05}, \cite{TZ20}, \cite{DV21}), \say{Non-Local integrability} (NLI) \cite{Na05} (to only quote a few)... These estimates are useful to prove exponential mixing in the context of hyperbolic flows, but can also be used to prove counting results about periodic orbits, for example. \\

 In some contexts, classical Dolgopyat's estimates are not enough, and one can be interested in adding supplementary \say{twist parameters} in the phase (this is a natural thing to do when studying group extensions of hyperbolic dynamical systems, as in \cite{Do02}). As an exemple, let us cite this Dolgopyat-type estimate, taken from \cite{OW17}, used to prove equidistribution result for holonomies of closed orbits under an expanding, conformal dynamical system on hyperbolic Julia sets in the complex plane:

 \begin{theorem}[\cite{OW17}]
 Let $f : \mathbb{C} \rightarrow \mathbb{C}$ be a rational map of degree $d \geq 2$, and suppose that $f$ is not conjugated to a monomial under a Moebius transformation. Suppose that its Julia set $J$ is included in $\mathbb{C}$, and that $f$ is expanding on $J$. Let $U \supset J$ be a small enough open neighborhood of $J$. Let $\delta > 0$ be such that  $\varphi := - \delta \ln |f'| \in C^{1}(U,\mathbb{R})$ is a normalized potential. For $\xi \in \mathbb{R}$ and $l \in \mathbb{Z}$, consider the following twisted transfer operator $\mathcal{L}_{\varphi,s,l} : C^{1}(U,\mathbb{C}) \rightarrow C^{1}(U,\mathbb{C})$, defined as
 $$ \mathcal{L}_{\varphi,\xi,l} h(x) := \sum_{y \in f^{-1}(x)} e^{\varphi(y)} |f'(y)|^{i \xi} \Big{(}\frac{f'(y)}{|f'(y)|} \Big{)}^{l} h(y). $$
 Then, there exists $C,p \geq 1$ and $\rho \in (0,1)$ such that, for any $(\xi,l) \in \mathbb{R} \times \mathbb{Z}$ with $|\xi| + |l| \geq C $, we have:
$$ \forall h \in C^{1}(U,\mathbb{C}), \ \| \mathcal{L}_{\varphi,\xi,l}^n h \|_\infty \leq C (|\xi|+|l|)^p \rho^n \|h\|_{C^1(U,\mathbb{C})}.$$
 \end{theorem}

For a sketch of proof, see Section 3.8. These kind of estimates will be very useful for us in the main part of our thesis: indeed, it appears that, using the \say{sum-product phenomenon}, one is sometimes able to reduce the problem of proving Fourier decay for equilibrium states into a problem in the spirit of exponential mixing for some suspension flow. This will be made precise later in the text: for now, let us discuss what this \say{sum-product phenomenon} is about.



\section{Additive combinatorics: the sum-product phenomenon}


\subsection{Additive combinatorics in $\mathbf{F}_p$}

When dealing with oscillatory integrals involving smooth measures, we have at our disposal all of the usual harmonic analysis tools to help us derive asymptotics. For example, let us recall the well known Van Der Corput Lemma.

\begin{lemma}[Van Der Corput]
Let $\psi : [0,1] \rightarrow \mathbb{R}$ be $C^{k+1}$, and suppose that its $k$-th derivative $\psi^{(k)}$ doesn't vanish. Then, there exists $C \geq 1$ such that
$$ \forall \xi \in \mathbb{R}^*, \ \left| \int_0^1 e^{i \xi \psi(x)} dx \right| \leq C |\xi|^{-1/k}. $$
\end{lemma}

There is an important intuitive remark to be made here. The condition on the derivative on $\psi$ can be thought of as a \say{non-concentration hypothesis} on the phase. Indeed, if $\psi$ is constant on some small open set, then the lemma doesn't applies. Still, the phase might concentrate a bit, like a monomial $x^k$ at zero. The more the phase concentrate, the greater one has to choose $k$, and so the less decay rate we get ($1/k$, here). More generally, if we replace $dx$ by any smooth measure $d\mu(x) = \rho(x)dx$, the same decay rate holds. Of course, in our context, $\mu$ is going to be an equilibrium state supported in a fractal set, and so this kind of results (that relies on integration by parts) are not available. \\

The goal of this section is to introduce what will be our replacement tool: the sum-product phenomenon. Let us give a rough sketch of the idea behind it. A good introduction can be found in Green's notes \cite{Gr09}. Say that we want to control the modulus of a sum of exponentials of the form $$ \frac{1}{n} \sum_{k=0}^n e^{i \xi X_k},$$
where $\xi \in \mathbb{R}$ and $n \geq 0$ are fixed, and where $X_k \in \mathbb{R}$ are some phases. If all the phases take the same values in $\mathbb{R}/2\pi \mathbb{Z}$, then this sum has modulus one, but in a generic setting one has no chance for this to happen. In fact, if we choose our phases \emph{randomly}, we can expect this sum to be less than one (the gap from one also depending on our choice of $n$ and $\xi$.) So, to produce theorems that gives us control over such sums, a good way to start would be to search for places where deterministic \say{pseudorandomness} occurs. This is where the sum-product phenomenon appears. \\

Let us discuss quickly about additive combinatorics: the field that studies the \say{additive structure of sets}. Let $R$ be a finite ring, and let $A \subset R$. One can measure the \say{additive structure} of $A$ by considering the set $A+A := \{ a+b | a,b \in A\}$, and by computing its cardinal. If $A \subset R$ is an additive subgroup of $R$, then $A+A=A$ has cardinal $|A|$. But if we choose $A$ randomly, we expect $|A+A|$ to have cardinal $\simeq |A|^2/2$. Similarly, one can quantify the \say{multiplicative structure} of $A$ by considering the set $A \cdot A := \{ ab | \ a,b \in A \}$. If $A \subset R$ is a multiplicative subgroup of $R^\times$, then $A \cdot A = A$ has cardinal $|A|$. If $A$ is taken randomly, we expect to find $|A \cdot A| \simeq |A|^2/2$. \\

More generally, one can study how the additive/multiplicative structure of two different sets $A,B \subset R$ behave with each others. A way to do so would be, for example for the additive structure, to consider the cardinal $|A-B|$. If $A$ is an additive subgroup of $R$, and if $B = b_0 + A$ is a translation of $A$, then $|A-B| = |A|^{1/2}|B|^{1/2}$. In this case, $A$ and $B$ have additive structure, and moreover, they are compatible. If $A$ and $B$ don't have a compatible additive structure, we expect something like $|A-B| \simeq |A| |B|$. \\

Let us fix ideas by taking $p$ a large prime number, and choosing $R := \mathbf{F}_p$, the finite field of cardinal $p$. The \say{sum-product phenomenon} (in $\mathbf{F}_p$) is the idea that multiplicative structure and additive structure can not coexist simultaneously, except if $A$ have very small or very large cardinal. In other words, the sum-product phenomenon (in $\mathbf{F}_p$) is a quantitave statement about the non-existence of proper subrings of $\mathbf{F}_p$. Since a baby version of this statement can be proved easily, we will follow Green and prove a result of this kind to help the reader get an intuition behind this. The next lemma states that the action $\mathbf{F}_p^\times \curvearrowright \mathbf{F}_p$ destroys any additive structure, and replaces it by \say{additive chaos}.

\begin{lemma}
Let $A \subset \mathbf{F}_p$. There exists $\xi \in \mathbf{F}_p^\times$ such that
$$ |A - \xi A| \geq \frac{1}{2} \min(|A|^2,p). $$
\end{lemma}

\begin{proof}

Define the additive energy of two nonempty sets $A,B \subset \mathbf{F}_p$ by the formula:
$$\omega_+(A,B) := |A|^{-3/2} |B|^{-3/2} \# \{ (a_1,a_2,b_1,b_2) \in A^2 \times B^2, \ a_1 - b_1 = a_2 - b_2 \}.$$
If the additive energy is small, then $A,B$ don't have a compatible additive structure. Indeed, denoting $\delta_{a,b}$ a Kronecker symbol, and using Cauchy-Schwarz inequality:
$$ \omega_+(A,B) = \frac{1}{|A|^{3/2} |B|^{3/2}} \sum_{a_1,a_2,b_1,b_2} \delta_{a_1-b_1, a_2-b_2}$$ $$ = \frac{1}{|A|^{3/2} |B|^{3/2}}  \sum_{x \in A-B} \sum_{a_1,a_2,b_1,b_2} \mathbb{1}_{\{x\}}(a_1-b_1) \mathbb{1}_{\{x\}}(a_2-b_2) $$
$$ = \frac{1}{|A|^{3/2} |B|^{3/2}}  \sum_{x \in A-B} \Big{(} \sum_{a,b} \mathbb{1}_{\{x\}}(a-b) \Big{)}^2 $$
$$ \geq \frac{1}{|A|^{3/2} |B|^{3/2}} \frac{1}{|A-B|} \Big( \sum_{x \in A-B} \sum_{a,b} \mathbb{1}_{\{x\}}(a-b) \Big)^2 = \frac{|A|^{1/2} |B|^{1/2}}{|A-B|}. $$
Now, let us compute the expected value of the additive energy between $A$ and $\xi A$, for $\xi \in \mathbf{F}_p^\times$ chosen uniformly randomly. We find:
$$ \frac{1}{p-1} \sum_{\xi \in \mathbf{F}_p^\times} \omega_+(A,\xi A) = \frac{1}{|A|^3 (p-1)} \#\{ (a_1,a_2,a_3,a_4,\xi) \in A^4 \times \mathbf{F}_p^\times, \ a_1 + a_2 \xi = a_3 + a_4 \xi  \} $$
$$ = \frac{1}{(p-1) |A|^3} \sum_{a_1,a_2,a_3,a_4} \#\{ \xi \in \mathbf{F}_p^\times, \ (a_1-a_3) = \xi (a_4-a_2) \}  $$
$$ = \frac{1}{(p-1) |A|^3} \Big{(} \underset{a_1-a_3=0}{\underset{a_2-a_4=0}{\sum}} (p-1)   +  \underset{a_1-a_3 \neq 0}{\underset{a_2-a_4=0}{\sum}} 0  + \underset{a_1-a_3 = 0}{\underset{a_2-a_4\neq 0}{\sum}} 0  + \underset{a_1-a_3 \neq 0}{\underset{a_2-a_4 \neq 0}{\sum}} 1 \Big{)} $$
$$ = \frac{1}{(p-1)|A|^3} \Big{(} |A|^2 (p-1) + |A|^2 (|A|-1)^2 \Big{)} = \frac{1}{|A|} + \frac{(|A|-1)^2}{(p-1) |A|} \leq 2 \max(1/|A|, |A|/p) .$$
In particular, there exists at least one $\xi \in \mathbf{F}_p^\times$ such that $\omega_+(A,\xi A) \leq 2 \max(1/|A|, |A|/p)$, and so, for this $\xi$:
$$ |A-\xi A| \geq |A|/2 \min( |A|, p/|A| ) = \min(|A|^2/2, p/2). $$
\end{proof}

Additive combinatorics can be used to prove subtler versions of this sum-product phenomenon. Let us cite a celebrated example of such results: a proof can be found in Green's notes \cite{Gr09}. This Theorem state that additive and multiplicative structure can not coexist (at intermediate scales between $0$ and $p$).

\begin{theorem}[Bourgain-Katz-Tao]
Let $\delta>0$. There exists $\delta'>0$ and $c>0$ such that the following holds.
For any prime $p$, for any $A \subset \mathbf{F}_p$ such that $p^\delta \leq |A| \leq p^{1-\delta}$, we have
$$ \max \left(|A+A|,|A \cdot A| \right) \geq c |A|^{1+\delta'} .$$
\end{theorem}


A nice corollary of the methods from additive combinatorics is the following result, taken again from \cite{Gr09}. This is a control for sums of exponentials where the phase has multiplicative structure. Intuitively, this have a chance to work because the exponential is a morphism that sees additive structure in its argument. So multiplicative structure in the phase is perceived as pseudorandomness, giving cancellations in the sum.

\begin{theorem}[Bourgain-Gilibichuk-Konyagin]
Suppose that $H \subset \mathbb{F}_p^\times$ is a multiplicative subgroup of size at least $p^\delta$, where $\delta>0$ and where $p$ is a large enough prime. Then, uniformly in $\xi \in \mathbb{F}_p^{\times}$, we have
$$\frac{1}{|H|} \left| \sum_{x \in H} e^{2 i \pi x \xi/p } \right| \leq C \ p^{- \delta'} $$
where $C,\delta'>0$ depends only on $\delta$. \end{theorem}

Let us give a more formal idea of the proof. For some set $A \subset \mathbf{F}_p$ and all $\alpha \in (0,1)$, define $$\text{Spec}_\alpha(A) := \Big{\{} \xi \in \mathbf{F}_p^\times , \ \Big{|} \sum_{x \in A} e^{2 i \pi x \xi/p } \Big{|} \geq \alpha |A| \Big{\}} .$$ 
If $A=H$ is a multiplicative subgroup of $\mathbf{F}_p$, it is clear that $\text{Spec}_\alpha(H)$ has some multiplicative structure, since for all $h \in H$, $h \cdot \text{Spec}_\alpha(H) = \text{Spec}_\alpha(H)$. But in some sense, $\text{Spec}_\alpha(A)$ also have a bit of weak additive structure, and that will force it to be a very large or very small set. Since we can expect decay \say{on average} in $\xi$, this set will be very small for some $\alpha$ (that will have order of magnitude $p^{-\delta'}$). Details can be found in the notes of Green. The additive structure of $\text{Spec}_\alpha(H)$ can be understood as follow. For any $\xi \in \text{Spec}_\alpha(H)$, let $c_\xi$ be a complex number of modulus one such that
$$ |H| \alpha \leq c_\xi \sum_{x \in H} e^{2 i \pi x \xi/p}. $$ By Cauchy-Scwhartz inequality, notice that
$$ \alpha^2 \leq \Big{|} \frac{1}{|\text{Spec}_\alpha(H)|} \sum_{\xi \in \text{Spec}_\alpha(H)}  \frac{c_\xi}{|H|} \sum_{x \in H}   e^{2 i \pi x \xi/p} \Big{|}^2$$
$$ \leq \frac{1}{|H|}  \sum_{x \in H} \frac{1}{|\text{Spec}_\alpha(H)|^2} \sum_{\xi,\eta \in \text{Spec}_\alpha(H)} c_\xi \overline{c_\eta} e^{2 i \pi x (\xi-\eta)/p} $$
$$ \leq \frac{1}{|\text{Spec}_\alpha(H)|^2} \sum_{\xi,\eta \in \text{Spec}_\alpha(H)} \Big{|} \frac{1}{|H|} \sum_{x \in H} e^{2 i \pi x (\xi-\eta)/p} \Big{|} .$$
It follows from this observation that a proportion of at least $\alpha^2/2$ of couples $(\xi,\eta) \in \text{Spec}_\alpha(H)^2$ satifies $\xi-\eta \in \text{Spec}_{\alpha^2/2}(H)$. This key observation can be used to prove Bourgain-Gilibichuk-Konyagin's Theorem.

\subsection{The sum-product phenomenon and fractal measures}


The philosophy of the sum-product phenomenon is that multiplicative structure in the phase of some sum of exponential (or, more generally, of some  oscillatory integral) is enough to ensure some decay properties of the sum (at some scale). From there, a natural idea is to look at \emph{multiplicative convolutions} of measures.

\begin{definition}
Let $\mathcal{A}$ be a normed algebra. Let $E,F \subset \mathcal{A}$ be compact sets, and let $(\mu,\nu) \in \mathcal{P}(E) \times \mathcal{P}(F)$. Define the multiplication operator $\text{mult}:E \times F \rightarrow E \cdot F$, where $E \cdot F := \{ x y | (x,y) \in E \times F\}$. The multiplicative convolution of $\mu$ and $\nu$ is defined as: $\mu \odot \nu := \text{mult}_*(\mu \otimes \nu) \in \mathcal{P}(E \cdot F)$. For any continuous map $f \in C^0(E \cdot F,\mathbb{C})$, we have:
$$\int_{E \cdot F} f d(\mu \odot \nu) = \int_E \int_F f(xy) d\mu(x) d\nu(y).$$
\end{definition}

We are ready to cite what will be our replacement of the Van Der Corput lemma in our fractal context. The following theorem will ensure decay of oscillatory integrals with enough multiplicative structure in the phase, at some scale, and under some additional non-concentration hypothesis for this scale. This first theorem was proved by J. Bourgain in $\sim 2010$.

\begin{theorem}[\cite{BD17},\cite{Bo10}]
For all $\gamma > 0$, there exist $\varepsilon_1, \varepsilon_2 > 0$ and $k \in \mathbb{N}$ such that the
following holds. Let $\mu \in \mathcal{P}([1/2, 1])$ and let $\eta > 0$ be large enough. Assume that: 
$$ \forall x \in \mathbb{R}, \ \forall \sigma \in [\eta^{-1}, \eta^{-\varepsilon_1}], \ \  \mu\left([x - \sigma, x + \sigma]\right) \leq \sigma^{\gamma}.$$
Then:
$$ |\widehat{\mu^{\odot k}}(\eta)| = \left| \int_{[1/2,1]^k} e^{2 i \pi  \eta x_1 \dots x_k} d\mu(x_1) \dots d\mu(x_k) \right| \leq \eta^{- \varepsilon_2} . $$
\end{theorem}

\begin{remark}
Notice that the sum-product phenomenon \textbf{does not say} that there exists $k$ such that $\widehat{\mu^{\odot k}}(\eta) \underset{\eta \rightarrow \infty}{\longrightarrow} 0$. What it does say is: if $\mu$ satisfies some regularity estimates \emph{at some scale} $\simeq 1/\eta_0$, then one can control $\widehat{\mu^{\odot k}}(\eta)$ when $\eta \simeq \eta_0$.
\end{remark}

\begin{example}
Let us do an example. Let $h \in (0,1)$, and consider the probability measure $\mu_h \in \mathcal{P}(1/2,1)$ as the uniform law on $h \mathbb{Z} \cap [1/2,1] =: \mathbf{Z}(h)$. For any interval $I \subset \mathbb{R}$ of lenght $\sigma \geq h$, we have a bound:
$$ \mu(I) = \frac{\#(I \cap \mathbf{Z}(h))}{\# \mathbf{Z}(h)} \leq C_0 h( 1 + \text{diam}(I) h^{-1} ) \leq C_1 \text{diam}(I) .$$
The sum-product phenomenon then gives us two absolute constants, an integer $k \geq 1$ and some $\varepsilon_2$, such that:
$$ \forall h \in (0,1), \  |\widehat{\mu_h^{\odot k }}(h^{-1})| \leq C_2 h^{-\varepsilon_2}. $$
\end{example}

There is also a stronger, more technical version of the sum-product phenomenon where we consider $k$ different measures $\mu_1, \dots, \mu_k$,  with a (seemingly) weaker non-concentration condition. 

\begin{corollary}[\cite{BD17}]
For all $\gamma > 0$, there exist $\varepsilon_2 > 0$ and $k \in \mathbb{N}$ such that the
following holds. Let $C_0 \geq 1$ and $\mu_1, \dots, \mu_k \in \mathcal{P}([1/2, 1])$ and let $\eta > 0$ large enough. Assume that for all $j$: 
$$  \forall \sigma \in [\eta^{-1}, \eta^{-\varepsilon_2}], \ \  (\mu_j \otimes \mu_j)\big( \{ (x,y) \in \mathbb{R}^2 | \ |x-y|< \sigma \} \big) \leq C_0 \sigma^{\gamma}.$$
Then, there exists $C_1$ depending only on $\gamma,C_0$ such that:
$$ \left| \int_{[1/2,1]^k} e^{2 i \pi  \eta x_1 \dots x_k} d\mu_1(x_1) \dots d\mu_k(x_k) \right| \leq C_1 \eta^{- \varepsilon_2} . $$
\end{corollary}
In concrete applications, it will be interesting for us to allow for slowly growing constants in front of our estimates. In this spirit, let us prove the following slightly strenghtened version. The proof is adapted from \cite{SS20}, Lemma 4.3.

\begin{corollary}
Fix $0< \gamma < 1$. There exist $ \varepsilon_1 > 0$ and $k \in \mathbb{N}$ such that the following holds for
$\eta \in \mathbb{R}$ with $|\eta|$ large enough. Let $1< R < |\eta|^{\varepsilon_1}$ and let $\lambda_1, \dots , \lambda_k$ be Borel measures supported on the interval $[R^{-1},R]$ with total mass less than $R$. Assume that each $\lambda_j$ satisfies the following non concentration property:
$$\forall \sigma \in [|\eta|^{-2}, |\eta|^{- \varepsilon_1}], \quad  \lambda_j \otimes \lambda_j \left( \{ (x,y) \in \mathbb{R}^2, \ |x-y| \leq \sigma \} \right) \leq \sigma^\gamma .$$
Then there exists a constant $c>0$ depending only on $\gamma$ such that
$$ \left| \int e^{i \eta z_1 \dots z_k}  d\lambda_1(z_1) \dots d\lambda_k(z_k) \right| \leq c |\eta|^{- \varepsilon_1} $$
\end{corollary}

\begin{proof}
 Fix $0<\gamma<1$, and let $\varepsilon_2$ and $k$ given by the previous corollary. Choose $\varepsilon_1 := \frac{\varepsilon_2}{2(2k+1)}$. Let $1<R<|\eta|^{\varepsilon_1}$, and let $\lambda_1 , \dots, \lambda_k$ be measures that satisfy the hypothesis of Corollary 1.3.10. We are going to use a dyadic decomposition. \\

Let $m := \lfloor \log_2(R) \rfloor + 1$.
Then $\lambda_j$ is supported in the interval $[2^{-m}, 2^m ]$.
Define, for $A$ a borel subset of $\mathbb{R}$ and for $r=-m+1,\dots , m$:  $$ \lambda_{j,r}(A) := R^{-1} \lambda_j\left( 2^{r} \left( A \cap [1/2,1[ \right) \right) $$
Those measures are all supported in $ [1/2,1[$, and have total mass $\lambda_{j,r}(\mathbb{R}) \leq 1$. \\

Moreover, a non concentration property is satisfied by each $\lambda_{j,r}$. If we fix some $r_1,\dots,r_k$ between $-m+1$ and $m$ and define ${\eta}_{r_1\dots r_k} := 2^{r_1+\dots r_k} \eta$, then $|\eta_{r_1,\dots, r_k}| \geq (2R)^{-k} |\eta| > 2^{-k} |\eta|^{1-k \varepsilon_1} > 1$ if $\eta$ is large enough. Let $\sigma \in [|\eta_{r_1,\dots, r_k}|^{-1}, |\eta_{r_1,\dots, r_k}|^{-{\varepsilon_2}}]$. Then
$$\lambda_{j,r} \otimes \lambda_{j,r} \left( \{ (x,y) \in \mathbb{R}^2, \ |x-y| \leq \sigma \} \right) = \int_{\mathbb{R}} \lambda_{j,r}( [x-\sigma,x+\sigma] ) d\lambda_{j,r}(x) $$
$$ \leq R^{-2} \int_{\mathbb{R}} \lambda_j\left( [2^r x - 2^r\sigma , 2^r x + 2^r \sigma ] \right) d\lambda_j( 2^r x ) $$
$$ = R^{-2} \lambda_{j} \otimes \lambda_j \left( \{ (x,y) \in \mathbb{R}^2, \ |x-y| \leq 2^r \sigma \} \right) $$
Since $2^{r} \sigma \in \left[ 2^{r} |\eta_{r_1,\dots,r_k}|^{-1} , 2^{r} |\eta_{r_1,\dots,r_k}|^{- \varepsilon_2} \right] \subset \left[ (2R)^{-(k+1)} |\eta|^{-1} , (2R)^{k+1}|\eta|^{- \varepsilon_2}\right] \subset \left[|\eta|^{-2}, |\eta|^{\varepsilon_1}\right]$ if $|\eta|$ is large enough, we can use the non-concentration hypothesis assumed for each $\lambda_j$ to get:
$$ \lambda_{j,r} \otimes \lambda_{j,r} \left( \{ (x,y) \in \mathbb{R}^2, \ |x-y| \leq \sigma \} \right) \leq R^{-2} (2^{r} \sigma)^\gamma \leq  \sigma^\gamma .$$
Hence, by the previous proposition, there exists a constant $C_1$ depending only on $\gamma$ such that
$$ \left| \int \exp( i \eta_{r_1 \dots r_k} z_1 \dots z_k)  d\lambda_{1,r_1}(z_1) \dots d\lambda_{k,r_k}(z_k) \right| \leq C_1 |\eta_{r_1 \dots r_k}|^{- \varepsilon_2} . $$
Finally, since $$ \lambda_j(A) = R \sum_{r=-m+1}^m \lambda_{j,r}(2^{-r} A) $$
we get that:
$$  \left| \int \exp( i  \eta z_1 \dots z_k)  d\lambda_1(z_1) \dots d\lambda_k(z_k) \right| $$
$$ \leq \sum_{r_1 , \dots r_k} R^k \left| \int \exp(i \eta z_1 \dots z_k)  d\lambda_{1,r_1}(2^{-r_1} z_1) \dots d\lambda_{k,r_k}(2^{-r_k} z_k) \right| $$
$$ =  \sum_{r_1 , \dots r_k} R^k \left| \int \exp(i \eta_{r_1 \dots r_k} z_1 \dots z_k)  d\lambda_{1,r_1}(z_1) \dots d\lambda_{k,r_k}(z_k) \right|$$
$$ \leq C_1 (2m)^k R^k |\eta_{r_1 \dots r_k}|^{- \varepsilon_2}  \leq 4^k C_1 m^k R^{2k} |\eta|^{- \varepsilon_2} $$
Since $m \leq \log_2(R) +1$, and since $k$ depends only on $\gamma$, there exists a constant $c$ that depends only on $\gamma$ such that $ 4^k C_1 m^k R^{2k} \leq c R^{2k+1}$ for any $R> 1$. Finally, $ c R^{2k+1} |\eta|^{- \varepsilon_2} \leq |\eta|^{-\varepsilon_1} $. \end{proof}
As we saw earlier, these theorems are useful even when the measures are sums of dirac masses. Let us precise the statement that we get in this particular context.
\begin{corollary}
Fix $0 < \gamma < 1$. There exist $k \in \mathbb{N}^*$, $c>0$ and $\varepsilon_1 > 0$ depending only on $\gamma$ such that the following holds for $\eta \in \mathbb{R}$ with $|\eta|$ large enough. Let $1 < R < |\eta|^{\varepsilon_1}$ , $N > 1$ and $\mathcal{Z}_1,\dots , \mathcal{Z}_k$ be finite sets such that $ \# \mathcal{Z}_j \leq RN$. Consider some maps $\zeta_j : \mathcal{Z}_j \rightarrow \mathbb{R} $, $j = 1, \dots , k$, such that, for all $j$:
$$ \zeta_j ( \mathcal{Z}_j ) \subset [R^{-1},R] $$ and 
$$\forall \sigma \in [|\eta|^{-2}, |\eta|^{- \varepsilon_1}], \quad \# \{\mathbf{b} , \mathbf{c} \in \mathcal{Z}^2_j , \ |\zeta_j(\mathbf{b}) - \zeta_j(\mathbf{c})| \leq \sigma \} \leq N^2 \sigma^{\gamma}.$$
Then 
$$ \left| N^{-k} \sum_{\mathbf{b}_1 \in \mathcal{Z}_1,\dots,\mathbf{b}_k \in \mathcal{Z}_k} \exp\left( i \eta \zeta_1(\mathbf{b}_1) \dots \zeta_k(\mathbf{b}_k) \right) \right| \leq c |\eta|^{-{\varepsilon_1}}.$$
\end{corollary}

\begin{proof}
Define our measures as sums of dirac mass: $$ \lambda_j := \frac{1}{N} \sum_{\mathbf{b} \in \mathcal{Z}_j} \delta_{\zeta_j(\mathbf{b})} .$$
We see that $\lambda_j$ is supported in  $ [R^{-1},R] $. The total mass is bounded by
$$ \lambda_j(\mathbb{R}) \leq N^{-1} \# \mathcal{Z}_j \leq R.$$
Then, if $\sigma \in [|\eta|^{-2}, |\eta|^{- \varepsilon_1}]$, we have, for any $a \in \mathbb{R}$:
$$ \lambda_j \otimes \lambda_j \left( \{ (x,y) \in \mathbb{R}^2 , \ |x-y|< \sigma \} \right) = \frac{1}{N^2} \# \left\{\mathbf{b} , \mathbf{c} \in \mathcal{Z}^2_j  , \  |\zeta_j(\mathbf{b}) - \zeta_j(\mathbf{c}) | \leq \sigma \right\} \leq  \sigma^\gamma. $$
Hence, the previous theorem applies directly, and gives us the desired result. \end{proof}
We have no reason to expect the sum-product phenomenon to hold only for measures supported on $\mathbb{R}$. In fact, more general versions of these theorems holds when we replace $\mathbb{R}$ by other algebras. Let us state some generalizations: the first one is in the context of complex multiplication, and was proved by J. Li in \cite{Li18}.
\begin{theorem}
Given $\gamma > 0$, there exist $ \varepsilon_2 \in \ (0,1)$ and $k \in \mathbb{N}^*$ such that the following holds for
$\eta \in \mathbb{C}$ with $|\eta| > 1$. Let $C_0 > 1$ and let $\lambda_1, \dots , \lambda_k$ be Borel measures supported on the annulus $\{ z \in \mathbb{C} \ , \ C_0^{-1} \leq |z| \leq C_0 \}$ with total mass less than $C_0$. Assume that each $\lambda_j$ satisfies the projective non concentration property, that is,
$$\forall \sigma \in [C_0 |\eta|^{-1}, C_0^{-1} |\eta|^{- \varepsilon_2}], \ \sup_{a,\theta \in \mathbb{R}} \ \lambda_j \{z \in \mathbb{C}, \ | \text{Re} (e^{i\theta} z) - a|  \leq \sigma \} \leq C_0 \sigma^\gamma .$$
Then there exists a constant $C_1$ depending only on $C_0$ and $\gamma$ such that
$$ \left| \int \exp(2 i \pi Re( \eta z_1 \dots z_k))  d\lambda_1(z_1) \dots d\lambda_k(z_k) \right| \leq C_1 |\eta|^{- \varepsilon_2} .$$
\end{theorem}
The non-concentration hypothesis to check here is stronger than earlier: this is a \say{projective non-concentration hypothesis}. As earlier, one can actually suppose that the constants $C_0$ are slowly growing with $\eta$ (the same proof follows through, doing a dyadic decomposition with annulus of the form $\{ z \in \mathbb{C}, \ |z| \in [2^{r-1},2^{r}[ \}$). If needed, the complete proof can be found in the author's paper \cite{Le21}. As before, by considering measures that are sums of dirac sums, one gets the following useful corollary:
\begin{corollary}
Fix $0 < \gamma < 1$. Then there exist $k \in \mathbb{N}^*$, $c>0$ and $\varepsilon_1 > 0$ depending only on $\gamma$ such that the following holds for $\eta \in \mathbb{C}$ with $|\eta|$ large enough. Let $1 < R < |\eta|^{\varepsilon_1}$ , $N > 1$ and $\mathcal{Z}_1, . . . , \mathcal{Z}_k$ be finite sets such that $ \# \mathcal{Z}_j \leq RN$. Consider some maps $\zeta_j : \mathcal{Z}_j \rightarrow \mathbb{C} $, $j = 1, \dots , k$, such that, for all $j$:
$$ \zeta_j ( \mathcal{Z}_j ) \subset \{ z \in \mathbb{C} \ , \ R^{-1} \leq |z| \leq R \}$$ and 
$$\forall \sigma \in [|\eta|^{-2}, |\eta|^{- \varepsilon_1}], \quad \sup_{a,\theta \in \mathbb{R}} \ \# \{\mathbf{b} \in \mathcal{Z}_j , \ | \text{Re} (e^{i\theta} \zeta_j(\mathbf{b})) - a|  \leq \sigma \} \leq N \sigma^{\gamma}.$$
Then: 
$$ \left| N^{-k} \sum_{\mathbf{b}_1 \in \mathcal{Z}_1,\dots,\mathbf{b}_k \in \mathcal{Z}_k} \exp\left(2 i \pi \text{Re}\left(\eta \zeta_1(\mathbf{b}_1) \dots \zeta_k(\mathbf{b}_k) \right) \right) \right| \leq c |\eta|^{-{\varepsilon_1}}.$$
\end{corollary}

Let us mention two further generalizations of these types of theorems in higher dimensional settings. We won't use them since we will be working in low dimensional settings, but they might be useful in harder contexts. In the next Theorem, we think of $\mathbb{R}^d$ as an algebra with the product $(x_j)_j (y_j)_j := (x_j y_j)_j$. This version is proved in Li's paper \cite{Li18}.
\begin{theorem}
Fix $\kappa > 0$. Then there exists $\varepsilon$ and $k \in \mathbb{N}$ such that the following holds for any $\tau$ large enough. Let $\lambda$ be a borel probability measure on $[1/2,1]^d \subset \mathbb{R}^d$ such that :
$$ \forall \rho \in [\tau^{-1},\tau^{\varepsilon}], \ \sup_{a \in \mathbb{R}, v \in \mathbb{S}^{d-1} } \lambda\{ x \ | \ v \cdot x \in [a - \rho , a + \rho] \} \leq \rho^{\kappa} .$$
Then, for all $\xi \in \mathbb{R}^n$ with $ | \xi | \in [ \tau/2 , \tau ] $, 
$$ \left| \int \exp(2 i \pi \xi \cdot (x^1 \dots x^k)) d\lambda(x^1) \dots d \lambda(x^k) \right| \leq \tau^{-\varepsilon}. $$
\end{theorem}

The final version takes place in an abstract setting and will probably prove to be useful when dealing with higher dimensional dynamical systems in the future. It is taken from He and de Saxcé's paper \cite{HdS22}. 
Let us define some notations. If $E$ is a normed simple algebra over $\mathbb{R}$, then we denote, for $W \subset E$ and $\rho>0$, $W^{(\rho)} := W + B(0,\rho)$ the $\rho$-neighborhood of $W$. For $a \in E$, define $\det(a)$ as the determinant of the linear map $x \in E \mapsto ax \in E$. For $\rho>0$, define the set of badly invertible elements of $E$ as $$ S_E(\rho) := \{ x \in E, \ |\det_E(x)| \leq \rho \}.$$
We are ready to quote our last version of the sum-product phenomenon.
\begin{theorem}[\cite{HdS22}]
Let $E$ be a normed simple algebra over $\mathbb{R}$ of finite dimension. Given $\kappa>0$, there exists $s \in \mathbb{N}$ and $\varepsilon>0$ such that for any parameter $\tau \in (0,\varepsilon \kappa)$ the following hold for any scale $\delta>0$ sufficiently small. Let $\mu$ be a Borel probability measure on $E$. Assume that
\begin{enumerate}
\item $\mu\Big{(} E \setminus B(0,\delta^{-\varepsilon}) \Big{)} \leq \delta^\tau$
\item For every $x \in E$, $\mu(x + S_E(\delta^\varepsilon)) \leq \delta^\tau$
\item For every $\rho \geq \delta$ and every proper affine subspace $W \subset E$, $\mu(W^{(\rho)}) \leq \delta^{-\varepsilon} \rho^\kappa$.
\end{enumerate}
Then, for all $\xi \in E^*$ with $\|\xi\|=\delta^{-1}$,
$$ |\widehat{\mu^{\odot s}}(\xi)| \leq \delta^{\varepsilon \tau} .$$
\end{theorem}
Let us conclude this chapter by stating that the authors of the paper mention (without proof) that a similar statement should hold for semisimple algebras.

\cleardoublepage
\chapter{On oscillatory integrals with Hölder phases}

\section{About oscillatory integrals}

\subsection{The Van Der Corput Lemma}

The goal of this Chapter is to explain our strategy to prove Fourier decay in a simple context. In this particular setting, this can be formulated as a statement on oscillatory integrals with Hölder \say{autosimilar} phases. Let us first recall the basics. \\

About oscillatory integrals, one of the first question that comes to mind is the following: give some sufficient conditions on the phase function $\psi$ so that the associated integral exhibit power decay, in the sense that there exists $\delta > 0$ such that for all large $|\xi|$, we have $$ \left| \int_0^1 e^{i \xi \psi(x)} dx \right| \leq  |\xi|^{-\delta} .$$
The first result of this kind is the useful Van Der Corput lemma, which we recall:

\begin{lemma}[Van Der Corput]

Let $\psi : [0,1] \rightarrow \mathbb{R}$ be a  $C^{k+1}$ phase ($k \geq 2$) such that the $k$-th derivative satisfies $\psi^{(k)} \geq 1$. Then, there exists $C_k>0$ such that, for all $\xi \geq 1$, 
$$ \left|  \int_0^1 e^{i \xi \psi(x)} dx \right| \leq C_k \xi^{-1/k} .$$

\end{lemma}

Essentially, Lemma 2.1.1 states that our oscillatory integral exhibit power decay as soon as our (smooth) phase satisfies a form of non-concentration hypothesis (the condition on the derivative). Variants of this lemma (e.g. the non-stationary phase) also relies on non-concentration of the phase, and on its smoothness. It is surprising to find that, to the author's knowledge, no deterministic example of (non-absolutely continuous) Hölder maps $\psi$ are known to satisfies this type of result. Yet, non-smooth maps appear regularly, and one must find a way to deal with them: for example, in the context of hyperbolic dynamical systems, the stable/unstable foliation is known to be only Hölder regular. Conjugacy between dynamical systems are often only Hölder, and invariant sets (such as Julia sets in the context of conformal dynamics) are often very non-smooth. \\

The goal of this Chapter is to construct deterministic examples of Hölder phases satisfying a Van Der Corput type of estimate, despite not being absolutely continuous. The lack of smoothness of the phase will be replaced by a form of autosimiliarity, which will play a key role in the proof.

\subsection{A probabilistic example: the Brownian motion}

Before stating our main result, we will discuss some estimates that have been proved in a random setting by Kahane \cite{Ka85}. Let $(X_n)_{n \geq 0}$ and $(Y_n)_{n \geq 1}$ be some i.i.d. random variables following a normalized Gaussian distribution $\mathcal{N}(0,1)$. We define the Brownian motion (or Wiener process) on $[0,1]$ by the following stochastic process:
$$ W(t) := X_0 t + \sqrt{2} \sum_{n=1}^\infty \frac{1}{2 \pi n} \Big( X_n \sin( 2 \pi n t) 
 + Y_n (1- \cos(2 \pi n t))\Big) ,$$
where convergence takes place almost surely in the $L^2(0,1)$ sense. (Indeed, the fact that \\
$\mathbb{E}( \sum_{k \geq 1} \exp(X_k^2/4) k^{-2})< \infty$ implies that the series inside is finite a.s., and so $X_k = O(\sqrt{\ln k})$.) It is known that, for any $\alpha < 1/2$, $W$ defines almost surely a $\alpha$-Hölder map, and for any $\alpha \geq 1/2$ is almost surely not $\alpha$-Hölder. (\cite{Ka85}, p. 235, Th. 2 p. 236 and Th. 3 p. 241) The following estimate holds.

\begin{proposition}[\cite{Ka85}, p. 255]

Almost surely, there exists $C>0$ such that for all $|\xi| \geq 1,$ we have $$ \left| \int_0^1 e^{i \xi W(t)} dt \right| \leq C |\xi|^{-1} \sqrt{\ln |\xi|}. $$

\end{proposition}

This proposition is very inspiring: it suggest that a form of Van Der Corput Lemma should hold in some generic sense for some genuinely \say{fractal} phases. More precisely, the lack of regularity of the Brownian motion is compensated by its statistical autosimilarity (\cite{Ka85}, Th. 1 p. 234): it is known that, for any $c>0$, the process
$$ t \mapsto \sqrt{c} \ W(t/c) $$
also defines a Brownian motion. Moreover, for any $a>0$, the following scaled increments all follow the same gaussian law: $$ \frac{W(t+a) - W(t) }{\sqrt{a}}\sim \mathcal{N}(0,1) .$$   Those properties of the Brownian motion tell us that $W(t)$ behaves in the same way in every scale, which allows us to easily \say{zoom in} in the proofs as $\xi$ grows. \\

It is thus natural for us to search for a deterministic candidate in the realm of \say{fractal} functions. In the case of the Brownian motion, the property $ \sqrt{c} \ W(t/c) \sim W(t)$ can be formally rewritten as a form of conjugacy: in a sense, $W$ acts like a conjugacy between $x \mapsto cx$ and $x \mapsto \sqrt{c}x$. This may be a hint for us to consider conjugacies of dynamical systems as good candidates for phases.

\subsection{Our deterministic setting}

Our explicit family of autosimilar phases $\psi:[0,1] \rightarrow \mathbb{R}$ will be constructed as conjugacies between the doubling map and some perturbation. We denote by $\mathbb{S}^1 := \mathbb{R}/\mathbb{Z}$ the circle. Define the doubling map $f_0 : \mathbb{S} \rightarrow \mathbb{S}$ by $f_0(x) := 2x$.
 To state the main result, we need to recall a useful fact on perturbations of expanding maps.

\begin{proposition}[\cite{KH95}, Th.19.1.2 and Th.18.2.1]

Let $0<\alpha<1$. Then there exists $\delta>0$ such that the following holds. Let $f:\mathbb{S}^1 \rightarrow \mathbb{S}^1$ be a $C^{1+\alpha}$ $\delta$-perturbation of the doubling map $f_0: x \in  \mathbb{S}^1 \mapsto 2x \in  \mathbb{S}^1$, meaning that
$$ \|f-f_0\|_{C^{1+\alpha}} < \delta .$$

Then, reducing $\alpha$ if necessary, there exists a $\alpha$-Hölder conjugacy $\psi : ( \mathbb{S}^1,f_0) \rightarrow ( \mathbb{S}^1,f)$. In other words, $\psi : \mathbb{S}^1 \rightarrow \mathbb{S}^1$ is a homeomorphism, Hölder with Hölder inverse, and $\psi \circ f_0 = f \circ \psi$. Generically, $\psi$ is not absolutely continuous (meaning that its derivative in the sense of distributions is not in $L^1(\mathbb{S})$). 

\end{proposition}

The fact that $\psi$ is generically not absolutely continuous stands because any such conjugacy must be $C^{1+\alpha}$ (see \cite{SS85}), which is not allowed if $f'(\psi(0)) \neq 2$, for example. (The derivative is taken in the sense that $f$ can be identified with an increasing, 1-periodic and $C^{1+\alpha}$ function $\mathbb{R} \rightarrow \mathbb{R}$). We are ready to state our main Theorem.

\begin{theorem}
Let $f_0:\mathbb{S}^1 \longrightarrow \mathbb{S}^1$ be the doubling map. Let $f$ be a $C^{3}$ $\delta-$perturbation of $f_0$. Let $\psi : ( \mathbb{S}^1,f_0) \rightarrow ( \mathbb{S}^1,f)$ be the $\alpha$-Hölder conjugacy. Then, there exists $C>0$ and $\rho>0$ such that:

$$ \forall \xi \in \mathbb{R}^*, \ \left| \int_0^1 e^{i \xi \psi(x)} dx \right| \leq C |\xi|^{-\rho} .$$

\end{theorem}

We can re-write Theorem 2.1.4 in a measure-theoretic form.

\begin{definition}

Denote by $\mu := \psi_* (dx) $ the pushforward of the lebesgue measure by $\psi$, so that
    $$ \int_0^1 e^{i \xi \psi(x)} dx = \int_0^1 e^{i \xi x} d\mu(x)$$
is the Fourier transform $\widehat{\mu}(\xi)$ of the measure $\mu$. This probability measure is $f$-invariant on $\mathbb{S}^1$, since the Lebesgue measure is $f_0$-invariant on the circle.

\end{definition}

\begin{remark}
The measure $\mu$ is known as the \emph{measure of maximal entropy} for the dynamical system $(\mathbb{S},f)$, as it is the push-forward via $\psi$ of the Lebesgue measure $\lambda$, which is the measure of maximal entropy for the doubling map. Indeed, the topological entropy $h(f_0)$ of $(\mathbb{S},f_0)$ is $\ln 2$ (see \cite{BS02} section 2.5) and the measure-theoretic entropy $h_\lambda(f_0)$ is also $\ln 2$ (see \cite{BS02}, section 9.4, and  Th. 9.5.4).
\end{remark}

To prove this result, we separate two cases: one where $f$ satisfies a linearity condition, and one where $f$ is \emph{totally non linear.} The \emph{total nonlinearity condition} (TNL) is defined as follows.

\begin{definition}
We say that $f$ satisfies (TNL) if there exists no $C^1$ with bounded derivative map $\theta: \mathbb{S}\setminus\{\psi(0),\psi(1/2)\} \rightarrow \mathbb{R}$ and no locally constant map $\kappa: \mathbb{S}\setminus\{\psi(0),\psi(1/2)\} \rightarrow \mathbb{R}$  such that, on $\mathbb{S}\setminus\{\psi(0),\psi(1/2)\}$, $$ \ln f' = \theta \circ f - \theta + \kappa .$$
\end{definition}

\begin{lemma}

If $f$ doesn't satisfies $(TNL)$, then Theorem 2.1.4 holds.

\end{lemma}

\begin{proof}

Suppose that there exists a $C^1$ with bounded derivative $\theta: \mathbb{S}\setminus\{\psi(0),\psi(1/2)\} \rightarrow \mathbb{R}$ and a locally constant map $\kappa: \mathbb{S}\setminus\{\psi(0),\psi(1/2)\} \rightarrow \mathbb{R}$  such that $$ \ln f' = \theta \circ f - \theta + \kappa .$$
 on $\mathbb{S}\setminus\{\psi(0),\psi(1/2)\}$. Since we are working with the doubling map $f_0$, $\psi(0)$ is a fixed point of $f$ that bounds our two intervals, and this implies that $\kappa$ is constant. Indeed, we see that
 $$ \ln f'(\psi(0^+)) = \theta( f(\psi(0^+)) ) -  \theta( \psi(0^+) ) + \kappa(\psi(0^+)) =  \kappa(\psi(0^+)) $$
 and, similarly, $\ln f'(\psi(0^-)) = \kappa(\psi(0^-))$. Hence $ \ln f' = \theta \circ f - \theta + \kappa $ for some constant $\kappa$. We say that $f$ is \emph{cohomologous} to a constant. It is then well known (\cite{PP90}, Th 3.6, and \cite{Ba18}, Th 2.2) that $\mu$, the measure of maximal entropy for $(\mathbb{S},f)$, is equal to the \emph{SRB measure}, that is, the only invariant probability measure that is absolutely continuous with respect to the Lebesgue measure. Moreover, in our case, the density of the $SRB$ measure is $C^1$, which is enough to ensure power decay for $\widehat{\mu}$. \end{proof}

The totally nonlinear case directly follows from the work of Sahlsten and Steven \cite{SS20}, dealing with the power decay of the Fourier transform of equilibrium states for one dimensional expanding maps. The proof is fairly technical in the general case, but can be drastically simplified in our perturbative case. Let us precise our setting: for any $C^{2}$ and 1-periodic function $\Phi:\mathbb{R} \rightarrow \mathbb{R}$, we set, for $\delta>0$ small enough:
$$ f_\delta(x) := z_\delta \int_0^{2x} e^{\delta \Phi(t)} dt $$
where $z_\delta>0$ is a normalization factor chosen so that $f_\delta(1/2)=1$. We further make the \say{non-constantness} hypothesis that $2 \Phi(0) \neq \Phi(1/3) + \Phi(2/3) $. Notice that $\ln(f_\delta') = \ln(2 z_\delta) + \delta \Phi$, and so $f_\delta$ satisfies (TNL) iff $\Phi$ is not cohomologous to a locally constant map. We will check later that this is always true if $\delta$ is small enough, thanks to our non-constantness hypothesis on $\Phi$. The map $f_\delta$ factors into a perturbation of the doubling map on the circle. Denote the associated conjugacy $\Psi_\delta$. We prove the following.

\begin{theorem}
There exists $C > 0$ and $\rho>0$ such that, if $\delta>0$ is small enough:

$$\forall |\xi| \geq 1, \ \left| \int_0^1 e^{i \xi \Psi_\delta(x)} dx \right| \leq C \delta^{-1} |\xi|^{- \rho} .$$

\end{theorem}

It is worth noticing that our explicit approach allows us to see that, even though the exponent $\rho$ depends on $\Phi$, it doesn't depends on $\delta$: this partly answers, in our particular case, a question found in \cite{SS20} about the dependence of this exponent on the dynamics. \\

The plan of this chapter is the following. Section 2.2, 2.3 and 2.4 shows how one can reduce Theorem 2.1.9 to checking a \say{nonlinearity estimate} on $f$ for a general $C^{3}$ perturbation of the doubling map. In section 2.5, we see how one can check those estimates for our perturbation defined above in an elementary way, inspired from \cite{BD17}. 

\begin{remark}
It was pointed out to me by Frederic Naud that another example of Hölder phase satisfying this \say{fractal Van der corput lemma} can be easily constructed. Consider the doubling map, but this time, seen as the map $z \mapsto z^2$ restricted to the unit circle $\mathbb{U} \subset \mathbb{C}$. If $c$ is a small enough complex number, then there exists a topological circle $J_c$ on which the dynamical system $z \mapsto z^2+c$ is well defined (\cite{Ly86}, section 1.16). The two dynamics are then conjugated by a quasiconformal mapping $\psi : \mathbb{U} \rightarrow J_c$ (which is Hölder, \cite{FS58}). It follows from the author previous work \cite{Le21} that there exists $\varepsilon>0$ and $C>0$ such that
$$ \forall z \in \mathbb{C},\ \left| \int_\mathbb{U} e^{i \text{Re}\left(z \psi(\theta)\right) } d\theta \right| \leq C (1+|z|)^{-\varepsilon} .$$
This will be a consequence of Chapter 3, noticing that $\psi_*(d\theta)$ is the measure of maximal entropy of $z \mapsto z^2+c$ on $J_c$.
\end{remark}

\section{Preliminary facts}

We choose a $C^{3}$ $\delta$-perturbation $f$ of the doubling map $f_0 : x \mapsto 2x \mod 1$, for $\delta$ small enough. Denote the Hölder conjugacy $\psi$. First of all, we define some inverse branches for the doubling map.

\begin{definition}

Define $S_0^{(0)} := [0,1/2)$ and $S_1^{(0)} := [1/2,1)$. This is a partition of $[0,1) \simeq \mathbb{S}^1 $ adapted to the doubling map. Define the associated inverse branches by:
$$ \begin{array}[t]{lrcl}
 g_0^{(0)}: & [0,1) & \longrightarrow & S_0^{(0)} \\
    & x  & \longmapsto &  x/2 \end{array}  \ , \quad \begin{array}[t]{lrcl}
 g_1^{(0)}: &  [0,1)  & \longrightarrow &  S_1^{(0)} \\
    & x  & \longmapsto &  (x+1)/2 \end{array}  $$
For any finite word $\mathbf{a} = a_1 \dots a_n \in \{0,1\}^n$, define 
$$ g_{\mathbf{a}}^{(0)} := g_{a_1}^{(0)} \dots g_{a_n}^{(0)} .$$
The cylinder set associated to the word $\mathbf{a}$ is defined by $ S_\mathbf{a}^{(0)} := g_{\mathbf{a}}^{(0)}(\mathbb{S}) $, and the collection  \\ $\{S_\mathbf{a}^{(0)} \ , \ \mathbf{a} \in \{0,1\}^n \}$ is a partition of $\mathbb{S}^1$.

\end{definition}

Recall that the Lebesgue measure is invariant by the doubling map. Moreover, for any measurable map $h : \mathbb{S}^1 \rightarrow \mathbb{C}$, we have the identity
$$ \int_{\mathbb{S}} h(x) dx = \frac{1}{2} \left( \int_{\mathbb{S}} h(g_0^{(0)}(x)) dx + \int_{\mathbb{S}} h(g_1^{(0)}(x)) dx \right) ,$$
which gives, by induction:
$$ \int_{\mathbb{S}} h(x) dx = \frac{1}{2^n} \sum_{\mathbf{a} \in \{0,1\}^n } \int_{\mathbb{S}} h\left( g_{\mathbf{a}}^{(0)}(x)\right) dx .$$ 
Now, we use our conjugacy $\psi$ to define similar inverse branches and partitions for $f$.

\begin{definition}
For any finite word $\mathbf{a} \in \{0,1\}^n$, define
$$ g_{\mathbf{a}} := \psi \circ g_{\mathbf{a}}^{(0)} \circ \psi^{-1} \text{ and } S_{\mathbf{a}} := \psi\left( S_{\mathbf{a}}^{(0)} \right) .$$
Notice that $S_{\mathbf{a}} = g_{\mathbf{a}}(\mathbb{S})$, and that $g_{\mathbf{a}}$ is a local inverse of $f$. In particular, it is $C^{3}$ on $[0,1)$.
\end{definition}

By definition of the inverse branches, the associated partition, and the pushforward measure $\mu$, we see that we have the following identity, holding for measurable maps $h:\mathbb{S} \rightarrow \mathbb{C}$:

$$ \int_\mathbb{S} h d\mu = \frac{1}{2^n} \sum_{\mathbf{a} \in \{0,1\}^n} \int_{\mathbb{S}} h(g_{\mathbf{a}}(x)) d\mu(x). $$

In spirit, this identity is a consequence of the autosimilarity of $\mu$ (which itself is a consequence of the autosimilarity of $\psi$, encoded by the fact that it is a conjugacy between expanding maps). It will allow us to work on small scales with control, as the maps $g_{\mathbf{a}}$ \say{zooms in} while respecting the structure of $\psi$. \\

We define some notations for \say{orders of magnitude}. If there exists a constant $C>0$ independent of $n$ such that $a_n \leq C b_n$, then we write $a_n \lesssim b_n$. If $a_n \lesssim b_n \lesssim a_n$, we denote $a_n \simeq b_n$. If there exist $C,\alpha>0$, independent of $n$ and $\delta$, such that $ C^{-1} e^{- \alpha \delta n} a_n  \leq b_n \leq C a_n e^{\alpha \delta n} $, then we denote it by $a_n \sim b_n$. (Recall that $\delta$ measure the $C^{3}$ distance between $f$ and $f_0$). Then:

\begin{lemma} 
Recall that the perturbation $f$ is supposed to satisfy $\|f-f_0\|_{C^{3}} \lesssim \delta$. The following order of magnitude holds, for $n \geq 1$ and $\mathbf{a} \in \{0,1\}$:
$$ g_{\mathbf{a}}' \sim 2^{-n} \ , \ \text{diam}(S_\mathbf{a}) \sim 2^{-n}.$$
Moreover, $ \mu(S_\mathbf{a}) = 2^{-n} $.
\end{lemma}

The proof is straightforward, as $(f^n)' \sim 2^n$ (denoting $f^n := f \circ \dots \circ f$). The second estimate is a consequence from the first, using the mean value theorem. The last equality is by definition of $\mu$ and $S_\mathbf{a}$. \\

Finally, we prove a nonconcentration estimate for $\psi$. 

\begin{lemma}
There exists $C,\delta_{\mu}>0$ such that:

$$ \forall x \in \mathbb{S}, \ \forall r>0, \ \mu\left( [x-r,x+r] \right) \leq C r^{\delta_\mu} . $$
\end{lemma}

This can be rewritten as $$ \lambda\left( \left\{ y \in \mathbb{S}, \ \psi(y) \in [x-r,x+r] \right\} \right) \leq C r^{\delta_\mu} ,$$
where $\lambda$ denotes the Lebesgue measure on the circle, which explains why we can see this estimate on $\mu$ as a nonconcentration estimate on $\psi$.

\begin{proof}

Fix $x \in \mathbb{S}$ and $r>0$ small enough. We have, for any $n \geq 1$:

$$ \mu\left( [x-r,x+r] \right) \leq \underset{S_{\mathbf{a}} \cap [x-r,x+r] \neq \emptyset }{\sum_{\mathbf{a} \in \{0,1\}^n}} \mu(S_{\mathbf{a}}) = 2^{-n} \cdot \#\left\{\mathbf{a} \in \{0,1\}^n, \ S_{\mathbf{a}} \cap [x-r,x+r] \neq \emptyset \right\} $$

Recall that $ \text{diam}(S_\mathbf{a}) \sim 2^{-n}$. In particular, there exists $C,\alpha>0$ such that $\text{diam}(S_\mathbf{a}) \geq C 2^{-n} e^{-\delta \alpha n}$. Choosing $n(r) := \lfloor \frac{ |\ln(r)| }{\ln 2 + \delta \alpha  }\rfloor$ yields $ \text{diam}(S_\mathbf{a}) \geq C r $, so that

$$ \#\left\{\mathbf{a} \in \{0,1\}^n, \ S_{\mathbf{a}} \cap [x-r,x+r] \neq \emptyset \right\} \leq 4 C, $$
and so 
$$ \mu\left( [x-r,x+r] \right) \leq 4 C \cdot 2^{-n(r)} \leq C' r^{\delta_\mu} $$
for some $C'>0$ and for $ \delta_\mu(\delta) := (1+\delta \alpha/ \ln 2)^{-1} < 1 $. Notice that $\delta_\mu$ approaches one as $\delta$ gets smaller. \end{proof}

\section{Reduction to sums of exponentials}

We are ready to reduce our Van Der Corput lemma to a \say{nonlinearity estimate}. The goal of this section is to approximate the integral by a finite sum of exponential, which will be controlled by a powerful theorem of Bourgain from additive combinatorics, as soon as those nonlinearity estimates are checked. In this section, 5 quantities will be at play: $\xi$, $n$, $k$, $\varepsilon_0$ and $\delta$. The only two variables are $\xi$ and $n$, and they are related by a relation of the form $n \simeq \ln |\xi|$. The quantities $k$ and $\varepsilon_0$ are constant parameters that will be fixed in section 2.4 while applying Theorem 2.4.1. The parameter $\delta$ will be chosen small before any other constant (depending on $\Phi$). \\

Our goal is to prove a bound of the form
$$ \forall \xi, \ \left| \widehat{\mu}(\xi) \right| \lesssim |\xi|^{- \rho}. $$
In the next lemma, we will consider a family of words $\mathbf{a}_{j}  \in \{0,1\}^{n}$, and we will denote their concatenation $\mathbf{A} := \mathbf{a}_0 \dots \mathbf{a}_k \in \{0,1\}^{(k+1)n}$. Same for words $\mathbf{b}_j \in \{0,1\}$ and their concatenation $\mathbf{B} := \mathbf{b}_1 \dots \mathbf{b}_k \in \{0,1\}^{kn}$. This section is devoted to the proof of the following reduction.

\begin{lemma}
Fix some $\varepsilon_0>0$ small enough. Define, for $j=1, \dots, k $, $\mathbf{A} = \mathbf{a}_0 \dots \mathbf{a}_k \in \{ 0,1\}^{(k+1)n}$ and $\mathbf{b} \in \{0,1\}^n$ : $$ \zeta_{\mathbf{A},j}(\mathbf{b}) := 4^{n} g_{\mathbf{a}_{j-1} \mathbf{b}}'\left( x_{\mathbf{a}_j} \right) \sim 1 $$
where $x_{\mathbf{a}} := g_{\mathbf{a}}(0) \in S_{\mathbf{a}}$. Then, for $|\xi|$ large enough, the following holds: $$|\widehat{\mu}(\xi)|^2 \lesssim e^{- \varepsilon_0  \delta_\mu n / 4} + 2^{-(k+1)n} \sum_{\mathbf{A} \in \{0,1\}^{(k+1)n}} \sup_{\eta \in [e^{\varepsilon_0 n/2}, e^{2 \varepsilon_0 n}]} 2^{-kn} \left| \sum_{\mathbf{B} \in \{0,1\}^{kn}} e^{ i \eta \zeta_{\mathbf{A},1}(\mathbf{b}_1) \dots \zeta_{\mathbf{A},k}(\mathbf{b}_k) } \right|,$$

where $n := \left\lfloor \frac{\ln |\xi|}{ (2k+2) \ln 2 - \varepsilon_0} \right\rfloor$.

\end{lemma}

\begin{proof}

First of all, using the autosimilarity formula for $\mu$, we get for any integer $N$:

$$ \int_0^1 e^{i \xi \psi(x)} dx = \int_0^1 e^{i \xi x} d\mu(x) $$
$$ = 2^{-N} \sum_{\mathbf{C} \in \{0,1\}^{N} } \int_{\mathbb{S}}  e^{i \xi g_{\mathbf{C}}(x)} d\mu(x) . $$

The actual value of $g_{\mathbf{C}}$ isn't important to us, the only relevant information for Fourier decay is its non-concentration. Hence, we are encouraged to use the Cauchy-Schwarz inequality like so:

$$ |\widehat{\mu}(\xi)|^2 \leq 2^{-N} \sum_{\mathbf{C} \in \{0,1\}^{N}} \left| \int_{\mathbb{S}}  e^{i \xi g_{\mathbf{C}}(x)} d\mu(x) \right|^2$$
$$ = 2^{-N} \sum_{\mathbf{C} \in \{0,1\}^{N}} \iint_{\mathbb{S}\times \mathbb{S}}  e^{i \xi \left(g_{\mathbf{C}}(x) - g_{\mathbf{C}}(y) \right)} d\mu(x) d\mu(y) .$$

Now, choose $N := (2k+1) n$, and set $\mathbf{C} = \mathbf{a}_0 \mathbf{b}_1 \mathbf{a}_1 \dots \mathbf{a}_{k-1} \mathbf{b}_k \mathbf{a}_k$, where $\mathbf{a}_i, \mathbf{b}_i \in \{0,1\}^n$. In a more compact fashion, we will denote $\mathbf{A} := \mathbf{a}_0 \mathbf{a}_1 \dots \mathbf{a}_k \in \{0,1\}^{(k+1)n} $, $ \mathbf{B} := \mathbf{b}_1 \dots \mathbf{b}_k \in \{0,1\}^{kn} $, and $\mathbf{A} * \mathbf{B} := \mathbf{C}$. This gives:

$$ |\widehat{\mu}(\xi)|^2 \lesssim 2^{-(2k+1)n} \sum_{\mathbf{A},\mathbf{B}} \iint_{\mathbb{S}\times \mathbb{S}}  e^{i \xi \left(g_{\mathbf{A*B}}(x) - g_{\mathbf{A*B}}(y) \right)} d\mu(x) d\mu(y). $$

Now, we will carefully linearize the phase. Define $\mathbf{A} \# \mathbf{B} = \mathbf{a}_0 \mathbf{b}_1 \mathbf{a}_1 \dots \mathbf{a}_{k-1} \mathbf{b}_k$. Notice that, by the mean value theorem, for all $x,y \in [0,1)$, there exists $z \in [0,1)$ such that

$$ g_{\mathbf{A*B}}(x) - g_{\mathbf{A*B}}(y) = g_{\mathbf{A\#B}}'(z)(\widehat{x}-\widehat{y}) $$
where $\widehat{x} := g_{\mathbf{a}_k}(x)$ and $\widehat{y} := g_{\mathbf{a}_k}(y)$. 
The main idea is that $g_{\mathbf{A\#B}}'(z)$ can be written as a product of $k$ functions, and this will allow us to apply the sum product-phenomenon to conclude (see section 2.4). We need to renormalize appropriately those functions. Define $$ \zeta_{\mathbf{A},j}(\mathbf{b}) := 4^n g_{\mathbf{a}_{j-1} \mathbf{b}}'( x_{\mathbf{a}_j} ) \sim 1 ,$$
where $x_{\mathbf{a}} := g_{\mathbf{a}}(0) \in S_{\mathbf{a}}$. The fact that $f'$ is Lipschitz gives us the following bounds:

$$ \left| 2^{-(2k+1)n} \sum_{\mathbf{A},\mathbf{B}} \iint_{\mathbb{S}\times \mathbb{S}} \left(  e^{i \xi \left(g_{\mathbf{A*B}}(x) - g_{\mathbf{A*B}}(y) - 4^{-kn} \zeta_{\mathbf{A},1}(\mathbf{b}_1) \dots \zeta_{\mathbf{A},k}(\mathbf{b}_k) (\widehat{x}-\widehat{y}) \right)} - 1 \right) d\mu(x) d\mu(y) \right| $$
$$ \lesssim |\xi|  2^{-(2k+1)n} \sum_{\mathbf{A},\mathbf{B}} \left| g_{\mathbf{A*B}}(x) - g_{\mathbf{A*B}}(y) - 4^{-kn} \zeta_{\mathbf{A},1}(\mathbf{b}_1) \dots \zeta_{\mathbf{A},k}(\mathbf{b}_k) (\widehat{x}-\widehat{y}) \right|  $$
$$ \lesssim  e^{\alpha \delta n} 2^{-(2k+2)n} |\xi| $$
for some $\alpha > 0$. This encourages us to relate $\xi$ and $n$ so that $$ |\xi| \simeq 2^{(2k+2)n} e^{- \varepsilon_0 n} $$
for some $\varepsilon_0 >0 $ small enough that will be fixed later.
This choice allows us to write, if $\delta>0$ is small enough, 
$$ |\widehat{\mu}(\xi)|^2 \lesssim  e^{- \varepsilon_0 n / 2} + 2^{-(2k+1)n} \sum_{\mathbf{A},\mathbf{B}} \iint_{\mathbb{S}\times \mathbb{S}}  e^{ i \xi 4^{kn} (\widehat{x} - \widehat{y}) \zeta_{\mathbf{A},1}(\mathbf{b}_1) \dots \zeta_{\mathbf{A},k}(\mathbf{b}_k) } d\mu(x) d\mu(y), $$
so that we may now work on the integral on the right side. Define $ \eta_{\mathbf{A}}(x,y) := \xi 4^{-k n} (\widehat{x} - \widehat{y}). $ The mean value Theorem gives us bounds of the form
$$ e^{\varepsilon_0 n} e^{-\alpha \delta n} (x-y) \lesssim |\eta_{\mathbf{A}}(x,y)| \lesssim e^{2 \varepsilon_0 n} .$$
To conclude, we just need to control the diagonal part of the integral. This is easily done using Lemma 2.2.4, as follows:
$$  2^{-(2k+1)n} \Big{|} \sum_{\mathbf{A},\mathbf{B}} \iint_{\left\{ |x-y| \leq e^{-(\varepsilon_0/2 - \alpha \delta) n} \right\}}  e^{ i\eta_{\mathbf{A}}(x,y) \zeta_{\mathbf{A},1}(\mathbf{b}_1) \dots \zeta_{\mathbf{A},k}(\mathbf{b}_k) } d\mu(x) d\mu(y) \Big{|} $$
$$ \lesssim \mu \otimes \mu \left( \left\{ (x,y) \in \mathbb{S} \times \mathbb{S} , \ |x-y| \leq e^{-(\varepsilon_0/2 - \alpha \delta) n} \right\} \right) $$ $$\lesssim e^{-(\varepsilon_0/2-\alpha \delta) \delta_\mu n} \lesssim e^{-\varepsilon_0 \delta_\mu n/4} .$$
So that now we may write, denoting by $D := \{ (x,y) \in \mathbb{R}^2 \ , \ |x-y| \leq e^{-(\varepsilon_0/2 - \alpha \delta)n} \}$ the previous neighborhood of the diagonal:
$$ |\widehat{\mu}(\xi)|^2 \lesssim  e^{- \varepsilon_0 n / 2} + e^{- \varepsilon_0  \delta_\mu n / 4} + 2^{-(2k+1)n} \Big{|} \sum_{\mathbf{A},\mathbf{B}} \iint_{\mathbb{S}\times \mathbb{S} \setminus D}  e^{ i \eta_\mathbf{A}(x,y) \zeta_{\mathbf{A},1}(\mathbf{b}_1) \dots \zeta_{\mathbf{A},k}(\mathbf{b}_k) } d\mu(x) d\mu(y) \Big{|} $$ 
$$\lesssim e^{- \varepsilon_0  \delta_\mu n / 4} + 2^{-(2k+1)n} \sum_{\mathbf{A} \in \{0,1\}^{(k+1)n}} \iint_{\mathbb{S}\times \mathbb{S} \setminus D} \left| \sum_{\mathbf{B} \in \{0,1\}^{kn}} e^{ i \eta_{\mathbf{A}}(x,y) \zeta_{\mathbf{A},1}(\mathbf{b}_1) \dots \zeta_{\mathbf{A},k}(\mathbf{b}_k) } \right| d\mu(x) d\mu(y) $$
$$ \lesssim e^{- \varepsilon_0  \delta_\mu n / 4} + 2^{-(k+1)n} \sum_{\mathbf{A} \in \{0,1\}^{(k+1)n}} \sup_{|\eta| \in [e^{\varepsilon_0 n/2}, e^{2 \varepsilon_0 n}]} 2^{-kn} \left| \sum_{\mathbf{B} \in \{0,1\}^{kn}} e^{ i \eta \zeta_{\mathbf{A},1}(\mathbf{b}_1) \dots \zeta_{\mathbf{A},k}(\mathbf{b}_k) } \right| .$$

\end{proof}

\section{The sum-product phenomenon}

The conclusion of the proof will be a consequence of a version of the sum-product phenomenon, which we recall (see Section 1.3.2 for similar statements):

\begin{theorem}[Sum-product phenomenon]

Fix $0 < \gamma_0 < 1$. There exist $k \in \mathbb{N}$, $c>0$ and $\varepsilon_1 > 0$ depending only on $\gamma_0$ such that the following holds for $\eta \in \mathbb{R}$ large enough. Let $\mathcal{Z}$ be a finite set, and fix some maps $\zeta_j : \mathcal{Z} \rightarrow \mathbb{R} $, $j = 1, \dots , k$, such that, for all $j$:
$$ \forall \mathbf{b} \in \mathcal{Z}, \ |\eta|^{-\varepsilon_1/2} \leq |\zeta_j(\mathbf{b})| \leq |\eta|^{\varepsilon_1/2} $$ and 
$$\forall \sigma  \in [ | \eta |^{-2} , |\eta|^{- \varepsilon_1} ], \quad  \# \{ (\mathbf{b},\mathbf{c}) \in \mathcal{Z}^2 , \ |\zeta_j(\mathbf{b})-\zeta_j(\mathbf{c})| \leq \sigma \} \leq \left( \# \mathcal{Z} \right)^2 \sigma^{\gamma_0}. \quad (*)$$
Then  
$$ \left| \frac{1}{\# \mathcal{Z}^k} \sum_{\mathbf{b}_1, \dots \mathbf{b}_k \in \mathcal{Z}} \exp\left( i \eta \zeta_1(\mathbf{b}_1) \dots \zeta_k(\mathbf{b}_k) \right)   \right| \leq c | \eta | ^{ - \varepsilon_1 } .$$

\end{theorem}

Our goal is to apply Theorem 2.4.1 with $ \mathcal{Z} := \{0,1\}^n $ and $\zeta_j := \zeta_{\mathbf{A},j}$. The fact that $\zeta_{\mathbf{A},j} \sim 1$ means that there exists a constant $\alpha>0$ such that $e^{-\alpha \delta n} \leq |\zeta_{\mathbf{A},j}(\mathbf{b})| \leq e^{\alpha \delta n}$, which gives the bound $|\eta|^{-\varepsilon_1/2} \leq |\zeta_{\mathbf{A},j}(\mathbf{b})| \leq |\eta|^{\varepsilon_1/2}$ for $n$ large enough and $\delta$ small enough. The only difficult requirement to check is the \say{non concentration hypothesis} $(*)$ (which is a non-linearity estimate on $f$). In previous work using this strategy to prove Fourier decay, various techniques were used to check this nonconcentration hypothesis. In
\cite{SS20} and \cite{Le21}, this estimate is checked using \emph{Dolgopyat's estimates}, such as found in \cite{Do98}.
In \cite{LNP19}, the non-concentration estimates are checked using regularity estimates for stationary measures of random walks. In the early work of Bourgain and Dyatlov \cite{BD17}, the nonconcentration estimates are checked directly, without the need of any additional technology. To get an elementary conclusion, we will present a new, elementary way to check the nonconcentration estimates that we need, inspired from \cite{BD17}. This strategy works even in a non-perturbative case, but we restricted ourselves to this case for simplicity. \\

Let us define the sets of \textbf{regular} words on which the sum-product phenomenon applies. Our goal will be to prove that they represent most of the words.

\begin{definition}
Fix $\gamma_0:=\gamma_\Phi/100$, where $\gamma_\Phi \in (0,1)$ is given by Remark 2.5.8. Then Theorem 2.4.1 fixes some $k \in \mathbb{N}$ and $\varepsilon_1 \in (0,1)$ (that are associated to $\gamma_0>0$). Fix $\varepsilon_0 := 1/20$ and $\delta$ small enough depending on $\Phi$. We call a block $\mathbf{A}=\mathbf{a_0} \dots \mathbf{a}_k \in \{0,1\}^{(k+1)n}$ regular if for all $j \in \llbracket 1,k\rrbracket$: 
$$\forall \sigma  \in [ e^{- 4 \varepsilon_0 n} , e^{-\varepsilon_0 \varepsilon_1 n/2} ], \quad \# \{ (\mathbf{b},\mathbf{c}) \in (\{0,1\}^{n})^2 , \  |\zeta_{\mathbf{A},j}(\mathbf{b})  - \zeta_{\mathbf{A},j}(\mathbf{c})| \leq \sigma \} \leq 4^n \sigma^{\gamma_0}.$$
Similarly, we call a couple $(\mathbf{a},\mathbf{d}) \in (\{0,1\}^{n})^2$ regular if: 
$$\forall \sigma  \in [ e^{- 4 \varepsilon_0 n} , e^{-\varepsilon_0 \varepsilon_1 n/2} ], \quad  \# \{ \mathbf{b} , \mathbf{c} \in \{0,1\}^{n} , \ | 4^n g_{\mathbf{a} \mathbf{b}}'(x_\mathbf{d}) - 4^n g_{\mathbf{a} \mathbf{c}}'(x_\mathbf{d}) | \leq \sigma \} \leq 4^n \sigma^{\gamma_0}.$$
We will denote the set of regular blocks by $\mathcal{R}^{k+1}_n$, and the set of regular couples $\mathcal{R}_{n}^2$.
\end{definition}
We postpone to the next section the proof of the following
\begin{lemma}
Most blocks are regular:
$$ 2^{-(k+1)n} \# \left( \{0,1\}^{(k+1)n} \setminus \mathcal{R}_n^{k+1} \right) \lesssim \delta^{-1} e^{- \varepsilon_0 \varepsilon_1 \gamma_\Phi n/400} .$$
\end{lemma}
This Lemma allows us to conclude the proof of Theorem 2.1.9: indeed, by Lemma 2.3.1, we already know that
$$ \left| \int_0^1 e^{i \xi \psi(x)} dx \right|^2 \lesssim e^{- \varepsilon_0 \delta_\mu n/4} + 2^{-(k+1)n} \sum_{\mathbf{A} \in \{0,1\}^{(k+1)n}} \sup_{\eta \in [e^{\varepsilon_0 n/2}, e^{2 \varepsilon_0 n}]} 2^{-kn} \left| \sum_{\mathbf{B} \in \{0,1\}^{kn}} e^{ i \eta \zeta_{\mathbf{A},1}(\mathbf{b}_1) \dots \zeta_{\mathbf{A},k}(\mathbf{b}_k) } \right| .$$
Using the previous bound yields:
$$ \left| \int_0^1 e^{i \xi \psi(x)} dx \right|^2 \leq C_{\delta}\left( e^{- \varepsilon_0 \delta_\mu n/4} +  e^{-\varepsilon_0 \varepsilon_1 \gamma_\Phi n/400} \right) $$ $$ + 2^{-{(k+1)n}} \sum_{\mathbf{A} \in \mathcal{R}_n^{k+1}} \sup_{\eta \in [e^{\varepsilon_0 n/2}, e^{2 \varepsilon_0 n}]} 2^{-kn} \left| \sum_{\mathbf{B} \in \{0,1\}^{kn}} e^{ i \eta \zeta_{\mathbf{A},1}(\mathbf{b}_1) \dots \zeta_{\mathbf{A},k}(\mathbf{b}_k) } \right|.$$
We then know that all regular blocks $\mathbf{A}$ produces maps $\zeta_{\mathbf{A},j}$ that satisfies the non concentration hypothesis required to apply Theorem 2.4.1. This gives the exponential bound:
$$  \left| \int_0^1 e^{i \xi \psi(x)} dx \right|^2 \leq C_{\delta}\left( e^{- \varepsilon_0 \delta_\mu n/4} +  e^{-\varepsilon_0 \varepsilon_1 \gamma_\Phi n/400} + e^{-\varepsilon_0 \varepsilon_1 n/2} \right). $$
Notice the following interesting fact: since $\delta_\mu$ approach one as the perturbation gets smaller, we see (in our particular case of a carefully chosen perturbation of the doubling map) that the exponent of decay might be chosen constant in $\delta$ (for $\delta$ small enough depending on $\Phi)$. \\

Recalling that $n := \left\lfloor \frac{(\ln |\xi|)}{ (2k+2) \ln 2 - \varepsilon_0} \right\rfloor$ then gives 
$$ \left| \int_0^1 e^{i \xi \psi(x)} dx \right| \leq C_\delta |\xi|^{-\rho} $$
For some $\rho(\Phi)>0$, constant in $\delta$. The fact that we get a constant $C_\delta \geq C \delta^{-1}$ in front of our power decay is an artefact of our method: the sum-product estimates needs some nonlinearity to holds, and $f_\delta$ is \say{more linear} as $\delta$ approaches zero. This partially answers a question found in \cite{SS20} about the dependence of the exponent $\rho$ on the dynamics.

\section{The non-concentration estimates}

Let us recall the family of perturbations of the doubling map that we work with. For any periodic function $\Phi : \mathbb{S} \rightarrow \mathbb{R}$, we are going to construct a perturbation $f$ of the doubling map so that $\ln f' = c_0 + c_1 \cdot \Phi$.

\begin{definition}

Let $\Phi:\mathbb{R} \rightarrow \mathbb{R}$ be $1$-periodic and $C^{2}$. Suppose that $2 \Phi(0) \neq \Phi(1/3)+\Phi(2/3)$. Then, for $\delta\in (0,1)$, set
$$ \varphi_\delta(x) := z_\delta \int_0^x e^{\delta \Phi(t)} dt. $$
where $$ z_\delta^{-1} := \int_{0}^{1} e^{\delta \Phi(t)} dt .$$

We see that $\varphi_\delta$ is a perturbation of the identity that factors into a $C^{3}$-diffeomorphism of the circle. More precisely, there exists a constant $C>0$ such that $ \| \varphi_\delta - I_d \|_{C^{3}} \leq C \delta. $ Moreover, for all $x$, $ x e^{- 2 \|\Phi\|_\infty \delta} \leq \varphi_\delta(x) \leq x e^{2 \| \Phi\|_\infty \delta} $, and $ e^{-2 \|\Phi\|_\infty \delta} \leq \varphi_\delta'(x) \leq e^{2 \|\Phi\|_\infty \delta}$.
\end{definition}

\begin{definition}

Our perturbation of the doubling map is defined as follows: for some  $\delta > 0$, set $$ f_\delta(x) := \varphi_\delta(2x) .$$
The parameter $\delta$ will be taken small enough at the end of the section. Notice that $\varphi_0(x)=x$ and so $f_0$ is the doubling map. We will omit $\delta$ and write $f,\varphi$ instead of $f_\delta, \varphi_\delta$ for the rest of the section. We define the inverse branches for $a \in \{0,1\}$ by $ g_a(x) := g_a^{(0)}(\varphi^{-1}(x)) $. We define, for a word $\mathbf{a} \in \{0,1\}^n$, $g_{\mathbf{a}} := g_{a_1 \dots a_n}$. Finally, set $S_\mathbf{a} := g_{\mathbf{a}}(\mathbb{S})$.

\end{definition}

\begin{remark}

Notice that $f$ was constructed such that $S_0 = [0,1/2)$ and $S_1 =[1/2,1)$. 
\end{remark}

\begin{lemma}
Let $\mathbf{a} \in \{0,1\}^n$.
Then $$ - \ln g_{\mathbf{a}}' =  n (\ln 2 + \ln z_\delta) + \delta \cdot S_n \Phi \circ g_{\mathbf{a}} ,$$
where $S_n \Phi := \sum_{k=0}^{n-1} \Phi \circ f^k$. In particular, $2^n g_{\mathbf{a}}' \in [ e^{- 2 \delta \|\Phi\|_\infty n}, e^{2 \delta \|\Phi\|_\infty n} ]$, and $\left|\frac{g_{\mathbf{a}}'(x)}{g_{\mathbf{a}}'(y)} - 1\right| \leq 3 \delta \|\Phi\|_{C^1} e^{2 \delta \|\Phi\|_\infty n} |x-y| $.
\end{lemma}

\begin{proof}
We see that
$$ \ln g_{\mathbf{a}}' = - \ln (f^n)'\circ g_{\mathbf{a}} =  -\sum_{k=0}^{n-1} (\ln f') \circ f^k \circ g_{\mathbf{a}} .$$
Moreover, $\ln f'(x) = \ln \varphi'(2x) + \ln 2$, and $\ln \varphi' = \ln z_\delta + \delta \cdot \Phi$. The first estimate is easy since $|\ln z_\delta| \leq \delta \|\Phi\|_\infty$. The second estimate can be checked as follows:

$$ \left|\frac{g_{\mathbf{a}}'(x)}{g_{\mathbf{a}}'(y)} - 1\right|  = \left| e^{\ln g_{\mathbf{a}}'(x) - \ln g_\mathbf{a}'(y) } - 1\right| $$ $$ \leq e^{ 2 n \delta \| \Phi\|_\infty } \left| \ln g_{\mathbf{a}}'(x) - \ln g_\mathbf{a}'(y)\right|  \leq \delta e^{2 n \delta \|\Phi\|_\infty} \sum_{j=1}^n \left| \Phi(g_{a_j \dots a_n} x) -\Phi(g_{a_j \dots a_n} y) \right| $$ $$ \leq \delta \|\Phi\|_{C^1} e^{ 2 \delta \| \Phi \|_\infty n} \sum_{j=1}^n (2/3)^{j} |x-y|  \leq 2 \delta \| \Phi\|_{C^1} e^{2 \delta \|\Phi\|_\infty n} |x-y| .$$ \end{proof}

Before heading into the proof of the non-concentration estimate, we borrow from \cite{AGY06} (Prop 7.4) a discussion on the \say{total nonlinearity condition} (TNL), relating it to a \say{uniform non integrability condition} (UNI) that will be more useful to us. 

\begin{proposition}
Since $\Phi \in C^2$, the following are equivalent:
\begin{enumerate}

\item $f$ satisfies (TNL).
\item There exists no locally constant map $K : S_0 \sqcup S_1 \rightarrow \mathbb{R}$ and no $C^1$ with bounded derivative $\theta : S_0 \sqcup S_1 \rightarrow \mathbb{R}$ such that
$$ \Phi = K + \theta \circ f - \theta .$$
\item (UNI 1) There exists $c_0 \in (0,1)$ such that for all $N \geq 0$, there exists $n \geq N$ and $\mathbf{a},\mathbf{b} \in \{0,1\}^n$, such that $$ \exists x \in \mathbb{S}, \ |(S_n \Phi \circ g_{\mathbf{a}}-S_n \Phi \circ g_{\mathbf{b}})'(x)| > c_0 $$
\item (UNI 2) There exists $c_0 \in (0,1)$ such that for all $N \geq 0$, there exists $n \geq N$ and $\mathbf{a},\mathbf{b} \in \{0,1\}^n$, such that $$ \forall x \in \mathbb{S}, \ |(S_n \Phi \circ g_{\mathbf{a}}-S_n \Phi \circ g_{\mathbf{b}})'(x)| > c_0 $$

\end{enumerate}

\end{proposition}

\begin{proof}

The fact that point 2) is equivalent to (TNL) is obvious from the definitions. \\

Let us prove  (TNL) $\Longleftrightarrow$ (UNI 1) . Suppose that there exists a locally constant $K$ and a $C^{1}$ with bounded derivative $\theta$ such that $\Phi = K + \theta \circ f - \theta$. Then $S_n \Phi \circ g_{\mathbf{a}} = S_n K + \theta - \theta \circ g_{\mathbf{a}}$. In particular, for any $x$, we get:
$$  |(S_n \Phi \circ g_\mathbf{a} - S_n \Phi \circ g_\mathbf{b})'(x) | \leq 2 |\theta'|_\infty \kappa_+^n $$
where $\kappa_+ \in (0,1)$ is such that $g_a'(x)<\kappa_+$ for all $x$ and $a \in \{0,1\}$ (we can choose $\kappa_+ = 2^{-1}e^{O(\delta)}$, or even $\kappa_+ := 2^{-1/2}$ to remove the dependency on $\delta$). This contradicts (UNI 1) and hence proves (UNI 1) $\Rightarrow$ (TNL). \\

Now, suppose that (UNI 1) doesn't hold. This means:
$$ \forall \varepsilon > 0, \exists N \geq 0, \forall n \geq N, \forall \mathbf{a},\mathbf{b} \in \{0,1\}^n, \forall x \in \mathbb{S}, \ |(S_n \Phi \circ g_{\mathbf{a}}-S_n \Phi \circ g_{\mathbf{b}})'(x)| < \varepsilon .$$
The point is that, in this case, a family of natural maps $(X_\mathbf{a})_{\mathbf{a}}$, indexed on $\mathbf{a} = \dots a_n a_{n-1} \dots a_2 a_1 \in \{0,1\}^\infty$, are going to all be equal to each other. For such an infinite word $\mathbf{a}$, define (on $\mathbb{S} \setminus \{0\}$):
$$ X_\mathbf{a}(x) := \sum_{i=1}^\infty \Phi'(g_{a_i \dots a_1} x) g_{a_i \dots a_1}'(x) = \lim_{n \rightarrow \infty} ( S_n \Phi \circ g_{a_n \dots a_1} )'(x) .$$
It is then easy to see that our condition yields $$ \forall \mathbf{a},\mathbf{b} \in \{0,1\}^\infty, X_\mathbf{a} = X_\mathbf{b} .$$
Moreover, $X := X_\mathbf{a}$ is $C^{0}$ (and bounded by $ |\Phi'|_\infty \kappa_+/(1-\kappa_+)$). We are going to use $X$ to construct a map $\theta$ that will allow us to contradict (TNL). Let $x \in S_b$ for some $b \in \{0,1\}$. Let $\mathbf{a} \in \{0,1\}^\infty$. Notice that:
$$ X \circ f(x) = X_{\mathbf{a} b}(f(x)) = \Phi'(g_b f(x)) g_b'(f(x)) + \sum_{i=1}^\infty \Phi'( g_{a_i \dots a_1 b} f(x)) g_{a_i \dots a_1 b}'(f(x)) $$
$$ = \frac{1}{f'(x)} \cdot \left(\Phi'(x) + \sum_{i=1}^\infty \Phi'(g_{a_i \dots a_1}(x)) g_{a_i \dots a_1}'(x) \right)  $$
$$ = \frac{1}{f'(x)} \cdot \left(\Phi'(x) + X_\mathbf{a}(x) \right) ,$$
and hence,
$$ \Phi' + X - (X \circ f) \cdot f' =0 .$$
We can conclude by integrating this relation. (Notice that since $X$ is defined on $\mathbb{S} \setminus \{0\}$, $X \circ f$ is defined on $S_0 \sqcup S_1$. So, in a more general setting, this argument indeed gives us a locally constant map, and not a constant map.) It is worth noticing that $X$ can be integrated in an explicit $\theta$. For example, fix some $x_0 \in \mathbb{S} \setminus \{0\}$, and let for $x \in \mathbb{S} \setminus \{0\}$:
$$ \theta(x) := \sum_{i=1}^\infty \left( \Phi(g_{a_i \dots a_1} x) -  \Phi(g_{a_i \dots a_1} x_0)\right) .$$
This map is $C^{1}$ and satisfies $\theta'=X$. In particular, $\Phi + \theta - \theta \circ f $ is locally constant, and we have proved (TNL) $\Rightarrow$ (UNI 1). \\

We conclude by showing (UNI 1) $\Rightarrow$ (UNI 2) in our smooth case. (The fact that (UNI 2) $\Rightarrow$ (UNI 1) is obvious.) Suppose that there exists $c_0$ and $N$ such that, for all $n \geq N$, there exists $\mathbf{a},\mathbf{b} \in \{0,1\}^n$, and there exists some $x_0$ such that $|(S_n \Phi \circ g_{\mathbf{a}}-S_n \Phi \circ g_{\mathbf{b}})'(x_0)| > c_0$. Fix $n_0 \geq N$ large enough so that $ 2 |\Phi'|_{\infty} \kappa^{n_0+1}_+ (1-\kappa_+)^{-1} < c_0/4$. Let $\mathbf{a},\mathbf{b} \in \{0,1\}^{n_0}$ and $x_0$ be given by (UNI 1). {Since $\Phi$ is $C^1$, the estimate holds on some neighborhood of $x_0$, replacing $c_0$ by $c_0/2$}. More precisely, there exists $n_1 \geq 0$ and $\mathbf{c} \in \{0,1\}^{n_1}$ such that:
$$ \forall x \in \mathbb{S}, \ |(S_{n_0} \Phi \circ g_{\mathbf{a}}-S_{n_0} \Phi \circ g_{\mathbf{b}})'(g_\mathbf{c} x)| > c_0/2 .$$
Notice that
$$ (S_{n_0+n_1} \Phi \circ g_{\mathbf{a} \mathbf{c}})'(x) = (S_{n_0} \Phi \circ g_{\mathbf{a}})'(g_\mathbf{c} x) g_\mathbf{c}'(x) + (S_{n_1} \Phi \circ g_{\mathbf{c}})'(x) ,$$
and hence 
$$ \forall x, \ | (S_{n_0+n_1} \Phi \circ g_{\mathbf{a} \mathbf{c}})'(x) -  (S_{n_0+n_1} \Phi \circ g_{\mathbf{b} \mathbf{c}})'(x)| > c_0 g_{\mathbf{c}}'(x)/2 .$$

This finally allows us to conclude. Set $\tilde{N} := n_0+n_1$ and set $\tilde{c}_0 := c_0 \kappa_-^{n_1}/ 4$, where $\kappa_- \in (0,1)$ is such that $\kappa_- < g_{a}'(x)$ for all $a \in \{0,1\}$ and for all $x$. (We can choose $\kappa_- = 2^{-1} e^{O(\delta)}$.) We are going to check (UNI 2) for those constants. \\

Let $n \geq \tilde{N}$. Choose any $\widehat{\mathbf{a}},\widehat{\mathbf{b}} \in \{0,1\}^{n-\tilde{N}}$ and fix $\mathbf{A} := \widehat{\mathbf{a}} {\mathbf{a}} {\mathbf{c}} \in \{0,1\}^n$ and $\mathbf{B} := \widehat{\mathbf{b}} {\mathbf{b}} {\mathbf{c}} \in \{0,1\}^n$, where $\mathbf{a},\mathbf{b},\mathbf{c}$ are like earlier. We get, for all $x$:
$$ |(S_{n} \Phi \circ g_{\mathbf{A}} -  S_{n} \Phi \circ g_{\mathbf{B}})'(x)| $$ $$= |(S_{n-\tilde{N}} \Phi \circ g_{\widehat{\mathbf{a}}})'(g_{\mathbf{ac}} x) g_{\mathbf{ac}}'(x) - (S_{n-\tilde{N}} \Phi \circ g_{\widehat{\mathbf{b}}})'( g_{\mathbf{b c}} x ) g_{\mathbf{bc}}'(x) + (S_{\tilde{N}} \Phi \circ g_{\mathbf{a} \mathbf{c}} - S_{\tilde{N}} \Phi \circ g_{\mathbf{b} \mathbf{c}})'(x) | $$
$$ \geq \frac{c_0 g_\mathbf{c}'(x)}{2} - \frac{2|\Phi'|_\infty \kappa_+}{1-\kappa^+} \kappa_+^{n_0} g_{\mathbf{c}}'(x)  > c_0 \kappa_-^{n_1} /4 .$$
\end{proof}

Now we understand how one can transform this (TNL) hypothesis into a condition involving Birkhoffs sums of $\Phi$. We still need to show that, in our context, these conditions are satisfied. We will also need to strenghten a bit the argument to show that, in our context, (UNI2) is satisfied for any $\delta>0$ small enough with constants $c_0$ and $N$ that can be chosen constant in $\delta$.

\begin{lemma}
Since $2\Phi(0) \neq \Phi(1/3)+\Phi(2/3)$, there exists $\delta_0(\Phi)\in (0,1)$, $c_0(\Phi) > 0$, $N(\Phi) \geq 1$ such that, for any $\delta \in (0,\delta_0)$, for any $n \geq N$, there exists $\mathbf{a},\mathbf{b} \in \{0,1\}^n$ such that  
$$ \forall x \in \mathbb{S}, \  |( S_n \Phi \circ g_{\mathbf{a}} - S_n \Phi \circ g_{\mathbf{b}}  )'(x)| > c_0. $$
\end{lemma}

\begin{proof}

First of all, let us show that $\Phi$ is not cohomologous to a locally contant map for the doubling map $f_0$. By contradiction, assume that there exists $\kappa : S_0 \sqcup S_1 \rightarrow \mathbb{R}$ locally constant, and $\theta : S_0 \sqcup S_1 \rightarrow \mathbb{R}$ $C^1$ with bounded derivative such that
$$ \Phi = \kappa + \theta \circ {f_0} - \theta. $$
We already saw (see Lemma 2.1.8) that this implies that $\kappa$ is a constant map. Now, since $f_0(0)=0$ and $f_0^2(1/3)=f_0(2/3)=1/3$, we get in this case:
$$ \Phi(0) = \kappa = \frac{1}{2}( \Phi(1/3) + \Phi(2/3)) $$
which is not possible by our choice of $\Phi$. It follows that $\Phi$ satisfies $(UNI2)$ when $\delta=0$: there exists $N \geq 0$ and $c_0>0$ such that, for every $n \geq N$, there exists $\mathbf{a},\mathbf{b} \in \{0,1\}^{n}$ such that
$$ \forall x \in \mathbb{S}, \ \Big| \sum_{i=1}^n \Phi'(g_{a_i \dots a_1}^{(0)}(x)) (g_{a_i \dots a_1}^{(0)})'(x) - \sum_{i=1}^n \Phi'(g_{b_i \dots b_1}^{(0)}(x)) (g_{b_i \dots b_1}^{(0)})'(x) \Big| \geq c_0 .$$
We let $n$ go to infinity and we find two infinite words $\mathbf{a},\mathbf{b} \in \{0,1\}^{\infty}$ such that
$$ \forall x \in \mathbb{S}, \ |X_\mathbf{a}^{(0)}(x) - X_\mathbf{b}^{(0)}(x)| > c_0 $$
where $X_\mathbf{a}^{(\delta)}(x) = \sum_{i=1}^\infty \Phi'(g_{a_i \dots a_1} x) g_{a_i \dots a_1}'(x)$. This quantity can be seen as a function of $(x,\delta) \in \mathbb{S} \times [0,1)$. One check easily that this function is uniformly continuous on a set $(x,\delta) \in \mathbb{S} \times [0,\delta_0]$ by normal convergence of this series. Hence, choosing $\delta_0$ small enough, one can ensure the following estimate:
$$ \forall \delta \in [0,\delta_0], \forall x \in \mathbb{S}, \ |X_\mathbf{a}^{(\delta)}(x) - X_{\mathbf{a}}^{(\delta)}(y)| > c_0/2. $$
We conclude  by cutting the series for a large enough $N$ (that can be chosen uniform in $\delta$), following the end of the proof of our previous lemma. \end{proof}

Now, we know that our choice of $\Phi$ implies a nonlinearity hypothesis on $f$, which can itself be rewritten into a condition involving derivatives of Birkhoff sums involving $\Phi$. Our goal is to use this condition to get our nonconcentration estimates. The idea of the proof of Lemma 2.4.3 can be decomposed in two main steps.
\begin{enumerate}
    \item The nonconcentration hypothesis can be rewritten in terms of a non-concentration estimate on Birkhoff sums involving $\Phi$, namely $S_n \Phi \circ g_{\mathbf{a}}$.
    \item To check that those Birkhoff sums doesn't concentrate too much, we show that the derivatives $\left(S_n \Phi \circ g_{\mathbf{a}} - S_n \Phi \circ g_{\mathbf{b}} \right)'$ are often away from zero.
\end{enumerate}

We begin by step 2. To this end, we prove a useful lemma, which is a consequence of Lemma 2.5.6. This looks like a dimension estimate for a Cantor set. The idea behind the proof is the following: picture an autosimilar tree, with leaves that looks like a Cantor set. Fix a small interval. Our goal is to count the number of leaves that encounter this interval. We begin by saying that we can cut down some of the branches at the first level. Then, we can conclude by induction on the tree structure, by seeing our big tree as an union of smaller trees. \\

In what follows, the \say{tree} is given by the map $ \mathbf{a} \in \{0,1\}^n \mapsto \sum_{i=1}^n \Phi'(g_{a_i \dots a_1} x_0) g_{a_i \dots a_1}'(x_0) \in \mathbb{R}$, and the small interval is $[a-\sigma,a+\sigma]$. 

\begin{lemma}[Tree lemma]

Let $n \in \mathbb{N}$. Let $\sigma > 0$. Then, for all $a \in \mathbb{R}$, and for all $x_0 \in \mathbb{S}$:

$$ \frac{1}{2^n}  \# \{ \mathbf{a} \in \{0,1\}^n  | \  (S_n \Phi \circ g_\mathbf{a})'(x_0) \in [a-\sigma,a+\sigma] \} \leq   \left(\frac{4}{c_0}\right)^{\gamma} \cdot \sigma^\gamma + \kappa_-^{-\gamma N} \cdot \kappa_-^{\gamma n} $$
where $c_0$ and $N$ are fixed constants given by Lemma 2.5.6, $ \gamma := \frac{1}{N} \frac{\ln(1-2^{-N})}{\ln \kappa_-},$ and $\kappa_- := 1/4 < \min_{a \in \{0,1\}} \inf_{x \in \mathbb{S}} g_a'(x)$.
\end{lemma}

\begin{proof}
The proof goes by induction on $n$. Fix $c_0>0$ and $N$ large enough so that Lemma 2.5.6 holds for those constants, and such that $$ 2 |\Phi'|_\infty \kappa_+^{N+1} (1-\kappa_+)^{-1} \leq c_0/4, $$
where $\kappa_+ := \sqrt{2} > \max_{a \in \{0,1\}} \sup_x g_a'(x)$.  If $n \leq N$, then the bound follows from $1 \leq \kappa_-^{-\gamma N} \kappa_-^{\gamma n}$. So let $n \geq N+1$, and suppose that Lemma 2.5.7 holds for all $k<n$. If $\sigma \geq c_0/4$, then the bound holds since $1 \leq \left(4 \sigma /c_0\right)^\gamma$. So suppose that $\sigma < c_0/4$.  \\

\underline{Step one}: We use UNI to get a bit of decay. (\say{Cutting a branch}) \\

We cut the words $\mathbf{a} \in \{0,1\}^n$ in two parts: $\mathbf{a} = \tilde{\mathbf{a}} \widehat{\mathbf{a}}$, with $\tilde{\mathbf{a}} \in \{0,1\}^{n-N}$ and $\widehat{\mathbf{a}} \in \{0,1\}^N$. Our quantity of interest can then be rewritten: 
$$ \frac{1}{2^n} \# \{ \mathbf{a} \in \{0,1\}^n, \ (S_n \Phi \circ g_\mathbf{a})'(x_0) \in [a-\sigma, a+\sigma] \} $$ $$ = \frac{1}{2^N}  \sum_{\widehat{\mathbf{a}} \in \{0,1\}^N } \frac{1}{2^{n-N}} \# \{ \tilde{\mathbf{a}} \in \{0,1\}^{n-N}, \ (S_n \Phi \circ g_\mathbf{a})'(x_0) \in [a-\sigma,a+\sigma] \} .$$
Now, (UNI2) allows us to see that there exists $\widehat{\mathbf{a}} \in \{0,1\}^N$ such that for all $\tilde{\mathbf{a}}$, $(S_n \Phi \circ g_\mathbf{a})'(x_0) \notin [a-\sigma,a+\sigma]$. Indeed, there exists $\widehat{\mathbf{a}},\widehat{\mathbf{b}} \in \{0,1\}^N$ such that
$$ (S_N \Phi \circ g_{\widehat{\mathbf{a}}} - S_N \Phi \circ g_{\widehat{\mathbf{b}}})'(x_0) > c_0 .$$
Hence, we know that one of the words $\widehat{\mathbf{a}}$ or $\widehat{\mathbf{b}}$, call it $\widehat{\mathbf{a}}^{(0)}$, satisfy
$$ | (S_N \Phi \circ g_{\widehat{\mathbf{a}}^{(0)}})'(x_0) - a | > c_0/2 .$$
Hence, for all  $\tilde{\mathbf{a}} \in \{0,1\}^{n-N}$,
$$ |(S_n \Phi \circ g_{\tilde{\mathbf{a}}\widehat{\mathbf{a}}^{(0)}})'(x_0) - a| > c_0/2 - 2 |\Phi'|_\infty \kappa_+^{N+1}(1-\kappa_+)^{-1} > c_0/4 ,$$
and so, since $\sigma< c_0/4$: $$ \forall \tilde{\mathbf{a}} \in \{0,1\}^{n-N}, \ (S_n \Phi \circ g_{ \tilde{\mathbf{a}} \widehat{\mathbf{a}}^{(0)}})'(x_0) \notin [a-\sigma,a+\sigma] .$$
In particular, we have:
$$\frac{1}{2^n} \# \{ \mathbf{a} \in \{0,1\}^n, \ (S_n \Phi \circ g_\mathbf{a})'(x_0) \in [a-\sigma, a+\sigma] \}$$
$$ = \frac{1}{2^N}  \underset{\widehat{\mathbf{a}} \neq \widehat{\mathbf{a}}^{(0)}}{\sum_{\widehat{\mathbf{a}} \in \{0,1\}^N }} \ \frac{1}{2^{n-N}} \# \Big\{ \tilde{\mathbf{a}} \in \{0,1\}^{n-N}, \ (S_n \Phi \circ g_\mathbf{a})'(x_0) \in [a-\sigma,a+\sigma] \Big\}. $$

\underline{Step two:} We use our induction hypothesis by taking advantage of the autosimilarity of $S_n \Phi \circ g_\mathbf{a}$. (\say{induction on the tree structure + rescaling}) \\

Notice that $$ (S_n \Phi \circ g_\mathbf{a})'(x_0) = (S_{n-N} \Phi \circ g_{\tilde{\mathbf{a}}})'(g_{\widehat{\mathbf{a}}} x_0) g_{\widehat{\mathbf{a}}}'(x_0) + (S_N \Phi \circ g_{\widehat{\mathbf{a}}})'(x_0) .$$
The condition $ (S_n \Phi \circ g_\mathbf{a})'(x_0) \in [a-\sigma,a+\sigma] $ can then be rewritten (fixing some $\widehat{\mathbf{a}}$):
$$ (S_{n-N} \Phi \circ g_{\tilde{\mathbf{a}}})'(y_0) \in [ \tilde{a} - \tilde{\sigma}, \tilde{a} - \tilde{\sigma} ] ,$$
where $y_0 := g_{\widehat{\mathbf{a}}}(x_0)$, 
$ \tilde{a} := ({a - S_N \tau \circ g_{\widehat{\mathbf{a}}}(x_0)})/{g_{\widehat{\mathbf{a}}}'(x_0)} $, and $\tilde{\sigma} := {\sigma}/{g_{\widehat{\mathbf{a}}}'(x_0)} $.
In particular, the induction hypothesis applies, and gives:
$$ \forall \widehat{\mathbf{a}} \in \{0,1\}^N, \frac{1}{2^{n-N}} \# \Big\{ \tilde{\mathbf{a}} \in \{0,1\}^{n-N}, \ (S_n \Phi \circ g_\mathbf{a})'(x_0) \in [a-\sigma,a+\sigma] \Big\}  \leq \left(\frac{4}{c_0}\right)^\gamma \ \tilde{\sigma}^\gamma \ + \ \kappa_-^{-\gamma N} \ \kappa_-^{\gamma (n-N)} .$$

We then inject this estimate into the sum at the end of step 1, use $\tilde{\sigma} \leq \sigma \kappa_-^{-N}$, and we get:

$$ \frac{1}{2^n} \# \{ \mathbf{a} \in \{0,1\}^n, \ (S_n \Phi \circ g_\mathbf{a})'(x_0) \in [a-\sigma,a+\sigma] \} $$
$$ \leq \left( 1-\frac{1}{2^N}\right) \kappa_-^{-N \gamma} \left(\left(\frac{4}{c_0}\right)^\gamma \cdot {\sigma}^\gamma + \kappa_-^{-\gamma N} \cdot \kappa^{\gamma n} \right) $$
$$ = \left(\frac{4}{c_0}\right)^\gamma \cdot {\sigma}^\gamma + \kappa_-^{-\gamma N} \cdot \kappa_-^{\gamma n} ,$$
since $$ \Big( 1 - \frac{1}{2^N} \Big) \kappa_-^{-N \gamma} =1 $$
by definition of $\gamma>0$. The proof is done.
\end{proof}

\begin{remark}
Notice that there exists $\delta_\Phi > 0$ such that, if $\delta \in (0,\delta_{\Phi})$, the decay rate $\gamma$ obtained previously is larger than $\gamma_\Phi := 2^{-N}/(10N) \in (0,1)$. Hence, if $\delta$ is small enough (depending on $\Phi$), we can suppose that the previous $\gamma$ gives a uniform bound independent of $\delta$. The previous lemma can be adapted to work for any equilibrium state (not only for the measure of maximal entropy), and also works even in a nonperturbative setting. (Our choice of setting was made to simplify the exposition in this chapter.)
\end{remark}

\begin{lemma}
There exists $C(\Phi) > 0$ such that, for any $n \geq 0$ and $\mathbf{a} \in \{0,1\}^n$, $$\|(S_n \Phi \circ g_\mathbf{a} )'' \|_\infty \leq C(\Phi).$$
\end{lemma}

\begin{proof}
Recall that $\Phi$ was supposed to be $C^2$. In particular, $(S_n \Phi \circ g_\mathbf{a})'$ is $C^1$. Moreover,
$$ (S_n \Phi \circ g_\mathbf{a})^{''}(x)$$ $$ =  \sum_{i=1}^n \Phi''(g_{a_i \dots a_1}(x)) g_{a_i \dots a_1}'(x)^2 + \sum_{i=1}^n \Phi'(g_{a_i \dots a_1}(x)) \left(\sum_{j=1}^i (\ln g_{a_j}')'(g_{a_{i-1} \dots a_1}(x)) g_{a_{i-1} \dots a_1}'(x) \right) g_{a_i \dots a_1}'(x) $$
satisfies $$ \|(S_n \Phi \circ g_\mathbf{a})''\|_\infty \leq \sum_{i} \|\Phi''\|_\infty 2^{-i/2} + \sum_i \sum_{j<i} \delta \|\Phi'\|_\infty^2 2^{-j/2} 2^{-i/2} \leq C(\Phi) < \infty$$
for some constant $C(\Phi) \geq 1$ that only depends on the $C^2$-norm on $\Phi$. 
\end{proof}

\begin{lemma}
Let $n$ be large enough. Let $\sigma \in [e^{-5 \varepsilon_0 n},\delta^{-1} e^{-\varepsilon_0 \varepsilon_1 n/3} ]$. Define $ \tilde{n} := \left\lfloor (\log_2 \sigma)/2 \right\rfloor $. This is a slowly increasing zoom factor, scaled so that $2^{-\tilde{n}} e^{2 \delta n} \simeq \sqrt{\sigma} e^{2 \delta n} \leq \sigma^{1/10}$. Fix any word $\mathbf{a} \in \{0,1\}^n$. The following bound holds:
$$ 2^{-2n-\tilde{n}}\#\{ (\mathbf{b},\mathbf{c}) \in (\{ 0,1 \}^{n})^2, \mathbf{d} \in \{0,1\}^{\tilde{n}}| \exists x \in S_{\mathbf{d}}, \left|\left(S_{2n} \Phi \circ g_{\mathbf{a} \mathbf{b}} - S_{2n} \Phi \circ g_{\mathbf{a} \mathbf{c}}\right)'(x)\right| \leq \sigma^{1/10} \} \leq \delta^{-1/2} \sigma^{\gamma_\Phi/50} . $$

\end{lemma}

\begin{proof}

We are going to reduce our bound to the previous lemma. Notice first that, for fixed words $\mathbf{a},\mathbf{b}$ and $\mathbf{c}$, we can compute the derivative of $S_{2n} \Phi \circ g_{\mathbf{a} \mathbf{b}} - S_{2n} \Phi \circ g_{\mathbf{a} \mathbf{c}} $ as follow:
$$ \left(S_{2n} \Phi \circ g_{\mathbf{a} \mathbf{b}} - S_{2n} \Phi \circ g_{\mathbf{a} \mathbf{c}}\right)' = g_{\mathbf{b}}' \left(S_n \Phi \circ g_{\mathbf{a}}\right)' \circ g_\mathbf{b} + \left( S_n \Phi \circ g_{\mathbf{b}} \right)' - g_{\mathbf{c}}' \left(S_n \circ g_{\mathbf{a}}\right)' \circ g_\mathbf{c} - \left( S_n \Phi \circ g_{\mathbf{c}} \right)'. $$
We see that the terms involving $\mathbf{a}$ becomes negligible. Indeed, $ \left| \left(S_n \Phi \circ g_{\mathbf{a}} \right)' \right|  \leq 2 \|\Phi'\|_\infty $, and $|g_{\mathbf{b}}'|, |g_{\mathbf{c}}'| \leq 2^{-n} e^{2 \delta |\Phi|_\infty n}$, so that
$$ \left| g_{\mathbf{b}}' \left(S_n \circ g_{\mathbf{a}}\right)' \circ g_\mathbf{b} - g_{\mathbf{c}}' \left(S_n \circ g_{\mathbf{a}}\right)' \circ g_\mathbf{c} \right| \leq 4 \|\Phi'\|_\infty \cdot 2^{-n} e^{2 \delta |\Phi|_\infty n} \leq \sigma^{1/10} $$
for $n$ large enough, since $\delta$ is small enough and $\varepsilon_0<1$.
Hence, if $ |\left(S_{2n} \Phi \circ g_{\mathbf{a} \mathbf{b}} - S_{2n} \Phi \circ g_{\mathbf{a} \mathbf{c}}\right)'|  \leq \sigma^{1/10} $, then $\left| \left( S_n \Phi \circ g_{\mathbf{b}} - S_n \Phi \circ g_{\mathbf{c}} \right)' \right| \leq 2\sigma^{1/10} $, and it follows that $$ 2^{-2n-\tilde{n}}\#\{ (\mathbf{b},\mathbf{c}) \in (\{ 0,1 \}^{n})^2, \mathbf{d} \in \{0,1\}^{\tilde{n}} \ |  \ \exists x \in S_{\mathbf{d}}, \ \left|\left(S_{2n} \Phi \circ g_{\mathbf{a} \mathbf{b}} - S_{2n} \Phi \circ g_{\mathbf{a} \mathbf{c}}\right)'(x)\right| \leq \sigma^{1/10} \}  $$
$$ \leq  2^{-2n-\tilde{n}}\#\{ (\mathbf{b},\mathbf{c}) \in (\{ 0,1 \}^{n})^2, \mathbf{d} \in \{0,1\}^{\tilde{n}} \ |  \ \exists x \in S_{\mathbf{d}}, \ \left|\left(S_{n} \Phi \circ g_{\mathbf{b}} - S_{n} \Phi \circ g_{\mathbf{c}}\right)'(x)\right| \leq 2\sigma^{1/10} \}. $$
$$ = 2^{-{\tilde{n}}-n} \underset{\mathbf{c} \in \{0,1\}^n}{\sum_{\mathbf{d} \in \{0,1\}^{\tilde{n}}}} 2^{-n} \#\left\{ \mathbf{b} \in \{ 0,1 \}^n \ | \ \exists x \in S_{\mathbf{d}}, \ (S_n \Phi \circ g_\mathbf{b})'(x) \in [a_\mathbf{c}(x) - 2 \sigma^{1/10}, a_{\mathbf{c}}(x) + 2 \sigma^{1/10}] \right\} $$
with $ a_\mathbf{c}(x) := (S_n \Phi \circ g_\mathbf{c})'(x)$. Now, the previous lemma tells us that $|a_\mathbf{c}(x)-a_\mathbf{c}(x_\mathbf{d})| \leq C(\Phi) \sigma^{1/10}$, and similarly for $a_\mathbf{b}$. It follows that:
$$ 2^{-{\tilde{n}}-n} \underset{\mathbf{c} \in \{0,1\}^n}{\sum_{\mathbf{d} \in \{0,1\}^{\tilde{n}}}} 2^{-n} \#\left\{ \mathbf{b} \in \{ 0,1 \}^n \ | \ \exists x \in S_{\mathbf{d}}, \ (S_n \Phi \circ g_\mathbf{b})'(x) \in [a_\mathbf{c}(x) - 2 \sigma^{1/10}, a_{\mathbf{c}}(x) + 2 \sigma^{1/10}] \right\} $$
$$ \leq 2^{-{\tilde{n}}-n} \underset{\mathbf{c} \in \{0,1\}^n}{\sum_{\mathbf{d} \in \{0,1\}^{\tilde{n}}}} 2^{-n} \#\left\{ \mathbf{b} \in \{ 0,1 \}^n \ | \  \ (S_n \Phi \circ g_\mathbf{b})'(x_\mathbf{d}) \in [a_\mathbf{c}(x_\mathbf{d}) - 4 C(\Phi) \sigma^{1/10}, a_{\mathbf{c}}(x_{\mathbf{d}}) + 4 C(\Phi) \sigma^{1/10}] \right\}. $$
Now, the \say{tree Lemma} 2.5.7  gives:
$$2^{-2n-\tilde{n}}\#\{ (\mathbf{b},\mathbf{c}) \in (\{ 0,1 \}^{n})^2, \mathbf{d} \in \{0,1\}^{\tilde{n}}| \exists x \in S_{\mathbf{d}}, \left|\left(S_{2n} \Phi \circ g_{\mathbf{a} \mathbf{b}} - S_{2n} \Phi \circ g_{\mathbf{a} \mathbf{c}}\right)'(x)\right| \leq \sigma^{1/10} \}$$
$$ \leq \Big(\frac{16 \ C(\Phi)}{c_0}\Big)^{\gamma_\Phi} \sigma^{\gamma_\Phi/10} + \kappa_-^{\gamma_\Phi N } \kappa_-^{\gamma_\Phi n} \leq \delta^{-1/2} \sigma^{\gamma_\Phi/50}$$
if $n$ is large enough, and since $\delta$ is small enough in front of $\varepsilon_0$ and $\varepsilon_1$. \end{proof}

\begin{lemma}
Let $n$ be large enough. Let $\sigma \in [e^{-5 \varepsilon_0 n},\delta^{-1} e^{-\varepsilon_0 \varepsilon_1 n/3} ]$. Let $\mathbf{a} \in \{0,1\}^n$.  Then:
$$ 8^{-n} \# \left\{ (\mathbf{b}, \mathbf{c},\mathbf{d}) \in \left(\{0,1\}^n\right)^3 , \  |S_{2n} \Phi \circ g_{\mathbf{a} \mathbf{b}}(x_{\mathbf{d}}) - S_{2n} \Phi \circ g_{\mathbf{a} \mathbf{c}}(x_{\mathbf{d}}) | \leq \sigma \right\} \leq  2 \delta^{-1/2} \sigma^{\gamma_\Phi/50} $$
\end{lemma}

\begin{proof}
 We cut the word $\mathbf{d}$ in two : $\mathbf{d} := \tilde{\mathbf{d}} \widehat{\mathbf{d}}$, with $\tilde{\mathbf{d}} \in \{0,1\}^{\tilde{n}}$ and $\widehat{\mathbf{d}} \in \{0,1\}^{n-\tilde{n}}$, where $\tilde{n} := \lfloor (\log_2 \sigma)/2 \rfloor$. The desired cardinal becomes
$$ 8^{-n} \#\left\{ (\mathbf{b},\mathbf{c},\tilde{\mathbf{d}},\widehat{\mathbf{d}}) \in (\{0,1\}^n)^2 \times \{0,1\}^{\tilde{n}} \times \{0,1\}^{n-\tilde{n}}, \ |S_{2n} \Phi \circ g_{\mathbf{a} \mathbf{b}}(x_{\tilde{\mathbf{d}} \widehat{\mathbf{d}}}) - S_{2n} \Phi \circ g_{\mathbf{a} \mathbf{c}}(x_{\tilde{\mathbf{d}} \widehat{\mathbf{d}}})|\leq \sigma \right\} .$$

From there, the strategy is taken from \cite{BD17}: we argue that for most of the words $\mathbf{b},\mathbf{c},\tilde{\mathbf{d}}$, the derivative of the inner function is large enough, thus spreading the $x_{\tilde{\mathbf{d}}\widehat{\mathbf{d}}}$.
Indeed, if we denote by $D_n(\sigma^{1/10})$ the set of all $(\mathbf{b},\mathbf{c},\tilde{\mathbf{d}})$ for which there exists $x \in \mathbb{S}_{\tilde{\mathbf{d}}}$ such that $$ |\left(S_{2n} \Phi \circ g_{\mathbf{a} \mathbf{b}} - S_{2n} \Phi \circ g_{\mathbf{a} \mathbf{c}}\right)'(x)| \leq \sigma^{1/10} ,$$ then the previous lemma bounds $2^{-2n-\tilde{n}} \# D_n(\sigma^{1/10})$, and we can write:

$$  8^{-n} \#\left\{ (\mathbf{b},\mathbf{c},\tilde{\mathbf{d}},\widehat{\mathbf{d}}), \ |S_{2n} \Phi \circ g_{\mathbf{a} \mathbf{b}}(x_{\tilde{\mathbf{d}} \widehat{\mathbf{d}}}) - S_{2n} \Phi \circ g_{\mathbf{a} \mathbf{c}}(x_{\tilde{\mathbf{d}} \widehat{\mathbf{d}}})|\leq \sigma \right\} . $$
$$ \leq 8^{-n} \# \left\{ (\mathbf{b}, \mathbf{c}, \tilde{\mathbf{d}}, \widehat{\mathbf{d}}) \ | \ (\mathbf{b},\mathbf{c},\tilde{\mathbf{d}}) \notin D_n(\sigma^{1/10}), \  |S_{2n} \Phi \circ g_{\mathbf{a} \mathbf{b}}(x_{\tilde{\mathbf{d}} \widehat{\mathbf{d}}}) - S_{2n} \Phi \circ g_{\mathbf{a} \mathbf{c}}(x_{\tilde{\mathbf{d}} \widehat{\mathbf{d}}})|\leq \sigma \right\} + \delta^{-1/2} \sigma^{\gamma_\Phi/50} .$$

Now, $(\mathbf{b},\mathbf{c},\tilde{\mathbf{d}}) \notin D_n(\sigma^{1/10})$ means that $$ \inf_{S_{\tilde{\mathbf{d}}}} \left| \left( S_{2n} \Phi \circ g_{\mathbf{a} \mathbf{b}} - S_{2n} \Phi \circ g_{\mathbf{a} \mathbf{c}} \right)' \right| \geq \sigma^{1/10} .$$
It is elementary to check that for any absolutely continuous map $f : I \rightarrow \mathbb{R}$ satisfying $\inf_I f' > 0$, we have, for any interval $J$, $\text{diam}(f^{-1}(J)) \leq (\inf_I f')^{-1} \text{diam}(J)$. Hence, if $(\mathbf{b},\mathbf{c},\tilde{\mathbf{d}}) \notin D_n(\sigma^{1/10})$, denoting by $I_{\mathbf{a},\mathbf{b},\mathbf{c},\tilde{\mathbf{d}}}(\sigma) := \{ x \in \mathbb{S}_{\tilde{\mathbf{d}}}, (S_{2n} \Phi \circ g_{\mathbf{a} \mathbf{b}} - S_{2n} \Phi \circ g_{\mathbf{a} \mathbf{c}})(x) \in [-\sigma,\sigma] \}$, we have $\text{diam}(I_{\mathbf{a},\mathbf{b},\mathbf{c},\tilde{\mathbf{d}}}(\sigma)) \leq 2 \sigma^{9/10}$, and:
$$ 2^{-(n-\tilde{n})} \# \{ \widehat{\mathbf{d}} \in \{0,1\}^{n-\tilde{n}} , \ | S_{2n} \Phi \circ g_{\mathbf{a} \mathbf{b}}(x_{\tilde{\mathbf{d}}\widehat{\mathbf{d}}}) - S_{2n} \Phi \circ g_{\mathbf{a} \mathbf{c}}(x_{\tilde{\mathbf{d}}\widehat{\mathbf{d}}}) | \leq \sigma  \}   $$
$$ = 2^{-(n-\tilde{n})} \# \{ \widehat{\mathbf{d}} \in \{0,1\}^{n-\tilde{n}} ,  \ x_{\tilde{\mathbf{d}} \widehat{\mathbf{d}}}  \in I_{\mathbf{a},\mathbf{b},\mathbf{c},\tilde{\mathbf{d}}}(\sigma)  \} $$
$$ \leq 2^{-(n-\tilde{n})}\left(1+\frac{\text{diam}(I_{\mathbf{a},\mathbf{b},\mathbf{c},\tilde{\mathbf{d}}}(\sigma))}{2^{-n} e^{-4 \delta |\Phi|_\infty n} }\right) \leq \sigma^{1/10} $$
since the $x_{\mathbf{c}}$ are spaced out by at least $2^{-n} e^{-4 \delta |\Phi|_\infty n}$ from each other (Lemma 2.2.3), and since $2^{\tilde{n}} \sim \sigma^{-1/2}$.
\end{proof}

\begin{lemma}

Let $n$ be large enough. Let $\sigma \in [e^{-4 \varepsilon_0 n}, e^{-\varepsilon_0 \varepsilon_1 n/2}]$. Let $\mathbf{a} \in \{0,1\}^n$.  Then:
$$ 8^{-n} \# \left\{ (\mathbf{b}, \mathbf{c}, \mathbf{d}) \in \left(\{0,1\}^n\right)^3 , \ | 4^n g_{\mathbf{a} \mathbf{b}}'(x_\mathbf{d}) - 4^{n} g_{\mathbf{a} \mathbf{c}}'(x_{\mathbf{d}}) | \leq \sigma \right\} \leq  {\delta^{-1}} \sigma^{\gamma_\Phi/60}  .$$
\end{lemma}

\begin{proof}

Let $(\mathbf{b},\mathbf{c},\mathbf{d})$ be such that $|4^n g_{\mathbf{a} \mathbf{b}}'(x_\mathbf{d}) - 4^n g_{\mathbf{a} \mathbf{c}}'(x_{\mathbf{d}})| \leq \sigma$. Then:
$$ |S_{2n} \Phi \circ g_{\mathbf{a} \mathbf{b}}(x_{\mathbf{d}}) - S_{2n} \Phi \circ g_{\mathbf{a} \mathbf{c}}(x_{\mathbf{d}})| = \delta^{-1} \left| \ln\left( 4^n g_{\mathbf{a} \mathbf{b}}'(x_{\mathbf{d}}) \right) - \ln\left( 4^n g_{\mathbf{a} \mathbf{c}}'(x_{\mathbf{d}}) \right) \right| $$
$$ \leq \delta^{-1} e^{2 \delta |\Phi|_\infty n} \left|4^n g_{\mathbf{a} \mathbf{b}}'(x_{\mathbf{d}})  - 4^{n}  g_{\mathbf{a} \mathbf{c}}'(x_{\mathbf{d}})  \right| \leq \delta^{-1} e^{2 \delta |\Phi|_\infty n} \sigma.$$

We can then conclude using the previous lemma:
$$ 8^{-n} \# \left\{ (\mathbf{b}, \mathbf{c}, \mathbf{d}) \in \left(\{0,1\}^n\right)^3 , \ | 4^n g_{\mathbf{a} \mathbf{b}}'(x_\mathbf{d}) - 4^{n} g_{\mathbf{a} \mathbf{c}}'(x_{\mathbf{d}}) | \leq \sigma \right\} $$
$$ \leq 8^{-n} \# \left\{ (\mathbf{b}, \mathbf{c},\mathbf{d}) \in \left(\{0,1\}^n\right)^3 , \  |S_{2n} \Phi \circ g_{\mathbf{a} \mathbf{b}}(x_{\mathbf{d}}) - S_{2n} \Phi \circ g_{\mathbf{a} \mathbf{c}}(x_{\mathbf{d}}) | \leq \delta^{-1} \sigma e^{2 \delta n} \right\} $$
$$ \leq 2 \delta^{-1/2} (\delta^{-1} \sigma e^{2 \delta |\Phi|_\infty n})^{\gamma_\Phi/50} \leq \delta^{-1} \sigma^{\gamma_\Phi/60}  .$$ \end{proof}

\begin{lemma}
Recall that $\mathcal{R}_{n}^2 \subset (\{0,1\}^{n})^2$ denotes the set of regular couples. Most couples are regular:
$$  4^{-n} \# \left( (\{0,1\}^{n})^2 \setminus \mathcal{R}_n^{2} \right) \lesssim  \delta^{-1} e^{- \varepsilon_0 \varepsilon_1 \gamma_\Phi n/400} .$$

\end{lemma}

\begin{proof}

We use a dyadic decomposition: for each $\sigma \in [e^{-4 \varepsilon_0 n }, e^{-\varepsilon_0 \varepsilon_1 n/2 }]$, there exists $l \in \llbracket \lfloor \varepsilon_0 \varepsilon_1 n/2 \rfloor, \lfloor 4 \varepsilon_0 n \rfloor \rrbracket$ such that $e^{-(l+1)} \leq \sigma \leq e^{-l}$. Hence
$$ \mathcal{R}_n^2 \subset \bigcap_{l \in \llbracket \lfloor \varepsilon_0 \varepsilon_1 n/2 \rfloor, \lfloor 4 \varepsilon_0 n \rfloor \rrbracket} \mathcal{R}_{n,l}^2 ,$$
where we denoted $$\mathcal{R}_{n,l}^2 := \left\{ (\mathbf{a},\mathbf{d}) \in (\{0,1\}^n)^2 \ \Big{|} \ 4^{-n} \# \{ \mathbf{b} , \mathbf{c} \in (\{0,1\}^{n})^2 , \ | 4^n g_{\mathbf{a} \mathbf{b}}'(x_\mathbf{d}) - 4^n g_{\mathbf{a} \mathbf{c}}'(x_\mathbf{d}) | \leq e^{-l} \} \leq e^{-(l+1) \gamma_\Phi /100}  \right\}.$$
Markov's inequality and the previous lemma gives us the bound
$$ 4^{-n} \# \left( (\{0,1\}^{n})^2 \setminus \mathcal{R}_{n,l} \right) \leq e^{(l+1) \gamma_\Phi /100} 16^{-n} \sum_{\mathbf{a},\mathbf{d}} \# \{ \mathbf{b} , \mathbf{c} \in (\{0,1\}^{n})^2 , \ | 4^n g_{\mathbf{a} \mathbf{b}}'(x_\mathbf{d}) - 4^n g_{\mathbf{a} \mathbf{c}}'(x_\mathbf{d}) | \leq e^{-l} \}$$ $$ \leq e^{(l+1) \gamma_\Phi/100}  \left( \delta^{-1} e^{- l \gamma_\Phi/60} \right) \lesssim  \delta^{-1} e^{-\varepsilon_0 \varepsilon_1 \gamma_\Phi n/300},$$
which implies, by summing over $\varepsilon_0 \varepsilon_1 n/2 \leq l \leq 4 \varepsilon_0 n$, for $n$ large enough:
$$ 4^{-n} \# \left( (\{0,1\}^{n})^2 \setminus \mathcal{R}_n^{2} \right) \leq  \delta^{-1} e^{- \varepsilon_0 \varepsilon_1 \gamma_\Phi n/400} .$$ \end{proof}

\begin{lemma}
Notice that a block $\mathbf{A}=\mathbf{a_0} \dots \mathbf{a}_k \in \{0,1\}^{(k+1)n}$ is regular if $(\mathbf{a}_{j-1},\mathbf{a}_j)$ is a regular couple, for all $j \in \llbracket 1,k\rrbracket$. Recall that the set of regular blocks is denoted by $\mathcal{R}^{k+1}_n$. Then, most blocks are regular:

$$ 2^{-(k+1)n} \# \left( \{0,1\}^{(k+1)n} \setminus \mathcal{R}_n^{k+1} \right) \lesssim  \delta^{-1} e^{- \varepsilon_0 \varepsilon_1 \gamma_\Phi n/400} .$$

\end{lemma}

\begin{proof}

The result follows from the previous lemma, noticing that $$ \mathcal{R}_n^{k+1} \subset \bigcap_{j=1}^k \{ \mathbf{A} \in \{0,1\}^{(k+1)n}, \ (\mathbf{a}_{j-1},\mathbf{a}_{j}) \in \mathcal{R}_n^2 \}. $$\end{proof}

\cleardoublepage

\chapter{The Fourier dimension of hyperbolic Julia sets}

\section{Introduction}
In this chapter, we adapt the strategy presented in Chapter 2 to study the (lower) Fourier dimension of \emph{Julia sets} for hyperbolic rational maps in the Riemann sphere (see section 3.2). In particular, we will prove the claim made in Remark 2.1.10. Our main result is the following.

\begin{theorem}
Let $f:\widehat{\mathbb{C}} \rightarrow \widehat{\mathbb{C}}$ be a hyperbolic rational map of degree $d \geq 2$. Let $J$ denote its Julia set. If $J$ is included in a circle, then $J$ has positive lower Fourier dimension, seen as a compact subset of $\mathbb{R}$ after conjugation with a Möbius transformation. If $J$ is not included in a circle, then $J$ has positive lower Fourier dimension, seen as a compact subset of $\mathbb{C}$.

\end{theorem}

In fact, the case where $J$ is included in a circle is already known: Theorem 9.8.1, page 227 and remark page 230 in \cite{Be91} tells us that in this case, $J$ is either a circle, or a Cantor set. In the first case, the Fourier dimension is 1. In the second case, our hyperbolicity assumption, and Theorem 3.9 in \cite{OW17} ensure that a \say{total non-linearity} condition is satisfied, allowing us to apply the work of Sahlsten and Stevens \cite{SS20} (that deals with equilibrium states for nonlinear, one-dimensional expanding maps). In the case where $J$ is not included in a circle, we prove the following result.

\begin{theorem}
Let $f:\widehat{\mathbb{C}} \rightarrow \widehat{\mathbb{C}}$ be a hyperbolic rational map of degree $d \geq 2$. Denote by $J \subset \mathbb{C}$ its Julia set, and suppose that $J$ is not included in a circle. Let $V$ be an open neighborhood of $J$, and consider any potential $\varphi \in C^1(V,\mathbb{R})$. Let $\mu_\varphi \in \mathcal{P}(J)$ be its associated equilibrium measure. Then, there exists $\rho_1,\rho_2 \in (0,1)$ such that the following hold. \\

For any $C^1$ map $\chi : \mathbb{C} 
\rightarrow \mathbb{C}$ supported on some open set $\Omega \subset \mathbb{C}$, there exists $C=C({f,\mu,\chi}) \geq 1$ such that, for any $\xi \geq 1$ and any $C^2$ phase $\psi: \Omega \rightarrow \mathbb{R}$ satisfying $$\|\psi\|_{C^2} + (\inf_{z \in \Omega} |\nabla \psi|)^{-1} \leq \xi^{\rho_1}, $$
we have:
$$   \Big{|} \int_J e^{i \xi \psi(z) } \chi(z) d\mu_\varphi(z) \Big{|} \leq C \xi^{-\rho_2} .$$
In particular, $\mu_\varphi$ has positive lower Fourier dimension (see Lemma 1.1.26).
\end{theorem}

In particular, our result applies to the \textbf{conformal measure} (also called the measure of maximal dimension) and to the \textbf{measure of maximal entropy}, see section 3.2. Since the measure of maximum entropy is related to the harmonic measure in a polynomial setting \cite{MR92}, one may expect our result to have some corollaries on the Dirichlet problem with boundary conditions on quasicircles or to the Brownian motion (which is related to the heat equation). Finally, one should stress out that the conclusion of Theorem 3.1.2 no longer applies if $J$ is a whole circle: in this case, $f$ is conjugated to $z \mapsto z^d$ (\cite{OW17}), and so any invariant probability measure which enjoys Fourier decay must be the Lebesgue measure on the circle (this is an easy exercise using Fourier series). \\

To prove Theorem 3.1.2, we will follow the ideas of \cite{SS20} (which generalizes the methods of Chapter 2) and adapt them to our case, where topological difficulties arise from the 2-dimensional setting. The strategy of the proof and organization of the Chapter goes as follows. 

\begin{itemize}
    \item In section 3.2, we collect facts about the thermodynamic formalism in the context of hyperbolic complex dynamics. Section 3.2.3 is devoted to the construction of \emph{two} families of open sets adapted to the dynamics. In  section 3.3.5 we state a large deviation result about Birkhoff sums.
    \item In section 3.3 we use large deviations theorems to derive order of magnitude for some dynamically-related quantities.
    \item The proof of Theorem 3.1.2 begins in section 3.4. Using the invariance of the equilibrium measure by a transfer operator, we relate its Fourier transform to a sum of exponentials by carefully linearizing the phase. We then use a version of the sum-product phenomenon to conclude.
    \item Section 3.5 is devoted to a proof of the non-concentration hypothesis that is needed to use the sum-product phenomenon. To this end, we use a generalization of Theorem 2.5 in \cite{OW17}, which is a version of Dolgopyat's estimates for a family of twisted transfer operators. 
\end{itemize}

Even if the strategy of the proof is borrowed from \cite{SS20}, they are some noticeable difficulties that arise in our setting that were previously invisible. In dimension 1, estimates of various diameters and linearization processes are made easier by the fact that connected sets are convex. In particular, in dimension 1, the dynamics map convex sets into convex sets.\\

In dimension 2, one may not associate to the Markov partition a family of open sets that are convex and still satisfy the properties that we usually ask for them: see the discussion after Proposition 3.2.4. We overcome this difficulty by constructing \emph{two} families of open sets associated to the dynamics: the usual open sets related to Markov pieces, in which the theory of \cite{Ru78} and \cite{OW17} applies, and a new one where computations and control are made easier. 
The second difficulty is that the dynamics may twist and deform even the sets in our second family. We overcome this difficulty by taking advantage of the conformality of the dynamics, through the use of the Koebe 1/4-Theorem which allows us to have a good control over such deformations. \\

Another difficulty comes in the proof of the non concentration hypothesis: the complex nature of the dynamics suggests non concentration in modulus and arguments of some dynamically related quantities.

\section{Thermodynamic formalism on hyperbolic Julia sets}

\subsection{Hyperbolic Julia sets}

We recall standard definitions and results about holomorphic dynamical systems. For more background, we recommend the notes of Milnor \cite{Mi90}. \\

Denote by $\widehat{\mathbb{C}}$ the Riemann sphere. Let $f: \widehat{\mathbb{C}} \rightarrow \widehat{\mathbb{C}}$ be a rational map of degree $d \geq 2$. \\
Recall that a family of holomorphic maps defined on an open set $D \subset \widehat{\mathbb{C}}$ is called \emph{normal} if from every sequence of maps from the family there exists a subsequence that converges locally uniformly.
The Fatou set of $f$ is the largest open set in $\widehat{\mathbb{C}}$ where the family of iterates $ \left\{ f^n , \ n \in \mathbb{N}\right\}$ is a normal family. Its complement is called the Julia set and is denoted by $J$. In our case, it is always nonempty and compact. (Lemma 3.5 in \cite{Mi90}) \\

Since $f(J) = f^{-1}(J) = J$, the couple $(f,J)$ is a well defined dynamical system, describing a chaotic behavior. For example, the action of $f$ on $J$ is topologically mixing: for any open set $U$ such that $U \cap J \neq \emptyset$, there exists $n \geq 0$ such that $f^n(U \cap J) = J$. (See Corollary 11.2 in \cite{Mi90}) \\

A case where the dynamics of $f$ on $J$ is particularly well understood is when $f$ is supposed to be \textbf{hyperbolic}, and we will assume it from now on. It means that the orbit of every critical point converges to an attracting periodic orbit. (In other words, if $p \in \widehat{\mathbb{C}}$ is a critical point for $f$, then there exists $p_0 \in \widehat{\mathbb{C}}$ and $m>0$ such that $p_0$ is an attracting fixed point for $f^m$ and $f^{km}(p) \underset{k \rightarrow \infty}{\longrightarrow} p_0$.) In this case $J \neq \widehat{\mathbb{C}}$, and so by conjugating $f$ with an element of $PSL(2,\mathbb{C})$ we can always see $J$ as a compact subset of $\mathbb{C}$. The hyperbolicity condition is equivalent to the existence of constants $c_0$ and $1<\kappa<\kappa_1$, and of a small open neighborhood $V$ of $J$ such that:
$$ \forall x \in V, \ \forall n \geq 0 \ , c_0 \kappa^n \leq |(f^n)'(x)| \leq \kappa_1^n  .$$
This is Theorem 14.1 in \cite{Mi90}. From now on, we also assume that $J$ is not contained in a circle.

\subsection{Pressure and equilibrium states}

\begin{definition}[\cite{Ru89}, \cite{OW17}, \cite{Ru78}, \cite{PU17}] Let $\varphi \in C^1(V,\mathbb{R})$ be a potential.
Denote by $\mathcal{M}_f$ the compact space of all $f$-invariant probability measures on $J$, equipped with the weak*-topology. Denote, for $\mu \in \mathcal{M}_f$, $h_f(\mu)$ the entropy of $\mu$.
Then, the map $$ \mu \in \mathcal{M}_f \longmapsto h_f(\mu) + \int_J \varphi d \mu $$
is upper semi-continuous, and admits a unique maximum, denoted by $P(\varphi)$. \\
The unique measure $\mu_\varphi$ that satisfies
$$ P(\varphi) = h_f(\mu_\varphi) + \int_J \varphi d \mu_\varphi $$
is called the equilibrium state associated to the potential $\varphi$. \\
This measure is ergodic on $(J,f)$ and its support is $J$.

\end{definition}

Two potentials are of particular interest. Define the \textbf{distortion function} by $\tau(x) := \log( |f'(x)| )$ on $V$. If $V$ is chosen small enough, $\tau$ is real analytic, in particular it is $C^1$. We then know that there exists a unique $\delta_J \in \mathbb{R}$ such that $P(- \delta_J \tau)=0$.
In this case, $\delta_J$ is the Hausdorff dimension of $J$, and the equilibrium state $\mu_{- \delta_J \tau}$ is equivalent to the $\delta_J$-dimensional Hausdorff measure on $J$. It is sometimes called the \textbf{conformal measure}, or the \textbf{measure of maximal dimension}. Moreover, we have the formula $$ \dim_H(J) = h_f(\mu_{-\delta_J \tau})/ \int \tau d\mu_{- \delta_J \tau} .$$ See \cite{PU09}, Corollary 8.1.7 and Theorem 8.1.4 for a proof.  \\

Another important example is the following. If we set $\varphi=0$, then the pressure is given by the largest entropy available for invariant measures $\mu$. The associated equilibrium state is then called the \textbf{measure of maximal entropy}. In the context where $f$ is a polynomial, this measure coincides with the \textbf{harmonic measure} with respect to $\infty$, see \cite{MR92}.

\subsection{Markov partitions}

Hyperbolic rational maps are especially easy to study thanks to the existence of \textbf{Markov partitions} of the Julia set. Proposition 3.2.2 and some of the following results are extracted from the section 2 of $\cite{Ru89}$ and $\cite{OW17}$.

\begin{proposition}[Markov partitions]

For any $\alpha_0 > 0$ we may write $J$ as a finite union $J = \cup_{a \in \mathcal{A}} P_a$ of compact nonempty sets $P_a$ with $\text{diam} P_a < \alpha_0$, and $|\mathcal{A}| \geq d$. Furthermore, with the topology of $J$,
\begin{itemize}
    \item $\overline{\text{int}_J P_a} = P_a $,
    \item $\text{int}_J P_a \cap \text{int}_J P_b = \emptyset$   if $a \neq b$,
    \item each $f(P_a)$ is a union of sets $P_b$.
\end{itemize}

\end{proposition}

Define $M_{ab}=1$ if $f(P_a) \supset P_b$ and $M_{ab}$ = 0 otherwise.
Then some power $M^N$ of the $|\mathcal{A}| \times |\mathcal{A}|$ matrix $(M_{ab})$ has all its entries positive. 

\begin{remark}
Julia sets are always singleton or perfect sets (Corollary 3.10 in \cite{Mi90}), and in our case, since any point in $J$ always has exactly $d \geq 2$ preimages in $J$, $J$ is always a perfect set. In particular, the condition $\text{int}_J P_a \neq \emptyset$ implies  $\text{diam} (P_a) > 0$ for all $a$. This condition of having $ \geq 2$ preimages for any point in $J$ is what makes the proof of Proposition 3.2.6 and 3.2.7 works.
\end{remark}

If $\alpha_0$ is chosen small enough, we may also consider open neighborhoods around the $(P_a)$ that behaves well with the dynamics. They will help us do computations with our smooth map $f$.

\begin{proposition}

If $\alpha_0$ is small enough, we can choose a Markov partition $(P_a)_{a \in \mathcal{A}}$ and two families of open sets $(U_a)_{a \in \mathcal{A}}$ and $(D_a)_{a \in \mathcal{A}}$ such that $ P_a \subset U_a \subset \overline{U_a} \subset D_a \subset \overline{D_a} \subset V $, and:

\begin{enumerate}

    \item $\text{diam}(D_a) \leq \alpha_0$
    \item $f$ is injective on $\overline{D_a}$, for all $a \in \mathcal{A}$
    \item $f$ is injective on $D_a \cup D_b$ whenever $D_a \cap D_b \neq \emptyset$

    \item For every $a,b$ such that $f(P_a) \supset P_b$, we have a local inverse $g_{ab} : \overline{D_b} \rightarrow \overline{D_a}$ for $f$. Moreover, $g_{ab}$ is holomorphic on a neighborhood of $\overline{D_b}$.
    
    \item If $f(P_a) \supset P_b$ for some $a,b \in \mathcal{A}$, then $f(U_a) \supset \overline{U_b}$ and $f(D_a) \supset \overline{D_b}$.
    
    \item $D_a$ is convex. 
    \item For any $a \in \mathcal{A}$, $P_a \nsubseteq \overline{\cup_{b \neq a} U_b} $

    \item There exists $r \in (0,1)$ such that for all $a \in \mathcal{A}$, \ $P_a \subset B(x_a,r/10) \subset B(x_a,r) \subset D_a $.

\end{enumerate}

\end{proposition}

Usually, only the sets $(U_a)_{a \in \mathcal{A}}$ are considered when dealing with hyperbolic conformal dynamics: they are the open sets introduced in \cite{OW17} and \cite{Ru89}, and so they are the sets where their papers apply. But they are sometimes not easy to work with, especially because they may not be connected. The $(D_a)_{a \in \mathcal{A}}$ have the advantage to be convex, which will make the computations of section 3.3 doable. But they have the disadvantage that some $P_b$ may be entirely contained in $D_a$ even if $b \neq a$. 

\begin{proof}
The construction of the sets $(U_a)$ is borrowed from \cite{Ru89}, where Ruelle does it for expanding maps. The main problem is that $f$ is not necessarily expanding here, so we will have to introduce a modified metric (sometimes called Mather's metric).  Since $f$ is hyperbolic, we know that there exists some $N \in \mathbb{N}$ such that $f^N$ satisfy $|(f^N)'(x)|>1$ on $J$. So, there exists some small neighborhood $V$ of $J$ where $f^{N} : V \rightarrow \mathbb{C}$ is well defined, and where $ |(f^N)'(x)| \geq \kappa > 1 $. \\

Define $\rho(z) := \sum_{k=0}^{N-1} \kappa^{-k/N} |(f^k)'(z)| $.
Since $f'$ doesn't vanish, it is a smooth and positive function on $V$, and so $ds:= \rho(z) |dz|$ is a well defined conformal metric on $V$. 
Moreover,
$$ \rho(f(z)) |f'(z)| = \sum_{k=0}^{N-1} \kappa^{-k/N} |(f^k)'(f(z))| \ |f'(z)| = \sum_{k=1}^{N} \kappa^{-(k-1)/N} |(f^k)'(z)|  $$
$$ \geq \kappa^{1/N} \sum_{k=0}^{N-1} \kappa^{-k/N} |(f^k)'(z)| = \kappa^{1/N} \rho(z) ,$$
and so $f$ is expanding for the distance $d_\rho$ induced by the conformal metric $\rho(z)|dz|$. In particular, reducing $V$ if necessary, there exists $\alpha > 0$ such that
$$ \forall x,y \in V, d_\rho(x,y) \leq \alpha \Rightarrow d_\rho(f(x),f(y)) \geq \kappa^{1/N} d_\rho(x,y) .$$
In addition, the euclidean distance and the constructed conformal metric are equivalent: there exists a constant $G \geq 1$ such that $$ G^{-1} |x-y| \leq d_\rho(x,y) \leq G|x-y| ,$$ and so the property may become, taking $\alpha$ smaller if necessary:
$$ \forall x,y \in V, |x-y| \leq \alpha \Rightarrow d_\rho(f(x),f(y)) \geq \kappa^{1/N} d_\rho(x,y) .$$
Finally, define $L>1$ a Lipschitz constant of $f$ for the distance $d_\rho$. Now, let $r< G^{-1} \alpha/4$, \ $\alpha_0 < r  \min(G^{-1}/20,L^{-1}(\kappa^{1/N}-\kappa^{1/(2N)})) $, and let $(P_a)_{a \in \mathcal{A}}$ be a Markov partition such that $\text{diam}(P_a) < \alpha_0$.
Define, for $a \in \mathcal{A}$, $$D_a := \text{Conv}\left(D_\rho( x_a , r ) \right) \supset P_a$$ for some fixed $x_a \in \text{int}_J P_a$, where $\text{Conv}$ denotes the euclidean convex hull, and where $D_\rho$ is an open ball for the distance $d_\rho$. The point (1) is satisfied taking $\alpha_0$ smaller if necessary, and the points (2) and (3) follows immediately, provided $\alpha_0$ is small enough, since $f$ is hyperbolic (hence a local biholomorphism). The points (6) and (8) follows by the definition of $D_a$ and by equivalence of distances. We prove point (5) for $(D_a)_a$.
Since $f$ is a biholomorphism from $D_a$ to $f(D_a)$, and since $\text{diam}_\rho( f(D_a) ) \leq L \alpha_0$, we get that
$$ f(D_a) \supset f\left(D_\rho(x_a, r)\right) \supset D_\rho(f(x_a), \kappa^{1/N} r) \supset  \overline{ D_\rho(x_b, \kappa^{1/(2N)} r) }  $$
for each $b \in \mathcal{A}$ such that $f(P_a) \supset P_b$.\\

We then have to show that $ D_\rho(x_b, \kappa^{1/(2N)} r) \supset {D}_b $. Recall that Caratheodory's Theorem state that any point $\omega \in {D}_b$ can be written as the convex sum of three points in $D_{\rho}(x_b,r).$ 
So let $x,y,z \in D_\rho(x_b, r)$, and for any $\lambda_1,\lambda_2,\lambda_3 \in [0,1]$ satisfying $\lambda_1+\lambda_2+\lambda_3=1$, let $\omega := \lambda_1 x + \lambda_2 y + \lambda_3 z \in D_b$. By definition, there exists three paths $\gamma_x,\gamma_y,\gamma_z$ from $x_b$ to $x,y,z$, each with with length $< r$. Set $\gamma := \lambda_1 \gamma_x + \lambda_2 \gamma_y + \lambda_3 \gamma_z$. It is a well defined path in $D_b$ from $x_b$ to $\omega$. Its length satisfies

$$ \int_0^1 |\gamma'(t)| \rho(\gamma(t)) dt \leq \left( \int_0^1 |\gamma'(t)| \rho(x_a) dt  \right) \ e^{ r \| \rho'/\rho \|_{\infty,D_b}  } $$
$$ \leq \left( \lambda_1 \int_0^1 |\gamma_x'(t)| \rho(x_a) dt + \lambda_2 \int_0^1 |\gamma_y'(t)| \rho(x_a) dt + \lambda_3 \int_0^1 |\gamma_z'(t)| \rho(x_a)dt \right)  \ e^{ r \| \rho'/\rho \|_{\infty,D_b}  } $$
$$  \leq \left( \lambda_1 \int_0^1 |\gamma_x'(t)| \rho(\gamma_x(t)) dt + \lambda_2 \int_0^1 |\gamma_y'(t)| \rho(\gamma_y(t)) dt + \lambda_3 \int_0^1 
|\gamma_z'(t)| \rho(\gamma_z(t)) dt \right)  \ e^{2 r \| \rho'/\rho \|_{\infty,D_b}  }   $$
$$ \leq  r e^{ 2 r \| \rho'/\rho \|_{\infty,D_b}  } < \kappa^{1/(2N)} r $$
as soon as $ r < \frac{\ln(\kappa)}{4N} \| \rho'/\rho \|_\infty^{-1}  $. Hence (5) is true for $(D_a)_a$. We finished constructing our sets $D_a$. Notice that we can choose $r$ arbitrary small, and so the diameters of the $D_a$ can be chosen as small as we want. The point (4) follows by considering the inverse branches of $f$ induced by (5): they are holomorphic and $\kappa^{1/N}$ contracting for $d_\rho$. \\

The construction of the sets $(U_a)_a$ is easier.
First of all, there exists $\beta>0$ such that $D_\rho(x_a,\beta) \cap P_b = \emptyset $ whenever $a \neq b$. Then, since all of the $P_a$ are compactly contained in $D_a$, there exists a parameter $s< \beta/3$ such that for all $x \in P_a$, $\overline{D_\rho(x,s)} \subset D_a $. Define:

$$ U_a := \{ x \in V, \ d_\rho(x,P_a) < s \} \subset D_a .$$

First,  the fact that $s< \beta/3$ ensures that $P_a \nsubseteq \overline{\cup_{b \neq a} U_b} $, since $x_a \notin \overline{\cup_{b \neq a} U_b}$, hence proving (7). We prove (5). Let $b$ be such that $f(P_a) \supset P_b$. We prove that $f(U_a) \supset U_b$. Let $z \in U_b$. By definition, there exists $z_b \in P_b$ such that $d_\rho(z,z_b)<s$. Notice that $z=f(g_{ab}(z))$: to conclude, it suffice to prove that $g_{ab}(z) \in U_a$ (we only know that it is in $D_a$ at this point). Since $f(P_a) \supset P_b$ and since $f:D_a \rightarrow \mathbb{C}$ is injective, we know that $g_{ab}(z_b) \in P_a$. The fact that $ d_\rho(g_{ab}(z),g_{ab}(z_b)) \leq \kappa^{-1/N} s < s $ allows us to conclude.

 \end{proof}

To study the dynamics, we need to introduce some notations. \\
A finite word $(a_n)_n$ with letters in $\mathcal{A}$ is called admissible if $M_{a_n a_{n+1}=1}$ for every $n$. Then define:

\begin{itemize}
    \item $ \mathcal{W}_n := \{ (a_k)_{k=1, \dots, n} \in \mathcal{A}^n , \ (a_k) \ \text{is admissible} \}  $
    \item For $\textbf{a} = a_1 \dots a_n \in \mathcal{W}_n$, define $g_{\mathbf{a}} := g_{a_1 a_2} g_{a_2 a_3} \dots g_{a_{n-1} a_n} : D_{a_n} \rightarrow D_{a_1} $.
    \item For $\textbf{a} = a_1 \dots a_n \in \mathcal{W}_n$, define $P_{\textbf{a}} := g_{\textbf{a}}(P_{a_n}) \subset P_{a_1}$, $U_{\textbf{a}} := g_{\textbf{a}}( U_{a_n}) \subset U_{a_1}$, and \mbox{$D_{\textbf{a}} := g_{\textbf{a}}( D_{a_n}) \subset D_{a_1}$.}
\end{itemize}
We begin by an easy remark on the diameters of the $D_\mathbf{a}$ and on the behavior of $\varphi$ on those sets.
\begin{remark}
Since our potential is $C^1$, it is Lipschitz on $\overline{D} := \overline{ \bigcup_a D_a } \subset V$. There exists a constant $C_\varphi>0$ such that:
$$ \forall x,y \in D, \ |\varphi(x) - \varphi(y)| \leq C_\varphi |x-y|.$$
Moreover, we have some estimates on the diameters of the $P_{\textbf{a}}$.
Since $f$ is supposed to be hyperbolic, we have the following estimate for the local inverses:
$$ \forall n \geq 1, \ \forall \textbf{a} \in \mathcal{W}_{n+1}, \ \forall x \in D_{\textbf{a}}, \ \kappa_1^{-n} \leq |g_{\textbf{a}}'(x)| \leq c_0^{-1} \kappa^{-n} .$$
Hence, since each $D_a$ are convex, $$ \text{diam}(P_\textbf{a}) \leq \text{diam}(D_{\textbf{a}}) = \text{diam}(g_{\textbf{a}}(D_{a_{n+1}})) \leq c_0^{-1} \kappa^{-n} $$ decreases exponentially fast. We will say that $\varphi$ has \textbf{exponentially decreasing variations}, as
$$ \max_{ \textbf{a} \in \mathcal{W}_n } \sup_{x,y \in D_{\textbf{a}}} |\varphi(x) - \varphi(y)| \leq C_\varphi c_0^{-1} \kappa^{-n} .$$
\end{remark}

Notice the following technical difficulty: for $a \in \mathcal{A}$, $D_a$ may be convex but for a word $\mathbf{a} \in \mathcal{W}_n$, $D_{\mathbf{a}}$ will eventually be twisted by $g_{\mathbf{a}}$ and not be convex anymore. Fortunately, we still have the following result, which relies heavily on the fact that $f$ is holomorphic:

\begin{lemma}

For all $\mathbf{a} \in \mathcal{W}_n, \ \text{Conv}(P_{\mathbf{a}}) \subset D_{\mathbf{a}} $.

\end{lemma}

\begin{proof}

We will need to recall some results on univalent holomorphic functions $g:\mathbb{D} \rightarrow \mathbb{C}$. First of all, we have the Koebe quarter theorem, which states that if $g$ is such a map, then $g(\mathbb{D}) \supset B\left(g(0),\frac{|g'(0)|}{4} \right)$ ($\mathbb{D}$ is the unit disk, and $B(\cdot,\cdot)$ denotes an open euclidean ball). Secondly, the Koebe distortion theorem states that in this case, we also have that $ |g(z)-g(0)| \leq |g'(0)| \frac{|z|}{\left(1-|z|\right)^2} $. Combining those two results gives us the following fact: for any injective holomorphic $g:B(z_0,r) \rightarrow \mathbb{C}$ , we have $$g(B(z_0,r/10)) \subset B(g(z_0),{|f'(0)| r}/{4}) \subset g(B(z_0,r)). $$
In particular, in this case, $$ \text{Conv}\Big( 
g(B(z_0,r/10)) \Big) \subset g(B(z_0,r)) .$$
Now recall from Lemma 3.2.5 that there exists $r \in (0,1)$ such that, for all $a \in \mathcal{A}$, $$ P_a \subset B(x_a,r/10) \subset B(x_a,r) \subset D_a. $$
Hence we can directly apply the previous fact to the map ${g_{\mathbf{a}}}_{| B(x_a, r)}$: it follows that
$$ \text{Conv}(P_\mathbf{a}) = \text{Conv} \Big(g_{\mathbf{a}}(P_a) \Big) \subset \text{Conv} \Big( g_{\mathbf{a}}(B(x_a,r/10)) \Big) \subset g_{\mathbf{a}}(B(x_a,r)) \subset g_{\mathbf{a}}(D_a) = D_\mathbf{a}. $$
\end{proof}

We end this topological part with some final remarks, extracted from \cite{OW17}. The following \say{partition result} is true:
$$ J = \bigcup_{\textbf{a} \in \mathcal{W}_n} P_{\textbf{a}} \text{ , and }\text{int}_J P_{\textbf{a}} \cap \text{int}_J P_{\textbf{b}} = \emptyset \text{ if }\textbf{a} \neq \textbf{b} \in \mathcal{W}_n. $$    
This allows us to see that $\bigcup_{n \geq 0} f^{-n} \Big( \bigcup_{a \in \mathcal{A}} \partial P_a \Big)$ is a $f$-invariant subset of $J$. By ergodicity, its measure is zero or one, but since $\mu_\varphi$ has full support and since it is a countable union of closed sets with empty interior, we find that $\mu_\varphi \left(\bigcup_{a \in \mathcal{A}}  \partial P_a \right) = 0$. In particular, it implies that, for any $n \geq 1$, we have the relation
$$ \forall f \in C^0(J,\mathbb{C}), \ \int_J f d\mu_\varphi = \sum_{\textbf{a} \in \mathcal{W}_n} \int_{P_{\textbf{a}}} f \ d\mu_\varphi ,$$
which will be useful later.

\subsection{Transfer operators}

Let $\varphi \in C^1(V,\mathbb{R})$ be a smooth potential.
Let $U := \bigcup_{a \in \mathcal{A}} U_a$, and notice that $\overline{f^{-1}(U)} \subset U$. \\
We define the associated transfer operator $\mathcal{L}_\varphi : C^1(U) \rightarrow C^1(U) $ by $$ \mathcal{L}_\varphi h(x) := \sum_{y , f(y)=x} e^{\varphi(y)} h(y) .$$
Notice that, if $x \in U_a$, then 
$$ \mathcal{L}_\varphi h(x) = \sum_{b, M_{ba}=1} e^{\varphi( g_{ba}(x) )} h(g_{ba}(x)) .$$
We have the following formula for the iterates:
$$ \mathcal{L}_\varphi^n h(x) = \sum_{f^n(y)=x} e^{S_n \varphi(y)} h(y) ,$$
where $S_n \varphi := \sum_{k=0}^{n-1} \varphi \circ f^k$ is a Birkhoff sum. This can be rewritten, if $x \in U_{b}$, in the following form:
$$ \mathcal{L}_\varphi^n h(x) = \underset{a_{n+1}=b}{\sum_{ \textbf{a} \in \mathcal{W}_{n+1} } } e^{S_n \varphi( g_{\textbf{a}}(x) )} h( g_{\textbf{a}}(x) ) .$$
Finally, note that our transfer operator also acts on the set of probability measures on $J$, by duality, in the following way:
$$ \forall h \in C^0(J,\mathbb{C}), \ \int_J h \ d \mathcal{L}_\varphi^* \nu := \int_J  \mathcal{L}_\varphi h \ d\nu.$$
Transfer operators satisfy the following Theorem, extracted from \cite{Ru89}, Theorem 3.6 :
\begin{theorem}[Perron-Frobenius-Ruelle] 

With our choice of open set $U \supset J$, and for any \underline{real} \\ potential $\varphi \in C^1(U,\mathbb{R})$:

\begin{itemize}
    \item the spectral radius of $\mathcal{L}_\varphi$, acting on $C^1(U,\mathbb{C})$, is equal to $e^{P(\varphi)}$.
    \item there exists a unique probability measure $\nu_\varphi$ on $J$ such that $\mathcal{L}_\varphi^* \nu_\varphi = e^{P(\varphi)} \nu_\varphi $.
    \item there exists a unique map $h \in C^1(U,\mathbb{C})$ such that $\mathcal{L}_\varphi h = e^{P(\varphi)} h$ and $\int h d\nu_\varphi =1$. \\ Moreover, $h$ is positive.
    \item The product $h \nu_\varphi$ is equal to the equilibrium measure $\mu_\varphi$.

\end{itemize}

\end{theorem}
The Perron-Frobenius-Ruelle Theorem allows us to link $\mu_\varphi$ to $\nu_\varphi$, and this will allow us to prove some useful estimates, called Gibbs estimates.

\begin{proposition}[Gibbs estimate, \cite{PP90}]

$$\exists C_0 \geq 1, \ \forall \textbf{a} \in \mathcal{W}_n, \ \forall x_\textbf{a} \in P_{\textbf{a}}, \ C_0^{-1} e^{S_n \varphi(x_{\textbf{a}}) - n P(\varphi)} \leq \mu_\varphi( P_{\textbf{a}} ) \leq C_0 e^{S_n\varphi(x_{\textbf{a}}) - n P(\varphi)} .$$

\end{proposition}

\begin{proof}

It is enough to prove the estimate for $\nu_\varphi$ since $h$ is continuous on the compact $J$, and since $h \nu_\varphi = \mu_\varphi$. We have
$$ \int_J e^{- \varphi} \mathbb{1}_{P_{a_1 \dots a_n}} d\nu_\varphi = e^{-P(\varphi)} \int_J \mathcal{L}_\varphi\left( e^{- \varphi} \mathbb{1}_{P_{a_1 \dots a_n}} \right) d\nu_\varphi  = e^{-P(\varphi)} \int_{J} \mathbb{1}_{P_{a_2 \dots a_n}} d\nu_\varphi = e^{-P(\varphi)} \nu_\varphi(P_{a_2 \dots a_n}). $$
Moreover, since $\varphi$ has exponentially decreasing variations, we can write that
$$\forall x_{\textbf{a}} \in P_{a_1 \dots a_n}, \ e^{-\varphi(x_{\textbf{a}}) - C \kappa^{-n}} \nu_\varphi(P_{a_1 \dots a_n}) \leq \int_J e^{- \varphi} \mathbb{1}_{P_{a_1 \dots a_n}} d\nu_\varphi \leq  e^{-\varphi(x_{\textbf{a}}) + C \kappa^{-n}} \nu_\varphi(P_{a_1 \dots a_n})  ,$$
and so $$ \forall x_{\textbf{a}} \in P_{a_1 \dots a_n}, \ e^{-\varphi(x_{\textbf{a}}) - C \kappa^{-n}}  \leq \frac{\nu_\varphi(P_{a_2 \dots a_n})}{\nu_\varphi(P_{a_1 \dots a_n}) } e^{-P(\varphi)} \leq  e^{-\varphi(x_{\textbf{a}}) + C \kappa^{-n}} . $$
Multiplying those inequalities gives us the desired relation, with $C_0 := e^{ C/(\kappa-1) }$. \end{proof}

It will be useful, in our future computations, to get rid of the pressure term in our exponential: in the case where $P(\varphi)=0$, we see that $\mu_\varphi(P_\textbf{a})  = e^{S_n \varphi(x_{\textbf{a}}) + O(1)}$ for $x_{\textbf{a}} \in P_{\textbf{a}}$.  

\begin{proposition}

Let $\psi \in C^1(U)$, and let $\mu_\psi$ be its associated equilibrium state.
There exists $\varphi \in C^1(U)$ such that $\mu_\varphi = \mu_{\psi}$, and that is normalized. \\ That is: $P(\varphi) =0$, $ \varphi < 0 \text{ on } J$, $ \mathcal{L}_\varphi 1 = 1 $ and $ \mathcal{L}_\varphi^* \mu_\varphi = \mu_\varphi $.  
\end{proposition}

\begin{proof}

Perron-Frobenius-Ruelle Theorem tells us that there exists a $C^1$ map $h>0$ and a probability measure $\nu_\psi$ such that $\mathcal{L}_\psi^* \nu_\psi = e^{P(\psi)} \nu_{\psi} $, $\mathcal{L}_\psi h = e^{P(\psi)} h$ and $\int h d\nu_{\psi} = 1$. \\
It is then a simple exercise to check that $\varphi := \psi - \log( h \circ f ) + \log(h) - P(\psi)$ defines a normalized potential, and that its equilibrium measure $\mu_\varphi$ is equal to $\mu_\psi$. \end{proof}

This proposition has the following consequence: we can always suppose that our equilibrium measure comes from a \emph{normalized} potential, by eventually choosing another smaller Markov partition afterwards. It gives us for free the invariance under some transfer operator, which completes the already fine properties of $f$-invariance and ergodicity. It also allows us to prove a useful regularity property.

\begin{proposition}

The equilibrium measure $\mu_\varphi$ is upper regular. More precisely, there exists $C,\delta_{AD}>0$ such that:
$$ \forall x \in \mathbb{C}, \ \forall r>0, \ \mu_\varphi(B(x,r)) \leq C  r^{\delta_{AD}}.$$

\end{proposition}

\begin{proof}

First of all, since $\mu_\varphi$ is a probability measure, we may only prove this estimate for $r$ small enough. Then, we know that $\mu_\varphi$ is supported in $J$, and so we just have to verify the estimate if $B(x_0,r) \cap J \neq \emptyset$. Without loss of generality, we can suppose that $x_0 \in J$. \\

The main idea is to cover $B(x_0,r) \cap J$ by some $P_{\mathbf{b}}$, but estimating the number of such $P_{\mathbf{b}}$ that are needed to do so is difficult. To bypass this difficulty, we use the notion of Moran cover. For any $x \in J$, define $n(x,r)$ as the only integer such that 

$$ |(f^{n(x,r)-1})'(x)|^{-1} \geq r \quad \text{and} \quad |(f^{n(x,r)})'(x)|^{-1} < r.$$
We get from the hyperbolicity condition $|(f^{n})'| \geq c_0 \kappa^n$ the following bound: $$ \forall x, \ -n(x,r) \leq \ln\left( 2 r c_0^{-1} \right)/\ln \kappa. $$
For any $x \in J \setminus \bigcup_{n \geq 0} f^{-n} \left( \bigcup_{a \in \mathcal{A}} \partial P_a \right)$ and for any $n$, there exists a unique $\mathbf{a} \in \mathcal{W}_n$ such that $x \in P_{\mathbf{a}}$. We denote it $P_n(x)$. Notice that $x \in P_{n(x,r)}(x)$. If $y \in P_{n(x,r)}(x)$ and $n(y,r) \leq n(x,r)$ then $ P_{n(x,r)}(x) \subset P_{n(x,r)}(y) $. Let $P(x)$ be the largest cylinder containing $x$ of the form $P_{n(y,r)}(x)$ for some $y \in P(x)$ and satisfying $P_{n(z,r)}(x) \subset P(x)$ for any $z \in P(x)$. The sets $(P(x))_{x \in J}$ are equal or disjoint (mod the boundary), and hence produce a cover of $J$ called a \emph{Moran cover}. Denote this Moran cover $\mathcal{P}_r$. An important property of this cover is the following: there exists a constant $M$ \emph{independent of $x_0$ and $r$} such that we can cover the ball $B(x_0,r)$ by $M$ elements of $\mathcal{P}_r$. Moreover, every element of the Moran cover have diameter strictly less than $r$. See \cite{PW97} page 243, \cite{Pe98} section 20, or \cite{WW17}. The proof uses the conformality of the dynamics. (In the references, the Moran cover are used in the context of expanding dynamics. In our case it is only eventually expanding, but working under the equivalent and conformal Mather metric, or replacing $f$ by $f^N$ for a large $N$, allows us see $f$ as an expanding map.) \\

We can then conclude our proof. The following holds:
$$ \mu_\varphi(B(x_0,r)) \leq \underset{B(x_0,r) \cap P \neq \emptyset}{\sum_{P \in \mathcal{P}_r} } \mu_\varphi(P) . $$
By Gibbs estimates, since each $P \in \mathcal{P}_r$ is of the form $P_{\mathbf{b}}$ for some $\mathbf{b} \in \mathcal{W}_{n(x,r)}$, $x \in J$, and by the bound on $n(x,r)$, we get:
$$ \forall P \in \mathcal{P}_r, \ \mu_\varphi(P) \leq C_0 e^{-n(x,r) |\sup_J \varphi|} \leq C r^{\delta_{AD}}  $$
for some $C, \delta_{AD} >0$.
Hence, since $B(x_0,r) \cap P$ occurs at most $M$ times, we get our desired bound $$ \mu_\varphi(B(x_0,r)) \leq M C r^{\delta_{AD}} .$$ \end{proof}

\subsection{Large deviation estimates }

From the ergodicity of $\mu_\varphi$, it is natural to ask if a large deviation result holds for the Birkhoff sums of potentials. The following theorem is true, we detail its proof in section 3.7.

\begin{theorem}
Let $\mu_\varphi$ be the equilibrium measure associated to a normalized $C^1(U,\mathbb{R})$ potential. Let $\psi \in C^1(U,\mathbb{R})$ be another potential. Then, for all $\varepsilon > 0$, there exists $C,\delta_0 > 0$ such that
$$\forall n \geq 1, \ \mu_\varphi\left( \left\{ x \in J \ , \ \left|\frac{1}{n} S_n \psi(x) - \int_J \psi d\mu_{\varphi} \right| > \varepsilon \right\} \right) \leq C e^{- n \delta_0} .$$

\end{theorem}

\begin{definition}

Let $\varphi \in C^1(U,\mathbb{R})$ be a normalized potential with equilibrium measure $\mu_\varphi$. \\ Let $\tau = \log |f'| \in C^1(U,\mathbb{R})$ be the distortion function. We call 
$$ \lambda_f(\mu_\varphi) := \int_J \tau \ d\mu_{\varphi} $$
the Lyapunov exponent of $\mu_\varphi$, and $ \delta := h_f(\mu_\varphi)/ \lambda(\mu_{\varphi}) $ the dimension of $\mu_\varphi$.

\end{definition}

\begin{remark}

The hyperbolicity and normalization assumptions ensure that $ h_f(\mu_\varphi), \lambda_f(\mu_\varphi) > 0$.
Indeed, we know that $\varphi<0$ on all $J$ and $P(\varphi)=0$, and so
$$ h_f(\mu_\varphi) = P(\varphi) - \int_J \varphi d\mu_\varphi > 0.$$
For the Lyapunov exponent, using the fact that $\mu_\varphi$ is $f$-invariant, we see that
$$ \lambda_f(\mu_\varphi) = \int_J \log |f'| d\mu_\varphi = \frac{1}{n} \sum_{k=0}^{n-1}  \int_J \log |f' \circ f^k| d\mu_\varphi $$ $$= \frac{1}{n} \int_J \log |(f^n)'| d\mu_\varphi \geq \frac{\log(c_0)}{n} + \log(\kappa) \rightarrow \log(\kappa) > 0 .$$

\end{remark}

\begin{proposition}
Let $\varphi \in C^1(U,\mathbb{R})$ be a normalized potential with equilibrium measure $\mu_\varphi$. Denote by $\lambda>0$ its Lyapunov exponent and $\delta>0$ its dimension. Then, for every $\varepsilon > 0$, there exists $C,\delta_0>0$ such that

$$ \forall n \geq 1, \ \mu_\varphi \left( \left\{ x \in J \ , \ \left|\frac{1}{n}S_n \tau(x) -  \lambda \right| \geq \varepsilon \text{ or } \left|\frac{S_n \varphi(x)}{S_n \tau(x)} + \delta \right| \geq \varepsilon \right\} \right) \leq C e^{- \delta_0 n} .$$

\end{proposition}

\begin{proof}

Let $\varepsilon>0$.
Applying Theorem 3.2.11 to $\psi = \tau$ gives  $$ \mu_\varphi \left( \left\{ x \in J \ , \ \left|\frac{1}{n}S_n \tau(x) - \lambda\right| \geq \varepsilon \right\} \right) \leq C e^{- \delta_0 n} $$
for some $C$ and $\delta_0>0$. Next, if $x \in J$ satisfies $$\left|\frac{S_n \varphi(x)}{S_n \tau(x)} + \delta \right| \geq \varepsilon, $$
then we have $$ |S_n \Phi(x)| \geq \varepsilon |S_n \tau(x)| \geq \varepsilon (n \log \kappa + \log c_0) $$
for the modified potential $\Phi := \varphi + \delta \tau$. Notice that this potential is $C^1$, and that $$ \int_J \Phi \ d\mu_\varphi = \int_J \varphi d\mu_\varphi + \frac{h_f(\mu_\varphi)}{\lambda_f(\mu_\varphi)} \int_J \tau \ d\mu_\varphi = 0 .$$ For $n$ large enough, we get $|S_n \Phi(x)| \geq \varepsilon n (\ln \kappa)/2 $, and so we can apply Theorem 3.2.11 to $\Phi$ again and conclude. \end{proof}
For clarity, we will replace $\mu_\varphi$ by $\mu$ in the rest of the Chapter. The dependence on $\varphi$ is implied. 


\section{Computing some orders of magnitude}

In this section, we derive various orders of magnitude of quantities that appear when we iterate our transfer operator. We need to recall some useful formalism used in $\cite{BD17}$. 

\begin{itemize}
    \item For $n \geq 1$, recall that $\mathcal{W}_n$ is the set of admissible words of length $n$. (A word $\mathbf{a}$ is admissible if $f(P_{a_i}) \supset P_{a_{i+1}}$ for all $i$.) If $\textbf{a} = a_1 \dots a_n a_{n+1} \in \mathcal{W}_{n+1}$, define $\mathbf{a}' := a_1 \dots a_{n} \in \mathcal{W}_{n}$. 
    
    \item For $\textbf{a}=a_1 \dots a_{n+1} \in \mathcal{W}_{n+1}$, $\textbf{b} = b_1 \dots b_{m+1} \in \mathcal{W}_{m+1}$, we write $\textbf{a} \rightsquigarrow \textbf{b}$ if $a_{n+1} = b_1$. Note that when $\textbf{a} \rightsquigarrow \textbf{b}$, the concatenation $\textbf{a}'\textbf{b}$ is an admissible word of length $n+m+1$.
    
    \item For $\mathbf{a} \in \mathcal{W}_{n+1}$, define $b(\mathbf{a}) := a_{n+1}$.

\end{itemize}
With those notations, we can reformulate our formula for the iterate of our transfer operator. For a function $h : U \rightarrow \mathbb{C}$, we have:
$$ \forall x \in P_b, \ \mathcal{L}_\varphi^n h(x) = \underset{\mathbf{a} \rightsquigarrow b}{\sum_{\mathbf{a} \in \mathcal{W}_{n+1}}} e^{S_n \varphi(g_{\mathbf{a}}(x))} h(g_{\mathbf{a}}(x)) = \underset{\mathbf{a} \rightsquigarrow b}{\sum_{\mathbf{a} \in \mathcal{W}_{n+1}}} h(g_{\mathbf{a}}(x)) w_{\mathbf{a}}(x) ,$$
where $$ w_{\mathbf{a}}(x) := e^{S_n\varphi(g_{\mathbf{a}}(x))} .$$
Iterating $\mathcal{L}_\varphi^n$ again leads us to the formula
$$ \forall x \in P_b, \ \mathcal{L}_\varphi^{nk} h(x)  = \sum_{\mathbf{a}_1 \rightsquigarrow \dots \rightsquigarrow \mathbf{a_k} \rightsquigarrow b} h(g_{\mathbf{a}_1 ' \dots \mathbf{a}_{k-1} ' \mathbf{a}_k}(x) ) w_{\mathbf{a}_1 ' \dots \mathbf{a}_{k-1} ' \mathbf{a}_k}(x) .$$
We are interested in the behavior of, for example, $w_\mathbf{a}$ for well behaved $\mathbf{a}$. For this, we use the previously mentioned large deviation estimate. \\

This part is adapted from \cite{SS20} and \cite{JS16}. Remember that $\overline{f^{-1}(D)} \subset D$.

\begin{definition}
For $\varepsilon>0$ and $n \geq 1$, write
$$ A_{n}(\varepsilon) := \left\{ x \in f^{-n}(D) \ , \ \left|\frac{1}{n}S_n \tau(x) -  \lambda \right| < \varepsilon \text{ and } \left|\frac{S_n \varphi(x)}{S_n \tau(x)} + \delta \right| < \varepsilon  \right\}. $$
Then Proposition 3.2.14 says that, for all $\varepsilon>0$, there exists $n_0(\varepsilon) \in \mathbb{N}$ and $\delta_0(\varepsilon)>0$ such that $$\forall n \geq n_0(\varepsilon), \ \mu(J \setminus A_n(\varepsilon)) \leq e^{- \delta_0(\varepsilon) n} .$$
\end{definition}

\begin{notations}
To simplify the reading, when two quantities dependent of $n$ satisfy $b_n \leq C a_n $ for some constant $C$, we denote it by $a_n \lesssim b_n$. If $a_n \lesssim b_n \lesssim a_n$, we denote it by $a_n \simeq b_n$. If there exists $c,C$ and $\alpha$, independent of $n$ and $\varepsilon$, such that $ c e^{- \varepsilon \alpha n} a_n \leq b_n \leq C e^{\varepsilon \alpha n} a_n$, we denote it by $a_n \sim b_n$. Throughout the chapter $\alpha$ will be allowed to change from line to line. It correspond to some positive constant. 
\end{notations}

Eventually, we will chose $\varepsilon$ small enough such that this exponentially growing term gets absorbed by the other leading terms, so we can neglect it.

\begin{proposition}
Let $\mathbf{a} \in \mathcal{W}_{n+1}$ be such that $D_{\mathbf{a}} \subset A_n(\varepsilon)$. Then:

\begin{itemize}
    \item uniformly on $x \in D_{b(\mathbf{a})}, \  |g_\mathbf{a}'(x)| \sim e^{- n \lambda}$
    \item $\mathrm{diam}(P_\mathbf{a}), \ \mathrm{diam}(U_\mathbf{a}), \ \mathrm{diam}(D_\mathbf{a})  \sim e^{- n \lambda} $
    \item uniformly on $x \in D_{\mathbf{a}}, \ w_{\mathbf{a}}(x) \sim e^{- \delta \lambda n} $
    \item $\mu(P_\textbf{a}) \sim e^{- \delta \lambda n}$
    
\end{itemize}

\end{proposition}

\begin{remark}
Intuitively speaking, here is what is happening.
Proposition 3.2.14 states that, for most $x \in J$, $\frac{1}{n} S_n \tau \cong \lambda$ and $ \frac{1}{n} S_n \varphi \cong - \lambda \delta $. Then, recall that $ \text{diam}( P_{\mathbf{a}} ) \simeq |(f^n)'|^{-1}  = e^{-S_n{ \tau } },$ and so $\text{diam} P_{\mathbf{a}} \sim e^{- \lambda n}$ for most words $\textbf{a}$. Notice that the presence of the Lyapunov exponent in the exponential is not surprising, since it is defined to represent a characteristic spatial frequency of our problem.
Samely, we can argue that since our equilibrium measure satisfies the Gibbs estimate, we have $ \mu( P_{\mathbf{a}} ) \simeq e^{S_n{\varphi}} ,$
and so $ \mu(P_{\mathbf{a}}) \sim e^{- \delta \lambda n} $
for most words $\mathbf{a}$. Again, it is no surprise that this exponent appears here: we recognize that $ \mu(P_{\mathbf{a}}) \sim \text{diam}(P_\mathbf{a})^\delta $, where $\delta$ is the dimension of our measure.

\end{remark}

\begin{proof}

Let $\mathbf{a} \in \mathcal{W}_{n+1}$ be such that $D_{\mathbf{a}} \subset A_n(\varepsilon)$. We have

$$ \forall x \in D_{b(\mathbf{a})}, \ |g_{\mathbf{a}}'(x)| = e^{-S_n \tau (g_{\mathbf{a}}(x))}  ,$$
and so $$ \forall x \in D_{b(\mathbf{a})}, \  e^{-n \lambda } e^{-n \varepsilon} \leq |g_{\mathbf{a}}'(x)| \leq e^{- n \lambda} e^{ n \varepsilon}. $$
For the diameters, the argument uses the conformal setting, through the Koebe quarter theorem. By  Lemma 3.2.6, $\text{Conv}(P_{\mathbf{a}}) \subset D_{\mathbf{a}} \subset A_n(\varepsilon)$. Hence:

$$ \forall x,y \in P_{a_n}, \ |x-y| = |f^n(g_{\mathbf{a}}(x)) - f^n(g_{\mathbf{a}}(y))| $$ 
$$ \leq \int_{0}^1 |(f^n)'\left(g_{\mathbf{a}}(y)+t(g_{\mathbf{a}}(x)-g_{\mathbf{a}}(y))\right)| dt \ |(g_{\mathbf{a}}(x)-g_{\mathbf{a}}(y))| \leq e^{\varepsilon n} e^{n \lambda} \text{diam}(P_{\mathbf{a}}) ,$$
and so $ e^{- \varepsilon n} e^{-\lambda n} \text{diam}(P_{a_n}) \leq \text{diam}(P_{\mathbf{a}}) $. Next, we write
$$ \text{diam}(P_{\mathbf{a}}) \leq \text{diam}(U_{\mathbf{a}}) \leq \text{diam}(D_{\mathbf{a}})$$
and 
$$ \text{diam}(D_{\mathbf{a}}) = \text{diam}(g_{\mathbf{a}}(D_{a_n})) \leq e^{\varepsilon n} e^{-\lambda n} \text{diam}(D_{a_n}) .$$
by convexity of $D_{a_n}$.
We have proved that $\text{diam}(P_\mathbf{a}), \ \text{diam}(U_\mathbf{a}), \ \text{diam}(D_\mathbf{a})  \sim e^{- n \lambda} $. \\
Next, consider the weight $w_{\mathbf{a}}(x)$. We have $$ w_{\mathbf{a}}(x)   = e^{S_n \varphi(g_{\mathbf{a}}(x))},$$
so
$$ e^{- \delta S_n \tau(x)} e^{- \varepsilon |S_n \tau(x)|} \leq w_{\mathbf{a}}(x) \leq e^{- \delta S_n \tau(x)} e^{\varepsilon |S_n \tau(x)|} ,$$ and hence
$$ e^{- \delta \lambda n} e^{- \varepsilon (\lambda + \delta + \varepsilon) n } \leq w_{\mathbf{a}}(x) \leq e^{- \delta \lambda n} e^{\varepsilon (\lambda + \delta + \varepsilon) n } .$$
Finally, since $\mu$ is a Gibbs measure for some constant parameter $C_0$, and with pressure 0, we can write:
$$ C_0^{-1} e^{- \delta \lambda n} e^{- \varepsilon (\lambda + \delta + \varepsilon) n } \leq \mu(P_{\mathbf{a}}) \leq C_0 e^{- \delta \lambda n} e^{\varepsilon (\lambda + \delta + \varepsilon) n }. $$ \end{proof}

\begin{definition}
Let $K \subset \mathbb{C}$ be a set.
Define the set of ($K$-localised) $\varepsilon$-regular words by
$$ \mathcal{R}_{n+1}(K, \varepsilon) := \left\{ \mathbf{a} \in \mathcal{W}_{n+1} \ | \ D_{\mathbf{a}} \subset A_{n}(\varepsilon),\ D_{\mathbf{a}} \cap K \neq \emptyset \right\}. $$
If $K=J$, we just denote $\mathcal{R}_{n+1}(\varepsilon) := \mathcal{R}_{n+1}(K,\varepsilon) $. We also define the set of ($K$-localised) $\varepsilon$-regular $k$-blocks by $$ \mathcal{R}_{n+1}^k(K,\varepsilon) = \left\{ \mathbf{A}=\mathbf{a}_1' \dots \mathbf{a}_{k-1}' \mathbf{a}_k \in \mathcal{W}_{nk+1} \ | \ \mathbf{a}_1 \in \mathcal{R}_{n+1}(K,\varepsilon), \ \forall i \geq 2, \ \mathbf{a}_i \in \mathcal{R}_{n+1}(\varepsilon) \right\} .$$ 
We write $\mathcal{R}_{n+1}^k(\varepsilon) := \mathcal{R}_{n+1}^k(J,\varepsilon)$. Finally, define the associated $\varepsilon$-regular geometric points to be $$ R_{n+1}^k(\varepsilon) := \bigcup_{\mathbf{A} \in \mathcal{R}_{n+1}^k(\varepsilon) } P_{\mathbf{A}}.$$
\end{definition}

\begin{lemma}
 There exists $n_1(\varepsilon)$ such that, for all $n \geq n_1(\varepsilon)$, we have:
$$  A_{n}(\varepsilon/2) \subset R_{n+1}(\varepsilon).$$

\end{lemma}

\begin{proof}
Let $x \in A_{n}(\varepsilon/2)$. There exists $\mathbf{a} \in \mathcal{W}_{n+1}$ such that $x \in D_{\mathbf{a}}$. To conclude, it suffices to show that $D_{\mathbf{a}} \subset A_n(\varepsilon)$. So let $y \in D_{\mathbf{a}}$.
We already saw in Remark 3.2.5 that Lipschitz potentials have exponentially decreasing variations. It implies in particular the existence of some constant $C>0$, which depends only on $f$ here, such that
$$ \left| S_n \tau(y) - S_n \tau(x) \right| \leq C .$$
Hence, we have
$$ \left| \frac{S_n \tau(y)}{n} - \lambda \right| = \left| \frac{S_n \tau(x)}{n} - \lambda\right| + \frac{1}{n}\left| S_n \tau(y) - S_n \tau(x) \right| \leq \varepsilon/2 + \frac{C}{n} \leq \varepsilon $$
as long as we chose $n$ large enough, depending on $\varepsilon$. Samely, we can write
$$ \left| \frac{S_n \varphi(y)}{S_n \tau(y)} + \delta \right| \leq \varepsilon/2 + \left| \frac{S_n \varphi(y)}{S_n \tau(y)} - \frac{S_n \varphi(x)}{S_n\tau(x)} \right| .$$
Since $S_n \tau = \log |(f^n)'| \geq \log(c_0) + n \log(\kappa)$ uniformly on $D$ for some $\kappa > 1$, we get 
$$ \left| \frac{S_n \varphi(y)}{S_n \tau(y)} + \delta \right| \leq \varepsilon/2 + \frac{1}{(\log(c_0) + n \log(\kappa))^2}\left| S_n \varphi(y) S_n \tau(x) - S_n \varphi(x) S_n \tau(y) \right| $$
$$ \leq \varepsilon/2 + \frac{|S_n \tau(x)| }{(\log(c_0) + n \log(\kappa))^2}\left| S_n \varphi(y)  - S_n \varphi(x) \right| + \frac{ |S_n \varphi(x)|}{(\log(c_0) + n \log(\kappa))^2}\left|  S_n \tau(x) - S_n \tau(y) \right| $$
$$ \leq \varepsilon/2 + \frac{C}{n} $$
for some constant $C$, where we used the fact that $(S_n \varphi)/n$ is uniformly bounded on $f^{-n}(\overline{D})$ and the preceding remark on potentials with exponentially vanishing variations.
Again, choosing $n$ large enough depending on $\varepsilon$ allows us to conclude. \end{proof}

\begin{proposition}
We have the following cardinality estimate:
$$ \# \mathcal{R}_{n+1}(\varepsilon) \sim  e^{\delta \lambda n} .$$
Moreover, there exists $n_2(\varepsilon)$ and $\delta_1(\varepsilon)>0$ such that
$$ \forall n \geq n_2(\varepsilon), \ \mu\left( J \setminus R_{n+1}(\varepsilon) \right) \leq  e^{ - \delta_1(\varepsilon) n} .$$
\end{proposition}

\begin{proof}
By the preceding lemma, we can write, for $n \geq n_1(\varepsilon)$:
$$ J \cap A_{n}(\varepsilon/2) \subset R_{n+1}(\varepsilon).$$
Moreover, we also know that there exists $n_0(\varepsilon/2)$ and $\delta_0(\varepsilon/2)$ such that, for all $n \geq n_0(\varepsilon/2)$, we have
$$ \mu\left( J \setminus A_n(\varepsilon/2) \right) \leq e^{- \delta_0(\varepsilon/2) n} .$$
So define $n_2(\varepsilon) := \max(n_1(\varepsilon), n_0(\varepsilon/2)/\varepsilon_0).$ For all $n \geq n_2(\varepsilon)$, we then have:
$$ \mu(J \setminus R_{n+1}(\varepsilon)) \leq \mu\left(J \setminus A_n(\varepsilon/2) \right) \leq e^{- \delta_0(\varepsilon/2) n } .$$
Next, the cardinality estimates follow from the bound on the measure. Indeed, we know that, for $n \geq n_2(\varepsilon)$:
$$ \mu(J) = \mu(J \cap R_{n+1}(\varepsilon)) + \mu(J \setminus R_{n+1}(\varepsilon)) \leq \sum_{\mathbf{a} \in \mathcal{R}_{n+1}(\varepsilon)} \mu(P_\mathbf{a})  + e^{- \delta_0(\varepsilon/2) n} ,$$
and so
$$ 1 - e^{- \delta_0(\varepsilon/2) n}  \leq \sum_{\mathbf{a} \in \mathcal{R}_{n+1}(\varepsilon)} \mu(P_\mathbf{a}) \leq 1 .$$
We then use the estimate obtained for $\mu(P_\mathbf{a})$, that is,
$$ C_0^{-1} e^{- \delta \lambda n} e^{- \varepsilon (\lambda + \delta + \varepsilon) n } \leq \mu(P_{\mathbf{a}}) \leq C_0 e^{- \delta \lambda n} e^{\varepsilon (\lambda + \delta + \varepsilon) n } ,$$
and we obtain
$$ C_0^{-1} e^{\delta \lambda n} e^{- \varepsilon(\lambda + \delta + \varepsilon)n} \left(  1 - e^{- \delta_0(\varepsilon/2) n} \right) \leq \# \mathcal{R}_{n+1}(\varepsilon) \leq C_0  e^{\delta \lambda n} e^{\varepsilon(\lambda + \delta + \varepsilon)}, $$
which proves that $\# \mathcal{R}_{n+1}(\varepsilon) \sim e^{\delta \lambda n}$. \end{proof}

\begin{proposition}
For all $n \geq n_2(\varepsilon)$,  $$ \mu\left(J \setminus R_{n+1}^k(\varepsilon) \right) \leq k e^{ - \delta_1(\varepsilon) n} ,$$
and so $$ \# \mathcal{R}_{n+1}^k(\varepsilon) \sim e^{k \delta \lambda n} .$$
\end{proposition}

\begin{proof}

Define $\tilde{R}_{n+1}(\varepsilon) := \bigsqcup_{\mathbf{a} \in \mathcal{R}_{n+1}} \text{int}_J P_{\mathbf{a}} $. From the point of view of the measure $\mu$, it is indistinguishable from $R_{n+1}(\varepsilon)$. First, we prove that
$$ \bigcap_{i=0}^{k-1} f^{-n i} \left( \tilde{R}_{n+1}(\varepsilon) \right) \subset R_{n+1}^k(\varepsilon) .$$
Let $x \in \bigcap_{i=0}^{k-1} f^{-n i} \left( \tilde{R}_{n+1}(\varepsilon) \right) $. Since there exists $\mathbf{A}=\mathbf{a}_1' \dots \mathbf{a}_{k-1}' \mathbf{a} _k\in \mathcal{W}_{kn+1}$ such that $x \in P_{\mathbf{A}}$, we see that for any $i$ we can write $f^{ni}(x) \in P_{\mathbf{a}_{1+i}' \dots \mathbf{a}_k} \cap \tilde{R}_{n+1}(\varepsilon) $.
So there exists $\mathbf{b}_{i+1} \in \mathcal{R}_{n+1}(\varepsilon)$ such that $ P_{\mathbf{a}_{1+i}' \dots \mathbf{a}_k} \cap \text{int}_J P_{\mathbf{b}_{i+1}} \neq \emptyset$. Then $\mathbf{b}_{i+1} = \mathbf{a}_{i+1}$, for all $i$, which implies that $\mathbf{A} \in \mathcal{R}_{n+1}^k(\varepsilon)$. \\
Now that the inclusion is proved, we see that $$ \mu\left( J \setminus R_{n+1}^k \right) \leq \sum_{i=0}^{k-1} \mu\left( f^{-n i} \left( J \setminus \tilde{R}_{n+1} \right) \right) $$
$$ = k \mu( J \setminus R_{n+1} ), $$
and we conclude by the previous Proposition. The cardinal estimate is done as before. \end{proof}


\section{Reduction to sums of exponentials}

We can finally begin the proof of the main Theorem 3.1.2. Recall that $f$ is a hyperbolic rational map of degree $d \geq 2$, and that $J \subset \mathbb{C}$ denotes its Julia set, which is supposed \textbf{not to be included in a circle}. Fix a small Markov partition $(P_a)_{a \in \mathcal{A}}$ and open sets $(U_a)_{a \in \mathcal{A}}$ and $(D_a)_{a \in \mathcal{A}}$ as in Proposition 3.2.2 and Proposition 3.2.4. Fix a normalized $\varphi \in C^1(V,\mathbb{R})$, and denote by $\mu$ its associated equilibrium state. Let $\chi \in C^1(\mathbb{C},\mathbb{C})$ be a compactly supported bump function, let $\xi \geq 1$ and let $\psi : \Omega \rightarrow \mathbb{R}$ be a $C^2$ phase defined on some open set $\Omega \subset \mathbb{C}$ satisfying $$ \|\psi\|_{C^2} + (\inf_{z \in \Omega} |\nabla \psi|)^{-1} \leq \xi^{\rho_1},$$ for some $\rho_1>0$ that will chosen small enough during the proof of Lemma 3.4.4 and 3.4.5. It will be convenient to see the gradient of $\psi$ as a complex number, so that we can write:
$$\forall h \in \mathbb{C}, \  (d\psi)_z(h) = \text{Re}\Big{(}\overline{\nabla \psi(z)} \ h \Big{)}. $$
Denote $K := \text{supp}(\chi)$ the support of $\chi$, and suppose that $K \cap J \neq \emptyset$, and that $K \subset \Omega$. In this case, if $n$ is large enough, any word $\mathbf{a} \in \mathcal{R}_{n+1}(K,\varepsilon)$ will then satisfy $D_\mathbf{a} \subset \Omega$, allowing us to use the lower bound $\inf_{D_{\mathbf{a}}} |\nabla \psi| \geq \xi^{-\rho_1}$. \\

Our goal is to prove power decay of the following oscillatory integral:
$$ \widehat{\psi_*(\chi d\mu)}(\xi) = \int_{J} e^{- 2 i \pi \xi \psi(z)} \chi(z) d\mu(z)  ,$$
where $\xi \in \mathbb{R}$. To do so, we will use the invariance of $\mu$ by the transfer operator $\mathcal{L}_\varphi$. This will make the maps $g_{\mathbf{a}}$ appear.
As we will be interested in intertwining blocks of words, let us introduce some notations, inspired from the one used in \cite{BD17}. For a fixed $n$ and $k$, denote:

\begin{itemize}
    \item $\textbf{A}=(\textbf{a}_0, \dots, \textbf{a}_k) \in \mathcal{W}_{n+1}^{k+1} \ , \ \textbf{B}=(\textbf{b}_1, \dots, \textbf{b}_k) \in \mathcal{W}_{n+1}^{k} $.
    \item We write $\textbf{A} \leftrightarrow \textbf{B}$ iff $\textbf{a}_{j-1} \rightsquigarrow \textbf{b}_j \rightsquigarrow \textbf{a}_j$ for all $j=1,\dots k$.
    \item If $\textbf{A} \leftrightarrow \textbf{B}$, then we define the words $\textbf{A} * \textbf{B} := \textbf{a}_0' \textbf{b}_1' \textbf{a}_1' \textbf{b}_2' \dots \textbf{a}_{k-1}' \textbf{b}_k' \textbf{a}_k$ and $\textbf{A} \# \textbf{B} :=  \textbf{a}_0' \textbf{b}_1' \textbf{a}_1' \textbf{b}_2' \dots \textbf{a}_{k-1}' \textbf{b}_k$.
    \item Denote by $b(\textbf{A}) \in \mathcal{A}$ the last letter of $\textbf{a}_k$.
    
\end{itemize}
Then, we can write:
$$ \forall x \in P_{b}, \ \mathcal{L}_\varphi^{(2k+1)n} h(x) = \underset{\mathbf{A} \rightsquigarrow b }{\sum_{\mathbf{A} \leftrightarrow \mathbf{B}}} h(g_{\mathbf{A} * \mathbf{B}}(x)) w_{\mathbf{A} * \mathbf{B}}(x) .$$
In particular, the invariance of $\mu$ under $\mathcal{L}_\varphi$ allows us to write the following formula:

$$ \widehat{\psi_*(\chi d\mu)}(\xi)  = \sum_{\mathbf{A} \leftrightarrow \mathbf{B}} \int_{P_{b(\mathbf{A})}} e^{-2 i \pi \psi( g_{\mathbf{A} * \mathbf{B}}(x))} \chi(g_{\mathbf{A}*\mathbf{B}}(x)) w_{\mathbf{A} * \mathbf{B}}(x)  d\mu(x) .$$
In this section, our goal is to relate this quantity to a well behaved sum of exponentials. To this end, we will need to introduce various parameters that will be chosen in section 3.5. Before going on, let us explain the role of these different quantities. \\

Five quantities will be at play: $\xi,n,k,\varepsilon_0$ and $\varepsilon$. The parameters $k,\varepsilon_0$ and $\varepsilon$ must be thought of as being \emph{fixed}. The integer $k$ will be chosen by an application of a version of the sum-product phenomenon (see Theorem 3.5.1). The small constant $\varepsilon_0$ will be chosen at the end of the proof of Proposition 3.6.8. The parameter $\varepsilon_0$ will be chosen small compared to $\lambda$, and $\varepsilon$ will be chosen small compared to $\varepsilon_0$, $\lambda$, $\delta$ and every other constant that might appear in the proof. The only variables are $\xi$ and $n$, but they are related. We think of $\xi$ as a large enough variable, and $n$ will be depending on $\xi$ with a relation of the form $ n \simeq \ln \xi $. We prove the following reduction.

\begin{proposition}
Define $$J_n := \{ e^{\varepsilon_0 n/2} \leq |\eta| \leq  e^{2 \varepsilon_0 n} \}$$ and $$  \zeta_{j,\mathbf{A}}(\mathbf{b}) := e^{2 \lambda n} g_{\mathbf{a}_{j-1}' \mathbf{b}}'(x_{\mathbf{a}_j}) $$
for some choice of $x_\mathbf{a} \in \text{int}_J P_{\mathbf{a}}$ for any finite admissible words $\mathbf{a}$.
There exists a constant $\alpha>0$ such that, for $ |\xi| \simeq e^{(2k+1) \lambda n} e^{\varepsilon_0 n} $ and $n$ large enough depending on $\varepsilon$:
$$ e^{- \varepsilon  \alpha n} |\widehat{\psi_*(\chi d\mu)}(\xi) |^2 \lesssim e^{-\lambda \delta (2k+1) n} \sum_{\mathbf{A} \in \mathcal{R}_{n+1}^{k+1}(\varepsilon)} \sup_{\eta \in J_n} \Bigg{|} \underset{\mathbf{A} \leftrightarrow \mathbf{B}}{\sum_{\mathbf{B} \in \mathcal{R}_{n+1}^k(\varepsilon)}} e^{2 i \pi \text{Re}\left( \eta \zeta_{1,\mathbf{A}}(\mathbf{b}_1) \dots \zeta_{k,\mathbf{A}}(\mathbf{b}_k) \right)} \Bigg| $$
$$ \quad \quad  \quad \quad  \quad \quad  \quad \quad +  e^{- \varepsilon  \alpha n} \mu(J \setminus R_{n+1}^{2k+1}(\varepsilon) )^2 + \kappa^{-2n} + e^{-(\lambda-\varepsilon_0) n/2 } + e^{- \varepsilon_0 \delta_{AD}n/4} .$$
\end{proposition}

Once Proposition 3.4.1 is established, if we manage to prove that the sum of exponentials enjoys exponential decay in $n$, then choosing $\varepsilon$ small enough will allow us to see that $|\widehat{\psi_*(\chi d\mu)}(\xi) |^2$ enjoys polynomial decay in $\xi$, and Theorem 3.1.2 will be proved. We prove Proposition 3.4.1 through a succession of lemmas.

\begin{lemma}

$$ |\widehat{\psi_*(\chi d\mu)}(\xi)|^2 \lesssim \Bigg{|} \underset{\mathbf{B} \in \mathcal{R}_{n+1}^k}{\underset{\mathbf{A} \in \mathcal{R}_{n+1}^{k+1}(K) }{\sum_{\mathbf{A} \leftrightarrow \mathbf{B}}}} \int_{P_{b(\mathbf{A})}} e^{-2 i \pi \psi( g_{\mathbf{A} * \mathbf{B}}(x))} \chi(g_{\mathbf{A}*\mathbf{B}}(x)) w_{\mathbf{A} * \mathbf{B}}(x)  d\mu(x)  \Bigg{|}^2 + \mu(J \setminus R_{n+1}^{2k+1}(\varepsilon) )^2 .$$

\end{lemma}

\begin{proof}

We have
$$ \widehat{\psi_*(\chi d\mu)}(\xi)  = \sum_{\mathbf{A} \leftrightarrow \mathbf{B}} \int_{P_{b(\mathbf{A})}} e^{-2 i \pi \psi( g_{\mathbf{A} * \mathbf{B}}(x))} \chi(g_{\mathbf{A}*\mathbf{B}}(x)) w_{\mathbf{A} * \mathbf{B}}(x)  d\mu(x) .$$
We are only interested on blocks $\mathbf{A}$ and $\mathbf{B}$ that allow us to get some control on the different quantities that will appear: those are the regular words. Furthermore, if $\mathbf{a}_0$ doesn't satify $P_{\mathbf{a}_0} \cap K \neq \emptyset$, then $\chi(g_{\mathbf{A}*\mathbf{B}}(z))=0$. It follows that:
$$ \widehat{\psi_*(\chi d\mu)}(\xi) = \underset{\mathbf{B} \in \mathcal{R}_{n+1}^k}{\underset{\mathbf{A} \in \mathcal{R}_{n+1}^{k+1}(K) }{\sum_{\mathbf{A} \leftrightarrow \mathbf{B}}}} \int_{P_{b(\mathbf{A})}} e^{-2 i \pi \psi( g_{\mathbf{A} * \mathbf{B}}(x))} \chi(g_{\mathbf{A}*\mathbf{B}}(x)) w_{\mathbf{A} * \mathbf{B}}(x)  d\mu(x) $$ $$+  \underset{\text{or} \ \mathbf{B} \notin \mathcal{R}_{n+1}^k}{\underset{\mathbf{A} \notin \mathcal{R}_{n+1}^{k+1}(K) }{\sum_{\mathbf{A} \leftrightarrow \mathbf{B}}}} \int_{P_{b(\mathbf{A})}} e^{-2 i \pi \psi( g_{\mathbf{A} * \mathbf{B}}(x))} \chi(g_{\mathbf{A}*\mathbf{B}}(x)) w_{\mathbf{A} * \mathbf{B}}(x)  d\mu(x)$$
where we see blocks in $\mathcal{R}_{n+1}^k$ as blocks in $\mathcal{W}_{n+1}^k$ in the obvious way. We can bound the contribution of the non-regular part by

$$ \Bigg{|} \underset{\text{or} \ \mathbf{B} \notin \mathcal{R}_{n+1}^k}{\underset{\mathbf{A} \notin \mathcal{R}_{n+1}^{k+1}(K) }{\sum_{\mathbf{A} \leftrightarrow \mathbf{B}}}} \int_{P_{b(\mathbf{A})}} e^{-2 i \pi \text{Re}\left( \overline{\xi} g_{\mathbf{A} * \mathbf{B}}(x) \right)} w_{\mathbf{A} * \mathbf{B}}(x) d\mu(x) \Bigg{|} \leq \sum_{\mathbf{C} \notin \mathcal{R}_{n+1}^{2k+1}} \int_{P_{b(\mathbf{C})}} w_{\mathbf{C}} d\mu $$
$$ \lesssim \sum_{\mathbf{C} \notin \mathcal{R}_{n+1}^{2k+1}} \mu(P_{\mathbf{C}} ) \leq \mu\left( J \setminus R_{n+1}^{2k+1}(\varepsilon) \right), $$
where we used the fact that $\mu$ is a Gibbs measure.
Once $\varepsilon$ will be fixed, this term will enjoy exponential decay in $n$, thanks to Proposition 3.3.8. 
\end{proof}

\begin{lemma}
There exists some constant $\alpha>0$ such that, for $n \geq n_2(\varepsilon)$ :

$$ e^{-\varepsilon \alpha n}  \Bigg{|} \underset{\mathbf{B} \in \mathcal{R}_{n+1}^k}{\underset{\mathbf{A} \in \mathcal{R}_{n+1}^{k+1}(K) }{\sum_{\mathbf{A} \leftrightarrow \mathbf{B}}}} \int_{P_{b(\mathbf{A})}} e^{-2 i \pi \psi( g_{\mathbf{A} * \mathbf{B}}(x))} \chi(g_{\mathbf{A}*\mathbf{B}}(x)) w_{\mathbf{A} * \mathbf{B}}(x)  d\mu(x) \Bigg{|}^2 \quad \quad \quad \quad \quad \quad \quad \quad \quad \quad \quad \quad $$ $$ \quad \quad \quad \quad  \lesssim e^{\lambda \delta (2k-1) n} \underset{\mathbf{B} \in \mathcal{R}_{n+1}^k}{\underset{\mathbf{A} \in \mathcal{R}_{n+1}^{k+1}(K) }{\sum_{\mathbf{A} \leftrightarrow \mathbf{B}}}} \left| \int_{P_{b(\mathbf{A})}} e^{-2 i \pi \psi( g_{\mathbf{A} * \mathbf{B}}(x) )}   w_{\mathbf{a}_k}(x) d\mu(x) \right|^2 + \kappa^{-2n}. $$

\end{lemma}

\begin{proof}

Notice that $w_{\mathbf{A} * \mathbf{B}}(x)$ and $w_{\mathbf{a}_k}(x)$ are related by
$$ w_{\mathbf{A} * \mathbf{B}}(x) = w_{\mathbf{A} \# \mathbf{B}}(g_{\mathbf{a}_k}(x)) w_{\mathbf{a}_k}(x) .$$
For each admissible word $\mathbf{a}$ of any length, fix once and for all a point $x_\mathbf{a} \in \text{int}_J P_{\mathbf{a}}$.
To get the term $w_{\mathbf{A} \# \mathbf{B}}(g_{\mathbf{a}_k}(x))$ out of the integral, we will compare it to $w_{\mathbf{A} \# \mathbf{B}}(x_{\mathbf{a}_k})$.
Recall that $\varphi$ has exponentially decreasing variations: we can write
$$ \max_{ \textbf{a} \in \mathcal{W}_n } \sup_{x,y \in P_{\textbf{a}}} |\varphi(x) - \varphi(y)| \lesssim \kappa^{-n} .$$
So we can write:
$$  \frac{w_{\mathbf{A} \# \mathbf{B}}(g_{\mathbf{a}_k}(x))}{w_{\mathbf{A} \# \mathbf{B}}(x_{\mathbf{a}_k})} = \exp \left( S_{2nk}\varphi(g_{\mathbf{A} \# \mathbf{B}}(g_{\mathbf{a}_k}(x))) - S_{2kn}\varphi( g_{\mathbf{A} \# \mathbf{B}}(x_{\mathbf{a}_k})) \right) ,$$
with 
$$ \left| S_{2nk}\varphi(g_{\mathbf{A} \# \mathbf{B}}(g_{\mathbf{a}_k}(x))) - S_{2kn}\varphi( g_{\mathbf{A} \# \mathbf{B}}(x_{\mathbf{a}_k}))  \right| \lesssim \sum_{j=0}^{2nk-1} \kappa^{- n(2k+1) + j } \lesssim \kappa^{-n} .$$
Hence, there exists some constant $C>0$ such that
$$ e^{-C \kappa^{-n}} w_{\mathbf{A} \# \mathbf{B}}(x_{\mathbf{a}_k}) \leq w_{\mathbf{A} \# \mathbf{B}}(g_{\mathbf{a}_k}(x)) \leq e^{C \kappa^{-n}} w_{\mathbf{A} \# \mathbf{B}}(x_{\mathbf{a}_k}) ,$$
which gives:
$$ \left| w_{\mathbf{A} \# \mathbf{B}}(g_{\mathbf{a}_k}(x)) -  w_{\mathbf{A} \# \mathbf{B}}(x_{\mathbf{a}_k})  \right| \leq \max\left| e^{\pm C \kappa^{-n}} -1 \right|  w_{\mathbf{A} \# \mathbf{B}}(g_{\mathbf{a}_k}(x)) \lesssim \kappa^{-n} w_{\mathbf{A} \# \mathbf{B}}(g_{\mathbf{a}_k}(x)) .$$
Furthermore, since $\chi$ also has exponentially vanishing variations:
$$ \forall x \in P_{b(\mathbf{A})}, \  |\chi(g_{\mathbf{A} * \mathbf{B}}(x)) - \chi(x_{\mathbf{a}_0})| \lesssim \kappa^{-n} .$$
Combining these estimates yields
$$ | \chi(g_{\mathbf{A} * \mathbf{B}}(x)) w_{\mathbf{A} \# \mathbf{B}}(g_{\mathbf{a}_k}(x))-\chi(x_{\mathbf{a}_0}) w_{\mathbf{A} \# \mathbf{B}}(x_{\mathbf{a}_k})| \lesssim \kappa^{-n} w_{\mathbf{A} \# \mathbf{B}}(g_{\mathbf{a}_k}(x)) .$$
From this, we get
$$ \Bigg{|} \underset{\mathbf{B} \in \mathcal{R}_{n+1}^k}{\underset{\mathbf{A} \in \mathcal{R}_{n+1}^{k+1}(K) }{\sum_{\mathbf{A} \leftrightarrow \mathbf{B}}}} \int_{P_{b(\mathbf{A})}} e^{-2 i \pi \psi(g_{\mathbf{A} * \mathbf{B}}(x))} \left(\chi(g_{\mathbf{A} * \mathbf{B}}(x)) w_{\mathbf{A} * \mathbf{B}}(x) - \chi(x_{\mathbf{a}_0})w_{\mathbf{A} \# \mathbf{B}}(x_{\mathbf{a}_k}) w_{\mathbf{a}_k}(x) \right) d\mu(x) \Bigg{|}   $$
$$ \leq \underset{\mathbf{B} \in \mathcal{R}_{n+1}^k}{\underset{\mathbf{A} \in \mathcal{R}_{n+1}^{k+1}(K) }{\sum_{\mathbf{A} \leftrightarrow \mathbf{B}}}} \int_{P_{b(\mathbf{A})}} \left| \chi(g_{\mathbf{A}*\mathbf{B}}(x)) w_{\mathbf{A} \# \mathbf{B}}(x) - \chi(x_{\mathbf{a}_0}) w_{\mathbf{A} \# \mathbf{B}}(x_{\mathbf{a}_k}) \right|w_{\mathbf{a}_k}(x)  d\mu(x) $$
$$ \lesssim \kappa^{-n} \underset{\mathbf{B} \in \mathcal{R}_{n+1}^k}{\underset{\mathbf{A} \in \mathcal{R}_{n+1}^{k+1}(K) }{\sum_{\mathbf{A} \leftrightarrow \mathbf{B}}}} \int_{P_{b(\mathbf{A})}}  w_{\mathbf{A} * \mathbf{B}}(x)  d\mu(x) \ \lesssim \kappa^{-n} $$
by the Gibbs estimates.
Moreover, by Cauchy-Schwartz,
$$ \Bigg{|} \underset{\mathbf{B} \in \mathcal{R}_{n+1}^k}{\underset{\mathbf{A} \in \mathcal{R}_{n+1}^{k+1}(K) }{\sum_{\mathbf{A} \leftrightarrow \mathbf{B}}}} \int_{P_{b(\mathbf{A})}} e^{-2 i \pi \psi( g_{\mathbf{A} * \mathbf{B}}(x) )}  \chi(x_{\mathbf{a}_0}) w_{\mathbf{A} \# \mathbf{B}}(x_{\mathbf{a}_k}) w_{\mathbf{a}_k}(x) d\mu(x) \Bigg{|}^2   $$
$$ = \Bigg{|} \underset{\mathbf{B} \in \mathcal{R}_{n+1}^k}{\underset{\mathbf{A} \in \mathcal{R}_{n+1}^{k+1}(K) }{\sum_{\mathbf{A} \leftrightarrow \mathbf{B}}}} \chi(x_{\mathbf{a}_0})  w_{\mathbf{A} \# \mathbf{B}}(x_{\mathbf{a}_k}) \int_{P_{b(\mathbf{A})}} e^{-2 i \pi \psi( g_{\mathbf{A} * \mathbf{B}}(x))}   w_{\mathbf{a}_k}(x) d\mu(x) \Bigg{|}^2  $$
$$ \leq  \underset{\mathbf{B} \in \mathcal{R}_{n+1}^k}{\underset{\mathbf{A} \in \mathcal{R}_{n+1}^{k+1}(K) }{\sum_{\mathbf{A} \leftrightarrow \mathbf{B}}}} |\chi(x_{\mathbf{a}_0})|^2 w_{\mathbf{A} \# \mathbf{B}}(x_{\mathbf{a}_k})^2 \underset{\mathbf{B} \in \mathcal{R}_{n+1}^k}{\underset{\mathbf{A} \in \mathcal{R}_{n+1}^{k+1}(K) }{\sum_{\mathbf{A} \leftrightarrow \mathbf{B}}}} \left| \int_{P_{b(\mathbf{A})}} e^{-2 i \pi \psi( g_{\mathbf{A} * \mathbf{B}}(x) )}   w_{\mathbf{a}_k}(x) d\mu(x) \right|^2 $$
$$  \lesssim e^{\varepsilon \alpha n} e^{-\lambda \delta (2k-1) n} \underset{\mathbf{B} \in \mathcal{R}_{n+1}^k}{\underset{\mathbf{A} \in \mathcal{R}_{n+1}^{k+1}(K) }{\sum_{\mathbf{A} \leftrightarrow \mathbf{B}}}} \left| \int_{P_{b(\mathbf{A})}} e^{-2 i \pi \psi( g_{\mathbf{A} * \mathbf{B}}(x))}   w_{\mathbf{a}_k}(x) d\mu(x) \right|^2 ,$$
by Proposition 3.3.3 and 3.3.8, where one could increase $\alpha$ if necessary.
\end{proof}

\begin{lemma}
Define $$ \zeta_{j,\mathbf{A}}(\mathbf{b}) = e^{2 \lambda n} g_{\mathbf{a}_{j-1}' \mathbf{b}}'(x_{\mathbf{a}_j})$$ and $$\eta_\mathbf{A}(x,y) := \overline{\nabla \psi(x_{\mathbf{a}_0})} \xi \left( g_{\mathbf{a}_k}(x) - g_{\mathbf{a}_k}(y) \right) e^{-2 k \lambda n} \in \mathbb{C}.$$
There exists $\alpha>0$ such that, for $ |\xi| \simeq e^{(2k+1) \lambda n} e^{\varepsilon_0 n} $ and $n$ large enough depending on $\varepsilon$:
$$ e^{-\varepsilon \alpha n} e^{-\lambda \delta (2k-1) n} \underset{\mathbf{B} \in \mathcal{R}_{n+1}^k}{\underset{\mathbf{A} \in \mathcal{R}_{n+1}^{k+1}(K) }{\sum_{\mathbf{A} \leftrightarrow \mathbf{B}}}} \left| \int_{P_{b(\mathbf{A})}} e^{-2 i \pi \psi( g_{\mathbf{A} * \mathbf{B}}(x) )}   w_{\mathbf{a}_k}(x) d\mu(x) \right|^2 $$
$$ \lesssim e^{-\lambda \delta (2k+1) n} \sum_{\mathbf{A} \in \mathcal{R}_{n+1}^{k+1}(K)} \iint_{P_{b(\mathbf{A})}^2 }  \Bigg{|} \underset{\mathbf{A} \leftrightarrow \mathbf{B}}{\sum_{\mathbf{B} \in \mathcal{R}_{n+1}^k}} e^{2 i \pi \text{Re}\left( \eta_{\mathbf{A}}(x,y) \zeta_{1,\mathbf{A}}(\mathbf{b}_1) \dots \zeta_{k,\mathbf{A}}(\mathbf{b}_k) \right)} \Bigg| d\mu^2(x,y) +  e^{-(\lambda-\varepsilon_0) n/2 } .$$

\end{lemma}

\begin{proof}
Remember that we allowed $\|\psi\|_{C^2} + (\inf_{x \in \Omega}|\nabla \psi|)^{-1}$ to grow like $\xi^{\rho_1}$ for some constant $\rho_1$ that we can choose as small as we want. We will carefully make these quantities appear in the computations of this lemma and in the next. We expand the integral term and use Proposition 3.3.3 to get

$$ e^{-\lambda \delta (2k-1) n} \underset{\mathbf{B} \in \mathcal{R}_{n+1}^k}{\underset{\mathbf{A} \in \mathcal{R}_{n+1}^{k+1}(K) }{\sum_{\mathbf{A} \leftrightarrow \mathbf{B}}}} \iint_{P_{b(\mathbf{A})}^2}  e^{2 i \pi \xi \left( \psi( g_{\mathbf{A} * \mathbf{B}}(x) ) - \psi( g_{\mathbf{A} * \mathbf{B}}(y) ) \right)}   w_{\mathbf{a}_k}(x) w_{\mathbf{a}_k}(y) d\mu^2(x,y)  $$
$$ \leq e^{-\lambda \delta (2k-1) n} \sum_{\mathbf{A} \in \mathcal{R}_{n+1}^{k+1}(K)} \iint_{P_{b(\mathbf{A})}^2} w_{\mathbf{a}_k}(x) w_{\mathbf{a}_k}(y) \Bigg{|} \underset{\mathbf{A} \leftrightarrow \mathbf{B}}{\sum_{\mathbf{B} \in \mathcal{R}_{n+1}^k}} e^{2 i \pi \xi \left( \psi(g_{\mathbf{A} * \mathbf{B}}(x)) -  \psi(g_{\mathbf{A} * \mathbf{B}}(y)) \right)} \Bigg| d\mu^2(x,y) $$
$$ \lesssim e^{\varepsilon \alpha n}  e^{-\lambda \delta (2k+1) n} \sum_{\mathbf{A} \in \mathcal{R}_{n+1}^{k+1}(K)} \iint_{P_{b(\mathbf{A})}^2}  \Bigg{|} \underset{\mathbf{A} \leftrightarrow \mathbf{B}}{\sum_{\mathbf{B} \in \mathcal{R}_{n+1}^k}} e^{2 i \pi \xi \left( \psi(g_{\mathbf{A} * \mathbf{B}}(x)) -  \psi(g_{\mathbf{A} * \mathbf{B}}(y)) \right)} \Bigg| d\mu(x) d\mu(y) .$$
The next step is to carefully linearize the phase. Here again, the construction of the $(D_a)_{a \in \mathcal{A}}$ as convex sets is really useful. \\

Fix some $\mathbf{A} \in \mathcal{R}_{n+1}^{k+1}$. For $x,y \in P_{b(\mathbf{A})}$, set $ \widehat{x} := g_{\mathbf{a}_k}(x)$ and $\widehat{y} = g_{\mathbf{a}_k}(y)$.
These are elements of $P_{b(\mathbf{B})}$, and so $[\widehat{x},\widehat{y}] \subset D_{b(\mathbf{B})}$. Hence, the following identity makes sense:
$$ \psi(g_{\mathbf{A} * \mathbf{B}}(x)) -  \psi(g_{\mathbf{A} * \mathbf{B}}(y)) = \psi(g_{\mathbf{A} \# \mathbf{B}}(\widehat{x})) - \psi( g_{\mathbf{A} \# \mathbf{B}}(\widehat{y})) = \text{Re} \int_{[\widehat{x},\widehat{y}]} \Big{(}\overline{\nabla \psi(g_{\mathbf{A} \# \mathbf{B}}(z))} g_{\mathbf{A} \# \mathbf{B}}'(z)\Big{)} dz.$$
Therefore, we get
$$ \left| \psi(g_{\mathbf{A} * \mathbf{B}}(x)) -  \psi(g_{\mathbf{A} * \mathbf{B}}(y)) - \text{Re}\Big{(} \overline{\nabla \psi(x_{\mathbf{a}_0})} g_{\mathbf{A} \# \mathbf{B}}'(x_{\mathbf{a}_k})(\widehat{x}-\widehat{y}) \Big{)} \right| $$ $$ = \left| \text{Re} \int_{[\widehat{x},\widehat{y}]} \left( \overline{\nabla \psi(x_{\mathbf{a}_0})} g_{\mathbf{A} \# \mathbf{B}}'(z) - \overline{\nabla \psi(g_{\mathbf{A} \# \mathbf{B}}(z))}  g_{\mathbf{A} \# \mathbf{B}}'(x_{\mathbf{a}_k}) \right) dz \right| $$
$$ = \left| \int_{[\widehat{x},\widehat{y}]} \left( \overline{\nabla \psi(x_{\mathbf{a}_0})} \Big(g_{\mathbf{A} \# \mathbf{B}}'(z) - g_{\mathbf{A} \# \mathbf{B}}'(x_{\mathbf{a}_k})\Big) - \Big(\overline{\nabla \psi(g_{\mathbf{A} \# \mathbf{B}}(z))} - \overline{\nabla \psi(x_{\mathbf{a}_0})} \Big) g_{\mathbf{A} \# \mathbf{B}}'(x_{\mathbf{a}_k}) \right) dz \right| $$
$$ \lesssim \| \psi \|_{C^2} \left| \int_{[\widehat{x},\widehat{y}]} \left(g_{\mathbf{A} \# \mathbf{B}}'(z) - g_{\mathbf{A} \# \mathbf{B}}'(x_{\mathbf{a}_k}) \right) dz \right| + \| \psi \|_{C^2} e^{\varepsilon \alpha n} e^{- (2k+2) \lambda n} ,$$
using Proposition 3.3.3. Notice that if $z \in [\widehat{x},\widehat{y}]$, we have $z \in D_{b(\mathbf{B})}$, and so $[z,x_{\mathbf{a}_k}] \subset D_{b(\mathbf{B})}$, and the following is well defined:

$$ g_{\mathbf{A} \# \mathbf{B}}'(z) - g_{\mathbf{A} \# \mathbf{B}}'(x_{\mathbf{a}_k}) = \int_{[z,x_{\mathbf{a}_k}]} g_{\mathbf{A} \# \mathbf{B}}''(\omega) d\omega .$$
Now, notice that since the maps are holomorphic, and since there exists $\beta>0$ such that $ P_a + B(0,2 \beta) \subset D_a $ for any $a \in \mathcal{A}$, we can write by Cauchy's formula
$$ \left| g_{\mathbf{A} \# \mathbf{B}}''(\omega) \right| = \left|\frac{1}{2 i \pi} \oint_{\mathcal{C}(\omega,\beta)} \frac{g_{\mathbf{A} \# \mathbf{B}}'(s)}{(s-\omega)^2} ds \right| \leq \frac{\|g_{\mathbf{A} \# \mathbf{B}}'\|_{\infty,\mathcal{C}(\omega,\beta)}}{\beta}  ,$$
where $\mathcal{C}(\omega,\beta)$ is a circle centered at $\omega$ with radius $\beta$.
And so 

$$|g_{\mathbf{A} \# \mathbf{B}}'(z) - g_{\mathbf{A} \# \mathbf{B}}'(x_{\mathbf{a}_k})| \lesssim \|g_{\mathbf{A} \# \mathbf{B} }' \|_{\infty,D_{b(\mathbf{B})}} |z-x_{\mathbf{a}_k}| \lesssim e^{\varepsilon \alpha n} e^{- 2 k \lambda n } e^{- \lambda n} $$
by Proposition 3.3.3. Hence,
$$ \left| \psi(g_{\mathbf{A} * \mathbf{B}}(x)) -  \psi(g_{\mathbf{A} * \mathbf{B}}(y)) - \text{Re}\Big{(} \overline{\nabla \psi(x_{\mathbf{a}_0})} g_{\mathbf{A} \# \mathbf{B}}'(x_{\mathbf{a}_k})(\widehat{x}-\widehat{y}) \Big{)} \right|  \lesssim \|\psi\|_{C^2} e^{\varepsilon \alpha n} e^{- (2 k + 2) \lambda n } .$$
Then we relate $g_{\mathbf{A} \# \mathbf{B}}'(x_{\mathbf{a}_k})$ to $g_{\mathbf{a}_0' \mathbf{b}_1}'(x_{\mathbf{a}_1}) \dots g_{\mathbf{a}_{k-1}' \mathbf{b}_k}'(x_{\mathbf{a}_k}) $, using Cauchy's formula again:
$$ \left| g_{\mathbf{A} \# \mathbf{B}}'(x_{\mathbf{a}_k}) - g_{\mathbf{a}_0' \mathbf{b}_1}'(x_{\mathbf{a}_1}) \dots g_{\mathbf{a}_{k-1}' \mathbf{b}_k}'(x_{\mathbf{a}_k}) \right| =  \left| \prod_{j=1}^k g_{\mathbf{a}_{j-1}'\mathbf{b}_j}'(g_{\mathbf{a}_j' \mathbf{b}_{j+1}' \dots \mathbf{a}_{k-1}' \mathbf{b}_k}(x_{\mathbf{a}_{k}})) - \prod_{j=1}^k g_{\mathbf{a}_{j-1}'\mathbf{b}_j}'(x_{\mathbf{a}_{j}}) \right|  $$
$$ \leq \sum_{i=0}^{k-1} \left| \prod_{j=1}^i g_{\mathbf{a}_{j-1}'\mathbf{b}_j}'(x_{\mathbf{a}_{j}})  \prod_{j=i+1}^k g_{\mathbf{a}_{j-1}'\mathbf{b}_j}'(g_{\mathbf{a}_j' \mathbf{b}_{j+1}' \dots \mathbf{b}_k}(x_{\mathbf{a}_{k}})) - \prod_{j=1}^{i+1} g_{\mathbf{a}_{j-1}'\mathbf{b}_j}'(x_{\mathbf{a}_{j}}) \prod_{j=i+2}^k g_{\mathbf{a}_{j-1}'\mathbf{b}_j}'(g_{\mathbf{a}_j' \mathbf{b}_{j+1}' \dots \mathbf{b}_k}(x_{\mathbf{a}_{k}}))   \right| $$
$$ \leq \sum_{i=0}^{k-1} e^{\varepsilon \alpha n} e^{-2 (k-1) \lambda n} | g_{\mathbf{a}_{i}'\mathbf{b}_{i+1}}'(g_{\mathbf{a}_{i+1}' \mathbf{b}_{i+2}' \dots \mathbf{b}_k}(x_{\mathbf{a}_{k}})) - g_{\mathbf{a}_{i}'\mathbf{b}_{i+1}}'(x_{\mathbf{a}_{i+1}}))| $$
$$ \lesssim e^{\varepsilon \alpha n} e^{-2 \lambda  (k-1) n} \|g_{\mathbf{a}_{i}' \mathbf{b}_{i+1}}'\|_\infty \text{diam}(P_{\mathbf{a}_{i+1}}) \lesssim e^{\varepsilon \alpha n} e^{ -(2k+1) \lambda n}  .$$
Hence, since $|\widehat{x} - \widehat{y}| \lesssim e^{\varepsilon \alpha n} e^{- \lambda n}$,
$$ \left| \psi(g_{\mathbf{A} * \mathbf{B}}(x)) - \psi(g_{\mathbf{A} * \mathbf{B}}(y)) - \text{Re}\Big( (\widehat{x}-\widehat{y})  \overline{\nabla \psi(x_{\mathbf{a}_0})} \prod_{j=1}^k g_{\mathbf{a}_{j-1}'\mathbf{b}_j}'(x_{\mathbf{a}_{j}}) \Big) \right| \lesssim  \|\psi\|_{C^2} e^{\varepsilon \alpha n} e^{ -(2k+2) \lambda n} .$$
From this estimate, we can relate our problem to a linearized one, as follows:
 $$ \scalemath{0.9}{ e^{-\lambda \delta (2k+1) n} \sum_{\mathbf{A} \in \mathcal{R}_{n+1}^{k+1}(K)} \iint_{P_{b(\mathbf{A})^2}}  \Bigg{|} \underset{\mathbf{A} \leftrightarrow \mathbf{B}}{\sum_{\mathbf{B} \in \mathcal{R}_{n+1}^k}} e^{2 i \pi \xi \left( \psi( g_{\mathbf{A} * \mathbf{B}}(x)) -  \psi(g_{\mathbf{A} * \mathbf{B}}(y)) \right) } - e^{2 i \pi \xi \text{Re}( \overline{\nabla \psi(x_{\mathbf{a}_0})}    g_{\mathbf{a}_{0}'\mathbf{b}_1}'(x_{\mathbf{a}_{1}}) \dots g_{\mathbf{a}_{k-1}'\mathbf{b}_k}'(x_{\mathbf{a}_{k}}) (\widehat{x}-\widehat{y}) )} \Bigg| d\mu^2(x,y) }$$
$$\scalemath{0.9}{ \leq e^{-\lambda \delta (2k+1) n} \sum_{\mathbf{A} \in \mathcal{R}_{n+1}^{k+1}(K)} \iint_{P_{b(\mathbf{A})^2}} \underset{\mathbf{A} \leftrightarrow \mathbf{B}}{\sum_{\mathbf{B} \in \mathcal{R}_{n+1}^k}} \Bigg{|}  e^{2 i \pi \xi \left( \psi(g_{\mathbf{A} * \mathbf{B}}(x)) - \psi(g_{\mathbf{A} * \mathbf{B}}(y)) - \text{Re}\Big( (\widehat{x}-\widehat{y})  \overline{\nabla \psi(x_{\mathbf{a}_0})} \prod_{j=1}^k g_{\mathbf{a}_{j-1}'\mathbf{b}_j}'(x_{\mathbf{a}_{j}}) \Big) \right)} - 1 \Bigg| d\mu^2(x,y) }$$
$$\scalemath{0.9}{ \lesssim e^{-\lambda \delta (2k+1) n} \sum_{\mathbf{A} \in \mathcal{R}_{n+1}^{k+1}(K)} \iint_{P_{b(\mathbf{A})^2}} \underset{\mathbf{A} \leftrightarrow \mathbf{B}}{\sum_{\mathbf{B} \in \mathcal{R}_{n+1}^k}}  | \xi| \Big| \psi(g_{\mathbf{A} * \mathbf{B}}(x)) - \psi(g_{\mathbf{A} * \mathbf{B}}(y)) - \text{Re}\Big( (\widehat{x}-\widehat{y})  \overline{\nabla \psi(x_{\mathbf{a}_0})} \prod_{j=1}^k g_{\mathbf{a}_{j-1}'\mathbf{b}_j}'(x_{\mathbf{a}_{j}}) \Big) \Big| d\mu^2(x,y) }$$

$$ \lesssim \|\psi\|_{C^2} e^{ \varepsilon \alpha n} e^{-\lambda \delta (2k+1) n} e^{(k+1) \delta \lambda n} e^{k \delta \lambda n}  | \xi|  e^{- (2 k + 2) \lambda n } \simeq  \|\psi\|_{C^2} |\xi| e^{\varepsilon \alpha n} e^{-(2k+2)\lambda n}. $$
Now we see why we need to relate $n$ to $\xi$. We follow \cite{SS20} and fix

$$ e^{(2k+1) \lambda (n-1)} e^{\varepsilon_0 (n-1)} \leq |\xi| \leq e^{(2k+1) \lambda n} e^{\varepsilon_0 n} .$$
This choice will ensure that the normalized phase $\eta$ will grow at a slow pace, of order of magnitude $e^{\varepsilon_0 n}$. It will be useful in section 3.6, where we will finally adjust $\varepsilon_0$.
This relationship being fixed from now on, we get:
$$ \scalemath{0.9}{ e^{-\lambda \delta (2k+1) n} \sum_{\mathbf{A} \in \mathcal{R}_{n+1}^{k+1}(K)} \iint_{P_{b(\mathbf{A})^2}}  \Bigg{|} \underset{\mathbf{A} \leftrightarrow \mathbf{B}}{\sum_{\mathbf{B} \in \mathcal{R}_{n+1}^k}} e^{2 i \pi \xi \left( \psi( g_{\mathbf{A} * \mathbf{B}}(x)) -  \psi(g_{\mathbf{A} * \mathbf{B}}(y)) \right) } - e^{2 i \pi \xi \text{Re}( \overline{\nabla \psi(x_{\mathbf{a}_0})}    g_{\mathbf{a}_{0}'\mathbf{b}_1}'(x_{\mathbf{a}_{1}}) \dots g_{\mathbf{a}_{k-1}'\mathbf{b}_k}'(x_{\mathbf{a}_{k}}) (\widehat{x}-\widehat{y}) )} \Bigg| d\mu^2(x,y) }$$ $$ \lesssim e^{\varepsilon \alpha n} \|\psi\|_{C^2} e^{-(\lambda-\varepsilon_0) n} \lesssim e^{-(\lambda-\varepsilon_0)n/2}$$ if $\rho_1>0$ is chosen so small that $\|\psi\|_{C^2} \leq \xi^{\rho_1} \lesssim e^{-(\lambda-\varepsilon_0) n/4}$.
Now we can focus on the term 
$$ e^{\varepsilon \alpha n}  e^{-\lambda \delta (2k+1) n} \sum_{\mathbf{A} \in \mathcal{R}_{n+1}^{k+1}(K)} \iint_{P_{b(\mathbf{A})}}  \Bigg{|} \underset{\mathbf{A} \leftrightarrow \mathbf{B}}{\sum_{\mathbf{B} \in \mathcal{R}_{n+1}^k}} e^{2 i \pi \text{Re}\left( \xi \overline{\nabla \psi(x_{\mathbf{a}_0})} g_{\mathbf{a}_{0}'\mathbf{b}_1}'(x_{\mathbf{a}_{1}}) \dots g_{\mathbf{a}_{k-1}'\mathbf{b}_k}'(x_{\mathbf{a}_{k}}) \left( \widehat{x} - \widehat{y} \right) \right)} \Bigg| d\mu^2(x,y) .$$
We re-scale the phase by defining, for any $\mathbf{b}$ such that $\mathbf{a}_{j-1} \rightsquigarrow \mathbf{b} \rightsquigarrow \mathbf{a}_j$:
$$ \zeta_{j,\mathbf{A}}(\mathbf{b}) = e^{2 \lambda n} g_{\mathbf{a}_{j-1}' \mathbf{b}}'(x_{\mathbf{a}_j}) ,$$
so that $\zeta_{j,\mathbf{A}}(\mathbf{b}) \sim 1$. We can then write
$$ \xi \overline{\nabla \psi(x_{\mathbf{a}_0})} g_{\mathbf{a}_{0}'\mathbf{b}_1}'(x_{\mathbf{a}_{1}}) \dots g_{\mathbf{a}_{k-1}'\mathbf{b}_k}'(x_{\mathbf{a}_{k}}) (\widehat{x}-\widehat{y}) = \eta_{\mathbf{A}}(x,y) \zeta_{1,\mathbf{A}}(\mathbf{b}_1) \dots \zeta_{k,\mathbf{A}}(\mathbf{b}_k),$$
where $\eta_{\mathbf{A}}(x,y) :=  \xi \overline{\nabla \psi(x_{\mathbf{a}_0})} \left( \widehat{x} - \widehat{y} \right) e^{-2 k \lambda n}   $.
\end{proof}

\begin{lemma}
Define $$ J_n := \{\eta \in \mathbb{C}, \  e^{\varepsilon_0 n/2} \leq |\eta| \leq e^{2 \varepsilon_0 n}  \} .$$
There exists $\alpha > 0$ such that, for $ |\xi| \simeq e^{(2k+1) \lambda n} e^{\varepsilon_0 n} $ and $n$ large enough depending on $\varepsilon$:
$$e^{-\varepsilon \alpha n} e^{-\lambda \delta (2k+1) n} \sum_{\mathbf{A} \in \mathcal{R}_{n+1}^{k+1}(K)} \iint_{P_{b(\mathbf{A})}^2}  \Bigg{|} \underset{\mathbf{A} \leftrightarrow \mathbf{B}}{\sum_{\mathbf{B} \in \mathcal{R}_{n+1}^k}} e^{2 i \pi \text{Re}\left( \eta_{\mathbf{A}}(x,y) \zeta_{1,\mathbf{A}}(\mathbf{b}_1) \dots \zeta_{k,\mathbf{A}}(\mathbf{b}_k) \right)} \Bigg| d\mu(x) d\mu(y)  $$
$$ \lesssim  e^{-\lambda \delta (2k+1) n} \sum_{\mathbf{A} \in \mathcal{R}_{n+1}^{k+1}} \sup_{\eta \in J_n} \Bigg{|} \underset{\mathbf{A} \leftrightarrow \mathbf{B}}{\sum_{\mathbf{B} \in \mathcal{R}_{n+1}^k}} e^{2 i \pi \text{Re}\left( \eta \zeta_{1,\mathbf{A}}(\mathbf{b}_1) \dots \zeta_{k,\mathbf{A}}(\mathbf{b}_k) \right)} \Bigg|  + e^{- \delta_{AD} \varepsilon_0 n/4} .$$

\end{lemma}

\begin{proof}

To estimate $\eta_\mathbf{A}(x,y)$, we need to control $|\widehat{x}-\widehat{y}|$. A first inequality is easy to get by convexity of $D_{b(\mathbf{A})}$:
$$ |\widehat{x}-\widehat{y}| = \left| \int_{[x,y]} g_{\mathbf{a}}'(z) dz \right| \leq e^{\varepsilon n} e^{- \lambda n} |x-y|. $$
The other inequality is subtler to get, but we already did the hard work. Recall Lemma 3.2.6: it tells us that $\text{Conv}(P_{\mathbf{a}_k}) \subset D_{\mathbf{a}_k}$, thanks to Koebe's quarter theorem. Consequently,  $[\widehat{x},\widehat{y}] \subset D_{\mathbf{a}_k}$, and so we can write safely that
$$ |x-y| = \left| \int_{[\widehat{x},\widehat{y}]} (f^n)'(z) dz \right| \leq e^{\varepsilon n} e^{\lambda n} |\widehat{x} - \widehat{y}| .$$
Combining this with the fact that $\zeta_{j,\mathbf{A}}(\mathbf{b}) \sim 1$ gives us the following estimate for $\eta$ : 
$$ \forall x,y \in P_{b(\mathbf{A})}, \ |\nabla \psi(x_{\mathbf{a}_0})| e^{- \varepsilon (4k+1) n} e^{\varepsilon_0 n} |x-y| \leq |\eta_{\mathbf{A}}(x,y)| \leq |\nabla \psi(x_{\mathbf{a}_0})| e^{\varepsilon (4k+1) n} e^{\varepsilon_0 n} |x-y| .$$
Now notice that, if $\rho_1$ is chosen small enough, we can suppose that we have a bound $ e^{-\varepsilon_0 n/4}\leq \xi^{-\rho_1} \leq |\nabla \psi(x_\mathbf{a})| \leq \xi^{\rho_1} \leq e^{\varepsilon_0 n/4} $ for all $\mathbf{a}_0 \in \mathcal{R}_{n+1}(K)$.
Choosing $\varepsilon$ and our Markov partition small enough then ensures that 
$$ \forall x,y \in P_{b(\mathbf{A})}, \  e^{- \varepsilon (4k+1) n} e^{3 \varepsilon_0 n/4} |x-y| \leq |\eta_{\mathbf{A}}(x,y)| \leq  e^{2 \varepsilon_0 n} .$$
Finally, we need to reduce our problem to the case where $\eta$ is not too small, so that the sum of exponential may enjoys some cancellations. For this, we use the upper regularity of $\mu$. We have, if $x \in P_{b(\mathbf{A})}$:
$$ \mu\left( \left\{ y \in J, \ |x-y| \leq  e^{\varepsilon (4k+1) n} e^{- \varepsilon_0 n/4} \right\} \right) \lesssim e^{\varepsilon (4k+1) n \delta_{AD} } e^{- \varepsilon_0 \delta_{AD}n/4} .$$
Integrating in $x$ yields 
$$ \mu \otimes \mu\left( \left\{ x,y \in J, \ |x-y| \leq  e^{\varepsilon (4k+1) n} e^{- \varepsilon_0 n/4} \right\} \right) \lesssim e^{\varepsilon (4k+1) n \delta_{AD} } e^{- \varepsilon_0 \delta_{AD}n/4} .$$
This allows us to reduce our principal integral term to the part where $\eta$ is not small. Indeed, we get:
$$ e^{-\lambda \delta (2k+1) n} \sum_{\mathbf{A} \in \mathcal{R}_{n+1}^{k+1}(K)} \iint_{P_{b(\mathbf{A})}^2}  \Bigg{|} \underset{\mathbf{A} \leftrightarrow \mathbf{B}}{\sum_{\mathbf{B} \in \mathcal{R}_{n+1}^k}} e^{2 i \pi \text{Re}\left( \eta_{\mathbf{A}}(x,y) \zeta_{1,\mathbf{A}}(\mathbf{b}_1) \dots \zeta_{k,\mathbf{A}}(\mathbf{b}_k) \right)} \Bigg| d\mu(x) d\mu(y)  $$
$$ \lesssim e^{-\lambda \delta (2k+1) n} \sum_{\mathbf{A} \in \mathcal{R}_{n+1}^{k+1}(K)} \iint_{ \{ x,y \in P_{b(\mathbf{A})}, \ |x-y| >  e^{\varepsilon (4k+1) n} e^{- \varepsilon_0 n/4} \} }  \Bigg{|} \underset{\mathbf{A} \leftrightarrow \mathbf{B}}{\sum_{\mathbf{B} \in \mathcal{R}_{n+1}^k}} e^{2 i \pi \text{Re}\left( \eta_\mathbf{A}(x,y) \zeta_{1,\mathbf{A}}(\mathbf{b}_1) \dots \zeta_{k,\mathbf{A}}(\mathbf{b}_k) \right)} \Bigg| d\mu(x) d\mu(y)  $$
$$ + e^{\varepsilon \alpha n} e^{\varepsilon n (4k+1) \delta_{AD} } e^{- \varepsilon_0 \delta_{AD}n/4} , $$
by the cardinality estimate on $\mathcal{R}_{n+1}^{k}$.
In this integral term, we get $$  |\eta_\mathbf{A}(x,y)| \geq  e^{- \varepsilon (4k+1) n} e^{3\varepsilon_0 n/4} e^{\varepsilon (4k+1) n} e^{- \varepsilon_0 n/4} = e^{\varepsilon_0 n/2}, $$
and so we can bound:
$$  e^{-\lambda \delta (2k+1) n} \sum_{\mathbf{A} \in \mathcal{R}_{n+1}^{k+1}(K)} \iint_{ \{ x,y \in P_{b(\mathbf{A})}, \ |x-y| >  e^{\varepsilon (4k+1) n} e^{- \varepsilon_0 n/4} \} }  \Bigg{|} \underset{\mathbf{A} \leftrightarrow \mathbf{B}}{\sum_{\mathbf{B} \in \mathcal{R}_{n+1}^k}} e^{2 i \pi \text{Re}\left( \eta_{\mathbf{A}}(x,y) \zeta_{1,\mathbf{A}}(\mathbf{b}_1) \dots \zeta_{k,\mathbf{A}}(\mathbf{b}_k) \right)} \Bigg| d\mu(x) d\mu(y)  $$
$$ \leq e^{-\lambda \delta (2k+1) n} \sum_{\mathbf{A} \in \mathcal{R}_{n+1}^{k+1}} \sup_{\eta \in J_n} \Bigg{|} \underset{\mathbf{A} \leftrightarrow \mathbf{B}}{\sum_{\mathbf{B} \in \mathcal{R}_{n+1}^k}} e^{2 i \pi \text{Re}\left( \eta \zeta_{1,\mathbf{A}}(\mathbf{b}_1) \dots \zeta_{k,\mathbf{A}}(\mathbf{b}_k) \right)} \Bigg|  .$$ \end{proof}
Combining all those lemmas gives us Proposition 3.4.1.


\section{The sum-product phenomenon}

The version of the sum-product phenomenon that we will use is a corollary of a version found in Li's work, see \cite{LNP19}, or \cite{Li18}. Similar results can be found in Section 1.3.2.

\begin{theorem}
Fix $0 < \gamma < 1$. Then there exist $k \in \mathbb{N}^*$, $c>0$ and $\varepsilon_1 > 0$ depending only on $\gamma$ such that the following holds for $\eta \in \mathbb{C}$ with $|\eta|$ large enough. Let $1 < R < |\eta|^{\varepsilon_1}$ , $N > 1$ and $\mathcal{Z}_1, \dots , \mathcal{Z}_k$ be finite sets with $ \# \mathcal{Z}_j \leq RN$. Consider some maps $\zeta_j : \mathcal{Z}_j \rightarrow \mathbb{C} $, $j = 1, \dots , k$, such that, for all $j$:
$$ \zeta_j ( \mathcal{Z}_j ) \subset \{ z \in \mathbb{C} \ , \ R^{-1} \leq |z| \leq R \}$$ and 
$$\forall \sigma \in [|\eta|^{-2}, |\eta|^{- \varepsilon_1}], \quad \sup_{a,\theta \in \mathbb{R}} \ \# \{\mathbf{b} \in \mathcal{Z}_j , \ | \text{Re} (e^{i\theta} \zeta_j(\mathbf{b})) - a|  \leq \sigma \} \leq N \sigma^{\gamma}. \quad  (*)$$
Then: 
$$ \left| N^{-k} \sum_{\mathbf{b}_1 \in \mathcal{Z}_1,\dots,\mathbf{b}_k \in \mathcal{Z}_k} \exp\left(2 i \pi \text{Re}\left(\eta \zeta_1(\mathbf{b}_1) \dots \zeta_k(\mathbf{b}_k) \right) \right) \right| \leq c |\eta|^{-{\varepsilon_1}}.$$
\end{theorem}

Our goal is to use Theorem 3.5.1 on the maps $\zeta_{j,\mathbf{A}}$.
Let's carefully define the framework. \\
For some fixed $\mathbf{A} \in \mathcal{R}_{n+1}^{k+1}$, define for $j=1, \dots, k$  $$ \mathcal{Z}_j := \{ \mathbf{b} \in \mathcal{R}_{n+1} , \mathbf{a}_{j-1} \rightsquigarrow \mathbf{b} \rightsquigarrow \mathbf{a}_j \  \} .$$
The maps $\zeta_{j,\mathbf{A}}(\mathbf{b}) := e^{2 \lambda n} g_{\mathbf{a}_{j-1}' \mathbf{b}}'(x_{{\mathbf{a}}_j})$ are defined on $\mathcal{Z}_j$. There exists a constant $\alpha>0$ (which will be fixed from now on) such that
$$\# \mathcal{Z}_j \leq e^{\varepsilon \alpha n} e^{ \delta \lambda n} $$
and
$$ \zeta_{j,\mathbf{A}}( \mathcal{Z}_j ) \subset \{ z \in \mathbb{C} \ , \ e^{-\varepsilon \alpha n} \leq |z| \leq e^{\varepsilon \alpha n} \} .$$
Let $\gamma>0$ small enough. Theorem 3.5.1 then fixes $k$ and some $\varepsilon_1$.
The goal is to apply Theorem 3.5.1 to the maps $\zeta_{j,\mathbf{A}}$, for $N := e^{\lambda \delta n}$, $R:=e^{\varepsilon \alpha n}$ and $\eta \in J_n$. Notice that choosing $\varepsilon$ small enough ensures that $R<|\eta|^{\varepsilon_1}$, and taking $n$ large enough ensures that $|\eta|$ is large. 
If we are able to prove the non-concentration hypothesis $(*)$ in this context, then Theorem 3.5.1 can be applied and we would be able to conclude the proof of the main Theorem 3.1.2.
Indeed, we already know that
$$ e^{- \varepsilon  \alpha n} |\widehat{\mu}(\xi)|^2 \lesssim e^{-\lambda \delta (2k+1) n} \sum_{\mathbf{A} \in \mathcal{R}_{n+1}^{k+1}} \sup_{\eta \in J_n} \Bigg{|} \underset{\mathbf{A} \leftrightarrow \mathbf{B}}{\sum_{\mathbf{B} \in \mathcal{R}_{n+1}^k}} e^{2 i \pi \text{Re}\left( \eta \zeta_{1,\mathbf{A}}(\mathbf{b}_1) \dots \zeta_{k,\mathbf{A}}(\mathbf{b}_k) \right)} \Bigg| $$
$$ \quad \quad  \quad \quad  \quad \quad  \quad \quad +  e^{- \varepsilon  \alpha n} \mu(J \setminus R_{n+1}^{2k+1}(\varepsilon) )^2 + \kappa^{-2n} + e^{-(\lambda-\varepsilon_0) n } + e^{-\varepsilon_0 \delta_{AD} n/2} $$
by Proposition 3.4.1. Since every error term already enjoys exponential decay in $n$, we just have to deal with the sum of exponentials. By Theorem 3.5.1, we can then write
$$ \sup_{\eta \in J_n} \Bigg{|} \underset{\mathbf{A} \leftrightarrow \mathbf{B}}{\sum_{\mathbf{B} \in \mathcal{R}_{n+1}^k}} e^{2 i \pi \text{Re}\left( \eta \zeta_{1,\mathbf{A}}(\mathbf{b}_1) \dots \zeta_{k,\mathbf{A}}(\mathbf{b}_k) \right)} \Bigg| \leq c e^{\lambda k \delta n} e^{- \varepsilon_0 \varepsilon_1 n/2 } ,$$
and hence we get
$$ e^{-\lambda \delta (2k+1) n} \sum_{\mathbf{A} \in \mathcal{R}_{n+1}^{k+1}} \sup_{\eta \in J_n} \Bigg{|} \underset{\mathbf{A} \leftrightarrow \mathbf{B}}{\sum_{\mathbf{B} \in \mathcal{R}_{n+1}^k}} e^{2 i \pi \text{Re}\left( \eta \zeta_{1,\mathbf{A}}(\mathbf{b}_1) \dots \zeta_{k,\mathbf{A}}(\mathbf{b}_k) \right)} \Bigg| $$ $$ \lesssim e^{\varepsilon \alpha n} e^{-\lambda \delta  (2k+1) n} e^{\lambda \delta (k+1) n} e^{\lambda \delta k n} e^{- \varepsilon_0 \varepsilon_1 n/2} \lesssim e^{\varepsilon \alpha n} e^{-\varepsilon_0 \varepsilon_1 n/2} .$$
Now, we see that we can choose $\varepsilon$ small enough so that all terms enjoy exponential decay in $n$, and since $ |\xi| \simeq e^{(2k+1) \lambda n} e^{\varepsilon_0 n} $, we have proved polynomial decay of $|\widehat{\mu}|^2$.

\section{The non-concentration hypothesis}

The last part of this Chapter is devoted to the proof of the non-concentration hypothesis that we just used. Unlike Chapter 2, we will this time directly use some Dolgopyat's estimates for well chosen twisted transfer operators.  From there, $\chi$ and $\psi$ won't play any role: the only actors now are the dynamics and the measure $\mu$ (the dependence on the measure coming from the definition of $\varepsilon$-regular words). 

\subsection{Statement of the non-concentration estimates}

\begin{definition}

For a given $\mathbf{A} \in \mathcal{R}_{n+1}^{k+1}$, define for $j=1, \dots, k$  $$ \mathcal{Z}_j := \{ \mathbf{b} \in \mathcal{R}_{n+1} , \ \mathbf{a}_{j-1} \rightsquigarrow \mathbf{b} \rightsquigarrow \mathbf{a}_j \  \} $$
Then define $$\zeta_{j,\mathbf{A}}(\mathbf{b}) := e^{2 \lambda n} g_{\mathbf{a}_{j-1}' \mathbf{b}}'(x_{{\mathbf{a}}_j})$$ on $\mathcal{Z}_j$. The following is satisfied, for some fixed constant $\alpha>0$:

 $$\# \mathcal{Z}_j \leq e^{\varepsilon \alpha n} e^{ \delta \lambda n} $$
 and
$$ \zeta_{j,\mathbf{A}}( \mathcal{Z}_j ) \subset \{ z \in \mathbb{C} \ , \ e^{-\varepsilon \alpha n} \leq |z| \leq e^{\varepsilon \alpha n} \} .$$

\end{definition}
We are going to prove the following fact, which will allow us to apply Theorem 3.5.1 for $\eta \in J_n$, $R:=e^{\varepsilon \alpha n}$ and $N:= e^{\lambda \delta n} $.

\begin{theorem}[non-concentration]
There exists $\gamma>0$, and we can choose $\varepsilon_0>0$, such that the following holds. Let $\eta \in \{ e^{\varepsilon_0 n/2} \leq |\eta| \leq e^{2 \varepsilon_0 n} \}$. Let $\mathbf{A} \in \mathcal{R}_{n+1}^{k+1}$. Then, if $n$ is large enough,
$$\forall \sigma \in [ |\eta|^{-2}, |\eta|^{- \varepsilon_1}], \quad \sup_{a,\theta \in \mathbb{R}} \ \# \left\{\mathbf{b} \in \mathcal{Z}_j , \ | \emph{\text{Re}} (e^{i\theta} \zeta_{j,\mathbf{A}}(\mathbf{b})) - a|  \leq \sigma \right\} \leq N \sigma^{\gamma}, $$
where $R:= e^{\varepsilon \alpha n}$, $N:= e^{\lambda \delta n}$ and $\varepsilon_1$ and $k$ are fixed by Theorem 3.5.1.
\end{theorem}

\subsection{Non-concentration estimates I}

The proof of Theorem 3.6.2 is in two parts. First of all, we see that the non-concentration hypothesis formulated above counts how many $\zeta_{j,\mathbf{A}}$ are in a strip. We begin by reducing the non-concentration to a counting problem in small disks. Recall that $R := e^{\varepsilon \alpha n}$ is a slowly growing term, and that $N := e^{\lambda \delta n} \sim \# \mathcal{R}_{n+1} $.

\begin{lemma}

If $\varepsilon_0$ and $\gamma$ are such that, for $\sigma \in [ e^{- 5 \varepsilon_0 n} ,  e^{ - \varepsilon_1 \varepsilon_0 n/5 }  ]$, $$\sup_{R^{-1} \leq |a| \leq R} \#  \{ \mathbf{b} \in \mathcal{Z}_j,  \ \zeta_{j,\mathbf{A}}(\mathbf{b}) \in B(a,\sigma) \} \leq N \sigma^{1+\gamma} ,$$
then Theorem 3.6.2 is true.
\end{lemma}

\begin{proof}
Suppose that the result in Lemma 3.6.3 is true.
Then, we know that squares $C_{c,\theta,\sigma} := e^{-i \theta} B_\infty(c,\sigma) =  \{ z \in \mathbb{C}, \ |\text{Re}(e^{i \theta}z-c)| \leq \sigma, |\text{Im}(e^{i \theta}z-c)| \leq \sigma \} $
are included in disks $B(c,\sigma \sqrt{2})$. (We note $B_\infty$ the balls for the $L^\infty$ norm.)
Hence,
$$ \forall \sigma \in [ e^{-  5 \varepsilon_0 n} , e^{ - \varepsilon_1 \varepsilon_0 n/5 }  ], $$
$$ \# \{ \mathbf{b} \in \mathcal{Z}_j , \ \zeta_{j,\mathbf{A}} \in C_{c,\theta,\sigma} \} \leq \# \{ \mathbf{b} \in \mathcal{Z}_j , \ \zeta_{j,\mathbf{A}} \in B(c,\sigma \sqrt{2}) \} \leq N \sqrt{2}^{1+\gamma} \sigma^{1+\gamma} .$$
Our next move is to cover the strip $S_{a,\theta,\sigma} := \left\{z \in \mathbb{C} , \ |  \text{Re} (e^{i\theta} z) - a|  \leq \sigma \right\}$
by squares $C_{c,\theta,\sigma}$. First of all, recall that $\zeta_{j,\mathbf{A}}(\mathcal{Z}_j) \subset B(0,R) $.
Hence, we can write, for a fixed $a$ and $\theta$:

$$ \# \{ \mathbf{b} \in \mathcal{Z}_j , \ \zeta_{j,\mathbf{A}}(\mathbf{b}) \in S_{a,\theta,\sigma}  \} \leq \sum_{c \in K(\sigma,R)} \# \{ \mathbf{b} \in \mathcal{Z}_j , \ \zeta_{j,\mathbf{A}} \in C_{c,\theta,\sigma}  \} $$
where $K(\sigma,n) := \{ e^{- i \theta}( a + i k \sigma ) \ | \ k = - \lfloor R/\sigma \rfloor, \dots , \lfloor R/\sigma \rfloor  \}$ is the set of the centers of the squares, chosen so that it covers our restricted strip. Hence, for $\sigma \in [ e^{-4 \varepsilon_0 n} , e^{ - \varepsilon_1 \varepsilon_0 n/2 }  ]$, 
$$ \# \{ \mathbf{b} \in \mathcal{Z}_j , \ \zeta_{j,\mathbf{A}}(\mathbf{b}) \in S_{a,\theta,\sigma}  \}  \lesssim \frac{R}{\sigma} N \sqrt{2}^{1+\gamma} \sigma^{1+\gamma} . $$
Then, since $\sigma$ goes to zero exponentially fast in $n$, and since $R$ grows slowly since $\varepsilon$ can be chosen as small as we want, we can just take $n$ large enough so that
$$  \# \{ \mathbf{b} \in \mathcal{Z}_j , \ \zeta_{j,\mathbf{A}}(\mathbf{b}) \in S_{a,\theta,\sigma}  \}  \lesssim \frac{R}{\sigma} N \sqrt{2}^{1+\gamma} \sigma^{1+\gamma} \leq N \sigma^{\gamma/2}, $$
and we are done. \end{proof}

\begin{definition}

Since $\exp : \mathbb{C} \rightarrow \mathbb{C}^* $ is a surjective, holomorphic morphism, with kernel $2 i \pi \mathbb{Z}$, it induces a biholomorphism $ \exp : \mathbb{C}/2 i \pi \mathbb{Z} \rightarrow \mathbb{C}^* $. Define by $ \log : \mathbb{C}^* \rightarrow \mathbb{C}/2 i \pi \mathbb{Z} $ its holomorphic inverse.
Note $ \text{mod}_{2 i \pi} : \mathbb{C} \rightarrow \mathbb{C}/2 i \pi \mathbb{Z} $ the projection.
\end{definition}

Now, we reduce the problem to a counting estimate on $\log(\zeta_{j,\mathbf{A}})$.

\begin{lemma}

If $\varepsilon_0$ and $\gamma$ are such that, for $\sigma \in [ e^{-6 \varepsilon_0 n} , e^{ - \varepsilon_1 \varepsilon_0 n/6 }  ]$, $$\sup_{a \in \mathbb{C}} \#  \left\{ \mathbf{b} \in \mathcal{Z}_j,  \ \log g_{\mathbf{a}_{j-1}' \mathbf{b}}' (x_{\mathbf{a}_j}) \in \emph{\text{mod}}_{2 i \pi} \left(B_\infty(a,\sigma) \right) \right\} \leq N \sigma^{1+\gamma} ,$$
then Theorem 3.6.2 is true.
\end{lemma}

\begin{proof}

Suppose that the estimate is true.
Let $ \sigma \in [ e^{-5 \varepsilon_0 n} , e^{ - \varepsilon_1 \varepsilon_0 n/5 }  ] $. Fix a euclidean ball $B(a,\sigma)$, where $a=r_0e^{i \theta_0}$ satisfies $r_0 \in [R^{-1},R]$ and $\theta_0 \in ]-\pi,\pi]$. Elementary trigonometry allows us to see that 
$$ B(a,\sigma) \subset \left\{ re^{i \theta} \in \mathbb{C} | \ r \in [r_0 - \sigma, r_0 + \sigma] , \ \theta \in [\theta_0 - \arctan(\sigma/r_0), \theta_0 + \arctan(\sigma/r_0) )] \ \right\} .$$
Then, since for $n$ large enough $$\ln(r_0+\sigma)-\ln(r_0-\sigma) = \ln(1+ \sigma r_0^{-1}) - \ln(1 - \sigma r_0^{-1}) \leq 4 \sigma r_0^{-1} \leq 4 \sigma R $$ and $$ 2\arctan(\sigma/r_0) \leq 4 \sigma r_0^{-1} \leq 4 \sigma R ,$$
we find that
$$ B(a,\sigma) \subset \exp B_\infty( (\ln(r_0),\theta_0), 4 \sigma R ) . $$
Hence:
$$  \#  \{ \mathbf{b} \in \mathcal{Z}_j,  \ \zeta_{j,\mathbf{A}}(\mathbf{b}) \in B(a,\sigma) \} \leq  \#  \{ \mathbf{b} \in \mathcal{Z}_j,  \ \zeta_{j,\mathbf{A}}(\mathbf{b}) \in \exp B_\infty( (\ln(r_0),\theta_0) , 4 R \sigma)  \} $$
$$ = \#  \{ \mathbf{b} \in \mathcal{Z}_j,  \ \log \zeta_{j,\mathbf{A}}(\mathbf{b}) \in \text{mod}_{2 i \pi} \left(  B_\infty( (\ln(r_0),\theta_0) , 4 R \sigma) \right) \}$$
$$ = \#  \{ \mathbf{b} \in \mathcal{Z}_j,  \ \log g_{\mathbf{a}_{j-1}' \mathbf{b}}'(x_{\mathbf{a}_j}) \in \text{mod}_{2 i \pi} \left(  B_\infty( (\ln(r_0) - 2 n \lambda,\theta_0) , 4 R \sigma) \right) \} $$
$$ \leq N (4 R \sigma)^{1+\gamma} \leq N \sigma^{1+\gamma/2} $$
provided $n$ is large enough. So the inequality of Lemma 3.6.3 is satisfied, and so Theorem 3.6.2 is true. \end{proof}

\subsection{Non-concentration estimates II}

We are going to prove that the estimate in Lemma 3.6.5 is satisfied for all $C^1(V,\mathbb{R})$ potentials $\varphi$. (The dependence in $\varphi$ is hidden in the definition of the $\varphi$-regular words.) For this, we will need a generalization of a result borrowed from $\cite{OW17}$.

\begin{theorem}
We work on $\mathfrak{U} := \bigsqcup_{a \in \mathcal{A}} U_a $, the formal disjoint union of the $U_a$. Define $$C^1_b(\mathfrak{U},\mathbb{C}) := \left\{ \mathfrak{h} = (h_a)_{a \in \mathcal{A}} \ | \ h_a \in C^1(U_a,\mathbb{C}) , \ \| (h_a)_{a} \|_{C^1_b(\mathfrak{U},\mathbb{C}) }< \infty \right\} , $$ where $\| \cdot \|_{C^1_b(\mathfrak{U},\mathbb{C})}$ is the usual $C^1$ norm
$$ \| \mathfrak{h} \|_{C^1_b(\mathfrak{U},\mathbb{C})} = \sum_{a \in \mathcal{A}} \left( \|h_a\|_{\infty,U_a} + \| \nabla h_a \|_{\infty,U_a} \right). $$

On this Banach space, for $\varphi$ a normalized potential, $s \in \mathbb{C}$ and $l \in \mathbb{Z}$, we define a twisted transfer operator $\mathcal{L}_{\varphi,s,l} : C^1_b(\mathfrak{U},\mathbb{C}) \rightarrow C^1_b(\mathfrak{U},\mathbb{C}) $ as follows:
$$ \forall x \in U_b, \ \mathcal{L}_{\varphi,s,l}\mathfrak{h}(x) := \sum_{M_{ab}=1} e^{\varphi(g_{ab}(x))} |g_{ab}'(x)|^s \left(\frac{g_{ab}'(x)}{|g_{ab}'(x)|} \right)^{-l} \mathfrak{h}( g_{ab}(x) ) ,$$
where $g_{ab} : U_b \rightarrow U_a$. Iterating this transfer operator yields:
$$ \forall x \in U_b, \ \mathcal{L}_{\varphi,s,l}^n \mathfrak{h} (x) = \underset{\mathbf{a} \rightsquigarrow b}{\sum_{\mathbf{a} \in \mathcal{W}_{n+1}}} w_{\mathbf{a}}(x) |g_{\mathbf{a}}'(x)|^s \left(\frac{g_{\mathbf{a}}'(x)}{|g_{\mathbf{a}}'(x)|} \right)^{-l} \mathfrak{h}(g_{\mathbf{a}}(x)) .$$
Since $J$ is supposed to be not included in a circle, we have the following result. There exists $C>0$ and $\rho<1$ such that, for any $s \in \mathbb{C}$ such that $\text{Re}(s)=0$ and $|\text{Im}(s)|+|l| >1$,
$$ \forall n \geq 1, \ \| \mathcal{L}_{\varphi,s,l}^n \|_{C^1_b(\mathfrak{U},\mathbb{C})} \leq C_0 (|\text{Im}(s)|+|l|)^2 \rho^n .$$

\end{theorem}
It means that this twisted transfer operator is eventually uniformly contracting for large $l$ and $\text{Im}(s)$. This result will play a key role in the proof of our non-concentration estimates. 

\begin{remark}
In \cite{OW17}, this Theorem was proved for the conformal measure. In \cite{ShSt20}, section 3.3, Sharp and Stylianou explain how we can generalize the result for a more general family of potentials, which covers the case of the measure of maximal entropy. The fully general theorem can be proved with some very minor modifications from the proof developed in \cite{OW17}: it will be explained in appendix B.
\end{remark}

\begin{proposition}

Define $\varepsilon_0 := \min( - \ln(\rho)/30,  \lambda/2  ) $. There exists $\gamma>0$ such that, \newline for $\sigma \in [e^{-6 \varepsilon_0 n} , e^{ - \varepsilon_1 \varepsilon_0 n/6 }  ]$ and if $n$ is large enough, $$\sup_{a} \#  \left\{ \mathbf{b} \in \mathcal{Z}_j,  \ -\log g_{\mathbf{a}_{j-1}' \mathbf{b}}' (x_{\mathbf{a}_j}) \in \emph{\text{mod}}_{2 i \pi} \left(B_\infty(a,\sigma) \right) \right\} \leq N \sigma^{1+\gamma} .$$

\end{proposition}

\begin{proof}
In the proof to come, all the $\simeq$ or $\lesssim$ will be uniform in $a$: the only relevant information here will be $\sigma$.
So fix $\sigma \in [ e^{- 6 \varepsilon_0 n} , e^{ - \varepsilon_1 \varepsilon_0 n/6 }  ]$, and fix a small square  $\text{mod}_{2 i \pi} \left(B_\infty(a,\sigma)\right) \subset \mathbb{C}/2 i \pi \mathbb{Z}$. The area of this square is $\sigma^2$.
Lift this square somewhere in $\mathbb{C}$, for example as $B_\infty(a,\sigma)$, and then define a bump function $\chi$ such that $\chi = 1$ on $B_\infty(a,\sigma)$, $\text{supp}(\chi) \subset B_\infty(a,2 \sigma)$ and such that $ \| \chi \|_{L^1(\mathbb{C})} \simeq \sigma^2 $. We can suppose that $ \| \partial_x^{k_1} \partial_y^{k_2} \chi \|_{L^1(\mathbb{C})} \simeq_{k_1,k_2} \sigma^{2-k_1-k_2} $. (For example, take $\chi(x) := \chi_0((x-a)\sigma^{-1})$ for $\chi_0$ a bump function around 0.) \\

Then, we can consider $h$, the $2 \pi \mathbb{Z}[i] := 2\pi(\mathbb{Z} + i \mathbb{Z})$ periodic map obtained by periodizing $\chi$. We can see it either as a smooth $2 \pi \mathbb{Z}[i]$-periodic map on $\mathbb{C}$, or as a smooth $2 \pi \mathbb{Z}$-periodic map on $\mathbb{C}/2 i \pi \mathbb{Z}$, or just as a smooth map on $\mathbb{C}/2  \pi \mathbb{Z}[i]$. \\

By construction, the periodicity of $h$ allows us to see that
$$ \mathbb{1}_{ \text{mod}_{2i \pi} \left( B_\infty(a,\sigma) \right) } \leq h .$$
Moreover, $h\left(-\log g_{\mathbf{a}_{j-1}' \mathbf{b}}' \right)(x_{\mathbf{a}_j}) $ is well defined, and so we can bound the desired cardinality with it. We have the following \say{convex combination} bound:

$$ \#  \left\{ \mathbf{b} \in \mathcal{Z}_j,  \ -\log g_{\mathbf{a}_{j-1}' \mathbf{b}}' (x_{\mathbf{a}_j}) \in \text{mod}_{2 i \pi} \left(B_\infty(a,\sigma) \right) \right\} \leq \sum_{\mathbf{b} \in \mathcal{Z}_j} h(-\log g_{\mathbf{a}_{j-1}' \mathbf{b}}' (x_{\mathbf{a}_j}) ) $$
$$ =  \sum_{\mathbf{b} \in \mathcal{Z}_j} \frac{w_{ \mathbf{b}}(x_{\mathbf{a}_j})}{w_{ \mathbf{b}}(x_{\mathbf{a}_j})}  h(-\log g_{\mathbf{a}_{j-1}' \mathbf{b}}' (x_{\mathbf{a}_j}) ) $$
$$ \leq R N{\sum_{\mathbf{b} \in \mathcal{Z}_{j}}} w_{ \mathbf{b}}(x_{\mathbf{a}_j}) h(-\log g_{\mathbf{a}_{j-1}' \mathbf{b}}' (x_{\mathbf{a}_j}) ) $$
$$ \leq R N \underset{\mathbf{a}_{j-1} \rightsquigarrow \mathbf{b} \rightsquigarrow \mathbf{a}_j}{\sum_{\mathbf{b} \in \mathcal{W}_{n+1}}} w_{ \mathbf{b}}(x_{\mathbf{a}_j}) h(-\log g_{\mathbf{a}_{j-1}' \mathbf{b}}' (x_{\mathbf{a}_j}) ) .$$
Then, since our map $h$ is $2 \pi \mathbb{Z}[i]$-periodic and smooth, we can develop it using Fourier series. We can write:
$$ \forall z=x+iy \in \mathbb{C}/2 i \pi \mathbb{Z}, \ h(z) = \sum_{(\mu,\nu) \in \mathbb{Z}^2 } c_{\mu \nu}(h) e^{i \left( \mu x + \nu y \right)} ,$$
where $$c_{\mu \nu}(h) = (4 \pi^2)^{-1} \int_{B_\infty(a,\pi)} h(x + iy) e^{-i \left( \mu x + \nu y \right)} dx dy .$$ 
Notice that $$ \mu^{k_1} \nu^{k_2} |c_{\mu \nu}(h)| \simeq |c_{\mu \nu}( \partial_x^{k_1} \partial_y^{k_2} h)| $$ $$\leq (4 \pi^2)^{-1} \|\partial_x^{k_1} \partial_y^{k_2} h\|_{L^1(B_\infty(a,\pi))} = (4 \pi^2)^{-1} \|\partial_x^{k_1} \partial_y^{k_2} \chi\|_{L^1(\mathbb{C})} \simeq \sigma^{2-k_1 - k_2} .$$
Plugging $ -\log g_{\mathbf{a}_{j-1}' \mathbf{b}}' (x_{\mathbf{a}_j}) $ in this expression yields
$$ h\left( -\log g_{\mathbf{a}_{j-1}' \mathbf{b}}' (x_{\mathbf{a}_j}) \right) = \sum_{(\mu,\nu) \in \mathbb{Z}^2 } c_{\mu \nu}(h) \exp\left( - i \mu \ln( | g_{\mathbf{a}_{j-1}' \mathbf{b}}' (x_{\mathbf{a}_j}) |  ) - i \nu \arg   g_{\mathbf{a}_{j-1}' \mathbf{b}}' (x_{\mathbf{a}_j})  \right) $$
$$ =  \sum_{(\mu,\nu) \in \mathbb{Z}^2 } c_{\mu \nu}(h) | g_{\mathbf{a}_{j-1}' \mathbf{b}}' (x_{\mathbf{a}_j})|^{- i \mu} \left( \frac{ g_{\mathbf{a}_{j-1}' \mathbf{b}}' (x_{\mathbf{a}_j})}{| g_{\mathbf{a}_{j-1}' \mathbf{b}}' (x_{\mathbf{a}_j})|} \right)^{- \nu} , $$
and so
$$  \#  \left\{ \mathbf{b} \in \mathcal{Z}_j,  \ -\log g_{\mathbf{a}_{j-1}' \mathbf{b}}' (x_{\mathbf{a}_j}) \in \text{mod}_{2i \pi}\left(B_\infty(a,\sigma) \right) \right\} $$
$$ \leq R N  \sum_{\mu \nu} c_{\mu \nu}(h) \underset{\mathbf{a}_{j-1} \rightsquigarrow \mathbf{b} \rightsquigarrow \mathbf{a}_j}{\sum_{\mathbf{b} \in \mathcal{W}_{n+1}}} w_{ \mathbf{b}}(x_{\mathbf{a}_j}) | g_{\mathbf{a}_{j-1}' \mathbf{b}}' (x_{\mathbf{a}_j})|^{- i \mu} \left( \frac{ g_{\mathbf{a}_{j-1}' \mathbf{b}}' (x_{\mathbf{a}_j})}{| g_{\mathbf{a}_{j-1}' \mathbf{b}}' (x_{\mathbf{a}_j})|} \right)^{- \nu}  .$$
For any word $\mathbf{a}$, define $\mathfrak g_{\mathbf{a}}'$ on $C_b^1(\mathfrak{U},\mathbb{C})$ by
$$ \forall x \in U_{b(\mathbf{a})}, \ \mathfrak g_{\mathbf{a}}'(x) := g_{\mathbf{a}}'(x) \quad , \quad \forall x \in U_b , \ b \neq b(\mathbf{a}), \ \mathfrak g_{\mathbf{a}}'(x) := 0 .$$
With this notation, we may rewrite the sum on $\mathbf{b}$ as follows:
$$  \underset{\mathbf{a}_{j-1} \rightsquigarrow \mathbf{b} \rightsquigarrow \mathbf{a}_j}{\sum_{\mathbf{b} \in \mathcal{W}_{n+1}}} w_{ \mathbf{b}}(x_{\mathbf{a}_j}) | g_{\mathbf{a}_{j-1}' \mathbf{b}}' (x_{\mathbf{a}_j})|^{ - i \mu} \left( \frac{ g_{\mathbf{a}_{j-1}' \mathbf{b}}' (x_{\mathbf{a}_j})}{| g_{\mathbf{a}_{j-1}' \mathbf{b}}' (x_{\mathbf{a}_j})|} \right)^{- \nu} $$
$$ = \underset{\mathbf{a}_{j-1} \rightsquigarrow \mathbf{b} \rightsquigarrow \mathbf{a}_j}{\sum_{\mathbf{b} \in \mathcal{W}_{n+1}}}  | g_{\mathbf{a}_{j-1}}' (g_{ \mathbf{b}}(x_{\mathbf{a}_{j}}))|^{- i \mu} \left( \frac{ g_{\mathbf{a}_{j-1} }' (g_{ \mathbf{b}}(x_{\mathbf{a}_{j}} ))} {| g_{\mathbf{a}_{j-1}}' (g_{ \mathbf{b}}(x_{\mathbf{a}_{j}}))|} \right)^{- \nu}  w_{\mathbf{b}}(x_{\mathbf{a}_j}) |g_{\mathbf{b}}' (x_{\mathbf{a}_j})|^{- i \mu} \left( \frac{ g_{ \mathbf{b}}' (x_{\mathbf{a}_j})}{| g_{ \mathbf{b}}' (x_{\mathbf{a}_j})|} \right)^{- \nu} $$
$$ = \underset{\mathbf{b} \rightsquigarrow \mathbf{a}_{j} }{\sum_{\mathbf{b} \in \mathcal{W}_{n+1}}} w_{\mathbf{b}}(x_{\mathbf{a}_j}) | g_{\mathbf{b}}' (x_{\mathbf{a}_j})|^{- i \mu} \left( \frac{ g_{ \mathbf{b}}' (x_{\mathbf{a}_j})}{| g_{ \mathbf{b}}' (x_{\mathbf{a}_j})|} \right)^{- \nu}  \left( |\mathfrak g_{\mathbf{a}_{j-1}}'|^{-i \mu}  \left(\frac{\mathfrak g_{\mathbf{a}_{j-1}}'}{|\mathfrak g_{\mathbf{a}_{j-1}}'|} \right)^{- \nu} \right) \left( g_{\mathbf{b}}(x_{\mathbf{a}_j}) \right) $$
$$ = \mathcal{L}_{\varphi,-i \mu,\nu}^n \left(  |\mathfrak g_{\mathbf{a}_{j-1}}'|^{-i \mu}  \left(\frac{\mathfrak g_{\mathbf{a}_{j-1}}'}{|\mathfrak g_{\mathbf{a}_{j-1}}'|} \right)^{- \nu} \right)(x_{\mathbf{a}_j}). $$
For clarity, set $\mathfrak{h}_{\mathbf{A},j} :=  |\mathfrak g_{\mathbf{a}_{j-1}}'|^{-i \mu}  \left(\frac{\mathfrak g_{\mathbf{a}_{j-1}}'}{|\mathfrak g_{\mathbf{a}_{j-1}}'|} \right)^{- \nu}  $.
The holomorphicity of the $(g_{\mathbf{a}})_\mathbf{a}$ allows us to use Cauchy's formula (like in the proof of Lemma 3.4.4) to see that
$$ \| \mathfrak{h}_{\mathbf{A},j} \|_{C_b^1(\mathfrak{U},\mathbb{C})} \lesssim (1+|\mu|+|\nu|). $$
We can now break the estimate into two pieces: high frequencies are controlled by the contraction property of this transfer operator, and the low frequencies are controlled by the Gibbs property of $\mu$. We also use the estimates on the Fourier coefficients of $h$. 
$$ \#  \left\{ \mathbf{b} \in \mathcal{Z}_j,  \ -\log g_{\mathbf{a}_{j-1}' \mathbf{b}}' (x_{\mathbf{a}_j}) \in \text{mod}_{2 i \pi} \left(B_\infty(a,\sigma) \right) \right\}  $$
$$ \leq R N \sum_{\mu \nu} c_{\mu \nu}(h)\underset{\mathbf{a}_{j-1} \rightsquigarrow \mathbf{b} \rightsquigarrow \mathbf{a}_j}{\sum_{\mathbf{b} \in \mathcal{W}_{n+1}}}  w_{\mathbf{b}}(x_{\mathbf{a}_j}) | g_{\mathbf{a}_{j-1}' \mathbf{b}}' (x_{\mathbf{a}_j})|^{ - i \mu} \left( \frac{ g_{\mathbf{a}_{j-1}' \mathbf{b}}' (x_{\mathbf{a}_j})}{| g_{\mathbf{a}_{j-1}' \mathbf{b}}' (x_{\mathbf{a}_j})|} \right)^{- \nu} $$
$$ \leq R N \left( \sum_{|\mu| + |\nu| \leq 1} |c_{\mu \nu}(h)|  \sum_{\mathbf{b} \in \mathcal{W}_{n+1}}  w_{\mathbf{b}}(x_{\mathbf{a}_j})  +  \sum_{|\mu| + |\nu| > 1} |c_{\mu \nu}(h)| |\mathcal{L}_{\varphi,-i \mu,\nu}^n (\mathfrak{h}_{\mathbf{A},j})(x_{\mathbf{a}_j})| \right) $$
$$ \lesssim R N \left( 5 \sigma^2 \sum_{\mathbf{b} \in \mathcal{W}_{n+1}} \mu(P_{\mathbf{b}})  + \sum_{|\mu| + |\nu| > 1} |c_{\mu \nu}(h)| \|\mathcal{L}_{\varphi,-i \mu,\nu}^n (\mathfrak{h}_{\mathbf{A},j})\|_{C_b^1(\mathfrak{U},\mathbb{C})} \right) $$
$$ \lesssim  R N \sigma^2 +R N  \sum_{|\mu| + |\nu| > 1} |c_{\mu \nu}(h)| (|\mu|+|\nu|)^2 \rho^n \| \mathfrak{h}_{\mathbf{A},j} \|_{C^1_b(\mathfrak{U},\mathbb{C})} $$

$$ \lesssim  R N \sigma^2 + R N \rho^n \sum_{|\mu| + |\nu| > 1} |c_{\mu \nu}(h)| (|\mu|+|\nu|)^5 (|\mu|+|\nu|)^{-2}    $$
$$ \leq C R N (\sigma^2 + \rho^n \sigma^{-3}) , $$
for some constant $C>0$.
We are nearly done. Since $\sigma \in  [ e^{-6 \varepsilon_0 n} , e^{ - \varepsilon_1 \varepsilon_0 n/6 }  ]$, we know that $ \sigma^{-3} \leq e^{18 \varepsilon_0 n } $.
Now is the time where we fix $\varepsilon_0$: choose $$ \varepsilon_0 := \min( - \ln(\rho)/30,  \lambda/2  ). $$
Then $ \rho^n \sigma^{-3} \leq  e^{ n(\ln(\rho) + 18 \varepsilon_0 )} \leq  e^{ - 12 \varepsilon_0 n} \leq \sigma^2 $ for $n$ large enough. Hence, we get
$$  \#  \left\{ \mathbf{b} \in \mathcal{Z}_j,  \ -\log g_{\mathbf{a}_{j-1}' \mathbf{b}}' (x_{\mathbf{a}_j}) \in \text{mod}_{2 i \pi} \left(B_\infty(a,\sigma) \right) \right\} \leq 2 C R N \sigma^{2} . $$
Finally, since $\sigma^{1/2}$ is quickly decaying compared to $R$, we have
$$ 2 C R N \sigma^{2} \leq N \sigma^{3/2} $$
provided $n$ is large enough. The proof is done.
\end{proof}

\section{Appendix A: Large deviations}

The goal of this section is to prove the large deviation Theorem 3.2.11, by using properties of the pressure. The link between the spectral radius of $\mathcal{L}_\varphi$ and the pressure given by the Perron-Frobenius-Ruelle theorem allows us to get the following useful formula (though, seing it that way may be a bit of a circular reasoning, as the following lemma is usually proved before Perron-Frobenius-Ruelle's Theorem). We extract the first one from \cite{Ru78}, Theorem 7.20 and Remark 7.28, and the second one from \cite{Ru89}, Lemma 4.5.

\begin{proposition}
 $$ P(\varphi) = \lim_{n \rightarrow \infty} \frac{1}{n} \log \sum_{f^n(x)=x} e^{S_n \varphi(x)}$$ $$\quad P(\varphi) = \lim_{n \rightarrow \infty} \frac{1}{n} \max_{b \in \mathcal{A}} \sup_{x \in P_b} \log \underset{\mathbf{a} \rightsquigarrow b}{\sum_{\textbf{a} \in \mathcal{W}_{n+1}} }  e^{S_n \varphi( g_{\textbf{a}}(x) )}  $$

\end{proposition}
We begin by proving another avatar of those spectral radius formulas (which is nothing new).

\begin{lemma}
Choose any $x_{\textbf{a}}$ in each of the $P_{\textbf{a}}$, $\textbf{a} \in \mathcal{W}_{n}$, $\forall n$.
Then
$$ P(\varphi) = \lim_n \frac{1}{n} \log \sum_{\textbf{a} \in \mathcal{W}_{n+1}} e^{S_n \varphi(x_\textbf{a})} .$$
\end{lemma}

\begin{proof}
Since $P_{\textbf{a}}$ is compact, and by continuity, for every $n$ there exists $b^{(n)} \in \mathcal{A}$ and $y_{b^{(n)}}^{(n)} \in P_{b^{(n)}}$ such that $$ \max_{b \in \mathcal{A}} \sup_{x \in P_b} \log \underset{\mathbf{a} \rightsquigarrow b}{\sum_{\textbf{a} \in \mathcal{W}_{n+1}} }  e^{S_n \varphi( g_{\textbf{a}}(x) )} = \log \underset{\mathbf{a} \rightsquigarrow b^{(n)}}{\sum_{\textbf{a} \in \mathcal{W}_{n+1}} } e^{S_n \varphi( g_{\textbf{a}}(y_{b^{(n)}}^{(n)}) )}. $$  
Define $y_{\textbf{a}} := g_{\textbf{a}}(y_{b^{(n)}}^{(n)}) \in P_{\textbf{a}}$ for clarity. The dependence on $n$ is not lost since it is contained in the length of the word. First of all, since $\varphi$ has exponentially vanishing variations, there exists a constant $C_1>0$ such that
$$ \forall x,y \in P_{\textbf{a}}, \ | S_n \varphi(x) - S_n \varphi(y) |\leq C_1 .$$
Now we want to relate the sums with the $x_\mathbf{a}$'s and the $y_{\mathbf{a}}$'s, but the indices are different. To do it properly, we are going to use the fact that $f$ is topologically mixing: there exists some $N \in \mathbb{N}$ such that the matrix $M^N$ has all its entries positive. In particular, it means that
$$ \forall b \in \mathcal{A}, \ \forall \textbf{a} \in \mathcal{W}_{n+1}, \ \exists \textbf{c} \in \mathcal{W}_{N}, \ \textbf{ac}b \in \mathcal{W}_{n+N+1} .$$
The point is that we are sure that the word is admissible. \\
For a given $\textbf{a} \in \mathcal{W}_{n+1}$, there exists a $\textbf{c} \in \mathcal{W}_{N}$ such that $ \textbf{ac}b^{(n+N+1)} \in \mathcal{W}_{n+N+1}$, and so, using the fact that $e^{S_n \varphi} \geq 0$, we get:

$$ e^{S_n \varphi(x_{\textbf{a}})} \leq  \underset{ \textbf{ac}b^{(n+N+1)} \in \mathcal{W}_{n+N+1}}{\sum_{\textbf{c} \in \mathcal{W}_{N} }} e^{C_1} e^{S_n \varphi(y_{\textbf{ac}b^{(n+N+1)}})} .$$
Then, since $S_{n}(\varphi) \leq S_{n+N}(\varphi) +N \|\varphi\|_{\infty,J}$, we have:

$$ e^{S_n \varphi(x_{\textbf{a}})} \leq  \underset{ \textbf{ac}b^{(n+N+1)} \in \mathcal{W}_{n+N+1}}{\sum_{\textbf{c} \in \mathcal{W}_{N} }} e^{C_2} e^{S_{n+N}(\varphi)(y_{\textbf{ac}b^{(n+N+1)}})} .$$
Hence
$$ \log \Big( {\sum_{\textbf{a} \in \mathcal{W}_{n+1}} }  e^{S_n \varphi( x_{\textbf{a}} )} \Big) \leq \log \Bigg( \sum_{\textbf{a} \in \mathcal{W}_{n+1}} \underset{ \textbf{ac}b^{(n+N+1)} \in \mathcal{W}_{n+N+1}}{\sum_{\textbf{c} \in \mathcal{W}_{N} }} e^{C_2} e^{S_{n+N} \varphi(y_{\textbf{ac}b^{(n+N+1)}})}  \Bigg) $$
$$ =C_2 + \log \Bigg( \underset{\mathbf{d} \rightsquigarrow b^{(n+N+1)}}{\sum_{\textbf{d} \in \mathcal{W}_{n+N+1} }} e^{S_n \varphi(y_{\textbf{d}})} \Bigg) ,$$
and so $$ \limsup_{n \rightarrow \infty} \frac{1}{n} \log \left( {\sum_{\textbf{a} \in \mathcal{W}_{n+1}} }  e^{S_n \varphi( x_{\textbf{a}} )}  \right) \leq P(\varphi) .$$
The other inequality is easier, we have
$$ \log \Bigg( \underset{ \mathbf{a} \rightsquigarrow b^{(n)}}{\sum_{\textbf{a} \in \mathcal{W}_{n+1}} } e^{S_n \varphi( y_{\textbf{a}}  )} \Bigg) \leq C_1 + \log \left( \sum_{\textbf{a} \in \mathcal{W}_{n+1}} e^{S_n \varphi(x_{\textbf{a}})} \right), $$
which gives us $$ P(\varphi) \leq \liminf \frac{1}{n} \log \left( \sum_{\textbf{a} \in \mathcal{W}_{n+1}} e^{S_n \varphi(x_{\textbf{a}})} \right) . $$ \end{proof}
Another useful formula is the computation of the differential of the pressure.
\begin{theorem}
The map $P : C^1(U,\mathbb{R}) \rightarrow \mathbb{R} $ is differentiable. 
If $\varphi \in C^1(U,\mathbb{R})$ is a normalized potential, then we have:
$$ \forall \psi \in C^1(U,\mathbb{R}), \  (dP)_{\varphi}(\psi) = \int_J \psi d\mu_{\varphi} .$$
\end{theorem}

\begin{proof}
This is Corollary 5.2 in \cite{Ru89}. Loosely, the argument goes as follows. \\

The differentiability is essentially a consequence of the fact that $e^{P(\psi)}$ is an isolated eigenvalue of $\mathcal{L}_\psi$. To compute the differential, consider $v_t \in C^1$ the normalized eigenfunction for $\mathcal{L}_{\varphi + t \psi}$ such that $v_0=1$. We have, for small $t$:
$$ \mathcal{L}_{\varphi + t \psi} v_t = e^{P(\varphi + t \psi)} v_t $$
Hence, $$ \mathcal{L}_{\varphi + t \psi} (\psi v_t ) + \mathcal{L}_{\varphi + t \psi} (\partial_t v_t) = v_t e^{P(\varphi + t \psi)} \frac{d}{dt} P(\varphi + t \psi) + e^{P(\varphi + t \psi)} \partial_t v_t $$
Taking $t=0$ and integrating against $\mu_\varphi$ gives
$$ (d P)_{\varphi}(\psi) = \int_J \mathcal{L}_{\varphi}(\psi) d\mu_\varphi = \int_J \psi d\mu_\varphi.$$ \end{proof}

Now, we are ready to establish large deviation estimates (Theorem 3.2.11). The proof is adapted from \cite{JS16}, subsection 4.

\begin{proof}

Let $\varphi$ be a normalized potential, and let $\psi$ be another $C^1$ potential.
Let $\varepsilon > 0$. Let $j(t) := P\left( (\psi - \int \psi d\mu_\varphi - \varepsilon)t + \varphi \right)$. We know by Theorem A.3 that $j'(0) = -\varepsilon < 0$. Hence, there exists $t_0> 0$ such that $P( (\psi - \int \psi d\mu_\varphi - \varepsilon)t_0 + \varphi ) < 0$. \\

Define $2 \delta_0 := -P\left( (\psi - \int \psi d\mu_\varphi - \varepsilon)t_0 + \varphi \right)   $. We then have
$$\mu_\varphi\left( \left\{ x \in J \ , \ \frac{1}{n} S_n \psi(x) - \int_J \psi d\mu_{\varphi}  \geq \varepsilon \right\} \right) \leq \sum_{\textbf{a} \in C_{n+1}} \mu_\varphi (P_\mathbf{a}) ,$$
where $C_{n+1} := \{ \mathbf{a} \in \mathcal{W}_{n+1} \ | \ \exists x \in P_{\mathbf{a}} , \ S_n \psi(x)/n - \int \psi d\mu_{\varphi}  \geq \varepsilon \} $.
For each $\mathbf{a}$ in some $C_{n+1}$, choose $x_{\mathbf{a} } \in P_{\mathbf{a} }$ such that $S_n \psi(x_{\mathbf{a} })/n - \int \psi d\mu_{\varphi}  \geq \varepsilon$. For the other $\mathbf{a}$, choose $x_{\mathbf{a}} \in P_{\textbf{a}}$ randomly. \\
Now, since $\mu_\varphi$ is a Gibbs measure, there exists $C_0>0$ such that:

$$ \sum_{\textbf{a} \in C_{n+1}} \mu_\varphi (P_\mathbf{a})  \leq C_0 \sum_{\textbf{a} \in C_{n+1}} \exp(S_n \varphi(x_\textbf{a})) $$
$$ \leq C_0 \sum_{\textbf{a} \in C_{n+1}} \exp\left({S_n\left( \left(\psi-\int \psi d\mu_\varphi - \varepsilon\right)t_0 + \varphi \right)(x_\textbf{a} )} \right) $$
$$ \leq C_0 \sum_{\textbf{a} \in \mathcal{W}_{n+1}} \exp\left({S_n\left( \left(\psi-\int \psi d\mu_\varphi - \varepsilon\right)t_0 + \varphi \right)(x_\textbf{a} )} \right) .$$
Then, by lemma 3.7.2, we can write for $n \geq n_0$ large enough:
$$ C_0 \sum_{\textbf{a} \in \mathcal{W}_{n+1}} \exp\left({S_n\left( \left(\psi-\int \psi d\mu_\varphi - \varepsilon\right)t_0 + \varphi \right)(x_\textbf{a} )} \right) \leq C e^{n \delta_0} e^{n P( (\psi - \int \psi d\mu_\varphi - \varepsilon)t_0 + \varphi )} \leq C e^{-  \delta_0 n} ,$$
and so $$ \mu_\varphi\left( \left\{ x \in J \ , \ \frac{1}{n} S_n \psi(x) - \int_X \psi d\mu_{\varphi}  \geq \varepsilon \right\} \right) \leq Ce^{- n \delta_0} .$$
The symmetric case is done by replacing $\psi$ by $-\psi$, and combining the two gives us the desired bound. \end{proof}

\section{Appendix B: Dolgopyat's estimates for a family of twisted transfer operator}

Here, we explain how to prove \textbf{Theorem 3.6.6}. It is a generalization of Theorem 2.5 in \cite{OW17}: we will explain what we need to change in the original paper for the theorem to hold more generally. \\

Proving that a complex transfer operator is eventually contracting is linked to analytic extensions results for dynamical zeta functions, and is often referred to as a spectral gap. Such results are of great interest to study, for example, periodic orbit distribution in hyperbolic dynamical systems (see for example the chapter 5 and 6 in \cite{PP90}), or asymptotics for dynamically defined quantities (as in \cite{OW17} or \cite{PU17}). One of the first result of this kind can be found in a work of Dolgopyat \cite{Do98}, in which he used a method that has been broadly extended since. We can find various versions of Dolgopyat's method in papers of Naud \cite{Na05}, Stoyanov \cite{St11}, Petkov \cite{PS16}, Oh-Winter \cite{OW17}, Li \cite{Li20}, and Sharp-Stylianou \cite{ShSt20}, to only name a few. See Section 1.2.2 for a discussion on Dolgopyat's estimates and their link to the rate of mixing of (suspension) flows. \\

In this section, we will outline the argument of Dolgopyat's method as explained in \cite{OW17} adapted to our general setting. We need three ingredients to make the method work:  (NLI) (non local integrability), (NCP) (another non concentration property), and a doubling property.

\begin{definition}
Define $\tau(x) := \log |f'(x)| \in \mathbb{R}$ and $\theta(x) := \arg f'(x) \in \mathbb{R}/2 \pi \mathbb{Z}.$ The transfer operator in Theorem $3.6.6$ acts on $C^1_b(\mathfrak{U},\mathbb{C})$ and may be rewritten in the form $$ \mathcal{L}_{\varphi,it,l} = \mathcal{L}_{\varphi - i t \mathcal{\tau} - i l \mathcal{\theta}} $$
for some normalized $\varphi \in C^1(U,\mathbb{R})$, $l \in \mathbb{Z}$ and $t \in \mathbb{R}$.
\end{definition}
With those notations, Theorem 3.6.6 can be rewritten as follows.
\begin{theorem}
Suppose that $J$ is not included in a circle.
For any $\varepsilon>0$, there exists $C>0$, $\rho<1$ such that for any $n \geq 1$ and any $t\in \mathbb{R}$, $l \in \mathbb{Z}$ such that $|t|+|l| > 1$,

$$ \| \mathcal{L}_{\varphi - i(t \tau + l \theta)}^n \|_{C^1_b(\mathfrak{U},\mathbb{C})} \leq C (|l|+|t|)^{1+\varepsilon} \rho^n $$

\end{theorem}

Let us recall the three main technical ingredients.

\begin{theorem}[NLI, \cite{OW17} section 3]

There exists $a_0 \in \mathcal{A}$,  $x_1 \in P_{a_0}$ , $N \in \mathbb{N}$, admissible words $\mathbf{a}, \mathbf{b} \in \mathcal{W}_{N+1}$ with $a_0\rightsquigarrow \mathbf{a},\mathbf{b}$, and an open neighborhood $U_0$ of $x_1$ such that for any $n \geq N$, the map $$ (\tilde{\tau}, \tilde{\theta}) := (S_n \tau \circ g_\mathbf{a} - S_n \tau \circ g_\mathbf{b}, S_n\theta \circ g_\mathbf{a} - S_n\theta \circ g_\mathbf{b}) : U_0 \rightarrow \mathbb{R} \times \mathbb{R}/2 \pi \mathbb{Z}$$
is a local diffeomorphism.
    
\end{theorem}

\begin{remark}
Remark 4.7 in \cite{SS20} and Proposition 3.8 in \cite{OW17} points out the fact that (NLI) is a consequence of our non-linear setting, which itself comes from the fact that we supposed that our Julia set is different from a circle. This takes the role of the (UNI) condition that we used in Chapter 2.
\end{remark}

\begin{theorem}[NCP, \cite{OW17} section 4]

For each $n \in \mathbb{N}$, for any $\mathbf{a} \in \mathcal{W}_{n+1}$, there
exists $0 < \delta < 1$ such that, for all $x \in P_\mathbf{a}$, all $w \in \mathbb{C}$ of unit length, and all $\varepsilon \in (0, 1) $,
$$B(x,\varepsilon) \cap \{ y \in P_\mathbf{a}, \ |\langle y - x, w \rangle| > \delta \varepsilon \} \neq \emptyset $$
where $\langle a + bi, c + di \rangle = ac + bd$ for $a, b, c, d \in \mathbb{R}$.
\end{theorem}

\begin{remark}
The NCP is a consequence of the fractal behavior of our Julia set. This time, if $J$ is included in any smooth set, the NCP fails. But in our case, this is equivalent to being included in a circle, see \cite{ES11}. Notice that this non concentration property has nothing to do with our previous non concentration hypothesis.
\end{remark}

\begin{theorem}[Doubling]

Let, for $a \in \mathcal{A}$, $\mu_a$ be the equilibrium measure $\mu_\varphi$ restricted to $P_a$. Then each $\mu_a$ is doubling, that is:

$$ \exists C>0, \ \forall x \in P_a, \ \forall r<1, \  \mu_a(B(x,2r)) \leq C \mu_a(B(x,r)). $$

\end{theorem}

\begin{proof}
It follows from Theorem A.2 in \cite{PW97} that $\mu_\varphi$ is doubling: the proof uses the conformality of the dynamics. To prove that $\mu_a := \mu_{| P_a}$ is still doubling, which is not clear a priori, we follow Proposition 4.5 in \cite{OW17} and prove that there exists $c>0$ such that for any $a \in \mathcal{A}$, for any $x \in P_a$, and for any $r>0$ small enough, $$ \frac{\mu_\varphi(B(x,r) \cap P_a)}{\mu_\varphi(B(x,r))} > c .$$
For this we use a Moran cover $\mathcal{P}_r$ associated to our Markov partition, see the proof of Proposition 3.2.10 for a definition. Recall that any element $P \in \mathcal{P}_r$ have diameter strictly less than $r$, and recall that there exists a constant $M>0$ independent of $x$ and $r$ such that we can cover the ball $B(x,r)$ with $M$ elements of $\mathcal{P}_r$. Moreover, Lemma 2.2 in \cite{WW17} allows us to do so using elements $P\in \mathcal{P}_r$ of the form $P_\mathbf{a}$ for $\mathbf{a}$ in some $\mathcal{W}_n$, $N_0 \leq n \leq N_0 + L$ for some $N_0(x,r)$ and some constant $L$ (independent of $x$ and $r$). We can then conclude as follows. Let $P^{(1)} , \dots P^{(M)} \in \mathcal{P}_r$ that covers $B(x,r)$. There exists $i$ such that $P^{(i)} \subset P_a$. Hence, by the Gibbs property of $\mu_\varphi$:

$$ \frac{\mu_\varphi(B(x,r) \cap P_a)}{\mu_\varphi(B(x,r))} \geq  \frac{\mu_\varphi(P^{(i)})}{ {\sum_{j=1}^M} \mu_\varphi(P^{(j)})} \geq M^{-1} C_0^{-2} e^{-(L+2)\| \varphi \|_\infty}.$$
 \end{proof}

\begin{remark}

The doubling property (or Federer property) is a regularity assumption made on the measure that is central for the execution of this version of Dolgopyat's method. It allows us to control integrals over $J$ by integrals over smaller pieces of $J$, provided some regularity assumption on the integrand. 

\end{remark}

Now we will outline the argument of Dolgopyat's method as used in $\cite{OW17}$. It can be decomposed into four main steps. \\

\underline{\textbf{Step 1:}} We reduce Theorem 3.8.2 to a $L^2(\mu)$ estimate. \\
\begin{quote}
We need to define a modified $C^1$ norm. Denote by $\| \cdot \|_r$ a new norm, defined by
$$ \|h\|_r := \left \{
\begin{array}{rcl}
\|h\|_{\infty,\mathfrak{U}} + \frac{\|\nabla h\|_{\infty,\mathfrak{U}}}{r}  \ \text{if} \ r \geq 1  \\
\|h\|_{\infty,\mathfrak{U}} + \| \nabla h \|_{\infty,\mathfrak{U}} \ \text{if} \ r < 1 
\end{array}
\right.$$
Moreover, we do a slight abuse of notation and write $\mu$ for $\sum_{a \in \mathcal{A}} \mu_a$, seen as a measure on $\mathfrak{U}$. This measure is supported on $J$, seen as the set $\bigsqcup_a P_a \subset \bigsqcup_a U_a = \mathfrak{U}$.
The first step is to show that Theorem 3.8.2 reduces to the following claim.

\begin{theorem}

Suppose that the Julia set of $f$ is not contained in a circle. Then there exists $C>0$ and $ \rho \in (0,1)$ such that for any $h \in C^1_b(\mathfrak{U},\mathbb{C})$
and any $n \in \mathbb{N}$,
$$ ||\mathcal{L}_{\varphi-i(t \tau + l \theta)}^n h||_{L^2(\mu)} \leq C \rho^n \|h \|_{|t|+|l|} $$
for all $t \in \mathbb{R}$ and $l \in \mathbb{Z}$ with $|t| + |l| \geq 1$.
\end{theorem}

A clear account for this reduction may be found in \cite{Na05}, section 5. This step holds in great generality without any major difficulty. Intuitively, Theorem 3.8.2 follows from Theorem 3.8.9 by the Lasota-Yorke inequality, by the quasicompactness of $\mathcal{L}_\varphi$, and by the Perron-Frobenius-Ruelle theorem, which implies that $\mathcal{L}^N_\varphi h$ is comparable to $\int h d\mu$ for $N$ large. The difference between the two can be controlled using $C^1$ bounds. \\
\end{quote}

\underline{\textbf{Step 2:}} We show that the oscillations in the sum induce enough cancellations.  \\
\begin{quote}

Loosely, the argument goes as follows. 
We write, for a well chosen and large $N$:

$$ \forall x \in U_b, \ \mathcal{L}^N_{\varphi - i (t \tau + l \theta)}h(x) = \underset{\mathbf{a}  \rightsquigarrow b} {\sum_{\mathbf{a} \in \mathcal{W}_{N+1}} } e^{ i (t S_N \tau + l S_N \theta )(g_\mathbf{a}(x))} h(g_\mathbf{a}(x))e^{S_N \varphi (g_\mathbf{a}(x) ) } .$$

If we choose $x$ in a suitable open set $\widehat{S} \subset U_0$, (NLI) and (NCP) tells us that we can extract words $\mathbf{a}$ and $\mathbf{b}$ from this sum such that some cancellations happen. Indeed, if we isolate the term given by the words from (NLI),
$$  e^{ i (t S_N \tau + l S_N \theta )(g_\mathbf{a}(x))} h(g_\mathbf{a}(x))e^{S_N \varphi (g_\mathbf{a}(x) ) } +  e^{ i (t S_N \tau + l S_N \theta )(g_\mathbf{b}(x))} h(g_\mathbf{b}(x))e^{S_N \varphi (g_\mathbf{b}(x) ) } ,$$
we see that a difference in argument might give us some cancellations. The effect of $h$ in the difference of argument can be carefully controlled by the $C^1$ norm of $h$. The interesting part comes from the complex exponential. The difference of arguments of this part is
$$ t (S_N \tau \circ g_\mathbf{a} -  S_N \tau \circ g_\mathbf{b}) + l (S_N \theta \circ g_{\mathbf{a}} - S_N \theta \circ g_{\mathbf{b}}) ,$$
which might be rewritten in the form
$$ \langle (t,l) , (\tilde{\tau},\tilde{\theta}) \rangle .$$
Then we proceed as follows. Choose a large number of points $(x_k)$ in $U_0$. If, for a given $x_k$, the difference of argument $\langle (t,l),(\tilde{\tau},\tilde{\theta}) \rangle (x_k)$ is not large enough, we might use (NCP) to construct another point $y_k$ next to $x_k$ such that $ \langle (t,l),(\tilde{\tau},\tilde{\theta}) \rangle (y_k) $ become larger. The construction goes as follows: (NLI) ensures that $\nabla \langle (t,l) , (\tilde{\tau},\tilde{\theta}) \rangle (x_k) =: w_k \neq 0 $. Hence, the direction $\widehat{w}_k := \frac{w_k}{|w_k|}$ is well defined. (NCP) then ensure the existence of some $y_k \in J$ which is very close to $x_k$ and such that $x_k - y_k$ is a vector pointing in a direction comparable to $\widehat{w}_k$. As we are following the gradient of $\langle (t,l) , (\tilde{\tau},\tilde{\theta}) \rangle$, we are sure that $\langle (t,l) , (\tilde{\tau},\tilde{\theta}) \rangle(y_k)$ will be larger than before. \\

We then let $S$ be the set containing the points where the difference in argument is large enough, so it contains some $x_k$ and some $y_k$. This large enough difference in argument that is true in $S$ is also true in a small open neighborhood $\widehat{S}$ of $S$. \\

We can then write, for $x \in \widehat{S}$, an inequality of the form: 
$$ \left|  e^{ i (t S_N \tau + l S_N \theta )(g_\mathbf{a}(x))} h(g_\mathbf{a}(x))e^{S_N \varphi (g_\mathbf{a}(x) ) } +  e^{ i (t S_N \tau + l S_N \theta )(g_\mathbf{b}(x))} h(g_\mathbf{b}(x))e^{S_N \varphi (g_\mathbf{b}(x) ) } \right| $$
$$ \leq (1-\eta) |h(g_\mathbf{a}(x))| e^{S_N \varphi(g_\mathbf{a}(x))} + |h(g_\mathbf{b}(x))| e^{S_N \varphi(g_\mathbf{b}(x))}  ,$$

where the $(1-\eta)$ comes in front of the part with the smaller modulus. This is, in spirit, Lemma 5.2 of \cite{OW17}.
We can then summarize the information in the form of a function $\beta$ that is $1$ most of the time, but that is less than $(1-\eta)^{1/2}$ on $\widehat{S}$.
This allows us to write the following bound:
$$ \mathcal{L}_{\varphi-i(t \tau + l \theta)}^N h \leq  \mathcal{L}_{\varphi}^N \left( |h| \beta \right).$$

One of the main difficulty of this part is to make sure that $\widehat{S}$ is a set of large enough measure, while still managing not to make the $C^1$-norm of $\beta$ explode. All the hidden technicalities in this part forces us to only get this bound for \emph{a well chosen $N$}. 
\end{quote}

\underline{\textbf{Step 3:}} These cancellations allow us to compare $\mathcal{L}$ to an operator that is contracting on a cone. 

\begin{quote}
We define the following cone, on which the soon-to-be-defined Dolgopyat operator will be well behaved. Define
$$ K_R(\mathfrak{U}) := \{ H \in C^1_b(\mathfrak{U}) \ | \ H \text{ is positive, and} \ |\nabla H| \leq R H  \} ,$$

and then define the Dolgopyat operator by $\mathcal{M} H := \mathcal{L}_\varphi^N (H \beta)$. We then show that, if $H \in K_R(\mathfrak{U})$ (for a well chosen $R$):
\begin{enumerate}
    \item $\mathcal{M}(H) \in K_R(\mathfrak{U}) $
    \item $ \| \mathcal{M}(H) \|_{L^2(\mu)}^2 \leq (1-\varepsilon) \| H \|_{L^2(\mu)}^2$.
\end{enumerate}

The first point is done using the Lasota-Yorke inequalities, see Lemma 5.1 in \cite{OW17}. The second point goes, loosely, as follows.

We write, using Cauchy-Schwartz:
$$ (\mathcal{M}H)^2 = \mathcal{L}_\varphi^N(H \beta)^2 \leq \mathcal{L}^N_\varphi(H^2)\mathcal{L}^N_\varphi(\beta^2)  . $$

On $\widehat{S}$, the cancellations represented in the function $\beta$ spread, thanks to the fact that $\varphi$ is normalized, as follows:

$$ \mathcal{L}_\varphi^N (\beta^2) = \sum_{\mathbf{a}} e^{ S_N \varphi \circ g_{\mathbf{a}} }   \beta^2 \circ g_\mathbf{a}  $$
$$ = \sum_{\mathbf{a} \text{ where } \beta=1 } e^{S_N \varphi \circ g_\mathbf{a} } \beta^2 \circ g_\mathbf{a} + \sum_{\mathbf{a} \text{ where } \beta \text{ is smaller} } e^{S_N \varphi \circ g_\mathbf{a} } \beta^2 \circ g_\mathbf{a} $$
$$ \leq \sum_{\mathbf{a} \text{ where } \beta=1 } e^{S_N \varphi \circ g_\mathbf{a} } + \sum_{\mathbf{a} \text{ where } \beta \text{ is smaller} } e^{S_N \varphi \circ g_\mathbf{a} } (1-\eta)  $$
$$ = 1 - \eta e^{-N \|\varphi\|_\infty}.$$

Then, we use the doubling property of $\mu=\sum_a \mu_a$ and the control given by the fact that $H \in K_R(\mathfrak{U})$ to bound the integral on all $J$ by the integral on $\widehat{S}$:

$$ \int_J  \mathcal{L}^N_\varphi(H^2) d\mu \leq C_0 \int_{\widehat{S}}  \mathcal{L}^N_\varphi(H^2) d\mu  .$$

Hence, we can write, using the fact that $\mathcal{L}_\varphi$ preserves $\mu$ and the previously mentioned Cauchy-Schwartz inequality: 
$$ \|H\|_{L^2(\mu)}^2 - \|\mathcal{M}H\|_{L^2(\mu)}^2 \geq \int_J \left(\mathcal{L}_\varphi^N(H^2) - \mathcal{L}_\varphi^N(H^2) \mathcal{L}_\varphi^N( \beta^2)  \right)d\mu $$
$$ \geq  \int_{\widehat{S}} \left(\mathcal{L}_\varphi^N(H^2) - \mathcal{L}_\varphi^N(H^2) \mathcal{L}_\varphi^N( \beta^2)  \right)d\mu $$
$$ \geq \eta e^{-N \|\varphi\|_\infty} \int_{\widehat{S}}\mathcal{L}_\varphi^N(H^2) d\mu  $$
$$ \geq \eta C_0^{-1} e^{-N \|\varphi\|_\infty} \| H \|_{L^2(\mu)}^2 := \varepsilon \| H \|_{L^2(\mu)}^2 .$$

Hence $$  \|\mathcal{M}H\|_{L^2(\mu)}^2 \leq (1-\varepsilon)  \|H\|_{L^2(\mu)}^2 . $$

\end{quote}
\underline{\textbf{Step 4:}} We conclude by an iterative argument. \\
\begin{quote}

To conclude, we need to see that we may bound $h$ by some $H \in K_R(\mathfrak{U})$, and also that the contraction property is true for all $n$, not just $N$.

For any $n=kN$, we can inductively prove our bound. If $k=1$, we can choose $H_0:= \|h\|_{|t|+|l|}$. Then, $|h| \leq H_0$ and so
$$ \| \mathcal{L}_{\varphi - i(t \tau + l \theta) }^N h \|_{L^2(\mu)} \leq \| \mathcal{M} H_0 \|_{L^2(\mu)} \leq (1-\varepsilon)^{1/2} \|h\|_{(|t|+|l|)} .$$

Then, choosing $H_{k+1} := \mathcal{M} H_k \in K_R(\mathfrak{U})$, we can proceed to the next step of the induction and get
$$  \| \mathcal{L}_{\varphi - i(t \tau + l \theta) }^{kN} h \|_{L^2(\mu)} \leq \| \mathcal{M} H_{k-1} \|_{L^2(\mu)} \leq (1-\varepsilon)^{k/2} \|h\|_{(|t|+|l|)} .$$

Finally, if $n=kN+r$, with $0 \leq j \leq N-1$, we write $$ \| \mathcal{L}_{\varphi - i(t \tau + l \theta) }^{kN+r} h \|_{L^2(\mu)} \leq (1-\varepsilon)^k \|\mathcal{L}_\varphi^r h\|_{(|t|+|l|)}  \lesssim (1-\varepsilon)^{n/(2N)} \|h\|_{(|t|+|l|)}  ,$$

and the proof is done.

\end{quote}

\cleardoublepage

\chapter{The Fourier dimension of basic sets I}

\section{Introduction}
In this Chapter, we begin our study of the lower Fourier dimension of basic sets for nonlinear Axiom A diffeomorphisms. The main goal of this Chapter is to prove a reduction statement, from \say{Fourier decay in the unstable direction} to a nonconcentration hypothesis to check. We will conclude this Chapter by checking this nonconcentration hypothesis in an easy case, where some bunching condition occurs. Chapter 5 will be devoted to establishing the nonconcentration estimates in the case where the dynamics is (nonlinear, and) area-preserving on a surface, thus proving positivity of the lower Fourier dimension in this context. \\ 

Let us precise our setting. Let $M$ be a complete Riemannian manifold, and let $f:M \rightarrow M$ be a $C^{2+\alpha}$ Axiom A diffeomorphism. We consider a basic set $\Omega \subset M$ for $f$. Recall that, if $p \in \Omega$, then we have a spliting of the tangent space at $p$ of the form $T_p M = E^u(p) \oplus E^s(p)$. The dynamics contract along the stable direction $E^s$, and expand along the unstable direction $E^u$. These distributions integrates into local stable (resp. local unstable) manifolds, denoted $W^s_{loc}(p)$ (resp. $W^u_{loc}(p)$). Those local laminations are transverse, and one can show that, for $p,q \in \Omega$ close enough, the point $ W^s_{loc}(p) \cap W^u_{loc}(q) =: [p,q] \in \Omega$ is well defined. \\

We will work under the assumption that $\dim(E^u)=1$. In this case, the vector bundle $E^s$ is $C^{1+\alpha}$ on $\Omega$. This is always true if $M$ is a surface. Before quoting our main results, let us define some notations: details on Axiom A diffeomorphisms and on the Thermodynamical formalism in this context will be given in the next subsection. 

\begin{definition}
For all $p \in \Omega$, denote by $v^u_p \in E^u(p)$ some unit vector. If $N$ is another Riemannian manifold, and if $\psi : M \rightarrow N$ is some $C^1$ map, we define, on $\Omega$, $|\partial_u \psi(x)| := |(d\psi)_x(v^u_x)|.$ We further define the (opposite of the) \emph{geometric potential} $\tau_f : \Omega \rightarrow \mathbb{R}$ by the formula
$$\tau_f(x) := \ln |\partial_u f(x)|.$$
\end{definition}

\begin{definition}
Fix once and for all an open neighborhood of the diagonal $\widetilde{\text{Diag}} \subset \Omega^2$ on which the bracket $(p,q) \in \widetilde{\text{Diag}} \mapsto [p,q] \in \Omega$ is well defined. We then define $\Delta : \widetilde{\text{Diag}} \rightarrow \mathbb{R} $ by the formula:
$$ \Delta(p,q) := \sum_{n \in \mathbb{Z}} \Big( \tau_f(f^n p) - \tau_f(f^n [p,q]) - \tau_f(f^n [q,p]) + \tau_f(f^n q) \Big). $$
\end{definition}

We are ready to quote the (sufficient) nonlinearity condition on the dynamics under which Fourier decay will hold. This is a non-concentration bound on $\Delta$.

\begin{definition}
Let $\mu$ be an equilibrium state on $\Omega$. We say that the dynamics $f$ satisfies the Quantitative NonLinearity condition (QNL) (with respect to $\mu$) if there exists $\gamma_{QNL} \in (0,1)$ and $C \geq 1$ such that:
$$ \forall \sigma > 0, \ \mu^{\otimes 2} \Big ( (p,q) \in \widetilde{\text{Diag}} \ , \ |\Delta(p,q)| \leq \sigma \Big) \leq C \sigma^{\gamma_{QNL}}. $$
\end{definition}

\begin{remark}
This condition can be though as a non-integrability condition on the suspension flow with base dynamics $f$ and roof function $\tau_f$. See  \cite{Do00} for useful details.
\end{remark}


We are ready to quote our main Theorem: we can reduce Fourier decay to checking (QNL). 

\begin{theorem}
Let $M$ be a complete manifold. Let $\alpha \in (0,1)$. Let $f:M \rightarrow M$ be a $C^{2+\alpha}$ Axiom A diffeomorphism. Let $\Omega$ be a basic set for $f$ with codimension one stable lamination. Let $\phi:\Omega \rightarrow \mathbb{R}$ be some Hölder regular potential, and denote its associated equilibrium state $\mu \in \mathcal{P}(\Omega)$. Suppose furthermore that (QNL) holds (for $\mu$). Then, there exists $\rho_1,\rho_2>0$ such that the following holds. \\

For any $\alpha$-Hölder map $\chi : \Omega 
\rightarrow \mathbb{C}$, there exists $C=C({f,\mu,\chi}) \geq 1$ such that, for any $\xi \geq 1$ and any $C^{1+\alpha}$ phase $\psi: \Omega \rightarrow \mathbb{R}$ satisfying $$\|\psi\|_{C^{1+\alpha}} + \big(\inf_{\text{supp}(\chi)} |\partial_u \psi| \ \big)^{-1} \leq \xi^{\rho_1}, $$
we have:
$$ \Big{|} \int_\Omega e^{i \xi \psi } \chi d\mu \Big{|} \leq C \xi^{-\rho_2} .$$
In particular,  $\underline{\text{dim}}^{E^u}_{F,C^{1+\alpha}}(\mu)>0$ (see Definition 1.1.27).
\end{theorem}

Once this Theorem is proved, the difficulty is to check the nonlinearity condition. The case where the dynamics is area-preserving on a surface is quite difficult and will be discussed in Chapter 5. In the case where $\Omega$ is an attractor satisfying a suitable bunching condition, we can actually prove a generic Fourier decay statement using some already existing Dolgopyat's estimates. This will be discussed in the end of this Chapter. Let us define the nonlinearity condition that we will use in this context.

\begin{definition}

Define $\Omega_{\text{per}} := \{ x \in \Omega \ \text{periodic} \}$. For any $x \in \Omega_{\text{per}}$ with minimal period $n$, define its local unstable Lyapunov exponent $\widehat{\lambda}(x)$ by the formula

$$ \widehat{\lambda}(x) = \frac{1}{n} \sum_{k=0}^{n-1} \tau_f(f^k(x)) .$$
We say that $\Omega$ satisfy the \say{Lyapunov NonLinearity condition} (LNL) if 
$$ \text{dim}_\mathbb{Q} \text{Vect}_\mathbb{Q} \widehat{\lambda}\left( \Omega_\text{per} \right) = \infty .$$
\end{definition}

We will prove in Section 4.7 that this nonlinearity condition (LNL) is generic on the dynamics. This nonlinearity condition will be enough to get Fourier decay in the unstable direction as soon as some bunching condition is satisfied. Let us give some details.

\begin{definition}
We introduce a bunching condition: see for example \cite{Ha97}, \cite{GRH21}, or \cite{ABV14}, \cite{HP69} for similar conditions in the case of flows. We say that the bunching condition (B) is satisfied if
$$ \forall p \in \Omega, \ |\partial_u f(p)| \cdot \| { {(df)_{p}}_{|E^s_p} } \| < 1 ,$$
where the norm on $(df)_{|E^s}$  is the operator norm induced by the Riemannian metric on $M$. It implies that $\Omega$ is a proper attractor. In this case, by continuity of $df$ and compactness of $\Omega$, there exists some $\alpha > 0$ such that $$ \forall p \in \Omega, \ |\partial_u f(p)| \cdot \| { {(df)_{p}}_{|E^s_p} } \|^{1/(1+\alpha)} \leq 1 .$$
In this case, the stable bunching parameter $b^s$ (see $\cite{GRH21}$, and \cite{Ha97} for the pointwise version) satisfies
$$ b^s(p) = 1 + \frac{\ln  \| { {(df)_{p}}_{|E^s_p} } \|^{-1} }{\ln |\partial_u f(p)| } \geq 2+\alpha. $$
This bunching condition implies that the stable lamination is $C^{2+\alpha}$. Also, notice that in dimension 2, the condition $$ \exists N \geq 0, \ \forall p \in \Omega, \ |\det (df^N)_p| < 1 $$
implies our bunching condition for some choice of riemannian metric. That is, we only need the dynamical system $(\Omega,f)$ to be dissipative.

\end{definition}

\begin{theorem}
  Let $M$ be a complete manifold.Let $\alpha \in (0,1)$. Let $f:M \rightarrow M$ be a $C^{2+\alpha}$ Axiom A diffeomorphism. Suppose that $f$ has an attractor $\Omega$ with codimension 1 stable lamination. Moreover, suppose that $f$ satisfies the nonlinearity condition (LNL) and the bunching condition (B). Let $\phi:\Omega \rightarrow \mathbb{R}$ be some Hölder regular potential, and denote its associated equilibrium state $\mu \in \mathcal{P}(\Omega)$. Then there exists $\rho_1,\rho_2>0$ such that the following holds. \\

For any $\alpha$-Hölder map $\chi : \Omega 
\rightarrow \mathbb{C}$, there exists $C=C({f,\mu,\chi}) \geq 1$ such that, for any $\xi \geq 1$ and any $C^{1+\alpha}$ phase $\psi: \Omega \rightarrow \mathbb{R}$ satisfying $$\|\psi\|_{C^{1+\alpha}} + \big( \inf_{\text{supp}(\chi)} |\partial_u \psi| \ \big)^{-1} \leq \xi^{\rho_1}, $$
we have:
$$ \Big{|} \int_\Omega e^{i \xi \psi } \chi d\mu \Big{|} \leq C \xi^{-\rho_2} .$$
In particular,  $\underline{\text{dim}}^{E^u}_{F,C^{1+\alpha}}(\mu)>0$ (see Definition 1.1.27).
\end{theorem}

If the dynamics acts on a surface, this can be reformulated as follow: every attractor for nonlinear and dissipative Axiom A diffeomorphism on a surface exhibit polynomial Fourier decay of equilibrium states in the unstable direction. \\

The strategy to prove Theorem 4.1.5 and Theorem 4.1.8 adapts the method we used in the previous Chapters in this context. Here, our main goal is to find a way to adapt the existing method for expanding maps in the case of Axiom A diffeomorphisms. For this, we use a standard construction using Markov partitions and then reduce the problem to the one-dimensional case.
This Chapter is divided in the following way: 

\begin{itemize}
    \item In section 4.2, we collect facts about the thermodynamical formalism in the context of Axiom A diffeomorphisms and recall classic constructions. Section 4.2.7 is devoted to a preliminary regularity result for equilibrium states in our context. In section 4.2.8 we state a large deviation result about Birkhoff sums.
    \item In section 4.3 we use  large deviations to derive order of magnitude for some dynamically-related quantities.
    \item The proof of Theorem 4.1.5 and Theorem 4.1.8 begins in section 4.4. Using the invariance of equilibrium state by the dynamics, we first reduce the Fourier transform to an integral on a union of local unstable manifolds, thus reducing the problem to a setting where the underlying dynamical system is a one dimensional expanding map, similarly to Chapter 2.
    \item We use a transfer operator to carefully approximate the integral by a sum of exponentials. We then apply a version of the sum-product phenomenon.
    \item We establish Theorem 4.1.5 in section 4.6, by working with the condition (QNL) and showing that it implies the non-concentration estimates needed to apply the sum-product phenomenon. 
    \item We establish Theorem 4.1.8 in section 4.7, following the method of Chapter 3 to check non-concentration.
    \item Section 4.8 is devoted to the proof of genericity of (LNL). Section 4.9 is devoted to the construction of a concrete example of solenoid on which Theorem 4.1.8 applies.
\end{itemize}

\section{Thermodynamic formalism for Axiom A diffeomorphisms}

\subsection{Axiom A diffeomorphisms and basic sets}

We recall standard results about Axiom A diffeomorphisms. An introduction to the topic can be found in \cite{BS02}. A more in-depth study can be found in \cite{KH95}. For an introduction to the thermodynamic formalism of Axiom A diffeomorphisms, we suggest the classic lectures notes of Bowen \cite{Bo75}. 

\begin{definition}

Let $f: M \rightarrow M$ be a diffeomorphism of a complete  $\mathcal{C}^\infty$ Riemannian manifold $M$. A compact set $\Lambda \subset M$ is said to be hyperbolic for $f$ if $f(\Lambda) = \Lambda$, and if for each $x \in \Lambda$, the tangent space $T_xM$ can be written as a direct sum $$ T_xM = E^s_x \oplus E^u_x $$
of subspaces such that
\begin{enumerate}
    \item $\forall x\in \Lambda$,  $(df)_x(E^s_x) = E^s_{f(x)}$ and $(df)_x(E^u_x) = E^u_{f(x)} $
    \item $\exists C>0, \ \exists \kappa \in (0,1), \ \forall x \in \Lambda, $
    $$ \forall v \in E^s_x, \ \forall n \geq 0,  \ \|(df^n)_x(v)\| \leq C \kappa^n \| v \| $$
    and $$ \forall v \in E^u_x, \ \forall n \geq 0, \  \|(df^n)_x(v)\| \geq C^{-1} \kappa^{-n} \| v \| .$$
\end{enumerate}

\end{definition}

It then follows that $E_x^s$ and $E_x^u$ are continuous sub-bundles of $T_xM$. 

\begin{remark}
We can always choose the metric so that $C=1$ in the previous definition: this is called an adapted metric (or a Mather metric) and we will fix one from now on. See \cite{BS02} for a quick proof.
\end{remark}

\begin{definition}
A point $x \in M$ is called non-wandering if, for any open neighborhood $U$ of $x$, $$ U \cap \bigcup_{n > 0} f^n(U)  \neq \emptyset.$$ We denote the set of non-wandering points by $\Omega(f)$.
\end{definition}

It is easy to check that the non-wandering set of $f$ is a closed invariant subset of $M$. Also, any periodic point of $f$ can be seen to lie in $\Omega(f)$. The definition of Axiom A diffeomorphisms is chosen so that the dynamical system $(f,\Omega(f))$ exhibit a chaotic behavior similar to the one found in symbolic dynamics. Namely:

\begin{definition}
A diffeomorphism $f: M \rightarrow M$ is said to be Axiom A if \begin{itemize}
    \item $\Omega(f)$ is a hyperbolic set for $f$ (in particular, it is compact),
    \item $ \Omega(f) = \overline{\{ x \in M \ | \ \exists n > 0, \ f^n(x) = x  \} } . $
\end{itemize} 
\end{definition}

In general, the non-wandering set of an Axiom A diffeomorphism can be written as the union of smaller invariant compact sets. Those are the sets on which we usually work.

\begin{theorem}
One can write $\Omega(f) = \Omega_1 \cup \dots \cup \Omega_k$, where $\Omega_i$ are nonempty compact disjoint sets, such that 
\begin{itemize}
    \item $f(\Omega_i)=\Omega_i$, and $f_{|\Omega_i}$ is topologically transitive,
    \item $\Omega_i = X_{1,i} \cup \dots \cup X_{r_i,i}$ where the $X_{j,i}$ are disjoint compact sets,  $f(X_{j,i}) = X_{j+1,i}$ ($X_{r_i+1,i}=X_{1,i}$) and $f^{r_i}_{| X_{j,i}}$ are all topologically mixing.
\end{itemize}
The sets $\Omega_i$ are called \textbf{basic sets}.
\end{theorem}

\begin{remark}
Notice that any basic set $\Omega$ is either a perfect set or a periodic orbit. Indeed, if $p \in \Omega$ were an isolated point, then by transitivity $p$ should have dense orbit in $\Omega$, but at the same time be a periodic point. 
We will suppose in all of this thesis that $\Omega$ is a perfect set: the conditions (QNL) or (LNL) forbid $\Omega$ to be a periodic orbit anyways.
\end{remark}
From now on, we fix a perfect basic set $\Omega \subset M$ for an Axiom A diffeomorphism $f$. (We do not ask anything on the dimension of $E^u$ yet.)

\subsection{Equilibrium states}

\begin{definition}

We say that a map $\varphi: X \subset M \rightarrow \mathbb{R}$ is Hölder if there exists $\alpha \in (0,1)$ and $C > 0$ such that $\forall x,y \in X, \ |\varphi(x) - \varphi(y)| \ \leq C d(x,y)^\alpha $, where $d$ is the natural geodesic distance induced by the metric on $M$. We note $C^\alpha(X)$ the set of $\alpha$-Hölder maps on $X$. If $\alpha > 1$, we denote by $C^{\alpha}(M)$ the set of functions that are $\lfloor \alpha \rfloor$ times differentiable, with $(\alpha - \lfloor \alpha \rfloor)$-Hölder derivatives.

\end{definition}

\begin{remark}

A theorem by McShane \cite{Mc34} proves that any $\alpha$-Hölder map defined on a subset of $M$ can always be extended to an $\alpha$-Hölder map on all $M$. Hence, even if for some definitions the potentials $\varphi$ need only to be defined on $\Omega$, we will always be able to consider them as maps in $C^\alpha(M)$ if necessary.

\end{remark}

\begin{definition}{\cite{Bo75}, \cite{Ru78}}
Let $\psi : \Omega \rightarrow \mathbb{R}$ be a Hölder potential.
Define the pressure of $\psi$ by
$$ P(\psi) := \sup_{\mu \in \mathcal{P}_f(\Omega)} \left\{ h_f(\mu) + \int_\Omega \psi d\mu \right\}, $$
where $\mathcal{P}_f(\Omega)$ is the compact set of all probability measures supported on $\Omega$ that are $f$-invariant, and where $h_f(\mu)$ is the entropy of $\mu$. There exists a unique measure $\mu_\psi \in M_f(\Omega)$ such that $$ P(\psi) = h_f(\mu_\psi) + \int_\Omega \psi d\mu_\psi .$$
This measure has support equal to $\Omega$, is ergodic on $(\Omega,f)$, and is called the equilibrium state associated to $\psi$.

\end{definition}

Two particular choices of potentials stands among the others. The first one is the \textbf{constant potential} $\psi := C$. In this case, the equilibrium measure is the \textbf{measure of maximal entropy} (which is known to be linked with the repartition of periodic orbits, see for example \cite{Be16}). \\

Another natural choice is the \textbf{geometric potential} $$ \psi := - \log |\det (df)_{|u} |, $$
where $\det (df)_{|u}$ stands for the determinant of the linear map $ {(df)_x} : E^u_x \rightarrow E^u_{f(x)} $, which is well defined up to a sign since the $E^u$ are equipped with a scalar product. In the case where $\Omega$ is an attractor, that is, if there exists a neighborhood $U$ of $\Omega$ such that $$\bigcap_{n \geq 0} f^n(U) = \Omega,$$ then the associated equilibrium measure is called a \textbf{SRB measure}. In this case, for any continuous function $g:U \rightarrow \mathbb{R}$ and for Lebesgue almost all $x \in U$,

$$ \frac{1}{n} \sum_{k=0}^{n-1} g(f^k(x)) \underset{n \rightarrow \infty}{\longrightarrow} \int_{\Omega} g \ d\mu_{SRB} ,$$

which allows us to think of this measure as the \say{physical equilibrium state} of the dynamical system. See \cite{Yo02} for more details on SRB measures, and also the last chapter of $\cite{Bo75}$.

\subsection{(Un)stable laminations, bracket and holonomies}

In this subsection, we will recall some results about the existence of stable/unstable laminations, some regularity results in our particular case, and some consequences on the regularity of holonomies.

\begin{definition}

Let $x \in \Omega$. For $\varepsilon>0$ small enough, we define the local stable and unstable manifold at $x$ by $$ W^s_{\epsilon}(x) := \{ y \in M \ | \ \forall n \geq 0, \ d(f^n(x),f^n(y)) \leq \varepsilon \}, $$
$$ W^u_\varepsilon(x) := \{ y \in M \ | \ \forall n \leq 0, \ d(f^n(x),f^n(y)) \leq \varepsilon \} .$$
We also define the global stable and unstable manifolds at $x$ by
$$ W^s(x) := \{ y \in M \ |  \ d(f^n(x),f^n(y)) \underset{n \rightarrow +\infty}{\longrightarrow} 0 \} ,$$
$$ W^u(x) := \{ y \in M \ | \ d(f^n(x),f^n(y)) \underset{n \rightarrow -\infty}{\longrightarrow} 0 \}.  $$

\end{definition}

\begin{theorem}[\cite{BS02}, \cite{KH95}, \cite{Bo75}]

Let $f$ be a $\mathcal{C}^r$ Axiom A diffeomorphism and let $\Omega$ be a basic set.
For $\varepsilon>0$ small enough and for $x \in \Omega$:
\begin{itemize}
    \item $W^s_\varepsilon(x)$ and $W^u_\varepsilon(x)$  contains $\mathcal{C}^r$ embedded disks,
    \item $\forall y \in \Omega \cap W^s_\varepsilon(x), \ T_y W^s_\varepsilon (x) = E^s_y$,
    \item $\forall y \in \Omega \cap W^u_\varepsilon(x), \ T_y W^u_\varepsilon (x) = E^u_y$,
    \item $f(W^s_\varepsilon(x)) \subset \subset W^s_\varepsilon(f(x))$ and $ f(W^u_\varepsilon(x)) \supset \supset W^u_\varepsilon(f(x)) $ where $\subset \subset$ means \say{compactly included},
    \item $ \forall y \in W^s_\varepsilon(x), \ \forall n \geq 0,  \ d^s(f^n(x),f^n(y)) \leq  \kappa^n d^s(x,y) $ where $d^s$ denotes the geodesic distance on the submanifold $W^s_\varepsilon$,
    \item $ \forall y \in W^u_\varepsilon(x), \ \forall n \geq 0, \ d^u(f^{-n}(x),f^{-n}(y)) \leq  \kappa^{n} d^u(x,y) $ where $d^u$ denotes the geodesic distance on the submanifold $W^u_\varepsilon$.
\end{itemize}
Moreover
$$ \bigcup_{n \geq 0} f^{-n}(W^s_\varepsilon(f^n(x))) = W^s(x)  $$
and $$ \bigcup_{n \geq 0} f^{n}(W^u_\varepsilon(f^{-n}(x))) = W^u(x),$$
and so the global stable and unstable manifolds are injectively immersed manifolds in $M$.

\end{theorem}

The family $(W^s_\varepsilon(x),W^u_\varepsilon(x))_{x \in \Omega}$ forms two transverse continuous laminations, which allows us to define the so-called bracket and holonomies maps. 

\begin{definition}[\cite{Bo75}, \cite{KH95}, \cite{BS02}]
For $\varepsilon>0$ small enough, there exists $\delta>0$ such that $W^s_\varepsilon(x) \cap W^u_\varepsilon(y)$ consists of a single point $[x,y]$ whenever $x,y \in \Omega$ and $d(x,y) < \delta$. In this case, $[x,y] \in \Omega$, and the map
$$ [\cdot, \cdot] : \{ (x,y) \in \Omega \times \Omega \ , \ d(x,y) < \delta \} \longrightarrow \Omega $$
is continuous. Moreover, there exists $C>0$ such that
$$ d^s([x,y],x) \leq C d(x,y), \quad \text{and} \quad d^u([x,y],y) \leq C d(x,y).$$

\end{definition}

\begin{definition} Fix $x,y$  two close enough points in $\Omega$ lying in the same local stable manifold. Let $U^u \subset W^u_\varepsilon(x) \cap \Omega$ be a small open neighborhood of $x$ relatively to $ W^u_\varepsilon(x) \cap \Omega$. The map
$$ \pi_{x,y} : U^u \subset W^u_{\varepsilon}(x) \cap \Omega \longrightarrow W^u_{\varepsilon}(y) \cap \Omega $$
defined by $\pi_{x,y}(z) := [z,y]$ is called a stable holonomy map. One can define an unstable holonomy map similarly on pieces of local stable manifolds intersected with $\Omega$. Since the stable and unstable laminations are Hölder regular in the general case, those honolomy maps are only Hölder regular in general (\cite{KH95}, Theorem 19.1.6).
\end{definition}

To work, we will need a bit of regularity on the unstable holonomies. Fortunately for us, in the particular case where the stable lamination are codimension 1, we have some regularity results: this is Theorem 1 page 25 of \cite{Ha89}, and Theorem  19.1.11 in \cite{KH95} for Anosov diffeomorphisms. For general hyperbolic sets, the proof is done is \cite{PR02}. 

\begin{theorem}

Let $f$ be an Axiom A diffeomorphism, and let $\Omega$ be a basic set. Suppose that $f$ has codimension one stable laminations. Then the stable lamination is $C^{1+\alpha}$ for some $\alpha>0$. In particular, the stable holonomies maps are $C^{1 + \alpha}$ diffeomorphisms.
\end{theorem}

\begin{remark}
Here, the holonomies being $C^{1+\alpha}$ means that the map $\pi_{x,y}$ extends to small curves $W^u_\delta(x) \longrightarrow W^u_\varepsilon(y)$, and that the extended map is a $C^{1+\alpha}$ diffeomorphism.
\end{remark}

From now on, we will work under the assumption that $f$ has codimension one stable laminations on $\Omega$. This is always true if $\dim M =2$, or if $\dim M=3$ by exchanging $f$ by $f^{-1}$ if necessary.  We fix a Hölder potential $\psi:\Omega \rightarrow \mathbb{R}$ and its associated equilibrium state $\mu$.

\subsection{Markov partitions}

In this subsection, we will construct the topological space on which we will work in this Chapter: the hypothesis made on the dimension of the unstable lamination will allow us to approximate the dynamics by a dynamical system on a (subset of a) finite disjoint union of smooth curves. For this we need to recall some results about Markov partitions.

\begin{definition}
A set $R \subset \Omega$ is called a \emph{rectangle} if 
$$ \forall x,y \in R, \ [x,y] \in R .$$
A rectangle is called proper if $\overline{\text{int}_\Omega(R)} = R$. If $x \in R$, and if $\text{diam}(R)$ is small enough with respect to $\varepsilon$, we define $$ W^s(x,R) := W^s_\varepsilon(x) \cap R  \quad \text{and} \quad W^u(x,R) := W^u_\varepsilon(x) \cap R .$$

\end{definition}

\begin{remark}
Notice that a rectangle isn't always connected, and might even have an infinite number of connected component, even if $\Omega$ is itself connected. This technicality is noticed in \cite{Pe19}, at the first paragraph of subsection 3.3, and is an obstruction to the existence of finite Markov partition with connected elements in general.
\end{remark}

\begin{definition}
A Markov partition of $\Omega$ is a finite covering $\{R_a\}_{a \in \mathcal{A}}$ of $\Omega$ by proper non-empty rectangles such that
\begin{itemize}
    \item $\text{int}_\Omega R_a \cap \text{int}_\Omega R_b = \emptyset$ if $a \neq b$
    \item $f\left( W^u(x,R_a) \right) \supset W^u(f(x),R_b)$ and $f\left( W^s(x,R_a) \right) \subset W^s(f(x),R_b)$ \\ when $x \in \text{int}_\Omega R_a \cap f^{-1}\left(\text{int}_\Omega R_b \right)$. 
\end{itemize}
\end{definition}

\begin{theorem}[\cite{Bo75},\cite{KH95} (chapter 18.7)]
Let $\Omega$ be a basic set for an Axiom A diffeomorphism $f$. Then $\Omega$ has Markov partitions of arbitrary small diameter.
\end{theorem}

From now on, we fix once and for all a Markov partition $\{R_a\}_{a \in \mathcal{A}}$ of $\Omega$ with small enough diameter. Remember that, since $\Omega$ is not an isolated cycle, it is a perfect set. In particular, $\text{diam} R_a >0$ for all $a \in \mathcal{A}$. 

\begin{definition}
We fix for the rest of this paper some periodic points $x_a \in \text{int}_\Omega R_a$ for all $a \in \mathcal{A}$ (this is possible by density of such points in $\Omega$). By periodicity, $x_a \notin W^u(x_b)$ when $a \neq b$. 
\end{definition}

\begin{definition}

We set, for all $a \in \mathcal{A}$, $$ S_a := W^s(x_a,R_a) \quad \text{and} \quad U_a := W^u(x_a,R_a) .$$
They are closed sets included in $R_a$, and are defined so that $[U_a,S_a] = R_a$. They will allow us to decompose the dynamics into a stable and an unstable part. Notice that the decomposition is unique:
$$ \forall x \in R_a, \exists ! (y,z) \in U_a \times S_a, \ x=[y,z] .$$

\end{definition}

The intuition of the construction to come is that, after a large enough number of iterates, $f^n$ can be approximated by a map that is only defined on $(U_a)_{a \in \mathcal{A}}$. Before heading into it, notice that the fractal nature of the sets $(U_a)_{a \in \mathcal{A}}$ might be a problem to do smooth analysis. It will be convenient for us to introduce some smooth curves that contains them and that are adapted to the dynamics.

\begin{definition}
For all $a \in \mathcal{A}$, $U_a \subset W^u_\varepsilon(x_a)$. Since $W^u_\varepsilon(x_a)$ is a smooth curve, it makes sense to consider the convex shell of $U_a$ seen as a subset of $W^u_\varepsilon(x_a)$. We denote it
$$ V_a := \text{Conv}_u(U_a) .$$
Each $V_a$ is then diffeomorphic to a compact interval, and contains $U_a$. Notice that it might happen that $V_a \cap \text{int}_\Omega R_b \neq \emptyset $ for some $a \neq b$. Notice also that $V_a \nsubseteq \Omega$: in particular, if $\Omega$ isn't an attractor, iterating forward the dynamics starting from a point $x \in V_a$ might sends it far away from our basic set.  
\end{definition}

\begin{remark}
By construction, the map $x \in V_a \mapsto |(df)_x(\vec{n})| \in \mathbb{R}$, where $\vec{n}$ is a unitary vector tangent to $V_a$ at $x$, is a Hölder map that coincides with $|\partial_u f|$ when $x \in U_a$. Since in this case $|\partial_u f(x)| \geq \kappa^{-1} >1$, and since $f$ is $C^{2+\alpha}$, we see that choosing our Markov partition with small enough diameters ensure that $f$ is still expanding along $V_a$.
\end{remark}

\subsection{A factor dynamics}

In this section we construct what will take the role of the shift map in our context. The construction is inspired by the symbolic case and already appear in the work of Dolgopyat \cite{Do98}. 

\begin{notations}
Let $a$ and $b$ be two letters in $\mathcal{A}$.
We note $a \rightarrow b$ if $f(\text{int}_\Omega R_a) \cap \text{int}_\Omega R_b \neq \emptyset $.

\end{notations}

\begin{definition}

We define $$\mathcal{R} := \bigsqcup_{a \in A} R_a , \quad \mathcal{S} := \bigsqcup_{a \in A} S_a , \quad \mathcal{U} := \bigsqcup_{a \in A} U_a , \quad \mathcal{V} := \bigsqcup_{a \in \mathcal{A}} V_a ,$$
where $\bigsqcup$ denote a formal disjoint union. We also define $$ \mathcal{R}^{(0)} := \bigsqcup_{a \in \mathcal{A}} \text{int}_\Omega R_a \subset \mathcal{R} $$
and $$ \mathcal{R}^{(1)} := \bigsqcup_{a \rightarrow b} (\text{int}_{\Omega} R_a) \cap f^{-1}(\text{int}_{\Omega} R_b) \subset \mathcal{R}^{(0)}, $$
so that the map $f:\Omega \rightarrow \Omega$ may be naturally seen as a map $f : \mathcal{R}^{(1)} \longrightarrow \mathcal{R}^{(0)}$. We then define 
$$ \mathcal{R}^{(k)} := f^{-k}( \mathcal{R}^{(0)}) $$
and, finally, we denote the associated residual set by
$$ \widehat{\mathcal{R}} := \bigcap_{k \geq 0} \mathcal{R}^{(k)} $$
so that $f : \widehat{\mathcal{R}} \longrightarrow \widehat{\mathcal{R}}$. 
Seen as a subset of $\Omega$, $\widehat{\mathcal{R}}$ has full measure, by ergodicity of the equilibrium measure $\mu$. Hence $\mu$ can naturally be thought of as a probability measure on $\widehat{\mathcal{R}}$. 
\end{definition}

\begin{definition}
Let $\mathcal{R}/\mathcal{S}$ be the topological space defined by the equivalence relation $ x \sim y  \Leftrightarrow \exists a \in \mathcal{A}, \ y \in W^s(x,R_a) $ in $\mathcal{R}$. Let $\pi:\mathcal{R} \rightarrow \mathcal{R}/\mathcal{S}$ denote the natural projection. The map $f:\widehat{\mathcal{R}} \rightarrow \widehat{\mathcal{R}}$ induces a factor map $F : \widehat{\mathcal{R}}/\mathcal{S} \rightarrow \widehat{\mathcal{R}}/\mathcal{S}$. Moreover, the measure $\nu := \pi_* \mu$ is an $F$-invariant probability measure on $\mathcal{R}/\mathcal{S}$.
\end{definition}

Let us check that the claim made in definition 4.2.26 holds.
We just have to verify that $f : \mathcal{R}^{(1)} \longrightarrow \mathcal{R}^{(0)}$ satisfy $$ f W^s(x,R_a) \subset W^s(f(x),R_b)   $$ for $x \in (\text{int}_{\Omega} R_a) \cap f^{-1}(\text{int}_{\Omega} R_b)$. This is true by definition of Markov partitions. The induced map satisfies $F \circ \pi = \pi \circ f$, and so $F_* \nu = \nu$. 

\begin{remark}

There is a natural isomorphism $\mathcal{U} \simeq \mathcal{R}/\mathcal{S}$ that is induced by the inclusion $\mathcal{U} \hookrightarrow \mathcal{R}$. This allows us to identify all the precedent construction to a dynamical system on $\mathcal{U}$. Namely:

\begin{itemize}
    \item The projection $\pi: \mathcal{R} \rightarrow \mathcal{R}/\mathcal{S}$ is identified with
    $$ \begin{array}[t]{lrcl}
 \pi: & \mathcal{R} \quad & \longrightarrow & \quad \mathcal{U} \\
    & x \in R_a & \longmapsto &  [x,x_a] \in U_a  \end{array}  $$
    
    \item The factor map $F:\widehat{\mathcal{R}}/\mathcal{S} \rightarrow \widehat{\mathcal{R}}/\mathcal{S}$ is identified with
    $$ \begin{array}[t]{lrcl}
 F: & \widehat{\mathcal{U}} \quad \quad \quad & \longrightarrow & \quad \widehat{\mathcal{U}} \\
    & x \in R_a \cap f^{-1}(R_b)  & \longmapsto &  [f(x),x_b] \in U_b  \end{array}   $$
    where $\widehat{\mathcal{U}}$ is defined similarly to $\widehat{\mathcal{R}}$, but with $F$ replacing $f$ in the construction.
    
    \item The measure $\nu$ is identified to the unique measure on $\mathcal{U}$ such that:
    
    $$ \forall h \in C^0(\mathcal{R},\mathbb{R}) \ \text{S-constant}, \ \int_\mathcal{U} h d\nu \ = \ \int_\mathcal{R} h d\mu $$
    
    where $S$-constant means:  $\forall a \in \mathcal{A}, \forall x,y \in R_a, \ x \in W^s(y,R_a) \Rightarrow h(x)=h(y)$.
    
\end{itemize}
\end{remark}

\begin{remark}
Since our centers $x_a$ are periodic, $x_a \in \widehat{\mathcal{R}}$, and hence $x_a \in \widehat{\mathcal{U}}$.
\end{remark}

\begin{lemma}
Since $\Omega$ is perfect, the set $\widehat{\mathcal{U}}$ is also a perfect set. In particular, $\text{diam} \ U_a > 0$ for all $a \in \mathcal{A}$.

\end{lemma}

\begin{proof}

First of all, we prove that $\widehat{\mathcal{U}}$ is infinite. Recall that if $\Omega$ is either a periodic orbit or it is perfect. By density of periodic orbits in $\Omega$, it follows if $\Omega$ is perfect that there exists an infinite family of distinct periodic orbits for $f$ in $\Omega$. In particular, the set of $f$-periodic points in $\mathcal{R}^{(0)}$ is infinite. \\

Let $x \neq y \in R_a$ be two $f$-periodic points with distinct periodic orbits. Then $\pi(x)$ and $\pi(y)$ are $F$-periodic with distinct periodic orbits. (In particular, they are in $\widehat{\mathcal{U}}$.) Indeed, if there existed some $n_0$ and $m_0$ such that $\pi(f^{n_0}(x)) = \pi(f^{m_0}y) $, then by definition of $\pi$ we would have $ f^{n_0}(x) \in W^s(f^{m_0} y) $. Hence $d(f^{n+n_0}(x), f^{n+m_0} (y)) \underset{n \rightarrow + \infty}{\longrightarrow} 0$, a contradiction.

So $\widehat{\mathcal{U}}$ is an infinite set. Now we justify that $F_{|\widehat{\mathcal{U}}}$ is topologically transitive. 

We see by the Baire category Theorem that $\widehat{\mathcal{R}}$ is a dense subset of $\Omega$. Moreover, $ f(\widehat{\mathcal{R}}) = \widehat{\mathcal{R}} $. It then follows from the transitivity of $f_{|\Omega}$ that $f_{|\widehat{\mathcal{R}}}$ is topologically transitive. Hence, $ F_{| \widehat{\mathcal{U}}} $ is topologically transitive, as a factor of $f_{|\widehat{\mathcal{R}}}$. We conclude the proof using the argument of Remark 4.2.6. \end{proof}

\begin{remark}
The fact that $\widehat{\mathcal{U}}$ is perfect, combined with the fact that holonomies extend to $C^{1+\alpha}$ maps, allows us to consider for $x \in \mathcal{R}$ the (absolute value of the) derivatives along the unstable direction of the holonomies in a meaningful way (without having to chose an extension). We denote them by $ \partial_u \pi(x) $.  \\

It is then known that those quantities are uniformly bounded (\cite{PR02}, Theorem 2.2). Even better, they can be chosen as close to $1$ as desired by taking the Markov partition small enough. In other words, for any $c < 1 < C$, one can take the Markov partition $\mathcal{R}$ so small that:
$$ \forall x \in \mathcal{R}, \ c < |\partial_u \pi(x)| \leq C .$$
This distortion bound holds even for the extensions of the holonomies on the local pieces of unstable manifolds.
\end{remark}

\begin{remark}
If the Markov partition is small enough, the map $F : \widehat{\mathcal{U}} \rightarrow \widehat{\mathcal{U}}$ is an expanding map. Taking $\kappa$ larger if necessary, one can write:
$$ \exists c>0, \forall n, \exists \delta>0, \forall x,y \in \widehat{U}, \ d(x,y) \leq \delta \Rightarrow d^u(F^n x, F^n y) \geq c \kappa^{-n} d^{u}(x,y). $$
\end{remark}

We will finally extend the dynamics on pieces of $\mathcal{V}$. 

\begin{definition}
Define, for $a \in \mathcal{A}$, $U_a^{(0)} := \text{int}_\Omega U_a$, $ V^{(0)}_a := \text{Conv}_u\left(  U_a^{(0)} \right) $, and $$ \mathcal{V}^{(0)} := \bigsqcup_{a \in \mathcal{A}} V_a^{(0)} \subset \mathcal{V}.$$ Define also, for $a,b \in \mathcal{A}$ such that $a \rightarrow b$, $U_{ab} := U_a \cap F^{-1}(U_b) $ and $U_{ab}^{(1)} := U_a^{(0)} \cap F^{-1}(U_b^{(0)})$. Finally, set $V_{ab} := \text{Conv}_u U_{ab} $, $ V_{ab}^{(1)} := \text{Conv}_u U_{ab}^{(1)}  $ and $$ \mathcal{V}^{(1)} := \bigsqcup_{a \rightarrow b} V_{ab}^{(1)} \subset \mathcal{V}^{(0)} .$$ We extends the dynamics on $\mathcal{V}$ by the formula:

$$ \begin{array}[t]{lrcl}
 F: &  \mathcal{V}^{(1)}  & \longrightarrow & \quad \mathcal{V}^{(0)} \\
    & x \in V_{ab}  & \longmapsto &  \tilde{\pi}_{f(x_a),x_b} (f(x)) \in V_b  \end{array}  $$
where $\tilde{\pi}_{f(x_a),x_b} : f(V_a) \rightarrow V_b$ is a $C^{1+\alpha}$ diffeomorphism that extends the stable holonomy between $f(U_a)$ and $U_b$. The previous remarks ensures that, as long as the Markov partition have small enough diameter, this is a well defined (and $C^{1+\alpha}$) expanding map.
\end{definition}

\subsection{A transfer operator}

The dynamical system $(F,\widehat{\mathcal{U}},\nu)$ is constructed to behave like a one-sided shift, and so it is natural to search for a transfer operator defined on $\mathcal{U}$ that leaves $\nu$ invariant. Recall that $\mu$ is the equilibrium measure for some Hölder potential $\psi$ on $M$: the fact that $\psi$ isn't necessarily $S$-constant suggest that we will have to search for another potential, cohomologous to $\psi$, that is $S$-invariant. Looking at the symbolic case, we notice from the proof of Proposition 1.2 in \cite{PP90} that the potential $\varphi_0 \in C^\alpha( \mathcal{R},\mathbb{R} )$ (take $\alpha$ smaller if necessary) defined by the formula $$  \varphi_0(x) := \psi(\pi(x)) + \sum_{n=0}^\infty \left( \psi( f^{n+1} \pi x ) - \psi(f^n \pi f x) \right) \quad \quad \quad  $$ might be interesting to consider. First of all, recall what is a transfer operator.

\begin{definition}
For some Hölder potential $\varphi:\mathcal{U} \rightarrow \mathbb{R}$, define the associated \emph{transfer operator} $ \mathcal{L}_\varphi : C^0( \mathcal{U}^{(0)} ,\mathbb{C}) \longrightarrow  C^0( \mathcal{U}^{(1)} ,\mathbb{C}) $
by 
$$ \forall x \in \mathcal{U}^{(0)}, \  \mathcal{L}_\varphi h (x) := \sum_{y \in F^{-1}(x)} e^{\varphi(y)} h(y) .$$
Iterating $\mathcal{L}_\varphi$ gives
$$ \forall x \in \mathcal{U}^{(0)}, \ \mathcal{L}_\varphi^n h (x) = \sum_{y \in F^{-n}(x)} e^{S_n \varphi(y)} h(y) ,$$
where $S_n \varphi(z) := \sum_{k=0}^{n-1} \varphi(F^k(z))$ is a Birkhoff sum. \\

By duality, $\mathcal{L}_\varphi$ also acts on the set of measures on $\mathcal{U}$. If $m$ is a measure on $\mathcal{U}^{(1)}$, then $\mathcal{L}_\varphi^* m$ is the unique measure on $\mathcal{U}^{(0)}$ such that

$$ \forall h \in C^0_c( \mathcal{U}^{(1)}, \mathbb{C}), \ \int_{\mathcal{U}} h \ d \mathcal{L}_\varphi^* m = \int_{\mathcal{U}} \mathcal{L}_\varphi h \ dm  .$$

\end{definition}

\begin{remark}
We may rewrite the definition by highlighting the role of the inverse branches of $F$.
For some $a \rightarrow b$, we see that $F : V_{ab}^{(1)} \rightarrow V_{b}^{(0)}$ is a $C^{1+\alpha}$ diffeomorphism. We denote by $g_{ab} : V_b \rightarrow V_ {ab} \subset V_a$ its local inverse. Notice that, restricted on $\mathcal{U}$, we get an inverse branch $g_{ab} : \widehat{U}_{b} \rightarrow \widehat{U}_{ab} \subset U_a$ of $F : \widehat{\mathcal{U}} \rightarrow \widehat{\mathcal{U}}$ that satisfies the formula
$$ \forall x \in U_b, \ g_{ab}(x) := f^{-1}([ x , f(x_a) ]) .$$
With these notations, the transfer operator can be rewritten as follow:

$$ \forall x \in U_b, \ \mathcal{L}_\varphi h(x) = \sum_{a \rightarrow b} e^{\varphi(g_{ab}(x))} h(g_{ab}(x)) .$$
\end{remark}

\begin{theorem}

There exists a Hölder function $h: \mathcal{U} \rightarrow \mathbb{R}$ such that $\mathcal{L}_{\varphi_0}^* (h \nu) = e^{P(\psi)} h \nu $, where $$\varphi_0(x) := \psi(\pi(x)) + \sum_{n=0}^\infty \left( \psi( f^{n+1} \pi x ) - \psi(f^n \pi f x) \right).$$

\end{theorem}

\begin{proof}

This fact follows from the symbolic setting, but a quick and clean proof can also be achieved using the following geometrical results, extracted from \cite{Le00} and \cite{Cl20}. See in particular Theorem 3.10 in \cite{Cl20}. Define, when $x \in W^s(y)$, $$\omega^+(x,y) := \sum_{n=0}^\infty \left( \psi(f^n(x)) - \psi(f^n(y)) \right),$$
and when $x \in W^u(y)$, $$ \omega^-(x,y) := \sum_{n=0}^\infty \left( \psi(f^{-n}(x)) - \psi(f^{-n}(y)) \right). $$
There exists two families of (nonzero and) finite measures $\mu_x^s$, $\mu_x^u$ indexed on $x \in \mathcal{R}^{(0)}$ such that $\text{supp} ( \mu_x^u) \subset W^u(x,R_a)$ and $\text{supp} ( \mu_x^s ) \subset W^s(x,R_a)$ if $x \in R_a$, and that satisfies the following properties:

\begin{itemize}
    \item If $\pi$ is a stable holonomy map between $W^u(y,R_a)$ and $W^u(\pi(y),R_a)$, then
    $$ \frac{d\left(\pi_*\mu_y^u\right)}{d \mu_{\pi(y)}^u}(\pi(x)) = e^{\omega^+(x,\pi x)}.$$
    and a similar formula holds for $\mu_y^s$ with an unstable holonomy map, replacing $\omega^+$ by $\omega^-$ (\cite{Cl20}, Theorem 3.7 and Theorem 3.9).
    \item The family of measures is $\psi$-conformal:
    $$ \frac{ d \left( f_* \mu_y^u \right)}{d \mu_{f(y)}^u}(f(x)) = e^{\psi(x) - P(\psi) } $$
    and a similar formula holds for $\mu^s_y$ (\cite{Cl20}, Theorem 3.4 and Remark 3.5).
\end{itemize}

Finally, this family of measures is related to $\mu$ in the following way (\cite{Cl20}, Theorem 3.10). There exists positive constants $(c_a)_{a \in \mathcal{A}}$ such that, for any measurable map $g:M \rightarrow \mathbb{C}$, the following formula holds:

$$ \int_M g d\mu = \sum_{a \in \mathcal{A}} \int_{U_a} \int_{S_a} e^{\omega^+([x,y],x)+\omega^-([x,y],y)} g([x,y]) c_a d\mu_{x_a}^s(y) d\mu_{x_a}^u(x)  .$$
From this we can link $\nu$ and the family $\mu_{x_a}^u$ in the following way. Define $h_0 : \mathcal{U} \rightarrow \mathbb{R}$ by the formula
$$ \forall a \in \mathcal{A}, \ \forall x \in U_a, \ h_0(x) := c_a \int_{S_a} e^{\omega^-([x,y],y) + \omega^+([x,y],y)} d\mu_{x_a}^s(y) > 0.$$
Then, for any $S$-constant map $g:\mathcal{R} \rightarrow \mathbb{R}$, we have
$$ \int_\mathcal{U} g d\nu = \int_M g d\mu = \sum_{a \in \mathcal{A}} \int_{U_a} g(x) h_0(x) d\mu_{x_a}^u(x) .$$

It follows from the properties of the bracket and the fact that $\psi$ is supposed Hölder that $h_0$ is Hölder. So $h:=h_0^{-1}$ is Hölder and satisfy $h \nu = \mu^u$, where $\mu^u_{|U_a} := \mu^u_{x_a}$. We check that $\mu^u$ is an eigenmeasure for $\mathcal{L}_{\varphi_0}$. 
Fix $a \rightarrow b \in \mathcal{A}$, and fix a continuous map $G: U_b \rightarrow \mathbb{C}$.
We have
$$ \int_{U_b} G(g_{ab}(x)) d \mu_{x_b}^u(x) = \int_{U_b} G \left( f^{-1} [x,f(x_a)] \right)  d \mu_{x_b}^u(x) $$
$$ = \int_{f(U_{ab})} G \left( f^{-1} x \right) e^{\omega^+(\pi x,x)} d \mu_{f(x_a)}^u(x)  $$
$$ = \int_{U_{ab}} G(x) e^{\omega^+(\pi f x, f x)} e^{-\psi(x)+P(\psi)} d \mu^u_{x_a}(x) $$
$$ = e^{P(\psi)} \int_{U_{ab}} G(x) e^{-\varphi_0(x)} d\mu_{x_a}^u .$$

In particular, for any continuous $G: \mathcal{U} \rightarrow \mathbb{R}$:
$$ \int_{\mathcal{U}} \mathcal{L}_{\varphi_0}(G) d\mu^u = \sum_{b \in \mathcal{A}} \int_{U_b} \sum_{a \rightarrow b} e^{\varphi_0( g_{ab}(x) ) } G(g_{ab}(x)) d \mu^u_{x_b}  $$
$$ = e^{P(\psi)} \underset{a \rightarrow b}{\sum_{a, b \in \mathcal{A}}} \int_{U_{ab}} G(x) d\mu^u_{x_a} = e^{P(\psi)} \int_{\mathcal{U}} G d\mu^u, $$
which concludes the proof. \end{proof}

\begin{lemma}
Define $\varphi := \varphi_0 - \ln h \circ F + \ln h - P(\psi)$. Then $\varphi:\mathcal{U} \rightarrow \mathbb{R}$ is Hölder and normalized, that is: $ \mathcal{L}_\varphi^* \nu = \nu $, $ \mathcal{L}_\varphi 1 = 1 $ on $\widehat{\mathcal{U}}$. Moreover, $\varphi \leq 0$ on $\widehat{\mathcal{U}}$, and there exists an integer $N \geq 0$ such that $S_N \varphi<0$ on $\widehat{\mathcal{U}}$.
\end{lemma}

\begin{proof}

It follows from the definition of $\varphi$ that $  \mathcal{L}_\varphi^* \nu = \nu $. Then, we see that $ \mathcal{L}_\varphi 1 = 1 $ on $L^2(\nu) $. Indeed, for any $L^2(\nu)$ map $G$, we have $$\int_{\mathcal{U}} \mathcal{L}_\varphi(1) G d\nu = \int_{\mathcal{U}}\mathcal{L}_\varphi( G \circ F  ) d\nu = \int_{\mathcal{U}} G \circ F d\nu = \int_{\mathcal{U}} G d\nu .$$

Since $\mathcal{L}_\varphi 1$ and $1$ are continuous functions on $\widehat{\mathcal{U}}$, and since $\nu$ has full support, it follows that $\mathcal{L}_\varphi 1 = 1$ on all $\widehat{\mathcal{U}}$, and even on $\mathcal{U}$ by density. We now check that $\varphi$ is eventually negative on the residual set. First of all, the equality $ 1 = \mathcal{L}_\varphi(1)$ implies that $\varphi \leq 0$ on $\widehat{\mathcal{U}}$. Then, since $F$ is uniformly expanding, there exists some integer $N$ such that $F^{-N}(x)$ always contains at least two distinct points, for any $x$. The equality
$ 1 = \mathcal{L}_\varphi^{N}(1)$ then allows us to conclude that $S_N \varphi<0$ on $F^{-N}(\widehat{\mathcal{U}})=\widehat{\mathcal{U}}$. \end{proof}

\subsection{Some regularity for $\nu$}

In this subsection, we introduce the usual symbolic formalism, recall the Gibbs estimates for $\nu$, and prove an upper regularity result for $\nu$. In particular, we prove that $\nu$ doesn't have atoms in our setting. Notice that if $\Omega$ were an isolated cycle, then the measure of maximal entropy would be the Dirac on the cycle, which is fully discrete, so this is not completely trivial.

\begin{notations}

For $n \geq 1$, a word $\mathbf{a} = a_1 \dots a_n \in \mathcal{A}^n$ is said to be admissible if $a_1 \rightarrow a_2 \rightarrow \dots \rightarrow a_n$. We define:
\begin{itemize}
    \item $\mathcal{W}_n := \{ \mathbf{a} \in \mathcal{A}^n \ | \ \mathbf{a} \ \text{is admissible} \ \}$.
    \item For $\mathbf{a} \in \mathcal{W}_n$, define $g_{\mathbf{a}} := g_{a_1 a_2}  g_{a_2 a_3}  \dots g_{a_{n-1} a_{n}} : V_{a_n} \longrightarrow V_{a_{1}}$. 
    \item For $\mathbf{a} \in \mathcal{W}_n$, define $V_{\mathbf{a}} := g_{\mathbf{a}}\left( V_{a_n} \right)$, $ U_{\mathbf{a}} := g_{\mathbf{a}}\left( U_{a_n} \right) $, and $ \widehat{U}_{\mathbf{a}} := g_{\mathbf{a}}\left( \widehat{\mathcal{U}} \cap U_{a_n} \right) $ 
    \item For $\mathbf{a} \in \mathcal{W}_n$, set $x_{\mathbf{a}} := g_{\mathbf{a}}\left( x_{a_n} \right) \in \widehat{U}_{\mathbf{a}} $. 
\end{itemize}

\end{notations}

\begin{remark}

Since $F$ is expanding, the maps $g_{\mathbf{a}}$ are contracting as $n$ becomes large. As $\mathcal{U}$ is included in a finite union of unstable curve $\mathcal{V}$, that are one dimensional Riemannian manifolds, it makes sense to consider absolute value of the derivatives of $F$ and $g_{\mathbf{a}}$, and we will do it from now on. For points in $\mathcal{U}$, this correspond to the absolute value of the derivative in the unstable direction. We find:
$$ \forall \mathbf{a} \in \mathcal{W}_n, \ \forall x \in V_{a_n}, \ |g_{\mathbf{a}}'(x)| \leq c^{-1} \kappa^n  , $$
for some constant $\kappa \in (0,1)$, and it follows that
$$\forall \mathbf{a} \in \mathcal{W}_n,  \ \text{diam}(V_{\textbf{a}}) = \text{diam} \ g_{\mathbf{a}}\left( V_{a_n} \right) \leq c^{-1} \kappa^n .$$
A consequence for our potential is that it has $\textbf{exponentially vanishing variations}$. Namely, since $\varphi$ is Hölder, the following holds:

$$ \exists C>0, \ \forall n \geq 1, \ \forall \mathbf{a} \in \mathcal{W}_n, \ \forall x,y \in U_{\mathbf{a}}, \ |\varphi(x) - \varphi(y)| \leq C \kappa^{\alpha n} .$$

\end{remark}

\begin{remark}

For a fixed $n$, the family $(U_{\mathbf{a}})_{\mathbf{a} \in W_n}$ is a partition of $\mathcal{U}$ (modulo a boundary set of zero measure). In particular, for any continuous map $g:\mathcal{U} \rightarrow \mathbb{C}$, we can write
$$ \int_{\mathcal{U}} g d \nu = \sum_{\mathbf{a} \in \mathcal{W}_n} \int_{U_{\mathbf{a}}} g d\nu.$$

\end{remark}

\begin{lemma}[Gibbs estimates, \cite{PP90}]

$$ \exists C_0>1, \ \forall n \geq 1, \ \forall \mathbf{a} \in \mathcal{W}_n, \  \forall x \in U_{\mathbf{a}}, \ \ C_0^{-1} e^{S_n \varphi(x)} \leq \nu( U_{\mathbf{a}} ) \leq C_0 e^{S_n \varphi(x)} $$

\end{lemma}

 \begin{proof}

We have
$$ \int_{\mathcal{U}} e^{- \varphi} \mathbb{1}_{U_{a_1 \dots a_n}} d\nu =  \int_{\mathcal{U}} \mathcal{L}_\varphi\left( e^{- \varphi} \mathbb{1}_{U_{a_1 \dots a_n}} \right) d\nu  = \int_{\mathcal{U}} \mathbb{1}_{U_{a_2 \dots a_n}} d\nu =  \nu(U_{a_2 \dots a_n}). $$
Moreover, since $\varphi$ has exponentially decreasing variations, we can write that:
$$\forall x \in U_{a_1 \dots a_n}, \ e^{-\varphi(x) - C \kappa^{\alpha n}} \nu(U_{a_1 \dots a_n}) \leq \int_{\mathcal{U}} e^{- \varphi} \mathbb{1}_{U_{a_1 \dots a_n}} d\nu \leq  e^{-\varphi(x) + C \kappa^{\alpha n}} \nu(U_{a_1 \dots a_n}) , $$
and so $$ \forall x \in U_{a_1 \dots a_n}, \ e^{-\varphi(x) - C \kappa^{\alpha n}}  \leq \frac{\nu(U_{a_2 \dots a_n})}{\nu(U_{a_1 \dots a_n}) } \leq  e^{-\varphi(x) + C \kappa^{\alpha n}} . $$
Multiplying those inequalities gives us the desired relation.
 \end{proof}

\begin{lemma}

The measure $\nu$ is upper regular, that is, there exists $C>0$ and $\delta_{\text{up}}>0$ such that

$$ \forall x \in \mathcal{U}, \ \forall r > 0, \ \nu( B(x,r) ) \leq C r^{\delta_{\text{up}}},$$

where $B(x,r)$ is a ball of center $x$ and radius $r$ in $\mathcal{U}$.

\end{lemma}

\begin{proof}
Let $x \in \mathcal{U}$, and let $r>0$ be small enough. If $B(x,r) \cap \widehat{\mathcal{U}}= \emptyset$ then $\nu(B(x,r)) = 0$ since $\widehat{\mathcal{U}}$ has full measure, and the proof is done. If $B(x,r) \cap \widehat{\mathcal{U}} \neq \emptyset$, then let $\tilde{x} \in B(x,r) \cap \widehat{\mathcal{U}}  $. We then see that $ \nu(B(x,r)) \leq \nu(B(\tilde{x},2r))$. \\

Recall that $F$ has bounded derivative on $\mathcal{U}$: $|F'|_\infty < \infty $. In particular, there exists a constant $\kappa_1 \in (0,1)$ such that $ |(F^n)'|_\infty \leq \kappa_1^{-n} $. Hence: $ \forall n \geq 1, \ \forall \mathbf{a} \in \mathcal{W}_n, \ \text{inf}|g_{\mathbf{a}}'| > \kappa_1^n $. In particular, $ \text{diam}(U_{\mathbf{a}}) > \kappa_1^n \eta_0$, where $\eta_0 := \min_{a \in \mathcal{A}} \text{diam} (U_a) > 0$ (This is where we use the fact that $\widehat{\mathcal{U}}$ is perfect). \\

Recall also from Lemma 4.2.36 that there exists an integer $N$ such that $S_N \varphi < 0$ on $\widehat{\mathcal{U}}$.  We then let $n(r)$ be the unique integer $n$ such that

$$ \eta_0 \kappa_1^{N(n+1)} < 2r \leq  \eta_0 \kappa_1^{Nn} .$$
Then, by construction, $ 2 r \leq \text{diam}(U_{\mathbf{a}}) , \forall \mathbf{a} \in \mathcal{W}_{n(r)}$. Furthermore, since the elements of the Markov partition are proper, we know that there exists $\delta \in (0,1)$ such that, for all $a \in \mathcal{A}$, $B(x_a,\delta) \cap \mathcal{U} \subset U_a$. Using the fact that $\mathcal{V}$ is one dimensional, it is then not difficult to see that for all $\mathbf{a} \in \mathcal{W}_{n}$, $ 
B(x_\mathbf{a},2 \delta r  ) \cap \mathcal{U} \subset B(x_\mathbf{a},\delta \text{diam}(U_\mathbf{a}) ) \cap \mathcal{U} \subset U_{\mathbf{a}}$. From the observation that any partition of an interval $[0,r]$ with subintervals of diameter $\delta$ has cardinal less than $\delta^{-1} r$, it follows that $B(\tilde{x},2r)$ is contained in at most $2\delta^{-1}$ sets of the form $U_{\mathbf{a}}$, $\mathbf{a} \in \mathcal{W}_n$. In particular, using Gibbs estimates, we find
$$ \nu\left( B(x,r) \right) \leq C_0 \underset{U_\mathbf{a} \cap B(\tilde{x},2r) \neq \emptyset}{\sum_{\mathbf{a} \in \mathcal{W}_n}} e^{S_{N n(r)} \varphi(x_\mathbf{a})}.$$
But then, notice that $ S_{n N} \varphi(x) < - n \inf_{\widehat{\mathcal{U}}}|S_N \varphi| $. Hence:
$$  \nu\left( B(x,r) \right) \leq 2 C_0 \delta^{-1} e^{- n(r) \inf_{\widehat{\mathcal{U}}}|S_N \varphi|} \leq C r^{\delta_{\text{up}}},$$
for some $C>0$ and $\delta_{\text{up}} >0$. \end{proof}
In particular, it follows that $\nu$ doesn't have atoms, which is some good news since we are willing to prove Fourier decay. Moreover, we get a regularity result on $\mu$.

\begin{corollary}
Let $f$ be an Axiom A diffeomorphism, and let $\Omega$ be a perfect basic set with codimension one stable foliation. Then any equilibrium state $\mu$ associated to a Hölder potential $\psi : \Omega \rightarrow \mathbb{R}$ is upper regular.
\end{corollary}

\begin{proof}
Let $x \in \Omega$ and $r > 0$ be small enough. Write $B(x,r) = \bigcup_{a \in \mathcal{A}} B(x,r) \cap R_a$. Notice that $B(x,r) \cap R_a \subset \{ y \in R_a \ | \ \exists z \in  B(x,r) \cap R_a, \ \pi(y) = \pi(z)\}$, and so:

$$ \mu(B(x,r)) = \sum_{a \in \mathcal{A}} \mu \left( B(x,r) \cap R_a \right) \leq \sum_{a \in \mathcal{A}} \nu\left( \pi\left( B(x,r) \cap R_a  \right) \right) \leq |\mathcal{A}| C \|\partial_u \pi\|_{\infty,\mathcal{R}}^{\delta_{up}} (2r)^{\delta_{up}} \leq \tilde{C} r^{\delta_{up}}.$$
using the previous lemma, and using the fact that $\text{diam}_u\left( \pi\left( B(x,r) \cap R_a\right) \right) \leq r \|\partial_u \pi\|_{\infty,\mathcal{R}}$.
\end{proof}

\subsection{Large deviations }

We finish this preliminary section by recalling some large deviation results. There exists a large bibliography on the subject, see for example \cite{Ki90}. Large deviations in the context of Fourier decay for some 1-dimensional shift was used in \cite{JS16} and in \cite{SS20}. A simple proof in the context of Julia sets using the pressure function can be found in section 3.7.

\begin{theorem}[\cite{LQZ03}]

Let $g : \Omega \rightarrow \mathbb{C}$ be any continuous map. Then, for all $\varepsilon>0$, there exists $n_0(\varepsilon)$ and $\delta_0(\varepsilon) > 0$ such that

$$ \forall n \geq n_0(\varepsilon), \ \mu \left( \left\{ x \in \Omega , \ \left| \frac{1}{n} \sum_{k=0}^{n-1} g(f^k(x)) - \int_{\Omega} g d\mu \right| \geq \varepsilon \right\} \right) \leq e^{-\delta_0(\varepsilon) n} $$

\end{theorem}

Notice that applying (a symbolic version of) Theorem 4.2.43 (like Theorem 1.2.13) to a $S$-constant function immediately gives the same statement for $(\nu,\widehat{\mathcal{U}},F)$. We apply it to some special cases.

\begin{definition}
Define $\tau_f := \ln |\partial_u f (x)| $  (so that $-\tau_f$ is the geometric potential) and $\tau_F := \ln |F'(x)|$. Beware that $\tau_F$ is \underline{only Hölder regular}. Then, define the associated global Lyapunov exponents by:
$$ \lambda := \int_{\Omega} \tau_f d\mu > 0 \quad , \quad \Lambda := \int_{\mathcal{U}} \tau_F d\nu > 0  $$
Define also the dimension of $\nu$ by the formula $$\delta := -\frac{1}{\Lambda} \int_{\mathcal{U}} \varphi d\nu > 0. $$
A consequence of Theorem 4.2.43 is the following:
\end{definition} 

\begin{corollary}[\cite{SS20}, \cite{Le21}]

For every $\varepsilon > 0$, there exists $n_1(\varepsilon)$ and $\delta_1(\varepsilon)>0$ such that

$$ \forall n \geq n_1(\varepsilon), \ \nu \left( \left\{ x \in \widehat{\mathcal{U}} \ , \ \left|\frac{1}{n}S_n \tau_F(x) -  \Lambda \right| \geq \varepsilon \text{ or } \left|\frac{S_n \varphi(x)}{S_n \tau_F(x)} + \delta \right| \geq \varepsilon \right\} \right) \leq e^{- \delta_1(\varepsilon) n} .$$

\end{corollary} 
These quantitative results tells us that, \say{for most $x \in \Omega$}, $|\partial_u f^n(x)|$ have order of magnitude $\exp(\lambda n)$, and $|(F^n)'(x)|$ have order of magnitude $\exp(\Lambda n)$. Studying those two quantities is therefore central for us. An important (and natural) remark is the following.

\begin{lemma}
The global unstable Lyapunov exponents $\lambda$ and $\Lambda$ are equal. In other words,
$$ \int_{\Omega} \ln |\partial_u f| d\mu = \int_{\mathcal{U}} \ln |F'| d\nu. $$

\end{lemma}

\begin{proof}
The fact that $ \pi \circ f = F \circ \pi  $ implies that $$ \tau_f = \tau_F(\pi(x)) - \ln |\partial_u \pi(f(x))| + \ln |\partial_u \pi(x)| .$$
In particular, $ \tau_f $ and $\tau_F \circ \pi$ are $f$-cohomologous.
This implies the desired equality, by definition of the measure $\nu$, and by $f$-invariance of $\mu$. \end{proof}

\section{Computing some orders of magnitudes}

The goal of this section is to use the large deviations to compute orders of magnitude for dynamical quantities. Once this preparatory step is done, we will be able to study oscillatory integrals involving $\mu$. First of all, we recall some useful symbolic notations.

\begin{itemize}
    \item Recall that $\mathcal{W}_n$ stands for the set of admissible words of length $n$. For $\mathbf{a}=a_1 \dots a_n a_{n+1} \in \mathcal{W}_{n+1}$, define $\mathbf{a}' := a_1 \dots a_{n} \in \mathcal{W}_n$.
    
    \item Let $\mathbf{a} \in \mathcal{W}_{n+1}$ and $\mathbf{b} \in \mathcal{W}_{m+1}$. Denote $\mathbf{a} \rightsquigarrow \mathbf{b} $ if $a_{n+1}=b_1$. In this case, $\mathbf{a}' \mathbf{b} \in \mathcal{W}_{n+m+1}$.
    
    \item Let $\mathbf{a} \in \mathcal{W}_{n+1}$. We denote by $b(\mathbf{a})$ the last letter of $\mathbf{a}$.
\end{itemize}
With these notations, we may rewrite the formula for the iterate of our transfer operator. For any continuous function $h : \mathcal{U} \longrightarrow \mathbb{C} $, we have

$$\forall b \in \mathcal{A}, \ \forall x \in U_b, \ \mathcal{L}_\varphi^n h(x) = \underset{\mathbf{a} \rightsquigarrow b}{\sum_{\mathbf{a} \in \mathcal{W}_{n+1}}} e^{S_n \varphi(g_{\mathbf{a}}(x))} h(g_{\mathbf{a}}(x)) = \underset{\mathbf{a} \rightsquigarrow b}{\sum_{\mathbf{a} \in \mathcal{W}_{n+1}}}  h(g_{\mathbf{a}}(x)) w_{\mathbf{a}}(x) ,$$
where $$ w_{\mathbf{a}}(x) := e^{S_n\varphi(g_{\mathbf{a}}(x))} = \nu(U_{\mathbf{a}})e^{O(1)} .$$
Iterating $\mathcal{L}_\varphi^n$ again leads us to the formula
$$ \forall k \geq 1,  \forall x \in U_b, \ \mathcal{L}_\varphi^{nk} h(x)  = \sum_{\mathbf{a}_1 \rightsquigarrow \dots \rightsquigarrow \mathbf{a}_k \rightsquigarrow b} h(g_{\mathbf{a}_1 ' \dots \mathbf{a}_{k-1} ' \mathbf{a}_k}(x) ) w_{\mathbf{a}_1 ' \dots \mathbf{a}_{k-1} ' \mathbf{a}_k}(x) .$$

\begin{definition}
Let, for small $\varepsilon>0$ and $n \geq 1$, $$ B_n(\varepsilon) := \left\{ x \in \mathcal{R} \ , \left|\frac{1}{n}S_n \tau_F(\pi(x)) -  \lambda \right| < \varepsilon \text{ and } \left|\frac{S_n \varphi(\pi(x))}{S_n \tau_F(\pi(x))} + \delta \right| < \varepsilon  \right\}  $$
$$\text{and} \ C_{n}(\varepsilon) :=  \left\{ x \in {\Omega} \ , \ \ \left| \frac{1}{n}\sum_{k=0}^{n-1} \tau_f(f^k(x)) - \lambda \right| < \varepsilon \right\} .$$
We can naturally identify $B_n(\varepsilon)$ with a subset of $\Omega$, modulo a zero-measure set. Finally, let $ A_{n}(\varepsilon) :=   B_n(\varepsilon) \cap C_n(\varepsilon) $.  Then $\mu(A_{n}(\varepsilon)) \geq 1 - e^{-\delta_2(\varepsilon) n}$ for $n$ large enough depending on $\varepsilon$ and for some $\delta_2(\varepsilon)>0$, by Theorem 4.2.43 and Corollary 4.2.45.
\end{definition}

\begin{notations}
To simplify the reading, when two quantities dependent of $n$ satisfy $b_n \leq C a_n $ for some constant $C$, we denote it by $a_n \lesssim b_n$. If $a_n \lesssim b_n \lesssim c_n$, we denote it by $a_n \simeq b_n$. If there exists $c,C$ and $\beta$, independent of $n$ and $\varepsilon$, such that $ c e^{- \varepsilon \beta n} a_n \leq b_n \leq C e^{\varepsilon \beta n} a_n$, we denote it by $a_n \sim b_n$. Throughout the text $\beta$ will be allowed to change from line to line. It correspond to some positive constant. 
\end{notations}
Eventually, we will chose $\varepsilon$ small enough such that this exponentially growing term gets absorbed by the other leading terms, so we can neglect it.

\begin{lemma}
Let $\mathbf{a} \in \mathcal{W}_{n+1}(\varepsilon)$ be such that $U_{\mathbf{a}} \cap A_{n}(\varepsilon) \neq \emptyset$.  Then:

\begin{itemize}
    \item uniformly on $x \in V_{\mathbf{a}}$,  $|\partial_u f^{n}(x)| \sim e^{n \lambda}$,
    \item uniformly on $x \in V_{b(\mathbf{a})}, \  |g_\mathbf{a}'(x)| \sim e^{- n \lambda}$,
    \item $\mathrm{diam}(\widehat{U}_\mathbf{a}), \ \mathrm{diam}(U_\mathbf{a}), \ \mathrm{diam}(V_{\mathbf{a}})  \sim e^{- n \lambda} $,
    \item uniformly on $x \in V_{\mathbf{a}}, \ w_{\mathbf{a}}(x) \sim e^{- \delta \lambda n} $,
    \item $\nu(\widehat{U}_\textbf{a})  \sim e^{- \delta \lambda n}$.
\end{itemize}

\end{lemma}

\begin{remark}

The definition of $\delta$ is chosen so that $ \mu(U_{\mathbf{a}}) \sim \text{diam}(U_\mathbf{a})^\delta $ when $U_{\mathbf{a}} \cap A_{n}(\varepsilon) \neq \emptyset$. Beware that the diameters are for the distance $d^u$.

\end{remark}

\begin{proof}

The first thing to notice is that, since any Hölder map have exponentially vanishing variations, these Birkhoff sums are easy to control when $n$ becomes large. In particular, the estimates that are true on a piece of $U_{\mathbf{a}}$ will extends on $V_{\mathbf{a}}$. For example, let $x \in V_{\mathbf{a}}$, and let $y_{\mathbf{a}} \in U_{\mathbf{a}} \cap A_{n}(\varepsilon)$. Since $\tau_f$ has exponentially vanishing variations, we can write
$$ \left| \frac{1}{n}\sum_{k=0}^{n-1} \tau_f(f^k(x)) - \lambda \right| \leq \left| \frac{1}{n}\sum_{k=0}^{n-1} \tau_f(f^k(x)) - \frac{1}{n}\sum_{k=0}^{n-1} \tau_f(f^k(y_{\mathbf{a}})) \right| + \left| \frac{1}{n}\sum_{k=0}^{n-1} \tau_f(f^k(y_{\mathbf{a}})) - \lambda \right| $$
$$ \leq  \frac{1}{n}\sum_{k=0}^{n-1} \| \tau_f \|_{C^\alpha} \text{diam}(f^k(V_{\mathbf{a}}))^\alpha + \varepsilon .$$
Now, notice that since $f$ is a diffeomorphism $\mathcal{V} \rightarrow f(\mathcal{V})$ which is a disjoint union of curves, and since $V_a = \text{Conv}_u(U_a)$, we find $V_\mathbf{a} = \text{Conv}_u(U_\mathbf{a})$, and 
 $f^k(V_{\mathbf{a}}) = \text{Conv}_u(f^k(U_{\mathbf{a}}))$. Since $f^k(U_{\mathbf{a}}) \subset R_{a_{k+1} \dots a_n}$ is diffeomorphic to $U_{a_{k+1} \dots a_n} $ by the mean of the holonomy map $\pi$ which have bounded differential, we obtain
$$ \left| \frac{1}{n}\sum_{k=0}^{n-1} \tau_f(f^k(x)) - \lambda \right|  \leq \frac{1}{n} \sum_{k=0}^{n-1} \|\tau_f\|_{C^\alpha}\text{diam}\left( f^k(U_{\mathbf{a}}) \right)^\alpha + \varepsilon  $$
$$ \lesssim \frac{1}{n} \sum_{k=0}^{n-1}\text{diam}(U_{a_{k+1} \dots a_{n}})^\alpha  + \varepsilon \lesssim  \frac{1}{n} \sum_{k=0}^{n-1} \kappa^{\alpha k} + \varepsilon \leq \frac{C}{n} + \varepsilon $$
for some constant $C$. In particular we see that for $n$ large enough depending on $\varepsilon$, the estimate
$$ \forall x \in V_{\mathbf{a}}, \ \left| \frac{1}{n} \sum_{k=0}^{n-1} \tau_f(f^k(x)) - \lambda \right| < 2 \varepsilon $$
holds, which proves $\partial_u f^n(x) \sim e^{\lambda n}$ on $V_\mathbf{a}$.
Similarly the estimates related to $B_n(\varepsilon)$ extends on $V_{\mathbf{a}}$ replacing $\varepsilon$ with $2 \varepsilon$ as long as $n$ is large enough. Once this is known, the estimates for $\partial_u(f^{n})$, $g_{\mathbf{a}}'$ and $w_{\mathbf{a}}$ are easy, and the estimates for the diameters is a direct consequence of our 1-dimensional setting and of the mean value theorem. 
The estimates for $\nu(U_\mathbf{a})$ follows from the Gibbs estimates. \end{proof}

\begin{definition}
For some fixed set $\omega \subset \mathcal{U}$, define the set of ($\omega$-localized) $\varepsilon$-regular words by
$$ \mathcal{R}_{n+1}(\omega) := \left\{ \mathbf{a} \in \mathcal{W}_{n+1} \ | \ \omega \cap U_{\mathbf{a}} \neq \emptyset, \ A_{n}(\varepsilon) \cap U_{\mathbf{a}} \neq \emptyset \right\}. $$
When no localization is asked, define $\mathcal{R}_{n+1} := \mathcal{R}_{n+1}(M)$. Then, define the set of ($\omega$-localized) $\varepsilon$-regular $k$-blocks by $$ \mathcal{R}_{n+1}^k(\omega) = \left\{ \mathbf{A}=\mathbf{a}_1' \dots \mathbf{a}_{k-1}' \mathbf{a}_k \in \mathcal{W}_{nk+1} \ | \ \forall i \geq 2, \ \mathbf{a}_i \in \mathcal{R}_{n+1}, \ \text{and} \ \mathbf{a}_1 \in \mathcal{R}_{n+1}(\omega) \right\} .$$ 
We also define $\mathcal{R}_{n+1}^k := \mathcal{R}_{n+1}^k(M)$. Finally, define the associated geometric points to be $$ R_{n+1}^k := \bigcup_{\mathbf{A} \in \mathcal{R}_{n+1}^k } U_{\mathbf{A}} .$$
\end{definition}

\begin{lemma}

For a fixed $k \geq 1$, and for $ n $ large enough depending on $\varepsilon$, $$ \# \mathcal{R}_{n+1}^k \sim e^{ k \delta \lambda n} .$$

\end{lemma}

\begin{proof}

First of all, notice that $ A_n(\varepsilon) \subset R_{n+1} $, so that $ \nu(\mathcal{U} \setminus R_{n+1}) \leq \nu( \mathcal{U} \setminus A_{n} ) \leq \nu(\mathcal{U} \setminus B_{n}) + \nu(\mathcal{U} \setminus C_{n})$. To control the term in $C_n$, we notice that, by the same argument as in Lemma 4.3.3,
$$ \forall x \in R_a, \  x \in C_n(\varepsilon) \Rightarrow [x,S_a] \subset C_n(2 \varepsilon) $$
as soon as $n$ is large enough. Hence, by definition of the measure $\nu$:
$$ \nu(\mathcal{U} \setminus C_n) = \mu\left( \mathcal{R} \setminus [\mathcal{U} \cap C_n, \mathcal{S}] \right) \leq \mu(\mathcal{R} \setminus C_n(\varepsilon/2)) \lesssim e^{- \delta_1 n} .$$
This proves that $ \nu(R_{n+1}) \geq 1 - e^{-\delta_0 n} $ for $n$ large enough depending on $\varepsilon$ and for some $\delta_0(\varepsilon)>0$. \\

Then, define $\tilde{R}_{n+1} := \bigsqcup_{\mathbf{a} \in \mathcal{R}_{n+1}} \text{int}_{\mathcal{U}} \left( U_{\mathbf{a}} \right) $. From the point of view of the measure $\nu$, it is indistinguishable from $R_{n+1}$. First, we prove that
$$ \bigcap_{i=0}^{k-1} F^{-n i} \left( \tilde{R}_{n+1} \right) \subset R_{n+1}^k .$$
Let $x \in \bigcap_{i=0}^{k-1} F^{-n i} \left( \tilde{R}_{n+1} \right) $. Since there exists $\mathbf{A}=\mathbf{a}_1' \dots \mathbf{a}_{k-1}' \mathbf{a} _k\in \mathcal{W}_{kn+1}$ such that $x \in U_{\mathbf{A}}$, we see that for any $i$ we can write $F^{ni}(x) \in U_{\mathbf{a}_{1+i}' \dots \mathbf{a}_k} \cap \tilde{R}_{n+1} $. 
So there exists $\mathbf{b}_{i+1} \in \mathcal{R}_{n+1}$ such that $ U_{\mathbf{a}_{1+i}' \dots \mathbf{a}_k} \cap \text{int}_{\mathcal{U}} U_{\mathbf{b}_{i+1}} \neq \emptyset$. Then $\mathbf{b}_{i+1} = \mathbf{a}_{i+1}$, for all $i$, which implies that $\mathbf{A} \in \mathcal{R}_{n+1}^k$.  
Now that the inclusion is proved, we see that $$ \nu\left( \mathcal{U} \setminus R_{n+1}^k \right) \leq \sum_{i=0}^{k-1} \nu\left( F^{-n i} \left( \mathcal{U} \setminus \tilde{R}_{n+1} \right) \right)  = k \nu( \mathcal{U} \setminus R_{n+1} ) ,$$
and so $ \nu( R_{n+1}^k ) \geq 1 - k e^{-\delta_1(\varepsilon) n} $ for $n$ large enough depending on $\varepsilon$. Now we may prove the cardinality estimate: since $ \nu(R_{n+1}^k) = \sum_{\mathbf{A} \in R_{n+1}^k} \nu(U_{\mathbf{A}}) $, we have

$$ \# \left(\mathcal{R}_{n+1}^k\right) e^{- \varepsilon \beta n} e^{- \delta \lambda k n} \lesssim \nu\left(R_{n+1}^k\right) \lesssim \# \left(\mathcal{R}_{n+1}^k\right) e^{\varepsilon \beta n} e^{- \delta \lambda k n} $$
and so
$$  e^{-\varepsilon \beta n} e^{\delta \lambda k n} \left( 1 - k e^{- \delta_1(\varepsilon) n} \right) \lesssim \# \mathcal{R}_{n+1}^k \lesssim  e^{\varepsilon \beta n} e^{\delta \lambda k n} .$$ \end{proof}

\section{Reduction to sums of exponential}

We can finally start the proof of Theorem 4.1.5 and Theorem 4.1.8. Recall the setting: we have fixed a $C^{2 + \alpha}$ Axiom A diffeomorphism $f:M \rightarrow M$ and a (perfect) basic set $\Omega \subset M$ on which $f$ has codimension 1 stable lamination. We have fixed a Hölder potential $\psi : \Omega \rightarrow \mathbb{R}$ and its associated equilibrium state $\mu$. We denote by $(F,\mathcal{U},\nu)$ the expanding dynamical system in factor defined in section 2. The measure $\nu$ is invariant by a transfer operator $\mathcal{L}_\varphi$, where $\varphi$ is some normalized and $\alpha$-Hölder potential on $\mathcal{U}$ (one can extend it in a Hölder fashion on $\mathcal{V}$ if desired). The Hölder exponent $\alpha$ is fixed for the rest of the Chapter. \\

We let $\chi \in C^\alpha(M,\mathbb{C})$ be a Hölder function with compact support, let $\xi \geq 1$, and we let $\psi: \omega \subset M \rightarrow \mathbb{R}$ be a $C^{1+\alpha}$ phase on some open set $\omega \subset M$ containing the support of $\chi$, satisfying:
$$ \|\psi\|_{C^{1+\alpha}} + (\inf_{x \in \omega} |\partial_u \psi|)^{-1} \leq \xi^{\rho_1} ,$$
for some $\rho_1>0$ that will be chosen small enough during the proof of Lemma 4.4.3 and Lemma 4.4.8. Our goal is to find a sufficient condition on the dynamics so that we get power decay of the following oscillatory integral:
$$ \widehat{\psi_*(\chi d\mu)}(\xi) = \int_{\Omega} e^{- 2 i \pi \xi \psi(x)} \chi(x) d\mu(x) .$$

To do so, six quantities will be at play: $\xi$, $K$, $n$, $k$, $\varepsilon_0$ and $\varepsilon$. We will think of $k$, $\varepsilon_0$ and $\varepsilon$ as being fixed. The constant $k$ is fixed using Theorem 4.5.1. The constant $\varepsilon_0>0$ will be fixed at the end of the Chapter (in the end of section 4.6 in the general case, and in Lemma 4.7.5 in the bunched case). The constant $\varepsilon>0$ is chosen at the very end of the proof to be small in front of every other constant that might appear. 
The only variable is $\xi$. We relate it to $n$ and $K$ by the formulas $$ n := \left\lfloor \frac{\ln |\xi|}{(2k+1) \lambda + \varepsilon_0} \right\rfloor \quad \text{and} \quad K := \left\lfloor \frac{((2k+1) \lambda + 2 \varepsilon_0)n}{\alpha |\ln \kappa|} \right\rfloor. $$

\begin{notations}
We recall a final set of notations, inspired from \cite{BD17}. For a fixed $n$ and $k$, denote:

\begin{itemize}
    \item $\textbf{A}=(\textbf{a}_0, \dots, \textbf{a}_k) \in \mathcal{W}_{n+1}^{k+1} \ , \ \textbf{B}=(\textbf{b}_1, \dots, \textbf{b}_k) \in \mathcal{W}_{n+1}^{k} $.
    \item We write $\textbf{A} \leftrightarrow \textbf{B}$ iff $\textbf{a}_{j-1} \rightsquigarrow \textbf{b}_j \rightsquigarrow \textbf{a}_j$ for all $j=1,\dots k$.
    \item If $\textbf{A} \leftrightarrow \textbf{B}$, then we define the words $\textbf{A} * \textbf{B} := \textbf{a}_0' \textbf{b}_1' \textbf{a}_1' \textbf{b}_2' \dots \textbf{a}_{k-1}' \textbf{b}_k' \textbf{a}_k$  \\ and  $\textbf{A} \# \textbf{B} :=  \textbf{a}_0' \textbf{b}_1' \textbf{a}_1' \textbf{b}_2' \dots \textbf{a}_{k-1}' \textbf{b}_k$.
    \item Denote by $b(\textbf{A}) \in \mathcal{A}$ the last letter of $\textbf{a}_k$.

\end{itemize}
\end{notations}

Our first goal is to prove the following reduction.
\begin{proposition}
Define $J_n := \{ e^{\varepsilon_0 n/2} \leq |\eta| \leq  e^{2 \varepsilon_0 n} \}$ and, for $j \in \{1, \dots, k\}$, $\mathbf{A} \in \mathcal{W}_{n+1}^{k+1}$, and $\mathbf{b} \in \mathcal{W}_{n+1}$ such that $\textbf{a}_{j-1} \rightsquigarrow \textbf{b} \rightsquigarrow \textbf{a}_j$, $$  \zeta_{j,\mathbf{A}}(\mathbf{b}) := e^{2 \lambda n} |g_{\mathbf{a}_{j-1}' \mathbf{b}}'(x_{\mathbf{a}_j})| .$$
There exists some constant $\beta>1$ such that for $n$ large enough depending on $\varepsilon$:
$$ |\widehat{\psi_*(\chi d\mu)}(\xi)|^2 \lesssim e^{ \varepsilon \beta n} e^{-\lambda \delta (2k+1) n} \sum_{\mathbf{A} \in \mathcal{R}_{n+1}^{k+1}(\omega)} \sup_{\eta \in J_n} \Bigg{|} \underset{\mathbf{A} \leftrightarrow \mathbf{B}}{\sum_{\mathbf{B} \in \mathcal{R}_{n+1}^k}} e^{i \eta \zeta_{1,\mathbf{A}}(\mathbf{b}_1) \dots \zeta_{k,\mathbf{A}}(\mathbf{b}_k) } \Bigg| $$
$$ \quad \quad  \quad \quad  \quad \quad  \quad \quad + e^{-  \varepsilon_0 n/2} +  e^{-\delta_1(\varepsilon) n} + e^{\varepsilon \beta n} \left( e^{- \lambda \alpha n} +  \kappa^{\alpha n} + e^{-(\alpha \lambda-\varepsilon_0) n/2} + e^{- \varepsilon_0 \delta_{up}n/6} \right).$$
\end{proposition}

Once Proposition 4.4.2 is established, if we manage to prove that the sum of exponentials enjoys exponential decay in $n$ under (QNL) (or under (B) and (LNL)), then choosing $\varepsilon$ small enough will allow us to see that $|\widehat{\psi_*(\chi d\mu)}(\xi)|^2$ enjoys polynomial decay in $\xi$, and Theorem 4.1.5 (or Theorem 4.1.8) will be proved. We prove Proposition 4.4.2 through a succession of lemmas. Our first step is to approximate our oscillatory integral over $\mu$ by an oscillatory integral over $\nu$.

\begin{lemma}

If $\rho_1$ is chosen small enough (depending on $\varepsilon_0$ and on the dynamics), then there exists $C>0$ such that, for $n$ large enough:

$$ \left| \int_{\Omega} e^{i \xi \psi} \chi d\mu - \int_{\mathcal{U}} e^{ i \xi \psi \circ f^K} \chi \circ f^K d \nu \right| \leq C e^{- \varepsilon_0 n/2} .$$
\end{lemma}

\begin{proof}

Let $h(x) := e^{ i \xi \psi(x)} \chi(x) $. Then, since $\mu$ is $f$ invariant, $$ \widehat{\psi_*(\chi d\mu)}(-\xi/2\pi) = \int_\Omega h \circ f^K d\mu .$$

We then check that $h \circ f^K$ is close to a $S$-constant function as $K$ grows larger.
First of all, notice that since $\psi$ is $C^1$, $e^{i \xi \psi}$ also is. In particular it is Lipschitz with constant $|\xi| \|\psi\|_{C^1} \leq |\xi|^{1+\rho_1}$. Since $\Omega$ is compact, it follows that $e^{- i \xi \psi}$ is $\alpha$-Hölder on $\Omega$ with constant $|\xi|^{1+\rho_1} \text{diam}(\Omega)^{1-\alpha}$. Since $\chi$ is $\alpha$-Hölder too, with some constant $C_{\chi}$, the product $h$ is also locally $\alpha$-Hölder, with some constant $\lesssim |\xi|^{1+\rho_1} $. \\

Now, let $x \in \mathcal{R}_a$. By definition, $\pi(x) \in U_a$ is in $W^s(x)$. Then $$ |h(f^K(x)) - h(f^K(\pi(x))) | \leq \|h\|_{C^\alpha} d( f^K(x),f^K(\pi(x)) )^\alpha $$ $$\leq \|h\|_{C^\alpha} \kappa^{\alpha K} \lesssim |\xi|^{1+\rho_1} \kappa^{\alpha K} .$$
The relation between $\xi$,$n$ and $K$ is chosen so that $|\xi| \kappa^{\alpha K} \simeq e^{-\varepsilon_0 n}$. If $\rho_1$ is taken small enough so that $|\xi|^{\rho_1}\lesssim e^{- \varepsilon_0 n/2}$, then there exists a constant $C>1$ such that 
$$ \| h \circ f^K - h \circ f^K \circ \pi \|_{\infty,\mathcal{R}} \leq C e^{- \varepsilon_0 n/2} .$$
The desired estimates follows from the definition of $\nu$. \end{proof}

Now we are ready to adapt the existing strategy for one dimensional expanding maps. From this point, we will follow \cite{SS20} and the strategy of Chapter 2 and 3. Notice however that we use an additional word $\mathbf{C}$ to cancel the distortions induced by $f^K$.

\begin{lemma}
$$ \left| \int_{\mathcal{U}} e^{i \xi \psi \circ f^K} \chi \circ f^K d \nu \right|^2 \lesssim \quad \quad \quad \quad \quad \quad \quad \quad \quad \quad \quad \quad \quad \quad \quad \quad \quad \quad \quad \quad \quad \quad \quad \quad \quad \quad \quad \quad \quad \quad $$ $$\quad \quad \quad  \Bigg{|} \sum_{\mathbf{C} \in \mathcal{R}_{K+1}} \underset{\mathbf{C} \rightsquigarrow \mathbf{A} \leftrightarrow \mathbf{B}}{ \underset{ \mathbf{B} \in \mathcal{R}_{n+1}^k}{ \underset{\mathbf{A} \in \mathcal{R}_{n+1}^{k+1} }{ \sum}} } \int_{U_{b(\mathbf{A})}} e^{i \xi \psi( f^K g_{\mathbf{C}'(\mathbf{A} * \mathbf{B})}(x))}  \chi(f^K g_{\mathbf{C}'(\mathbf{A} * \mathbf{B})}(x))  w_{\mathbf{C}'(\mathbf{A} * \mathbf{B})}(x)d\nu(x) \Bigg{|}^2 + e^{- \delta_1(\varepsilon) n }. $$

\end{lemma}

\begin{proof}

Since $\nu$ is invariant by $\mathcal{L}_\varphi$, we can write 
$$ \int_{\mathcal{U}} h \circ f^K d \nu = \sum_{\mathbf{C} \in \mathcal{W}_{K+1}}  \int_{U_{b(\mathbf{C})}} h\left(f^K(g_{ \mathbf{C}}(x))\right) w_{\mathbf{C}}(x) d\nu(x) .$$
\newpage
If we look at the part of the sum where $\mathbf{C}$ is not a regular word, we get by the Gibbs estimates:
$$ \Bigg| \sum_{\mathbf{C} \notin \mathcal{R}_{K+1}} \int_{U_{b(\mathbf{A})}} h\left(f^K(g_{\mathbf{C}}(x))\right) w_{\mathbf{C}}(x) d\nu(x) \Bigg| \lesssim \sum_{\mathbf{C} \notin \mathcal{R}_{K+1}} \nu(U_{\mathbf{C}}) \lesssim \nu(\mathcal{U} \setminus R_{K+1}), $$
which decays exponentially in $K$, hence in $n$, by (the proof of) Lemma 4.3.6.
Then, we iterate our transfer operator again on the main sum, to get the following term:
$$  \sum_{\mathbf{C} \in \mathcal{R}_{K+1}} \underset{\mathbf{C} \rightsquigarrow \mathbf{A} \leftrightarrow \mathbf{B}}{ \underset{ \mathbf{B} \in \mathcal{W}_{n+1}^k}{ \underset{\mathbf{A} \in \mathcal{W}_{n+1}^{k+1} }{ \sum}} } \int_{U_{b(\mathbf{A})}} e^{i \xi \psi( f^K g_{\mathbf{C}'(\mathbf{A} * \mathbf{B})}(x))}  \chi(f^K g_{\mathbf{C}'(\mathbf{A} * \mathbf{B})}(x))  w_{\mathbf{C}'(\mathbf{A} * \mathbf{B})}(x)d\nu(x) . $$
Looking at the part of the sum where words are not regular again, we see by the Gibbs estimates again that
$$ \Bigg| \sum_{\mathbf{C} \in \mathcal{R}_{K+1}} \underset{\text{or }\mathbf{B} \notin \mathcal{R}_{n+1}^k}{\underset{\mathbf{A} \notin \mathcal{R}_{n+1}^{k+1} }{\sum_{\mathbf{C} \rightsquigarrow \mathbf{A} \leftrightarrow \mathbf{B}}}} \int_{U_{b(\mathbf{A})}} h\left(f^K(g_{\mathbf{C}'(\mathbf{A} * \mathbf{B})}(x))\right) w_{\mathbf{C}' (\mathbf{A} * \mathbf{B})}(x) d\nu(x) \Bigg| \quad \quad  \quad \quad  \quad \quad  \quad \quad  \quad \quad $$ $$  \quad \quad  \quad \quad  \quad \quad  \quad \quad  \lesssim \sum_{\mathbf{C} \in \mathcal{R}_{K+1}} \underset{\text{or }\mathbf{B} \notin \mathcal{R}_{n+1}^k}{\underset{\mathbf{A} \notin \mathcal{R}_{n+1}^{k+1} }{\sum_{\mathbf{C} \rightsquigarrow \mathbf{A} \leftrightarrow \mathbf{B}}}} \nu(U_{\mathbf{C}}) \nu(U_{\mathbf{A} * \mathbf{B}}) \lesssim \nu(\mathcal{U} \setminus R_{n}^{2k+1}), $$
and the desired estimate follows from Lemma 4.3.6. \end{proof}

\begin{lemma}

Define $ \chi_{\mathbf{C}}(\mathbf{a}_0) := \chi(f^K g_{\mathbf{C}}(x_{\mathbf{a}_0}))$. There exists some constant $\beta>0$ such that, for $n$ large enough:

$$  \Bigg{|}\sum_{\mathbf{C} \in \mathcal{R}_{K+1}} \underset{\mathbf{C} \rightsquigarrow \mathbf{A} \leftrightarrow \mathbf{B}}{ \underset{ \mathbf{B} \in \mathcal{R}_{n+1}^k}{ \underset{\mathbf{A} \in \mathcal{R}_{n+1}^{k+1} }{ \sum}} } \int_{U_{b(\mathbf{A})}} e^{i \xi \psi( f^K g_{\mathbf{C}'(\mathbf{A} * \mathbf{B})}(x))} w_{\mathbf{C}'(\mathbf{A} * \mathbf{B})}(x) \chi(f^K g_{\mathbf{C}'(\mathbf{A} * \mathbf{B})}(x)) d\nu(x) \Bigg{|}^2  \quad \quad \quad \quad \quad \quad \quad $$ $$ \quad \quad \quad \quad  \lesssim  \Bigg{|} \sum_{\mathbf{C} \in \mathcal{R}_{K+1}} \underset{\mathbf{C} \rightsquigarrow \mathbf{A} \leftrightarrow \mathbf{B}}{ \underset{ \mathbf{B} \in \mathcal{R}_{n+1}^k}{ \underset{\mathbf{A} \in \mathcal{R}_{n+1}^{k+1}(\omega) }{ \sum}} } \chi_{\mathbf{C}}(\mathbf{a}_0) \int_{U_{b(\mathbf{A})}} e^{i \xi \psi( f^K g_{\mathbf{C}'(\mathbf{A} * \mathbf{B})}(x)))} w_{\mathbf{C}'(\mathbf{A} * \mathbf{B})}(x) d\nu(x) \Bigg{|}^2 + e^{\varepsilon \beta n} e^{- \lambda \alpha n}. $$

\end{lemma}

\begin{proof}

Our first move is to get rid of terms in the sum where $\chi(f^K g_{\mathbf{C}} g_{\mathbf{A} * \mathbf{B}}) =0 $. To this end, notice that $f^K g_{\mathbf{C}}$ is an holonomy map that sends $U_{b(\mathbf{C})}$ in $R_{b(\mathbf{C})}$. It doesn't really play a role for this question. The only word that matter here is $\mathbf{a}_0$ : if it isn't in $R_{n+1}(\omega)$, then by definition it implies that $\omega \cap U_{\mathbf{a}_0} = \emptyset$, and so $\chi(f^K g_{\mathbf{C}} g_{\mathbf{A} * \mathbf{B}}) = 0$. \\

Hence our main term is equal to the same sum, but where we have restricted $\mathbf{A}$ in $\mathcal{R}_{n+1}(\omega)$. \\

Next, denote $ \chi_{\mathbf{C}}(\mathbf{a}_0) := \chi(f^K g_{\mathbf{C}}(x_{\mathbf{a}_0})). $ The orders of magnitude of Lemma 4.3.3 combined gives us
$$ \left| \chi(f^K g_{\mathbf{C}} g_{\mathbf{A}*\mathbf{B}}(x)) -  \chi_{\mathbf{C}}(\mathbf{a}_0)  \right| \leq \|\chi\|_{C^\alpha} d\left(f^K g_{\mathbf{C}} g_{\mathbf{A}*\mathbf{B}}(x),f^K g_{\mathbf{C}} (x_{\mathbf{a}_0} ) \right)^\alpha $$ 
$$ \lesssim \|\partial_u (f^K)\|_{\infty,U_{\mathbf{C}}}^\alpha \| g_{\mathbf{C}}' \|_{\infty,U_{b(\mathbf{C})}}^\alpha  \text{diam}(U_{\mathbf{a}_0 })^\alpha  \lesssim e^{\varepsilon \beta n} e^{- \lambda \alpha n} .$$
Hence, by the Gibbs estimates 
$$ { \Bigg{|} {\sum_{\mathbf{A},\mathbf{B},\mathbf{C}}} \int_{U_{b(\mathbf{A})}} e^{i \xi \psi( f^K g_{\mathbf{C}'(\mathbf{A} * \mathbf{B})} )} \chi( f^K g_{\mathbf{C}'(\mathbf{A} * \mathbf{B})})  w_{\mathbf{C}'(\mathbf{A} * \mathbf{B})} d\nu  -   \sum_{\mathbf{A},\mathbf{B},\mathbf{C}} \chi_{\mathbf{C}}(\mathbf{a}_0)  \int_{U_{b(\mathbf{A})}} e^{i \xi \psi(f^K g_{\mathbf{C}'(\mathbf{A} * \mathbf{B})})} w_{\mathbf{C}'(\mathbf{A} * \mathbf{B})} d\nu \Bigg{|} }$$
$$ \lesssim e^{\varepsilon \beta n} e^{- \lambda \alpha n}  \sum_{\mathbf{A},\mathbf{B},\mathbf{C}} \nu(U_{\mathbf{C}'(\mathbf{A} * \mathbf{B}) }) \lesssim e^{\varepsilon \beta n} e^{- \lambda \alpha n} .$$
\end{proof}

\begin{lemma}
There exists some constant $\beta>0$ such that, for $n$ large enough:
$$ \Bigg{|} \sum_{\mathbf{C} \in \mathcal{R}_{K+1}} \underset{\mathbf{C} \rightsquigarrow \mathbf{A} \leftrightarrow \mathbf{B}}{ \underset{ \mathbf{B} \in \mathcal{R}_{n+1}^k}{ \underset{\mathbf{A} \in \mathcal{R}_{n+1}^{k+1}(\omega) }{ \sum}} } \chi_{\mathbf{C}}(\mathbf{a}_0) \int_{U_{b(\mathbf{A})}} e^{i \xi \psi( f^K g_{\mathbf{C}'(\mathbf{A} * \mathbf{B})}(x)))} w_{\mathbf{C}'(\mathbf{A} * \mathbf{B})}(x) d\nu(x) \Bigg{|}^2 \quad \quad \quad \quad \quad \quad \quad  $$
$$ \quad \quad \quad \quad \lesssim  e^{-(2k-1) \lambda \delta n} e^{- \lambda \delta K} \sum_{\mathbf{C} \in \mathcal{R}_{K+1}} \underset{\mathbf{C} \rightsquigarrow \mathbf{A} \leftrightarrow \mathbf{B}}{ \underset{ \mathbf{B} \in \mathcal{R}_{n+1}^k}{ \underset{\mathbf{A} \in \mathcal{R}_{n+1}^{k+1}(\omega) }{ \sum}} } \Bigg{|} \int_{U_{b(\mathbf{A})}}  e^{i \xi \psi( f^K( g_{\mathbf{C}'(\mathbf{A} * \mathbf{B})}(x)))} w_{\mathbf{a}_k}(x) d\nu(x) \Bigg{|}^2 + e^{\varepsilon \beta n} \kappa^{\alpha n}. $$
\end{lemma}

\begin{proof}

Notice that $w_{\mathbf{C}'(\mathbf{A} * \mathbf{B})}(x)$ and $w_{\mathbf{a}_k}(x)$ are related by
$$ w_{\mathbf{C}'(\mathbf{A} * \mathbf{B})}(x) = w_{\mathbf{C}'(\mathbf{A} \# \mathbf{B})}(g_{\mathbf{a}_k}(x)) w_{\mathbf{a}_k}(x) .$$
Moreover:
$$  \frac{w_{\mathbf{C}'(\mathbf{A} \# \mathbf{B})}(g_{\mathbf{a}_k}(x))}{w_{\mathbf{C}'(\mathbf{A} \# \mathbf{B})}(x_{\mathbf{a}_k})} = \exp \left( S_{K+2kn}\varphi(g_{\mathbf{C}'(\mathbf{A} \# \mathbf{B})}(g_{\mathbf{a}_k}(x))) - S_{K+2kn}\varphi( g_{\mathbf{C}'(\mathbf{A} \# \mathbf{B})}(x_{\mathbf{a}_k})) \right) ,$$
with 
$$ \left| S_{K+2kn}\varphi(g_{\mathbf{C}'(\mathbf{A} \# \mathbf{B})}(g_{\mathbf{a}_k}(x)) - S_{K+2kn}\varphi( g_{\mathbf{C}'(\mathbf{A} \# \mathbf{B})}(x_{\mathbf{a}_k}))  \right| \lesssim \sum_{j=0}^{K+2kn-1} \kappa^{\alpha(K+n(2k+1) - j )} \lesssim \kappa^{\alpha n} $$
since $\varphi$ is $\alpha$-Hölder. Hence, there exists some constant $C>0$ such that
$$ e^{-C \kappa^{\alpha n}} w_{\mathbf{C}'(\mathbf{A} \# \mathbf{B})}(x_{\mathbf{a}_k}) \leq w_{\mathbf{C}'(\mathbf{A} \# \mathbf{B})}(g_{\mathbf{a}_k}(x)) \leq e^{C \kappa^{\alpha n}} w_{\mathbf{C}'(\mathbf{A} \# \mathbf{B})}(x_{\mathbf{a}_k}), $$
which gives
$$ \left| w_{\mathbf{C}'(\mathbf{A} \# \mathbf{B})}(g_{\mathbf{a}_k}(x)) -  w_{\mathbf{C}'(\mathbf{A} \# \mathbf{B})}(x_{\mathbf{a}_k})  \right| \leq \max\left| e^{\pm C \kappa^{\alpha n}} -1 \right|  w_{\mathbf{C}'(\mathbf{A} \# \mathbf{B})}(x_{\mathbf{a}_k}) \lesssim \kappa^{\alpha n} w_{\mathbf{C}'(\mathbf{A} \# \mathbf{B})}(x_{\mathbf{a}_k}) .$$
Hence
$$ \Bigg{|}  \sum_{\mathbf{C} \in \mathcal{R}_{K+1}} \underset{\mathbf{C} \rightsquigarrow \mathbf{A} \leftrightarrow \mathbf{B}}{ \underset{ \mathbf{B} \in \mathcal{R}_{n+1}^k}{ \underset{\mathbf{A} \in \mathcal{R}_{n+1}^{k+1}(\omega) }{ \sum}} } \chi_{\mathbf{C}}(\mathbf{a}_0) \int_{U_{b(\mathbf{A})}} e^{i \xi \psi(f^K g_{\mathbf{C}'(\mathbf{A} * \mathbf{B})}(x))} \left(w_{\mathbf{C}'(\mathbf{A} * \mathbf{B})}(x) - w_{\mathbf{C}'(\mathbf{A} \# \mathbf{B})}(x_{\mathbf{a}_k}) w_{\mathbf{a}_k}(x) \right) d\nu(x) \Bigg{|}   $$

$$ \lesssim \kappa^{\alpha n}  \sum_{\mathbf{A},\mathbf{B},\mathbf{C}}\int_{U_{b(\mathbf{A})}}  w_{\mathbf{C}'(\mathbf{A} \# \mathbf{B})}(x_{\mathbf{a}_k}) w_{\mathbf{a}_k}(x) d\nu(x) \ \lesssim e^{\varepsilon \beta n} \kappa^{\alpha n} $$
by the Gibbs estimates. It follows that $$ \Bigg{|} \sum_{\mathbf{C} \in \mathcal{R}_{K+1}} \underset{\mathbf{C} \rightsquigarrow \mathbf{A} \leftrightarrow \mathbf{B}}{ \underset{ \mathbf{B} \in \mathcal{R}_{n+1}^k}{ \underset{\mathbf{A} \in \mathcal{R}_{n+1}^{k+1}(\omega) }{ \sum}} } \chi_{\mathbf{C}}(\mathbf{a}_0) \int_{U_{b(\mathbf{A})}} e^{i \xi \psi( f^K g_{\mathbf{C}'(\mathbf{A} * \mathbf{B})}(x)))} w_{\mathbf{C}'(\mathbf{A} * \mathbf{B})}(x) d\nu(x) \Bigg{|}^2 \quad \quad \quad \quad \quad \quad \quad  $$
$$ \quad \quad \lesssim  e^{\varepsilon \beta n} 
\kappa^{\alpha n} +  \Bigg{|}  \sum_{\mathbf{C} \in \mathcal{R}_{K+1}} \underset{\mathbf{C} \rightsquigarrow \mathbf{A} \leftrightarrow \mathbf{B}}{ \underset{ \mathbf{B} \in \mathcal{R}_{n+1}^k}{ \underset{\mathbf{A} \in \mathcal{R}_{n+1}^{k+1}(\omega) }{ \sum}} } w_{\mathbf{C}'(\mathbf{A} \# \mathbf{B})}(x_{\mathbf{a}_k})  \chi_{\mathbf{C}}(\mathbf{a}_0) \int_{U_{b(\mathbf{A})}} e^{i \xi \psi(f^K g_{\mathbf{C}'(\mathbf{A} * \mathbf{B})}(x) ) }   w_{\mathbf{a}_k}(x) d\nu(x) \Bigg{|}^2 . $$
Moreover, by Cauchy-Schwartz and by the orders of magnitude of Lemma 4.3.3,
$$ \Bigg{|}  \sum_{\mathbf{C} \in \mathcal{R}_{K+1}} \underset{\mathbf{C} \rightsquigarrow \mathbf{A} \leftrightarrow \mathbf{B}}{ \underset{ \mathbf{B} \in \mathcal{R}_{n+1}^k}{ \underset{\mathbf{A} \in \mathcal{R}_{n+1}^{k+1}(\omega) }{ \sum}} } w_{\mathbf{C}'(\mathbf{A} \# \mathbf{B})}(x_{\mathbf{a}_k})  \chi_{\mathbf{C}}(\mathbf{a}_0) \int_{U_{b(\mathbf{A})}} e^{i \xi \psi(f^K g_{\mathbf{C}'(\mathbf{A} * \mathbf{B})}(x) ) }   w_{\mathbf{a}_k}(x) d\nu(x) \Bigg{|}^2   $$
$$  \lesssim e^{\varepsilon \beta n} e^{-\lambda \delta (2k-1) n} e^{- \lambda \delta K} \sum_{\mathbf{C} \in \mathcal{R}_{K+1}} \underset{\mathbf{C} \rightsquigarrow \mathbf{A} \leftrightarrow \mathbf{B}}{ \underset{ \mathbf{B} \in \mathcal{R}_{n+1}^k}{ \underset{\mathbf{A} \in \mathcal{R}_{n+1}^{k+1}(\omega) }{ \sum}} }\left| \int_{U_{b(\mathbf{A})}} e^{i \xi \psi(f^K g_{\mathbf{C}' (\mathbf{A} * \mathbf{B})}(x) ) }   w_{\mathbf{a}_k}(x) d\nu(x) \right|^2 ,$$
where one could increase $\beta$ if necessary.\end{proof}

\begin{lemma}
Define $$ \Delta_{\mathbf{A},\mathbf{B},\mathbf{C}}(x,y) :=  \psi( f^K g_{\mathbf{C}'(\mathbf{A} * \mathbf{B})}(x)) - \psi( f^K g_{\mathbf{C}'(\mathbf{A} * \mathbf{B})}(y)). $$
There exists some constant $\beta>0$ such that, for $n$ large enough:
$$ e^{-\lambda \delta (2k-1) n} e^{- \lambda \delta K} \sum_{\mathbf{C} \in \mathcal{R}_{K+1}} \underset{\mathbf{C} \rightsquigarrow \mathbf{A} \leftrightarrow \mathbf{B}}{ \underset{ \mathbf{B} \in \mathcal{R}_{n+1}^k}{ \underset{\mathbf{A} \in \mathcal{R}_{n+1}^{k+1}(\omega) }{ \sum}} } \Bigg{|} \int_{U_{b(\mathbf{A})}}  e^{i \xi \psi( f^K g_{\mathbf{C}'(\mathbf{A} * \mathbf{B})}(x))} w_{\mathbf{a}_k}(x) d\nu(x) \Bigg{|}^2 $$
$$ \lesssim e^{\varepsilon \beta n } e^{-\lambda \delta (2k+1) n} e^{- \lambda \delta K} \underset{\mathbf{C} \rightsquigarrow \mathbf{A}}{\underset{\mathbf{A} \in \mathcal{R}_{n+1}^{k+1}(\omega)}{\sum_{\mathbf{C} \in \mathcal{R}_{K+1}}}} \iint_{U_{b(\mathbf{A})}^2 }  \Bigg{|} \underset{\mathbf{A} \leftrightarrow \mathbf{B}}{\sum_{\mathbf{B} \in \mathcal{R}_{n+1}^k}} e^{i \xi \left|\Delta_{\mathbf{A},\mathbf{B},\mathbf{C}}\right|(x,y) } \Bigg| d\nu(x) d\nu(y)  .$$

\end{lemma}

\begin{proof}

We first open up the modulus squared:

$$  \underset{\mathbf{B} \in \mathcal{R}_{n+1}^k}{\underset{\mathbf{A} \in \mathcal{R}_{n+1}^{k+1}(\omega) }{\sum_{\mathbf{A} \leftrightarrow \mathbf{B}}}}  \Bigg{|} \int_{U_{b(\mathbf{A})}}  e^{i \xi \psi( f^K g_{\mathbf{C}'(\mathbf{A} * \mathbf{B})}(x))} w_{\mathbf{a}_k}(x) d\nu(x) \Bigg{|}^2  $$ $$= \underset{\mathbf{B} \in \mathcal{R}_{n+1}^k}{\underset{\mathbf{A} \in \mathcal{R}_{n+1}^{k+1}(\omega) }{\sum_{\mathbf{A} \leftrightarrow \mathbf{B}}}} \iint_{U_{b(\mathbf{A})}^2} w_{\mathbf{a}_k}(x) w_{\mathbf{a}_k}(y) e^{i \xi \Delta_{\mathbf{A},\mathbf{B},\mathbf{C}}(x,y)  }  d\nu(x) d\nu(y).  $$
Since this quantity is real, we get:
$$ = \underset{\mathbf{B} \in \mathcal{R}_{n+1}^k}{\underset{\mathbf{A} \in \mathcal{R}_{n+1}^{k+1}(\omega) }{\sum_{\mathbf{A} \leftrightarrow \mathbf{B}}}} \iint_{U_{b(\mathbf{A})}^2} w_{\mathbf{a}_k}(x) w_{\mathbf{a}_k}(y) \cos({ \xi \Delta_{\mathbf{A},\mathbf{B},\mathbf{C}}(x,y)  })  d\nu(x) d\nu(y) $$

$$ = \sum_{\mathbf{A} \in \mathcal{R}_{n+1}^{k+1}(\omega)} \iint_{U_{b(\mathbf{A})}^2 }  w_{\mathbf{a}_k}(x) w_{\mathbf{a}_k}(y) \underset{\mathbf{A} \leftrightarrow \mathbf{B}}{\sum_{\mathbf{B} \in \mathcal{R}_{n+1}^k}}  \cos\left( \xi  |\Delta_{\mathbf{A},\mathbf{B},\mathbf{C}}|(x,y) \right) d\nu(x) d\nu(y),$$
and then we conclude using the triangle inequality and the estimates of Lemma 4.3.3, as follow:
$$ \lesssim e^{\varepsilon \beta n } e^{- 2 \lambda \delta n} \sum_{\mathbf{A} \in \mathcal{R}_{n+1}^{k+1}(\omega)} \iint_{U_{b(\mathbf{A})}^2 }  \Bigg{|} \underset{\mathbf{A} \leftrightarrow \mathbf{B}}{\sum_{\mathbf{B} \in \mathcal{R}_{n+1}^k}} e^{i \xi \left|\Delta_{\mathbf{A},\mathbf{B},\mathbf{C}}\right|(x,y) } \Bigg| d\nu(x) d\nu(y) . $$

\end{proof}

\begin{lemma}
Define $$\zeta_{j,\mathbf{A}}(\mathbf{b}) := e^{2 \lambda n} |g_{\mathbf{a}_{j-1}' \mathbf{b}}'(x_{\mathbf{\mathbf{a}}_j})|$$
and $$ J_n := \{ e^{\varepsilon_0 n/2} \leq |\eta| \leq e^{2 \varepsilon_0 n}  \} .$$
There exists $\beta > 0$ such that, for $n$ large enough depending on $\varepsilon$,

$$  e^{- \lambda \delta K} \sum_{\mathbf{C} \in \mathcal{R}_{K+1}} e^{-\lambda \delta (2k+1) n}   \underset{\mathbf{C} \rightsquigarrow \mathbf{A}}{\underset{\mathbf{A} \in \mathcal{R}_{n+1}^{k+1}(\omega)}{\sum}} \iint_{U_{b(\mathbf{A})}^2 }  \Bigg{|} \underset{\mathbf{A} \leftrightarrow \mathbf{B}}{\sum_{\mathbf{B} \in \mathcal{R}_{n+1}^k}} e^{i \xi \left|\Delta_{\mathbf{A},\mathbf{B},\mathbf{C}}\right|(x,y) } \Bigg| d\nu(x) d\nu(y)  \quad \quad \quad \quad \quad \quad   $$
$$ \quad \quad \quad \quad \quad \quad \lesssim  e^{-\lambda \delta (2k+1) n} \sum_{\mathbf{A} \in \mathcal{R}_{n+1}^{k+1}(\omega)} \sup_{\eta \in J_n} \Bigg{|} \underset{\mathbf{A} \leftrightarrow \mathbf{B}}{\sum_{\mathbf{B} \in \mathcal{R}_{n+1}^k}} e^{i \eta \zeta_{1,\mathbf{A}}(\mathbf{b}_1) \dots \zeta_{k,\mathbf{A}}(\mathbf{b}_k) } \Bigg|  + e^{ \varepsilon \beta n} \left( e^{-  (\alpha \lambda - \varepsilon_0) n/2}  + e^{- \varepsilon_0 n \delta_{\text{up}}/6} \right) .$$

\end{lemma}

\begin{proof}

Our goal is to carefully approximate $\Delta_{\mathbf{A},\mathbf{B},\mathbf{C}}$ by a product of derivatives of $g_{\mathbf{a}_{j-1}'\mathbf{b}_j}$, and then to renormalize the phase. Using arc length parameterization on $\mathcal{V}$, our 1-dimensional setting allows us to apply the mean value theorem: for all $x,y \in U_{b(\mathbf{A})}$, there exists $z \in V_{{\mathbf{a}_k}}$ such that
$$ |\psi( f^K g_{\mathbf{C}} g_{\mathbf{A}* \mathbf{B}} (x) ) -  \psi( f^K g_{\mathbf{C}} g_{\mathbf{A}* \mathbf{B}} (y) )| = |\psi( f^K g_{\mathbf{C}} g_{\mathbf{A} \# \mathbf{B}} (g_{\mathbf{a}_k} x) ) -  \psi( f^K g_{\mathbf{C}} g_{\mathbf{A} \# \mathbf{B}} (g_{\mathbf{a}_k} y) )|$$ $$ = |\partial_u \psi(  f^K g_{\mathbf{C}'(\mathbf{A} \# \mathbf{B})} z )| |\partial_u f^K(g_{\mathbf{C}'(\mathbf{A} \# \mathbf{B})} z )| |g_{\mathbf{C}}'(g_{\mathbf{A} \# \mathbf{B} } z)| \left[ \prod_{j=1}^{k} |g_{\mathbf{a}_{j-1}'\mathbf{b}_j}'(g_{\mathbf{a}_j' \mathbf{b}_{j+1}' \dots \mathbf{a}_{k-1}' \mathbf{b}_k}z) )| \right]  \ d^u(g_{\mathbf{a}_k} x,g_{\mathbf{a}_k} y) .$$
The estimates of Lemma 4.3.3 gives
$$ \left| {\Delta_{\mathbf{A},\mathbf{B},\mathbf{C}}(x,y)} \right| \lesssim  |\psi|_{C^{1+\alpha}} e^{\varepsilon \beta n} e^{- (2k+1) \lambda n} \lesssim |\xi|^{\rho_1}  e^{\varepsilon \beta n} e^{- (2k+1) \lambda n}  .$$
We then relate $|\Delta_{\mathbf{A},\mathbf{B}, \mathbf{C}}|$ to the $\zeta_{j,\mathbf{A}}$. Set $$ |\tilde{\Delta}_{\mathbf{A},\mathbf{B},\mathbf{C}}(x,y)| := |\partial_u \psi(  f^K g_{\mathbf{C}}(x_{\mathbf{a}_0} )| |\partial_u f^K(g_{\mathbf{C}} (x_{\mathbf{a}_0}) )| |g_{\mathbf{C}}(x_{\mathbf{a}_0})| \left[ \prod_{j=1}^{k} |g_{\mathbf{a}_{j-1}'\mathbf{b}_j}'(x_{\mathbf{a}_j}) )| \right]  \ d^u(g_{\mathbf{a}_k}x,g_{\mathbf{a}_k}y).$$
Then, as before, $$ \left| {\tilde{\Delta}_{\mathbf{A},\mathbf{B},\mathbf{C}}(x,y)} \right| \lesssim |\xi|^{\rho_1}  e^{\varepsilon \beta n} e^{- (2k+1) \lambda n}  .$$ 
Hence, using the fact that $|e^{s} - e^{t}| \leq e^{\max(s,t)} |s-t| $, we get

$$ \left| |\Delta_{\mathbf{A},\mathbf{B},\mathbf{C}}(x,y)| - |\tilde{\Delta}_{\mathbf{A},\mathbf{B},\mathbf{C}}(x,y)| \right| \leq |\xi|^{\rho_1} e^{\varepsilon \beta n} e^{-(2k+1)\lambda n} \left| \ln |\Delta_{\mathbf{A},\mathbf{B},\mathbf{C}}|(x,y) - \ln |\tilde{\Delta}_{\mathbf{A},\mathbf{B},\mathbf{C}}|(x,y) \right|.$$

Moreover, using the estimates of Lemma 4.3.3, exponentially vanishing variations of Hölder maps, the lower bound $\partial_u \psi(f^K g_{\mathbf{C}}(x_{\mathbf{a}_0})) \gtrsim |\xi|^{-\rho_1} $, and using $|\ln a - \ln b| \leq \frac{|a-b|}{\min(a,b)}$, we get:
\begin{itemize}
    \item $ \left| \ln|\partial_u \psi(  f^K g_{\mathbf{C}'(\mathbf{A}\# \mathbf{B})} z )| - \ln |\partial_u \psi(f^K g_\mathbf{C}(x_{\mathbf{a}_0}))| \right| \lesssim |\xi|^{2 \rho_1} e^{\varepsilon \beta n} ( e^{\lambda K} e^{- \lambda K} e^{-\lambda n})^\alpha \lesssim |\xi|^{2 \rho_1} e^{\varepsilon \beta n }e^{- \alpha \lambda n} ,$
    \item $ \left| \ln |\partial_u f^K(g_{\mathbf{C}'(\mathbf{A} \# \mathbf{B})} z )| - \ln |\partial_u f^K(g_{\mathbf{C}} (x_{\mathbf{a}_0}) )| \right|  \lesssim e^{\varepsilon \beta n }e^{-\alpha \lambda n} $,
    \item $ \left| \ln |g_{\mathbf{C}}'(g_{\mathbf{A} \# \mathbf{B}} z )| - \ln |g_{\mathbf{C}}' (x_{\mathbf{a}_0}) | \right|  \lesssim e^{\varepsilon \beta n }e^{- \alpha \lambda n} $,
    \item $ \left| \ln|g_{\mathbf{a}_{j-1}'\mathbf{b}_j}'(g_{\mathbf{a}_j' \mathbf{b}_{j+1}' \dots \mathbf{a}_{k-1}' \mathbf{b}_k}z) )| - \ln|g_{\mathbf{a}_{j-1}' \mathbf{b}_j}'(x_{\mathbf{a}_j})| \right| \lesssim e^{\varepsilon \beta n} e^{-\alpha \lambda n} $,
\end{itemize}
so that summing every estimates gives us $$ \left| \ln |\Delta_{\mathbf{A},\mathbf{B},\mathbf{C}}|(x,y) - \ln |\tilde{\Delta}_{\mathbf{A},\mathbf{B},\mathbf{C}}|(x,y) \right| \lesssim |\xi|^{2 \rho_1} e^{\varepsilon \beta n} e^{- \alpha \lambda n}. $$
Hence, $$  \left| \Delta_{\mathbf{A},\mathbf{B},\mathbf{C}}(x,y) - \tilde{\Delta}_{\mathbf{A},\mathbf{B},\mathbf{C}}(x,y) \right| \lesssim |\xi|^{3 \rho_1} e^{\varepsilon \beta n} e^{-(2k+1+\alpha) \lambda n} ,$$
which allows us to approximate our main integral as follows:
$$e^{-(2k+1) \delta \lambda n - \delta \lambda K} \Bigg{|}  \underset{\mathbf{C} \rightsquigarrow \mathbf{A}}{\underset{\mathbf{A} \in \mathcal{R}_{n+1}^{k+1}(\omega)}{\sum_{\mathbf{C} \in \mathcal{R}_{K+1}}}} \iint_{U_{b(\mathbf{A})}^2 }  \Big{|} \underset{\mathbf{A} \leftrightarrow \mathbf{B}}{\sum_{\mathbf{B} \in \mathcal{R}_{n+1}^k}} e^{i \xi \left|\Delta_{\mathbf{A},\mathbf{B},\mathbf{C}}\right| } \Big| d\nu \otimes d\nu -  \underset{\mathbf{C} \rightsquigarrow \mathbf{A}}{\underset{\mathbf{A} \in \mathcal{R}_{n+1}^{k+1}(\omega)}{\sum_{\mathbf{C} \in \mathcal{R}_{K+1}}}} \iint_{U_{b(\mathbf{A})}^2 }  \Big{|} \underset{\mathbf{A} \leftrightarrow \mathbf{B}}{\sum_{\mathbf{B} \in \mathcal{R}_{n+1}^k}} e^{i \xi \left|\tilde{\Delta}_{\mathbf{A},\mathbf{B},\mathbf{C}}\right| } \Big| d\nu \otimes d\nu \Bigg{|}  $$
$$ \lesssim |\xi|^{1+3\rho_1} e^{-(2k+1+\alpha) \lambda n} \lesssim e^{-(\alpha \lambda - \varepsilon_0)n/2},$$
since $|\xi| \simeq  e^{(2k+1) \lambda n} e^{\varepsilon_0 n} $, and for $\rho_1$ chosen small enough so that $|\xi|^{3 \rho_1} \leq e^{(\alpha \lambda - \varepsilon_0) n/2}$. 
To conclude, we notice that $ |\xi| |\tilde{\Delta}_{\mathbf{A},\mathbf{B},\mathbf{C}}|$ can be written as a product like so:
$$ |\xi| |\tilde{\Delta}_{\mathbf{A},\mathbf{B},\mathbf{C}}|(x,y) =  \eta_{\mathbf{A},\mathbf{C}}(x,y)  \zeta_{1,\mathbf{A}}(\mathbf{b}_1) \dots  \zeta_{k,\mathbf{A}}(\mathbf{b}_k),  $$
where $$ \eta_{\mathbf{A},\mathbf{C}}(x,y) := |\xi| |\partial_u \psi(  f^K g_{\mathbf{C}'}(x_{\mathbf{a}_0} )| |\partial_u f^K(g_{\mathbf{C}'} (x_{\mathbf{a}_0}) )| |g_{\mathbf{C}'}(x_{\mathbf{a}_0})| e^{-2 k \lambda n }d^u(g_{\mathbf{a}_k}x,g_{\mathbf{a}_0}y). $$
We estimate $\eta_{\mathbf{A},\mathbf{C}}$ using the hypothesis made on $\partial_u \psi$, the estimates of Lemma 4.3.3, and the mean value theorem, to get
$$ e^{-\varepsilon \beta n} e^{2 \varepsilon_0 n/3} d^u(x,y) \leq |\xi|^{-\rho_1} e^{-\varepsilon \beta n} e^{\varepsilon_0 n} d^u(x,y) \lesssim \eta_{\mathbf{A},\mathbf{C}}(x,y)  \lesssim |\xi|^{\rho_1} e^{\varepsilon \beta n} e^{\varepsilon_0 n} \leq e^{2 \varepsilon_0 n} ,$$
as long as $\rho_1>0$ is chosen so small that $|\xi|^{\rho_1} \leq e^{\varepsilon_0 n/3}$. Notice that for the lower inequality to hold, it was critical for $\mathbf{a}_0$ to be in $\mathcal{R}_{n+1}(\omega)$, and not just in $\mathcal{R}_{n+1}$. \\

We then see that $\eta_{\mathbf{A},\mathbf{C}}(x,y) \in J_n$ as soon as $d^u(x,y) \geq e^{\varepsilon \beta n - \varepsilon_0 n/6}$. To get rid of the part of the integral where $d^u(x,y)$ is too small, we use the upper regularity of $\nu$ proved in Lemma 4.2.41. For all $y \in \mathcal{U}$, the ball $B(y, e^{\varepsilon \beta n - \varepsilon_0 n /6} )$ has measure $ \lesssim e^{-( \varepsilon_0/6  - \beta \varepsilon) \delta_{up} n} $, so that by integrating over $y$,

$$ \nu \otimes \nu \left( \{ (x,y) \in \mathcal{U}^2 \ , \ |x-y|<  e^{\varepsilon \beta n - \varepsilon_0 n/6} \} \right) \lesssim e^{-( \varepsilon_0/6  - \beta \varepsilon) \delta_{up} n} $$
as well. Hence we can cut the double integral in two, the part near the diagonal which is controlled by the previous estimates, and the part far away from the diagonal where $\eta_{\mathbf{A},\mathbf{C}}(x,y) \in J_n$. Once this is done, the sum over $\mathbf{C}$ disappears, as there is no longer a dependence over $\mathbf{C}$ in the phase once taken the supremum for $\eta \in J_n$. \end{proof}

\section{The sum-product phenomenon}

The version of the sum-product phenomenon that we will use is the following. Similar statements can be found in Section 1.3.2.

\begin{theorem}
Fix $0 < \gamma < 1$. There exist an integer $k \geq 1$, $c>0$ and $\varepsilon_1 > 0$ depending only on $\gamma$ such that the following holds for $\eta \in \mathbb{R}$ with $|\eta|$ large enough. Let $1 < R < |\eta|^{\varepsilon_1}$ , $N > 1$ and $\mathcal{Z}_1,\dots , \mathcal{Z}_k$ be finite sets such that $ \# \mathcal{Z}_j \leq RN$. Consider some maps $\zeta_j : \mathcal{Z}_j \rightarrow \mathbb{R} $, $j = 1, \dots , k$, such that, for all $j$:
$$ \zeta_j ( \mathcal{Z}_j ) \subset [R^{-1},R] $$ and 
$$\forall \sigma \in [|\eta|^{-2}, |\eta|^{- \varepsilon_1}], \quad \# \{\mathbf{b} , \mathbf{c} \in \mathcal{Z}^2_j , \ |\zeta_j(\mathbf{b}) - \zeta_j(\mathbf{c})| \leq \sigma \} \leq N^2 \sigma^{\gamma}.$$
Then: 
$$ \left| N^{-k} \sum_{\mathbf{b}_1 \in \mathcal{Z}_1,\dots,\mathbf{b}_k \in \mathcal{Z}_k} \exp\left( i \eta \zeta_1(\mathbf{b}_1) \dots \zeta_k(\mathbf{b}_k) \right) \right| \leq c |\eta|^{-{\varepsilon_1}}$$
\end{theorem}

We will use Theorem 4.5.1 on the maps $\zeta_{j,\mathbf{A}}$.
Let's carefully define the framework. \\
For some fixed $\mathbf{A} \in \mathcal{R}_{n+1}^{k+1}(\omega)$, define for $j=1, \dots, k$  $$ \mathcal{Z}_j := \{ \mathbf{b} \in \mathcal{R}_{n+1} , \mathbf{a}_{j-1} \rightsquigarrow \mathbf{b} \rightsquigarrow \mathbf{a}_j \  \} ,$$
so that the maps $\zeta_{j,\mathbf{A}}(\mathbf{b}) := e^{2 \lambda n} |g_{\mathbf{a}_{j-1}' \mathbf{b}}'(x_{{\mathbf{a}}_j})|$ are defined on $\mathcal{Z}_j$. Recall from Lemma 4.3.3 and Lemma 4.3.6 that there exists a constant $\beta>0$  such that:
$$\# \mathcal{Z}_j \leq e^{\varepsilon \beta n} e^{ \delta \lambda n} $$
and
$$ \zeta_{j,\mathbf{A}}( \mathcal{Z}_j ) \subset \left[ e^{- \varepsilon \beta n}, e^{\varepsilon \beta n} \right] .$$

Let $\gamma>0$ be small enough. Theorem 4.5.1 then fixes $k$ and some $\varepsilon_1$.
The goal is to apply Theorem 4.5.1 to the maps $\zeta_{j,\mathbf{A}}$, for $N := e^{\lambda \delta n}$, $R:=e^{\varepsilon \beta n}$ and $\eta \in J_n$. Notice that choosing $\varepsilon$ small enough ensures that $R<|\eta|^{\varepsilon_1}$, and taking $n$ large enough ensures that $|\eta|$ is large. 
If we are able to check the non-concentration hypothesis for most words $\mathbf{A}$, then Theorem 4.5.1 can be applied and we would be able to conclude the proof of Theorem 4.1.5.
Indeed, we already know that
$$ |\widehat{\psi_*(\chi d\mu)}(\xi)|^2 \lesssim e^{ \varepsilon \beta n} e^{-\lambda \delta (2k+1) n} \sum_{\mathbf{A} \in \mathcal{R}_{n+1}^{k+1}(\omega)} \sup_{\eta \in J_n} \Bigg{|} \underset{\mathbf{A} \leftrightarrow \mathbf{B}}{\sum_{\mathbf{B} \in \mathcal{R}_{n+1}^k}} e^{i \eta \zeta_{1,\mathbf{A}}(\mathbf{b}_1) \dots \zeta_{k,\mathbf{A}}(\mathbf{b}_k) } \Bigg| $$
$$ \quad \quad  \quad \quad  \quad \quad  \quad \quad + e^{- \varepsilon_0 n/2} +  e^{-\delta_1(\varepsilon) n} + e^{\varepsilon \beta n} \left( e^{- \lambda \alpha n} +  \kappa^{\alpha n} + e^{-(\alpha \lambda-\varepsilon_0)n/2} +  e^{- \varepsilon_0 \delta_{up}n/6} \right) $$
by Proposition 4.4.2. Since every error term already enjoys exponential decay in $n$, we just have to deal with the sum of exponentials. Say that we are able to show an estimate like:
$$ e^{- \lambda \delta (k+1) n} \# \{ \mathbf{A} \in \mathcal{R}_{n+1}^{k+1} \ , \ \text{the sum-product phenomenon doesn't apply for} (\zeta_{j,\mathbf{A}})_j\} \leq \rho^n, $$
for some $\rho \in (0,1)$. Then, by Theorem 4.5.1, we can write for all $\mathbf{A}$ such that the sum-product phenomenon applies:
$$ \sup_{\eta \in J_n} \Bigg{|} \underset{\mathbf{A} \leftrightarrow \mathbf{B}}{\sum_{\mathbf{B} \in \mathcal{R}_{n+1}^k}} e^{i \eta \zeta_{1,\mathbf{A}}(\mathbf{b}_1) \dots \zeta_{k,\mathbf{A}}(\mathbf{b}_k) } \Bigg| \leq c e^{\lambda k \delta n} e^{- \varepsilon_0 \varepsilon_1 n/2 } ,$$
and hence we get
$$ e^{-\lambda \delta (2k+1) n} \sum_{\mathbf{A} \in \mathcal{R}_{n+1}^{k+1}(\omega)} \sup_{\eta \in J_n} \Bigg{|} \underset{\mathbf{A} \leftrightarrow \mathbf{B}}{\sum_{\mathbf{B} \in \mathcal{R}_{n+1}^k}} e^{ i  \eta \zeta_{1,\mathbf{A}}(\mathbf{b}_1) \dots \zeta_{k,\mathbf{A}}(\mathbf{b}_k)} \Bigg| $$ $$ \lesssim  e^{ \varepsilon \beta n} \rho^n + e^{\varepsilon \beta n} e^{-\lambda \delta  (2k+1) n} e^{\lambda \delta (k+1) n} e^{\lambda \delta k n} e^{- \varepsilon_0 \varepsilon_1 n/2} \lesssim  e^{\varepsilon \beta n} \Big(\rho^n + e^{-\varepsilon_0 \varepsilon_1 n/2} \Big) .$$
Now, we see that we can choose $\varepsilon$ small enough so that all terms enjoy exponential decay in $n$, and since $ |\xi| \simeq e^{\left( (2k+1) \lambda + \varepsilon_0 \right) n} $, we have proved polynomial decay of $|\widehat{\psi_* (\chi d\mu)}|^2$. So, to conclude the proof of Theorem 4.1.5 (resp. 4.1.8), we have to show that the sum-product phenomenon can be applied often under the condition (QNL) (resp. (B) and (LNL)).

\section{The non-concentration estimates under (QNL)}

This section is devoted to the proof of the non-concentration hypothesis that we just used under (QNL). The strategy is similar to the one used in the previous Chapters.

\begin{definition}
For a given $\mathbf{A} \in \mathcal{R}_{n+1}^{k+1}(\omega)$, define for $j=1, \dots, k$  $$ \mathcal{Z}_j := \{ \mathbf{b} \in \mathcal{R}_{n+1} , \ \mathbf{a}_{j-1} \rightsquigarrow \mathbf{b} \rightsquigarrow \mathbf{a}_j \  \} $$
Then define $$\zeta_{j,\mathbf{A}}(\mathbf{b}) := e^{2 \lambda n} |g_{\mathbf{a}_{j-1}' \mathbf{b}}'(x_{{\mathbf{a}}_j})|$$ on $\mathcal{Z}_j$. The following is satisfied, for some fixed constant $\beta>0$:
 $$\# \mathcal{Z}_j \leq e^{\varepsilon \beta n} e^{ \delta \lambda n} $$
 and
$$ \zeta_{j,\mathbf{A}}( \mathcal{Z}_j ) \subset \left[ e^{- \varepsilon \beta n} , e^{\varepsilon \beta n} \right] .$$
\end{definition}
Denote further $R:=e^{\varepsilon \beta n}$ and $N:= e^{\lambda \delta n} $. The goal of this section is to prove the next proposition. Once done, Theorem 4.1.5 will follow by our previous discussion.

\begin{proposition}
Under (QNL), there exists $\gamma>0$ such that, for $\varepsilon_0$ and $k$ given by Theorem 4.5.1, the following hold: there exists $\rho_3 \in (0,1)$ such that:
$$ \#\Big\{ \mathbf{A} \in \mathcal{R}_{n+1}^{k+1} \Big| \ \exists \sigma \in [e^{-4 \varepsilon_0 n}, e^{- \varepsilon_0 \varepsilon_1 n /2}], \ \exists j \in \llbracket 1,k \rrbracket, \ \# \{\mathbf{b} , \mathbf{c} \in \mathcal{Z}^2_j , \ |\zeta_{j,\mathbf{A}}(\mathbf{b}) - \zeta_{j,\mathbf{A}}(\mathbf{c})| \leq \sigma \} \geq N^2 \sigma^{\gamma} \Big\} $$ $$ \lesssim N^{k+1}  \rho_3^n. $$
\end{proposition}

This will be done in a succession of reductions. The first reduction follow an idea of Chapter 2 (already found in \cite{BD17}): using Markov's inequality to reduce this estimate to a bound on some expected value. 

\begin{lemma}
Suppose that there exists $\gamma>0$ such that, for $k\geq 0$, $\varepsilon_0>0$ given by Theorem 4.5.1, the following hold: for all $ \sigma \in [e^{-4\varepsilon_0 n-1}, e^{-\varepsilon_0 \varepsilon_1 n/2 +1}]$,
$$ \ \#\Big\{ (\mathbf{a},\mathbf{b},\mathbf{c},\mathbf{d}) \in \widehat{\mathcal{R}}_{n+1}^4 \ \Big| \  \ \big|e^{-2 \lambda n} |g_{\mathbf{a}'\mathbf{b}}'(x_{\mathbf{d}})| - e^{-2 \lambda n} |g_{\mathbf{a}'\mathbf{c}}'(x_{\mathbf{d}})|\big| \leq \sigma \Big\} \leq N^4 \sigma^{2 \gamma}, $$
where $\widehat{\mathcal{R}}_{n+1}^4 := \{ (\mathbf{a},\mathbf{b},\mathbf{c},\mathbf{d}) \in \mathcal{R}_{n+1}^4 \ | \ \mathbf{a} \rightsquigarrow \mathbf{b} \rightsquigarrow \mathbf{d}, \ \mathbf{a} \rightsquigarrow \mathbf{c} \rightsquigarrow \mathbf{d}\}$.\\

Then the conclusion of Proposition 4.6.2 holds.
\end{lemma}

\begin{proof}
We will denote $\mathcal{Z}_{\mathbf{a},\mathbf{d}} := \{ \mathbf{b} \in \mathcal{R}_{n+1} \ | \ \mathbf{a} \rightsquigarrow \mathbf{b} \rightsquigarrow \mathbf{d} \}$. To prove our lemma, we use a dyadic decomposition. For each integer $l \geq 0$ such that $\varepsilon_0 \varepsilon_1 n/2 - 1 \leq l \leq 4 \varepsilon_0 n +1$, notice that we can write, using Markov's inequality:
$$ N^{-2} \#\Big\{ (\mathbf{a},\mathbf{d}) \in \mathcal{R}_{n+1}^2 \Big| \ \#\{ (\mathbf{b},\mathbf{c}) \in \mathcal{Z}_{\mathbf{a},\mathbf{d}}^2, \ \big| e^{-2 \lambda n} |g_{\mathbf{a}'\mathbf{b}}'(x_{\mathbf{d}})| - e^{-2 \lambda n} |g_{\mathbf{a}'\mathbf{c}}'(x_{\mathbf{d}})| \big| \leq e^{-l} \} \geq N^2 e^{-\gamma (l+1)} \Big\}  $$
$$ \leq N^{-4} e^{\gamma (l+1)} \#\Big\{ (\mathbf{a},\mathbf{b},\mathbf{c},\mathbf{d}) \in \widehat{\mathcal{R}}_{n+1}^4 \ \Big| \  \ \big|e^{-2 \lambda n} |g_{\mathbf{a}'\mathbf{b}}'(x_{\mathbf{d}})| - e^{-2 \lambda n} |g_{\mathbf{a}'\mathbf{c}}'(x_{\mathbf{d}})| \big| \leq e^{-l} \Big\} \leq e^{- \gamma (l-1)}. $$
Now, using the fact that for all $\sigma \in [e^{-4 \varepsilon_0 n}, e^{-\varepsilon_0 \varepsilon_1 n/2}]$, there exists $l$ such as above satisfying $ e^{-(l+1)} \leq \sigma \leq e^{-l}$ yields:
$$ N^{-(k+1)} \#\Big\{ \mathbf{A} \in \mathcal{R}_{n+1}^{k+1} \Big| \ \exists \sigma \in [e^{-4 \varepsilon_0 n}, e^{- \varepsilon_0 \varepsilon_1 n /2}], \ \exists j \in \llbracket 1,k \rrbracket, \ \# \{\mathbf{b} , \mathbf{c} \in \mathcal{Z}^2_j , \ |\zeta_{j,\mathbf{A}}(\mathbf{b}) - \zeta_{j,\mathbf{A}}(\mathbf{c})| \leq \sigma \} \geq N^2 \sigma^{\gamma} \Big\}$$
$$ \leq k N^{-2} \#\Big\{ (\mathbf{a},\mathbf{d}) \in \mathcal{R}_{n+1}^2 \Big| \exists \sigma \in [e^{-4 \varepsilon_0 n}, e^{- \varepsilon_0 \varepsilon_1 n/2}] \ \#\{ (\mathbf{b},\mathbf{c}) \in \mathcal{Z}_{\mathbf{a},\mathbf{d}}^2, \ e^{-2 \lambda n} \big| |g_{\mathbf{a}'\mathbf{b}}'(x_{\mathbf{d}})| -|g_{\mathbf{a}'\mathbf{c}}'(x_{\mathbf{d}})| \big| \leq \sigma \} \geq N^2 \sigma^{\gamma} \Big\} $$
$$ \leq k \sum_{l = \lfloor \varepsilon_0 \varepsilon_1 n/2 \rfloor}^{ \lceil 4\varepsilon_0 n \rceil} N^{-2} \#\Big\{ (\mathbf{a},\mathbf{d}) \in \mathcal{R}_{n+1}^2 \Big| \ \#\{ (\mathbf{b},\mathbf{c}) \in \mathcal{Z}_{\mathbf{a},\mathbf{d}}^2, \ \big|e^{-2 \lambda n} |g_{\mathbf{a}'\mathbf{b}}'(x_{\mathbf{d}})| - e^{-2 \lambda n} |g_{\mathbf{a}'\mathbf{c}}'(x_{\mathbf{d}})| \big| \leq e^{-l} \} \geq N^2 e^{-\gamma (l+1)} \Big\}   $$
$$ \lesssim  \sum_{l = \lfloor \varepsilon_0 \varepsilon_1 n/2 \rfloor}^{ \lceil 4\varepsilon_0 n \rceil} e^{-\gamma l} \lesssim n \ e^{ - \varepsilon_0 \varepsilon_1 \gamma n /2} \lesssim \rho_3^n, $$
for some $\rho_3 \in (0,1)$.
\end{proof}

\begin{lemma}
Suppose that there exists $\gamma>0$ such that, for $k\geq 0$, $\varepsilon_1>0$ given by Theorem 4.5.1, the following hold: for all $\sigma \in [e^{-5 \varepsilon_0 n}, e^{-\varepsilon_1 \varepsilon_0 n/4}]$,
$$  \# \Big\{ (\mathbf{a},\mathbf{b},\mathbf{c},\mathbf{d}) \in \widehat{\mathcal{R}}_{n+1}^4  \Big| \  \big|S_{2n} \tau_F\big( g_{\mathbf{a}'\mathbf{b}}(x_\mathbf{d}) \big) - S_{2n} \tau_F\big( g_{\mathbf{a}'\mathbf{c}}(x_\mathbf{d}) \big) \big| \leq \sigma \Big\} \leq N^4 \sigma^{3 \gamma} .$$
(Recall that $\tau_F$ is defined in Definition 4.2.44.) Then the conclusion of Proposition 4.6.2 holds.
\end{lemma}

\begin{proof}
Suppose that the estimate is true.
Let $\sigma \in [e^{-4 \varepsilon_0 n-1}, e^{-\varepsilon_1 \varepsilon_0 n/2 +1}]$.
Our goal is to check the bound of Lemma 4.6.3. Since $g_{\mathbf{a}'\mathbf{b}}'(x_\mathbf{c})e^{- 2 \lambda n} \in [R^{-1},R]$ (with $R = e^{\epsilon \beta n}$), we have:
$$  \#\Big\{ (\mathbf{a},\mathbf{b},\mathbf{c},\mathbf{d}) \in \widehat{\mathcal{R}}_{n+1}^4 \ \Big| \  \ \big|e^{-2 \lambda n} |g_{\mathbf{a}'\mathbf{b}}'(x_{\mathbf{d}})| - e^{-2 \lambda n} |g_{\mathbf{a}'\mathbf{c}}'(x_{\mathbf{d}})| \big| \leq \sigma \Big\}   $$
$$ \leq \#\Big\{ (\mathbf{a},\mathbf{b},\mathbf{c},\mathbf{d}) \in \widehat{\mathcal{R}}_{n+1}^4  , \  \ \Big| \frac{|g_{\mathbf{a}'\mathbf{b}}'(x_{\mathbf{d}})|}{|g_{\mathbf{a}'\mathbf{c}}'(x_{\mathbf{d}})|} - 1 \Big| \leq R \sigma \Big\} .$$
Now, notice that $\ln(|g_{\mathbf{a}'\mathbf{b}}'(x_\mathbf{c})|) = S_{2n} \tau_F(g_{\mathbf{a}'\mathbf{b}}(x_\mathbf{c}))$. It follows that, for all $(\mathbf{a},\mathbf{b},\mathbf{c},\mathbf{d}) \in \widehat{\mathcal{R}}_{n+1}^4$,
$$ \Big| S_{2n} \tau_F(g_{\mathbf{a}'\mathbf{b}}(x_\mathbf{d})) - S_{2n} \tau_F(g_{\mathbf{a}'\mathbf{c}}(x_\mathbf{d})) \Big| = \Big|\ln \Big( \frac{|g_{\mathbf{a}'\mathbf{b}}'(x_{\mathbf{d}})|}{|g_{\mathbf{a}'\mathbf{c}}'(x_{\mathbf{d}})|} \Big) \Big| \leq R^2 \Big| \frac{|g_{\mathbf{a}'\mathbf{b}}'(x_{\mathbf{d}})|}{|g_{\mathbf{a}'\mathbf{c}}'(x_{\mathbf{d}})|} - 1 \Big|. $$
Hence:
$$ \#\Big\{ (\mathbf{a},\mathbf{b},\mathbf{c},\mathbf{d}) \in \widehat{\mathcal{R}}_{n+1}^4 \ \Big| \  \ \big|e^{-2 \lambda n} |g_{\mathbf{a}'\mathbf{b}}'(x_{\mathbf{d}})| - e^{-2 \lambda n} |g_{\mathbf{a}'\mathbf{c}}'(x_{\mathbf{d}})| \big| \leq \sigma \Big\}   $$
$$ \leq \# \Big\{ (\mathbf{a},\mathbf{b},\mathbf{c},\mathbf{d}) \in \widehat{\mathcal{R}}_{n+1}^4  \Big| \  \big|S_{2n} \tau_F\big( g_{\mathbf{a}'\mathbf{b}}(x_\mathbf{d}) \big) - S_{2n} \tau_F\big( g_{\mathbf{a}'\mathbf{c}}(x_\mathbf{d}) \big) \big| \leq R \sigma \Big\}
\leq N^4 (R^3 \sigma)^{3 \gamma} \leq N^4 \sigma^{2 \gamma}$$
for $n$ large enough, since $\varepsilon$ is chosen small enough. The conclusion follows from Lemma 4.6.3. \end{proof}

\begin{lemma}
Suppose that there exists $\gamma>0$ such that, for $k \geq 0$ and $\varepsilon_1$ given by Theorem 4.5.1, the following hold: for all $\sigma \in [e^{-5 \varepsilon_0 n}, e^{- \varepsilon_0 \varepsilon_1 n/5}]$, 
$$ \sum_{a \in \mathcal{A}}  \underset{\mathbf{c} \rightsquigarrow a}{\underset{\mathbf{b} \rightsquigarrow a}{\sum_{(\mathbf{b},\mathbf{c}) \in \mathcal{R}_{n+1}^2}}} \nu^{\otimes 2}\Big( (x,y) \in U_a^2 \Big| \ |S_{n} \tau_F ( g_\mathbf{b}(x)) -  S_{n} \tau_F( g_\mathbf{b}(y)) -  S_{n} \tau_F( g_\mathbf{c}(x)) +  S_{n} \tau_F( g_\mathbf{c}(y))| \leq \sigma \Big) \leq N^2 \sigma^{7 \gamma} $$
Then the conclusion of Proposition 4.6.2 holds.
\end{lemma}

\begin{proof}
Suppose that the previous estimate holds. We will check that Lemma 4.6.4 applies. Let $\sigma \in [e^{-5 \varepsilon_0 n},e^{-\varepsilon_0 \varepsilon_1 n/4}]$. First of all, notice that the family of maps $S_{2n} \tau_F \circ g_{\mathbf{a}'\mathbf{b}}$ are uniformly Hölder regular. In particular, there exists $C \geq 1$ such that, for all $\mathbf{a},\mathbf{b},\mathbf{d}$:
$$ \forall x \in U_{\mathbf{d}}, \ |S_{2n} \tau_F(g_{\mathbf{a}'\mathbf{b}}(x_\mathbf{d})) - S_{2n} \tau_F(g_{\mathbf{a}'\mathbf{b}}(x)) | \leq C \text{diam}(U_\mathbf{d})^{\alpha n} \leq C \kappa^{\alpha n}. $$
If the constant $\varepsilon_0$ is chosen sufficiently small so that $\kappa^{\alpha n} \leq e^{- 5 \varepsilon_0 n}$, then we find
$$ \forall x \in U_{\mathbf{d}}, \ |S_{2n} \tau_F(g_{\mathbf{a}'\mathbf{b}}(x_\mathbf{d})) - S_{2n} \tau_F(g_{\mathbf{a}'\mathbf{b}}(x)) | \leq C \sigma ,$$
which implies
$$ \nu(U_{\mathbf{d}}) \mathbb{1}_{[-\sigma,\sigma]}\Big( S_{2n} \tau_F(g_{\mathbf{a}'\mathbf{b}}(x_\mathbf{d})) - S_{2n} \tau_F(g_{\mathbf{a}'\mathbf{c}}(x_\mathbf{d})) \Big) $$ $$ \leq R \int_{U_\mathbf{d}} \mathbb{1}_{[-3 C \sigma,3 C \sigma]}(S_{2n} \tau_F(g_{\mathbf{a}'\mathbf{b}}(x)) - S_{2n} \tau_F(g_{\mathbf{a}'\mathbf{c}}(x)) ) d\nu(x) .$$
Using the fact that $\nu(U_{\mathbf{d}}) \sim N^{-1}$ from Lemma 4.3.3 and summing the previous estimates yields

$$ N^{-4} \# \Big\{ (\mathbf{a},\mathbf{b},\mathbf{c},\mathbf{d}) \in \widehat{\mathcal{R}}_{n+1}^4  \Big| \  \big|S_{2n} \tau_F\big( g_{\mathbf{a}'\mathbf{b}}(x_\mathbf{d}) \big) - S_{2n} \tau_F\big( g_{\mathbf{a}'\mathbf{c}}(x_\mathbf{d}) \big) \big| \leq \sigma \Big\} $$
$$ \leq R N^{-3} \sum_{\mathbf{a},\mathbf{b},\mathbf{c}} \sum_{\mathbf{d}} \nu(U_\mathbf{d})  \mathbb{1}_{[-\sigma,\sigma]}\Big( S_{2n} \tau_F(g_{\mathbf{a}'\mathbf{b}}(x_\mathbf{d})) - S_{2n} \tau_F(g_{\mathbf{a}'\mathbf{c}}(x_\mathbf{d})) \Big)   $$
$$ \leq R N^{-3} \sum_{d \in \mathcal{A}} \underset{\mathbf{a} \rightsquigarrow \mathbf{c} \rightsquigarrow d}{\underset{\mathbf{a} \rightsquigarrow \mathbf{b} \rightsquigarrow d}{\sum_{\mathbf{a},\mathbf{b},\mathbf{c}}}} \nu\Big( x \in U_d, \ |S_{2n} \tau_F(g_{\mathbf{a}'\mathbf{b}}(x)) - S_{2n} \tau_F(g_{\mathbf{a}'\mathbf{c}}(x))| \leq 3C \sigma \Big).$$
Now, the idea is to use Cauchy-Schwarz inequality to make the $y$-variable appear. This technique can be found in a paper of Tsujii \cite{Ts01} (I don't know if it appeared before). We have:
$$ \Big( N^{-4} \# \Big\{ (\mathbf{a},\mathbf{b},\mathbf{c},\mathbf{d}) \in \widehat{\mathcal{R}}_{n+1}^4  \Big| \  \big|S_{2n} \tau_F\big( g_{\mathbf{a}'\mathbf{b}}(x_\mathbf{d}) \big) - S_{2n} \tau_F\big( g_{\mathbf{a}'\mathbf{c}}(x_\mathbf{d}) \big) \big| \leq \sigma \Big\} \Big)^2 $$
$$ \lesssim R^2 N^{-3} \sum_{d \in \mathcal{A}} \underset{\mathbf{a} \rightsquigarrow \mathbf{c} \rightsquigarrow d}{\underset{\mathbf{a} \rightsquigarrow \mathbf{b} \rightsquigarrow d}{\sum_{\mathbf{a},\mathbf{b},\mathbf{c}}}} \nu\Big( x \in U_d, \ |S_{2n} \tau_F(g_{\mathbf{a}'\mathbf{b}}(x)) - S_{2n} \tau_F(g_{\mathbf{a}'\mathbf{c}}(x))| \leq 3C \sigma \Big)^2 $$
Now, notice that for any measurable function $h : U_d \rightarrow \mathbb{R}$ and for any interval $I \subset \mathbb{R}$ of length $|I|$, we can write:
$$ \nu(x \in U_d, \ h(x) \in I)^2 = \iint_{U_d^2} \mathbb{1}_{I}(h(x)) \mathbb{1}_{I}(h(y)) d\nu(x) d\nu(y) $$
$$ \leq \iint_{\mathbf{U}_d^2} \mathbb{1}_{[-2|I|,2|I|]}(h(x)-h(y)) d\nu(x) d\nu(y) $$ $$= \nu^{\otimes 2}( (x,y) \in U_d, \ |h(x)-h(y)| \leq 2 |I| ). $$
Applying this elementary estimate to our case yields
$$ \Big( N^{-4} \# \Big\{ (\mathbf{a},\mathbf{b},\mathbf{c},\mathbf{d}) \in \widehat{\mathcal{R}}_{n+1}^4  \Big| \  \big|S_{2n} \tau_F\big( g_{\mathbf{a}'\mathbf{b}}(x_\mathbf{d}) \big) - S_{2n} \tau_F\big( g_{\mathbf{a}'\mathbf{c}}(x_\mathbf{d}) \big) \big| \leq \sigma \Big\} \Big)^2  $$
$$  \lesssim R^2 N^{-3} \sum_{d \in \mathcal{A}} \underset{\mathbf{a} \rightsquigarrow \mathbf{c} \rightsquigarrow d}{\underset{\mathbf{a} \rightsquigarrow \mathbf{b} \rightsquigarrow d}{\sum_{\mathbf{a},\mathbf{b},\mathbf{c}}}} \nu^{\otimes 2}\Big( (x,y) \in U_d^2, \ |H_{2n}(\mathbf{a},\mathbf{b},\mathbf{c},x,y)| \leq 6C \sigma \Big) $$
where $H_{2n}(\mathbf{a},\mathbf{b},\mathbf{c},x,y) := S_{2n} \tau_F(g_{\mathbf{a}'\mathbf{b}}(x)) - S_{2n} \tau_F(g_{\mathbf{a}'\mathbf{c}}(x)) - S_{2n} \tau_F(g_{\mathbf{a}'\mathbf{b}}(y)) + S_{2n} \tau_F(g_{\mathbf{a}'\mathbf{c}}(y))$. To conclude, we will show that this expression is close to another one that doesn't depend on $\mathbf{a}$. Notice that we have:
$$ S_{2n} \tau_F(g_{\mathbf{a}'\mathbf{b}}(x)) - S_{2n} \tau_F ( g_{\mathbf{a}' \mathbf{b}}(y)) =  (S_{n} \tau_F \circ g_\mathbf{a})(g_\mathbf{b}(x)) - (S_{n} \tau_F \circ g_{\mathbf{a}}) (g_\mathbf{b}(y)) + S_n \tau_F(g_\mathbf{b}(x)) - S_n \tau_F(g_\mathbf{b}(y)). $$
$$ = S_n \tau_F(g_\mathbf{b}(x)) - S_n \tau_F(g_\mathbf{b}(y)) + O(\sigma) $$
since $S_n \tau_F \circ g_\mathbf{a}$ are uniformly Hölder maps, and $\varepsilon_0$ is chosen sufficiently small. Hence $$ \Big|  S_{n} \tau_F ( g_\mathbf{b}(x)) -  S_{n} \tau_F( g_\mathbf{b}(y)) -  S_{n} \tau_F( g_\mathbf{c}(x)) +  S_{n} \tau_F( g_\mathbf{c}(y)) - H_{2n}(\mathbf{a},\mathbf{b},\mathbf{c},x,y) \Big| \leq 2 C \sigma, $$
and we finally get:
$$ N^{-4} \# \Big\{ (\mathbf{a},\mathbf{b},\mathbf{c},\mathbf{d}) \in \widehat{\mathcal{R}}_{n+1}^4  \Big| \  \big|S_{2n} \tau_F\big( g_{\mathbf{a}'\mathbf{b}}(x_\mathbf{d}) \big) - S_{2n} \tau_F\big( g_{\mathbf{a}'\mathbf{c}}(x_\mathbf{d}) \big) \big| \leq \sigma \Big\} $$
$$ \lesssim R \Big( N^{-3} \sum_{d \in \mathcal{A}} \underset{\mathbf{a} \rightsquigarrow \mathbf{c} \rightsquigarrow d}{\underset{\mathbf{a} \rightsquigarrow \mathbf{b} \rightsquigarrow d}{\sum_{\mathbf{a},\mathbf{b},\mathbf{c}}}} \nu^{\otimes 2}\Big( (x,y) \in U_d^2, \ |H_{2n}(\mathbf{a},\mathbf{b},\mathbf{c},x,y)| \leq 6C \sigma \Big) \Big)^{1/2} $$
$$ \lesssim R \Big( N^{-2} \sum_{d \in \mathcal{A}} \underset{\mathbf{a} \rightsquigarrow \mathbf{c} \rightsquigarrow d}{\underset{\mathbf{a} \rightsquigarrow \mathbf{b} \rightsquigarrow d}{\sum_{\mathbf{a},\mathbf{b},\mathbf{c}}}} \nu^{\otimes 2}\Big( (x,y) \in U_d^2, \ |S_{n} \tau_F ( g_\mathbf{b}(x)) -  S_{n} \tau_F( g_\mathbf{b}(y)) -  S_{n} \tau_F( g_\mathbf{c}(x)) +  S_{n} \tau_F( g_\mathbf{c}(y))| \leq 8 C \sigma \Big) \Big)^{1/2} $$
$$ \leq R (8C \sigma)^{7\gamma/2} \leq \sigma^{3 \gamma} $$
for $n$ large enough. Hence Lemma 4.6.4 applies, and the conclusion of Proposition 4.6.2 holds.
\end{proof}

Our last step to establish Proposition 4.6.2 is to properly recognise that the expression that we just introduced is actually (close to) $\Delta$.

\begin{lemma}
For $a_0 \in \mathcal{A}$, and $p,q \in \widehat{R}_{a_0}$, recall that $\Delta(p,q) \in \mathbb{R}$ is defined by
$$ \Delta(p,q) := \sum_{n \in \mathbb{Z}} \Big( \tau_f(f^n p) - \tau_f(f^n [p,q]) - \tau_f(f^n [q,p]) + \tau_f(f^n q) \Big). $$
Then, if we denote $x := \pi(p) \in U_{a_0}$, $y := \pi(q) \in U_{a_0}$, and if we let $(a_k)_{k \geq 0}$ and $(b_k) \in \mathcal{A}^\mathbb{N}$ be defined by $$\forall k \geq 0, \ f^{-k}(p) \in R_{a_k} , f^{-k}(q) \in R_{b_k} , $$
then:
$$ S_{n} \tau_F ( g_\mathbf{a}(x)) -  S_{n} \tau_F( g_\mathbf{a}(y)) -  S_{n} \tau_F( g_\mathbf{b}(x)) +  S_{n} \tau_F( g_\mathbf{b}(y)) = \Delta(p,q) + O(\sigma)  $$
for all $\sigma \in [e^{-5 \varepsilon_0 n}, e^{- \varepsilon_0 \varepsilon_1 n/5}]$.
\end{lemma}

\begin{proof}
First of all, recall that, by the proof of Lemma 4.2.46, we know that $\tau_f$ and $\tau_F \circ \pi$ are cohomologous. Let $\theta : \mathcal{R} \rightarrow \mathbb{R}$ be a Hölder map such that $\tau_f = \tau_F \circ \pi + \theta \circ f - \theta$ on $\mathcal{R}$. It is then easy to check that, for all $p,q \in R_a$ for some $a \in \mathcal{A}$,
$$ \Delta(p,q) = 
 \sum_{n \in \mathbb{Z} } \Big( \tau_F( \pi f^n p) - \tau_F(\pi f^n [p,q]) - \tau_F(\pi f^n [q,p]) + \tau_F(\pi f^n q) \Big) .$$
We then notice that, for nonnegative $n$, we have $\pi(f^n p) = \pi(f^n([p,q]))$ and $\pi(f^n q) = \pi(f^n([q,p]))$. It follows that
$$ \Delta(p,q) = \sum_{k=1}^\infty \Big( \tau_F(\pi f^{-k} p) - \tau_F(\pi f^{-k} [p,q]) - \tau_F(\pi f^{-k} [q,p]) + \tau_F(\pi f^{-k} q) \Big). $$
Now, we notice that $F^k( \pi f^{-k} p ) = \pi(p) =x$, and since $f^{-j}(p) \in R_{a_j}$, we find $\pi f^{-k}(p) = g_{a_k \dots a_0}(x)$. A similar computation for the other points gives the expression
$$ \Delta(p,q) = \sum_{k=1}^\infty \Big( \tau_F(g_{a_k \dots a_0} x) - \tau_F(g_{a_k \dots a_0} y) - \tau_F(g_{b_k \dots b_0} x) + \tau_F(g_{b_k \dots b_0} y) \Big) .$$
Now, since $\tau_F$ is Hölder, we find again
$$ \Big| \tau_F(g_{a_k \dots a_0} x) - \tau_F(g_{a_k \dots a_0} y) - \tau_F(g_{b_k \dots b_0} x) + \tau_F(g_{b_k \dots b_0} y) \Big| \leq  \|\tau_F \|_{C^\alpha} \big( \text{diam}(U_{a_k \dots a_0})^\alpha + \text{diam}(U_{b_k \dots b_0})^\alpha  \big) \lesssim \kappa^{\alpha k} ,$$
which gives, summing those bounds for all $k \geq n+1$:
$$ \Big|\sum_{k=n+1}^\infty \Big(\tau_F(g_{a_k \dots a_0} x) - \tau_F(g_{a_k \dots a_0} y) - \tau_F(g_{b_k \dots b_0} x) + \tau_F(g_{b_k \dots b_0} y) \Big) - \Delta(p,q)\Big| \lesssim \kappa^{\alpha n} \leq \sigma $$
for large enough $n$. Now we conclude noticing that
$$ S_{n} \tau_F ( g_\mathbf{a}(x)) -  S_{n} \tau_F( g_\mathbf{a}(y)) -  S_{n} \tau_F( g_\mathbf{b}(x)) + S_n \tau_F (g_\mathbf{c}(y)) $$ $$ = \sum_{k=1}^n \Big(\tau_F(g_{a_k \dots a_0} x) - \tau_F(g_{a_k \dots a_0} y) - \tau_F(g_{b_k \dots b_0} x) + \tau_F(g_{b_k \dots b_0} y) \Big). $$
\end{proof}

We are ready to prove Proposition 4.6.2.

\begin{proof}[Proof (of Proposition 4.6.2).]
We suppose (QNL). Our goal is to check that the estimates of Lemma 4.6.5 holds. Let us denote by $R^{a_0 \dots a_n} := f^{n}(R_{a_n}) \cap \dots f(R_{a_1}) \cap R_{a_0}$. We have, using our previous Lemma:

$$\mathbb{1}_{[-\sigma,\sigma]}\Big( S_{n} \tau_F ( g_\mathbf{a}(x)) -  S_{n} \tau_F( g_\mathbf{a}(y)) -  S_{n} \tau_F( g_\mathbf{b}(x)) +  S_{n} \tau_F( g_\mathbf{b}(y)) \Big)$$ $$ \leq \frac{1}{\mu(R^{\mathbf{a}}) \mu(R^{\mathbf{b}})} \iint_{R^{\mathbf{a}} \times R^{\mathbf{b}}} \mathbb{1}_{[-2 \sigma, 2 \sigma] }\Big( \Delta(p,q) \Big) d\mu^2(p,q). $$
Since the measure $\mu$ is $f$-invariant, we find that $\mu(R^{\mathbf{a}}) = \mu(f^{-n}(R^{\mathbf{a}})) = \nu(U_{\mathbf{a}}) \sim N^{-1}$. Integrating in $d\nu^{\otimes 2}(x,y)$ yields
$$ \nu^{\otimes 2}\Big( (x,y) \in U_a^2 \Big| \ |S_{n} \tau_F ( g_\mathbf{b}(x)) -  S_{n} \tau_F( g_\mathbf{b}(y)) -  S_{n} \tau_F( g_\mathbf{c}(x)) +  S_{n} \tau_F( g_\mathbf{c}(y))| \leq \sigma \Big)  $$
$$ \leq R^2 N^2 \iint_{R^{\mathbf{a}} \times R^{\mathbf{b}}} \mathbb{1}_{[-2 \sigma, 2 \sigma] }\Big( \Delta(p,q) \Big) d\mu^2(p,q), $$
and then summing those estimates gives
$$ \sum_{a \in \mathcal{A}}  \underset{\mathbf{c} \rightsquigarrow a}{\underset{\mathbf{b} \rightsquigarrow a}{\sum_{(\mathbf{b},\mathbf{c}) \in \mathcal{R}_{n+1}^2}}} \nu^{\otimes 2}\Big( (x,y) \in U_a^2 \Big| \ |S_{n} \tau_F ( g_\mathbf{b}(x)) -  S_{n} \tau_F( g_\mathbf{b}(y)) -  S_{n} \tau_F( g_\mathbf{c}(x)) +  S_{n} \tau_F( g_\mathbf{c}(y))| \leq \sigma \Big)  $$
$$ \leq R^2 N^2 \sum_{a \in \mathcal{A}}  \underset{\mathbf{c} \rightsquigarrow a}{\underset{\mathbf{b} \rightsquigarrow a}{\sum_{(\mathbf{b},\mathbf{c}) \in \mathcal{R}_{n+1}^2}}} \iint_{R^{\mathbf{a}} \times R^{\mathbf{b}}} \mathbb{1}_{[-2 \sigma, 2 \sigma] }\Big( \Delta(p,q) \Big) d\mu^2(p,q) $$
$$ = R^2 N^2 \sum_{a \in \mathcal{A}} \mu^{\otimes 2}\Big( (p,q) \in (R_a)^2, \  |\Delta(p,q)| \leq 2 \sigma \Big) \leq R^2 N^2 \sigma^{\gamma_{QNL}} . $$
where the last inequality comes from (QNL). We can finally fix $\gamma := \gamma_{QNL}/8$ and let $k,\varepsilon_1$ being given by Theorem 4.5.1. 
We choose $\varepsilon_0 := {\alpha |\ln(\kappa)|}/{8} $, and
we see that the estimates of Lemma 4.6.4 are satisfied, thus concluding the proof of Proposition 4.6.2. 
\end{proof}

We prove Theorem 4.1.4 by applying Proposition 4.6.2 to the computations of section 4.5. 

\section{The non-concentration estimates under (B) and (LNL)}

This section is devoted to study the easier case where $\Omega$ is \textbf{a bunched attractor} with codimension one stable foliation (for example, when $\Omega$ is a solenoid). We suppose that the bunching condition (B) holds and that the Lyapunov NonLinearity condition (LNL) holds. In this context, the important thing to notice is the following: the map $\tau_F : \mathcal{U} \rightarrow \mathbb{R}$ is $C^{1+\alpha}$, since the stable holonomy is going to be $C^{2+\alpha}$ \cite{Ha97}. This allows us to conclude following the ideas of either Chapter 2 (in an elementary fashion, under a \emph{total nonlinearity condition}) or to directly use already existing Dolgopyat's estimates. We choose the latter. To prove Theorem 4.1.8, we will check that the non-concentration hypothesis used in section 4.5 holds for all blocks $\mathbf{A}$. This is the content of the next proposition, which will be proved through a succession of reductions.

\begin{proposition}[non-concentration]
Suppose (B) and (LNL). There exists $\gamma>0$, and we can choose $\varepsilon_0>0$, such that the following holds. Let $\eta \in J_n$. Let $\mathbf{A} \in \mathcal{R}_{n+1}^{k+1}$. Then, if $n$ is large enough,
$$\forall \sigma \in [ |\eta|^{-2}, |\eta|^{- \varepsilon_1}], \quad  \# \left\{(\mathbf{b}, \mathbf{c}) \in \mathcal{Z}^2_j , \  |\zeta_{j,\mathbf{A}}(\mathbf{b})-\zeta_{j,\mathbf{A}}(\mathbf{c}) | \leq \sigma \right\} \leq N^2 \sigma^{\gamma}, $$
where $R:= e^{\varepsilon \beta n}$, $N:= e^{\lambda \delta n}$ and $\varepsilon_1,k$ are fixed by Corollary 4.5.1.
\end{proposition}

Our first step is to reduce our non-concentration estimate to a statement about Birkhoff sums.

\begin{lemma}

If $\varepsilon_0$ and $\gamma$ are such that, for $\sigma \in [ e^{- 5 \varepsilon_0 n} , e^{ - \varepsilon_1 \varepsilon_0 n/4 }  ]$, $$\sup_{a \in \mathbb{R}} \#  \left\{ \mathbf{b} \in \mathcal{Z}_j,  \ S_{2n} {\tau}_F \left( g_{\mathbf{a}_{j-1}' \mathbf{b}} (x_{\mathbf{a}_j})  \right) \in [a-\sigma,a+\sigma] \right\} \leq N \sigma^{2 \gamma} ,$$

then the conclusion of Proposition 4.7.1 holds.
\end{lemma}

\begin{proof}
Suppose that the estimate is true.
Let $|\eta| \in [e^{\varepsilon_0 n/2}, e^{2 \varepsilon_0 n}] $, and then let $\sigma \in [ |\eta|^{-2}, |\eta|^{-\varepsilon_1} ] \subset [e^{-4 \varepsilon_0 n}, e^{- \varepsilon_0 \varepsilon_1 n/2}].$  Let $a \in [R^{-1},R]$ (it is enough to conclude). Since for $n$ large enough $$\ln(a+\sigma)-\ln(a-\sigma) = \ln(1+ \sigma a^{-1}) - \ln(1 - \sigma a^{-1}) \leq 4 \sigma a^{-1} \leq 4 \sigma R ,$$ 
we find that
$$ \ln\left( [a-\sigma,a+\sigma] \right) \subset [\ln a - 4 R \sigma, \ \ln a + 4 R \sigma ] . $$
Hence:
$$  \#  \{ \mathbf{b} \in \mathcal{Z}_j,  \ {\zeta}_{j,\mathbf{A}}(\mathbf{b}) \in [a-\sigma,a+\sigma] \} \leq  \#  \{ \mathbf{b} \in \mathcal{Z}_j,  \ \ln {\zeta}_{j,\mathbf{A}}(\mathbf{b}) \in [\ln a - 4 R \sigma, \ln a + 4 R \sigma]  \} $$
$$ = \#  \{ \mathbf{b} \in \mathcal{Z}_j,  \ S_{2n} {\tau}_F \left( g_{\mathbf{a}_{j-1}' \mathbf{b}} (x_{\mathbf{a}_j})  \right)  \in \left[ -\ln a + 2 n \lambda - 4 \sigma R , -\ln a + 2 n \lambda + 4 \sigma R \right] \} $$
$$ \leq N (4 R \sigma)^{2 \gamma} \leq R^{-1} N \sigma^{\gamma} ,$$
since $4R \sigma \in [e^{-5 \varepsilon_0 n}, e^{-\varepsilon_1 \varepsilon_0 n/4}]$ and $4^{2\gamma} R^{2 \gamma + 1} \sigma^\gamma \leq 1$ for large enough $n \geq 1$. Finally,

$$ \# \left\{(\mathbf{b}, \mathbf{c}) \in \mathcal{Z}^2_j , \  |{\zeta}_{j,\mathbf{A}}(\mathbf{b})-{\zeta}_{j,\mathbf{A}}(\mathbf{c}) | \leq \sigma \right\}  \quad \quad \quad \quad \quad \quad \quad \quad \quad \quad \quad $$ $$ \quad \quad \quad \quad \quad = \sum_{\mathbf{c} \in \mathcal{Z}_j}   \left\{ \mathbf{b} \in \mathcal{Z}_j,  \ {\zeta}_{j,\mathbf{A}}(\mathbf{b}) \in [\zeta_{j,\mathbf{A}}(\mathbf{c})-\sigma,\zeta_{j,\mathbf{A}}(\mathbf{c})+\sigma] \right\} \leq N^2 \sigma^{\gamma}, $$
and so the conclusion of Proposition 4.7.1 holds. \end{proof}
The version of the Dolgopyat's estimates that we will use is taken from \cite{DV21} (Theorem 6.4). 

\begin{theorem}
Define, for $s \in \mathbb{C}$, a twisted transfer operator $\mathcal{L}_{s} : C^\alpha(\mathcal{U},\mathbb{C}) \rightarrow C^\alpha(\mathcal{U},\mathbb{C}) $ as follows:

$$ \forall x \in U_b, \ \mathcal{L}_{s} h (x) := \sum_{a \rightarrow b} e^{(\varphi+s \tau_F)(g_{ab}(x))} h(g_{ab}(x)) .$$
Iterating this transfer operator yields:
$$ \forall x \in U_b, \ \mathcal{L}_{s}^n h (x) = \underset{\mathbf{a} \rightsquigarrow b}{\sum_{\mathbf{a} \in \mathcal{W}_{n+1}}} w_{\mathbf{a}}(x) e^{s S_{n} {\tau}_F(g_{\mathbf{a}}(x)) } h(g_{\mathbf{a}}(x)) .$$
Under our nonlinearity condition (LNL) and the bunching condition (B), the following holds. There exists $\rho \in \mathbb{N}$ such that, for any $s \in \mathbb{C}$ such that $\text{Re}(s)=0$ and $|\text{Im}(s)|>\rho$,

$$ \forall h \in C^\alpha(\mathcal{U},\mathbb{C}), \ \forall n \geq 0, \ \| \mathcal{L}_{s}^n h \|_{L^\infty(\mathcal{U})} \leq \rho |\text{Im}(s)|^{\rho} e^{-n/\rho} \|h\|_{C^\alpha(\mathcal{U},\mathbb{C})} .$$
\end{theorem}

\begin{remark}
This theorem is stated in \cite{DV21} under a \emph{total nonlinearity condition} made on $\tau_F$ (See Section 2.5 for details on this condition and the link with \say{Uniform Non Integrability} in this $C^{1+\alpha}$ context). The condition goes as follow: there exists no locally constant map  $\mathfrak{c} : \mathcal{U}^{(1)} \rightarrow \mathbb{R}$  and  $\theta \in C^1(\mathcal{U}^{(1)},\mathbb{R})$  such that  $$\tau_F = \mathfrak{c} - \theta \circ F + \theta. $$

This condition is satisfied in our setting. Indeed, suppose that some locally constant $\mathfrak{c} : \mathcal{U}^{(1)} \rightarrow \mathbb{R}$ satisfy $ \tau_F = \mathfrak{c} + \theta \circ F - \theta $ for some $\mathcal{C}^1$ map $\theta$. Then, recall that $\tau_F \circ \pi$ and $\tau_f$ are $f$-cohomologous (Lemma 4.2.46). Hence, if $x \in \Omega_{per}$ is a $f$-periodic point with period $n_x$, the unstable Lyapunov exponent of the associated periodic orbit is $$ \widehat{\lambda}(x) = \frac{1}{n_x} S_{n_x} \tau_F(\pi(x)) \in \text{Span}_\mathbb{Q} \left( \mathfrak{c}(\mathcal{U})\right) , $$
which implies that
$$ \text{dim}_\mathbb{Q} \text{Vect}_\mathbb{Q} \widehat{\lambda}\left( \Omega_\text{per} \right) \leq  \text{dim}_\mathbb{Q}  \text{Span}_\mathbb{Q} \left( \mathfrak{c}(\mathcal{U})\right) , $$
and this is an obvious contradiction to the nonlinearity hypothesis (LNL). \\
\end{remark}

\begin{lemma}
Define $ \varepsilon_0 := \min\left( 1/(5 \rho (3+\rho)), \alpha \lambda/8  \right)$. Fix $\gamma:=1/4$. Then Theorem 4.5.1 gives constants $\varepsilon_1$ and $k$ associated with this choice of $\gamma$. For $\sigma \in [e^{-5 \varepsilon_0 n} , e^{ - \varepsilon_1 \varepsilon_0 n/4 }  ]$ and if $n$ is large enough, $$ \sup_{a \in \mathbb{R}} \#  \left\{ \mathbf{b} \in \mathcal{Z}_j,  \ S_{2n} {\tau}_F \left( g_{\mathbf{a}_{j-1}' \mathbf{b}} (x_{\mathbf{a}_j})  \right) \in [a-\sigma,a+\sigma] \right\} \leq N \sigma^{1/2} .$$

\end{lemma}

\begin{proof}
In the proof to come, all the $\simeq$ or $\lesssim$ will be uniform in $a$: the only relevant information is $\sigma$.
So fix $\sigma \in [ e^{- 5 \varepsilon_0 n} , e^{ - \varepsilon_1 \varepsilon_0 n/4 }  ]$, and fix an interval of length $\sigma$, $[a-\sigma,a+\sigma]$. Then choose a bump function $\chi$ such that $\chi = 1$ on $[a-\sigma,a+\sigma]$, $\text{supp}(\chi) \subset [a-2\sigma,a+2\sigma]$ and such that $ \| \chi \|_{L^1(\mathbb{R})} \simeq \sigma $. We can suppose that $ \| \chi^{(l)} \|_{L^1(\mathbb{R})} \simeq_l \sigma^{1-l} $.  \\

Then, we can consider $h$, the $2 \pi \mathbb{Z}$ periodic map obtained by periodizing $\chi$. This will allows us to use Fourier series. By construction, we see that
$ \mathbb{1}_{ [a-\sigma, a+\sigma] } \leq h .$
Hence:

$$ \#  \left\{ \mathbf{b} \in \mathcal{Z}_j,  \ S_{2n} {\tau}_F ( g_{\mathbf{a}_{j-1}' \mathbf{b}} (x_{\mathbf{a}_j})  ) \in [a-\sigma,a+\sigma] \right\} \leq \sum_{\mathbf{b} \in \mathcal{Z}_j} h\left( S_{2n} {\tau}_F ( g_{\mathbf{a}_{j-1}' \mathbf{b}} (x_{\mathbf{a}_j})  ) \right) $$
$$ \leq R N{\sum_{\mathbf{b} \in \mathcal{Z}_{j}}} w_{ \mathbf{b}}(x_{\mathbf{a}_j}) h\left( S_{2n} {\tau}_F ( g_{\mathbf{a}_{j-1}' \mathbf{b}} (x_{\mathbf{a}_j}) ) \right) $$
$$ \leq R N \underset{\mathbf{a}_{j-1} \rightsquigarrow \mathbf{b} \rightsquigarrow \mathbf{a}_j}{\sum_{\mathbf{b} \in \mathcal{W}_{n+1}}} w_{ \mathbf{b}}(x_{\mathbf{a}_j})  h\left( S_{2n} \tau_F ( g_{\mathbf{a}_{j-1}' \mathbf{b}} (x_{\mathbf{a}_j}) ) \right) .$$
Then we develop $h$ using Fourier series. We have:
$$ \forall x \in \mathbb{R}, \ h(x) = \sum_{ \mathfrak{n} \in \mathbb{Z} } c_{ \mathfrak{n} }(h) e^{i  \mathfrak{n} x} ,$$
where $$c_{\mathfrak{n}}(h) := (2 \pi)^{-1} \int_{a-\pi}^{a+\pi} h(x) e^{-i  \mathfrak{n} x} dx .$$ 
Notice that $$  \mathfrak{n}^{l} |c_{ \mathfrak{n}}(h)| \simeq |c_{ \mathfrak{n}}(h^{(l)})| \lesssim  \sigma^{ 1 - l} .$$
Plugging $ S_{2n} \tau_F ( g_{\mathbf{a}_{j-1}' \mathbf{b}} (x_{\mathbf{a}_j}) ) $ in this expression yields
$$ h\left( S_{2n} \tau_F ( g_{\mathbf{a}_{j-1}' \mathbf{b}} (x_{\mathbf{a}_j}) ) \right) = \sum_{ \mathfrak{n} \in \mathbb{Z} } c_{\mathfrak{n}}(h) e^{i \mathfrak{n} S_{2n} \tau_F ( g_{\mathbf{a}_{j-1}' \mathbf{b}} (x_{\mathbf{a}_j}) )} ,$$
and so
$$  \#  \left\{ \mathbf{b} \in \mathcal{Z}_j,  \ S_{2n} {\tau}_F ( g_{\mathbf{a}_{j-1}' \mathbf{b}} (x_{\mathbf{a}_j}) ) \in [a-\sigma,a+\sigma] \right\} $$
$$ \leq R N  \sum_{\mathfrak{n} \in \mathbb{Z}} c_{\mathfrak{n} }(h) \underset{\mathbf{a}_{j-1} \rightsquigarrow \mathbf{b} \rightsquigarrow \mathbf{a}_j}{\sum_{\mathbf{b} \in \mathcal{W}_{n+1}}} w_{ \mathbf{b}}(x_{\mathbf{a}_j}) e^{i \mathfrak{n} S_{2n} \tau_F ( g_{\mathbf{a}_{j-1}' \mathbf{b}} (x_{\mathbf{a}_j}) )} .$$
For any block $\mathbf{A} \in \mathcal{W}_{n+1}^{k+1}$, integer $j \in \{1, \dots, k\}$ and integer $\mathfrak{n} \in \mathbb{Z}$, define $\mathfrak h_{\mathbf{A},j}^{\mathfrak{n}} \in   C^\alpha(\mathcal{U},\mathbb{R})$ by
$$ \forall x \in U_{b(\mathbf{a}_{j-1})}, \ \mathfrak h_{\mathbf{A},j}^{\mathfrak{n}} (x) := \exp\left({ i \mathfrak{n} S_n \tau_F(g_{\mathbf{a}_{j-1}}(x))} \right) \quad , \quad \forall x \in U_b , \ b \neq b(\mathbf{a}_{j-1}), \  \mathfrak h_{\mathbf{A},j}^{\mathfrak{n}} (x) := 0 .$$
With this notation, we may rewrite the sum on $\mathbf{b}$ as follows:
$$  \underset{\mathbf{a}_{j-1} \rightsquigarrow \mathbf{b} \rightsquigarrow \mathbf{a}_j}{\sum_{\mathbf{b} \in \mathcal{W}_{n+1}}} w_{ \mathbf{b}}(x_{\mathbf{a}_j}) \exp\left({i \mathfrak{n} S_{2n} {\tau}_F ( g_{\mathbf{a}_{j-1}' \mathbf{b}} (x_{\mathbf{a}_j}) )} \right)$$
$$ = \underset{\mathbf{a}_{j-1} \rightsquigarrow \mathbf{b} \rightsquigarrow \mathbf{a}_j}{\sum_{\mathbf{b} \in \mathcal{W}_{n+1}}}    w_{\mathbf{b}}(x_{\mathbf{a}_j}) e^{i \mathfrak{n} S_n {\tau}_F( g_{\mathbf{b}}(x_{\mathbf{a}_j}) ) } \exp\left({ i \mathfrak{n} S_n {\tau}_F(g_{\mathbf{a}_{j-1}} g_{\mathbf{b}}(x_{\mathbf{a}_j}))} \right) $$
$$ = \underset{\mathbf{b} \rightsquigarrow \mathbf{a}_{j} }{\sum_{\mathbf{b} \in \mathcal{W}_{n+1}}} w_{\mathbf{b}}(x_{\mathbf{a}_j}) e^{i \mathfrak{n} S_n {\tau}_F( g_{\mathbf{b}}(x_{\mathbf{a}_j}) ) } \mathfrak h_{\mathbf{A},j}^{\mathfrak{n}} \left( g_{\mathbf{b}}(x_{\mathbf{a}_j}) \right) $$
$$ = \mathcal{L}_{i \mathfrak{n}}^n \left( \mathfrak{h}_{\mathbf{A},j}^{\mathfrak{n}} \right)(x_{\mathbf{a}_j}). $$
A direct computation allows us to estimate the $C^\alpha$ norm of $\mathfrak{h}_{\mathbf{A},j}^{\mathfrak{n}}$. We get, uniformly in $n$:
$$ \| \mathfrak{h}_{\mathbf{A},j}^{\mathfrak{n}} \|_{C^{\alpha}(\mathcal{U},\mathbb{R})} \lesssim (1+\mathfrak{n}) .$$

We can now break the estimate into two pieces: high frequencies are controlled by the contraction property of this transfer operator, and the low frequencies are controlled by the Gibbs property of $\nu$. We also use the estimates on the Fourier coefficients on $h$. 

$$ \#  \left\{ \mathbf{b} \in \mathcal{Z}_j,  \ S_{2n} {\tau}_F ( g_{\mathbf{a}_{j-1}' \mathbf{b}} (x_{\mathbf{a}_j}) ) \in [a-\sigma,a+\sigma] \right\}   $$
$$\leq R N  \sum_{\mathfrak{n} \in \mathbb{Z}} c_{\mathfrak{n} }(h) \underset{\mathbf{a}_{j-1} \rightsquigarrow \mathbf{b} \rightsquigarrow \mathbf{a}_j}{\sum_{\mathbf{b} \in \mathcal{W}_{n+1}}} w_{ \mathbf{b}}(x_{\mathbf{a}_j}) e^{i \mathfrak{n} S_{2n} \tau_F ( g_{\mathbf{a}_{j-1}' \mathbf{b}} (x_{\mathbf{a}_j}) )}  $$
$$ \leq R N \left( \sum_{|\mathfrak{n}| \leq \rho} |c_{\mathfrak{n}}(h)|  \sum_{\mathbf{b} \in \mathcal{W}_{n+1}}  w_{\mathbf{b}}(x_{\mathbf{a}_j})  +  \sum_{|\mathfrak{n}| > \rho} |c_{\mathfrak{n}}(h)| |\mathcal{L}_{i \mathfrak{n}}^n (\mathfrak{h}_{\mathbf{A},j})(x_{\mathbf{a}_j})| \right) $$
$$ \lesssim R N \left(  \sigma \sum_{\mathbf{b} \in \mathcal{W}_{n+1}} \nu(U_{\mathbf{b}})  + \sum_{|\mathfrak{n}| > \rho} |c_{\mathfrak{n}}(h)| \|\mathcal{L}_{ i \mathfrak{n}}^n (\mathfrak{h}_{\mathbf{A},j})\|_{L^\infty(\mathcal{U})} \right) $$
$$ \lesssim  R N \sigma +R N  \sum_{|\mathfrak{n} | > \rho} |c_{\mathfrak{n}}(h)| \mathfrak{n}^{\rho}  e^{-n/\rho} \| \mathfrak{h}_{\mathbf{A},j}^{\mathfrak{n}} \|_{C^\alpha(\mathfrak{U},\mathbb{C})} $$

$$ \lesssim  R N \sigma + R N e^{-n/\rho} \sum_{|\mathfrak{n}| > \rho} |c_{\mathfrak{n}}(h)| |\mathfrak{n}|^{\rho+3} |\mathfrak{n}|^{-2}    $$
$$ \leq C R N (\sigma + e^{-n/\rho} \sigma^{-(\rho+2)}) , $$
for some constant $C>0$ (using $\mathfrak{n}^{l} |c_{ \mathfrak{n}}(h)| \lesssim \sigma^{1-l}$ to get the last inequality).
We are nearly done. Since $\sigma \in  [ e^{-5 \varepsilon_0 n} , e^{ - \varepsilon_1 \varepsilon_0 n/4 }  ]$, we know that $ \sigma^{-(\rho+2)} \leq e^{5 (\rho+2) \varepsilon_0 n } $.
Now is the time where we fix $\varepsilon_0$: choose $$ \varepsilon_0 := \min\left( \frac{1}{5 \rho (3+\rho)}, \frac{\alpha \lambda}{8}  \right). $$
Then $ e^{-n/\rho} \sigma^{-(\rho+2)} \leq  e^{ ( -1/\rho + 5(\rho+2) \varepsilon_0 ) n } \leq  e^{ - 5 \varepsilon_0 n} \leq \sigma $ for $n$ large enough. Hence, we get
$$  \#  \left\{ \mathbf{b} \in \mathcal{Z}_j,  \ S_{2n} \tau_F ( g_{\mathbf{a}_{j-1}' \mathbf{b}} (x_{\mathbf{a}_j}) ) \in [a-\sigma,a+\sigma] \right\}  \leq 2 C R N \sigma . $$
Finally, since $\sigma^{1/2}$ is quickly decaying compared to $R$, we have
$$ 2 C R N \sigma \leq N \sigma^{1/2} $$
provided $n$ is large enough. The proof is done.
\end{proof}

\section{Appendix A: the nonlinearity condition (LNL) is generic}

\begin{theorem}

The condition $\text{dim}_\mathbb{Q} \text{Vect}_\mathbb{Q} \widehat{\lambda}\left( \Omega_\text{per} \right) = \infty$ is generic in the following sense: for any given Axiom A diffeomorphism $f : M \rightarrow M$ and a fixed basic set $\Omega$ that is not an isolated periodic orbit, a generic $C^k$ perturbation $\tilde{f}$ of $f$ have an hyperbolic set $\tilde{\Omega}$ on which the dynamic is conjugated with  $(f,\Omega)$, and $\text{dim}_\mathbb{Q} \text{Span}_\mathbb{Q} \widehat{\lambda}\left( \tilde{\Omega}_\text{per} \right) = \infty$.

\end{theorem}

\begin{proof}

Let $f:M \longrightarrow M$ be a $C^k$ Axiom A diffeomorphism, and fix $\Omega$ a basic set for $f$ that is not an isolated periodic orbit. Recall that this implies that $\Omega$ is infinite (and even perfect). Since $f$ is Axiom A, $\Omega_{\text{per}}$ is then infinite. Let $U \supset \Omega$ be a small open neighborhood in $M$. Consider a small enough open neighborhood around $f$ in the space of $C^k$ maps, $\mathfrak{U} \subset C^k(M,M)$. Then, there exists a map $$\Phi: \mathfrak{U} \rightarrow 2^M \times \mathcal{C}^0(\Omega,M) $$ 
such that for any $\tilde{f} \in \mathfrak{U}$, $\Phi(\tilde{f}) = (\tilde{\Omega},h)$ satisfies the following properties:

\begin{itemize}
     \item $\tilde{\Omega} \subset U$ is a hyperbolic set for $\tilde{f} : M \rightarrow M$,
    \item $\tilde{\Omega}_{\text{per}}$ is dense in $\tilde{\Omega}$,
    \item $h: \Omega \rightarrow \tilde{\Omega}$ is an homeomorphism and conjugates $(\Omega,f)$ with $(\tilde{\Omega},\tilde{f})$.
    \item The map $\tilde{f} \in \mathfrak{U} \mapsto h \in C^0(\Omega,M)$ is continuous.
\end{itemize}

This is Theorem 5.5.3 in \cite{BS02}. See also \cite{KH95}, page 571: \say{stability of hyperbolic sets}. \\
So let $\tilde{f} \in \mathfrak{U}$. Since $h$ conjugates $(f,\Omega)$ with $(\tilde{f},\tilde{\Omega})$, it is also a bijection between $\Omega_\text{per}$ and $\tilde{\Omega}_\text{per}$. For $x \in \Omega_\text{per}$, we write $\tilde{x} \in \tilde{\Omega}$ for $h(x)$. \\

We are ready to prove that the condition $$ \dim_\mathbb{Q} \text{Vect}_\mathbb{Q} \widehat{\lambda}\left( \tilde{\Omega}_{\text{per}} \right) = \infty$$
is generic in $\tilde{f} \in \mathfrak{U}$. The set where this condition holds condition may be rewritten in the following way:
$$ \left\{ \tilde{f} \in \mathfrak{U} \ | \ \forall N \geq 1, \exists \tilde{x_1}, \dots, \tilde{x_N} \in \tilde{\Omega}_{\text{per}} \ , \ \left(\widehat{\lambda}(\tilde{x}_1) , \dots, \widehat{\lambda}(\tilde{x}_N) \right) \ \text{is linearly independent over} \ \mathbb{Q} \right\}, $$
ie
$$ \bigcap_{N \geq 1} \bigcup_{ \Omega_{\text{per}}^N } \ \left\{ \tilde{f} \in \mathfrak{U} \  | \  \left(\widehat{\lambda}(\tilde{x}_1) , \dots, \widehat{\lambda}(\tilde{x}_N) \right) \ \text{is} \ \mathbb{Q}-\text{independent} \right\}. $$
Fix once and for all a sequence of distinct periodic orbits $(y_N)_{N \geq 1} \in \Omega_{\text{per}}$. Then:
$$ \left\{ f \in \mathfrak{U} \ | \  \dim_\mathbb{Q} \text{Vect}_\mathbb{Q} \widehat{\lambda}\left( \tilde{\Omega}_{\text{per}} \right) = \infty \right\} \supset \bigcap_{N \geq 1}  \left\{ \tilde{f} \in \mathfrak{U} \  | \  \left(\widehat{\lambda}(\tilde{y}_1) , \dots, \widehat{\lambda}(\tilde{y}_N) \right) \ \text{is} \ \mathbb{Q}-\text{independent} \right\}$$
$$ = \bigcap_{N \geq 1} \bigcap_{(m_1 , \dots, m_N) \in \mathbb{Z}^N \setminus \{ 0 \}} \left\{ \tilde{f} \in \mathfrak{U} \ | \ \sum_{i=1}^N m_i \widehat{\lambda}(\tilde{y}_i) \neq 0 \right\}, $$
and this is a countable intersection of dense open sets. Indeed, fix some $N \geq 1$ and some integers $(m_1, \dots, m_N) \in \mathbb{Z}^N \setminus\{0\}$, and denote $ \mathfrak{U}_{m_1,\dots,m_N} := \left\{ \tilde{f} \in \mathfrak{U} \ | \ \sum_{i=1}^N m_i \widehat{\lambda}(\tilde{y}_i) \neq 0 \right\} $. Without loss of generality, we may suppose that $m_1 \neq 0$.\\

First of all, $\mathfrak{U}_{m_1,\dots,m_N}$ is open. This is because the function 
$$ \begin{array}[t]{lrcl}
  & \mathfrak{U}  & \longrightarrow & \mathbb{R}^N \\
    & \tilde{f} & \longmapsto &  (\widehat{\lambda}(\tilde{y}_1),\dots, \widehat{\lambda}(\tilde{y}_N))  \end{array} $$
is continuous. Indeed, recall that $$\widehat{\lambda}(\tilde{y}_i) = \frac{1}{n_i} \sum_{k=0}^{n_i} \log | \partial_u \tilde{f} (\tilde{f}^k( \tilde{y}_i )) |, $$ where $n_i$ is the period of $y_i \in \Omega_{\text{per}}$ so is constant in $\tilde{f}$. Moreover, $\tilde{f}$ varies smoothly in $C^k$ norm, $\tilde{y}_i$ varies continuously, and finally the unstable direction of $\tilde{f}$ also varies continuously in $\tilde{f}$, see  Corollary 2.9 in \cite{CP15}. \\

Now, we check that $\mathfrak{U}_{m_1,\dots, m_N}$ is dense in $\mathfrak{U}$. Without loss of generality, it suffices to prove that $f \in \overline{\mathfrak{U}_{m_1, \dots, m_N}}$. If $f \in \mathfrak{U}_{m_1,\dots, m_N}$ then we have nothing to prove. So suppose $f \notin \mathfrak{U}_{m_1,\dots, m_N}$. We are going to construct a perturbation $\tilde{f}$ in $\mathfrak{U}_{m_1,\dots,m_N}$. \\

Notice that the set $$ \Lambda := \left\{ f^k(y_i) \ | \ k \in \mathbb{Z}, i \in \llbracket 1 , N \rrbracket \right\} \subset {\Omega}_{\text{per}} $$
is discrete and finite. So there exists a small open neighborhood $U$ of $y_1$ such that $U \cap \Lambda = \{y_1\}$. Then, choose $\varphi:M \rightarrow M$ a $C^k$ diffeomorphism of $M$ that is a small perturbation of the identity on $M$ such that $\varphi=Id$ on $M\setminus U$, $\varphi(y_1)=y_1$, $(d\varphi)_{y_1}(E^u(y_1)) = E^u(y_1) $ and such that $\partial_u \varphi(y_1) = 1+\varepsilon$ for some small $\varepsilon>0$. \\

Define $\tilde{f} := f \circ \varphi$. Then $\tilde{f}$ is a small perturbation of $f$. Moreover, the $y_i$ are periodic orbits for $\tilde{f}$. Since the conjugacy $h$ between $f$ and $\tilde{f}$ sends periodic orbits to periodic orbits of the same period, and since $h$ is a small perturbation of the identity in the $C^0$ topology, it follows necessarily that $\tilde{y_i} = h(y_i) = y_i$. Hence, the Lyapunov exponents of interests are:
$$ \widehat{\lambda}(\tilde{y}_1)  = \frac{1}{n_1} \log( | \partial_u f(y_1) (1+\varepsilon) | ) + \frac{1}{n_1} \sum_{k=1}^{n_1-1} \log |\partial_u \tilde{f} (f^k(y_1)) |  = \frac{1}{n_1} \log(1+\varepsilon) + \widehat{\lambda}(y_1) .$$
$$ \widehat{\lambda}(\tilde{y}_i) = \widehat{\lambda}(y_i), \forall i \in \llbracket 2,N \rrbracket. $$
Hence $$ \sum_{i=1}^n m_i  \widehat{\lambda}(\tilde{y}_i) = \frac{m_1}{n_1} \log(1+\varepsilon) + \sum_{i=1}^{N} m_i  \widehat{\lambda}(y_i) = \frac{m_1}{n_1} \log(1+\varepsilon) \neq 0, $$
since $f$ is supposed to be $\mathfrak{U} \setminus \mathfrak{U}_{n_1,\dots, n_N}$ and $m_1 \neq 0$. The proof is done. \end{proof}

\section{Appendix B: construction of an explicit nonlinear solenoid}

\subsection{A nonlinear perturbation of the doubling map}

Our goal is to construct a nonlinear perturbation of the doubling map on the circle on which we will be able to compute some periodic orbits and Lyapunov exponents. We want them to be linearly independent over $\mathbb{Q}$.

\begin{lemma}
Let $\Lambda := \{ \frac{1}{2^K - 1} \ | \ K \geq 2 \}$ and $$ \tilde{\Lambda} := \bigsqcup_{K \geq 2 } \left[ \frac{1}{2^K -1} - \frac{1}{8^K} \ , \ \frac{1}{2^K -1} + \frac{1}{8^K} \right] .$$
Let $N \geq 2$ and $k \in \llbracket 1, N-1 \rrbracket$. Then $$ \frac{2^k}{2^N-1} \notin \tilde{\Lambda} .$$
\end{lemma}

\begin{proof}

First of all, the union is indeed disjoint. Indeed, for $K \geq 2$, we see that 
$$ \frac{1}{2^{K+1}-1} + \frac{1}{8^{K+1}} < \frac{1}{2^{K}-1} - \frac{1}{8^{K}} ,$$
since $$ \frac{1}{2^{K}-1} - \frac{1}{2^{K+1}-1} >  \frac{1}{2^{K}} - \frac{1}{\frac{3}{2} 2^{K}} = \frac{1}{3} \frac{1}{2^K} > \frac{9}{8} \frac{1}{8^K} .$$
Now, fix $N \geq 2$. We first do the case where $k=N-1$. 
In this case, $$ \frac{2^k}{2^N-1} = \frac{1}{2} \frac{1}{1-2^{-N}} > 1/2 ,$$
which proves that $ \frac{2^k}{2^N-1} \notin \tilde{\Lambda} $ since $ \tilde{\Lambda} \subset [0,0.4] $. 

We still have to do the case where $ k \in \llbracket 1, N-2 \rrbracket$, and $N \geq 3$. First of all, we check that
$$  \frac{2^k}{2^N-1} < \frac{1}{2^{N-k}-1} - \frac{1}{8^{N-k}} .$$
Indeed,
$$ \frac{1}{2^{N-k}-1} - \frac{2^k}{2^{N}-1} = \frac{1}{2^{N-k}-1} - \frac{1}{2^{N-k}-2^{-k}} $$
$$ \geq  \frac{1}{2^{N-k}-1} - \frac{1}{2^{N-k}-2^{-1}} $$
$$ = \frac{1}{2 (2^{N-k}-1)(2^{N-k}-2^{-1})}  $$
$$ > \frac{1}{2 \cdot 4^{N-k}} > \frac{1}{8^{N-k}}.$$
Second, we check that $$ \frac{1}{2^{N-k+1}-1} + \frac{1}{8^{N-k+1}} < \frac{2^k}{2^N - 1} . $$
Indeed, 
$$ \frac{2^k}{2^N - 1} - \frac{1}{2^{N-k+1}-1} = \frac{1}{2^{N-k} - 2^{-k}} - \frac{1}{2^{N-k+1}-1} $$
$$ > \frac{1}{2^{N-k}} - \frac{1}{\frac{3}{2} \cdot 2^{N-k}} $$
$$ = \frac{1}{3} \frac{1}{2^{N-k}} > \frac{1}{8^{N-k+1}}, $$
and this conclude the proof.
\end{proof}

\begin{lemma}

Fix $(\alpha_N)_{N \geq 2}$ a family of real numbers so that $ \sum_{K \geq 2} |\alpha_K| 8^{K r} < \infty $ for any $r \geq 1$. There exists a $\mathcal{C}^\infty$ function $g : \mathbb{R} \rightarrow \mathbb{R}$ such that:

\begin{itemize}
    \item $ \text{supp g} \subset \tilde{\Lambda} $,
    \item $ \| g \|_{C^1} < 1 $.
    \item $ \forall N \geq 2, \ \forall k \in \llbracket 1, N-1 \rrbracket, \ g\left( \frac{2^k}{2^N-1} \right) = 0$
    \item $ \forall N \geq 2, \ g'\left( \frac{1}{2^N-1} \right) = \alpha_N$.
\end{itemize}

\end{lemma}

\begin{proof}

Let $\theta : \mathbb{R} \rightarrow [0,1]$ be a smooth bump function such that $\text{supp}(\theta) \subset [-1/2,1/2] $ and such that $\theta = 1$ on $[-1/4,1/4]$. Let $\chi(x) := x \theta(x)$. \\

The map $\chi$ is supported in $[-1/2,1/2]$ and satisfy $\chi(0)=0$, $\chi'(0)=1$.
Define, for $N \geq 2$, $$ \chi_N(x) := \chi\left(8^N\left(x-\frac{1}{2^N-1}\right)\right).$$
Then $\chi_N$ is supported in $\left[ \frac{1}{2^N -1} - \frac{1}{8^N} \ , \ \frac{1}{2^N -1} + \frac{1}{8^N} \right]$, and so the function
$$ g(x) := \sum_{K \geq 2} \alpha_K {8^{-K}} {\chi_K(x)} $$
is well defined and supported in $\tilde{\Lambda}$. Since, for all $r \geq 1$, $ \sum_{K \geq 2} |\alpha_K| 8^{K r} < \infty $, $g$ is a $\mathcal{C}^\infty$ function, with derivatives $$ g^{(r)}(x) = \sum_{K \geq 2} \alpha_K 8^{(r-1)K} \chi^{(r)}\left(8^N\left(x-\frac{1}{2^N-1}\right)\right) .$$
In particular, $$ \|g\|_{C^r} \leq \sum_{K \geq 2} |\alpha_K| 8^{(r-1)K}, $$
and so we see that replacing the sequence $(\alpha_K)_{K \geq 2}$ by $(\alpha_{K+K_0})_{K \geq 2}$ allows us to choose $\|g\|_{C^r}$ as small as we want if so desired. The fact that $g$ vanishes on $\Lambda$ and the computation of its derivative on this set follows from the previous formulas. \end{proof}

\begin{lemma}
The map $f(x) := 2x + g(x)$ can be seen as a smooth map $\mathbb{S}^1 \longrightarrow \mathbb{S}^1$.
\end{lemma}

\begin{proof}
We will identify $[0,1]/\{0 \sim 1\}$ to the circle. The only thing to check is if $g$ stays smooth after the quotient. It stays smooth indeed, since $g^{(r)}(0)=g^{(r)}(1)=0$ for all $r \geq 0$. \end{proof}

\begin{lemma}

For $g$ small enough (in the $C^1$ topology), $(f,\mathbb{S}^1)$ is a nonlinear perturbation of the usual doubling map. A family of periodic points for this dynamical system is the set $ \Lambda $, seen as a subset of the circle. The associated Lyapunov exponents are  $\ln(2) + \frac{1}{N}\ln\left(1+\alpha_N / 2 \right)$.

\end{lemma}

\begin{proof}

Let $N \geq 2$ and let $x_N := \frac{1}{2^N -1}$. We check that $x_N$ is periodic with period $N$. Indeed, for all $k \in \llbracket 1, N-1 \rrbracket$, we see that $ f^k(x_N) = 2^k x_N $ by construction of $g$, and then
$$ f^N(x_N) = \frac{2^N}{2^N -  1} = \frac{1}{2^N - 1} \ \text{mod} \ 1 .$$
Hence $x_N$ is $f$-periodic with period $N$. Its Lyapunov exponent is then
$$ \widehat{\lambda}( x_N ) := \frac{1}{N} \sum_{k=0}^{N-1} \ln |f'(f^k(x_N))| 
= \frac{1}{N} \sum_{k=0}^{N-1} \ln |2+g'(2^k x_N))| = \ln(2) + \frac{1}{N}\ln\left(1+\alpha_N / 2 \right) .  $$  \end{proof}

\begin{lemma}

We can choose the sequence $(\alpha_N)$ so that the Lyapunov exponents $\widehat{\lambda}( x_N )$ is a family of real numbers that are linearly independent over $\mathbb{Q}$.

\end{lemma}

\begin{proof}
Choosing $(\alpha_N)$ essentially randomly gives us what we want, but to stay deterministic let us give another argument.
By the Lindemann–Weierstrass theorem (\cite{Ba90}, Theorem 1.4), we just have to choose the $\alpha_N$ so that
$$ \widehat{\lambda}(x_N) = e^{\beta_N} $$
where $\beta_N$ are distinct algebraic numbers (while still ensuring that the sums of Lemma 4.9.2 converges). In other words, it suffice to choose $\alpha_N$ of the form
$$ \alpha_N = 2 \left( 2^{-N} e^{N \exp(\beta_N) }  - 1\right) $$
with $\beta_N$ distinct algebraic numbers converging to $\ln \ln 2$ quickly enough.
For example, we can fix:
$$ \beta_N := \left\lfloor 10^{N^2} \ln \ln 2 \right\rfloor 10^{- N^2} ,$$
and in this case $\alpha_N = O(N \cdot 10^{-N^2})$, which is quick enough to ensure convergence of the serie in Lemma 4.9.2. \end{proof}

We have constructed a chaotic and \emph{nonlinear} dynamical system on the circle with some prescribed Lyapunov exponent. With our choice of $\alpha_N$, we have $$ \widehat{\lambda}(x_N) = e^{ \left\lfloor 10^{N^2} \ln \ln 2 \right\rfloor 10^{- N^2} }.$$
In particular, $$ \text{dim}_\mathbb{Q} \text{Vect}_\mathbb{Q} \left\{ \widehat{\lambda}(x) \ | \ x \text{ is periodic} \right\} = \infty$$
which was what we wanted to construct.

\subsection{A nonlinear solenoid}

In this section we construct an explicit nonlinear perturbation of the usual solenoid. Denote by $\mathbb{T} := \mathbb{R}/\mathbb{Z} \times \overline{\mathbb{D}} $ the full torus.
Define $ F : \mathbb{T} \rightarrow \mathbb{T} $ by the formula

$$ F(\theta,x,y) := \left(f(\theta), \frac{1}{4}x + \frac{1}{4 \pi} \cos(2 \pi \theta) ,  \frac{1}{4}y + \frac{1}{4 \pi} \sin(2 \pi \theta)  \right)  $$
where $f(\theta) = 2 \theta + g(\theta)$ is the function defined previously. Notice that $F$ is a diffeomorphism onto its image (injectivity can be checked quickly by remembering that $f$ is conjugated to the doubling map via a conjugacy close to the identity for the uniform norm, see Theorem 18.2.1 in \cite{KH95}). In particular, we can use it to glue $\mathbb{T}$ with (a copy of) $F(\mathbb{T}) =: \tilde{\mathbb{T}}$ along their boundaries, which allows us to see $F$ as a diffeomorphism of a genuine closed 3-manifold $M = (\mathbb{T} \sqcup \tilde{\mathbb{T}})/({\partial \mathbb{T} \sim {\partial} \tilde{\mathbb{T}}})$ (\cite{Bo78}, chapter 1) which contains $\mathbb{T}$. The diffeomorphism obtained is Axiom A (see \cite{Ro99}, chapter 8.7 for details), with two basic sets, which are copies of the same solenoid: one seens as an attractor (lying in $\mathbb{T}$), and one seen as a repeller (lying in $\tilde{\mathbb{T}}$). Let us discuss the attractor.
\\

Since $F(\mathbb{T}) \subset \overset{\circ}{\mathbb{T}} $, we see that $F^n(\mathbb{T})$ is a strictly decreasing sequence of compact sets. The intersection $ S := \bigcap_{n \geq 0} F^n(\mathbb{T}) $ is called a (nonlinear) solenoid. The set $S$ is an attractor (and a basic set) for $F$.  Moreover, it has codimension 1 stable lamination, as it contract in the $(x,y)$ variables. \\

Indeed, at a given point $p=(\theta,x,y) \in S$, we see that the Jacobian of $F$ is
$$ \text{Jac}_F(p) = \begin{pmatrix} 2+g'(\theta) & 0 & 0 \\ - \frac{1}{2}\sin(2 \pi \theta) & 1/4 & 0 \\ \frac{1}{2}\cos(2 \pi \theta) & 0 & 1/4 \end{pmatrix} .$$
In particular, the subspace $E^s(p) := \{(0,h,k) \ | \ h,k \in \mathbb{R} \} \subset T_p \mathbb{T} = \mathbb{R}^3$ is independent of $p$, and we check that $(dF)_p (E^s(p)) \subset E^s(F(p))$ is a contracting linear map. So we have found the stable direction. \\

This allows us to compute the derivative in the unstable direction at a periodic point $p$ (of period $n$). Indeed, since $$ (dF^n)_p : E^s(p) \oplus E^u(p) \rightarrow E^s(p) \oplus E^u(p) $$
also sends $E^u(p)$ into $E^u(p)$, we can compute the determinant of the jacobian of $F^n$ by making the unstable derivative appear like so:
$$ \det (d(F^n)_p) = \partial_u (F^n)(p) \times \det( d(F^n)_{\ |E^s(p)} ) $$
Since we already know that $ \det( d(F^n)_{\ |E^s(p)} ) = \det(I/4^n) = 16^{-n} $,
and that $$ \det(d(F^n))_p = 16^{-n} \prod_{k=0}^{n-1}(2+g'(f^k(\theta))) ,$$
we find:
$$ \partial_u (F^n)(p) = \prod_{k=0}^{n-1} f'(f^k(\theta)) .$$
If $p$ is not a periodic point, then the angle $\vartheta(p)$ between $E^u$ and $(E^s)^{\perp}$ appear in the formula. Indeed, some linear algebra yields, in general:
$$ \forall p \in S, \ \det (dF)_p = \partial_u f(p) \det((dF)_{E^s}) \frac{\cos(\vartheta(F(p))}{\cos(\vartheta(p))} = 16^{-1} \partial_u f(p) \frac{\cos(\vartheta(F(p))}{\cos(\vartheta(p))}. $$
Iterating this formula gives
$$ \forall p \in S, \forall n \geq 0, \ \det (dF^n)_p = 16^{-n} \partial_u (f^n)(p)  \frac{\cos(\vartheta(F^n(p))}{\cos(\vartheta(p))} = 16^{-n} \partial_u (f^n)(p) e^{O(1)} .$$
Since $|\det(d(F^n)_p)| 
\leq 4^{-n}$, we know that the bunching condition (B) is satisfied, replacing $F$ by $F^N$ for $N$ large enough if necessary. This would only multiply the associated Lyapunov exponents by $N$, which doesn't change the dimension of $\text{dim}_\mathbb{Q} \text{Vect}_\mathbb{Q}(\widehat{\lambda}(S_{\text{per}}))$. \\

Finally, we are going to construct some periodic orbits for $F$ and compute their unstable Lyapunov exponents. If $p \in S$ is a periodic point for $F$, then it is clear that its angular coordinate is periodic for $f$. Reciprocally, let $\theta_0$ be a periodic point for $f$: there exists an integer $n_0$ such that $f^{n_0}(\theta_0)=\theta_0$.
Then, the map $ F^{n_0} $ satisfies
$$ F^{n_0}( \{  \theta_0 \} \times \overline{\mathbb{D}} ) \subset \{ \theta_0 \} \times \mathbb{D} $$
and is contracting on this disk. Hence, there exists a unique associated fixed point $p_0 \in \{ \theta_0\} \times \mathbb{D}$. \\

This allows us to exhibit some periodic orbits and compute the associated Lyapunov exponents. Let $N \geq 2$ and consider $\theta_N := \frac{1}{2^N-1}$ the periodic point for $f$ constructed in the last subsection. Let $p_N$ be the unique associated fixed point, noted $$ p_N = (\theta_N, x_N, y_N ) \in S .$$
Then we have proved that its associated unstable Lyapunov exponent is
$$ \widehat{\lambda}(p_N) = \widehat{\lambda}(\theta_N).$$
In particular, we have by Lemma 4.9.5:
$$ \dim_\mathbb{Q} \text{Vect}_\mathbb{Q} \widehat{\lambda} \left( S_\text{per} \right) = \infty, $$
where $S_{\text{per}} := \{ p \in S \ | \ p \ \text{is} \ F \ \text{periodic}. \}$. Hence, Theorem 4.1.8 applies. For example, the SRB measure and the measure of maximal entropy enjoys polynomial Fourier decay in the unstable direction. Indeed, let $\chi$ be a Hölder map with localized support at a point $p_0$ on the solenoid. Let $\mu$ be an equilibrium state. Then, the Fourier transform of $ \chi d \mu $ can be written, for $\xi \in \mathbb{R}^3$, as:
$$ \widehat{\chi d \mu}(\xi) = \int_{S} e^{- 2 i \pi x \cdot \xi} \chi(x) d\mu(x) = \int_{S} e^{- i t \phi_v} \chi d\mu $$
where $v := \xi/|\xi|$, $t=2 \pi |\xi|$ and $ \phi_v(x):= x \cdot v$.
Fix a direction $v$ such that $v$ is not orthogonal to $E^u(p_0)$. Then, if $\chi$ has small enough support, $ |\partial_u \phi_v| > 0 $ around the support of $\chi$. It follows from Theorem 4.1.8 that  $\widehat{\chi d \mu}(t v) \underset{t \rightarrow \infty}{\longrightarrow} 0$ at a polynomial rate. The convergence to zero is even uniform on any cone in the unstable direction.

\cleardoublepage

\chapter{The Fourier dimension of basic sets II}

\section{Introduction}

Let $M$ be a complete and smooth Riemannian surface, and let $f:M \rightarrow M$ be a $C^\infty$ Axiom A diffeomorphism. Suppose moreover that $M$ is oriented and that $f$ preserves the orientation. Let $\Omega$ be a basic set for $f$. Denote by $\mu \in \mathcal{P}(\Omega)$ an equilibrium state associated to some Hölder regular potential $\psi:\Omega \rightarrow \mathbb{R}$. (See the previous Chapter for details on Axiom A diffeomorphisms and the Thermodynamical formalism in this context.) Recall that we denote by $\tau_f : \Omega \rightarrow \mathbb{R}$ the map defined by $\tau_f(x) := \ln |\partial_u f(x)|$ for $x \in \Omega$, where
$$ \partial_u f(x) := \| (df)_{|E^u(x)} \|,$$
where we denoted $E^u(x) \subset T_x M$ the (one dimensional) unstable direction at $x$. Denote by $\widetilde{\text{Diag}}  \subset \Omega^2$ a small enough neighborhood of the diagonal, so that for any $(p,q) \in \widetilde{\text{Diag}}$, $\{[p,q] \} := W^s_{loc}(p) \cap W^u_{loc}(q)$ and $\{[q,p]\} := W^s_{loc}(q) \cap W^u_{loc}(p)$ are well defined (where $W^s_{loc}(x)$ and $W^u_{loc}(x)$ denote local stable/unstable manifolds passing through $x$.). Then, define, for $(p,q) \in \widetilde{\text{Diag}}$,
$$ \Delta_f(p,q) := \sum_{n \in \mathbb{Z}} \Big(\tau_f(f^n p) - \tau_f(f^n [p,q]) - \tau_f(f^n [q,p]) + \tau_f(f^n q) \Big). $$
A consequence of Theorem 4.1.5 is the following.

\begin{theorem}
Let $f:M \rightarrow M$ be a smooth Axiom A map on a complete riemannian surface. Let $\Omega$ be a basic set such that for $\det df = 1$ on $\Omega$. Let $\mu$ be an equilibrium state associated to some Hölder potential, and suppose that for some choice of Markov Partition $(R_a)_{a \in \mathcal{A}}$, the following Quantitative NonLinearity condition (QNL) hold:
$$ \exists \gamma \in (0,1), \  \exists C_0 \geq 1, \ \forall a \in \mathcal{A}, \ \forall \sigma >0, \quad \mu^{\otimes 2} \Big( (p,q) \in R_a^2, \  |\Delta_f(p,q)| \leq \sigma
 \Big) \leq C_0 \sigma^\gamma. $$
Finally, let $\alpha \in (0,1)$. Then there exists $\rho \in (0,1)$ such that the following hold. For any $\alpha$-Hölder map $\chi : \Omega 
\rightarrow \mathbb{C}$, for any $C^{1+\alpha}$ local chart $\varphi: U \rightarrow \mathbb{R}^2$ with $U \supset \text{supp}(\chi)$, there exists $C=C({f,\mu,\chi,\varphi}) \geq 1$ such that:
$$\forall \xi \in \mathbb{R}^2 \setminus \{0\}, \ \Big{|} \int_\Omega e^{i \xi \cdot \varphi(x) } \chi(x) d\mu(x) \Big{|} \leq C |\xi|^{-\rho} .$$
In particular, $\underline{\dim}_{F} (\mu) > 0$ and $\underline{\dim}_F (\Omega) >0$.
\end{theorem}

\begin{proof}
We see from the proof of Theorem 4.1.5 (especially in the proof of Proposition 4.6.2 at the end of section 4.6) that our condition is enough to apply Theorem 4.1.5: that is, we get Fourier decay in the unstable direction. The central remark to get full Fourier decay is the following. \\

By hypothesis, $f$ is area-preserving on $\Omega$. Combined with the fact that $(df)(E^u(x)) = (df)(E^u(f(x)))$ and $(df)(E^s(x)) = (df)(E^s(f(x)))$, it follows from elementary linear algebra on a plane that there exists a Hölder regular map $h: \Omega \rightarrow \mathbb{R}$ such that
$$ |\partial_u f(x) \partial_s f(x)| = e^{h(f(x)) - h(x)}.$$
In particular, $-\tau_f$ is cohomologous to the map $\tilde{\tau}_f(x) := \ln \partial_s f(x)$. It follows that, in this context,
$$ \Delta_f(p,q) = - \sum_{n \in \mathbb{Z}} \Big(\tilde{\tau}_f(f^n p) - \tilde{\tau}_f(f^n [p,q]) - \tilde{\tau}_f(f^n [q,p]) + \tilde{\tau}_f(f^n q) \Big) ,$$
and so our (QNL) condition is also satisfied if we reverse the dynamics. This gives Fourier decay in the stable direction. We conclude the proof following the ideas of Lemma 1.1.29 and Lemma 1.1.26.
\end{proof}

The difficult question that we will try to answer now is: how can we check (QNL) in this area-preserving context ? Is this a generic condition on the dynamics ? Unfortunately, we are not able to prove genericity on the dynamics, but only a form of genericity in the space of potentials when we perturb $\tau_f$. Let us precise what we mean. (Recall that a map defined on a subset $K$ of a manifold is said smooth in the sense of Whitney if it can be extended to a smooth map on the whole manifold. In this setting, usual Taylor expansions makes sense, even if $K$ has empty interior.)


\begin{definition}
Let $\tau : \Omega \rightarrow \mathbb{R}$. We say that $\tau \in \text{Reg}_u^{1+\alpha}(\Omega)$ (\say{regular in the unstable direction}) for some $\alpha \in (0,1)$ if $\tau \in C^{1+\alpha}(\Omega,\mathbb{R})$ (in the sense of Whitney),  if moreover, for any $p \in \Omega$, the map $\tau_{|W^u_{loc}(p)} : W^u_{loc}(p) \cap \Omega \rightarrow \mathbb{R}$ is $C^{N}$ for $N := 5 + \lceil 1/\alpha \rceil$ (uniformly in $p$), the maps
$$ r \in W^u_{loc}(p) \cap \Omega \longmapsto (\partial_s \tau(r),\partial_u \tau(r)) \in \mathbb{R}^2  $$
are uniformly $C^{1+\alpha}$, and if finally the map $(\partial_u \partial_s \tau, \partial_u \partial_u \tau)$ is $\alpha$-Hölder on $\Omega$.
\end{definition}

We prove in section 5.7 that this space of function contains $\tau_f$.

\begin{definition}
Let $\tau \in \text{Reg}_u^{1+\alpha}(\Omega)$. 
For $(p,q) \in \widetilde{\text{Diag}}$, and for any $n \in \mathbb{Z}$, define:
$$ T_n(p,q) := \tau(f^n(p)) - \tau(f^n([p,q])) - \tau(f^n([q,p])) + \tau(f^n(q)) .$$
Define also $\Delta : \widetilde{\text{Diag}}\longrightarrow \mathbb{R}$ by the formula $ \Delta(p,q) := \sum_{n \in \mathbb{Z}} T_n(p,q)$.
\end{definition}

The main result of this Chapter is the following: under a (generic in $\tau$) cohomology condition on a potential $\Phi_\tau : \Omega \rightarrow \mathbb{R}$, which is essentially a mixed derivative $\partial_u \partial_s \tau$, we are able to prove  the nonconcentration of $\Delta$. See remark 5.4.5 for a rigorous definition of $\Phi_\tau$.

\begin{theorem}
Let $f:M \rightarrow M$ be a smooth Axiom A map on a complete riemannian surface. Let $\Omega$ be a basic set such that for $\det df = 1$ on $\Omega$. Let $\tau : \Omega \rightarrow \mathbb{R}$ be an observable in $\text{Reg}_u^{1+\alpha}(\Omega)$. Denote by $\Phi_\tau : \Omega \rightarrow \mathbb{R}$ its associated \say{mixed derivative potential}. If $\Phi_\tau \in C^{\alpha}$ is not cohomologous to zero (this is a $\text{Reg}_u^{1+\alpha}$-generic condition on $\tau$), then for any equilibrium state (associated to a Hölder potential on $\Omega$), and for any Markov Partition $(R_a)_{a \in \mathcal{A}}$, there exists $C_0 \geq 1$, $\gamma>0$ such that:
$$  \forall a \in \mathcal{A}, \forall \sigma >0, \ \mu^{\otimes 2} \Big{(} (p,q) \in R_a^2, \ |\Delta(p,q)| \leq 
\sigma \Big{)} \leq C_0 \sigma^\gamma. \quad (NC)$$
\end{theorem}

This doesn't prove genericity on the dynamics, but this is still quite encouraging.  Theorem 5.1.4 will be proved via a succession of reductions, by adapting ideas found in M.Tsujii and Z.Zhang's paper \say{Smooth mixing Anosov flows in dimension three are exponentially mixing} \cite{TZ20}. It is possible that one might get rid of the area-preserving hypothesis with additionnal work, as this hypothesis is not needed in \cite{TZ20}. Still, the general setting would add some substancial technicalities, and so we stay in the area-preserving case for simplicity.
\section{Localised non-concentration}
We let $\tau \in \text{Reg}_u^{1+\alpha}(\Omega)$ and we fix some Markov partition $(R_a)_{a \in \mathcal{A}}$. From now on, we fix a parameter $\beta_Z >1$. It will be chosen large enough later in the text (in the end of section 5.4): its role will be to make some terms negligeable in some future Taylor expansions. We will denote by $\text{Rect}_{\beta_Z}(\sigma)$ the set of rectangles of the form $R = \bigcap_{k = -k_1}^{k_2} f^k(R_{a_k})$ for some word $(a_k)_{k} \in \mathcal{A}^{\llbracket -k_1, k_2 \rrbracket}$ with unstable diameter $\simeq \sigma$ and stable diameter $\simeq \sigma^{\beta_Z}$. To be more precise, we choose our rectangles as follow. \\

Recall that for each $a$, the boundary of $R_a$ have zero measure. For each $a$, write $R_a = [U_a,S_a]$. Now, for each $n \geq 0$, define $({U^{(n)}_{\mathbf{a}} })_{\mathbf{a} \in J^n}$ as the partition:
$$ U^{(n)}_{\mathbf{a}} := U_{a_1} \cap f^{-1} (R_{a_2}) \cap \dots \cap f^{-n}(R_{a_n}).$$ For $\mu$-almost every $x \in R_a$, for each $n \geq 0$, there exists a unique $\mathbf{a}^{(n)} \in \mathcal{A}^n$ such that $[x,x_a] \in U_{\mathbf{a}^{(n)}}^{(n)}$. As $n$ grows, the diameter of those goes to zero exponentially quickly. Let $n(x)$ be the smallest integer $n \geq 0$ such that $x \in U_{\mathbf{a}^{(n(x))}}^{(n(x))}$ and $\text{diam}^u(U_{\mathbf{a}^{(n(x))}}^{(n(x))}) \leq \sigma$. The unstable part of the partition is then given by $(U(x))_{x \in \cup_a U_a}$, where $U(x) := U_{\mathbf{a}^{(n(x))}}^{(n(x))}$. This is a finite partition of $(U_a)$ with elements of diameter $\simeq \sigma$. We then do a similar construction $(S(x))_{x \in \cup_a S_a}$ in the stable direction, but for the scale $\sigma^{\beta_Z}$. We can then consider the partition of $\cup_a R_a$ given by the rectangles $R(x):=[U(x),S(x)]$. This partition will be called $\text{Rect}_{\beta_Z}(\sigma)$. \\

The point of this section is to prove the following reduction:

\begin{proposition}
Suppose that there exists $C_0 \geq 1$ and $\gamma>0$ such that:
$$ \forall \sigma > 0, \ \forall R \in \text{Rect}_{\beta_Z}(\sigma), \ \mu^{\otimes 2}\Big( (p,q) \in R^2, \ |\Delta(p,q)| \leq 9 \sigma^{1+\beta_Z+\alpha} \Big) \leq C_0 \mu(R)^2 \sigma^\gamma. $$
Then, the conclusion of Theorem 5.1.4 holds.
\end{proposition}

To this aim, we use two elementary localisation lemma, each based on Cauchy-Schwarz inequality, to prove a similar statement for product measures. Localizing will be usefull to be able to use local asymptotic estimates. We then use the local product structure of equilibrium states.

\begin{lemma}[First localization lemma]
Let $\lambda$ be a Borel probability measure on a topological space $X$ and $(X_i)_{i \in I}$ a finite partition with $\lambda(X_i) >0$. Let $h:X \rightarrow \mathbb{R}$ be a measurable function. Then, for all $\sigma>0$,
$$ \lambda^{\otimes 2}( (x,x') \in X^2, \ |h(x)-h(x')| \leq \sigma)^2 \leq \sum_{i \in I} \frac{1}{\lambda(X_i)} \lambda^{\otimes 2}( (x,\tilde{x}) \in X_i^2, \ |h(x)-h(\tilde{x})| \leq 2 \sigma).  $$
\end{lemma}

\begin{proof} The proof is a succession of Cauchy-Schwarz inequalities. We have:
$$\lambda^{\otimes 2}( (x,x') \in X^2, \ |h(x)-h(x')| \leq \sigma)^2 = \Big(\int_X \int_X \mathbb{1}_{[-\sigma,\sigma]}(h(x)-h(x')) d\lambda(x) d\lambda(x')\Big)^2 $$
$$ \leq \int_X \Big(\int_X \mathbb{1}_{[-\sigma,\sigma]}(h(x)-h(x')) d\lambda(x) \Big)^2 d\lambda(x') $$
$$ = \int_X \Big( \sum_{i \in I} \lambda(X_i) \frac{1}{\lambda(X_i)} \int_{X_i} \mathbb{1}_{[-\sigma,\sigma]}(h(x)-h(x')) d\lambda(x) \Big)^2 d\lambda(x')  $$
$$ \leq \int_X \sum_{i \in I} \lambda(X_i) \Big( \frac{1}{\lambda(X_i)} \int_{X_i} \mathbb{1}_{[-\sigma,\sigma]}(h(x)-h(x')) d\lambda(x) \Big)^2 d\lambda(x')  $$
$$ = \int_X \sum_{i \in I} \frac{1}{\lambda(X_i)} \iint_{X_i^2} \mathbb{1}_{[-\sigma,\sigma]}(h(x)-h(x')) \mathbb{1}_{[-\sigma,\sigma]}(h(\tilde{x})-h(x')) d\lambda^2(x,\tilde{x})  d\lambda(x') $$
$$ \leq  \sum_{i \in I} \frac{1}{\lambda(X_i)} \iint_{X_i^2} \mathbb{1}_{[-2 \sigma, 2 \sigma]}(h(x)-h(\tilde{x})) d\lambda^2(x,\tilde{x})  $$
$$ = \sum_{i \in I} \frac{1}{\lambda(X_i)} \lambda^{\otimes 2}( (x,\tilde{x}) \in X_i^2, \ |h(x)-h(\tilde{x})| \leq 2 \sigma). $$
\end{proof}

\begin{lemma}[Second localization lemma]
Let $\lambda_X$ (resp. $\lambda_Y$) be a borel measure on a topological space $X$ (resp. $Y$) and let $\lambda_Z := \lambda_X \otimes \lambda_Y$ be a probability measure on $Z := X \times Y$. Let $(X_i)$ (resp. $(Y_i)$) be a partition of $X$ (resp. $Y$) with $\lambda_X(X_i) > 0$ (resp. $\lambda_Y(Y_i)>0$). Denote $Z_{i,j} := X_i \times Y_j$. Let $h:Z \rightarrow \mathbb{R}$ be a measurable map. Then, for all $\sigma>0$, $$\lambda_Z^{\otimes 2}\Big( ((x,y),(x',y')) \in Z^2, \ |h(x,y)-h(x',y) - h(x,y') + h(x',y') | \leq \sigma \Big)^4 $$
$$ \leq \sum_{(i,j) \in I \times J} \frac{1}{\lambda_Z(Z_{i,j})} \lambda_Z^{\otimes 2}\Big( ((x,y),(x',y')) \in Z_{i,j}^2, \ | h(x,y) - h(x',y) - h(x,y') + h(x',y') | \leq 4\sigma ) \Big) $$
\end{lemma}

\begin{proof}
The proof uses successive Cauchy-Schwarz inequalities and the first localization lemma. We have:
$$\lambda_Z^{\otimes 2}\Big( ((x,y),(x',y')) \in Z^2, \ |h(x,y)-h(x',y) - h(x,y') + h(x',y') | \leq \sigma \Big)^4$$
$$ = \Big( \iint_{Y^2} \lambda_X^{\otimes 2}( (x,x') \in X^2, \big| \big(h(x,y) - h(x,y')\big) - \big(h(x',y) - h(x',y')\big) \big| \leq \sigma ) d\lambda_Y^{\otimes 2}(y,y') \Big)^4 $$
$$ \leq \Big( \iint_{Y^2} \lambda_X^{\otimes 2}( (x,x') \in X^2, \big| \big(h(x,y) - h(x,y')\big) - \big(h(x',y) - h(x',y')\big) \big| \leq \sigma )^2  d\lambda_Y^{\otimes 2}(y,y') \Big)^2 $$
$$ \leq \Big( \iint_{Y^2} \sum_{i \in I} \frac{1}{\lambda(X_i)} \lambda_X^{\otimes 2}( (x,x') \in X_i^2, \big| \big(h(x,y) - h(x,y')\big) - \big(h(x',y) - h(x',y')\big) \big| \leq 2 \sigma  ) d\lambda_Y^{\otimes 2}(y,y') \Big)^2 $$
$$ = \Big( \sum_{i \in I} \lambda_X(X_i) \frac{1}{\lambda_X(X_i)^2} \iint_{X_i^2} \lambda_Y^{\otimes 2}( (y,y') \in Y^2, \ \big| \big(h(x,y) - h(x,y')\big) - \big(h(x',y) - h(x',y')\big) \big| \leq 2\sigma ) d\lambda_X^{\otimes 2}(x,x') \Big)^2 $$
$$ \leq  \sum_{i \in I} \lambda_X(X_i) \frac{1}{\lambda_X(X_i)^2} \iint_{X_i^2} \lambda_Y^{\otimes 2}( (y,y') \in Y^2, \ \big| \big(h(x,y) - h(x',y)\big) - \big(h(x,y') - h(x',y')\big) \big| \leq 2\sigma )^2 d\lambda_X^{\otimes 2}(x,x') $$
$$ \leq \sum_{(i,j) \in I \times J} \frac{1}{\lambda(Z_{i,j})} \lambda_Z^{\otimes 2}\Big( ((x,y),(x',y')) \in Z_{i,j}^2, \ | h(x,y) - h(x,y') - h(x',y) + h(x',y') | \leq 4\sigma ) \Big) .$$
\end{proof}

Before the proof of Proposition 5.2.1, let us recall the statement about the local product structure of equilibrium states (see Section 4.2.6). For all $x \in \Omega$, denoting $S_x := W^s_{loc}(x) \cap \Omega$, $U_x := W^u_{loc}(x) \cap \Omega$ and $R_x := [U_x,S_x]$, there exists two measures $\mu_x^u \in \mathcal{P}(U_x)$ and $\mu_x^s \in \mathcal{P}(S_x)$, and a (uniformly in $x$) bounded map $\omega_x : U_x \times S_x \rightarrow \mathbb{R}$ such that, for all measurable map $h : R_x \rightarrow \mathbb{R}$, we have
$$ \int_{R_x} h d\mu = \int_{U_x} \int_{S_x} h([z,y]) e^{\omega_x(z,y)} d\mu_x^s(y) d\mu_x^u(z). $$

\begin{proof}[Proof (of proposition 5.2.1)]

Let $a \in \mathcal{A}$ and let $\sigma>0$. Denote $\Gamma := 1+\beta_Z+\alpha$. Since $\tau$ is Lipschitz, and by uniform hyperbolicity of the dynamics, there exists $n_0(\sigma)\geq 1$ such that
$$ \forall (p,q) \in R_a, \ |\Delta(p,q) - H_{n_0(\sigma)}(p,q)| \leq \sigma^{\Gamma}, $$
where $H_{n_0}(p,q) = \sum_{n=-n_0}^{n_0} T_n(p,q)$. It follows that
$$ \mu^{\otimes 2}\Big( (p,q) \in R_a^2, \ |\Delta(p,q)| \leq \sigma^{\Gamma} \Big) \leq \mu^{\otimes 2}\big( (p,q) \in R_a^2, \ |\sum_{n=-n_0}^{n_0} T_n(p,q)| \leq 2 \sigma^{\Gamma} \Big). $$
We fix some $x_a \in R_a$ and write, using the local product structure:
$$ \mu^{\otimes 2}\Big( (p,q) \in R_a^2, \ |H_{n_0}(p,q)| \leq 2 \sigma^{\Gamma} \Big) $$
$$ = \int_{U_{x_a}} \int_{S_{x_a}} \mathbb{1}_{[-2\sigma^{\Gamma},2\sigma^{\Gamma}]}( H_{n_0}([z,y],[z',y']) ) e^{\omega_{x_a}(z,y)} e^{\omega_{x_a}(z',y')} (d\mu_{x_a}^s)^{\otimes 2}(y) (d\mu_{x_a}^u)^{\otimes 2}(z) $$
$$ \leq e^{2 \| \omega_{a} \|_{\infty,U_a \times S_a}} (\mu_a^s \otimes \mu_a^u)^{\otimes 2}\Big( ((z,y),(z',y')), \ |h_{n_0}( z,y ) - h_{n_0}(z',y) - h_{n_0}(z,y') + h_{n_0}(z',y') | \leq 2 \sigma^{\Gamma} \Big),  $$
where $h_{n_0}(z,y) := \sum_{n=-n_0}^{n_0} \tau(f^n([z,y]))$. The second localisation lemma then gives the bound
$$ (\mu_a^s \otimes \mu_a^u)^{\otimes 2}\Big( ((z,y),(z',y')), \ |h_{n_0}( z,y ) - h_{n_0}(z',y) - h_{n_0}(z,y') + h_{n_0}(z',y') | \leq 2 \sigma^{\Gamma} \Big)^4 $$
$$ \leq \sum_{R \in \text{Rect}_{\beta_Z}(\sigma)} \frac{(\mu_a^s \otimes \mu_a^u)^{\otimes 2}\Big( ((z,y),(z',y')) \in \tilde{R}^2, \ |h_{n_0}( z,y ) - h_{n_0}(z',y) - h_{n_0}(z,y') + h_{n_0}(z',y') | \leq 8 \sigma^{\Gamma} \Big)}{\mu_a^s(S_a) \mu_a^u(U_a)} . $$
Using the fact that $\omega_a \in L^\infty$ allows us to get the bound
$$ \mu^{\otimes 2}\Big( (p,q) \in R_a^2, \ |H_{n_0}(p,q)| \leq 2 \sigma^{\Gamma} \Big)^4  $$
$$ \leq \sum_{R \in \text{Rect}_{\beta_Z}(\sigma)} \frac{e^{C_1(\mu)}}{\mu(R)} \mu^{\otimes 2}\Big( (p,q \in R^2, \ |H_{n_0}(p,q)| \leq 8 \sigma^{\Gamma} \Big) .$$
Now, using the fact that $ \forall (p,q) \in R_a, \ |\Delta(p,q) - H_{n_0}(p,q)| \leq \sigma^{\Gamma} $ again, we finally find our last estimate:
$$ \mu^{\otimes 2}\Big( (p,q) \in R_a^2, \ |\Delta(p,q)| \leq  \sigma^{\Gamma} \Big)^4 $$
$$ \leq \sum_{R \in \text{Rect}_{\beta_Z}(\sigma)} \frac{e^{C_1(\mu)}}{\mu(R)} \mu^{\otimes 2}\Big( (p,q \in R^2, \ |\Delta(p,q)| \leq 9 \sigma^{\Gamma} \Big). $$
This is enough to conclude. Indeed, plugging the local nonconcentration estimate of the hypothesis of proposition 5.2.1 in the sum yields
$$ \mu^{\otimes 2}\Big( (p,q) \in R_a^2, \ |\Delta(p,q)| \leq  9 \sigma^{\Gamma} \Big) \leq \Big(  \sum_{R \in \text{Rect}_{\beta_Z}(\sigma)} \frac{e^{C_1(\mu)}}{\mu(R)} C_0 \mu(R)^2 9 \sigma^\gamma \Big)^{1/4} \leq 2 C_0 e^{C_1(\mu)} \sigma^{\gamma/4} .$$
\end{proof}

We conclude this section with a final reformulation. For $p \in \Omega$ and $\sigma>0$ we will denote $S_p(\sigma) := W^s_{loc}(p) \cap \Omega \cap B(p,\sigma)$, $U_p(\sigma) := W^u_{loc}(p) \cap \Omega \cap B(p,\sigma)$ and $R_p^{\beta_Z}(\sigma) := [U_p(\sigma),S_p(\sigma^{\beta_Z})]$.

\begin{lemma}
Suppose that there exists $\gamma>0$ and $C_0 \geq 1$ such that, for any $\sigma > 0$, for all $p \in \Omega$,  $$ \mu( q \in R_p^{\beta_Z}(\sigma) , \ |\Delta(p,q)| \leq \sigma^{1+\alpha+\beta_Z} ) \leq C_0 \mu(R_p^{\beta_Z}(\sigma))\sigma^\gamma . $$
Then the conclusion of Theorem 5.1.4 holds.
\end{lemma}

\begin{proof}
Suppose that the estimate of Lemma 5.2.4 holds. We are going to check that Proposition 5.2.1 applies. Let $R \in \text{Rect}_{\beta_Z}(\sigma)$. We have:
$$ \mu^{\otimes 2}\Big( (p,q) \in R^2, \ |\Delta(p,q)| \leq 9 \sigma^{1+\alpha+\beta_Z} \Big) = \int_{R} \mu(q \in R, \ |\Delta(p,q)| \leq 9 \sigma^{1+\alpha+\beta_Z}) d\mu(p). $$
The key fact to use is that the measure $\mu$ is doubling in our 2-dimensional setting: this will be proved in subsection 5.9. In fact, it will be proved in this subsection that there exists $C_1,C_2 \geq 10$ such that, for any $p \in R$, $$ R \subset R_p^{\beta_Z}(C_1 \sigma) \quad ; \quad \mu(R_p^{\beta_Z}(C_1 \sigma)) \leq C_2 \mu(R). $$
From this fact, it follows that
$$ \mu(q \in R, \ |\Delta(p,q)| \leq 9\sigma^{1+\alpha+\beta_Z}) \leq \mu(q \in R_p^{\beta_Z}(C_1 \sigma), \ |\Delta(p,q)| \leq (C_1 \sigma)^{1+\alpha+\beta_Z}) $$
$$ \leq (C_1 \sigma)^\gamma \mu(R_p^{\beta_Z}(C_1 \sigma)) \leq C_1^\gamma C_2 \sigma^\gamma \mu(R). $$
which proves the hypothesis of Proposition 5.2.1.
\end{proof}

Now we know that, to conclude, it suffices to understand the oscillations of $\Delta(p,q)$, for any fixed $p$, when $q$ gets close to $p$. To do so, we will introduce some coordinate systems associated to the dynamics.

\section{Adapted coordinates and templates}

In this section, we construct a family of adapted coordinates in which the dynamics is going to be (almost) linearized. We also define templates (linear form version, and vector field version). We need to introduce some notations first. For each $x \in \Omega$, we will denote
$$ \lambda_x := \partial_u f(x) \quad, \quad \mu_x := \partial_s f(x). $$
Recall that we can always modify the metric so that $\lambda_x \in (1,\infty)$ and $\mu_x \in (0,1)$. We suppose that we work with such a metric. By compacity of $\Omega$ and continuity of $x \mapsto (\lambda_x,\mu_x)$, there exists $\lambda_-,\lambda_+,\mu_-,\mu+$ such that:
$$ \forall x \in \Omega, 0 < \mu_+ < \mu_x <\mu_- < 1 < \lambda_- < \lambda_x < \lambda_+ < \infty. $$
The area-preserving hypothesis ensure that, for any periodic point $x$ of period $n$, we have
$$ \lambda_x \dots \lambda_{f^{n-1}(x)} = \mu_x^{-1} \dots \mu_{f^{n-1}(x)}^{-1} .$$
In other words, for any periodic point $x$ of period $n$, the Birkhoff sums $\sum_{k=0}^{n-1} \ln(\mu_{f^k(x)} \lambda_{f^{k}(x)})$ vanishes. Livsic's Theorem \cite{Liv71} then ensure the cohomology relation $\ln(\mu_x \lambda_x) \sim 0$. In particular, we can always write:
$$\forall x \in \Omega, \forall n \geq 0, \  \lambda_x \dots \lambda_{f^{n-1}(x)} = \mu_x^{-1} \dots \mu_{f^{n-1}(x)}^{-1} e^{O(1)} . $$
(These estimates can also be checked directly by doing some elementary linear algebra in the plane, using the fact that the stable and unstable distributions are invariant under the dynamics.) Our goal is to introduce some local coordinates along stable and unstable manifolds, for some base point $x \in \Omega$. There is a small technicality coming from the possible non-orientability of the stable and unstable laminations that we will solve by adding the data of some local orientation with our base-point. This is the meaning of the following definition:

\begin{definition}
We define: $$\widehat{\Omega} := \Big\{ (x,e_u,e_s) \in \Omega \times TM^2 \Big|  (e_u,e_s) \in E^u(x) \times E^s(x), \ |e_u|=|e_s|=1, \ (e_u,e_s) \text{ is a direct basis of } T_xM  \Big\}.$$
There is a natural projection $p: \widehat{\Omega} \rightarrow \Omega$. Since we work on an oriented surface $M$, the fibers are isomorphic to $\mathbb{Z}/2\mathbb{Z}$: given $x \in \Omega$, there is two possible choices for $e_u$, and then there exists a unique choice of $e_s$ that makes $(e_u,e_s)$ into a direct basis of $T_x M$. The only other element in the fiber is then $(x,-e_u,-e_s)$. Notice also that there is a natural dynamics $\widehat{f} : \widehat{\Omega} \rightarrow \widehat{\Omega}$ such that $f \circ p = p \circ \widehat{f}$. A formula for $\widehat{f}$ is given by
$$ \widehat{f}(x,e_u,e_s) = \Big(f(x), \frac{(df)_x(e_u)}{|(df)_x(e_u)|}, \frac{(df)_x(e_s)}{|(df)_x(e_s)|}\Big). $$
Seeing $\widehat{\Omega}$ as a subset of the frame bundle naturally allows us to see it as a metric space. In the following, we will often denote by $\widehat{x}$ an element of the fiber $p^{-1}(x) \in \widehat{\Omega}$. Finally, notice the following fact: let $\phi:\Omega \rightarrow \mathbb{R}$ be some smooth map, and fix $x \in \Omega$. Now let $\widehat{x} = (x,e_u,e_s) \in \widehat{\Omega}$ in the fiber. We can then consider the quantity $(d^2 \phi)_x(e_u,e_s)$. The remark is the following: this quantity doesn't depend on the choice of $\widehat{x}$ in the fiber, and actually only depends on $x$.
\end{definition}

\begin{lemma}
There exists a family of uniformly smooth maps $(\Phi^s_{\widehat{x}})_{\widehat{x} \in \widehat{\Omega}}$, such that for all $\widehat{x} \in \widehat{\Omega}$ $\Phi_{\widehat{x}}^s : \mathbb{R} \longrightarrow W^s(x)$
is a smooth parametrization of $W^s(x)$, $\Phi_{\widehat{x}}^s(0)=x$, $(\Phi_x^u)'(0)=e_s(\widehat{x})$, and:
$$ \forall x \in \widehat{\Omega}, \forall y \in \mathbb{R},  \ f\left( \Phi_{\widehat{x}}^s(y) \right) = \Phi_{\widehat{f}(\widehat{x})}^s( \mu_x y ),  $$
where ${\mu}_x := \partial_s f(x) \in (\mu_+,\mu_-) \subset (0,1)$. The dependence in $\widehat{x}$ of $(\Phi_{\widehat{x}}^s)_{\widehat{x} \in \widehat{\Omega}}$ is Hölder.
\end{lemma}

\begin{proof}

The proof is taken from \cite{KK07}.
The idea is to first define $\Phi_{\widehat{x}}^s$ on $W^s_{loc}(x)$ (on which $|\partial_s f(x)|$ makes sense and is smooth along $W^s_{loc}(x)$, even outside $\Omega$), and then to extend our maps on $W^s(x)$ using the conjugacy relation that we want. Define, for any $x \in \Omega$, and for any $y \in W^s_{\text{loc}}(x)$, the function $\rho_x(y)$ by the formula:
$$ \rho_x(y) := \sum_{n=0}^\infty \left[ (\ln \partial_s f)(f^n y) - (\ln \partial_s f)(f^n x) \right] .$$
This is well defined and smooth along $W_{loc}^s(x)$ (and this, uniformly in $x \in \Omega$). We can then define, for $y \in W^s(x)$:
$$ |\Psi^s_x(y)| := \Big|\int_{[x,y] \subset W^s_{loc}(x)} e^{\rho_x(y')} dy' \Big|,$$
in the sense that we integrate from $x$ to $y$, following the local stable manifold $W^s_{loc}(x)$ (w.r.t. the arclenght). For some choice of $\widehat{x}$ there is a unique smooth natural map $\Psi^s_{\widehat{x}}:W_{loc}^s(x) \rightarrow (-\varepsilon,\varepsilon)$ such that $d(\Psi^s_{\widehat{x}})_x(e_s(x)) = 1$ and $|\Psi^s_{\widehat{x}}(y)| = |\Psi^s_{x}(y)|$. This function is smooth along $W^s_{loc}(x)$, and is invertible. We denote by $\Phi_{\widehat{x}}^s : (-\varepsilon,\varepsilon) \rightarrow W^s(x)$ its inverse. One can choose a uniform $\varepsilon$ for all these maps.  \newpage

We check that the dynamics is linearized in these coordinates. Notice that, since $\Psi_{\widehat{x}}^s(x)=0$, and since the orientation of compatible, the desired relation is equivalent to:
$$ \forall y \in W^s_{loc}(x), \ \partial_s\Psi_{\widehat{f}(\widehat{x})}^s( f(y))\partial_s f(y) = \mu_x \partial_s \Psi_{\widehat{x}}^s(y). $$
But this is obviously true, by construction of $\Psi_{\widehat{x}}^s$. It follows that, for all $y \in W^s_{loc}(x)$, $\Psi_{\widehat{f}(\widehat{x})}^s(f(y)) = {\mu}_x \Psi_{x}^s(y) $. In particular, notice that iterating this relation yields, for $y \in W^s_{loc}(x)$, $\Psi_{\widehat{f}^n(\widehat{x})}^s(f^n(y)) = \mu_{f^{n-1}x} \dots \mu_{x} \Psi_{\widehat{x}}^s(y)$. \newline

To conclude the proof, we need to extend $\Psi^s_{\widehat{x}}$ on the whole stable manifold of $x \in \Omega$. We proceed as follow. Let $y \in W^s(x)$. If $n$ is large enough (depending on $y$), one sees that $f^n y \in W^s_{loc}(f^n x)$. Hence, it makes sense to define:
$$ \Psi_{\widehat{x}}^s(y) := {\mu}_x^{-1} \dots {\mu}_{f^{n-1} x}^{-1}  \Psi_{\widehat{f}^n {\widehat{x}}}^s(f^n(y)). $$
The previous discussion ensure that this is well defined. Moreover, it is easy to check that the map $\Psi_{\widehat{x}}^s : W^s(x) \rightarrow \mathbb{R}$ is a smooth diffeomorphism (when we see $W^s(x)$ as a manifold equipped with the arclenght.) The inverse of $\Psi_{\widehat{x}}^s$ is defined to be $\Phi_{\widehat{x}}^s : \mathbb{R} \rightarrow W^s(x)$. The commutation relation is then easy to check. The Hölder regularity in $x$ is tedious to detail but shouldn't be surprising.
\end{proof}

\begin{lemma}
There exists a family of uniformly smooth maps $(\Phi^u_{\widehat{x}})_{{\widehat{x}} \in \widehat{\Omega}}$, such that for all ${\widehat{x}} \in \widehat{\Omega}$ $\Phi_{\widehat{x}}^u : \mathbb{R} \longrightarrow W^u(x)$
is a smooth parametrization of $W^u(x)$, $\Phi_x^u(0)=x$, $(\Phi_x^u)'(0)=e_u({\widehat{x}})$, and:
$$ \forall {\widehat{x}} \in \widehat{\Omega}, \forall z \in \mathbb{R}, \ f\left( \Phi_{\widehat{x}}^u(z) \right) = \Phi_{\widehat{f}({\widehat{x}})}^u( {\lambda}_x z ),  $$
where $\lambda_x = \partial_u f(x) \in (\lambda_-,\lambda_+) \subset (1,\infty)$. The dependence in ${\widehat{x}}$ of $(\Phi_{\widehat{x}}^u)$ is Hölder.
\end{lemma}

\begin{remark}
Denote by $\mathfrak{s}:\widehat{\Omega}\rightarrow \widehat{\Omega}$ the map $\mathfrak{s}(x,e_u,e_s):=(x,-e_u,-e_s)$. Then, for any $\widehat{x} \in \widehat{\Omega}$ and for any $y,z \in \mathbb{R}$:
$$ \Phi_{\mathfrak{s}(\widehat{x})}^u(z) = \Phi_{\widehat{x}}^u(-z) \quad ; \quad \Phi_{\mathfrak{s}(\widehat{x})}^s(y) = \Phi_{\widehat{x}}^s(-y).$$
\end{remark}

\begin{definition}
These parametrizations often goes outside $\Omega$, but we are only interested by what's happening inside $\Omega$. So let us define:
$$ \Omega_{\widehat{x}}^u := (\Phi_{\widehat{x}}^u)^{-1}(\Omega) \subset \mathbb{R} \quad, \quad \Omega_{\widehat{x}}^s := (\Phi_{\widehat{x}}^s)^{-1}({\Omega}) \subset \mathbb{R} .$$
Notice that, for all ${\widehat{x}} \in \widehat{\Omega}$,  $ 0 \in \Omega_{\widehat{x}}^u $. Also, $\Omega_{\mathfrak{s}(\widehat{x})}^u = - \Omega_{\widehat{x}}^u $. Moreover:
$$ \forall {\widehat{x}} \in \widehat{\Omega}, \forall z \in \Omega_{{\widehat{x}}}^u, \ {\lambda}_x z \in \Omega_{\widehat{f}({\widehat{x}})}^u \subset \mathbb{R} .$$
Similar statements hold for $\Omega_{\widehat{x}}^s$.
\end{definition}

\begin{remark}
Let us define some further notations. Define, for $n \in \mathbb{Z}$ and $x \in \Omega$:
$$ \lambda_x^{\langle n \rangle} := \partial_u (f^n)(x) \quad; \quad \mu_x^{\langle n \rangle} := \partial_s (f^n)(x)  .$$
Notice that $\lambda_x^{\langle 0 \rangle} = \mu_x^{\langle 0 \rangle} = 1$, $ \lambda_x^{\langle -n \rangle} = (\lambda_{f^{-n}(x)}^{\langle n \rangle})^{-1} $ and $\mu_x^{\langle -n \rangle} = (\mu_{f^{-n}(x)}^{\langle n \rangle})^{-1}$. 
Moreover, we can write some relations involving $(\Phi_{\widehat{x}}^u)$ and $(\Phi_{\widehat{x}}^s)$. For all $n \in \mathbb{Z}$, ${\widehat{x}} \in \widehat{\Omega}$, $y \in \Omega_{\widehat{x}}^s$, $z \in \Omega_{\widehat{x}}^u$, we have:
$$ f^n (\Phi_{\widehat{x}}^u(z)) = \Phi_{\widehat{f}^n({\widehat{x}})}^u(\lambda_x^{\langle n \rangle} z) \quad ; \quad f^n (\Phi_{\widehat{x}}^s(y)) = \Phi_{{\widehat{f}}^n({\widehat{x}})}^s(\mu_x^{\langle n \rangle} y) .$$
\end{remark}

\begin{lemma}[change of parametrizations]
Let $x_0 \in \Omega$ and let $x_1 \in \Omega \cap W_{loc}^u(x)$. Then the real maps $\text{aff}_{\widehat{x_1},\widehat{x_0}} := (\Phi_{\widehat{x_1}}^u)^{-1} \circ \Phi_{\widehat{x_0}}^u : \mathbb{R} \longrightarrow \mathbb{R}$ are affine.
Moreover, there exists $C \geq 1$ and $\alpha>0$ such that $\ln |{\text{aff}_{\widehat{x_1},\widehat{x_0}}}'(0)| \leq C d(x_0,x_1)^\alpha$.
\end{lemma}

\begin{proof}
Notice that, for all $z \in \mathbb{R}$, and for all $n \geq 0$:
$$ \left(\Phi_{\widehat{x_1}}^{u}\right)^{-1}\left( \Phi_{\widehat{x_0}}^{u}(z) \right) =  {\lambda}_{x_1}^{\langle n \rangle } \left(\Phi_{\widehat{f}^{-n}(\widehat{x_1})}^{u}\right)^{-1}\left( \Phi_{\widehat{f}^{-n}(\widehat{x_0})}^{u}( \lambda_{x_0}^{\langle - n \rangle} z ) \right) . $$
In particular, without loss of generality, we see that we can reduce our problem to show that $\text{aff}_{\widehat{x_1},\widehat{x_0}}$ is affine on a neighborhood of zero, and this property should spread. In this case, we can compute the log of the absolute value of the differential of $\left(\Phi_{\widehat{x_1}}^{u}\right)^{-1} \circ \Phi_{\widehat{x_0}}^{u}$, and we get:

$$ \ln \left(  \Big{|} \Big{(} \left( \Phi_{\widehat{x_1}}^{u}\right)^{-1} \circ \Phi_{\widehat{x_0}}^{u} \Big{)}'(z) \Big{|} \right) $$ 
$$= \rho_{x_1}( \Phi_{\widehat{x_0}}^{u}(z) ) - \rho_{x_0}( \Phi_{\widehat{x_0}}^{u}(z) ) = \rho_{x_1}(x_0), $$
which is constant in $z$. The bound on ${\text{aff}_{\widehat{x_1},\widehat{x_0}}}'(0)$ follows from an easy bound on $\rho_{x_1}(x_0)$.
\end{proof}

These coordinates are interesting but only linearize the dynamics along the stable or unstable direction. Of course, we can't expect to fully linearize the dynamics in smooth coordinates, but we can still try to introduce coordinates that will linearize the dynamics in a weaker sense, in some particular places. This construction is directly taken from \cite{TZ20}, appendix B. We call them nonstationnary normal coordinates because the base point is not necesseraly a fixed point, and because the dynamics can be written in a \say{normal form} in these coordinates.

\begin{lemma}[Nonstationary normal coordinates]
There exists two small constants $\rho_1 < \rho_0 < 1$, and a family of uniformly smooth coordinates charts $\{ \iota_{\widehat{x}}:(-\rho_0,\rho_0)^2 \rightarrow M \}_{\widehat{x} \in \widehat{\Omega}}$ that preserves the orientation such that:
\begin{itemize}
\item For every $\widehat{x} \in \Omega$, we have $$ \iota_{\widehat{x}}(0,0) = x, \quad  \iota_{\widehat{x}}(z,0) = \Phi^u_{\widehat{x}}(z), \quad \iota_{\widehat{x}}(0,y) = \Phi_{\widehat{x}}^s(y) ,$$
\item the map $f_{\widehat{x}} := \iota_{\widehat{f}({\widehat{x}})}^{-1} \circ f \circ \iota_{\widehat{x}} : (-\rho_1,\rho_1)^2 \longrightarrow (-\rho_0,\rho_0)^2$ is smooth (uniformly in ${\widehat{x}}$) and satisfies
$$ \pi_y( \partial_y f_{\widehat{x}}(z,0)) = \mu_x , \quad \pi_z(\partial_z f_{\widehat{x}}(0,y)) = {\lambda}_x,$$
where $\pi_z$ (resp. $\pi_y$) is the projection on the first (resp. second) coordinate.
\end{itemize}
Furthermore, one can assume the dependence in ${\widehat{x}}$ of $(\iota_{\widehat{x}})_{{\widehat{x}} \in \widehat{\Omega}}$ to be Hölder regular.
\end{lemma}

\begin{proof}
Since the stable/unstable manifolds are smooth, and since they intersect uniformly transversely, we know that we can construct a system of smooth coordinate charts $(\check{\iota}_{\widehat{x}})_{{\widehat{x}} \in \widehat{\Omega}}$ such that, for all ${\widehat{x}} \in \widehat{\Omega}$, $$ \check{\iota}_{\widehat{x}}(0,0) = x, \quad  \check{\iota}_{\widehat{x}}(z,0) = \Phi^u_{\widehat{x}}(z), \quad \check{\iota}_{\widehat{x}}(0,y) = \Phi_{\widehat{x}}^s(y) .$$
One can also assume the dependence in ${\widehat{x}}$ of these to be Hölder regular, since the stable/unstable laminations are Hölder (in our context, they are even $C^{1+\alpha}$, see section 4.1 and 4.2 for details). Define $\check{f}_{\widehat{x}} := \check{\iota}_{\widehat{f}({\widehat{x}})}^{-1} \circ f \circ \check{\iota}_{{\widehat{x}}}$. This is a smooth map defined on a neighborhood of zero, with a (hyperbolic) fixed point at zero. Notice also that $(d \check{f}_{\widehat{x}})_0$ is a diagonal map with coefficients $(\lambda_x,\mu_x)$. Those coordinates won't do, but we can straighten them into doing what we want. Define
$$ \check{\rho}_{\widehat{x}}^u(z) := \sum_{n=1}^\infty \left( \ln | \pi_y \partial_y \check{f}_{\widehat{f}^{-n}({\widehat{x}})}(\lambda_x^{\langle -n \rangle} z,0)| - \ln |\mu_{f^{-n}({x})}| \right) $$
and $$ \check{\rho}_{\widehat{x}}^s(y) := \sum_{n=0}^\infty \left( \ln |\pi_z \partial_z \check{f}_{\widehat{f}^{n}({\widehat{x}})}(0,\mu_x^{\langle n \rangle} y)| - \ln |\lambda_{f^{n}(x)}| \right) .$$
Finally, set $\check{\mathcal{D}}_{\widehat{x}}^u(z,y) := (z, y  e^{\check{\rho}_{\widehat{x}}^u(z)} )$, $\check{\mathcal{D}}_{\widehat{x}}^s(z,y) := (z e^{-\check{\rho}_{\widehat{x}}^s(y)}, y)$, $\check{\mathcal{D}}_{\widehat{x}}:= \check{\mathcal{D}}_{\widehat{x}}^u \circ \check{\mathcal{D}}_{\widehat{x}}^s$ and $ \iota_{\widehat{x}} := \check{\iota}_{\widehat{x}} \circ \check{\mathcal{D}}_{\widehat{x}}$. Let us check that $f_{\widehat{x}} := \iota_{\widehat{f}({\widehat{x}})}^{-1} \circ f \circ \iota_{\widehat{x}}$ satisfies the desired relations.
First of all, notice that $\check{\rho}_{\widehat{x}}^u$ and $\check{\rho}_{\widehat{x}}^s$ are smooth and satisfy $\check{\rho}_{\widehat{x}}^u(0)=\check{\rho}_{\widehat{x}}^s(0)=0$. In particular, $\check{\mathcal{D}}_{\widehat{x}}, \check{\mathcal{D}}_{\widehat{x}}^u$ and $ \check{\mathcal{D}}_{\widehat{x}}^s$ are smooth, and coincide with the identity on $\{ (z,y) \ , \ z=0 \text{ or } y=0\}$.
Moreover, $$ \check{\rho}_{\widehat{f}({\widehat{x}})}^u({\lambda}_x z) = \ln |\pi_y \partial_y \check{f}_{{\widehat{x}}}(z,0)| - \ln |\mu_{x}| + \check{\rho}_{{\widehat{x}}}^{u}(z) $$
and $$ \check{\rho}_{\widehat{x}}^s(y) = \ln |\pi_z \partial_z \check{f}_{\widehat{x}}(0,y)| - \ln |\lambda_x| + \check{\rho}^s_{\widehat{f}({\widehat{x}})}(\mu_x y) .$$
Now let us write $f_{\widehat{x}}$ in terms of $\check{f}_{\widehat{x}}$: we have
$$ f_{\widehat{x}} = \iota_{\widehat{f}({\widehat{x}})}^{-1} \circ f \circ \iota_{\widehat{x}} = (\check{\mathcal{D}}_{\widehat{f}({\widehat{x}})}^s)^{-1} \circ (\check{\mathcal{D}}_{\widehat{f}({\widehat{x}})}^u)^{-1} \circ \check{f}_{\widehat{x}} \circ \check{\mathcal{D}}_{{\widehat{x}}}^u \circ \check{\mathcal{D}}_{{\widehat{x}}}^s. $$
Hence:
$$ (df_{\widehat{x}})_{(z,0)} = d((\mathcal{D}_{\widehat{f}({\widehat{x}})}^s)^{-1})_{(\lambda_x z,0)} \circ d((\mathcal{D}_{\widehat{f}({\widehat{x}})}^u)^{-1})_{(\lambda_x z,0)} \circ (d\check{f}_{\widehat{x}})_{(z,0)} \circ (d \mathcal{D}_{\widehat{x}}^u)_{(z,0)} \circ (d \mathcal{D}_{\widehat{x}}^s)_{(z,0)} $$
Abusing a bit notations, we can write in matrix form:
$$ (df_{\widehat{x}})_{(z,0)} = \begin{pmatrix} 1 & (*) \\ 0 & 1 \end{pmatrix} \begin{pmatrix} 1 & 0 \\ 0 & e^{-\check{\rho}_{\widehat{f}({\widehat{x}})}(\lambda_x z)} \end{pmatrix} \begin{pmatrix} \lambda_x & \pi_z \partial_y \check{f}_{\widehat{x}}(z,0) \\ 0 & \pi_y \partial_y \check{f}_{\widehat{x}}(z,0) \end{pmatrix} \begin{pmatrix} 1 & 0 \\ 0 & e^{\check{\rho}_{\widehat{x}}^u(z)} \end{pmatrix} \begin{pmatrix} 1 & (*) \\ 0 & 1 \end{pmatrix}  $$
$$ = \begin{pmatrix} \lambda_x & (*) \\ 0 & e^{-\check{\rho}_{\widehat{f}({\widehat{x}})}^u(\lambda_x z) + \check{\rho}_{\widehat{x}}^u(z)} \pi_y \partial_y \check{f}_{\widehat{x}}(z,0) \end{pmatrix} = \begin{pmatrix} \lambda_x & (*) \\ 0 & \mu_x \end{pmatrix} ,$$
which implies in particular that $\pi_y \partial_y f_{\widehat{x}}(z,0) = \mu_x$. A similar computation shows that
$$ (df_{\widehat{x}})_{(0,y)} = \begin{pmatrix} \lambda_x & 0 \\ (*) & \mu_x \end{pmatrix} .$$
In particular, $\pi_z \partial_z f_{\widehat{x}}(0,y) = \lambda_x$. 
\end{proof}

\begin{remark}
Notice the following convenient property. As soon as the quantities written make sense, we have the identities:
$$ (df^{\langle n \rangle}_{\widehat{x}})_{(z,0)} = \begin{pmatrix} \lambda_x^{\langle n \rangle} & (*) \\ 0 & \mu_x^{\langle n \rangle} \end{pmatrix} \quad  ; \quad (df_x^{\langle n \rangle})_{(0,y)} = \begin{pmatrix} \lambda_x^{\langle n \rangle} & 0 \\ (*) & \mu_x^{\langle n \rangle} \end{pmatrix}$$
where $f_{\widehat{x}}^{\langle n \rangle} := \iota_{\widehat{f}^{n}({\widehat{x}})} \circ f^n \circ \iota_{\widehat{x}} : (-(\rho_1/\rho_0)^n \rho_0,(\rho_1/\rho_0)^n \rho_0)^2 \longrightarrow (-\rho_0, \rho_0)^2$.
\end{remark}

This coordinate system is not \say{canonically attached} to the dynamics, since the behavior of $\iota_x$ outside the \say{cross} $\{ (z,y) \ , \ zy=0 \}$ might be completely arbitrary. But the behavior of those coordinates near the cross seems to give rise to less arbitrary objects. Those objects will be called \say{templates} in these notes. They are inspired from the \say{templates} appearing in \cite{TZ20}.

\begin{definition}[Templates, dual version]
Let $x \in \Omega$. A template based at $x$ is a continuous 1-form $\xi_x:W^u_{loc}(x) \rightarrow \Omega^1(M)$ such that:
$$ \forall z \in W^u_{loc}(x), \ \text{Ker}(\xi_x)_z \supset E^u(z). $$
We will denote by $\Xi(x)$ the space of templates based at $x$.
\end{definition}

\begin{remark}
Notice that, since $E^u(z)$ moves smoothly along the unstable local manifold $W^u_{loc}(x)$, it makes sense to consider \textbf{smooth} templates. Notice further that, since $(df)_z E^u(z) = E^u(f(z))$, the diffeomorphism $f$ acts naturally on templates by taking the pullback. This yields a map $f^* : \Xi(f(x)) \rightarrow \Xi(\widehat{x})$.
\end{remark}

\begin{lemma}[Some important templates.]
There exists a family $(\xi_{\widehat{x}}^s)_{{\widehat{x}} \in \widehat{\Omega}}$ of smooth templates, where $\xi_{\widehat{x}}^s \in \Xi(x)$, that satifies the following invariance relation:
$$\forall z \in W^u_{loc}(x), \ (f^* (\xi_{\widehat{f}(\widehat{x})}^s))_{z} = \mu_x (\xi_{\widehat{x}}^s)_{z} .$$
Moreover, the dependence in $\widehat{x}$ of $\xi_{\widehat{x}}^s$ is Hölder.
\end{lemma}

\begin{proof}
For all ${\widehat{x}} \in \widehat{\Omega}$, define $\xi_{\widehat{x}}^s := (\iota_{\widehat{x}}^{-1})^*(dy)$. It is clear that this defines a smooth template at ${\widehat{x}}$. Furthermore, using remark 5.3.7 for $n=1$:
$$ f^* \xi_{\widehat{f}({\widehat{x}})}^s = f^* ((\iota_{\widehat{f}({\widehat{x}})}^{-1})^* (dy)) = (\iota_{{\widehat{x}}}^{-1})^* ((f_{\widehat{x}})^*(dy)) = (\iota_{{\widehat{x}}}^{-1})^* (d (\pi_y f_{\widehat{x}})) = (\iota_{{\widehat{x}}}^{-1})^* (\mu_x dy) = \mu_x \xi_{{\widehat{x}}}^s .$$
The dependence is Hölder because $(\iota_{\widehat{x}})$ depends on ${\widehat{x}}$ in a Hölder manner.
\end{proof}

One could say that one of the reasons these templates are usefull is because they give rise to an isomorphism $E^s(z) \simeq TM/E^u(z)$: the point being that one of those two line bundle is $C^{\infty}$ along local unstable manifolds, whereas the other is only $C^{1+\alpha}$. It is natural to try to find a \say{vector field} version of those templates. We suggest a way to proceed in the following.

\begin{definition}[Templates, vector version]
Let $x \in \Omega$. A (vector) template based at $x$ is a continuous section of the line bundle $TM/E^u$ along $W^u_{loc}(x)$. We will denote by $\Gamma(x)$ the space of (vector) templates at $x$. 
\end{definition}

\begin{remark}
If $X$ is some continuous vector field defined along $W^u_{loc}(x)$, we can take its class modulo $E^u$ to get a (vector) template $[X]$. Notice also that it makes sense to talk about smooth (vector) templates. Notice further that $f$ acts naturally on (vector) templates, since $(df)_z E^u(z) = E^u(f(z))$, by taking the pushforward. This define a map $f_* : \Gamma(x) \rightarrow \Gamma(f(x)))$.
\end{remark}

\begin{lemma}[Some important templates.]
There exists a family $([\partial_s^{\widehat{x}}])_{{\widehat{x}} \in \widehat{\Omega}}$ of smooth (vector) templates, where $[\partial_s^{\widehat{x}}] \in \Gamma(x)$, such that:
$$ \forall z \in W^u_{loc}(x), \ (f_*[\partial_s^{\widehat{x}}])_{f(z)} = \mu_x [\partial_s^{\widehat{f}({\widehat{x}})}]_{f(z)}.$$
Moreover, the dependence in ${\widehat{x}}$ of 
$([\partial_s^{\widehat{x}}])$ is Hölder.
\end{lemma}

\begin{proof}
For all ${\widehat{x}} \in \widehat{\Omega}$, define $\partial_s^{\widehat{x}} := (\iota_{\widehat{x}})_*(\partial/\partial y)$ along $W^u_{loc}(x)$. This is smooth. Moreover, using remark 5.3.7, we find:
$$ f_* \partial_s^{\widehat{x}} = f_* (\iota_{\widehat{x}})_*(\partial_y)= (\iota_{\widehat{f}({\widehat{x}})})_* ((f_x)_* \partial_y) = (\iota_{\widehat{f}({\widehat{x}})})_* ( \mu_x \partial_y + (*) \partial_z ) = \mu_x \partial_s^{\widehat{f}({\widehat{x}})} + (*) \partial_u .$$
Hence, taking the class modulo $E^u$, we find
$$ f_*[\partial_s^{\widehat{x}}] = \mu_x [\partial_{s}^{\widehat{f}({\widehat{x}})}] ,$$
which is what we wanted. The dependence in ${\widehat{x}}$ is then Hölder because of the properties of $\iota_{\widehat{x}}$.
\end{proof}

\begin{remark}[A quick duality remark]
We can define a kind of \say{duality bracket} \newline $\Xi(x) \times \Gamma(x) \rightarrow C^0(W^u_{loc}(x),\mathbb{R})$ by the following formula:
$$ \langle \xi , [X] \rangle := \xi(X) .$$
Our special templates $(\xi_{\widehat{x}}^s)$ and $([\partial_s^{\widehat{x}}])$ can be chosen normalised so that $\langle \xi_{\widehat{x}}^s, [\partial_s^{\widehat{x}}] \rangle = 1$. This will not be useful for us, but this is an indication that $\Xi({\widehat{x}})$ and $\Gamma({\widehat{x}})$ could countain the same informations.
\end{remark}

\begin{remark}[Templates acting on a space of functions]
It is natural to search for a space of functions on which (vector) templates could acts. A way to do it is as follow.
For each $x \in \Omega$, define $\mathcal{F}(x)$ as the set of functions $h_x$ defined on a neighborhood of $W^u_{loc}(x)$ that are $C^1$ along the stable direction and that vanish along $W^u_{loc}(x)$. In this case, for any point $z \in W^u_{loc}(x)$, we know that $\partial_s h_x(z)$ makes sense, and we know that $\partial_u h_x(z) = 0$ also makes sense. One can then make (vector) templates $[X]$ acts on $h_x$ by setting:
$$\forall z \in W^u_{loc}(x), \  ([X] \cdot h_x)(z) := (X \cdot h_x)(z) .$$
This is well defined.
In the particular case where $[X] = [\partial_s^{\widehat{x}}]$, we get the formula:
$$ \forall z \in (-\rho_0,\rho_0), \ ([\partial_s^{\widehat{x}}] \cdot h_x)(\Phi_{{\widehat{x}}}^u(z)) = \partial_y (h_x \circ \iota_{\widehat{x}})(z,0) .$$

Notice that $f$ acts naturally on these space of functions, by taking a pullback $f^* : \mathcal{F}(f(x)) \rightarrow \mathcal{F}(x)$. If we fix $h_{f(x)} \in \mathcal{F}(f(x))$, and if we set $h_x := h_{f(x)} \circ f \in \mathcal{F}(x)$, notice finally that one can write
$$ [\partial_s^{\widehat{x}}] \cdot h_{x} = [\partial_s^{\widehat{x}}] \cdot f^*h_{f(x)} = f_* [\partial_s^{\widehat{x}}] \cdot h_{f(x)} = \mu_x [\partial_s^{\widehat{f}({\widehat{x}})}] \cdot h_{f(x)}.$$
\end{remark}

\begin{lemma}[Changing basepoint]
Let $x_0 \in \Omega$. For $x_1 \in \Omega \cap W^u_{loc}(x)$, let $$H(x_0,x_1) = \exp\Big{(}\sum_{n=0}^\infty \left(\ln \mu_{f^{-n}(x_0)} - \ln \mu_{f^{-n}(x_1)} \right)\Big{)} .$$ Then:
$$\forall z \in W^u_{loc}(x), \quad ([\partial_s^{\widehat{x_1}}])_{z} = \pm H(x_1,x_0) ([\partial_s^{\widehat{x_0}}])_{z},$$
where the sign is positive iff $\widehat{x_0}$ and $\widehat{x_1}$ share the same local orientation.\newline In particular, for any $x \in \Omega$, $ [\partial_s^{\mathfrak{s}(\widehat{x})}] = - [\partial_s^{\widehat{x}}] $.
\end{lemma}

\begin{proof}
First of all, notice that if we fix some $x \in \Omega$, then our special vector template satisfies (only at $x$ for now):
$$ [\partial_s^{\mathfrak{s}(\widehat{x_0})}]_{x_0} = - [\partial_s^{\widehat{x_0}}]_x. $$
Now, remember that $TM/E^u$ is a line bundle, and that $[\partial_s^{\widehat{x_0}}]$ doesn't vanish. In particular, there exists a function $a_{\widehat{x_0},\widehat{x_1}}:W^u_{loc}(x_0) \rightarrow \mathbb{R}$ such that:
$$ \forall z \in W^u_{loc}(x_0), \quad ([\partial_s^{\widehat{x_1}}])_{z} = a_{\widehat{x_0},\widehat{x_1}}(z) ([\partial_s^{\widehat{x_0}}])_{z} .$$
The main point is to show that $a_{\widehat{x_0},\widehat{x_1}}$ is $z$-constant. Since the family $([\partial_s^{\widehat{x}}])$ depends in $\widehat{x}$ in a Hölder manner (and locally uniformly in $z$) we know that $a_{\widehat{x_0},\widehat{x_1}}(z) = \pm (1 + O(d(x_0,x_1)^{\alpha}))$ for some $\alpha$, where the sign depends on the orientations. Recall that $f$ was supposed orientation preserving. The invariance properties of those (vector) templates yields an invariance property for $a_{\widehat{x_0},\widehat{x_1}}(z)$:
$$ \forall z \in W^u_{loc}(x), \  a_{\widehat{x_0},\widehat{x_1}}(z) = \frac{\mu_{x}^{\langle - n \rangle}}{\mu_{\tilde{x}}^{\langle -n \rangle}} a_{\widehat{f}^{-n}(\widehat{x_0}),\widehat{f}^{-n}(\widehat{x_1})}(f^{-n}(z)).$$
Taking the limit as $n \rightarrow + \infty$ gives the result.
\end{proof}


\section{The local behavior of $\Delta$}

Now that we have the right tools, we come back to our study of $\Delta$. Recall that $\Delta : \widetilde{\text{Diag}} \rightarrow \mathbb{R}$ is defined as
$$\Delta(p,q) := \sum_{n \in \mathbb{Z}} T_n(p,q) ,$$
where $T_n(p,q) := \tau(f^n(p)) - \tau(f^n([p,q])) - \tau(f^n([q,p])) + \tau(f^n(q))$. The goal of this section is to understand the local behavior of $\Delta(p,q)$ when $q$ gets close to $p$. To this aim, we cut the sum in two, like so: 
$$ \Delta^+(p,q) := \sum_{n=0}^\infty T_n(p,q) \quad ; \quad \Delta^-(p,q) := \sum_{n=1}^\infty T_{-n}(p,q),$$
so that $\Delta = \Delta^- + \Delta^+$. Intuitively, $\Delta^-(p,\cdot)$ is smooth along the unstable direction but is Hölder in the stable direction. The behavior of $\Delta^+(p,\cdot)$ is the opposite: Hölder along the unstable direction and smooth in the stable direction. To catch oscillations for $\Delta$, we will fix an unstable manifold close to $W^u_{loc}(p)$ and try to understand the behavior of $\Delta(p,\cdot)$ along this manifold. The sum $\Delta^-(p,\cdot)$ will be approximated by a polynomial, and $\Delta^+(p,\cdot)$ will be approximated by an autosimilar function behaving like a Weierstrass function. We will thus reduce catching oscillations of $\Delta(p,\cdot)$ to understanding the oscillations of $\Delta^+(p,\cdot)$, along the unstable foliation, modulo polynomials. Let us begin by the local description of $\Delta^+(p,\cdot)$.\\

Let us fix some $p \in \Omega$, and set:
$$ \Delta_p^+(q) := \Delta^{+}(p,q) \quad , \quad T_{p,n}(q) := T_n(p,q) .$$
For each $p$ and $n$, $T_{p,n}$ is $C^{1+\alpha}$, and moreover taking the derivative along the local stable lamination yields: $\partial_s T_{p,n}(q) = -(\partial_s \tau)( f^n([p,q])) |\mu_{[p,q]}^{\langle n\rangle}| \partial_s \pi_p(q) + \partial_s \tau(f^n(q))) |\mu_q^{\langle n \rangle}|$, where $\pi_p(q) := [p,q]$. It follows that $\Delta_{p}^+$ is indeed $C^{1}$ along the local stable lamination. Moreover, $T_{p,n}$ vanish on $W^u_{loc}(p)$, and so does $\Delta_p^+$. It follows that
$$ \Delta_p^{+} \in \mathcal{F}(p) \quad, \quad T_{p,n} \in \mathcal{F}(p) ,$$
where $\mathcal{F}(p)$ denotes the space of function defined in remark 5.3.17. Now one of the main point of this section is to make the following intuition rigorous: since $\Delta_p^+ \in \mathcal{F}(p)$, we can do a Taylor expansion along the stable foliation. We get an expression like this (denoting $r=[q,p] \in W^u_{loc}(p) \cap W^s_{loc}(q)$):
$$ \Delta_p^+(q) \simeq  \Delta_p^+(r) \pm \partial_s \Delta_p^+(r) d^u(q,r) = \pm \partial_s \Delta_p^+(r) d^u(q,r). $$
It is then convenient to replace $\partial_s \Delta_p^+$ by $[\partial_s^{\widehat{p}}] \cdot \Delta_{p}^+$, to use more easily the autosimilarity properties of $\Delta_p^+$. We then expect all the oscillations of $\Delta_p^+$ to come from wild oscillations of $[\partial_s^{\widehat{p}}] \cdot \Delta_{p}^+$. Now let us makes thing rigorous: the fact that $\Delta_p^+ \in \mathcal{F}(p)$ ensure that the next definition makes sense.

\begin{definition}
For each ${\widehat{x}} \in \widehat{\Omega}$, for each $z \in \Omega_{\widehat{x}}^u \subset \mathbb{R}$, define $X_{\widehat{x}} \in C^{\alpha}((-\rho_0,\rho_0),\mathbb{R}) $ by:
$$ X_{\widehat{x}}(z) := ([\partial_s^{\widehat{x}}] \cdot \Delta_{x}^+)(\Phi_{\widehat{x}}^u(z)).$$
The family $(X_{\widehat{x}})_{x \in \Omega}$ depends on ${\widehat{x}}$ in a Hölder manner. Moreover, $X_{\mathfrak{s}(\widehat{x})}(z) = -X_{\widehat{x}}(-z) $.
\end{definition}

\begin{lemma}[autosimilarity]
We have $X_{\widehat{x}}(0)=0$. Moreover, the family $(X_{\widehat{x}})_{\widehat{x} \in \widehat{\Omega}}$ satisfies the following autosimilarity relation:
$$\forall z \in \Omega_{\widehat{x}}^u \cap (-\rho_0 \lambda_x^{-1},\rho_0 \lambda_x^{-1}), \ X_{\widehat{x}}(z) = \widehat{\tau}_{\widehat{x}}(z) + \mu_x X_{\widehat{f}({\widehat{x}})}(\lambda_x z) ,$$
where $\widehat{\tau}_{\widehat{x}}(z) := ([\partial_{\widehat{x}}^s] \cdot T_{x,0})(\Phi_{\widehat{x}}^u(z)) \in C^{\alpha}$. Notice that $\widehat{\tau}_{\mathfrak{s}(\widehat{x})}(z) = - \widehat{\tau}_{\widehat{x}}(-z)$.
\end{lemma}

\begin{proof}
Notice that $T_{p,n+1}(q) = T_{f(p),n}(f(q))$. It follows that:
$$ \Delta_p^+(q) = T_{p,0}(q) + \sum_{n=0}^\infty T_{f(p),n}(f(q)) = T_{p,0}(q) + \Delta_{f(p)}^+(f(q)) .$$
Making the vector template $[\partial_s^{\widehat{p}}]$ acts on this along $W^u_{loc}(p)$ yields (using the invariance properties of the family $([\partial_s^{\widehat{x}}])_{\widehat{x} \in \widehat{\Omega}}$):
$$([\partial_s^{\widehat{p}}] \cdot \Delta_p^+) = ([\partial_s^{\widehat{p}}] \cdot T_{p,0}) + ([\partial_s^{\widehat{p}}] \cdot (\Delta_{f(p)}^+ \circ f)) = ([\partial_s^{\widehat{p}}] \cdot T_{p,0}) + \mu_p ([\partial_s^{\widehat{f}({\widehat{p}})}] \cdot \Delta_{f(p)}^+) \circ f  .$$
Testing this equality on $\Phi_{\widehat{p}}^u(z)$ gives the desired relation, since $f \circ \Phi_{\widehat{p}}^u = \Phi_{\widehat{f}({\widehat{p}})}^u \circ (\lambda_p Id )$.
\end{proof}

\begin{lemma}[regularity of $\widehat{\tau}_{\widehat{x}}$]
The function $\widehat{\tau}_{\widehat{x}}$ is smoother than expected: it is $C^{1+\alpha}((-\rho_0,\rho_0),\mathbb{R})$. It vanish at $z=0$, and its derivative at zero is:
$$ ({\widehat{\tau}_{\widehat{x}}})'(0) = \partial_z \partial_y (\tau \circ \iota_{{\widehat{x}}})(0,0) + n_{\widehat{x}}'(0) \partial_z(\tau \circ \iota_{\widehat{x}})(0,0),$$
where $n_{\widehat{x}}(z) \in C^{1+\alpha}$ is defined such that $\partial_y + n_{\widehat{x}}(z) \partial_z \in \iota_{\widehat{x}}^{-1}(E^s)$ points in the stable direction at coordinates $(z,0)$. Notice that $({\widehat{\tau}_{\widehat{x}}})'(0) = ({\widehat{\tau}_{\mathfrak{s}(\widehat{x})}})'(0)$.
\end{lemma}

\begin{proof}

Let us do an explicit computation of $\widehat{\tau}_{\widehat{x}}$. By definition of $[\partial_s^{\widehat{x}}]$:
$$ \widehat{\tau}_{\widehat{x}}(z) = ([\partial_s^{\widehat{x}}] \cdot T_{x,0})(\Phi_{\widehat{x}}^u(z)) = \partial_y (T_{x,0} \circ \iota_{\widehat{x}})(z,0) .$$
Recall that $T_{x,0}(\iota_{\widehat{x}}(z,y)) = \tau(\iota_{\widehat{x}}(z,y)) - \tau([x,\iota_{\widehat{x}}(z,y)])-\tau([\iota_{\widehat{x}}(z,y),x]) + \tau(x) \in C^{1+\alpha}$. Define $$\pi_{\widehat{x}}^s(z,y) := \iota_{\widehat{x}}^{-1}( [x,\iota_{\widehat{x}}(z,y)] ) \in \{ (0,y') , y' \in (-\rho_0,\rho_0) \}$$ and $$\pi_{\widehat{x}}^u(z,y) := \iota_{\widehat{x}}^{-1}( [\iota_{\widehat{x}}(z,y),x] ) \in \{ (z',0) , z' \in (-\rho_0,\rho_0) \}.$$
Define also $\tau_{\widehat{x}} := \tau \circ \iota_{\widehat{x}}$. Then:
$$ T_{x,0}(\iota_{\widehat{x}}(z,y)) = \tau_{\widehat{x}}(z,y) - \tau_{\widehat{x}}(\pi_{\widehat{x}}^u(z,y)) - \tau_{\widehat{x}}(\pi_{\widehat{x}}^s(z,y)) + \tau_{\widehat{x}}(0). $$
For each point $z \in (-\rho_0,\rho_0)$, let $\vec{N}_{\widehat{x}}(z)$ be a vector pointing along the direction $\iota_{\widehat{x}}^{-1}(E^s)$, and normalize it so that $\vec{N}_{\widehat{x}}(z) = \partial_y + n_{\widehat{x}}(z) \partial_z$. By regularity of $E^s$, we can choose $\vec{N}_{\widehat{x}}(z)$ to be $C^{1+\alpha}$ in $z$. We can then, for each ${\widehat{x}},z$, find a (smooth) path $ t \mapsto \gamma_{\widehat{x}}(z,t)$ such that $\pi_{\widehat{x}}^{u} \circ \gamma_{\widehat{x}}(z,t) = (z,0) $ and such that $\partial_t \gamma_{\widehat{x}}(z,0) = \vec{N}_{\widehat{x}}(z)$. (We just follow the stable lamination in coordinates.) Using this path, we can compute the derivative of $T_{x,0} \circ \iota_{\widehat{x}}$ as follow:
$$\partial_y (T_{x,0} \circ \iota_{\widehat{x}})(z,0) = \Big{(}(\partial_y + n_{\widehat{x}}(z) \partial_z) \cdot (T_{x,0} \circ \iota_{\widehat{x}}) \Big{)}(z,0) = \frac{d}{dt} T_{x,0}(\iota_{\widehat{x}}(\gamma_{\widehat{x}}(z,t)))_{|t=0}$$
$$ = \frac{d}{dt}_{|t=0} \Big{(} \tau_{\widehat{x}}( \gamma_{\widehat{x}}(z,t) ) - \tau_{\widehat{x}}(\pi_{\widehat{x}}^u( \gamma_{\widehat{x}}(z,t) )) - \tau_{\widehat{x}}(\pi_{\widehat{x}}^s( \gamma_{\widehat{x}}(z,t))) + \tau_{\widehat{x}}(0) \Big{)}  $$
$$ = \partial_y \tau_{\widehat{x}}(z,0) + n_{\widehat{x}}(z) \partial_z \tau_{\widehat{x}}(z,0) - (d\tau_{\widehat{x}})_{(0,0)} \circ (d\pi_{\widehat{x}}^s)_{(z,0)} ( \vec{N}_{\widehat{x}}(z) ) $$
$$ = \partial_y \tau_{\widehat{x}}(z,0) + n_{\widehat{x}}(z) \partial_z \tau_{\widehat{x}}(z,0) - \partial_y \tau_{\widehat{x}}(0,0) \pi_y( \partial_y \pi_{\widehat{x}}^s(z,0)),$$
since $(d\pi_{\widehat{x}}^s)_{(z,0)}(\partial_z) = 0$, as $\pi_{\widehat{x}}^s(z,0) = (0,0)$, and since $\pi_z \circ \pi_{\widehat{x}}^s = 0$.
In this expression, everything is $C^{1+\alpha}$ (since $\tau \in \text{Reg}_u^{1+\alpha}$); except eventually the last term $\eta_{\widehat{x}}(z) := \pi_y \partial_y \pi_{\widehat{x}}^s(z,0)$. Let us prove that $\eta_{\widehat{x}}(z) $ is, in fact, constant and equal to one. First of all, the maps $(\eta_{\widehat{x}})$ are at least continuous (and this, uniformly in ${\widehat{x}}$). Moreover, we have $\pi_y \pi_{\widehat{x}}^s(0,y) = y$, and hence $\eta_{\widehat{x}}(0)=1$. To conclude, let us use the fact that the stable lamination is $f$ invariant. This remark, written in coordinates, yields (as soon as the relation makes sense):
$$ \pi_{\widehat{x}}^s = f_{\widehat{f}^{-n}({\widehat{x}})}^{\langle n \rangle} \circ \pi_{\widehat{f}^{-n}({\widehat{x}})}^s \circ f_{\widehat{x}}^{\langle -n \rangle}. $$
Taking the differential at $(z,0)$ yields, using remark 5.3.7:
$$ \forall n \geq 0, \forall z \in (-\rho_0,\rho_0), \ (d\pi_{\widehat{x}}^s)_{(z,0)} = (df_{\widehat{f}^{-n}({\widehat{x}})}^{\langle n \rangle})_{(0,0)} \circ (d\pi^s_{\widehat{f}^{-n}({\widehat{x}})})_{(\lambda_x^{\langle -n \rangle}z,0)} \circ (df_{\widehat{x}}^{\langle -n \rangle})_{(z,0)} $$ $$=\begin{pmatrix} \lambda_{f^{-n}(x)}^{\langle n \rangle} & 0 \\ 0 & \mu_{f^{-n}(x)}^{\langle n \rangle} \end{pmatrix} \begin{pmatrix} 0 & 0 \\ 0 & \eta_{\widehat{f}^{-n}({\widehat{x}})}(\lambda_x^{\langle -n \rangle} z) \end{pmatrix}  \begin{pmatrix} \lambda_x^{\langle -n \rangle} & (*) \\ 0 & \mu_{x}^{\langle -n \rangle} \end{pmatrix} = \begin{pmatrix} 0 & 0 \\ 0 & \eta_{\widehat{f}^{-n}({\widehat{x}})}(\lambda_x^{\langle -n \rangle} z) \end{pmatrix} .$$
It follows that: 
$$ \forall z, \forall n \geq 0, \ \eta_{\widehat{x}}(z) = \eta_{\widehat{f}^{-n}({\widehat{x}})}(\lambda_{{x}}^{\langle -n \rangle} z) \underset{n \rightarrow \infty}{\longrightarrow} 1 .$$
In conclusion, we get the following expression for $\widehat{\tau}_{\widehat{x}}$:
$$ \widehat{\tau}_{\widehat{x}}(z) = \partial_y \tau_{\widehat{x}}(z,0) + n_{\widehat{x}}(z) \partial_z \tau_{\widehat{x}}(z,0) - \partial_y \tau_{\widehat{x}}(0,0), $$
which is a $C^{1+\alpha}$ function that vanish at zero. Let us compute its derivative at zero: we have
$$ (\widehat{\tau}_{\widehat{x}})'(0) = \partial_z \partial_y \tau_{\widehat{x}}(0,0) + n_{\widehat{x}}'(0) \partial_z \tau_{\widehat{x}}(0,0) + n_{\widehat{x}}(0) \partial_z^2 \tau_x(0,0).$$
The fact that $n_{\widehat{x}}(0)=0$ gives us the desired formula.\end{proof}

\begin{remark}
Notice that, denoting by $\pi_p^S(q) := [p,q]$, our previous arguments prove that the map $\partial_s \pi_p^S : W^u_{loc}(p) \cap \Omega \rightarrow \mathbb{R}$ is $C^{1+\alpha}$ (and this, uniformly in $p$).
\end{remark}




\begin{remark}
Recall that, in our area-preserving context, there exists a Hölder map $h:\Omega \rightarrow \mathbb{R}$ such that $\lambda_x \mu_x = \exp(h(f(x))-h(x))$. Fix one such $h$ for the rest of the Chapter.
For each $\tau \in \text{Reg}_u^{1+\alpha}(\Omega,\mathbb{R})$, let us denote by $\Phi_\tau \in C^\alpha(\Omega,\mathbb{R})$ the map defined by $$\Phi_\tau : x \in \Omega \longmapsto (\widehat{\tau}_{\widehat{x}})'(0) e^{h(x)} \in \mathbb{R},$$ where $\widehat{x}$ is any orientation in the fiber $p^{-1}(x) \in \widehat{\Omega}$. This doesn't depend on the choice of $\widehat{x}$. The linear map $ \tau \in \text{Reg}_u^{1+\alpha}(\Omega,\mathbb{R}) \mapsto \Phi_\tau \in C^{\alpha}(\Omega,\mathbb{R})$ is obviously continuous. Moreover, it is easy to see that, generically in $\tau \in \text{Reg}_u^{1+\alpha}$, the map $\Phi_\tau$ is not cohomologous to zero. Indeed, one can take a fixed point $p_0$ (or a periodic ordit) and look at the value of $\Phi_\tau(p_0)$: if it is zero, then it is easy to $\text{Reg}_u^{1+\alpha}-$perturb $\tau$ on a neighborhood of $p_0$ so that $\Phi_{\tau}(p_0)$ becomes non-zero.
\end{remark}

\begin{lemma}[Change of basepoint]
Let $x_0 \in \Omega$. Let $x_1 \in W^u_{loc}(x)$ be close enough to $x_0$. Then, there exists $\text{Aff}_{\widehat{x_0},\widehat{x_1}}:\mathbb{R} \rightarrow \mathbb{R}$, an (invertible) affine map, such that: 
$$ \text{Aff}_{{\widehat{x_0}},{\widehat{x_1}}}\Big{(} X_{{\widehat{x_1}}}(z) \Big{)}=  X_{\widehat{x_0}}(\text{aff}_{{\widehat{x_0}},{\widehat{x_1}}}(z)), $$
where $\text{aff}_{{\widehat{x_0}},{\widehat{x_1}}} = (\Phi_{{\widehat{x_0}}}^u)^{-1} \circ \Phi_{{\widehat{x_1}}}^{u}$ is the affine change of charts defined in lemma 5.3.7.
Moreover, there exists $C \geq 1$ and $\alpha>0$ such that $\ln |{\text{Aff}_{{\widehat{x_0}},{\widehat{x_1}}}}'(0)| \leq C d(x_0,x_1)^\alpha$.
\end{lemma}

\begin{proof}
Let $x_0 \in \Omega$ and let $x_1 \in W^u_{loc}(x)$ be close enough to $x_0$. We have, for $q$ in a neighborhood of $x_1$:
$$ \Delta^+( x_1, q ) = \Delta^+( x_1, [x_0,q] ) + \Delta^+(x_0,q). $$
(Picture two rectangles with one shared side, the first one touching $x_1$, $x_0$, $[x_0,q]$, and the second one touching $x_0$, $[x_0,q]$ and $q$.) We differentiate (w.r.t. $q$) with the vector template $[\partial_s^{\widehat{x_1}}]$ along $W^u_{loc}(x_1)$ to find:
$$ X_{\widehat{x_1}}(z) = [\partial_s^ {\widehat{x_1}}] \cdot (\Delta_{x_1}^+ \circ [x_0,\cdot])(\Phi_{{\widehat{x_1}}}^u(z)) + ([\partial_s^{{\widehat{x_1}}}] \cdot \Delta_{x_0}^+)(\Phi_{{\widehat{x_1}}}^u(z)) .$$
The first thing to recall is that $\Phi_{{\widehat{x_1}}}^u = \Phi_{\widehat{x_0}}^u \circ (\Phi_{\widehat{x_0}}^{u})^{-1} \circ \Phi_{{\widehat{x_1}}}^u = \Phi_{\widehat{x_0}}^u \circ \text{aff}_{{\widehat{x_0}},{\widehat{x_1}}} $, and moreover, by Lemma 5.3.18, $[\partial_s^{{\widehat{x_1}}}] = \pm H(x_0,x_1) [\partial_s^{\widehat{x_0}}]$. From this, we see that the last term is $\pm H(x_0,x_1) X_{\widehat{x_0}}(\text{aff}_{{\widehat{x_0}},{\widehat{x_1}}}(z))$. To conclude, we only need to show that
$$ [\partial_s^ {{\widehat{x_1}}}] \cdot (\Delta_{{\widehat{x_1}}}^+ \circ [x_0,\cdot]) = \pm H(x_0,x_1) [\partial_s^{{\widehat{x_0}}}] \cdot (\Delta_{{\widehat{x_1}}}^+ \circ [x_0,\cdot])  $$
is constant along $W^u_{loc}({\widehat{x_0}})$. In coordinates, we see that:
$$ [\partial_s^{{\widehat{x_0}}}] \cdot (\Delta_{{\widehat{x_1}}}^+ \circ [x_0,\cdot])(\Phi_{\widehat{x_0}}(z))  = \partial_y (\Delta_{{\widehat{x_1}}}^+ \circ [x_0,\cdot] \circ \iota_{\widehat{x_0}})(z,0) = \partial_y( \Delta_{{\widehat{x_1}}}^+ \circ \iota_{\widehat{x_0}} \circ \pi_{\widehat{x_0}}^s)(z,0), $$
where $\pi_{\widehat{x_0}}^s$ is defined in the proof of Lemma 5.4.3. Recall from this proof that
$$ (d\pi_{\widehat{x_0}}^s)_{(z,0)} = \begin{pmatrix} 0 & 0 \\ 0 & 1 \end{pmatrix} $$
is constant in $z$ (and in ${\widehat{x_0}}$).
It follows that:
$[\partial_s^{{\widehat{x_0}}}] \cdot (\Delta_{{\widehat{x_1}}}^+ \circ [x_0,\cdot])(\Phi_{\widehat{x_0}}(z)) = \partial_y(\Delta_{{\widehat{x_1}}}^+ \circ \iota_{\widehat{x_0}})(0,0) $, which is a constant expression in $z$. The proof is done.
\end{proof}

We conclude this section by showing that one can reduce the study of the oscillations of $\Delta_x$ to the study of the oscillations of $X_{\widehat{x}}$. The proof is in three parts: we first establish a proper asymptotic expansion for $\Delta_p^+$, then for $\Delta^-_p$, and finally we reduce (NC) to a \say{non-concentration modulo polynomials} statement about $(X_{\widehat{x}})_{{\widehat{x}} \in \widehat{\Omega}}$. 

\begin{theorem}
Let $p,q \in \Omega$ be close enough. We will denote $\pi_p^{S}(q) := [p,q] =: s \in \Omega \cap W^s_{loc}(p)$, and  $\pi_p^U(q) := [q,p] =: r \in \Omega \cap W^u_{loc}(p)$. We have $q = [r,s]$. These \say{coordinates} are $C^{1+\alpha}$. Suppose that $d^s(p,s) \leq \sigma^{\beta_Z}$ and $d^u(p,r) \leq \sigma$ for $\sigma>0$ small enough. If $\beta_Z>1$ is fixed large enough, then the following asymptotic expansion hold:
$$ \Delta_p^+([r,s]) = \pm  \partial_s \Delta_p^+(r) d^s(p,s) + O(\sigma^{1+\beta_Z+\alpha}) ,$$
where $\partial_s$ denotes the derivative in the stable direction.
\end{theorem}

\begin{proof}
Let us introduce some notations. Define, for any $p \in \Omega$, the $C^{1+\alpha}$ map $\nabla_p:W^s_{loc}(p) \rightarrow \mathbb{R}$ as $$ \nabla_p(s) := \sum_{n=0}^\infty \left(\tau( f^{n}(p) ) - \tau(f^n(s)) \right),$$
and notice that $$ \Delta_p^+(q) = \nabla_p(s) - \nabla_{r}(q) .$$
A Taylor expansion yields:
$$ \nabla_p(s) = \nabla_p(p) \pm \partial_s \nabla_p(p) d^s(p,s) + O(d^s(p,s)^{1+\alpha}) $$
$$ = \pm \partial_s \nabla_p(p) d^s(p,s) + O(\sigma^{(1+\alpha)\beta_Z}) $$
$$= \pm \partial_s \nabla_p(p) d^s(p,s) + O(\sigma^{1+\alpha+\beta_Z}) $$
as soon as $\beta_Z \alpha > 1 + \alpha$. Hence, $$ \Delta_p^+(q) = \pm\left( \partial_s \nabla_p(p) d^{s}(p,s) - \partial_s \nabla_r(r) d^{s}(r,q) \right) + O(\sigma^{1+\alpha+\beta_Z}) .$$
Now, we want to make $\partial_s \Delta_p^+(r)$ appear. We compute it and find out that
$$ \partial_s \Delta_p^+(r) = \partial_s \nabla_p( p ) \partial_s \pi_p^S(r) - \partial_s \nabla_r(r). $$
We can make this term appear in our asymptotic expansion like so:
$$ \Delta_p^+(q) = \pm\Big(  \partial_s \nabla_p(p) \frac{d^s(p,s)}{d^s(r,q)} - \partial_s \nabla_r(r) \Big) d^s(r,q) + O(\sigma^{1+\alpha+\beta_Z}) $$
$$ = \pm\Big(  \partial_s \Delta_p^+(r) + \varphi_{p,s}(r) \Big) d^s(r,q) + O(\sigma^{1+\alpha+\beta_Z}) $$
where $$\varphi_{p,s}(r) := \partial_s \nabla_p(p) \left(\frac{d^s(p,s)}{d^s(r,q)} - \partial_s\pi_p^S(r)\right).$$ Now, we would like to show that $\varphi_{p,s}(r) = O(\sigma^{1+\alpha})$.
(In this case $\varphi_{p,s}(r) d^s(r,q) = O(\sigma^{1+\alpha+\beta_Z})$.) To do so, notice the following. The holonomy is (uniformly in $r$) a $C^{1+\alpha}$ map $W^s_{loc}(r) \cap \Omega \rightarrow W^s_{loc}(p) \cap \Omega$. A Taylor expansion at $q=r$ yields:
$$ \frac{d^s(s,p)}{d^s(r,q)} = \frac{d^s(\pi_p^S(q),\pi_p^S(r))}{d^s(q,r)} = \partial_s \pi_p^S(r) + O(d^s(q,r)^\alpha) = \partial_s \pi_p^S(r) + O(\sigma^{\alpha \beta_Z}).$$
If $\beta_Z>2$ is chosen so large that $\alpha \beta_Z > 1+\alpha$ again, then we find the expansion:
$$ \frac{d^s(s,p)}{d^s(q,r)} = \partial_s \pi_p^S(r) + O(\sigma^{1+\alpha}),$$
which is what we wanted to find. This gives us the following expansion for $\Delta^+$:
$$ \Delta_p^+(q) = \pm \partial_s \Delta_p^+(r) d^s(r,q) + O(\sigma^{1+\alpha + \beta_Z}). $$
To conclude, we re-apply our estimate of $d^s(p,s)$ (but inverting the role of $(p,s)$ and $(r,q)$) to find
$$ \frac{d^s(r,q)}{d^s(p,s)} = \partial_s \pi_{s}^S([r,s]) + O(\sigma^{1+\alpha}) = 1 + \alpha_{p,s} d^u(p,r) + O(\sigma^{1+\alpha}) ,$$
(where $\alpha_{p,s}=O(1)$ is some constant that depends only on $p,s$), since $r \mapsto \partial_s \pi_{s}^S([r,s])$ is $C^{1+\alpha}$ (uniformly in $p,s$). It follows that
$$ \Delta_p^+(q) = \pm \partial_s \Delta_p^+(r) d^s(p,s) \pm \alpha_{p,s} \partial_s \Delta_p^+(r) d^u(p,r) d^s(p,s)  + O(\sigma^{1+\alpha+\beta_Z}) $$ $$= \pm \partial_s \Delta_p^+(r) d^s(p,s) +  O(\sigma^{1+\alpha+\beta_Z}) ,$$
since $\partial_s \Delta_p^+(r)$ is $\alpha$-Hölder and vanish at $r=p$. \end{proof}

Now that we have a local description of $\Delta^+$, we turn to a study of the local behavior of $\Delta^-$ along the unstable direction. To achieve a clean proof, we will use the following asymptotic expansion of the distance function, whose proof is postponed to Section 5.7.2.

\begin{proposition}
Fix $p \in \Omega$ and $s \in W^s_{loc}(p) \cap \Omega$. There exists a $C^\infty$ (uniformly in $p$) function $\phi_{p}: W^u_{loc}(p) \cap \Omega \rightarrow \mathbb{R}$ such that
$$ d^u(s,[r,s]) = d^u(p,r) + \phi_{p}(r) d^s(p,s) + O\Big( 
d^s(p,s)(d^u(p,r)^{1+\alpha} + d^s(p,s)^{\alpha}) \Big) $$
\end{proposition}

\begin{theorem}
Let $p,q \in \Omega$ be close enough, and denote $[p,q] =: s \in W^s_{loc}(p) \cap \Omega$, and  $[q,p] =: r \in W^u_{loc}(p) \cap \Omega$. Suppose that $d^s(p,s) \leq \sigma^{\beta_Z}$ and $d^u(p,r) \leq \sigma$ for $\sigma>0$ small enough. If $\beta_Z>1$ is fixed large enough, then there exists a polynomial $P_{p,s} \in \mathbb{R}_{d_Z}[X]$ (where $d_Z :=\lfloor \beta_Z \rfloor + 2$) whose coefficients depends on $p,s$, and a (uniformly) $C^{1+\alpha}$ function $\psi_p : W^u_{loc}(p) \cap \Omega \rightarrow \mathbb{R}$ such that
$$ \Delta^-(p,q) = P_{p,s}\big( \pm(\Phi_{\widehat{p}}^u)^{-1}(r) \big) + 
\psi_p(r) d^s(p,s) + O(\sigma^{1+\beta_Z+\alpha}) .$$
\end{theorem}

\begin{proof}
Like before, let us decompose $\Delta^-(p,q)$ into two contributions:
$$ \Delta^-(p,q) = \overline{\nabla}_p(r) - \overline{\nabla}_s(q), $$
where $\overline{\nabla}_p: W^u(p) \rightarrow \mathbb{R}$ is defined by $$ \overline{\nabla}_p(r) := \sum_{n=1}^\infty \Big( \tau(f^{-n}(p)) - \tau(f^{-n}(r)) \Big) .$$
This map is clearly smooth along $W^u_{loc}(p)$. A Taylor expansion along $W^u_{loc}(p)$ of order $d_Z := \lceil \beta_Z \rceil + 2 \in \mathbb{N}$ yields
$$ \overline{\nabla}_p(r) = \sum_{k=1}^{d_Z} \alpha_k(p) d^u(p,r)^k + O(d^u(p,r)^{d_Z+1}) \quad ; \quad \overline{\nabla}_s(q) = \sum_{k=1}^{d_Z} \alpha_k(s) d^u(s,q)^k + O(d^u(s,q)^{d_Z+1}) ,$$
for some functions $\alpha_k : \Omega \rightarrow \mathbb{R}$, $k \in \llbracket 1,d_Z \rrbracket$. Notice that $d_Z$ is so large that $d^u(p,r)^{d_Z+1} , d^u(q,s)^{d_Z+1} = O(\sigma^{1+\alpha+\beta})$. This gives the expansion
$$ \Delta^-(p,q) = \sum_{k=1}^{d_Z} \Big( \alpha_k(p) d^u(p,r)^k - \alpha_k(s) d^u(s,q)^k \Big) + O(\sigma^{\beta_Z + \alpha + 1}) .$$
Now, our previous technical lemma ensure that there exists a (uniformly in $p$) $C^{1+\alpha}$ map $\phi_p : W^u(p) \cap \Omega \rightarrow \mathbb{R}$ such that
$$d^u(s,[r,s]) = d^u(p,r) + \phi_{p}(r) d^s(p,s) + O\Big( 
d^s(p,q)(d^u(p,r)^{1+\alpha} + d^s(p,s)^{\alpha}) \Big)$$
$$ = d^u(p,r) + \phi_{p}(r) d^s(p,s) + O\Big( 
\sigma^{1+\beta_Z+\alpha} \Big). $$
It follows that, for all $k \geq 1$, we can write
$$ d^u(s,q)^k = \Big( d^u(p,r) + \phi_{p}(r) d^s(p,s) + O\big( 
\sigma^{1+\beta_Z+\alpha} \big) \Big)^k $$
$$ = d^u(p,r)^k + k d^u(p,r)^{k-1} \phi_{p}(r) d^s(p,s) + O(\sigma^{2 \beta_Z} + \sigma^{1+\beta+\alpha}) $$
$$ = d^u(p,r)^k + k d^u(p,r)^{k-1} \phi_{p}(r) d^s(p,s) + O(\sigma^{1+\beta+\alpha})  $$
since $2 \beta_Z > 1 +\alpha + \beta_Z$, as $\beta_Z>2$ and $\alpha \in (0,1)$. Hence
$$ \Delta^-(p,q) = \sum_{k=1}^{d_Z} \Big( \alpha_k(p) d^u(p,r)^k - \alpha_k(s) (d^u(p,r)^k + k d^u(p,r)^{k-1} \phi_{p}(r) d^s(p,s)) \Big) + O(\sigma^{\beta_Z + \alpha + 1})  $$
$$ = \sum_{k=1}^{d_Z} \Big( (\alpha_k(p) - \alpha_k(s)) d^u(p,r)^k\Big) - d^s(p,s) \sum_{k=1}^{d_Z} \alpha_k(s)  k d^u(p,r)^{k-1} \phi_{p}(r)   + O(\sigma^{\beta_Z + \alpha + 1})  $$
$$ = P_{p,s}(r) + d^s(p,s) \psi_p(r) + O(\sigma^{1+\beta_Z+\alpha}) ,$$
where $P_{p,s}(r)$ is a polynomial in $d^u(p,r)$ of degree $d_Z$ whose coefficients depends on $p,s$, and where $\psi_p$ is a (uniformly in $p$) $C^{1+\alpha}$ function of $r$. To conclude, we write $P_{p,s}(r) = (P_{p,s} \circ \Phi_{\widehat{p}}^u)(z)$ for $z := (\Phi_{\widehat{p}}^u)^{-1}(r)$ and for some choice of orientation $\widehat{p}$. This is a (uniformly in $\widehat{p},s$) smooth function of $z$, and a Taylor expansion of order $d_Z$ allows us to write $P_{p,s}(r) = Q_{p,s}(\pm z) + O(\sigma^{1+\alpha+\beta_Z})$, for some $Q_{p,s} \in \mathbb{R}_{d_Z}[X]$. \end{proof}

\begin{remark}
Recall again that $\mu$ has a local product structure \cite{Cl20}, in the following sense. For all $x \in \Omega$, there exists $\mu_x^u$ and $\mu_x^s$ two measures supported on $U_x := \Omega \cap W^u_{loc}(x)$ and $S_x := \Omega \cap W^s_{loc}(x)$ such that, for all measurable $h:M \rightarrow \mathbb{C}$ supported in a small enough neighborhood of $x$, we have
$$ \int_\Omega h d\mu = \int_{U_x} \int_{S_x} h([z,y]) e^{\omega_x(z,y)} d\mu_x^s(y) d\mu_x^u(z) ,$$
where $\omega_x$ is some (uniformly in $x$) Hölder map. In particular, for $C(\mu) := \sup_x \|\omega_x\|_{L^\infty( W^u_{loc}(x) \times W^s_{loc}(x))}$, we find when $h$ is nonnegative:
$$ e^{-C(\mu)} \int_{U_x} \int_{S_x} h([z,y]) d\mu_x^s(y) d\mu_x^u(z) \leq  \int_\Omega h d\mu \leq e^{C(\mu)} \int_{U_x} \int_{S_x} h([z,y]) d\mu_x^s(y) d\mu_x^u(z).$$

Notice in particular that, for some rectangle $R_p = [U_p,S_p] \ni p$, and for any borel set $I \subset \mathbb{R}$ :
$$ \mu( q \in R_p , \ h([q,p]) \in I ) \simeq \mu_p^s(S_p) \mu_p^u(z \in U_p, h(z) \in I)). $$
The family of measures $(\mu_x^u)_{x \in \Omega}$ satisfies some invariance properties under $f$ that will be useful later. 
\end{remark}

\begin{lemma}
Denote by $\emeta_{\widehat{x}} := (\Phi_{\widehat{x}}^u)^*\mu_x^u$, the measure $\mu_x^u$ seen in the coordinates $\Phi_{\widehat{x}}^u$. Suppose that the family $(X_{\widehat{x}})_{\widehat{x}}$ satisfies the following \say{non-concentration modulo polynomials} property: there exists $\alpha,\gamma,\sigma_0>0$ such that, for all ${\widehat{x}} \in \widehat{\Omega}$ and for all $0 < \sigma < \sigma_0$:
$$ \sup_{P \in \mathbb{R}_{d_Z}[X]} \emeta_{\widehat{x}}( z \in [-\sigma,\sigma], \ |X_{\widehat{x}}(z) - P(z)| \leq \sigma^{1+\alpha/2} ) \leq \sigma^{\gamma} \cdot \emeta_{\widehat{x}}([-\sigma,\sigma]). $$
Then the conclusion of Theorem 5.1.4 holds.
\end{lemma}

\begin{proof}
Recall that, by Lemma 5.2.4, to check (NC), it suffice to establish the following bound:
$$ \mu \left( q \in R, |\Delta(p,q)| \leq \sigma^{1+\alpha+\beta_Z} \right) \lesssim \sigma^\gamma \mu(R) $$
where the bound is uniform in $p$ and $R$, for $R \in \text{Rect}_{\beta_Z}(\sigma)$ a rectangle containing $p$ with stable (resp. unstable) diameter $\sigma^{\beta_Z}$ (resp. $\sigma$). Let us check this estimate by using the Taylor expansion of $\Delta^+$ and $\Delta^-$. We can write, using the local product structure of $\mu$:
$$ \mu \left( q \in R, |\Delta(p,q)| \leq \sigma^{1+\alpha+\beta_Z} \right) \lesssim \int_{S} \mu_p^u \left( r \in U_p^\sigma \ , \ |\Delta_p^+([r,s]) + \Delta^-(p,[r,s])| \leq \sigma^{1+\alpha+\beta_Z}  \right) d\mu_p^s(s)  $$
$$ \lesssim \int_{S} \mu_p^u\left( r \in U_p^\sigma , \ |(\partial_s \Delta_p^+(r) + \psi_p(r)) d^s(p,s) + P_{p,s}(\pm (\Phi_{\widehat{p}}^u)^{-1}(r))| \leq C \sigma^{1+\beta_Z+\alpha} \right) d\mu_p^s(s) $$
$$ = \int_{S} \mu_p^u\left( r \in U_p^\sigma , \ |\partial_s \Delta_p^+(r) + \psi_p(r) + Q_{p,s}(\pm (\Phi_{\widehat{p}}^u)^{-1}(r)))| d^s(p,s) \leq C \sigma^{1+\beta_Z+\alpha} \right) d\mu_p^s(s) $$
where $Q_{p,s} := {d^s(p,s)}^{-1} {P_{p,s}}$. It is easy to see, using Gibbs estimates and the fact that $\text{diam}^s(S) \simeq \sigma^{\beta_Z}$, that there exists $\delta_{reg}>0$ such that $$\mu_p^s( B(p,\sigma^{\beta_Z+\alpha/2}) \cap S ) \leq \sigma^{\alpha \delta_{reg}/2} \mu_p^s(S) .$$
It follows that one can cut the integral over $S$ in two parts: the part where $s$ is $\sigma^{\beta_Z+\alpha/2}$-close to $p$, and the other part. We get, using the aforementioned regularity estimates:
$$ \mu \left( q \in R, |\Delta(p,q)|\leq \sigma^{1+\alpha+\beta_Z} \right) $$ $$ \lesssim \sigma^{\delta_{reg} \alpha/2} \mu(R) + \int_{S} \mu_p^u\left( r \in U_p^\sigma , \ |\partial_s \Delta_p^+(r) + \psi_p(r) + Q_{p,s}(\pm (\Phi_{\widehat{p}}^u)^{-1}(r))| \leq C \sigma^{1+\alpha/2} \right) d\mu_p^s(s). $$
We just have to control the integral term to conclude. To do so, notice that, for all $s$, we can write:
$$ \mu_p^u\left( r \in U_p^\sigma , \ |\partial_s \Delta_p^+(r) + \psi_p(r) + Q_{p,s}(\pm (\Phi_{\widehat{p}}^u)^{-1}(r))| \leq C \sigma^{1+\alpha/2} \right)  $$ $$= \left((\Phi_{\widehat{p}}^u)^*\mu_p^u \right)\Big( z \in (\Phi_{\widehat{p}}^u)^{-1}(U_p^\sigma), \ |\partial_s \Delta_p^+( \Phi_{\widehat{p}}^u(z) ) + \psi_p( \Phi_{\widehat{p}}^u(z) ) + Q_{p,s}(\pm z) | \leq  C \sigma^{1+\alpha/2} \Big) $$
$$ \leq \left((\Phi_{\widehat{p}}^u)^*\mu_p^u \right)\Big( z \in [-C \sigma,C \sigma], \ |\partial_s \Delta_p^+(\Phi_{\widehat{p}}^u(z)) + \psi_p(\Phi_{\widehat{p}}^u(z)) + Q_{p,s}(\pm z)| \leq C \sigma^{1+\alpha/2} \Big) . $$
Beware that, if $\psi_p \circ \Phi_{\widehat{p}}^u$ is uniformly in $p$ a $C^{1+\alpha}$ function of $z$, allowing us do approximate it by its $1+\alpha$-th order Taylor expansion, the same doesn't hold for $Q_{p,s}$, which may have unbounded coefficients as $s$ gets close to $p$. Still, writing $$ \psi_p(\Phi_{\widehat{p}}^u(z)) = \alpha_p^{(1)} \pm \alpha_p^{(2)} z + O(\sigma^{1+\alpha}), $$
we find, replacing $Q_{s,p}$ by $\tilde{Q}_{s,p}(z) := \alpha_p^{(0)} \pm \alpha_p^{(1)} z + Q_{s,p}(\pm z)$,
$$ \mu \left( q \in R, |\Delta(p,q)|\leq \sigma^{1+\alpha+\beta_Z} \right) $$
$$ \lesssim \sigma^{\delta_{reg} \alpha/2} \mu(R) + \int_{S} \emeta_p\left( z \in [-C\sigma,C\sigma] , \ |\partial_s \Delta_p^+(\Phi_{\widehat{p}}^u(z)) + \tilde{Q}_{p,s}(\pm z)| \leq C \sigma^{1+\alpha/2} \right) d\mu_p^s(s). \quad (*)$$
Now, since $TM/E^u$ is a line bundle, and since $[\partial_s]$ and $[\partial_s^{\widehat{p}}]$ are $C^{1+\alpha}$ sections of this line bundle ($[\partial_s]$ is $C^{1+\alpha}$ because $E^s$ is $C^{1+\alpha}$ in our 2-dimensional context, and $[\partial_s^{\widehat{p}}]$ is actually smooth by construction), there exists a nonvanishing $C^{1+\alpha}$ function $a_{\widehat{p}}(z)$ such that $[\partial_s]_{\Phi_{\widehat{p}}^u(z)} = a_{\widehat{p}}(z) [\partial_s^p]_{\Phi_{\widehat{p}}^u(z)} $. We have $a_p(z)= \pm(1 + O(\sigma))$. Hence, using the bound $X_{\widehat{x}}(z) = O(\sigma^\alpha)$ since $X_{{\widehat{x}}}$ is (uniformly) Hölder regular and vanish at zero:
$$\emeta_{\widehat{p}}\left( z \in [-C\sigma,C\sigma] , \ |\partial_s \Delta_p^+(\Phi_{\widehat{p}}^u(z)) + \tilde{Q}_{p,s}(\pm z)| \leq C \sigma^{1+\alpha/2} \right) $$
$$ = \emeta_{\widehat{p}}\left( z \in [-C\sigma,C\sigma] , \ |a_{\widehat{p}}(z) X_{\widehat{p}}(z) + \tilde{Q}_{p,s}(\pm z)| \leq C \sigma^{1+\alpha/2} \right) $$
$$ \leq \emeta_{\widehat{p}}\Big( z \in [-C \sigma,C \sigma], \ |X_{\widehat{p}}(z) \pm \tilde{Q}_{p,s}(\pm z)| \leq C' \sigma^{1+\alpha/2} \big) $$
$$ \leq C'' \emeta_{\widehat{p}}( [C' \sigma,C' \sigma]) \sigma^{\gamma}  ,$$
where the last control is given by the nonconcentration hypothesis made on $(X_{\widehat{x}})_{{\widehat{x}} \in \widehat{\Omega}}$. To conclude, notice that by regularity of the parametrizations $\Phi_{\widehat{p}}^u$, and since the measure $\mu_x^u$ is doubling (thus constants can be neglected, see section 5.9), we get $\emeta_{\widehat{p}}\left( [-C \sigma, C\sigma] \right) \leq C \mu_p^u(U)$. Injecting this estimate in $(*)$ yields
$$ \mu \left( q \in R, |\Delta(p,q)|\leq \sigma^{1+\beta_Z+\alpha} \right) \leq C \left( \sigma^{\delta_{reg} \alpha/2} + \sigma^{\gamma} \right) \mu(R) ,$$
which is what we wanted. \end{proof}

Notice that the estimate required to apply Lemma 5.4.11 actually doesn't depend on the choice of the orientation $\widehat{x} \in \widehat{\Omega}$ in the fiber above $x$. If the estimate is true for $\widehat{x}$, then it is also satisfied for $\mathfrak{s}(\widehat{x})$. We are reduced to understand oscillations of $ z \mapsto X_{\widehat{x}}(z)$ modulo polynomials of a fixed degree $d_{Z} \geq 1$. The next section will be devoted to proving a \say{blowup} result on the family $(X_{\widehat{x}})_{{\widehat{x}} \in \widehat{\Omega}}$, which will help us understand deeper the oscillations of those functions. This \say{blowup} result will allow us to exhibit a rigidity phenomenon. The final section is devoted to proving the non-concentration estimates in the hypothesis of Lemma 5.4.1, under our generic condition defined in Remark 5.4.5.

\section{Autosimilarity, polynomials, and rigidity}
Let us recall our setting. We are given a family of Hölder maps $(X_{\widehat{x}})_{{\widehat{x}} \in \widehat{\Omega}}$, where $X_{\widehat{x}} : \Omega_{\widehat{x}}^u \cap (-\rho,\rho) \longrightarrow \mathbb{R}$ is defined only on a (fractal) neighborhood of zero and vanish at $z=0$. Recall that $\Omega_{\widehat{x}}^u := (\Phi_{\widehat{x}}^u)^{-1}(\Omega) \ni 0$. Recall also that $\Omega_{\mathfrak{s}(\widehat{x})}^u = - \Omega_{\widehat{x}}^u$ and that $X_{\mathfrak{s}(\widehat{x})}(z) = - X_{\widehat{x}}(-z)$. Furthermore, we have an autosimilarity relation: for any ${\widehat{x}} \in \widehat{\Omega}$, and any $z \in \Omega_{\widehat{x}}^u \cap (-\rho,\rho)$, we have
$$ X_{\widehat{x}}(z) = \widehat{\tau}_{\widehat{x}}(z) + \mu_x X_{\widehat{f}({\widehat{x}})}(z \lambda_x) ,$$
where $\widehat{\tau}_{\widehat{x}} :  \Omega_{\widehat{x}}^u \cap (-\rho,\rho) \rightarrow \mathbb{R}$ is a $C^{1+\alpha}$ map (in the sense of Whitney) that vanish at zero. Recall that $\widehat{\tau}_{{\widehat{x}}}'(0)$ only depends on $x \in \Omega$, and recall moreover that $\text{Reg}_u^{1+\alpha}$-generically in the choice of $\tau$, we can suppose that the function
$$ \Phi_{\tau} : x \in \Omega \mapsto (\widehat{\tau}_{\widehat{x}})'(0) e^{h(x)} \in \mathbb{R}$$
is not cohomologous to zero (where $h: \Omega \rightarrow \mathbb{R}$ is such that $\mu_x \lambda_x = \exp(h(f(x))-h(x))$). We will establish quantitative estimates on the oscillations of $(X_{\widehat{x}})$ under the cohomology condition $\Phi_\tau \nsim 0$. To do so, we start by proving a \say{blowup} result, inspired from Appendix B in  \cite{TZ20}. The point of this lemma is to only keep, in the autosimilarity formula of $(X_{\widehat{x}})_{{\widehat{x}} \in \widehat{\Omega}}$, the germ of $\widehat{\tau}_{\widehat{x}}$ (in the form of its Taylor expansion at zero at some order). Depending on the contraction/dilatation rate on the dynamics, the order of this Taylor expansion is different: in our area-preserving case, it is enough to approximate $\widehat{\tau}_{\widehat{x}}(z)$ by $(\widehat{\tau}_{\widehat{x}})'(0)z$.
\begin{lemma}[Blowup]
There exists two families of functions $(Y_{\widehat{x}})_{{\widehat{x}} \in \widehat{\Omega}}$, $(Z_{\widehat{x}})_{{\widehat{x}} \in \widehat{\Omega}}$ such that:
\begin{itemize}
\item For all ${\widehat{x}} \in \widehat{\Omega}$ and $z \in \Omega_{\widehat{x}}^u \cap (-\rho,\rho)$, $$X_{\widehat{x}}(z) = Y_{\widehat{x}}(z) + Z_{\widehat{x}}(z) .$$
\item The map $Y_{\widehat{x}} :  \Omega_{\widehat{x}}^u \cap (-\rho,\rho) \rightarrow \mathbb{R}$ is $C^{1+\alpha}$, and there exists $C \geq 1$ such that, for all ${\widehat{x}} \in \widehat{\Omega}$ and $z \in \Omega_{\widehat{x}}^u \cap (-\rho,\rho)$: $$|Y_{\widehat{x}}(z)| \leq C |z|^{1+\alpha}.$$
\item The family $(Z_{\widehat{x}})_{{\widehat{x}} \in \widehat{\Omega}}$ satisfies an autosimilarity formula: for any ${\widehat{x}} \in \widehat{\Omega}$, $z \in \Omega_{\widehat{x}}^u \cap (-\rho,\rho)$
$$ Z_{\widehat{x}}(z) = (\widehat{\tau}_{\widehat{x}})'(0) z + \mu_x Z_{\widehat{f}({\widehat{x}})}(z \lambda_x) .$$
Moreover, the dependence in ${\widehat{x}}$ of $(X_{\widehat{x}})_{{\widehat{x}} \in \widehat{\Omega}}$, $(Y_{\widehat{x}})_{{\widehat{x}} \in \widehat{\Omega}}$ and $(Z_{\widehat{x}})_{{\widehat{x}} \in \widehat{\Omega}}$ is Hölder.
\end{itemize}

\begin{proof}
In the original proof, there is an implicit argument used, which is the fact that polynomials (of order one, here) are maps with vanishing (second order) derivative. In our fractal context, this is not true, as $\Omega_{\widehat{x}}^u$ may not be connected: so we have to replace this derivative with a notion adapted to our fractal context. For $\beta<\alpha$, define a \say{$(1+\beta)$-order fractal derivative} as follows: if $h:\Omega_x^u \cap (-\rho,\rho) \rightarrow \mathbb{R}$ is ${C^{1+\alpha}}$ in the sense of Whitney, then its Taylor expansion at zero makes sense, and we can consider the function:
$$ \delta^{1+\beta}(h)(z) := \frac{h(z) - h(0) - h'(0) z}{|z|^{1+\beta}}. $$
This is a continuous function on $\Omega_{\widehat{x}}^u \cap (-\rho,\rho)$, which is bounded and vanish at zero at order $|z^{(\alpha-\beta)-}|$. Moreover, notice that $\delta^{1+\beta}(h)=0$ is equivalent to saying that $h$ is affine. Notice further that $$ \delta^{1+\beta}\Big{(} \mu h(\lambda \cdot) \Big{)}(z) = (\mu |\lambda|^{1+\beta}) \cdot \delta^{1+\beta}(h)(z \lambda ).$$
Now, let us begin the actual proof. Consider the autosimilarity equation of $(X_{\widehat{x}})$, and formally take the $(1+\beta)$-th fractal derivative. We search for a $C^{1+\alpha}$ solution $(Y_{\widehat{x}})$ of this equation:
$$ \delta^{1+\beta}(Y_{\widehat{x}})(z) = \delta^{1+\beta}(\widehat{\tau}_{\widehat{x}})(z) + \mu_x \lambda_x^{1+\beta} \cdot \delta^{1+\beta}(Y_{\widehat{x}})(z \lambda_x). $$
Notice that $\kappa_x := \mu_x \lambda_x^{1+\beta}$ behaves like a greater-than-one multiplier. Indeed, if we denote, for $x \in \Omega$ and $n \in \mathbb{Z}$, $$\kappa_x^{\langle n \rangle}  := \kappa_x \dots \kappa_{f^{n-1}(x)} $$
(if $n \geq 0$, and similarly if $n \leq 0$ as in the definition of $\lambda_x^{\langle n \rangle}) $,
we see that $\kappa_x^{\langle n \rangle} \geq (\lambda_-^n)^{\beta}$, where $\lambda_- > 1$. We can solve this equation by setting
$$ \delta^{1+\beta}(Y_{\widehat{x}})(z) := -\sum_{n=1}^\infty \delta^{1+\beta}(\widehat{\tau}_{\widehat{f}^{-n}({\widehat{x}})})(z \lambda_x^{\langle - n \rangle}) \cdot \kappa_x^{\langle - n \rangle} =: \tilde{Y}_{\widehat{x}}(z).$$
This defines a continuous function that vanishes at zero. We then define $Y_{\widehat{x}}$ as the only $C^{1+\alpha}$ function such that $ Y_{\widehat{x}}(0)=Y_{\widehat{x}}'(0)=0 $ and $\delta^{1+\beta}(Y_{\widehat{x}}) = \tilde{Y}_{\widehat{x}}$. In other words, $Y_{\widehat{x}}(z) := |z|^{1+\beta} \tilde{Y}_{\widehat{x}}(z)$. Using the sum formula of $\tilde{Y}_{\widehat{x}}$, we find the autosimilarity formula:
$$ Y_{\widehat{x}}(z) = \widehat{\tau}_{\widehat{x}}(z) - (\widehat{\tau}_{\widehat{x}})'(0)z + \mu_x Y_{\widehat{f}({\widehat{x}})}(z \lambda_x) .$$
We can then conclude by setting $Z_{\widehat{x}} := X_{\widehat{x}} - Y_{\widehat{x}}$.
\end{proof}
\end{lemma}

\begin{remark}
One could wonder what it means for $(X_{\widehat{x}})_{\widehat{x}}$ (or $(Y_{\widehat{x}})_{{\widehat{x}}}$ and $(Z_{\widehat{x}})_{\widehat{x}}$) to be Hölder in ${\widehat{x}}$, when the domain of definition of $X_{\widehat{x}}$ changes with ${\widehat{x}}$. We mean the following. If ${\widehat{x_1}} \in \Omega$ is close to ${\widehat{x_0}}$, then first of all they share the same local orientation. Denote $x_2 := [x_1,x_0] \in W^u_{loc}(x) \cap W^s_{loc}(x)$, and then consider $\widehat{x_2}$ in the fiber of $x_2$ with the same local orientation than $\widehat{x_1}$ and $\widehat{x_0}$. The holonomy between $W^u_{loc}(x_1)$ and $W^u_{loc}(x_2)$, written in coordinates $\Phi_{\widehat{x_1}}^u$ and $\Phi_{\widehat{x_2}}^u$, yields a map $\pi_{x_2,x_1} : \Omega_{\widehat{x_2}}^u \cap (-\rho,\rho) \rightarrow \Omega_{\widehat{x_1}}^u \cap (-\rho',\rho')$. Then, the affine map $\text{aff}_{\widehat{x_0},\widehat{x_2}} = \Phi_{{\widehat{x_0}}}^u \circ (\Phi_{\widehat{x_2}}^u)^{-1} : \Omega_{\widehat{x_2}}^u \cap (-\rho',\rho') \rightarrow \Omega_{\widehat{x_0}}^u \cap (-\rho'',\rho'')$ allows us to compare functions on $\Omega_{{\widehat{x_2}}}$ with functions on $\Omega_{{\widehat{x_0}}}$. It is not difficult to see that the map  $\text{aff}_{{\widehat{x_0}},\widehat{x_2}} \circ \pi_{\widehat{x_2},{\widehat{x_1}}}$ is $C^{1+\alpha}$, and moreover that $\| \text{aff}_{{\widehat{x_0}},\widehat{x_2}} \circ \pi_{\widehat{x_2},{\widehat{x_1}}} - Id \|_{C^{1+\alpha}} \lesssim d(\widehat{x_0},\widehat{x_1})^\alpha$ \cite{PR02}.
What we mean for the Hölder dependence in $x$, then, is the following:
$$ \| X_{\widehat{x_1}} - X_{\widehat{x_0}}  \circ \text{aff}_{{\widehat{x_0}},\widehat{x_2}} \circ \pi_{\widehat{x_2},{\widehat{x_1}}}\|_{L^\infty(\Omega_{\widehat{x_1}} \cap (-\rho,\rho))} \lesssim d(\widehat{x_0},\widehat{x_1})^\alpha.$$ \end{remark}

To get from this construction a deeper understanding of the local behavior of $(X_{\widehat{x}})$, we need first to state some quantitative statements about equivalence of norms in our context. The proof of the next lemma will be done in section 5.8.

\begin{lemma}
There exists $\kappa>0$ such that, for all ${\widehat{x}}\in \widehat{\Omega}$, for all $P \in \mathbb{R}_{d_Z}[X]$, there exists $z_0 \in \Omega_{\widehat{x}}^u \cap (-\rho/2,\rho/2)$ such that
$$ \forall z \in [z_0 - \kappa, z_0 + \kappa], \ |P(z)| \geq \kappa \|P\|_{C^\alpha((-\rho,\rho))}. $$
If we denote $P = \sum_{k=0}^{d_Z} a_k z^k \in \mathbb{R}_{d_Z}[X]$, then moreover
$$\max_{0 \leq k \leq d_Z} |a_k| \leq \kappa^{-1} \|P\|_{L^\infty( (-\rho,\rho) \cap \Omega_{\widehat{x}}^u )}. $$
\end{lemma}

From this, we deduce a strong statement on the possible forms of $Z_{\widehat{x}}$.

\begin{lemma}
Suppose that there exists $P \in \mathbb{R}_{d_Z}[X]$ such that $Z_{{\widehat{x_0}}} = P$ on $(-\rho,\rho) \cap \Omega_{{\widehat{x_0}}}^u$, for some ${\widehat{x_0}} \in \widehat{\Omega}$. Then $P$ is linear. 
\end{lemma}

\begin{proof}
Let $\kappa>0$ be such that $$ \forall P = \sum_{k=0}^{d_Z} a_k z^k \in \mathbb{R}_{d_Z}[X], \quad \max_{0 \leq k \leq d_Z} |a_k| \leq \kappa^{-1} \inf_x \|P\|_{L^\infty( (-\rho,\rho) \cap \Omega_x )}.$$
We suppose that there exists $P = \sum_k a_k z^k \in \mathbb{R}_{d_Z}[X]$ such that $Z_{{\widehat{x_0}}}(z) = P(z)$ for all $z \in (-\rho,\rho) \cap \Omega_{{\widehat{x_0}}}^u$. Since
$$ \forall n \geq 1, \forall z \in (-\rho,\rho) \cap \Omega_{\widehat{f}^{n}({\widehat{x_0}})}^u , \ Z_{\widehat{f}^n({\widehat{x_0}})}(z) = \text{linear} + \mu_x^{\langle - n \rangle} Z_{{\widehat{x_0}}}(\lambda_x^{\langle - n \rangle} z), $$
it follows that $Z_{\widehat{f}^{n}({\widehat{x_0}})}$ is polynomial on $\Omega_{\widehat{f}^n({\widehat{x_0}})}^u \cap (-\rho,\rho)$, of the form:
$$ Z_{\widehat{f}^{-n}({\widehat{x_0}})}(z) = \text{linear} + \sum_{k=2}^{d_Z} \mu_{x_0}^{\langle n \rangle} (\lambda_{x_0}^{\langle n \rangle})^k a_k z^k. $$
Equivalence of norms yields, for all $k \geq 2$,
$$ \forall n\geq 1, \ | \mu_{x_0}^{\langle n \rangle} (\lambda_{x_0}^{\langle n \rangle})^k a_k  | \leq \kappa^{-1} \sup_{{\widehat{x}} \in \widehat{\Omega}} \|Z_{\widehat{x}}\|_{L^\infty(\Omega_{\widehat{x}}^u \cap (-\rho,\rho))} < \infty ,$$
which implies that $a_k=0$ for $k \geq 2$. Hence, $Z_{{\widehat{x_0}}}$ is linear.
\end{proof}
 
The idea now is to consider the distance from $Z_{\widehat{x}}$ to the space of linear/polynomial maps. By the autosimilarity formula of $(Z_{\widehat{x}})$, there is going to be some invariance that will prove useful. Notice, from the proof of lemma 5.5.1, that $Z_{\mathfrak{s}({\widehat{x}})}(z)=-Z_{{\widehat{x}}}(-z)$. 

\begin{definition}
For any $\rho>0$ small enough, consider the functions $\mathcal{D}_\rho^{lin}, \mathcal{D}_\rho^{poly}  : \Omega \longrightarrow \mathbb{R}_+$ defined as
$$ \mathcal{D}_\rho^{lin}(x) := \inf_{a \in \mathbb{R}} \sup_{z \in \Omega_{\widehat{x}}^u \cap (-\rho,\rho)} | Z_{\widehat{x}}(z) - a z | $$
and $$ \mathcal{D}_\rho^{poly}(x) := \inf_{P \in \mathbb{R}_{d_Z}[X]} \sup_{z \in \Omega_{\widehat{x}}^u \cap (-\rho,\rho)} | Z_{\widehat{x}}(z) - P(z) |. $$
where $\widehat{x} \in \widehat{\Omega}$ is any choice of element in the fiber of $x \in \Omega$. These functions are continuous, since $\widehat{x} \in \Omega \mapsto Z_{\widehat{x}} \in C^0$ is continuous in the sense of Remark 5.5.2, and since we are computing a distance to a finite-dimensional vector space. Moreover,
$$ 0 \leq \mathcal{D}_{\rho}^{poly} \leq \mathcal{D}_{\rho}^{lin}. $$
\end{definition}
In the following, for some $C^{1+\alpha}$ function $Z$ on $\Omega_x^u \cap(-\rho,\rho)$, let us denote $$ \delta^{1+\alpha}(Z)(z) := \frac{Z(z)-Z(0)-Z'(0)z}{|z|^{1+\alpha}} .$$
\begin{lemma}
We have the following criterion. The following are equivalent, for some fixed $x \in \Omega$ and $\rho>0$:
\begin{itemize}
\item $\mathcal{D}_\rho^{lin}(x)=0$
\item For some $\widehat{x} \in \Omega$ in the fiber of $x$, for all $n \geq 0$, $Z_{f^{-n}(\widehat{x})} \in C^{1+\alpha}\big{(}\Omega_{f^{-n}(\widehat{x})}^u \cap (-\rho \lambda_x^{\langle -n \rangle } , \rho \lambda_x^{\langle -n \rangle}\big{)},\mathbb{R}) $, and there exists $C \geq 1$ such that, for all $n \geq 0$, $$\| \delta^{1+\alpha}(Z_{\widehat{f}^{-n}(\widehat{x})}) \|_{L^\infty( \Omega_{\widehat{f}^{-n}(\widehat{x})}^u \cap (-\rho \lambda_x^{\langle -n \rangle } , \rho \lambda_x^{\langle -n \rangle}) )} \leq C.$$
\end{itemize}
\end{lemma}

\begin{proof}
Suppose that $\mathcal{D}_\rho(x)=0$. Choose any $\widehat{x} \in \widehat{\Omega}$ in the fiber of $x \in \Omega$. Since $Z_{\widehat{x}}(0)=0$, there exists $a \in \mathbb{R}$ such that $Z_{\widehat{x}}(z) = az$ on $(-\rho,\rho)$. The autosimilarity relation $Z_{\widehat{x}}(z) = (\widehat{\tau}_{\widehat{x}})'(0)z + \mu_x Z_{\widehat{f}({\widehat{x}})}(\lambda_x z)$ gives, with a change of variable,
$$ Z_{\widehat{f}^{-1}({\widehat{x}})}(z \lambda_x^{\langle -1 \rangle}) \mu_x^{\langle -1 \rangle} = (\widehat{\tau}_{\widehat{f}^{-1}({\widehat{x}})})'(0) \lambda_x^{\langle -1 \rangle} \mu_x^{\langle -1 \rangle} z + Z_{{\widehat{x}}}(z) .$$
Iterating this yields $$ Z_{\widehat{f}^{-n}({\widehat{x}})}(z \lambda_x^{\langle -n \rangle}) \mu_x^{\langle -n \rangle} = \text{linear} +  Z_{\widehat{x}}(z) .$$
In particular, if $Z_{\widehat{x}}$ is linear on $\Omega_{\widehat{x}}^u \cap (-\rho,\rho)$, then $Z_{\widehat{f}^{-n}({\widehat{x}})}$ is linear on $\Omega_{\widehat{f}^{-n}({\widehat{x}})}^u \cap (-\rho \lambda_x^{\langle -n \rangle } , \rho \lambda_x^{\langle -n \rangle}) )$. In particular, it is $C^{1+\alpha}$ and the bound on $\delta^{1+\alpha}(Z_{\widehat{f}^{-n}({\widehat{x}})}) = 0$ holds. Reciprocally, if the second point holds, then we can write, on $\Omega_{\widehat{x}}^u \cap (-\rho,\rho)$:
$$ |\delta^{1+\alpha}(Z_{\widehat{x}})(z)| = |\mu_x^{ \langle - n \rangle} (\lambda_x^{ \langle -n \rangle })^{1+\alpha} \cdot \delta^{1+\alpha}(Z_{\widehat{f}^{-n}({\widehat{x}})})(z \lambda_x^{\langle -n \rangle})| \leq C' (|\lambda_x|^{\langle -n \rangle})^{\alpha} \underset{n \rightarrow \infty}{\longrightarrow} 0,  $$
where we used the fact that $\mu_x^{\langle n \rangle} \lambda_x^{\langle n \rangle} = e^{O(1)}$ by our area-preserving hypothesis made on the dynamics $f$. Hence $\delta^{1+\alpha}(Z_{\widehat{x}})=0$ on $\Omega_{\widehat{x}}^u \cap (-\rho,\rho)$, which means that $Z_{\widehat{x}}$ is linear on this set.
\end{proof}

\begin{lemma}[Rigidity lemma]
Suppose that there exists $x_0 \in \Omega$ such that $\mathcal{D}_\rho^{poly}(x_0) = 0$, then $\mathcal{D}_\rho^{lin} = 0$ on $\Omega$.
\end{lemma}

\begin{proof}
The proof is in three steps. Suppose that $\mathcal{D}_\rho^{poly}(x_0)=0$ for some $\rho>0$ and some $x_0 \in \Omega$.
\begin{itemize}
\item We first show that there exists $0 < \rho' < \rho$ and a set $\omega \subset W^u_{loc}(x_0) \cap \Omega$ which is an open neighborhood of $x$ for the topology of $W^u_{loc}(x_0) \cap \Omega$, such that $\mathcal{D}_{\rho'}^{lin}(\tilde{x})=0$ if $\tilde{x} \in \omega$. 
\end{itemize}
So suppose that $\mathcal{D}_\rho^{poly}(x_0)=0$. This means that $Z_{\widehat{x_0}}$ coincide with a polynomial of degree $d_Z$ on $(-\rho,\rho) \cap \Omega_{\widehat{x_0}}^u$. Lemma 5.5.4 ensures that, in this case, $Z_{\widehat{x_0}}$ is in fact linear on $(-\rho,\rho) \cap \Omega_{\widehat{x_0}}^u$. Now, notice that since $Y_{\widehat{x_0}}$ is $C^{1+\alpha}$, we know that $X_{\widehat{x_0}}$ is $C^{1+\alpha}$ on $(-\rho,\rho) \cap \Omega_{\widehat{x_0}}^u$. Recall then that, by Lemma 5.4.6, we know that if $x_1 \in W^u_{loc} \cap \Omega$ is close enough to $x_0$, (and if $\widehat{x_1}$ share the same local orientation than $\widehat{x_0}$) we can write
$$ \text{Aff}_{\widehat{x_0},\widehat{x_1}}\Big(X_{\widehat{x_1}}(z)\Big{)} = X_{\widehat{x_0}}(\text{aff}_{\widehat{x_0},\widehat{x_1}}(z)), $$
where $\text{Aff}_{\widehat{x_0},\widehat{x_1}}$ and $\text{aff}_{\widehat{x_0},\widehat{x_1}}$ are affine functions (that gets close to the identity as $x_1 \rightarrow x_0$). If follows that $X_{\widehat{x_1}}$ is $C^{1+\alpha}$ on some (smaller) open neighborhood of zero, $(-\rho',\rho') \cap \Omega_{\widehat{x_1}}^u$. In particular, $Z_{\widehat{x_1}}$ is also $C^{1+\alpha}$ on this set. Let us show that $\mathcal{D}_{\rho'}^{lin}(\widehat{x_1})=0$ by checking the criterion given in the previous lemma. We have:
$$ Z_{\widehat{x_1}}(z) = X_{\widehat{x_1}}(z) - Y_{\widehat{x_1}}(z) = \text{Aff}_{\widehat{x_0},\widehat{x_1}}^{-1} \Big( X_{\widehat{x_0}}(\text{aff}_{\widehat{x_0},\widehat{x_1}}(z))  \Big) - Y_{\widehat{x_1}}(z) $$
$$ = \text{Aff}_{\widehat{x_0},\widehat{x_1}}^{-1} \Big( Z_{\widehat{x_0}}(\text{aff}_{\widehat{x_0},\widehat{x_1}}(z))  \Big) + \text{Aff}_{\widehat{x_0},\widehat{x_1}}^{-1} \Big( Y_{\widehat{x_0}}(\text{aff}_{\widehat{x_0},\widehat{x_1}}(z))  \Big) - Y_{\widehat{x_1}}(z) $$
By hypothesis, $\text{Aff}_{\widehat{x_0},\widehat{x_1}}^{-1} \Big( Z_{\widehat{x_0}}(\text{aff}_{\widehat{x_0},\widehat{x_1}}(z))  \Big)$ is affine in $z$, and so its $(1+\alpha)$-th derivative is zero. We can then write, for all $n\geq 0$ and $z \in (-\rho' \lambda_x^{\langle -n \rangle},\rho' \lambda_x^{\langle -n \rangle}) \cap \Omega_{\widehat{f}^{-n}(\widehat{x_1})}^u$:
$$ \delta^{1+\alpha}(Z_{\widehat{f}^{-n}(\widehat{x_1})})(z) = \alpha_{\widehat{f}^{-n}(\widehat{x_0}),\widehat{f}^{-n}(\widehat{x_1})} \delta^{1+\alpha}(Y_{\widehat{f}^{-n}(\widehat{x_0})})(\text{aff}_{\widehat{f}^{-n}(\widehat{x_0}),\widehat{f}^{-n}(\widehat{x_1})}(z)) - \delta^{1+\alpha}(Y_{\widehat{f}^{-n}(\widehat{x_1})})(z) , $$
where $\alpha_{\widehat{x_0},\widehat{x_1}} := (\text{Aff}^{-1}_{\widehat{x_1},\widehat{x_0}})'(0) (\text{aff}_{\widehat{x_1},\widehat{x_0}})'(0)^{1+\alpha} = 1 + O(d^u(x_0,x_1))$.
Since $|\delta^{1+\alpha}(Y_{\widehat{x_0}})| \leq \|Y_{\widehat{x_0}}\|_{C^{1+\alpha}}$, the criterion applies.
\begin{itemize}
\item Second, we show that if $\mathcal{D}_{\rho'}^{lin}(x)=0$ for some $x$ and small $\rho'$, then $\mathcal{D}_{\min(\rho' \lambda_x,\rho)}^{lin}(f(x))=0$. \end{itemize}
This directly comes from the autosimilarity formula. We have, for $z \in (-\rho,\rho) \cap \Omega_{\widehat{x}}^u$:
$$ Z_{\widehat{x}}(z) = (\widehat{\tau}_{\widehat{x}})'(0)z + \mu_x Z_{\widehat{f}({\widehat{x}})}(z \lambda_x). $$
In particular, if $ Z_{\widehat{x}}$ is linear on $(-\rho',\rho') \cap \Omega_{\widehat{x}}^u$, then $Z_{\widehat{f}({\widehat{x}})}$ is linear on $(-\rho' \lambda_x,\rho' \lambda_x) \cap (-\rho,\rho) \cap \Omega_{\widehat{f}({\widehat{x}})}^u$.
\begin{itemize}
\item We conclude, using the transitivity of the dynamics and the continuity of $\mathcal{D}_\rho^{lin}$. \end{itemize}
We know that $\mathcal{D}_\rho^{lin}(x_0)=0$, by hypothesis. By step one, there exists $\omega$, some unstable neighborhood of $x_0$, and $\rho'<\rho$ such that $\mathcal{D}_{\rho'}^{lin}=0$ on $\omega$. Step 2 then ensures that $\mathcal{D}_{\min(\rho' \lambda_x^{\langle n \rangle} , \rho)}^{lin}=0$ on $f^n(\omega)$. Choosing $N$ large enough, we conclude that
$$ \forall x \in \bigcup_{n \geq N} f^n(\omega), \ \mathcal{D}_\rho^{lin}(x) = 0 .$$
Since the dynamics $f$ is transitive on $\Omega$, we know that $\bigcup_{n \geq N} f^n(\omega)$ is dense in $\Omega$. The function $\mathcal{D}_\rho^{lin}$ being continuous, it follows that $\mathcal{D}_\rho^{lin}=0$ on $\Omega$. \end{proof}

\begin{lemma}[Oscillations everywhere in $x$]
Under the $\text{Reg}_u^{1+\alpha}$-generic condition $\Phi_\tau \nsim 0$, the following hold. There exists $\kappa \in (0,\rho/10)$ such that, for all ${\widehat{x}} \in \widehat{\Omega}$, for all $P \in \mathbb{R}_{d_Z}[X]$, there exists $z_0 \in \Omega_{\widehat{x}}^u \cap (-\rho/2,\rho/2)$ such that
$$ \forall z \in \Omega_{\widehat{x}}^u \cap (z_0 - \kappa, z_0 + \kappa), \  |Z_{\widehat{x}}(z) - P(z)| \geq \kappa. $$
\end{lemma}

\begin{proof}
Our previous lemma gives us the following dichotomy: either $\mathcal{D}_\rho^{poly} > 0$ on $\Omega$, or either $\mathcal{D}_\rho^{lin} =0$ on $\Omega$. Suppose the later. In this case, for all ${\widehat{x}}$, $Z_{\widehat{x}} \in C^{1+\alpha}$. Write the autosimilarity relation and take the (usual) first derivative in $z$. We find:
$$ Z_{\widehat{x}}'(0) = (\widehat{\tau}_{\widehat{x}})'(0) + \mu_x \lambda_x Z_{\widehat{f}({\widehat{x}})}'(0) .$$
Recall that, since $f$ is area preserving, we can write $\mu_x \lambda_x = \exp(h(f(x)) - h(x))$ for some Hölder function $h : \Omega \rightarrow \mathbb{R}$. Our previous relation can then be rewritten as
$$ Z_{\widehat{x}}'(0) e^{h(x)} = e^{h(x)} (\widehat{\tau}_{\widehat{x}})'(0) + Z_{\widehat{f}({\widehat{x}})}'(0) e^{h(f(x))} .$$
Now notice that $Z_{\widehat{x}}'(0) e^{h(x)} =: \theta(x)$ doesn't depend on the choice of $\widehat{x}$ in the fiber of $x$. We can rewrite our expression as $\Phi_\tau = \theta \circ f - \theta$ where $\Phi_\tau(x) = e^{h(x)}(\widehat{\tau}_x)'(0)$, that is, $\Phi_\tau \sim 0$. So our generic condition $\Phi_\tau \nsim 0$ ensures that $\mathcal{D}_{\rho}^{poly} > 0$ on $\Omega$. By continuity of $\mathcal{D}_\rho^{poly}$, and by compacity of $\Omega$, there exists some $\kappa>0$ such that $\mathcal{D}_\rho^{poly}(x) \geq \kappa$ for all $x \in \Omega$. One can do the same proof replacing $\rho$ with $\rho/2$, so we can directly says that $\mathcal{D}_{\rho/2}^{poly}(x) \geq \kappa$, taking $\kappa$ smaller if necessary. Now, this means the following: for every $P \in \mathbb{R}_{d_Z}[X]$, for every ${\widehat{x}} \in \widehat{\Omega}$, there exists $z_0(x,a,b) \in \Omega_{\widehat{x}}^u \cap (- \rho/2,\rho/2)$ such that
$$ |Z_{\widehat{x}}(z_0)-P(z)| \geq \kappa. $$
We still have to show that this doesn't only hold for some point $z_0$, but on a whole small interval. The proof uses the technical Lemma 5.5.3 on equivalence of norms $\mathbb{R}_{d_Z}[X]$. The proof in different, depending if $\|P\|_{C^\alpha( \Omega_x^u \cap (-\rho,\rho) )}$ is small or large. \\

Let ${\widehat{x}} \in \widehat{\Omega}$ and $P \in \mathbb{R}_{d_Z}[X]$. There exists $z_0 \in \Omega_{\widehat{x}}^u \cap (-\rho/2,\rho/2)$ such that $|P(z_0)-Z_{\widehat{x}}(z_0)| \geq \kappa$. Hence:
$$ \forall z \in (z_0-\tilde{\kappa},z_0+\tilde{\kappa}) \cap \Omega_{\widehat{x}}^u, \ |Z_{\widehat{x}}(z)-P(z)| \geq \kappa - \|Z_{\widehat{x}} - P\|_{C^\alpha(\Omega_{\widehat{x}}^u \cap (-\rho,\rho))} \tilde{\kappa}^{\alpha} $$
$$ \geq \kappa - ( \sup_{x' \in \widehat{\Omega}} \|Z_{x'}\|_{C^\alpha(\Omega_{x'}^u \cap (-\rho,\rho))} + \|P\|_{C^\alpha(-\rho,\rho)} ) \tilde{\kappa}^\alpha. $$
This arguments gives us what we want if $\|P\|_{C^\alpha(-\rho,\rho)}$ is bounded by some constant $M \geq \sup_{x'} \|Z_{x'}\|_{C^\alpha}$. Indeed, in this case, choosing $\tilde{\kappa} := (\kappa/2M)^{1\alpha}$ yields
$$ \forall z \in (z_0-\tilde{\kappa},z_0+\tilde{\kappa}) \cap \Omega_{\widehat{x}}^u, \ |Z_{\widehat{x}}(z)-P(z)| \geq \kappa/2.$$
To conclude, we need to understand the case where $\|P\|_{C^\alpha((-\rho,\rho))}$ is large.
Recall that Lemma 5.5.3 ensure that there exists $\kappa_{0} \in (0,1)$ and $z_1 \in (-\rho/2,\rho/2) \cap \Omega_{\widehat{x}}^u$ such that $$ \forall z \in (z_1 - \kappa_0, z_1 + \kappa_0), \ |P(z)| \geq \kappa_0 \|P\|_{C^\alpha((-\rho,\rho))} .$$
It follows that if $M$ is fixed so large that $M \geq \sup_{x'}\| Z_{x'} \|_{C^\alpha} \kappa_0^{-1} + \kappa$, then in the case where $\|P\|_{C^\alpha(-\rho,\rho)} \geq M$ we find
$$ \forall z \in (z_1 - \kappa_0, z_1 + \kappa_0), \ |P(z) - Z_{\widehat{x}}(z)| \geq |P(z)| - \|Z_{\widehat{x}}\|_{\infty} $$
$$ \geq \kappa_0 \|P\|_{C^{\alpha}} - \|Z_{\widehat{x}}\|_{\infty} \geq \kappa ,$$
which conclude the proof. \end{proof}

\begin{lemma}[Oscillation everywhere in $x$, at all scales in $z$]
Under the generic condition $\Phi_\tau \nsim 0$, the following hold. There exists $\kappa>0$ such that, for all ${\widehat{x}} \in \widehat{\Omega}$, for all $P \in \mathbb{R}_{d_Z}[X]$, for all $n \geq 0$, there exists $z_0 \in \Omega_{\widehat{x}}^u \cap (-\rho/2,\rho/2)$ such that
$$ \forall z \in \Omega_{\widehat{x}}^u \cap (z_0 - \kappa, z_0 + \kappa), \  |Z_{\widehat{x}}^{\langle n \rangle}(z) - P(z)| \geq \kappa ,$$
where $Z_{\widehat{x}}^{\langle n \rangle}(z) := \mu_x^{\langle -n \rangle} Z_{\widehat{f}^{-n}({\widehat{x}})}( z \lambda_x^{\langle -n \rangle})$. The family $(Z_{\widehat{x}}^{\langle n \rangle})$ is a (n-th times) zoomed-in and rescaled version of $(Z_{\widehat{x}})$.
\end{lemma}
\begin{proof}
We know that $Z_{\widehat{x}}^{\langle n \rangle}(z) = \text{linear} + Z_{\widehat{x}}(z)$ on $(-\rho,\rho) \cap \Omega_{\widehat{x}}^u$. The result follows from the previous lemma.
\end{proof}

\begin{lemma}
Under the generic condition $\Phi_\tau \nsim 0$, the following hold. There exists $\kappa>0$ and $n_0 \geq 0$ such that, for all $\widehat{x} \in \widehat{\Omega}$, for all $P \in \mathbb{R}_{d_Z}[X]$, for all $n \geq n_0$, there exists $z_0 \in \Omega_{{\widehat{x}}}^u \cap (-\rho/2,\rho/2)$ such that
$$\forall z \in \Omega_{\widehat{x}}^u \cap (z_0 - \kappa, z_0 + \kappa), \  |X_{\widehat{x}}^{\langle n \rangle}(z) - P(z)| \geq \kappa ,$$
where $X_{\widehat{x}}^{\langle n \rangle}(z) := \mu_x^{\langle -n \rangle} X_{\widehat{f}^{-n}({\widehat{x}})}( z \lambda_x^{\langle -n \rangle})$. 
\end{lemma}
\begin{proof}
Recall that $X_{\widehat{x}} = Y_{\widehat{x}} + Z_{\widehat{x}}$, and that $|Y_{\widehat{x}}(z)| \leq C|z|^{1+\alpha}$. Zooming in, we find, for all $n \geq 0$:
$$ X_{\widehat{x}}^{\langle n \rangle}(z) = Y_{\widehat{x}}^{\langle n \rangle}(z) + Z_{\widehat{x}}^{\langle n \rangle}(z) ,$$
where $Y_{\widehat{x}}^{\langle n \rangle}(z) := \mu_x^{\langle -n \rangle} Y_{\widehat{f}^{-n}({\widehat{x}})}(z \lambda_x^{\langle -n \rangle}) = O( (\lambda_x^{\langle - n \rangle})^\alpha ).$ Taking $n_0$ large enough so that this is less than $\kappa/2$ for the $\kappa$ given by the previous lemma allows us to conclude.
\end{proof}

We can rewrite this into a statement about $X_{\widehat{x}}$ only.

\begin{proposition}
Under the generic condition $\Phi_\tau \nsim 0$, the following hold. There exists $0<\kappa<1/10$ and $\sigma_0 > 0$ such that, for all ${\widehat{x}} \in \widehat{\Omega}$, for all $P \in \mathbb{R}_{d_Z}[X]$, for all $0<\sigma<\sigma_0$, there exists $z_0 \in \Omega_{{\widehat{x}}}^u \cap (- \sigma/2, \sigma/2)$ such that
$$\forall z \in \Omega_{\widehat{x}}^u \cap (z_0 - \kappa \sigma, z_0 + \kappa \sigma) \subset(- \sigma, \sigma), \  |X_{\widehat{x}}(z) - P(z)| \geq \kappa \sigma .$$
\end{proposition}

\begin{proof}
For each $\sigma$ small enough, define $n_x(\sigma)$ as the largest positive integer such that $ \sigma \leq \rho \lambda_x^{\langle -n_x(\sigma) \rangle}$.
We then have $\sigma \simeq \lambda_x^{\langle -n_x(\sigma) \rangle } \simeq (\mu_x^{\langle -n_x(\sigma) \rangle})^{-1}$, and we see that we can deduce our statement written with $\sigma$ by our statement written with $n_x(\sigma)$.
\end{proof}

We conclude this section by establishing what we will call the \say{Uniform Non Integrability condition} (UNI) in our context.

\begin{definition}[UNI]
We say that the condition (UNI) is satisfied if there exists $0<\kappa<1/10$ and $\sigma_0 > 0$ such that, for all ${\widehat{x}} \in \widehat{\Omega}$, for all $P \in \mathbb{R}_{d_Z}[X]$, for all $0<\sigma<\sigma_0$, for any $z_0 \in \Omega_{\widehat{x}}^u \cap (-\rho,\rho)$, there exists $z_1 \in \Omega_{{\widehat{x}}}^u \cap (z_0 -  \sigma/2,z_0+ \sigma/2)$ such that
$$\forall z \in \Omega_{\widehat{x}}^u \cap (z_1 - \kappa \sigma, z_1 + \kappa \sigma) \subset(z_0 - \sigma,z_0+\sigma), \  |X_{\widehat{x}}(z) - P(z)| \geq \kappa \sigma .$$
\end{definition}

\begin{proposition}
Under the generic condition $\Phi_\tau \nsim 0$, (UNI) holds.
\end{proposition}

\begin{proof}
Let us fix the $\kappa$ and $\sigma_0$ from Proposition 5.5.11. Let ${\widehat{x}} \in \widehat{\Omega}$, let $\sigma<\sigma_0$, and let $z_0 \in \Omega_{\widehat{x}}^u$. Let $P \in \mathbb{R}_{d_Z}[X]$. Define $x_1 := \Phi_x^u(z_0)$. Let $\widehat{x_1}$ be the element in the fiber of $x_1$ that share he same local orientation than $\widehat{x}$. Recall that, by Lemma 5.4.6, there exists $\text{Aff}_{{\widehat{x}},{\widehat{x_1}}}$ and $\text{aff}_{{\widehat{x}},{\widehat{x_1}}}$, two affine functions with $e^{O(1)}$ linear coefficients, such that
$$ \text{Aff}_{{\widehat{x}},{\widehat{x_1}}}(X_{{\widehat{x_1}}}(z)) = X_{\widehat{x}}(\text{aff}_{{\widehat{x}},{\widehat{x_1}}}(z)).$$
Furthermore, $\text{aff}_{{\widehat{x}},{\widehat{x_1}}} = (\Phi_{\widehat{x}}^u)^{-1} \circ (\Phi_{{\widehat{x_1}}}^u)$. Notice that $\text{aff}_{{\widehat{x}},{\widehat{x_1}}}(0)=z_0$.
Since $(\text{Aff}_{{\widehat{x}},{\widehat{x_1}}})'(0) = e^{O(1)}$, the previous lemma applied to $X_{{\widehat{x_1}}}$ gives us some $z_2 \in (-\sigma/2,\sigma/2)$ such that:
$$ \forall z \in (z_2 - \kappa \sigma,z_2 + \kappa \sigma), \ |\text{Aff}_{{\widehat{x}},{\widehat{x_1}}}(X_{{\widehat{x_1}}}(z)) - P(z)| \geq \kappa \sigma ,$$
choosing $\kappa$ smaller if necessary.
Setting $z_1:=(\text{aff}_{{\widehat{x}},{\widehat{x_1}}})^{-1}(z_2)$ yields the desired result.
\end{proof}

\section{The non-concentration estimates under (UNI)}

In this last section, we establish our last estimate, given by the following proposition, which takes the role of the \say{tree Lemma} 2.5.7 of Chapter 2. The idea of the proof is the same, the main difficulty for us was to replace the easy (UNI) condition found in Chapter 2 by our subtle (UNI) condition in our context. But once a version of (UNI) is proved, the \say{tree lemma} follows quite easily, even in our context.

\begin{proposition}
Under (UNI), there exists $\gamma>0$ and $\alpha>0$ such that for all $\sigma>0$ small enough, for all ${\widehat{x}} \in \widehat{\Omega}$, for all $P \in \mathbb{R}_{d_Z}[X]$, we have
$$ \emeta_{\widehat{x}}\Big( z \in [-\sigma,\sigma], \ |X_{\widehat{x}}(z) - P(z)| \leq \sigma^{1+\alpha} \Big) \leq \sigma^\gamma \emeta_{\widehat{x}}([-\sigma,\sigma]) ,$$
where $\emeta_{\widehat{x}} := (\Phi_{\widehat{x}}^u)^*\mu_x^u$.
\end{proposition}
Once this proposition is proved, Lemma 5.4.11 will ensure that the nonconcentration estimates (NC) are true under the generic condition $\Phi_\tau \nsim 0$. Before heading into the proof, we need to establish a cutting lemma. The difficulty here is to cut $\Omega_{\widehat{x}}^u$ into \say{equal} pieces, but the fractal nature of $\Omega_{\widehat{x}}^u$ makes it a bit subtle (especially because I prefer not to use Markov partitions here). See Section 5.8 for related statements about the uniform perfection of the sets $\Omega_{\widehat{x}}^u \cap (-\rho,\rho)$.

\begin{lemma}
Let $\kappa>0$. There exists $\mathcal{C} \geq 1$, and $\eta_{cut,1},\eta_{cut,2} \in (0,1)$, with $\eta_{cut,1} \leq \kappa/100$, such that the following hold. Let $x \in \Omega$, $\sigma \in (0,\rho]$ and $n \geq 1$. Let $U^{(\sigma)}_{x} \subset W^u_{loc}(x) \cap \Omega$ be such that $B(x,\sigma) \cap W^u_{loc}(x) \cap \Omega \subset U^{(\sigma)}_{x} \subset B(x,10\sigma) \cap W^u_{loc}(x)$. There exists a finite set $\mathcal{A} := \mathcal{A}(x,\sigma,n)$ with $|\mathcal{A}| \leq \mathcal{C}$ and a family of intervals $(I_\mathbf{a})_{\mathbf{a} \in \bigcup_{k=0}^n\mathcal{W}_k}$, where $\mathcal{W}_k \subset \mathcal{A}^k$ is a set of words on the alphabet $\mathcal{A}$, and where $I_{\mathbf{a}} \subset W^u_{loc}(x)$ are such that: 
\begin{enumerate}
    \item $I_{\emptyset} \cap \Omega = U^{(\sigma)}_x$
    \item For all $\mathbf{a} \in \mathcal{W}_k$, $0 \leq k \leq n$, we have $I_{\mathbf{a}} \cap \Omega = \underset{\mathbf{a}b \in \mathcal{W}_{n+1}}{\bigcup_{b \in \mathcal{A}}} (I_{\mathbf{a}b} \cap \Omega)$, and this union is disjoint modulo a zero-measure set.
    \item For all $\mathbf{a} \in \mathcal{W}_k$, $0 \leq k \leq n$, there exists $x_{\mathbf{a}} \in \Omega \cap I_{\mathbf{a}}$ such that $$W^u_{loc}(x) \cap 
  B(x_{\mathbf{a}}, 10^{-1} \text{diam}^u(I_{\mathbf{a}}) ) \subset I_{\mathbf{a}}$$
    \item For all $\mathbf{a} \in \mathcal{W}_k$, $0 \leq k \leq n-1$, we have for any $b \in \mathcal{A}$ such that $\mathbf{a}b \in \mathcal{W}_{k+1}$: $$ \eta_{cut,1}^2 \text{diam}^u(I_\mathbf{a}) \leq \text{diam}^u(I_{\mathbf{a}b}) \leq \eta_{cut,1} \text{diam}^u(I_{\mathbf{a}}) \quad , \quad \mu_x^u(I_{\mathbf{a}b}) \geq \eta_{cut,2} \mu_x^u(I_{\mathbf{a}}).$$
\end{enumerate}
\end{lemma}

\begin{proof}
The proof is done in a couple of steps.
\begin{itemize}
\item \underline{Step 1}: We cut, once ($n=1$), the sets $U_x^{(\sigma)}$ when $\sigma \in [\rho \lambda_+^{-2},\rho]$. \end{itemize}
First of all, recall that for all $\varepsilon$, there exists $\delta>0$ such that for any $z \in W^u_{loc}(x) \cap \Omega$, $\mu_x^u(B(z,\varepsilon) \cap W_{loc}^u(x))>\delta$. Since, moreover, the measure $\mu_x^u$ is upper regular, that is:
$$ \forall U \subset W^u_{loc}(x), \ \mu_x^u(U) \leq C \text{diam}(U)^{\delta_{reg}}, $$
it follows that reciprocally, if $\mu_x^u(U) \geq \varepsilon$, then $\text{diam}(U) \geq \delta$. \newline

Consider a finite cover of $U_x^{(\sigma)}$ by a set of balls $E$ of the form $B(x_i,\varepsilon)$ with $x_i \in U_x^{(\sigma)}$, for $\varepsilon$ small enough (for example, $\varepsilon := \kappa \rho \lambda_+^{-2} \cdot 10^{-10}$ should be enough). By Vitali's lemma, there exists a subset $D \subset E$ such that the balls in $D$ are disjoints, and such that $\cup_{B \in D} 3B \supset U^{(\sigma)}_x$. Since $W^u_{loc}(x)$ is one-dimensional, we can set $\mathcal{A} := D$, and choose $B \subset I_a \subset 3B$ such that the $(I_a)$ have disjoint union and covers $U^{(\sigma)}$. By construction, the diameter of those intervals is small, but in a controlled way: the measure of them is greater than some constant, and they all contain a ball of radius $\varepsilon$ for some $\varepsilon$. This construct our $(I_a)_{a \in \mathcal{A}}$ in this case. Notice that the cardinal of $\mathcal{A}$ is uniformly bounded. We see that it suffice to take $\varepsilon$ small enough to find other constants $\eta_{cut,1},\eta_{cut,2}$ with $\eta_{cut,1} \leq \kappa/100$ that satisfies what we want in this macroscopic context. 

\begin{itemize}
\item \underline{Step 2:} We cut, once $(n=1)$, sets $U_x^{(\sigma)}$ when $\sigma>0$ is small.
\end{itemize}
The idea is to use the properties of $\mu_x^u$. If we denote by $\varphi:\Omega \rightarrow \mathbb{R}$ the potential defining our equilibrium state $\mu$, then we know \cite{Cl20} that $ f_* d\mu_x^u = e^{\varphi \circ f^{-1} - P(\varphi)} d\mu_{f(x)}^u $. Iterating yields $f_*^n d\mu_x^u = e^{S_n\varphi \circ f^{-n} - nP(\varphi)} d\mu_{f^n(x)}^u$. The Hölder regularity of $\varphi$ allows us to see that, for all intervals $J \subset I := W^u_{loc}(f^n(x)) \cap B(f^n(x),\rho) $, we have $$ \frac{\mu_x^u( f^{-n}(J))}{\mu_x^u(f^{-n}(I))} =  \frac{\mu_x^u(J)}{\mu_x^u(I)}\Big{(}1 + O(\rho^\alpha)\Big{)} \quad, \quad \frac{\text{diam}( f^{-n}(J))}{\text{diam}(f^{-n}(I))} =  \frac{\text{diam}(J)}{\text{diam}(I)}\Big{(}1 + O(\rho^\alpha)\Big{)} .$$
So we see that, using the dynamics, (and reducing $\rho$ if necessary) we can reduce our setting to the setting of step 1. This might deform a bit the constants though, which is why we chosed the $10^{-1}$ in point three instead of $6^{-1}$ (which was the constant given by Vitali's Lemma).
\begin{itemize}
\item \underline{Step 3:} We iterate this construction to the subintervals $I_b$. 
\end{itemize}
We consider $U^{(\sigma)}$ as in the statement of the lemma. We use step 2 to construct $(I_a)_{a \in \mathcal{A}}$. Then, we can iterate our construction on each of the $I_a$, since they satisfy the necessary hypothesis to do so. This gives us intervals $(I_{a_1 a_2})_{a_1 a_2 \in \mathcal{W}_2}$, with $\mathcal{W}_2 \subset \mathcal{A}^2$ (taking $\mathcal{A}$ with more letters if necessary). Doing this again and again yield the desired construction.
\end{proof}

Using this partition lemma, we can prove Proposition 5.6.1. Notice that, looking at this lemma in coordinates $\Phi_{\widehat{x}}^u$, one can get the same construction  replacing subintervals of $W^u_{loc}$ by intervals of $\mathbb{R}$,  $\Omega$ by $\Omega_{\widehat{x}}^u$, $\mu_x^u$ by $\emeta_{\widehat{x}}$, etc.

\begin{proof}[Proof (Proposition 5.6.1)]
Assume (UNI), and denote $(\kappa,\sigma_0)$ the constants given by (UNI). Let $0<\sigma<\sigma_0$ be small enough, let $P \in \mathbb{R}_{d_Z}[X]$. Let $x \in \Omega$. Define $k(\sigma) \in \mathbb{N}$ as the largest integer such that $$  \sigma^\alpha \leq (\kappa/20) \cdot \Big({\eta}_{cut,2}\Big)^{2 k(\sigma)}.$$ We have $k(\sigma) \simeq C |\ln \sigma|$. By applying the previous construction to $\Phi_{\widehat{x}}^u([-\sigma,\sigma])$ (for the constant $\kappa$), with depth $k(\sigma)$, we find a family of intervals $(I_{\mathbf{a}})_{\mathbf{a} \in \cup_{j \leq k(\sigma)}\mathcal{W}_j}$ that satisfies some good partition properties. Then, the heart the proof is as follow: we want to construct (up to renaming some of the $I_{\mathbf{a}}$) a word $\mathbf{b} \in W_{k(\sigma)} \subset \mathcal{A}^{k(\sigma)}$ such that
$$ \{ z \in [-\sigma,\sigma] , \  |X_{\widehat{x}}(z) - P(z)| \leq \sigma^{1+\alpha} \} \subset \underset{\forall i, a_i \neq b_i}{\bigcup_{\mathbf{a} \in \mathcal{W}_{k(\sigma)}}} I_{\mathbf{a}} .$$
This is nothing more than a technical way of keeping track of the oscillations happening at all scales between $\sigma$ and $\sigma^{1+\alpha}$. Notice that, once this is proved, then the conclusion is easy: we will have
$$ \emeta_{\widehat{x}}\Big( \underset{\forall i, a_i \neq b_i}{\bigcup_{\mathbf{a} \in \mathcal{W}_{k(\sigma)}}} I_{\mathbf{a}} \Big) = \underset{\forall i, a_i \neq b_i}{\sum_{\mathbf{a} \in \mathcal{W}_{k(\sigma)}}} \emeta_{\widehat{x}}(I_{\mathbf{a}}) $$
$$ \leq (1-\eta_{cut,2})\underset{\forall i, a_i \neq b_i}{\sum_{\mathbf{a} \in \mathcal{W}_{k(\sigma)-1}}} \emeta_{\widehat{x}}(I_{\mathbf{a}}) $$
$$ \leq (1-\eta_{cut,2})^2 \underset{\forall i, a_i \neq b_i}{\sum_{\mathbf{a} \in \mathcal{W}_{k(\sigma)-2}}} \emeta_{\widehat{x}}(I_{\mathbf{a}}) $$
$$ \leq \dots \leq (1-\eta_{cut,2})^{k(\sigma)} \emeta_x([-\sigma,\sigma]) \simeq \sigma^\gamma \emeta_{\widehat{x}}([-\sigma, \sigma])  $$
for some $\gamma>0$, since $k(\sigma) \simeq \ln(\sigma^{-1})$. Now, let us construct this word $\mathbf{b}$: the idea is that since $2 \sigma^{1+\alpha} \simeq \sigma (\kappa/10) \cdot \Big({\eta}_{cut,2}\Big)^{2 k(\sigma)}$, for each $i$, we can find oscillations at scales $\sim \sigma \eta_{cut,1}^i$ of magnitude $\sim \sigma (\kappa/10) (\eta_{cut,1})^{2i} \geq 2 \sigma^{1+\alpha}$, and conclude. \\

Let us begin by the case $i=0$.
By Lemma 5.6.2, we know that there exists a point $z_{\emptyset} \in \Omega_{\widehat{x}}^u \cap(-\sigma/2,\sigma/2)$ such that:
$$ \forall z \in (z_{\emptyset} \pm \kappa\sigma), \ |X_{\widehat{x}}(z)-P(z)| \geq \kappa \sigma.$$
There exists $b_1 \in \mathcal{A}$ such that $z_{\emptyset} \in I_{b_1}$. Since $\text{diam}(I_{b_1}) \leq \eta_{cut,1} \sigma \leq \frac{\kappa}{100} \sigma$, we find 
$$ \forall z \in I_{b_1}, \ |X_{\widehat{x}}(z)-P(z)| \geq \kappa \sigma.$$
Hence: $$ \{ z \in [-\sigma,\sigma] , \ |X_{\widehat{x}}(z)-P(z)| \leq \sigma^{1+\alpha} \} \subset \{ z \in [-\sigma,\sigma] , \ |X_{\widehat{x}}(z)-P(z)| \leq \sigma \kappa/2 \} \subset \bigcup_{a_1 \in \mathcal{A}\setminus \{b_1\}} I_{a_1}.$$
Now, let $a_1 \in \mathcal{A} \setminus \{b_1\}$, and work on $I_{a_1}$. We know that there exists $z_{a_1} \in \Omega_{\widehat{x}}^u$ such that \\$(z_{a_1} \pm 10^{-1} \text{diam}(I_{a_1})) \subset I_{a_1}$. Now, (UNI) applied at this interval yields:
$$ \exists \tilde{z}_{a_1} \in I_{a_1} \cap \Omega_{\widehat{x}}^u, \forall z \in (\tilde{z}_{a_1} \pm (\kappa/10) \text{diam}(I_{a_1})), \ |X_{\widehat{x}}(z) - P(z)| \geq (\kappa/10) \text{diam}(I_{a_1}) \geq (\kappa/10) \eta_{cut,1}^2 \sigma. $$
Since $\text{diam}(I_{a_1 b}) \leq \eta_{cut,1} \text{diam}(I_{a_1})$, with $\eta_{cut,1} \leq \kappa/100$, we find, for some $b_2 \in \mathcal{A}$ (we can say that this is the same $b_2$ for any $a_1 \in \mathcal{A}$, up to renaming the intervals),
$$ \forall z \in I_{a_1 b_2}, |X_{\widehat{x}}(z) - P(z)| \geq (\kappa / 10) \eta_{cut,1}^2 \sigma. $$
Hence:
$$ I_{a_1} \cap \{ z \in [-\sigma,\sigma] , \ |X_{\widehat{x}}(z) - P(z)| \leq \sigma^{1+\alpha} \} $$ $$\subset I_{a_1} \cap \{ z \in [-\sigma,\sigma] , \ |X_{\widehat{x}}(z) - P(z)| \leq (\kappa/20) {\eta_{cut,1}^2} \sigma \} \subset \bigcup_{a_2 \neq b_2} I_{a_1 a_2}. $$
Hence $$ \{ z \in [-\sigma,\sigma] , \ |X_{\widehat{x}}(z) - P(z)| \leq \sigma^{1+\alpha} \}  \subset \bigcup_{a_1 \neq b_1} \bigcup_{a_2 \neq b_2} I_{a_1 a_2} . $$
Let us do the next step, and then we will stop here because the construction will be clear enough. We could formally conclude by induction.
Take $a_1 \neq b_1$ and $a_2 \neq b_2$. We know that there exists $z_{a_1 a_2} \in \Omega_{\widehat{x}}^u$ such that $(z_{a_1 a_2} \pm 10^{-1} \text{diam}(I_{a_1 a_2})) \subset I_{a_1 a_2}$.  Applying (UNI) to this interval yield
$$ \exists \tilde{z}_{a_1 a_2} \in I_{a_1 a_2} \cap \Omega_{\widehat{x}}^u, \forall z \in (\tilde{z}_{a_1 a_2} \pm (\kappa/10) \text{diam}(I_{a_1 a_2})), $$ $$ |X_{\widehat{x}}(z)-P(z)| \geq (\kappa/10) \text{diam}(I_{a_1 a_2}) \geq (\kappa/10) \eta_{cut,1}^4 \sigma.$$
There exists $b_3 \in \mathcal{A}$ such that $\tilde{z}_{a_1 a_2} \in I_{a_1 a_2 b_3}$. Since $\text{diam}(I_{a_1 a_2 b_3}) \leq \eta_{cut,1} \text{diam}(I_{a_1 a_2})$, with $\eta_{cut,1} \leq \kappa/100$, we have
$$ \forall z \in I_{a_1 a_2 b_3}, \ |X_{\widehat{x}}(z) - P(z)| \geq (\kappa/10) \eta_{cut,1}^4 \sigma \geq 2 \sigma^{1+\alpha} .$$
Hence,
$$ I_{a_1 a_2} \cap \{ z \in [-\sigma,\sigma], \ |X_{\widehat{x}}(z) - P(z)| \leq \sigma^{1+\alpha} \} \subset \bigcup_{a_3 \neq b_3} I_{a_1 a_2 a_3} ,$$
and so $$ \{ z \in [-\sigma,\sigma], \ |X_{\widehat{x}}(z) - P(z)| \leq \sigma^{1+\alpha} \} \subset \bigcup_{a_1 \neq b_1}  \bigcup_{a_2 \neq b_2}   \bigcup_{a_3 \neq b_3}  I_{a_1 a_2 a_3} .$$
This algorithm is done until the $k(\sigma)$-th step. This conclude this construction, hence the proof. \end{proof}

Theorem 5.1.4 is then proved using Proposition 5.6.1, Proposition 5.5.13, and Lemma 5.4.11.

\section{Appendix A: A collection of technical regularity statements}

\subsection{Some regularity for the geometric potential}

In this section, we prove that (in our 2-dimensional, Axiom A, area-preserving context) even though the geometric potential $\tau_f(x) := \ln \| (df)_{E^u(x)} \|$ is only $C^{1+\alpha}(\Omega,\mathbb{R})$,  for any $p \in \Omega$, the family of functions 
$$ \partial_s \tau_f : W^u_{loc}(p) \cap \Omega \longrightarrow \mathbb{R} $$
are $C^{1+\alpha}(W^u_{loc}(p),\mathbb{R})$ (and this, uniformly in $p$). Moreover, the differential of $\tau_f$ varies (at least) in a $C^{1+\alpha}$ manner along local unstable manifolds, and $\partial_u \partial_s \tau_f$ is Hölder on $\Omega$. In fact, $\tau_f \in \text{Reg}_u^{1+\alpha}(\Omega)$. Let us introduce some notations.

\begin{definition}
For any $\widehat{x} \in \widehat{\Omega}$ and $(z,y) \in (-\rho,\rho)^2 \cap \iota_{\widehat{x}}(\Omega)^{-1}$, denote by $\theta^u_{\widehat{x}}(z,y) \in \mathbb{R}$ the only real number such that
$$ (\iota_{\widehat{x}})^*(E^u(\iota_{\widehat{x}}(z,y))) = \text{Span}\Big{(}\partial_z + \theta^u_{\widehat{x}}(z,y) \partial_y\Big{)} ,$$
where $(\iota_{\widehat{x}})_{{\widehat{x}} \in \widehat{\Omega}}$ is the family of smooth coordinates defined in Lemma 5.3.6. 
The real valued map $\theta_{\widehat{x}}^u$ is $C^{1+\alpha}$, and this, uniformly in ${\widehat{x}} \in \widehat{\Omega}$.
\end{definition}

\begin{remark}
Notice that, for all $z$, $\theta_{\widehat{x}}^u(z,0) = 0$. In particular, $z \mapsto \partial_z \theta_{\widehat{x}}^u(z,0)=0$ is smooth.
\end{remark}

\begin{lemma}
The map $z \mapsto \partial_y \theta_{\widehat{x}}^u(z,0)$ is $C^{1+\alpha}((-\rho,\rho) \cap \Omega_{\widehat{x}},\mathbb{R})$, and this, uniformly in ${\widehat{x}}$.
\end{lemma}

\begin{proof}
Recall that $f_{\widehat{x}} := \iota_{\widehat{f}({\widehat{x}})}^{-1} \circ f \circ \iota_{\widehat{x}}$ denotes the dynamics in coordinates, and that $f_{\widehat{x}}(z,0) = (\lambda_x z,0)$.
Let us introduce four families of smooth functions $a_{\widehat{x}},b_{\widehat{x}},c_{\widehat{x}},d_{\widehat{x}} : (-\rho,\rho)^2 \rightarrow \mathbb{R}$ such that $$ (df_{\widehat{x}})_{(z,y)} = \begin{pmatrix} a_{\widehat{x}}(z,y) & b_{\widehat{x}}(z,y) \\ c_{\widehat{x}}(z,y) & d_{\widehat{x}}(z,y) \end{pmatrix} . $$
Recall that, by Remark 3.7, we have the following properties:
$$ a_{\widehat{x}}(z,0) = \lambda_x, \quad c_{\widehat{x}}(z,0) = 0, \quad d_{\widehat{x}}(z,0) = \mu_x.$$
Now, the invariance relation $(df)(E^u({\widehat{x}})) = E^u(f({\widehat{x}}))$, written in coordinates, yields
$$ \theta^u_{\widehat{f}({\widehat{x}})}(f_{\widehat{x}}(z,y)) = \frac{c_{\widehat{x}}(z,y)+d_{\widehat{x}}(z,y) \theta_{\widehat{x}}^u(z,y)}{a_{\widehat{x}}(z,y)+b_{\widehat{x}}(z,y) \theta_{\widehat{x}}^u(z,y)} .$$
Taking the derivative in $y$ and choosing $y=0$ gives, for the LHS:
$$ \frac{\partial}{\partial y}_{|y=0}\theta^u_{\widehat{f}({\widehat{x}})}(f_{\widehat{x}}(z,y)) = \partial_z \theta_{\widehat{f}({\widehat{x}})}^u(f_{\widehat{x}}(z,0)) b_{\widehat{x}}(z,0) + \partial_y \theta_{\widehat{f}({\widehat{x}})}^u(f_{\widehat{x}}(z,0)) d_{\widehat{x}}(z,0) = \partial_y \theta_{\widehat{f}({\widehat{x}})}^u(\lambda_x z,0) \mu_x. $$
Hence:
$$\partial_y \theta_{\widehat{f}({\widehat{x}})}^u(\lambda_x z,0) \mu_x = \frac{\Big{(} \partial_y c_{\widehat{x}}(z,0) + d_{\widehat{x}}(z,0) \partial_y \theta_{\widehat{x}}^u(z,0) \Big{)} a_{\widehat{x}}(z,0) - \Big{(} \partial_y a_{\widehat{x}}(z,0) + b_{\widehat{x}}(z,0) \partial_y \theta_{\widehat{x}}^u(z,0) \Big{)} c_{\widehat{x}}(z,0)}{a_{\widehat{x}}(z,0)^2} $$
$$ = \frac{\left(\partial_y c_{\widehat{x}}(z,0) + \mu_x \partial_y \theta_{\widehat{x}}(z,0)\right) \lambda_x}{\lambda_x^2}, $$
and so:
$$ \partial_y \theta_{\widehat{f}({\widehat{x}})}^u(\lambda_x z,0) = \frac{\partial_y c_{\widehat{x}}(z,0)}{\mu_x \lambda_x} + \lambda_x^{-1} \partial_y \theta_{\widehat{x}}(z,0) .$$
Now, set $\phi_{\widehat{x}}(z) := \frac{\partial_y c_{\widehat{x}}(z,0)}{\mu_x \lambda_x}$. This is a smooth function of $z$. A change of variable yields
$$ \partial_y \theta_{{\widehat{x}}}^u(z,0) = \phi_{\widehat{f}^{-1}({\widehat{x}})}(z \lambda_{f^{-1}(x)}^{-1}) + \lambda_{f^{-1}(x)}^{-1} \partial_y \theta_{\widehat{f}^{-1}({\widehat{x}})}(z \lambda_{f^{-1}(x)}^{-1},0)  ,$$
and so we find that
$$ \partial_y \theta_{\widehat{x}}^u(z,0) = \sum_{n=1}^\infty \lambda_x^{\langle -n \rangle} \phi_{\widehat{f}^{-n}({\widehat{x}})}(z \lambda_x^{\langle -n \rangle}) .$$
This expression is clearly smooth in $z$, and varies in a Hölder manner in ${\widehat{x}}$. Its derivative at zero,
$$ \partial_z \partial_y \theta_{\widehat{x}}^u(0,0) = \sum_{n=1}^\infty (\lambda_x^{\langle -n \rangle})^2 \phi_{\widehat{f}^{-n}({\widehat{x}})}'(0) ,$$
varies in a Hölder manner in ${\widehat{x}}$.
\end{proof}

\begin{definition}
Let $\mathbb{P}^1 M$ denote the projective tangent blundle of $M$, defined as the quotient of the unit tangent bundle $T^1 M$ under the action of the antipodal map $v \in T^1_x M \mapsto -v \in T^1_x M$. We consider the map:
$$ \mathbf{s}^u : \Omega \hookrightarrow \mathbb{P}^1 M $$
defined by $\mathbf{s}^u(x) := E^u(x) \in \mathbb{P}^1 M$. This is a section of the bundle $\mathbb{P}^1 M$: in other words, if $\pi_{\mathbb{P}^1 M} : \mathbb{P}^1 M \rightarrow M$ denotes the projection on the basepoint, we have $\pi_{\mathbb{P}^1M} \circ \mathbf{s}^u = Id_{\Omega}$.
\end{definition}

\begin{corollary}
The map $\mathbf{s}^u : \Omega \rightarrow \mathbb{P}^1 M$ is $C^{1+\alpha}$. Moreover, for any $p \in \Omega$, the map
$$ r \in W^u_{loc}(p) \cap \Omega \mapsto \partial_s \mathbf{s}^u(r) \in T(\mathbb{P}^1 M) $$
is $C^{1+\alpha}$ (and this, uniformly in $p$). Moreover, $\partial_u \partial_s \mathbf{s}^u$ and $\partial_u \partial_u \mathbf{s}^u$ are Hölder regular.
\end{corollary}

\begin{proof}
We use the coordinate chart $\iota_{\widehat{x}} : (-1,1)^2 \rightarrow M$ to create an associated coordinate chart
$ \overline{\iota}_{\widehat{x}} : (-1,1)^3 \rightarrow \mathbb{P}^1 M $ as follow:
$$ \overline{\iota}_{\widehat{x}}(z,y,\theta) := \text{Span} \Big( (\iota_{\widehat{x}})_* (\partial_z + \theta \partial_y) \Big) \in \mathbb{P}^1_{\iota_{\widehat{x}}(z,y)} M.$$
Using these coordinates, we see that
$$ \mathbf{s}^u( \iota_{\widehat{x}}(z,y) ) = E^u(\iota_{\widehat{x}}(z,y)) = (\iota_{\widehat{x}})_*(\partial_z + \theta_{\widehat{x}}^u(z,y) \partial_z ) = \overline{\iota}_{\widehat{x}}(z,y,\theta_{\widehat{x}}^u(z,y)).$$
In other words, $\mathbf{s}^u$ can be written, through the coordinates $\iota_{\widehat{x}}$ and $\overline{\iota}_{\widehat{x}}$, as $(z,y) \mapsto (z,y,\theta_{\widehat{x}}^u(z,y))$.
We know that $\theta_{\widehat{x}}(z,y)$ is $C^{1+\alpha}$, and that $z \mapsto \partial_y \theta_{\widehat{x}}^u(z,0) $ is smooth (and varies in a Hölder manner in ${\widehat{x}}$).
We also know that $z \mapsto \partial_z \theta_{\widehat{x}}^u(z,0) = 0$ is smooth. Hence, $z \mapsto (d\theta_{\widehat{x}}^u)_{(z,0)}$ is smooth, and varies in a Hölder manner in ${\widehat{x}}$. It follows that $r \in W^u_{loc}(p) \mapsto (d \mathbf{s}^u)_r $ is also smooth. Finally, $E^s$ is $C^{1+\alpha}$, and so 
$$r \in W^u_{loc}(p) \mapsto \partial_s \mathbf{s}^u(r) = (d \mathbf{s}^u)_r(\partial_s) \in T(\mathbb{P}^1 M) $$
is also $C^{1+\alpha}$. The Hölder property of $\partial_u \partial_s \mathbf{s}^u$ and $\partial_u \partial_u \mathbf{s}^u$ also follows from the behavior of $\theta_x^u$ in coordinates.
\end{proof}

\begin{theorem}
Define, for $p \in \Omega$, $\tau_f(p) := \ln \|{ (df)_p}_{|E^u(p)} \|$. In our area-preserving, Axiom A, 2-dimensional context, $\tau_f \in C^{1+\alpha}(\Omega,\mathbb{R})$. Moreover, for all $p \in \Omega$, the map $$ r \in W^u_{loc}(p) \cap \Omega \longrightarrow \partial_s \tau_f(r) \in \mathbb{R} $$
is $C^{1+\alpha}$, and this, uniformly in $p \in \Omega$. Finally, the map $\partial_u \partial_s \tau_f$ is Hölder regular on $\Omega$.
\end{theorem}

\begin{proof}
Let $\tilde{\Omega}$ denote a small enough open neighborhood of $s^u(\Omega) \subset \mathbb{P}^1 M$. Then $$ \overline{\tau}_f : V \in \tilde{\Omega} \longrightarrow \ln \| (df)_{|V} \| \in \mathbb{R} $$
is smooth. Moreover, $\tau_f = \overline{\tau}_f \circ \mathbf{s}^u$. It follows that $$ \partial_s \tau_f(r) = (d \overline{\tau}_f)_{s^u(r)}( \partial_s \mathbf{s}^u(r) ) $$ is $C^{1+\alpha}(W^u_{loc}(p),\mathbb{R})$. Moreover, $\partial_u \partial_s \tau_f(p)$ is Hölder regular for $p \in \Omega$, since $\overline{\tau}_f$ is smooth, and $\mathbf{s}^u, \partial_s \mathbf{s}^u,\partial_u \mathbf{s}^u, \partial_u \partial_s \mathbf{s}^u$ are Hölder on $\Omega$.
\end{proof}

\subsection{Some regularity for the distance function}

The goal of this subsection is to prove Proposition 5.4.8. Let us first define some notations.

\begin{definition}
Let ${\widehat{x}} \in \widehat{\Omega}$. For each $y \in \Omega_{\widehat{x}}^s \cap (-\rho,\rho)$, we parametrise the local unstable manifold $\widetilde{W}^u_{{\widehat{x}}}(y) := \iota_{\widehat{x}}^{-1}(W^u_{loc}(\Phi_{\widehat{x}}^s(y)))$ in the following way: $\widetilde{W}^u_{{\widehat{x}}}(y) = \{ (z,\mathcal{G}_{\widehat{x}}(z,y)), \ z \in (-\rho,\rho) \}$.
The function $z \mapsto \mathcal{G}_{\widehat{x}}(z,y)$ is smooth, and the function $\mathcal{G}_{\widehat{x}}$ is $C^{1+\alpha}$ by the properties of the local unstable foliation in our two-dimensional context.
\end{definition}

\begin{lemma}
For all $z \in (-\rho,\rho)$, $\partial_y \mathcal{G}_{\widehat{x}}(z,0) = 1$.
\end{lemma}

\begin{proof}
We will use the invariance of the unstable foliation under the dynamics. The formula $f_{\widehat{x}} \widetilde{W}_{\widehat{x}}^s(y) = \widetilde{W}_{\widehat{f}({\widehat{x}})}^s( \mu_x y )$ can be rewritten as
$$ f_{\widehat{x}}(z,\mathcal{G}_{\widehat{x}}(z,y)) = ( \psi_{\widehat{x}}(z,y), \mathcal{G}_{\widehat{f}({\widehat{x}})}(\psi_{\widehat{x}}(z,y),\mu_x y) ) ,$$
where $\psi_{\widehat{x}}(z,y)$ is some (uniformly) $C^{1+\alpha}$ function. We can compute its value when $y=0$: since $\mathcal{G}_{\widehat{x}}(z,0) = 0$, we find $f_{\widehat{x}}(z,\mathcal{G}_{\widehat{x}}(z,0)) = f_{\widehat{x}}(z,0) = (\lambda_x z,0)$ and so $\psi_{\widehat{x}}(z,0) = \lambda_x z$. We apply $\frac{\partial}{\partial y}_{|y=0}$ to our invariance relation and we find:
$$ \partial_y f_{\widehat{x}}(z,0) \partial_y \mathcal{G}_{\widehat{x}}(z,0) = ( \partial_y \psi_{\widehat{x}}(z,0), \partial_z \mathcal{G}_{\widehat{f}({\widehat{x}})}(\psi_{\widehat{x}}(z,0),0) ) \partial_y \psi_{\widehat{x}}(z,0) + \partial_y \mathcal{G}_{\widehat{f}({\widehat{x}})}(\psi_{\widehat{x}}(z,0),0)  \mu_x ). $$
Since $\partial_z \mathcal{G}_{\widehat{f}({\widehat{x}})}(\lambda_x z,0) = 0$, the equation given by the second coordinate can be rewritten as
$$ \mu_x \partial_y \mathcal{G}_{\widehat{x}}(z,0) = \partial_y \mathcal{G}_{\widehat{f}({\widehat{x}})}(\lambda_x z,0) \mu_x .$$
Hence, for all $z$, and for all $n \geq 0$, $$ \partial_y \mathcal{G}_{\widehat{x}}(z,0) = \partial_y \mathcal{G}_{\widehat{f}^{-n}({\widehat{x}})}( \lambda_x^{\langle -n \rangle} z,0) \underset{n \rightarrow \infty}{\longrightarrow} 1, $$ 
since $\partial_y \mathcal{G}_{\widehat{x}}(0,0) = 1$ (because $\partial_y \mathcal{G}_{\widehat{x}}(0,y) = y$), and since all the maps $\mathcal{G}_{\widehat{x}}(z,y)$ are uniformly $C^{1+\alpha}$.
\end{proof}

\begin{remark}
The path $\gamma_{{\widehat{x}},y}(z) := (z,\mathcal{G}_{\widehat{x}}(z,y))$ have normalised derivative
$$ \frac{\gamma_{{\widehat{x}},y}'(z)}{|\gamma_{{\widehat{x}},y}(z)|} = \Big( \frac{1}{\sqrt{1+\partial_z \mathcal{G}_{\widehat{x}}(z,y)^2}} , \frac{\partial_z \mathcal{G}_{\widehat{x}}(z,y)}{\sqrt{1+\partial_z \mathcal{G}_{\widehat{x}}(z,y)^2}} \Big) .$$
This function have the same regularity than $E^u$: that is, smooth in the $z$ variable, and $C^{1+\alpha}$ in the $(z,y)$ variable. It follows that $\partial_z \mathcal{G}_{\widehat{x}}$ is $C^{1+\alpha}$. Moreover, it follows from the previous subsection that $ z \mapsto \partial_y \partial_z \mathcal{G}_{\widehat{x}}(z,0)$ is smooth.
\end{remark}

\begin{definition}
Let us denote by $S_{\widehat{x}}(z,y)$ the symmetric $2\times 2$ matrix representing the Riemannian metric of $M$ through the coordinates $\iota_{\widehat{x}}$. In other words, for any path $\gamma(t)  \in (-\rho,\rho)^2$ in coordinates, its lenght seen as a path $\iota_{\widehat{x}} \circ \gamma$ on the manifold can be computed by the formula
$$ \text{Lenght}_{\widehat{x}}(\gamma) = \int_a^b \langle S_{\widehat{x}}( \gamma(t)) \dot{\gamma}(t), \dot{\gamma}(t) \rangle dt ,$$
where $\langle \cdot , \cdot \rangle$ is the usual scalar product in $\mathbb{R}^2$. The function $S_{\widehat{x}}$ is smooth (and this, uniformly in ${\widehat{x}}$).
\end{definition}

We are ready to state our main technical lemma.

\begin{lemma}
Denote by $\mathbf{L}_{\widehat{x}}(y,z) := \text{Lenght}_{\widehat{x}}((\gamma_{{\widehat{x}},y})_{|[0,z]})$. Then, $\mathbf{L}_{\widehat{x}}$ is smooth in $z$, $C^{1+\alpha}$ in $(z,y)$, and moreover the map $z \mapsto \partial_y \mathbf{L}_{\widehat{x}}(0,z)$ is smooth.
\end{lemma}

\begin{proof}
Denoting $\gamma_{\widehat{x}}(z,y) = (z,\mathcal{G}_{\widehat{x}}(z,y))$, we have the formula
$$ \mathbf{L}_{\widehat{x}}(z,y) = \int_0^z \langle S_{\widehat{x}}(\gamma_{\widehat{x}}(t,y)) \partial_z{\gamma}_{\widehat{x}}(t,y) , \partial_z{\gamma}_{\widehat{x}}(t,y) \rangle dt .$$
Since $\partial_z \mathcal{G}_{\widehat{x}}(\cdot,y)$ and $\mathcal{G}_{\widehat{x}}(\cdot,y)$ are smooth, the map $\mathbf{L}_{\widehat{x}}(\cdot,y)$ is also smooth. Moreover, since $\partial_z \mathcal{G}_{\widehat{x}}$ and $\mathcal{G}_{\widehat{x}}$ are $C^{1+\alpha}$, the map $\mathbf{L}_{\widehat{x}}$ also is. To compute $\partial_y \mathbf{L}_{\widehat{x}}(z,0)$, notice first that $$\frac{d}{dy}_{|y=0} S_{\widehat{x}}(\gamma_{\widehat{x}}(z,y)) = \partial_y S_{\widehat{x}}(\gamma_{\widehat{x}}(z,0)) \partial_y \mathcal{G}_{\widehat{x}}(z,0) = \partial_y S_{\widehat{x}}(z,0) \partial_y \mathcal{G}_{\widehat{x}}(z,0),$$ so that:
$$ \partial_y\mathbf{L}_{\widehat{x}}(z,0) = \int_0^z \big\langle  \partial_y S_{\widehat{x}}(t,0) \partial_y \mathcal{G}_{\widehat{x}}(t,0) \partial_z \gamma_{\widehat{x}}(t,0), \partial_z \gamma_{\widehat{x}}(t,0) \big\rangle dt $$ $$+ \int_0^z \langle S_{\widehat{x}}(\gamma_{\widehat{x}}(t,0)) \partial_z \gamma_{\widehat{x}}(t,0), \partial_y \partial_z \gamma_{\widehat{x}}(t,0) \rangle dt $$
$$ = \int_0^z \big\langle  \partial_y S_{\widehat{x}}(t,0) \mathbf{e}_z + S_{\widehat{x}}(t,0) \partial_y \partial_z \gamma_{\widehat{x}}(t,0), \ \mathbf{e}_z \big\rangle dt, $$
where $\mathbf{e}_z := (1,0)$, and $\partial_y \partial_z \gamma_{\widehat{x}}(t,0) = (0,\partial_y \partial_z \mathcal{G}_{\widehat{x}}(z,0))$. It follows that $\partial_y \mathbf{L}_{\widehat{x}}(\cdot,0)$ is smooth.
\end{proof}

Before getting to our main technical estimate, we need one last preliminary lemma.

\begin{lemma}
For ${\widehat{x}} \in \widehat{\Omega}$ and $z \in \Omega_{\widehat{x}}^u \cap (-\rho,\rho)$, $y \in \Omega_{\widehat{x}}^s \cap (-\rho,\rho)$, define
$ \mathbf{C}_{\widehat{x}}(z,y) := \iota_{\widehat{x}}^{-1}( [\Phi_{\widehat{x}}^u(z),\Phi_{\widehat{x}}^s(y)] ) $. The map $\mathbf{C}_{\widehat{x}}$ is (uniformly in ${\widehat{x}}$) $C^{1+\alpha}$, and moreover, denoting $\mathbf{C}_{\widehat{x}}(z,y) = ( \pi_z \mathbf{C}_{\widehat{x}}(z,y),\pi_y \mathbf{C}_{\widehat{x}}(z,y) )$, we have:
\begin{itemize}
\item $\forall z, \ \pi_y \partial_y \mathbf{C}_{\widehat{x}}(z,0) = 1$
\item There exists $\alpha_{\widehat{x}} \in \mathbb{R}$ (continuous in $x$) such that $$ \pi_z \partial_y \mathbf{C}_{\widehat{x}}(z,0) = \alpha_{\widehat{x}} z + O\big(|z|^{3/2}\big) .$$
\end{itemize}
\end{lemma}

\begin{proof}
The relation $f([p,q]) = [f(p),f(q)]$ yields, in coordinates, the formula
$$ f_{\widehat{x}}( \mathbf{C}_{\widehat{x}}(z,y) ) = \mathbf{C}_{\widehat{f}({\widehat{x}})}(\lambda_x z,\mu_x y). $$
Applying $\frac{d}{dy}_{|y=0}$ then gives the relation $$ (df_{\widehat{x}})_{(z,0)}( \partial_y \mathbf{C}_{\widehat{x}}(z,0) ) = \partial_y \mathbf{C}_{\widehat{f}({\widehat{x}})}(\lambda_x z,0) \mu_x .$$
Keeping only the $y$ coordinates gives
$$ \partial_y \pi_y \mathbf{C}_{\widehat{x}}(z,0) = \partial_y \pi_y \mathbf{C}_{\widehat{f}({\widehat{x}})}(\lambda_x z,0) , $$
and it follows that $$ \partial_y \pi_y \mathbf{C}_{\widehat{x}}(z,0) = \lim_{n \rightarrow \infty} \partial_y \pi_y \mathbf{C}_{\widehat{f}^{-n}({\widehat{x}})}(\lambda_x^{\langle-n\rangle} z,0) = 1 $$
since $\mathbf{C}_{\widehat{x}}(0,y) = (0,y)$. Now, denoting $b_{\widehat{x}}(z) := \pi_z \partial_y f_{\widehat{x}}(z,0)$ (a smooth map in $z$), we see by isolating the $z$ coordinate that
$$ \lambda_x  \partial_y \pi_z \mathbf{C}_{\widehat{x}}(z,0) + b_{\widehat{x}}(z)  \partial_y \pi_y \mathbf{C}_{\widehat{x}}(z,0) = \partial_y \pi_z \mathbf{C}_{\widehat{f}({\widehat{x}})}(\lambda_x z,0) \mu_x , $$
which can be rewritten as
$$ \partial_y \pi_z \mathbf{C}_{\widehat{x}}(z,0) = B_{\widehat{x}}(z) + \mu_x \lambda_x^{-1} \partial_y \pi_z \mathbf{C}_{\widehat{f}({\widehat{x}})}(\lambda_x z,0), $$
where $B_{\widehat{x}}(z) = - \lambda_x^{-1} b_{\widehat{x}}(z)$. The properties of our coordinate system ensure that $B_{\widehat{x}}(0)=0$ and $\partial_y \pi_z \mathbf{C}_{\widehat{x}}(0,0)=0$. It follows that $\widetilde{B}_{\widehat{x}}(z) := z^{-1} B_{\widehat{x}}(z)$ is a smooth map. If we denote further $\widetilde{\mathbf{C}}_{\widehat{x}}(z) := z^{-1}{\partial_y \pi_z \mathbf{C}_{\widehat{x}}(z,0)}$, our previous relation can be rewritten
$$ \widetilde{\mathbf{C}}_{\widehat{x}}(z) = \widetilde{B}_{\widehat{x}}(z) + \mu_x \widetilde{\mathbf{C}}_{\widehat{f}({\widehat{x}})}(\lambda_x z).$$
It would be convenient to be able to see $\widetilde{\mathbf{B}}_{\widehat{x}}(z)$ as an infinite sum involving $\tilde{B}_{\widehat{x}}$, but notice that $\lambda_x^{\langle n \rangle} z$ is eventually larger than $\rho$ and falls away from the domain of definition of our coordinate system. So, for each (small) $z$, define $N_x(z) \geq 1$ such that $\lambda_x^{\langle n \rangle} |z| \simeq \sqrt{|z|}$, ie $\mu_x^{\langle n \rangle} \sim \sqrt{|z|}$ since $f$ is area preserving. We find:
$$ \widetilde{\mathbf{C}}_{\widehat{x}}(z) = \sum_{n=0}^{N_x(z)-1} \mu_x^{\langle n \rangle} \widetilde{B}_{\widehat{f}^n({\widehat{x}})}(\lambda_x^{\langle n \rangle} z) + \mu_x^{\langle N_x(z) \rangle} \widetilde{\mathbf{C}}_{\widehat{f}^{N_x(z)}({\widehat{x}})}(\lambda_x^{\langle N_x(z) \rangle} z) = \sum_{n=0}^{N_x(z)-1} \mu_x^{\langle n \rangle} \widetilde{B}_{\widehat{f}^n({\widehat{x}})}(\lambda_x^{\langle n \rangle} z) + O(\sqrt{z}).  $$
Hence:
$$ \Big| \widetilde{\mathbf{C}}_{\widehat{x}}(z) - \sum_{n=0}^\infty \mu_x^{\langle n \rangle} \widetilde{B}_{\widehat{f}^n({\widehat{x}})}(0) \Big| \leq \sum_{n=0}^{N_x(z)-1} |\mu_x^{\langle n \rangle}| |\widetilde{B}_{\widehat{f}^n({\widehat{x}})}(\lambda_x^{\langle n \rangle} z) - \widetilde{B}_{\widehat{f}^n({\widehat{x}})}(0) | + O(\sqrt{|z|}) = O(\sqrt{|z|}), $$
which gives the desired Taylor expansion.
\end{proof}

\begin{corollary}
Fix $p \in \Omega$ and $s \in W^s_{loc}(p) \cap \Omega$. There exists a smooth (uniformly in $p$) function $\phi_{p}: W^u_{loc}(p) \cap \Omega \rightarrow \mathbb{R}$ such that
$$ d^u(s,[r,s]) = d^u(p,r) + \phi_{p}(r) d^s(p,s) + O\Big( 
d^s(p,s)(d^u(p,r)^{1+\alpha} + d^s(p,s)^{\alpha}) \Big) $$
\end{corollary}

\begin{proof}

Let $p=x \in \Omega$ and $s \in W^s_{loc}(x) \cap \Omega$. Choose some ${\widehat{x}} \in \widehat{\Omega}$ in the fiber of $x$. Denote $y := (\Phi_{\widehat{x}}^s)^{-1}(s)$, $z := \Phi_{\widehat{x}}^u(r)$.
 First of all, we can write, denoting $\mathbf{e}_y = (0,1)$:
$$ d^s(p,s) = \int_0^y \langle S_{\widehat{x}}(0,t) \mathbf{e}_y, \mathbf{e}_y \rangle dt = y \int_0^1  \langle S_{\widehat{x}}(y t,0) \mathbf{e}_y, \mathbf{e}_y \rangle dt = y \langle S_{\widehat{x}}(0,0) \mathbf{e}_y, \mathbf{e}_y \rangle + O(y^2). $$
Recall that $|(\Phi_{\widehat{x}}^s)'(0)| = 1$, so that $\langle S_{\widehat{x}}(0,0) \mathbf{e}_y, \mathbf{e}_y \rangle=1$.
It follows that $y = d^s(p,s) + O(d^s(p,s)^2)$. Similarly, $z = d^u(p,r) + O(d^u(p,r)^2)$. Then, denoting $\mathbf{C}_{\widehat{x}}(z,y) = (\iota_{\widehat{x}})^{-1}( [ \Phi_{\widehat{x}}^u(z), \Phi_{\widehat{x}}^s(y) ] )$, notice that we have:
$$ d^u(s,[r,s]) = \mathbf{L}_{\widehat{x}}(\pi_z \mathbf{C}_{\widehat{x}}(z,y),y).$$
This expression is $C^{1+\alpha}$, and a Taylor expansion in the $y$ variable gives:
$$ \mathbf{L}_{\widehat{x}}(\pi_z \mathbf{C}_{\widehat{x}}(z,y),y) = \mathbf{L}_{\widehat{x}}(\pi_z \mathbf{C}_{\widehat{x}}(z,0),0) + y \cdot \frac{d}{dy}  \mathbf{L}_{\widehat{x}}( \pi_z \mathbf{C}_{\widehat{x}}(z,y),y)_{|y=0} + O(y^{1+\alpha})  $$
$$ =  \mathbf{L}_{\widehat{x}}(z,0) + \big(\partial_z \mathbf{L}_{\widehat{x}}(z,0) \partial_y \pi_z \mathbf{C}_{\widehat{x}}(z,0) + \partial_y \mathbf{L}_{\widehat{x}}(z,0)\big)y + O(y^{1+\alpha})  $$
$$ = \mathbf{L}_{\widehat{x}}(z,0) + \big( \alpha_{\widehat{x}} z \partial_z \mathbf{L}_{\widehat{x}}(z,0) + \partial_y \mathbf{L}_{\widehat{x}}(z,0)\big)y + O( 
z^{1+\alpha} y + y^{1+\alpha}). $$
The expression $\big( \alpha_{\widehat{x}} z \partial_z \mathbf{L}_{\widehat{x}}(z,0) + \partial_y \mathbf{L}_{\widehat{x}}(z,0)\big)$ is smooth in $z$. Using $z=(\Phi_{\widehat{x}}^u)^{-1}(r)$, $y = d^s(p,s) + O(d^s(p,s)^2)$, and $\mathbf{L}_{\widehat{x}}(z,0) = d^u(p,r)$ finally gives us
$$ d^u(s,[r,s]) = d^u(p,r) + \phi_{p}(r) d^s(p,s) + O\Big( 
d^s(p,s)(d^u(p,r)^{1+\alpha} + d^s(p,s)^{\alpha}) \Big) ,$$
for some smooth function $\phi_p : W^u_{loc}(p) \cap \Omega \rightarrow \mathbb{R}$.
\end{proof}

\section{Appendix B: Polynomials on perfect sets}

The goal of this subsection is to prove quantitative \say{equivalence of norms} lemmas for polynomials that we restrict to (uniformly) perfect sets. In particular, we prove Lemma 5.5.3.

\begin{lemma}
Let $d \geq 1$. Let $a_0 \dots a_d \in [-1,1]$ be $(d+1)$ distinct points. Then, for all $P \in \mathbb{R}_d[X]$,
$$  2^{-d} d^{-1} \min_{i \neq j} |a_i-a_j| \cdot \|P\|_{\infty,[-1,1]} \leq \max_i |P(a_i)| \leq \|P\|_{\infty,[-1,1]} . $$
\end{lemma}

\begin{proof}
One inequality is obvious. We prove the other one using Lagrange interpolating polynomials. Define
$$ L_i(x) := \frac{\prod_{j \neq i}(x-a_j)}{\prod_{j \neq i}(a_j-a_i)}. $$
From the formula $$\forall x \in \mathbb{R}, \ P(x) = \sum_{i=0}^d P(a_i) L_i(x),$$
we find
$$ \|P\|_{\infty,[-1,1]} \leq \sum_i |P(a_i)| \|L_i\|_{\infty,[-1,1]} \leq \frac{d \ 2^d}{\min_{i \neq j} |a_i - a_j|} \max_i |P(a_i)|.$$
\end{proof}

\begin{definition}
A metric space $(K,d)$ is $\kappa$-uniformly perfect if
$$ \forall \sigma \in (0,1),  \forall x \in K,  \exists y \in K, \ d(x,y) \in [\kappa \sigma, \sigma] .$$
We say that $K$ is uniformly perfect if $K$ is $\kappa$-uniformly perfect for some $\kappa \in (0,1)$.
\end{definition}

\begin{remark}
Notice that $K$ is uniformly perfect if and only if
$$\exists \kappa\in (0,1), \exists \sigma_0 \in (0,1), \forall \sigma \in (0,\sigma_0),  \forall x \in K,  \exists y \in K, \ d(x,y) \in [\kappa \sigma, \sigma] .$$
Indeed, in this case, $K$ would be $\kappa_0 \sigma_0$-uniformly perfect. Notice also that, if $K$ is uniformly perfect, then for any $x \in K$ and $r>0$, the space $K \cap B(x,r)$ is $\kappa r$-uniformly perfect.
\end{remark}

\begin{lemma}
There exists $\kappa>0$ such that, for all ${\widehat{x}} \in \widehat{\Omega}$, the set $\Omega_{\widehat{x}}^u \cap (-\rho,\rho)$ is $\kappa$-uniformly perfect. 
\end{lemma}

\begin{proof}
Let $\rho$ be small enough so that, if $d(x,y) \leq \rho$, then $[x,y]$ and $[y,x]$ makes sense. Since $\Omega$ is a basic set (and since it is not a periodic orbit), it is a perfect set. Moreover, the periodic orbits are dense in $\Omega$. Let us then fix a finite family $(x_i)_{i \in I}$ of periodic points of $\Omega$ such that $\bigcup_{i \in I} B(x_i,\rho/2) = \Omega$. For each $i \in I$, we construct a companion for $x_i$ in the following way: by density of periodic points, there exists $\tilde{y}_i \in B(x_i,\rho/2) \setminus \{x_i\}$ which is periodic. It follows that $\tilde{y}_i \notin W^s_{loc}(x_i)$. We then set $y_i := [\tilde{y_i},x_i] \in W^u_{loc}(x_i) \setminus \{x_i\}$. \\

Let us define $\delta := \min_{i \in I} \inf\{ d^u( [x_i,z],[y_i,z] ) \ | \ z \in B(x_i,\rho)  \}> 0$. For any $z \in \Omega$, there exists $i \in I$ such that $x \in B(x_i,\rho)$. Now, since $d^u( [x_i,z],[y_i,z] ) \geq \delta$, we know that $d^u(z,[x_i,z]) \geq \delta/3$ or $d^u(z,[y_i,z]) \geq \delta/3$. It follows from this discussion the following property:
$$ \exists \kappa_0 \in (0,1), \forall x \in \Omega, \exists z \in W^u_{loc}(x), \ d^u(x,z) \in [\kappa_0 \rho, \rho]. \quad (*) $$
Now, we will use the invariance of the unstable foliation under the dynamics to extend the property $(*)$ to any scale.
Notice that there exists $C_0 \geq 1$ such that, if $x,\tilde{x}$ are in the same unstable leaf with $d^u(x,\tilde{x}) \leq \rho$, then
$$ \sum_{k=0}^\infty |\tau_f( f^{-k} x) - \tau_f( f^{-k} \tilde{x}) | \leq C_0 .$$
Now, let $x \in \Omega$ and let $\sigma \in (0,\rho)$. Let $n(x,\sigma) \geq 1$ be the largest integer such that $\partial_u f^{n(x,\sigma)}(x) \leq \sigma^{-1} \rho e^{-C_0}$. Define $\overline{x} := f^{n(x,\sigma)}(x)$. By the property $(*)$, there exists $\overline{z} \in W^u_{loc}(x)$ such that $d^u(\widehat{x},\widehat{z}) \in [\kappa_0 \rho,\rho]$. Define $z := f^{-n(x,\sigma)}(\overline{z})$. We have, by the mean value theorem, for some $t \in W^u_{loc}(x)$ between $x$ and $z$:
$$ d(x,z) = \partial_u f^{-n(x,\sigma)}(t) d^u(\overline{x},\overline{z}) \in [e^{-2C_0} \sigma \kappa_0 \rho, \sigma] .$$
We have thus proved the following property:
$$ \exists \kappa_1,\sigma_0 \in (0,1), \forall \sigma \in (0,\sigma_0), \forall x \in \Omega, \exists z \in W^u_{loc}(x), \ d^u(x,z) \in [\kappa_1 \sigma, \sigma]. \quad (**) $$
The conclusion follows from the fact that the maps $\Phi_{\widehat{x}}^u : (-\rho,\rho) \rightarrow W^u_{loc}(x)$ are uniformly $C^{1}$ (thus, uniformly Lipschitz) maps. 
\end{proof}

\begin{corollary}
Let $M$ be a complete Riemannian surface and let $f:M \rightarrow M$ be a smooth Axiom A diffeomorphism. Let $\Omega \subset M$ be a basic set. Then $\Omega$ is uniformly perfect.
\end{corollary}

\begin{proof}
Lemma 5.8.4 ensure that there exists $\kappa$ such that, for any ${\widehat{x}} \in \widehat{\Omega}$, $\Omega_{\widehat{x}}^u \cap (-\rho,\rho)$ is $\kappa$-uniformly perfect. Symmetrically, (reducing $\kappa$ if necessary,) for any ${\widehat{x}} \in \widehat{\Omega}$, $\Omega_{\widehat{x}}^s \cap (-\rho,\rho)$ is $\kappa$-uniformly perfect. The local product structure of $\Omega$ and the $C^{1}$ regularity of the holonomies yields the result.
\end{proof}

\begin{lemma}
Let $K$ be a $\kappa$-uniformly perfect metric set. For all $k \geq 1$, there exists $\delta_{k} \in (0,1)$ that depends only on $\kappa$ and $k$ such that 
$$ \exists (x_i)_{i \in \llbracket 1, k \rrbracket} \in K^k,  \forall i \neq j, \ B(x_i,\delta_k) \cap B(x_j,\delta_k) = \emptyset. $$
\end{lemma}

\begin{proof}
The proof goes by induction on $k$. For $k=1$, this is trivial. Suppose the lemma true for some $k \geq 1$. There exists a family of points $x_i \in K$, for $i = 1, \dots ,k$, and there exists $\delta_{k} \in (0,1)$ that only depends on $k$ and $\kappa$, such that the balls $ B(x_i,\delta_k) $ are disjoint. Now, since $K$ is $\kappa$-uniformly perfect, there exists a point $x_{k+1} \in K$ such that $d(x_1,x_{k+1}) \in [\kappa \delta_k/3, \delta_k/3]$.
Setting $\delta_{k+1} := \kappa \delta_k/10$ then ensure that, for any $i \neq j \in \llbracket 1,k+1 \rrbracket$, one have
$$ B(x_i,\delta_{k+1}) \cap B(x_j,\delta_{k+1}) = \emptyset. $$
which concludes the proof.
\end{proof}

\begin{lemma}
Let $K \subset (-\rho,\rho)$ be a $\kappa$-uniformly perfect set that contains zero (where $\kappa< \rho/100$). Let $d \geq 1$. Then there exists $\tilde{\kappa} \in (0,1)$ (that depends only on $\kappa$ and $d$) such that for all $P \in \mathbb{R}_d[X]$, there exists $z_0 \in K \cap (-\rho/2,\rho/2)$ such that
$$ \forall z \in [z_0-\tilde{\kappa},z_0+\tilde{\kappa}], \ |P(z)| \geq \tilde{\kappa} \|P\|_{\infty,[-1,1]}. $$
\end{lemma}

\begin{proof}
Using Lemma 5.8.6, we first cut $(-\rho/2,\rho/2)$ into $2d$ subintervals $(I_i)$ such that for all $i$, there exists $x_i \in I_i \cap K$ satisfying $B(x_i,\delta) \subset I_i$. We know that we can choose $\delta$ to depend only on $\kappa$ and $d$.\\

Then, similarly, for each $i \in \llbracket 1, 2d \rrbracket$, we can choose two families $A_i = (a_k^{(i)})_k$ and $B_i = (b_k^{(i)})$ of $d+1$ points lying in $B(x_i,\delta) \cap K$, satisfying:
$$ \min_{j \neq k} |a_j^{(i)} - a_k^{(i)}| \geq \delta' ,\quad \min_{j \neq k} |b_j^{(i)} - b_k^{(i)}| \geq \delta' ,\quad d(A_i,B_i) \geq \delta'  $$
for some $\delta' < \delta$ that depends only on $\kappa$ and $d$. In fact, we can further suppose that for all $i \in I$, there exists a point $c^{(i)} \in K \cap I_i$ such that $$d(c^{(i)},A_i) \geq \delta' , d(c^{(i)},B_i) \geq \delta'.$$
The role of each point being symetric, we can further assume by renaming our points that we have $a_j^{(i)} \leq c^{(i)} \leq b_j^{(i)}$ for each $i,j$. Denoting by $\|P\|_{A_i} := \max_k |P(a^{(i)}_k)|$ and $\|P\|_{B_i} := \max_k |P(b^{(i)}_k)|$, Lemma 5.8.1 ensures that
$$2^{-d} d^{-1} \delta' \|P\|_{\infty,[-1,1]} \leq \|P\|_{A_i} \leq \|P\|_{\infty,[-1,1]} ,\quad 2^{-d} d^{-1} \delta' \|P\|_{\infty,[-1,1]} \leq \|P\|_{B_i} \leq \|P\|_{\infty,[-1,1]} .$$
Now, let $P \in \mathbb{R}_d[X]$. Since $P'$ has degree $d-1$, it vanish at most $d-1$ times. Moreover, $P$ vanish at most $d$ times. It follows that there exists $i_{0} \in \llbracket 1,2d \rrbracket$ such that $P_{|I_{i_0}}$ is monotonous and doesn't change signs. The inequality
$$ \min( \|P\|_{A_{i_0}}, \|P\|_{B_{i_0}} )\geq \frac{\delta'}{d 2^d} \|P\|_{\infty,[0,1]}$$
then ensures that there exists $j_0$ and $k_0$ such that
$ \min (|P(a_{j_0}^{(i_0)})|, |P(b_{k_0}^{(i_0)})| ) \geq \frac{\delta}{d 2^d} \|P\|_{\infty,[-1,1]} $.
The monotonicity of $P$ on $I_i$ and the fact that $P(a_{j_0}^{(i_0)})$ and $P(b_{k_0}^{(i_0)})$ have same sign ensure that
$$ \forall z \in [a_{j_0}^{(i_0)},b_{k_0}^{(i_0)}], \ |P(z)| \geq \frac{\delta}{d 2^d} \|P\|_{\infty,[0,1]} . $$
The fact that $[c^{(i_0)} - \delta',c^{(i_0)} + \delta'] \subset [a_{j_0}^{(i_0)},b_{k_0}^{(i_0)}]$ proves the desired estimate.
\end{proof}

\begin{corollary}
Let $f$ be an Axiom A diffeomorphism on a surface $M$. Denote by $\Omega$ one of its basic sets, suppose that $\det df = 1$ on $\Omega$. There exists $\kappa>0$ such that, for all ${\widehat{x}} \in \widehat{\Omega}$, for all $P \in \mathbb{R}_{d_Z}[X]$, there exists $z_0 \in \Omega_{\widehat{x}}^u \cap (-\rho/2,\rho/2)$ such that
$$ \forall z \in [z_0 - \kappa, z_0 + \kappa], \ |P(z)| \geq \kappa \|P\|_{C^\alpha((-\rho,\rho))}. $$
If we denote $P = \sum_{k=0}^{d_Z} a_k z^k \in \mathbb{R}_{d_Z}[X]$, then
$$\max_{0 \leq k \leq d_Z} |a_k| \leq \kappa^{-1} \|P\|_{L^\infty( (-\rho,\rho) \cap \Omega_{\widehat{x}}^u )}. $$
\end{corollary}

\begin{proof}
This follows easily from Lemma 5.8.7, Lemma 5.8.4, and from the fact that the norms \newline
$\| \cdot \|_{C^\alpha((-\rho,\rho))}$, $\| \cdot \|_{\infty,[-1,1]}$ and $P \mapsto \max_k |a_k|$ are equivalent on $\mathbb{R}_{d_Z}[X]$.
\end{proof}

\begin{remark}
This Corollary holds even if $\det df \neq 1$ on $\Omega$. In this setting, the construction of the linearizing coordinates $\Phi_{\widehat{x}}^u$ follows through without difficulty and the same result applies. The only important condition is the fact that $M$ is a surface.
\end{remark}

\section{Appendix C: The doubling property}

In this section, we prove the doubling property for equilibrium states in our 2-dimensional context. For $x \in \Omega$ and $r >0$, we will define $U(x,r) := B(x,r) \cap \Omega \cap W^u_{loc}(x)$ and $S(x,r) := B(x,r) \cap \Omega \cap W^s_{loc}(x)$. The following Proposition is taken from \cite{Cl20} (Theorem 3.1, Theorem 3.4, Theorem 3.7.)

\begin{proposition}
Let $\varphi:\Omega \rightarrow \mathbb{R}$ a Hölder map.  There exists a family of (Borel) measures $(\mu_x^u)_{x \in \Omega}$, where $\mu_x^u$ is supported on the whole unstable manifold $W^u(x)$, such that the following properties are satified:
\begin{itemize}
\item $\forall \rho \in (0,1), 
\exists K \geq 1, \forall x\in \Omega, \forall r \in [\rho,\rho^{-1}], \ \mu_x^u(U(x,r)) \in [K^{-1},K] $,
\item $\forall x \in \Omega, \ (f_* d\mu_x^u)(f(z)) = e^{\varphi(z)-P(\varphi)} d\mu_{f(x)}^u(f(z))$.
\end{itemize}
\end{proposition}

A similar statement holds in the stable direction (see Theorem 3.9 in \cite{Cl20}).

\begin{proposition}
Let $\varphi:\Omega \rightarrow \mathbb{R}$ a Hölder map.  There exists a family of (Borel) measures $(\mu_x^s)_{x \in \Omega}$, where $\mu_x^s$ is supported on the whole unstable manifold $W^s(x)$, such that the following properties are satified:
\begin{itemize}
\item $\forall \rho \in (0,1), 
\exists K \geq 1, \forall x\in \Omega, \ \forall r \in [\rho,\rho^{-1}], \ \mu_x^s(S(x,r)) \in [K^{-1},K] $,
\item $\forall x \in \Omega, \ (f_* d\mu_x^s)(f(y)) = e^{-\varphi(y)+P(\varphi)} d\mu_{f(x)}^u(f(y))$.
\end{itemize}
\end{proposition}
Our starting point is the following \say{local product structure} result for equilibrium states (See Theorem 3.10 in \cite{Cl20} or \cite{Le00}). 

\begin{theorem}
Let $R=[U_x,S_x] \subset \Omega$ be a rectangle with small diameter (with $x \in U_x \subset W^u_{loc}(x)$ and $x \in S_x \subset W^s_{loc}(x)$). Let $\mu$ be an equilibrium state associated to $\varphi:\Omega \rightarrow \mathbb{R}$. Then using the previous construction, we have (up to a scalar multiple):
$$ \mu(R) = \int_{S_x} \int_{U_x} e^{\omega(z,y)} d\mu_x^u(z) d\mu_x^s(y), $$
where $\omega(z,y) = O(1)$ is a continuous map that vanish when $z=y=x$.
\end{theorem}

\begin{remark}
It follows from the last Theorem that there exists $K \geq 1$ such that, for any rectangle $R=[U_x,S_s]$ of diameter $\leq \rho$, where $x \in U_x \subset W^u_{loc}(x)$ and $x \in S_x \subset W^s_{loc}(x)$:
$$  K^{-1} \mu_x^u(U_x) \mu_x^s(S_x) \leq \mu(R) \leq K \mu_x^u(U_x) \mu_x^s(S_x) .$$
\end{remark}

\begin{lemma}
There exists $K \geq 1$ such that, for any $n \geq 1$ and for any $r \in [\rho,\rho^{-1}]$,
$$ K^{-1} e^{-n P(\varphi) + S_n \varphi(f^{-n} x)} \leq \mu_{f^{-n}(x)}^u(f^{-n}(U(x,r))) \leq K e^{-n P(\varphi) + S_n \varphi(f^{-n} x)}, $$
and similarly:
$$ K^{-1} e^{-n P(\varphi) + S_n \varphi(x)} \leq \mu_{f^n(x)}^s(f^{n}(S(x,r))) \leq K e^{-n P(\varphi) + S_n \varphi(x)}. $$
\end{lemma}

\begin{proof}
This follows easily from the fact that $\varphi$ is Hölder. In particular, for $y \in W^u_{loc}(z)$ with $d^u(x,y) \leq \rho^{-1}$, notice that
$$ \Big|\sum_{k=0}^n \varphi(f^{-k} x) - \sum_{k=0}^n \varphi(f^{-k} y) \Big| \leq C(\varphi). $$
Hence, Proposition 5.9.2 yields
$$ \mu_{f^{-n}(x)}^u(f^{-n} U(x,r)) = \int_{U(x,r)} e^{-n P(\varphi) + S_n \varphi(f^{-n z})} d\mu_x^u(z) \in K^{\pm 1} e^{-nP(\varphi) + S_n \varphi(f^{-n} z)} \mu_x^u(U(x,r)).$$
\end{proof}

\begin{lemma}
Let $x \in \Omega$ and let $r \in (0,\rho)$. There exists $C_0 \geq 1$ such that the following holds. If we let $n(x,r)$ be the smallest integer such that $\partial_u f^n(x) \geq r/\rho$, then:
$$ U(f^{n(x,r)}(x),e^{-C_0}) \subset f^{n(x,r)}( U(x,r) ) \subset  U(f^{n(x,r)}(x),e^{C_0}) .$$
Furthermore, for any $D \geq 1$, there exists $n_0(D) \geq 1$ such that for any $x \in \Omega$ and $r>0$: $|n(x,Dr) - n(x,r)| \leq n_0(D)$.
\end{lemma}

\begin{proof}

This follows from the following observation. Let $t \in W^u_{loc}(x)$ be such that $d^u(x,t) \leq \rho^{-1}$. Then, for all $n \geq 1$:
$$ \frac{\partial_u f^{-n}(x)}{\partial_u f^{-n}(t)} = \exp \big( \sum_{k=0}^n \tau_f(f^{-k} x) - \tau_f(f^{-k} t) \big) \leq \exp\big( \|\tau_f\|_{C^1} \sum_{k \leq n} \lambda_-^{-k} d^u(x,t) \big) \leq e^{C_0} $$
for some constant $C_0$. The same computation replacing the role of $x$ and $t$ yields $$ \frac{\partial_u f^{-n}(x)}{\partial_u f^{-n}(t)} \in [e^{-C_0},e^{C_0}] .$$
Now, since $W^u_{loc}(x)$ is a 1-dimensional manifold, the mean value theorem ensure that, for all $r \leq \rho^{-1}$, there exists $t \in U(x,r)$ such that
$$ \text{diam}(f^{-n}(U(x,r))) = \partial_u f^{-n}(t) \text{diam}(U(x,r)), $$
and the result follows.
\end{proof}

\begin{lemma}
For all $D \geq 1$, there exists a doubling constant $C_{dou}(D) \geq 1$ such that, for any $x \in \Omega$, and for any $r \in (0,\rho)$:
$$ \mu_x^u( U(x,D r) ) \leq C_{dou} \mu_x^u(U(x,r)) \quad , \quad \mu_x^s( S(x,D r) ) \leq C_{dou} \mu_x^s(S(x,r)) .$$
\end{lemma}

\begin{proof}
This is essentially a consequence of the Gibbs estimates and our two-dimensional context. We have :
$$ \mu_x^u( U(x,2 r) ) \leq \mu_x^u( f^{-n(x,2r)} U(f^{n(x,2r} x,K)  )  $$
$$ \leq K e^{-n(x,2r) P(\varphi) + S_{n(x,2r)} \varphi(f^{-n(x,2r)})(x))} $$
$$ \leq K e^{n_0 P(\varphi) + n_0 \|\varphi\|_\infty} e^{-n(x,r) P(\varphi) + S_{n(x,r)} \varphi(f^{-n(x,r)})(x))} $$
$$ \leq K^2 e^{n_0 P(\varphi) + n_0 \|\varphi\|_\infty} \mu_x^u(f^{-n} U(f^nx,K^{-1})) $$
$$  \leq K^2 e^{n_0 P(\varphi) + n_0 \|\varphi\|_\infty} \mu_x^u(U(x,r)). $$
The same reasoning applies to $\mu_x^s$.
\end{proof}

\begin{lemma}
Let $f$ be a smooth Axiom A map on a surface, let $\Omega$ be a basic set, and let $\mu$ be an equilibrium state associated to a Hölder potential $\varphi:\Omega \rightarrow \mathbb{R}$. The equilibrium state $\mu$ is doubling, that is, there exists $C \geq 1$ such that, for all $x \in \Omega$ and $r \in (0,\rho)$:
$$ \mu(B(x,2r)) \leq C \mu(B(x,r)) .$$
\end{lemma}

\begin{proof}
First of all, properties of the bracket ensure that there exists $K \geq 1$ such that, for all $x \in \Omega$ and $r \in (0,\rho)$:
$$R(x,K^{-1}r) \subset B(x,r) \subset R(x,Kr) ,$$
where $R(x,r) := [U(x,r),S(x,r)]$. It follows that
$$ \mu(B(x,2r)) \leq \mu(R(x,2Kr)) \leq K \mu_x^u(U(x,2Kr)) \mu_x^s(S(x,2Kr)) $$
$$ \leq K C_{dou}(2K^2)^2 \mu_x^u(U(x,K^{-1} r)) \mu_x^s(S(x,K^{-1}r)) \leq K^2 C_{dou}(2K^2)^2 \mu(R(x,K^{-1}R)) $$
$$ \leq K^2 C_{dou}(2K^2)^2 \mu(B(x,r)), $$
which proves the doubling property. \end{proof}

\begin{remark}
In fact, we see from this proof that, if we denote $R_{\beta}(x,\sigma) := [U(x,\sigma),S(x,\sigma^{\beta})]$, then we can write a doubling property of the form:
$$ \forall C_1,\beta >1 ,\exists C_2 \geq 1, \forall x \in \Omega, \forall \sigma \in (0,1), \  \mu(R_\beta(x,C_1 \sigma)) \leq C_2 \mu(R_\beta(x,\sigma)) .$$
\end{remark}

We conclude this part by a corollary on rectangles coming from a Markov partition, holding in our 2-dimensional context.

\begin{lemma}
Fix $(R_a)_{a \in \mathcal{A}}$ a finite markov partition. Let $\beta_Z>1$, and denote $\text{Rect}_{\beta_Z}(\sigma)$ the set of rectangles of the form $R := \bigcap_{k={-k_1}}^{k_2} R_{a_k} $ where $\text{diam}^u(R) \simeq \sigma$ and $\text{diam}^s(R) \simeq \sigma^{\beta_Z}$. Then, there exists $C_1,C_2 \geq 10$ such that for all $R \in \text{Rect}_{\beta_Z}(\sigma)$ and for all $p \in R$:
$$ R \subset R_{\beta_Z}(p,C_1 \sigma) \quad ; \quad \mu(R_{\beta_Z}(p,C_1 \sigma)) \leq C_2 \mu(R).$$
\end{lemma}

\begin{proof}
Let $R \in \text{Rect}_{\beta_Z}(\sigma)$. Denoting $(R_a)_{a \in \mathcal{A}}$ the Markov partition, we fix for each $a \in \mathcal{A}$, $x_a \in \overset{\circ}{R_a}$, and $x_a \in U_a \subset W^u_{loc}(x) \cap \Omega$, $x_a \in S_a \subset W^s_{loc}(x) \cap \Omega$ such that $R_a = [U_a,S_a]$. By construction, there exists $\delta > 0$ such that $B(x_a,\delta) \cap \Omega \subset R_a$. Then, since $W^u$ and $W^s$ are one-dimensional, it is easy to see that, denoting $U_{\mathbf{a}} := U_{a_1} \cap f^{-1} R_{a_2} \cap \dots \cap f^{-n}(R_{a_n})$, we have
$$ B(x_\mathbf{a},\delta \text{diam}(U_\mathbf{a})  ) \subset U_\mathbf{a} $$
for some point $x_\mathbf{a} \in U_\mathbf{a}$. A similar statement hold for $S_\mathbf{b} := S_{b_1} \cap f(S_{b_2}) \cap \dots \cap f(S_{b_n})$, and combining the two gives the following statement:
$$ \forall R \in \text{Rect}_{\beta_Z}(\sigma), \ \exists x_R \in R, \ R_{\beta_Z}(x_{R}, \delta \sigma ) \subset R.$$
Now, for any $p \in R$, we always have $R \subset R_{\beta_Z}(p,C_1 \sigma)$ for some constant $C_1$, and finally, if $\delta$ is chosen small enough, we find:
$$ R_{\beta_Z}(x_R, \delta \sigma) \subset R \subset R_{\beta_Z}(p, C_1 \sigma) \subset R_{\beta_Z}(x_R, \delta^{-1} \sigma ). $$
The doubling property finally yields:
$$ \mu(R_{\beta_Z}(p,C_1 \sigma)) \leq  \mu(R_{\beta_Z}(x_R ,\delta^{-1} \sigma)) \leq C_2(\delta) \mu(R_{\beta_Z}(x_R ,\delta \sigma)) \leq C(\delta) \mu(R), $$
for some constant $C(\delta)$, that depends only on $\delta$, which is given by the doubling property of $\mu$.
\end{proof}







\cleardoublepage

\chapter{About the geodesic flow on hyperbolic surfaces}

\section{Introduction}
In this final Chapter, we are interested in studying the Fourier properties of equilibrium states for the geodesic flow on convex-cocompact surfaces of constant negative curvature. More details on our setting will be explained during the Chapter, but let's quickly introduce the main objects at play. A useful reference is \cite{PPS15}. \\

We work on hyperbolic manifolds, that is, a Riemannian manifold $M$ that may be written as $M = \mathbb{H}^d/\Gamma$, where $\mathbb{H}^d$ is the hyperbolic space of dimension $d$, and where $\Gamma$ is a (non-elementary, discrete, without torsion, orientation preserving) group of isometries of $\mathbb{H}^d$. The geodesic flow $\phi = (\phi_t)_{t \in \mathbb{R}}$ acts on the unit tangent bundle of $M$, denoted by $T^1 M$. We say that a point $v \in T^1 M$ is wandering for the flow if there exists an open neighborhood $U \subset T^1 M$ of $v$, and a positive number $T>0$ such that:
$$ \forall t >T, \ \phi_t(U) \cap U = \emptyset .$$
The set of non-wandering points for $\phi$, denoted by $NW(\phi) \subset T^1 M$, is typically \say{fractal} and is invariant by the geodesic flow. We will work under the hypothesis that the group $\Gamma$ is convex-cocompact, which exactly means that $\text{NW}(\phi)$ is supposed compact. In particular, \emph{the case where $M$ is itself compact is authorized}. Under this condition, the flow $\phi$ restricted to $NW(\phi)$ is Axiom A. \\

In this context, for any choice of Hölder regular potential $F:T^1 M \rightarrow \mathbb{R}$, and for any probability measure $m \in \mathcal{P}(T^1 M)$ (the set of borel probability measures on $T^1 M$) invariant by the geodesic flow, one can consider the \emph{metric pressure} associated to $m$, defined by:
$$ P_{\Gamma,F}(m) := h_m(\phi) + \int_{\text{NW}(\phi)} F dm ,$$
where $h_m(\phi)$ denotes the entropy of the time-1 map of the geodesic flow with respect to the measure $m$. Notice that any probability measure invariant by the geodesic flow must have support included in the non-wandering set of $\phi$.
The \emph{topological pressure} is then defined by
$$ P(\Gamma,F) := \sup_m P_{\Gamma,F}(m) ,$$
where the sup is taken over all the $\phi$-invariant probability measures $m$. Those quantities generalize the variational principle for the topological and metric entropy (that we recover when $F=0$). 
It is well known that this supremum is, in fact, a maximum: see for example \cite{BR75} or $\cite{PPS15}$.

\begin{theorem}
Let $\Gamma$ be convex-cocompact, $M := \mathbb{H}^d/\Gamma$, and $F:T^1 M \rightarrow \mathbb{R}$ be a Hölder regular potential. Then there exists a unique probability measure $m_F$ invariant by $\phi$ such that $P_{\Gamma,F}(m_F) = P({\Gamma,F})$. This measure is called the equilibrium state associated to $F$ and its support is the non-wandering set of the geodesic flow. When $F=0$, $m_F$ is the measure of maximal entropy.
\end{theorem}

Theorem 6.1 in \cite{PPS15} also gives us a description of equilibrium states. To explain it, recall that the \emph{Hopf coordinates} allows us to identify $T^1 \mathbb{H}^d$ with $\partial_\infty \mathbb{H}^{d} \times \partial_\infty \mathbb{H}^{d} \times \mathbb{R}$, where $\partial_\infty \mathbb{H}^d $ denotes the \emph{ideal boundary} of the hyperbolic space (diffeomorphic to a sphere in our context). The measure $m_F$ lift into a $\Gamma$-invariant measure $\tilde{m}_F$ on $T^1 \mathbb{H}^d$, which can then be studied in these coordinates. The interesting remark is that $\tilde{m}_F$ may be seen as a product measure, involving what we call $(\Gamma,F)$-Patterson-Sullivan densities, which are generalizations of the usual Patterson-Sullivan probability measures. More precisely, there exists $\mu_F$ and $\mu_{F}^\iota$, two Patterson-Sullivan densities supported on the ideal boundary $\partial_\infty \mathbb{H}^d$, such that one may write (in these Hopf coordinates):

$$ d\tilde{m}_F(\xi,\eta,t) = \frac{d\mu_F(\xi) \otimes d\mu_{F}^\iota(\eta) \otimes dt}{D_{F}(\xi,\eta)^2} ,$$
where $D_F$ is the \say{potential gap} (or gap map), that we will define later. More details on Patterson-Sullivan densities can be found in section 6.2 (which will be devoted to recalling various preliminary results). 
Since the Hopf coordinates are smooth on $\mathbb{H}^d$, we see that one may reduce Fourier decay for $m_F$ to proving Fourier decay for Patterson-Sullivan densities. This reduction is the content of section 6.4. Then, to prove Fourier decay for those measures, several possibilities exists. With our current techniques, this will only be achieved when $d=2$, so that Patterson-Sullivan densities are supported on the circle. \\

The first possibility would be to use the fact that, in this low dimensional context, there exists a coding of the dynamics of the group $\Gamma$ on the ideal boundary: see for exemple \cite{BS79} or \cite{AF91}. Using these, one should be able to get Fourier decay for Patterson-Sullivan densities by adapting the proof of Bougain and Dyatlov in \cite{BD17}, which is similar to the method used in Chapter 2, Chapter 3, and Chapter 4. The second possibility would be to adapt the argument found in Li's appendix \cite{LNP19} to prove that (some) Patterson-Sullivan densities are actually stationary measures with exponential moment (for a random walk on $\Gamma$). Since in dimension 2, isometries of $\mathbb{H}^2$ may be seen as elements of $\text{PSL}_2(\mathbb{R})$, one could then apply Li's work \cite{Li20} to get Fourier decay. This is the strategy that we choose to follow in section 6.3. Finally, let us stress the fact that we are only able to work under a regularity condition (R) (see definition 6.2.13) that ensure upper regularity for our measures of interest. We now state our main results.

\begin{theorem}[Compare Theorem 6.3.2]
Let $\Gamma$ be a (non-elementary, discrete, without torsion, orientation preserving) convex-cocompact group of isometries of $\mathbb{H}^d$, and let $F : T^1(\mathbb{H}^d/\Gamma) \rightarrow \mathbb{R}$ be a Hölder potential satisfying (R). There exists $\mu \in \mathcal{P}(\partial_\infty \mathbb{H}^d)$ , a $(\Gamma,F)$ Patterson-Sullivan density, such that there exists $\nu \in \mathcal{P}(\Gamma)$ with exponential moment so that $\mu$ is $\nu$-stationary and so that the support of $\nu$ generates $\Gamma$.
\end{theorem}

Theorem 6.1.2 is the main technical result of this Chapter. The strategy is inspired by the appendix of \cite{LNP19}, but in our setting, some additional difficulties appear since the potential may be non-zero. For example, the proof of Lemma A.12 in \cite{LNP19} fails to work in our context. Our main idea to replace this lemma is to do a carefull study of the action of $\Gamma$ on the sphere at infinity: we will be particularly interested in understanding its contractions properties. This is the content of section 6.2. The proof of Theorem 6.1.2 is in section 6.3. Once this main technical result is proved, one can directly use the work of Li \cite{Li20} and get:

\begin{corollary}[\cite{Li20}, Theorem 1.5]

Let $\Gamma$ be a (non-elementary, discrete, without torsion, orientation preserving) convex-cocompact group of isometries of $\mathbb{H}^2$, and let $F : T^1(\mathbb{H}^2/\Gamma) \rightarrow \mathbb{R}$ be a Hölder potential satisfying (R). Let $\mu \in \mathcal{P}(\partial_\infty \mathbb{H}^2)$ be any $(\Gamma,F)$ Patterson-Sullivan density. Let $\alpha \in (0,1)$. There exists $\rho_1,\rho_2>0$ such that the following hold. There exists $C \geq 1$ such that, for any $s \in \mathbb{R}^*$, for any $\alpha$-Hölder $\chi:\partial_\infty \mathbb{H}^2 \simeq \mathbb{S}^1 \rightarrow \mathbb{R} $ and for any $C^2$ function $\varphi : \partial_\infty \mathbb{H}^2 \rightarrow \mathbb{R}$ such that $\|\varphi\|_{C^2} + (\inf_K |\varphi'|)^{-1} + 
\|\chi\|_{C^\alpha} \leq |s|^{\rho_1}$, we have:

$$ \left|\int_{\partial_\infty \mathbb{H}^2} e^{i s \varphi} \chi d\mu \right| \leq \frac{C}{ |s|^{\rho_2}} .$$
\end{corollary}

Strictly speaking, Theorem 1.5 in \cite{Li20} only allows us to get Fourier decay for the one Patterson-Sullivan density given by Theorem 6.1.2. But the fact that all Patterson-Sullivan densities are absolutely continuous between each other, with a Holder regular Radon-Nikodym derivative, allows us to get a Fourier decay statement for every Patterson-Sullivan densities. \\

Then, using the previous Corollary 6.1.3 and using Hopf coordinates, we can conclude Fourier decay for equilibrium states on convex-cocompact hyperbolic surfaces. The proof is done in section 6.4.

\begin{theorem}[Compare Theorem 6.4.5]
Let $\Gamma$ be a (non-elementary, discrete, without torsion, orientation preserving) convex-cocompact group of isometries of $\mathbb{H}^2$, and let $F : T^1(\mathbb{H}^2/\Gamma) \rightarrow \mathbb{R}$ be a Hölder potential satisfying (R). Let $m_F$ be the associated equilibrium state. Let $\alpha \in (0,1)$. There exists $\rho_1,\rho_2>0$ and $C \geq 1$ such that the following holds.  For all $\chi : T^1\mathbb{H}^2 \rightarrow \mathbb{R}$ a $\alpha$-Hölder map with small enough support $K$,  for all $\zeta \in \mathbb{R}^3 \setminus \{0\}$ and for any $\varphi : T^1 \mathbb{H}^2 \rightarrow \mathbb{R}^3$, $C^2$ local chart containing the support of $\chi$, satisfying $$  \|\varphi\|_{C^2} + \sup_{x \in K} 
\big( \| ^t (d\varphi)_x^{-1} \| \big) + \|\chi\|_{C^\alpha} \leq |\zeta|^{\rho_1} ,$$ we have:
$$ \left|\int_{NW(\phi)} e^{i \zeta \cdot \varphi(v)} \chi(v) d m_F(v) \right| \leq \frac{C}{|\zeta|^{\rho_2}},$$
where $\zeta \cdot \zeta'$ and $|\zeta|$ denotes the euclidean scalar product and the euclidean norm on $\mathbb{R}^3$. In other word, the pushforward measure $\varphi_*(\chi d m_F) \in \mathcal{P}(\mathbb{R}^3)$ exhibit power Fourier decay. In fact, we find $\underline{\text{dim}}_F(m_F) >0$ and $\underline{\text{dim}}_F(NW(\phi)) > 0$.
\end{theorem}

\begin{remark}
We will see in section 6.4 that the argument to prove Theorem 6.1.4 from Corollary 6.1.3 is fairly general. In particular, if one is able to prove Fourier decay for $(\Gamma,F)$-Patterson-Sullivan densities in some higher dimensional context, this would prove Fourier decay for equilibrium states in higher dimensions. For example, \cite{LNP19} precisely proves Fourier decay for Patterson-Sullivan densities with the potential $F=0$ when $\Gamma < \text{PSL}(2,\mathbb{C})$ is a Zariski-dense Kleinian Schottky group. This yields power decay for the measure of maximal entropy on $M := \mathbb{H}^3/\Gamma$ in this context, and even better, this proves $\underline{\text{dim}}_F(NW(\phi)) >0$ (seen as a subset of a $5$-dimensional manifold). 
\end{remark}

\section{Preliminaries}

\subsection{Moebius transformations preserving the unit ball}

In this first paragraph we recall well known properties of Moebius transformations. useful references for the study of such maps are \cite{Be83}, \cite{Ra06} and \cite{BP92}. The group of all Moebius transformations of $\mathbb{R}^d \cup \{\infty\}$ is the group generated by inversion of spheres and reflexions. This group contains dilations and rotations. Denote by $\text{Mob}(B^b)$ the group of all Moebius transformations $\gamma$ such that $\gamma$ preserves the orientation of $\mathbb{R}^d$, and such that $\gamma(B^d)=B^d$, where $B^d$ denotes the open unit ball in $\mathbb{R}^d$. These maps also acts on the unit sphere $\mathbb{S}^{d-1}$. These transformations can be put in a \say{normal form} as follows.

\begin{lemma}[\cite{Ra06}, page 124]
Define, for $b \in B^d$, the associated \say{hyperbolic translation} by:
$$ \tau_b(x) = \frac{(1-|b|^2) x + (|x|^2 + 2 x \cdot b + 1) b}{|b|^2 |x|^2 + 2 x \cdot b + 1} .$$
Then $\tau_b \in \text{Mob}(B^d)$. Moreover, for every $\gamma \in \text{Mob}(B^d)$, $\tau_{\gamma(0)}^{-1} \gamma  \in SO(d,\mathbb{R})$. 
\end{lemma}

It follows that the distortions of any Moebius transformation $\gamma \in \text{Mob}(B^d)$ can be understood by studying the distortions of hyperbolic translations. The main idea is the following: if $\gamma(o)$ is close to the unit sphere, then $\gamma$ contracts strongly on a large part of the sphere. Let us state a quantitative statement:

\begin{lemma}[First contraction lemma]

Let $\gamma \in \text{Mob}(B^d)$. Suppose that $|\gamma(o)| \geq c_0 > 0$. Denote $x_\gamma^m := \gamma(o)/|\gamma(o)|$, and let $\varepsilon_\gamma := 1-|\gamma(o)|$. Then: 
\begin{enumerate}
\item There exists $c_1,c_2>0$ that only depends on $c_0$ such that
$$ \forall x \in \mathbb{S}^{d-1}, \ |x - x^m_\gamma| \geq c_1 \varepsilon_\gamma^2 \Longrightarrow |\gamma^{-1} x - \gamma^{-1} x^m_\gamma | \geq c_2 .$$
\item For all $c \in (0,1)$, there exists $C \geq 1$ and a set $A_\gamma \subset \mathbb{S}^{d-1}$ such that $\text{diam}(A_\gamma) \leq  C \varepsilon_\gamma$ and such that:
$$ \forall x \in \mathbb{S}^{d-1} \setminus A_\gamma, \ |\gamma(x)-x_\gamma^m| \leq c \varepsilon_\gamma .$$
\end{enumerate}
\end{lemma}

\begin{proof}
Let $\gamma \in \text{Mob}(B^{d})$. Since $\gamma = \tau_{\gamma(o)} \Omega$ for some $\Omega \in SO(n,\mathbb{R})$, we see that we may suppose $\gamma = \tau_b$ for some $b \in B^d$. Without loss of generality, we may even choose $b$ of the form $\beta e_d$, where $e_d$ is the d-th vector of the canonical basis of $\mathbb{R}^d$, and where $\beta = |\gamma(o)| \in [c_0,1[$. Denote by $\pi_d$ the projection on the $d-th$ coordinate. We find:
$$ \forall x \in \mathbb{S}^{d-1}, \ \pi_d \tau_b(x) = \frac{(1+\beta^2) x_d + 2\beta}{2 \beta x_d + (1+\beta^2)} =: \varphi(x_d).$$
The function $\varphi$ is continuous and increasing on $[-1,1]$, and fixes $\pm 1$. Computing its value at zero gives $ \varphi(0) = \frac{2 \beta}{1+\beta^2} \geq 1 - \varepsilon_\gamma^2 $,
which proves the first point. Computing its value at $-\beta$ gives $\varphi(-\beta) = \beta$, which (almost) proves the second point. The second point is proved rigorously by a direct computation, noticing that
$$ 1-\varphi(-1+C \varepsilon_\gamma) = \frac{\varepsilon_\gamma}{C} \frac{1-C \varepsilon_\gamma/2}{1- (1-1/(2C))\varepsilon_\gamma} \leq \varepsilon_\gamma/C .$$
\end{proof}
Finally, let us recall a well known way to see $\text{Mob}(B^d)$ as a group of matrices.
\begin{lemma}
Let $q : \mathbb{R}^{d+1} \rightarrow  \mathbb{R}^{d+1}$ be the quadratic form $q(t,\omega) := -t^2 + \sum_i \omega_i^2$  on $\mathbb{R}^{d+1}$. We denote by $SO(d,1)$ the set of linear maps with determinant one that preserves $q$. Let $H := \{ (t,\omega) \in \mathbb{R} \times \mathbb{R}^d \ , \ q(t,\omega)=-1 \ , t > 0\}.$ Define the stereographic projection $\zeta : B^d \rightarrow H$ by $\zeta(x) := \left( \frac{1+|x|^2}{1-|x|^2}, \frac{2x}{1-|x|^2} \right)$. Then, for any $\gamma \in \text{Mob}(B^d)$, the map $\zeta \gamma \zeta^{-1} : H \rightarrow H$ is the restriction of an element of $SO(d,1)$ to $H$.
\end{lemma}

\begin{proof}

It suffices to check the lemma when $\gamma$ is a rotation or a hyperbolic translation. A direct computation shows that, when $\Omega \in SO(d,\mathbb{R})$, then $\zeta \Omega \zeta^{-1}$ is a rotation leaving invariant the $t$ coordinate, and so it is trivially an element of $SO(d,1)$. We now do the case where $\gamma = \tau_{\beta e_d}$ is a hyperbolic translation. We denote by $x$ the variable in $B^d$ and $(t,\omega)$ the variables in $\mathbb{R}^{d+1}$. The expression $\zeta(x) = (t,\omega)$ gives
$$ \frac{1+|x|^2}{1-|x|^2} = t, \quad \frac{2x}{1-|x|^2} = \omega, \ \text{and} \quad \zeta^{-1}(t,\omega) = \frac{\omega}{1+t} .$$
For $\alpha \in \mathbb{R}$, denote $s_\alpha := \sinh(\alpha)$ and $c_\alpha := \cosh(\alpha)$. There exists $\alpha$ such that $\beta = s_\alpha /(c_\alpha+1) = (c_\alpha-1)/s_\alpha$. For this $\alpha$, we also have $\beta^2 = (c_\alpha-1)/(c_\alpha+1)$, and $1-\beta^2 = 2/(c_\alpha+1)$. Now, we see that
$$ \tau_{\beta e_d}(x) = \frac{(1-\beta^2) x + (|x|^2 + 2 x_d \beta + 1)  \beta e_d}{\beta^2 |x|^2 + 2 x_d \beta + 1}  = \frac{\frac{2}{c_\alpha+1} x + (|x|^2 + 2 x_d \frac{c_\alpha-1}{s_\alpha} + 1) \frac{s_\alpha}{c_\alpha+1} e_d}{\frac{c_\alpha-1}{c_\alpha+1} |x|^2 + 2 x_d \frac{s_\alpha}{c_\alpha+1}  + 1} $$
$$ = \frac{2x + (s_\alpha(1+|x|^2) + 2x_d(c_\alpha-1) ) e_d }{1-|x|^2+ c_\alpha(1+|x|^2) + 2 s_\alpha x_d} = \frac{\omega + \left(s_\alpha t + (c_\alpha -1) \omega_d \right)e_d}{1+ c_\alpha t + s_\alpha \omega_d} $$
$$ = \zeta^{-1}\left( c_\alpha t + s_\alpha \omega_d , \omega + \left( s_\alpha t + (c_\alpha-1)\omega_d \right)e_d \right) ,$$
and so $\zeta \tau_{\beta e_d} \zeta^{-1}(t,\omega) = \left( c_\alpha t + s_\alpha \omega_d \ , \ \omega + \left( s_\alpha t + (c_\alpha-1)\omega_d \right)e_d \right)  $ is indeed linear in $(t,\omega)$. In this form, checking that $\zeta \tau_{\beta e_d} \zeta^{-1} \in SO(d,1)$ is immediate.
\end{proof}

\begin{remark}
From now on, we will allow to directly identify elements of $\text{Mob}(B^d)$ with matrices in $SO(d,1)$. (By continuity of $\gamma \mapsto \zeta \gamma \zeta^{-1}$, we even know that thoses matrices lies in $SO_0(d,1)$, the connected component of the identity in $SO(d,1)$). 
It follows from the previous explicit computations that, for any matrix norm on $SO(d,1)$, and for any $\gamma \in \text{Mob}(B^d)$ such that $|\gamma(o)| \geq c_0$, there exists $C_0$ only depending on $c_0$ such that  $$ \|\gamma\| \leq C_0 \varepsilon_\gamma^{-1}.$$
\end{remark}

\subsection{The Gibbs cocycle}

In this paragraph, we introduce our geometric setting. For an introduction to geometry in negative (non-constant) curvature, the interested reader may refer to \cite{BGS85}, or to the first chapters of \cite{PPS15}. \\

Let $M = \mathbb{H}^d/\Gamma$ be a hyperbolic manifold of dimension $d$, where $\Gamma \subset Iso^+(\mathbb{H}^d)$ denotes a non-elementary and discrete group of isometries of the hyperbolic space without torsion and that preserves the orientation. Let $T^1 M$ denote the unit tangent bundle of $M$, and denote by $p:T^1 M \rightarrow M$ the usual projection. The projection lift to a $\Gamma$-invariant map $\tilde{p}:T^1 \mathbb{H}^d \rightarrow \mathbb{H}^d$. Fix $F : T^1 M \rightarrow \mathbb{R}$ a Hölder map: we will call it a potential. The potential $F$ lift to a $\Gamma$-invariant map $\tilde{F} : T^1 \mathbb{H}^d \rightarrow \mathbb{R}$. All future constants appearing in this Chapter will implicitely depend on $\Gamma$ and $F$. \\

To be able to use our previous results about Moebius transformations, we will work in the conformal ball model for a bit. In this model, we can think of $\mathbb{H}^d$ as being the unit ball $B^d$ equipped with the metric $ds^2 := \frac{4\|dx\|^2}{(1-\|x\|^2)^2}$. The ideal boundary $\partial_\infty \mathbb{H}^d$  (see \cite{BGS85} for a definition) of $\mathbb{H}^d$ is then naturally identified with $\mathbb{S}^{d-1}$, and its group of orientation-preserving isometries with $\text{Mob}(B^d)$. On the ideal boundary, there is a natural family of distances $(d_x)_{x \in \mathbb{H}^d}$ called visual distances (seen from $x$), defined as follow:
$$ d_x(\xi,\eta) := \lim_{t \rightarrow +\infty}  \exp\left( -\frac{1}{2} \left(  d(x,\xi_t) + d(x,\eta_t) - d(\xi_t,\eta_t) \right) \right) \in [0,1], $$
where $\xi_t$ and $\eta_t$ are any geodesic rays ending at $\xi$ and $\eta$. To get an intuition behind this quantity, picture a finite tree with root $x$ and think of $\xi$ and $\eta$ as leaves in this tree. 

\begin{lemma}[\cite{PPS15} page 15 and \cite{LNP19} lemma A.5]
The visual distances are all equivalent and induces the usual euclidean topology on $\mathbb{S}^{d-1} \simeq \partial_\infty \mathbb{H}^d$. More precisely:
$$ \forall x,y \in \mathbb{H}^d, \forall \xi,\eta \in \partial_\infty \mathbb{H}^d, \  \ e^{-d(x,y)} \leq \frac{d_x(\xi,\eta)}{d_y(\xi,\eta)} \leq e^{d(x,y)} .$$
In the ball model, the visual distance from the center of the ball is the sine of (half of) the angle.
\end{lemma}

The sphere at infinity $\partial_\infty \mathbb{H}^d$ takes an important role in the study of $\Gamma$. 
Since $\Gamma$ is supposed non-elementary, for any $x \in \mathbb{H}^d$, the orbit $\Gamma x$ accumulates on $\mathbb{S}^{d-1}$ (for the euclidean topology) into a (fractal) \emph{limit set} denoted $\Lambda_\Gamma$. This limit set is independent of $x$. We will denote by $\text{Hull}(\Lambda_\Gamma)$ the convex hull of the limit set: that is, the set of points $x \in \mathbb{H}^d$ such that $x$ is in a geodesic starting and ending in $\Lambda_\Gamma$. Since $\Gamma$ acts naturally on $\Lambda_\Gamma$, $\Gamma$ acts on $\text{Hull}(\Lambda_\Gamma)$. Without loss of generality, we can assume that $o \in \text{Hull}( \Lambda_\Gamma)$, and we will do so from now on. We will say that $\Gamma$ is convex-cocompact if $\Gamma$ is discrete, without torsion and if $\text{Hull}( \Lambda_\Gamma) / \Gamma$ is compact. In particular, in this Chapter, we allow $M$ to be compact. \\

We will suppose throughout the Chapter that $\Gamma$ is convex-cocompact. In this context, the set $\Omega \Gamma = p^{-1}( \text{Hull} \Lambda_\Gamma / \Gamma ) \subset T^1 M$ is compact, and it follows that $\sup_{\Omega \Gamma} |F| < \infty$. In particular, $\tilde{F}$ is bounded on $\tilde{p}^{-1}( \text{Hull}(\Lambda_\Gamma))$, which is going to allow us to get some control over line integrals involving $F$. Recall the notion of line integral in this context:
if $x,y \in \mathbb{H}^d$ are distinct points, then there exists a unique unit speed geodesic joining $x$ to $y$, call it $c_{x,y}$. We then define:
$$ \int_x^y \tilde{F} := \int_0^{d(x,y)} \tilde{F}(\dot{c}_{x,y}(s)) ds .$$
Beware that if $\tilde{F}(-v) \neq \tilde{F}(v)$ for some $v \in T^1 M$, then $\int_x^y \tilde{F}$ and $\int_y^x \tilde{F}$ might not be equal. \\
We are ready to introduce the \emph{Gibbs cocycle} and recall some of its properties.

\begin{definition}[\cite{PPS15}, page 39]
The following \say{Gibbs cocycle} $C_F : \partial_\infty \mathbb{H}^d \times \mathbb{H}^d \times \mathbb{H}^d \rightarrow \mathbb{R}$ is well defined and continuous:
$$ C_{F,\xi}(x,y) := \underset{t \rightarrow + \infty}{\lim} \left( \int_y^{\xi_t} \tilde{F} - \int_x^{\xi_t} \tilde{F}  \right), $$
where $\xi_t$ denotes any geodesic converging to $\xi$.
\end{definition}

\begin{remark}
Notice that if $\xi$ is the endpoint of the ray joining $x$ to $y$, then $$C_{F,\xi}(x,y) = - \int_x^y \tilde{F}.$$
\end{remark}
For $A \subset \mathbb{H}^d$, we call \say{shadow of $A$ seen from $x$} the set $\mathcal{O}_x A$ of all $\xi \in \partial_\infty \mathbb{H}^d$ such that the geodesic joining $x$ to $\xi$ intersects $A$. (The letter $\mathcal{O}$ stands for \say{Ombre} in French.)

\begin{proposition}[\cite{PPS15}, Proposition 3.4 and 3.5]
We have the following estimates on the Gibbs cocycle.

\begin{enumerate}
\item For all $R>0$, there exists $C_0>0$ such that for all $\gamma \in \Gamma$ and for all $\xi \in \mathcal{O}_o B(\gamma(o),R)$ in the shadow of the (hyperbolic) ball $B(\gamma(o),R)$ seen from $o$, we have:
$$ \left| C_{F,\xi}(o,\gamma(o)) + \int_o^{\gamma(o)} \tilde{F} \right| \leq C_0 .$$
\item 
There exists $\alpha \in (0,1)$ and $C_0>0$ such that, for all $\gamma \in \Gamma$ and for all $\xi,\eta \in \Lambda_\Gamma$ such that $d_o(\xi,\eta) \leq e^{-d(o,\gamma(o))-2}$,
$$ \left| C_{F,\xi}(o,\gamma(o)) - C_{F,\eta}(o,\gamma(o)) \right| \leq C_0 e^{ \alpha d(o,\gamma(o))} d_o(\xi,\eta)^{\alpha}. $$
In fact, for every $x,y \in \mathbb{H}^d$, the map $\xi \in \Lambda_\Gamma \mapsto C_{F,\xi}(x,y) \in \mathbb{R}$ is Hölder regular (but not uniformly in $x,y$).
\end{enumerate}

\end{proposition}

\subsection{Patterson-Sullivan densities}

In this paragraph, we recall the definition of $(\Gamma,F)$ Patterson-Sullivan densities, and we introduce a regularity condition. To begin with, recall the definition and some properties of the critical exponent of $(\Gamma,F)$.

\begin{definition}[\cite{PPS15}, Lemma 3.3]
Recall that $F$ is supposed Hölder, and that $\Gamma$ is convex-cocompact. The critical exponent of $(\Gamma,F)$ is the quantity $\delta_{\Gamma,F} \in \mathbb{R}$ defined by:
$$ \delta_{\Gamma,F} := \underset{n \rightarrow \infty}{\limsup} \  \underset{n-c < d(x,\gamma y) \leq n}{\frac{1}{n} \ln \ {\sum_{\gamma \in \Gamma}} \ e^{\int_x^{\gamma y} \tilde{F} }} ,$$
for any $x,y \in \mathbb{H}^d$ and any $c>1$. The critical exponent dosen't depend on the choice of $x,y$ and $c$.
\end{definition}

\begin{theorem}[\cite{PPS15}, section 3.6 and section 5.3]

Let $\Gamma \subset Iso^+(\mathbb{H}^d)$ be convex-cocompact, and note $M := \mathbb{H}^d / \Gamma$.
Let $F : T^1 M \rightarrow \mathbb{R}$ be a Hölder regular potential. Then there exists a unique (up to a scalar multiple) family of finite nonzero measures $(\mu_x)_{x \in \mathbb{H}^d}$ on $\partial_\infty \mathbb{H}^d$ such that, for all $\gamma \in \Gamma$, for all $x,y \in \mathbb{H}^d$ and for all $\xi \in \partial_\infty \mathbb{H}^d$:
\begin{itemize}
\item $\gamma_* \mu_x = \mu_{\gamma x}$
\item $d \mu_x(\xi) = e^{-C_{F-\delta_{\Gamma,F},\xi}(x,y)} d\mu_y(\xi)$
\end{itemize}
Moreover, these measures are all supported on the limit set $\Lambda_\Gamma$. We call them $(\Gamma,F)$-Patterson Sullivan densities.
\end{theorem}

\begin{remark}
Notice that Patterson-Sullivan densities only depend on the normalized potential $F-\delta_{\Gamma,F}$. Since $\delta_{\Gamma,F + \kappa} = \delta_{\Gamma,F} + \kappa $, replacing $F$ by $F-\delta_{\Gamma,F}$ allows us to work without loss of generality with potentials satisfying $\delta_{\Gamma,F}=0$. We call such potentials \emph{normalized}.
\end{remark}

The next estimate tells us, in a sense, that we can think of $\mu_o$ as a measure of a \say{fractal solid angle}, pondered by the potential. This is better understood by recalling that since the area of a hyperbolic sphere of large radius $r$ is a power of $\sim e^r$, then the solid angle of an object of diameter $1$ lying in that sphere is a power of $\sim e^{-r}$. In the following \say{Shadow lemma}, the object is a ball $B(y,R)$, at distance $d(x,y)$ from an observer at $x$.

\begin{proposition}[Shadow Lemma, \cite{PPS15} Lemma 3.10]
Let $R>0$ be large enough. There exists $C>0$ such that, for all $x,y \in \text{Hull}(\Lambda_\Gamma)$:

$$ C^{-1} e^{\int_x^{y} (\tilde{F} - \delta_{\Gamma,F})} \leq \mu_x\left( \mathcal{O}_x B(y,R) \right) \leq C e^{\int_x^{y} (\tilde{F}-\delta_{\Gamma,F})} .$$
\end{proposition}

\begin{definition}
The shadow lemma calls for the following hypothesis: we say that the potential $F$ satisfy the regularity assumptions (R) if $F$ is Hölder regular and if $\sup_{\Omega \Gamma} F < \delta_{\Gamma,F}$. 
\end{definition}

 \begin{remark}
By Lemma 3.3 in \cite{PPS15}, we see that we can construct potentials satisfying (R) as follow: choose some potential $F_0$ satisfying (R) (for example, the constant potential) and then choose any Hölder map $E : T^1 M \rightarrow \mathbb{R}$ satisfying $2 \sup_{\Omega \Gamma} |E| < \delta_{\Gamma,F} - \sup_{\Omega {\Gamma}} F$. Then $F := F_0 + E$ satisfies the assumption (R). A similar assumption is introduced in \cite{GS14}.
 \end{remark}

The point of the assumption (R) is to ensure that Patterson-Sullivan densities exhibit some regularity. This is possible because we have a tight control over the geometry of shadows.

\begin{lemma}
Let $\Gamma \subset Iso^+(\mathbb{H}^d)$ be convex-cocompact, and let $F : T^1 ({\mathbb{H}^d/\Gamma}) \rightarrow \mathbb{R}$ satisfy the regularity assumptions (R). Let $\delta_{reg} \in (0,1)$ such that $\delta_{reg} < \delta_{\Gamma,F} - \sup_{\Omega \Gamma} F$. Let $\mu$ denote some $(\Gamma,F)$-Patterson-Sullivan density. Then
$$ \exists C>0, \ \forall \xi \in \partial_\infty \mathbb{H}^{d}, \ \forall r>0, \  \mu(B(\xi,r)) \leq C r^{\delta_{reg}} ,$$
where the ball is in the sense of some visual distance.
\end{lemma}

\begin{proof}

First of all, since for all $p,q \in \mathbb{H}^d$, $\xi \in \partial_\infty \mathbb{H}^d \mapsto e^{C_{\xi}(p,q)}$ is continuous and since the ideal boundary is compact, we can easily reduce our statement to the case where $\mu$ is a Patterson-Sullivan density based on the center of the ball $o$. Moreover, since all the visual distances are equivalent, one can suppose that we are working for the visual distance based at $o$ too. Finally, since the support of $\mu_o$ is $\Lambda_\Gamma$, we can suppose without loss of generality that $B(\xi,r) \cap \Lambda_\Gamma \neq \emptyset$. Since in this case there exists some $\tilde{\xi} \in \Lambda_\Gamma$ such that $B(\xi,r) \subset B(\tilde{\xi},2r)$, we may further suppose without loss of generality that $\xi \in \Lambda_\Gamma$. \\

Now fix $\xi \in \Lambda_\Gamma$ and $r>0$. Let $x \in \text{Hull}(\Lambda_\Gamma)$ lay in the ray starting from $o$ and ending at $\xi$. Let $\rho \in [0,1]$, let $y \in \mathbb{H}^d$ such that $[o,y]$ is tangent to the sphere $S(x,\rho)$, and note $\eta$ the ending of the ray starting from $o$ and going through $y$. The hyperbolic law of sine (see Lemma A.4 and A.5 in \cite{LNP19}) allows us to compute directly:
$$ d_o(\xi,\eta) = \frac{1}{2} \cdot \frac{\sinh(\rho)}{\sinh(d(o,x))} .$$
It follows that there exists $C>0$ such that for all $\xi \in \Lambda_\Gamma$, for all $r>0$, there exists $x \in \text{Hull}(\Lambda_\Gamma)$ such that $e^{-d(o,x)} \leq C r$ and $B_o(\xi,r) \subset \mathcal{O}_o B(x,1) $. 
The desired bound follows from the shadow lemma, since the geodesic segment joining $o$ and $x$ lays in $\text{Hull}(\Lambda_\Gamma)$. \end{proof}

The regularity of Patterson-Sullivan densities is going to allow us to state a second version of the contraction lemma. First, let us introduce a bit of notations. We fix, for all the duration of the Chapter, a large enough constant $C_{\Gamma}>0$. For $\gamma \in \Gamma$, we define $\kappa(\gamma) := d(o,\gamma o)$, $r_\gamma := e^{-\kappa(\gamma)}$ and $B_\gamma := \mathcal{O}_o B(\gamma o,C_\Gamma)$. By the hyperbolic law of sine, the radius of $B_\gamma$ is $\sinh(C_\Gamma)/\sinh(r_\gamma)$ (Lemma A.5 in \cite{LNP19}.) If $C_\Gamma$ is chosen large enough and when $\kappa(\gamma)$ is large, we get a radius of $\sim e^{C_\Gamma} r_\gamma \geq r_\gamma$. We have the following covering result.
\begin{lemma}[\cite{LNP19}, Lemma A.8]
Define $r_n := e^{-4 C_\Gamma n}$, and let $ S_n := \{ \gamma \in \Gamma \ , e^{-2 C_\Gamma} r_n \leq r_\gamma < r_n \}$. For all $n \geq 1$, the family $\{ B_\gamma \}_{\gamma \in S_n}$ covers $\Lambda_\Gamma$. Moreover, there exists $C>0$ such that:
$$ \forall n, \forall \xi \in \Lambda_\Gamma, \ \#\{ \gamma \in S_n \ , \ \xi \in B_\gamma \} \leq C .$$
\end{lemma}

Now, we are ready to state our second contraction lemma. Since the potential is not supposed bounded, a lot of technical bounds will only be achieved by working on the limit set or on its convex hull. One of the main goal of the second contraction lemma is then to replace $x_\gamma^m$ by a point $\eta_\gamma$ lying in the limit set.

\begin{lemma}[Second contraction lemma]
Let $\Gamma \subset \text{Iso}^+(\mathbb{H}^d)$ be convex-cocompact. Let $F : T^1( \mathbb{H}^d/\Gamma ) \rightarrow \mathbb{R}$ be a potential satisfying (R). Denote by $\mu_o \in \mathcal{P}(\Lambda_\Gamma)$ the associated Patterson-Sullivan density at $o$. Then there exists a family of points $(\eta_\gamma)_{\gamma \in \Gamma}$ such that, for any $\gamma \in \Gamma$ with large enough $\kappa(\gamma)$, we have $\eta_\gamma \in \Lambda_\Gamma \cap B_\gamma$ (in fact $d_o(\eta_\gamma,x_\gamma^m) \lesssim r_\gamma^2$), and moreover:

\begin{enumerate}
    \item there exists $c>0$ independent of $\gamma$ such that $d_o(\xi,\eta_\gamma) \geq r_\gamma/2 \Rightarrow d_o(\gamma^{-1} \xi, \gamma^{-1} \eta_\gamma) \geq c$,
    \item for all $\varepsilon_0 \in (0,\delta_{reg})$, there exists $C$ independent of $\gamma$ such that:
    $$ \int_{\Lambda_\Gamma} d_o(\gamma(\xi),\eta_\gamma)^{\varepsilon_0} d\mu_o(\xi) \leq C r_\gamma^{ \varepsilon_0 } .$$
\end{enumerate}

\end{lemma}

\begin{proof}
Recalling that the visual distance and the euclidean distance are equivalent on the unit sphere, if we forget about $\eta_\gamma$ and replace it by $x_\gamma^m$ instead, then the first point is a direct corollary of the first contraction Lemma 6.2.2. We just have to check two points: first, since $\Gamma$ is discrete without torsion, there exists $c_0 >0$ such that for all $\gamma \in \Gamma \setminus \{Id\}$, $d(o,\gamma o) > c_0$. The second point is to check that the orders of magnitude of $r_\gamma$ and $\varepsilon_\gamma$ (quantity introduced in the first contraction lemma) are compatible. This can be checked using an explicit formula relating the hyperbolic distance with the euclidean one in the ball model: $r_\gamma = e^{-\kappa(\gamma)} = \frac{1-|\gamma(o)|}{1+|\gamma(o)|} \sim \varepsilon_\gamma/2$ (see \cite{Ra06}, exercise 4.5.1). We even have a large security gap for the first statement to hold (recall that the critical scale is $\sim \varepsilon_\gamma^2$). \\

We will use the strong contraction properties of $\Gamma$ to construct a point $\eta_\gamma \in \Lambda_\Gamma$ very close to $x_\gamma^m$. Since $\Gamma$ is convex-cocompact, we know in particular that $\text{diam}(\Lambda_\Gamma)>0$. Now let $\gamma \in \Gamma$ such that $\kappa(\gamma)$ is large enough. The first contraction lemma says that there exists $A_\gamma \subset \partial_\infty \mathbb{H}^d$ with $\text{diam} (A_\gamma) \leq C r_\gamma$ such that $\text{diam}(\gamma(A_\gamma^c)) \leq r_\gamma/10$. It follows that we can find a point $\widehat{\eta}_\gamma \in \Lambda_\Gamma$ such that $d_o(A_\gamma, \widehat{\eta}_\gamma)> \text{diam}(\Lambda_\Gamma)/3$. Fixing $\eta_\gamma := \gamma(\widehat{\eta}_\gamma)$ gives us a point satisfying $\eta_\gamma \in \Lambda_\Gamma$ and $d_o(\eta_\gamma,x_\gamma^m) \lesssim r_\gamma^{2}$. Hence $\eta_\gamma \in B_\gamma$, and moreover, any point $\xi$ satisfying $d_o(\xi,\eta_\gamma) \geq r_\gamma/2$ will satisfy $\gamma^{-1}(\xi) \in A_\gamma$, and so $d_o(\gamma^{-1}(\xi),\gamma^{-1}(\eta_\gamma)) \geq \text{diam}(\Lambda_\Gamma)/3$. This proves the first point. \\

For the second point: since the set $A_\gamma \subset \partial_\infty \mathbb{H}^{d}$ is of diameter $\leq C r_\gamma$ and satisfy that, for all $\xi \notin A_\gamma$, $d_o(\gamma(\xi),x_\gamma^m) \leq r_\gamma/10$, the upper regularity of $\mu_o$ yields
$$ \int_{\Lambda_{\Gamma}} d_o(\gamma(\xi),x_\gamma^m)^{\varepsilon_0} d\mu_o(\xi) \leq C \mu_o(A_\gamma) + \int_{A_\gamma^c} C r_\gamma^{\varepsilon_0} d\mu_o \leq C r_\gamma^{\varepsilon_0} .$$
The desired bound follows from $d_o(x_\gamma^m,\eta_\gamma) \lesssim r_\gamma$, using the triangle inequality.\end{proof}

\section{Patterson-Sullivan densities are stationary measures}

\subsection{Stationary measures}

In this subsection we define stationary measures and state our main theorem. 
\begin{definition}
Let $\nu \in \mathcal{P}(\Gamma)$ be a probability measure on $\Gamma \subset SO(n,1)$. Let $\mu \in \mathcal{P}(\partial_\infty \mathbb{H}^d)$. We say that $\mu$ is $\nu$-stationary if:
$$ \mu = \nu * \mu := \int_\Gamma \gamma_* \mu \ d\nu(\gamma) .$$
Moreover, we say that the measure $\nu$ has exponential moment if there exists $\varepsilon>0$ such that $ \int_\Gamma \|\gamma\|^{\varepsilon} d\nu(\gamma) < \infty$. Finally, we denote by $\Gamma_\nu$ the subgroup of $\Gamma$ generated by the support of $\nu$.
\end{definition}

\begin{theorem}
Let $\Gamma \subset \text{Iso}^+(\mathbb{H}^d)$ be a (non-elementary, discrete, without torsion, orientation preserving) convex-cocompact group, and let $F : T^1(\mathbb{H}^d/\Gamma) \rightarrow \mathbb{R}$ be a potential on the unit tangent bundle satisfying (R). Let $x \in \text{Hull}(\Lambda_\Gamma)$ and let $\mu_x \in \mathcal{P}(\Lambda_\Gamma)$ denotes the $(\Gamma,F)$ Patterson-Sullivan density from $x$. Then there exists $\nu \in \mathcal{P}(\Gamma)$ with exponential moment (seen as a random walk in $SO(d,1)$) such that $\mu_x$ is $\nu$-stationary and such that $\Gamma_\nu = \Gamma$.
\end{theorem}

\begin{remark}
This result for $d=2$ was announced without proof by Jialun Li in \cite{Li17} (see remark 1.9). A proof in the case of constant potentials is done in the appendix of \cite{LNP19}. Our strategy is inspired by this appendix. For more details on stationary measures, see the references therein.
\end{remark}

First of all, a direct computation allows us to see that if $(\mu_x)_{x \in \mathbb{H}^d}$ are $(\Gamma,F)$ Patterson-Sullivan densities, then, for any $\eta \in SO_0(d,1)$ , $ (\eta_* \mu_{\eta^{-1} x})_{x \in \mathbb{H}^d} $ are $(\eta \Gamma \eta^{-1}, \eta_* F)$ Patterson-Sullivan densities. This remark allows us to reduce our theorem to the case where the basepoint $x$ is the center of the ball $o$, and $o \in \text{Hull}(\Lambda_\Gamma)$. Our goal is to find $\nu \in \mathcal{P}(\Gamma)$ such that $\nu * \mu_o = \mu_o$.  Assuming that $F$ is normalized, this can be rewritten as follows:
$$ d\mu_o(\xi) = \sum_\gamma \nu(\gamma) d(\gamma_* \mu_o)(\xi) =  \sum_\gamma \nu(\gamma) e^{C_{F,\xi}(o,\gamma o)} d\mu_o(\xi) .$$
Hence, $\mu_o$ is $\nu$-stationary if
$$ \sum_{\gamma \in \Gamma} \nu(\gamma) f_\gamma = 1 \ \text{ on } \Lambda_\Gamma,$$
where $$f_\gamma(\xi) := e^{C_{F,\xi}(o,\gamma o)}.$$ 

\begin{remark}
Our main goal is to find a way to decompose the constant function $1$ as a sum of $f_\gamma$. Here is the intuition behind our proof. \\

Define $r_\gamma^{-F} := e^{-\int_o^{\gamma o} \tilde{F}} = f_\gamma(x^m_\gamma) \simeq f_\gamma(\eta_\gamma)$. The first thing to notice is that $f_\gamma$ looks like an approximation of unity centered at $x^m_\gamma$. Renormalizing yields the intuitive statement $r_\gamma^{F} f_\gamma \sim \mathbb{1}_{B_\gamma}$. The idea is that this approximation gets better as $\kappa(\gamma)$ becomes large. Once this observation is done, there is a natural \say{n-th approximation} operator that can be defined. For some positive function $R$, one can write:
$$ R \simeq \sum_{\gamma \in S_n} R(\eta_\gamma) \mathbb{1}_{B_\gamma} \simeq \sum_{\gamma \in S_n} R(\eta_\gamma) r_\gamma^{F} f_\gamma =: P_n R.$$
Proving that the operator $P_n$ does a good enough job at approximating some functions $R$ is the content of the \say{approximation lemma} 6.3.8. In particular, it is proved that, under some assumptions on $R>0$, we have $c R \leq P_n R \leq C R$. \\

The conclusion of the proof is then easy. We fix a constant $\beta>0$ small enough so that $c R \leq  \beta P_n R \leq R$. Then, we define by induction $R_0 = 1$ and $0 < R_{n+1} := R_n - \beta P_{n+1} R_n \leq (1-c) R_n$. By induction, this gives $R_n \leq (1-c)^n$, and hence $1 = R_0 - \lim_n R_n = \sum_{k} (R_{k}-R_{k+1})$ is a decomposition of $1$ as a sum of $f_\gamma$.
\end{remark}

\subsection{The approximation operator}

First, we collect some results on $f_\gamma$ that will allows us to think of it as an approximation of unity around $x_m^\gamma$ with width $r_\gamma$. The first point studies $f_\gamma$ near $x_\gamma^m$, the second point study the decay of $f_\gamma$ away from it, and the last point is a regularity estimate at the scale $r_\gamma$. To quantify this decay, we recall the notion of \emph{potential gap} (or gap map).

\begin{definition}
The following \say{potential gap} $D_F : \mathbb{H}^d \times \partial_\infty \mathbb{H}^d \times \partial_\infty \mathbb{H}^d$ is well defined and continuous:
$$D_{F,x}(\eta,\xi) := \exp \frac{1}{2} \lim_{t \rightarrow \infty} \left( \int_x^{\xi_t} \tilde{F} + \int_{\eta_t}^x \tilde{F} - \int_{\eta_t}^{\xi_t} \tilde{F} \right),$$
where $(\eta_t)$ and $(\xi_t)$ denotes any unit speed geodesic ray converging to $\eta$ and $\xi$.
\end{definition}

Under our assumptions, the gap map (for some fixed $x$) behaves like a distance on $\Lambda_\Gamma$: see \cite{PPS15}, section 3.4 for details. Finally, we denote by $\iota(v) = -v$ the \emph{flip map} on the unit tangent bundle. 

\begin{lemma}[Properties of $f_\gamma$]
There exists $C_{\text{reg}} \geq 2$ (independent of $C_\Gamma$) and $C_0 \geq 1$ (depending on $C_\Gamma)$ such that, for all $\gamma \in \Gamma$:
\begin{enumerate}
\item For all $\xi \in B_\gamma$, $$ r_\gamma^{F} f_\gamma(\xi) \in [e^{-C_0},e^{C_0}] .$$
 \item For all $\xi \in \Lambda_\Gamma$ such that $d_o(\xi,\eta_\gamma) \geq r_\gamma/2$,
$$ C_0^{-1}  D_{F,o}(\xi,\eta_\gamma)^{-2}  \leq r_\gamma^{-F \circ \iota} f_\gamma(\xi) \leq C_0 D_{F,o}(\xi,\eta_\gamma)^{-2} .$$
 \item For all $\xi,\eta \in \Lambda_\Gamma$ such that $d_o(\xi,\eta) \leq r_\gamma/e^2$,
 $$ \left| f_\gamma(\xi)/f_\gamma(\eta)-1 \right| \leq C_{\text{reg}} \cdot d_o(\xi,\eta)^\alpha r_\gamma^{-\alpha}. $$
 
\end{enumerate}

\end{lemma}

\begin{proof}
Recall that $f_\gamma(x_\gamma^m) = r_\gamma^{-F}$: the first point is then a consequence of Proposition 6.2.8. The third point also directly follows without difficulty. For the second item, recall that $\eta_\gamma \in B_\gamma$, so that by the same argument we get $$ r_\gamma^{-F \circ \iota} \simeq e^{C_{F \circ \iota,\eta_{\gamma}}(o,\gamma o)}. $$
Then, a direct computation yields:
$$ f_\gamma(\xi) r_\gamma^{-F \circ \iota} \simeq \exp \lim_{t \rightarrow \infty} \left( \left( \int_{\gamma o}^{\xi_t} \tilde{F} - \int_o^{\xi_t} \tilde{F}\right) - \left( \int_{(\eta_\gamma)_t}^{o} \tilde{F} - \int_{(\eta_\gamma)_t}^{\gamma o} \tilde{F} \right) \right) $$
$$ = \lim_{t \rightarrow \infty} \exp {\left( \left( - \int_o^{\xi_t} \tilde{F} - \int_{(\eta_\gamma)_t}^o \tilde{F}  + \int_{(\eta_\gamma)_t}^{\xi_t} \tilde{F} \right) + \left( \int_{\gamma o}^{\xi_t} \tilde{F} + \int_{(\eta_\gamma)_t}^{\gamma o} \tilde{F} - \int_{(\eta_\gamma)_t}^{\xi_t} \tilde{F} \right) \right)} $$
$$ = \frac{D_{F,\gamma o}(\eta_\gamma,\xi)^2}{D_{F,o}(\eta_\gamma,\xi)^2} .$$
Under our regularity hypothesis (R), and because $\Gamma$ is convex-cocompact and $\xi,\eta \in \Lambda_\Gamma$, it is known that there exists $c_0>0$ such that $d_{\gamma o}(\xi,\eta_\gamma)^{c_0} \leq D_{F,\gamma o}(\eta_\gamma,\xi) \leq 1$ (see \cite{PPS15}, page 56). The second contraction lemma allows us to conclude, since $d_{\gamma o}(\xi,\eta_\gamma) = d_o(\gamma^{-1} \xi,\gamma^{-1}(\eta_\gamma)) \geq c$ under the hypothesis $d_o(\xi,\eta_\gamma) \geq r_\gamma/2$.
\end{proof}
To get further control over the decay rate of $f_\gamma$ away from $\eta_\gamma$, the following \say{role reversal} result will be helpful.
\begin{lemma}[symmetry]
Let $n \in \mathbb{N}$ and let $\eta \in \Lambda_\Gamma$. Since $\{B_\gamma\}_{\gamma \in S_n}$ is a covering of $\Lambda_\Gamma$, and by choosing $C_\Gamma$ larger if necessary, we know that there exists $\tilde{\gamma}_\eta \in S_{n+1}$ such that $\eta \in B_{\tilde{\gamma}_\eta}$ and $d(\eta,\eta_{\tilde{\gamma}_\eta}) \leq C_{reg}^{-2/\alpha} r_{\tilde{\gamma}_\eta}$. Now let $\gamma \in \Gamma$, and suppose that $d_o(\eta,\eta_\gamma) \geq r_\gamma$. Then:
$$ C^{-1} f_{\tilde{\gamma}_\eta}(\eta_\gamma)  \leq f_\gamma(\eta) \leq C f_{\tilde{\gamma}_\eta}(\eta_\gamma) $$
for some constant $C$ independent of $n$, $\gamma$ and $\eta$.
\end{lemma}

\begin{proof}
First of all, by the third point of the previous lemma, and since $d(\eta,\eta_{\tilde{\gamma}_\eta}) \leq C_{\text{reg}}^{-2/\alpha} r_\gamma$, we can write $ f_{\gamma}(\eta)/f_{\gamma}(\eta_{\tilde{\gamma}_\eta}) \in [1/2,3/2] $.
Then, by the previous lemma again, we see that
$$ \frac{f_\gamma(\eta)}{f_{\tilde{\gamma}_\eta}(\eta_\gamma)} \simeq \frac{f_{\gamma}(\eta_{\tilde{\gamma}_\eta})}{f_{\tilde{\gamma}_\eta}(\eta_\gamma)} \simeq \frac{D_{F,o}(\eta_{\tilde{\gamma}_\eta},\eta_\gamma)^2}{D_{F,o}(\eta_{\gamma},\eta_{\tilde{\gamma}_\eta})^{2}} \simeq 1,$$
where we used the quasi symmetry of the gap map (\cite{PPS15}, page 47).
\end{proof}

We are ready to introduce the $n$-th approximation operator. For some positive function $R : \Lambda_\Gamma \rightarrow \mathbb{R}^*_+$, define, on $\partial_\infty \mathbb{H}^d$, the following positive function:
$$ P_n R(\eta) := \sum_{\gamma \in S_n} R(\eta_\gamma) r_\gamma^F f_\gamma(\eta) .$$ 
The function $P_n R$ has the regularity of $f_\gamma$ for $\gamma \in S_n$.

\begin{lemma}
Choosing $C_\Gamma$ larger if necessary, the following hold. Let $n \in \mathbb{N}$ and let $\xi,\eta \in \Lambda_{\Gamma}$ such that $d_o(\xi,\eta) \leq r_{n+1}$. Then:
$$ \left| \frac{P_nR(\xi)}{P_n R(\eta)} - 1 \right| \leq \frac{1}{2} d_o(\xi,\eta)^\alpha r_{n+1}^{-\alpha} .$$

\end{lemma}

\begin{proof}
The regularity estimates on $f_\gamma$ and the positivity of $R$ yields:
$$ \left| P_nR(\xi) - P_nR(\eta) \right| \leq \sum_{\gamma \in S_n} R(\eta_\gamma) r_\gamma^F |f_\gamma(\xi)-f_\gamma(\eta)| $$
$$ \leq C_{\text{reg}}  \sum_{\gamma \in S_n} R(\eta_\gamma) r_\gamma^F f_\gamma(\eta) d_0(\xi,\eta)^{\alpha} r_{\gamma}^{-\alpha} \leq P_nR(\eta) \cdot C_{\text{reg}} e^{-2 \alpha C_{\Gamma}}  d_0(\xi,\eta)^{\alpha} r_{n+1}^{-\alpha} .$$
The bound follows.
\end{proof}

Now is the time where we combine all of our preliminary lemmas to prove our main technical lemma: the approximation operator does a good enough job at approximating on $\Lambda_\Gamma$. Some natural hypothesis on $R$ are required: the function to approximate has to be regular enough at scale $r_n$, and has to have mild global variations (so that the decay of $f_\gamma$ away from $x_\gamma^m$ is still useful).

\begin{lemma}[Approximation lemma]

Let $\varepsilon_0 \in (0,\delta_{reg})$ and $C_0 \geq 1$. Let $n \in \mathbb{N}$, and let $R : \Lambda_\Gamma \rightarrow \mathbb{R}$ be a positive function satisfying:
\begin{enumerate}
    \item For $\xi,\eta \in \Lambda_\Gamma$, if $d_o(\xi,\eta) \leq r_{n+1}$, then
    $$ \left|\frac{R(\xi)}{R(\eta)} - 1\right| \leq \frac{1}{2} \left(\frac{d_o(\xi,\eta)}{r_{n+1}}\right)^{\alpha} .$$
    \item For $\xi,\eta \in \Lambda_\Gamma$, if $d_o(\xi,\eta) > r_{n+1}$, then
    $$ R(\xi)/R(\eta) \leq C_0 d_o(\xi,\eta)^{\varepsilon_0} r_n^{-\varepsilon_0} .$$
\end{enumerate}
Then there exists $A \geq 1$ that only depends on $\varepsilon_0$ and $C_0$ such that, for all $\eta \in \Lambda_\Gamma$:
$$ A^{-1} R(\eta) \leq P_{n+1} R(\eta) \leq A R(\eta) .$$
\end{lemma}

\begin{proof}
Let $\eta \in \Lambda_\Gamma$. We have:
$$ P_{n+1} R(\eta) = \sum_{\gamma \in S_{n+1}} R(\eta_\gamma) r_\gamma^F f_{\gamma}(\eta) = \underset{\eta \in B_\gamma}{\sum_{\gamma \in S_{n+1}}} R(\eta_\gamma) r_\gamma^F f_{\gamma}(\eta) + \underset{\eta \notin B_\gamma}{\sum_{\gamma \in S_{n+1}}} R(\eta_\gamma) r_\gamma^F f_{\gamma}(\eta) .$$
The first sum is easily controlled: if $\eta \in B_\gamma$, then $R(\eta_\gamma) \simeq R(\eta)$ and $r_\gamma^F f_\gamma(\eta) \simeq 1$. Since $\eta$ is in a (positive and) bounded number of $B_\gamma$, we find
$$C^{-1} R(\eta) \leq \underset{\eta \in B_\gamma}{\sum_{\gamma \in S_{n+1}}} R(\eta_\gamma) r_\gamma^F f_{\gamma}(\eta) \leq C R(\eta),$$
which gives the lower bound since $R$, $r_\gamma^F$ and $f_\gamma$ are positive. To conclude, we need to get an upper bound on the residual term. Using $\text{diam}(B_\gamma) \lesssim r_n$, the symmetry lemma on $f_\gamma(\eta)$, the regularity and mild variations of $R$, and using the shadow lemma $r_\gamma^F \simeq \mu_o(B_\gamma)$, we get:
$$ \underset{\eta \notin B_\gamma}{\sum_{\gamma \in S_{n+1}}} R(\eta_\gamma) r_\gamma^F f_{\gamma}(\eta) \lesssim R(\eta) r_n^{-\varepsilon_0} \sum_{\eta \notin B_\gamma} r_\gamma^F d_o(\eta_\gamma,\eta)^{\varepsilon_0} f_{\tilde{\gamma}_\eta}( \eta_\gamma ) $$ $$ \lesssim R(\eta) r_n^{-\varepsilon_0} \sum_{\eta \notin B_\gamma} \mu(B_\gamma) d_o(\eta_\gamma,\eta)^{\varepsilon_0} f_{\tilde{\gamma}_\eta}(\eta_\gamma) .$$
Now notice that, for all $\gamma \in S_{n+1}$ such that $\eta \notin B_\gamma$, and for all $\xi \in B_\gamma$, we have
$$  d_o(\eta_\gamma,\eta)^{\varepsilon_0} f_{\tilde{\gamma}_\eta}(\eta_\gamma) \lesssim \Big(d_o(\eta_\gamma,\xi)^{\varepsilon_0} + d_o(\xi,\eta)^{\varepsilon_0} \Big) f_{\tilde{\gamma}_\eta}(\xi) \lesssim \Big(r_n^{\varepsilon_0} + d_o(\xi,\eta)^{\varepsilon_0} \Big) f_{\tilde{\gamma}_\eta}(\xi)  , $$
so that integrating in $\xi \in B_\gamma$ gives: $$ \mu(B_\gamma) d_o(\eta_\gamma,\eta)^{\varepsilon_0} f_{\tilde{\gamma}_\eta}(\eta_\gamma) \lesssim r_n^{\varepsilon_0} \int_{B_\gamma} \Big{(}1+ \frac{d_o(\xi,\eta)^{\varepsilon_0}}{r_n^{\varepsilon_0}} \Big) d\mu_o(\xi).$$ Injecting this estimate in the previous sum, and using Lemma 6.2.16 gives $$ \underset{\eta \notin B_\gamma}{\sum_{\gamma \in S_{n+1}}} R(\eta_\gamma) r_\gamma^F f_{\gamma}(\eta) \lesssim R(\eta) \int_{B(\eta,r_{n+1})^c} \Big(1+\frac{d_o(\xi,\eta)^{\varepsilon_0}}{r_n^{\varepsilon_0}} \Big) f_{\tilde{\gamma}_\eta}(\xi) d\mu_o(\xi).$$
Finally, the second contraction lemma and the bound $d_o(\eta,\eta_{\tilde{\gamma}_\eta}) \lesssim r_n$ yields:
$$ \int_{B(\eta,r_{n+1})^c} \Big(1+\frac{d_o(\xi,\eta)^{\varepsilon_0}}{r_n^{\varepsilon_0}} \Big) \ f_{\tilde{\gamma}_\eta}(\xi) d\mu_o(\xi) \leq \int_{\Lambda_\Gamma} \Big(1+\frac{d_o(\xi,\eta)^{\varepsilon_0}}{r_n^{\varepsilon_0}} \Big) 
 \ f_{\tilde{\gamma}_\eta}(\xi) d\mu_o(\xi)  $$ $$= \int_{\Lambda_{\Gamma}} \Big(1+\frac{d_o(\xi,\eta)^{\varepsilon_0}}{r_n^{\varepsilon_0}} \Big)d\mu_{\tilde{\gamma}_\eta o}(\xi) =  \int_{\Lambda_\Gamma} \Big(1+\frac{d_o(\tilde{\gamma}_\eta(\xi),\eta)^{\varepsilon_0}}{r_n^{\varepsilon_0}} \Big) d\mu_{o}(\xi)$$ 
$$ \lesssim  1 + r_n^{-\varepsilon_0} \int_{\Lambda_\Gamma} d_o(\tilde{\gamma}_\eta(\xi),\eta_{\tilde{\gamma}_\eta})^{\varepsilon_0} d\mu_{o}(\xi)  \lesssim 1 ,$$
which concludes the proof.
\end{proof}

\subsection{The construction of $\nu$}

In this last subsection, we construct the measure $\nu$ and conclude that some $(\Gamma,F)$ Patterson-Sullivan densities are stationary measures for a random walk on $\Gamma$ with exponential moment. We can conclude by following the end of \cite{LNP19} very closely, but we will recall the last arguments for the reader's convenience. \\

Recall that the large constant $C_\Gamma \geq 1$ was fixed just before Lemma 6.2.16, and that $r_n := e^{-4 C_\Gamma n}$. Recall also that $\alpha>0$ is fixed by Lemma 6.2.8. We fix $\beta \in (0,1)$ small enough so that $1-\beta \geq e^{-4 C_\Gamma \alpha} + \beta$, and we choose $\varepsilon_0$ so that $r_n^{\varepsilon_0} = (1-\beta)^n$. By taking $\beta$ even smaller, we can suppose that $\varepsilon_0 < \delta_{reg}$. For this choice of $\varepsilon_0$, and for $C_0(\varepsilon_0) := 2 (1-\beta)^{-2} e^{4 C_\Gamma \varepsilon_0}$, the approximation lemma gives us a constant $A > 1$ such that, under the hypothesis of Lemma 6.3.9:
$$ \frac{\beta}{A^2} R \leq \frac{\beta}{A} P_{n+1} R \leq \beta R .$$
We then use $P_n$ to successively take away some parts of $R$. Define, by induction, $R_0 := 1$ and
$$ R_{n+1} := R_{n} - \frac{\beta}{A} P_{n+1} R_n \leq R_n.$$
For the process to work as intended, we need to check that $R_n$ satisfies the hypothesis of the approximation lemma.

\begin{lemma}
Let $n \in \mathbb{N}$. The function $R_n$ is positive on $\Lambda_{\Gamma}$, and for any $\xi,\eta \in \Lambda_{\Gamma}$:
\begin{enumerate}
\item If $d_o(\xi,\eta) \leq r_{n+1}$, then $$ \left| \frac{R_n(\xi)}{R_n(\eta)} - 1 \right| \leq \frac{1}{2} d_o(\xi,\eta)^{\alpha} r_{n+1}^{-\alpha}. $$
\item If $d_o(\xi,\eta) > r_{n+1}$, then $$ R_n(\xi)/R_n(\eta) \leq C_0(\varepsilon_0) \cdot  d_o(\xi,\eta)^{\varepsilon_0} (1-\beta)^{-n} .$$
\end{enumerate}
\end{lemma}

\begin{proof}

The proof goes by induction on $n$. The case $n=0$ is easy: the first point holds trivially and the second holds since $C_0(\varepsilon_0) r_{1}^{\varepsilon_0} = C_0(\varepsilon_0) e^{- 4 C_\Gamma \varepsilon_0} \geq 1$. Now, suppose that the result hold for some $n$. In this case, the approximation lemma yields
$$ R_{n+1} = R_n - \frac{\beta}{A} P_{n+1} R_n \geq (1-\beta) R_n, $$
and in particular $R_{n+1}$ is positive. Let us prove the first point: consider $\xi,\eta \in \Lambda_{\Gamma}$ such that $d_o(\xi,\eta) \leq r_{n+2}$. Then Lemma 6.3.8 gives
$$ \left| \frac{\beta}{A} P_{n+1} R_n(\xi) - \frac{\beta}{A} P_{n+1} R_n(\eta) \right| \leq \frac{1}{2} \left(\frac{\beta}{A} P_{n+1} R_n(\eta)\right) \cdot d_o(\xi,\eta)^{\alpha} r_{n+2}^{-\alpha} \leq \frac{1}{2} \beta R_n(\eta) \cdot d_o(\xi,\eta)^{\alpha} r_{n+2}^{-\alpha}. $$
Hence, using the induction hypothesis:
$$ \left| R_{n+1}(\xi) - R_{n+1}(\eta) \right| \leq \left| R_{n}(\xi) - R_{n}(\eta) \right| + \left| \frac{\beta}{A} P_{n+1} R_n(\xi) - \frac{\beta}{A} P_{n+1} R_n(\eta) \right| $$ $$ \leq \frac{1}{2} \left( r_{n+1}^{-\alpha} + \beta r_{n+2}^{-\alpha} \right) R_{n}(\eta) d_o(\xi,\eta)^{\alpha}  $$
$$ \leq \frac{1}{2} \frac{e^{-4 C_\Gamma \alpha} + \beta}{1-\beta} R_{n+1}(\eta) d_o(\xi,\eta)^{\alpha} r_{n+2}^{-\alpha} .$$
Recalling the definition of $\beta$ gives the desired bound. It remains to prove the second point. First of all, notice that, for any $\xi$ and $\eta$, we have:
$$ \frac{R_{n+1}(\xi)}{R_{n+1}(\eta)} \leq (1-\beta)^{-1} \frac{R_n(\xi)}{R_n(\eta)} .$$
Now, suppose that $d_o(\xi,\eta) \in (r_{n+2},r_{n+1}]$. The induction hypothesis gives $R_n(\xi)/R_n(\eta) \leq 1+|R_n(\xi)/R_n(\eta)-1| \leq 2$, and so:
$$ \frac{R_{n+1}(\xi)}{R_{n+1}(\eta)} \leq \frac{2}{1-\beta} =  \frac{2}{1-\beta} r_{n+2}^{\varepsilon_0} (1-\beta)^{-(n+2)} \leq \frac{2}{(1-\beta)^2} \cdot d_o(\xi,\eta)^{\varepsilon_0} (1-\beta)^{-{n+1}} ,$$
which proves the bound. Finally, suppose that $d_o(\xi,\eta) > r_{n+1}$. In this case, the induction hypothesis directly yields
$$ \frac{R_{n+1}(\xi)}{R_{n+1}(\eta)} \leq \frac{1}{1-\beta} \frac{R_n(\xi)}{R_n(\eta)} \leq C_0(\varepsilon_0) d_o(\xi,\eta)^{\varepsilon_0} (1-\beta)^{-(n+1)} ,$$
and the proof is done. \end{proof}

We are ready to prove Theorem 6.3.2, following Li in \cite{LNP19}.

\begin{proof}

The previous lemma ensure that for all $n$, the function $R_n$ satisfies the hypothesis of the approximation lemma. Hence, we can write for all $n$,

$$ R_{n+1} = R_n - \frac{\beta}{A} P_{n+1} R_n \leq \left(1-\frac{\beta}{A^2}\right) R_n, $$
so that by induction:
$$ R_n \leq \left(1- \frac{\beta}{A^2} \right)^n \longrightarrow 0.$$
It follows that $$ 1 = R_0 - \lim_{n} R_n = \sum_{n=1}^\infty (R_{n-1} - R_{n} ) = \frac{\beta}{A} \sum_{n=1}^\infty P_{n} (R_{n-1}). $$
In other words,
$$ 1 = \sum_{n = 1}^\infty \sum_{\gamma \in S_n} \frac{\beta}{A} R_{n-1}(\eta_\gamma) r_\gamma^F \cdot f_\gamma .$$
Letting $$ \nu(\gamma) :=  \frac{\beta}{A} R_{n-1}(\eta_\gamma) r_\gamma^F \ \text{if} \ \gamma \in S_n, \quad \text{and} \quad \nu(\gamma):=0 \ \text{if} \ \gamma \notin \bigcup_k S_k , $$
 gives us a probability measure on $\Gamma$ (since $\int f_\gamma d\mu_o =1$) satisfying $\nu * \mu_o = \mu_o$, by the remarks made section in 6.3.1. Checking that the measure has exponential moment is easy since $\|\gamma\| \lesssim r_\gamma^{-1}$ by Remark 6.2.4. Indeed, by the shadow lemma $r_\gamma^F \simeq \mu_o(B_\gamma)$, and since $S_n$ covers each point a bounded number of time, we get:
$$ \int_{\Gamma} \|\gamma\|^{\varepsilon} d\nu \lesssim \sum_{n} \sum_{\gamma \in S_n} \nu(\gamma) r_\gamma^{-\varepsilon}  $$
$$ \lesssim \sum_n \left( \sum_\gamma \mu_o(B_\gamma) \right) (1-\beta/A^2)^n e^{\varepsilon 4 C_\Gamma n} < \infty $$
if $\varepsilon$ is small enough. Finally, we show that $\Gamma_\nu$, the group spanned by the support of $\nu$, is $\Gamma$. To see this, say that $C_\Gamma$ was chosen so large that $C_\Gamma \geq 6 \text{diam} (\text{Hull}(\Lambda_\Gamma)/\Gamma)$. In this case, there exists $\gamma_1 \in S_1$ such that $d(o,\gamma_1 o) \in [|\ln r_1| + C_\Gamma/2, |\ln r_1| + 3 C_\Gamma/2]$. Then, any $\gamma \in \Gamma$ such that $d(o,\gamma o) \leq C_\Gamma/2$ satisfies $\gamma_1 \gamma \in S_1$. In particular:
$$ \{ \gamma \in \Gamma \ , \ d(o,\gamma o) \leq C_\Gamma/2 \} \subset \Gamma_\nu,$$
and it is then well known (see for example Lemma A.14 in \cite{LNP19}) that this set spans the whole group $\Gamma$ as soon as $C_\Gamma/2$ is larger than 3 times the diameter of $\text{Hull}({\Lambda_\Gamma})/\Gamma$. 
\end{proof}

\section{Consequences on equilibrium states}

Now that we have proved Theorem 6.1.2, Corollary 6.1.3 follows directly from \cite{Li20} (since, when $d=2$, being Zariski dense is equivalent to being non-elementary). To see how this statement induces some knowledge over equilibrium states, let us recall more precisely the link between the latter and Patterson-Sullivan densities. First, recall that the Hopf coordinates $$\text{Hopf} : \left((\partial_\infty \mathbb{H}^2 \times \partial_\infty \mathbb{H}^2)\setminus \mathcal{D}\right) \times \mathbb{R} \longrightarrow T^1 \mathbb{H}^2$$ allows us to smoothly identify the unit tangent bundle of $\mathbb{H}^2$ with a torus minus the diagonal times $\mathbb{R}$ by the following process. For any $v^+ \neq v^- \in \partial_\infty \mathbb{H}^2$, and for any $t \in \mathbb{R}$, $ \text{Hopf}(v^+,v^-,t) := v $ is the unique vector $v \in T^1 M$ lying on the geodesic $]v^-,v^+[$ such that $\tilde{p}(\phi_{-t}(v))$ is the closest point to $o$ on this geodesic. 
We will denote by $(\partial_{v^+},\partial_{v^-}, \partial_{t})$ the induced basis of $T(T^1 M)$ in these coordinates. Finally, recall that $\iota$ denotes the flip map.

\begin{theorem}[\cite{PPS15}, Theorem 6.1]

Let $\Gamma \subset Iso^+(\mathbb{H}^d)$ be convex cocompact, $M := \mathbb{H}^d/\Gamma$, and $F:T^1 M \rightarrow \mathbb{R}$ be a normalized and Hölder-regular potential. Denote by $m_F \in \mathcal{P}(T^1 M)$ the associated equilibrium state, and let $\tilde{m}_F$ be its $\Gamma$-invariant lift on $T^1 \mathbb{H}^d$. Denotes by $\mu_{x}^F$ the $(\Gamma,F)$ Patterson-Sullivan density with basepoint $x$. Then, for any choice of $x \in \mathbb{H}^d$, the following identity hold in the Hopf coordinates (up to a multiplicative constant $c_0 > 0$):
$$ c_0 \cdot d\tilde{m}_F(v^+,v^-,t) = \frac{d\mu_{x}^F(v^+) d\mu^{F \circ \iota}_{x}(v^-)  dt}{D_{F,x}(v^+,v^-)^{2}} .$$
\end{theorem}

We are now ready to prove Fourier decay for $m_F$. To do a clean proof, we write down three lemmas corresponding to Fourier decay in the three directions $(\partial_{v^+}, \partial_{v^-},\partial_t)$. We will then combine all of them to get the desired result (in the spirit of sections 1.1.3 and 1.1.4).

\begin{lemma}
Let $\alpha \in (0,1)$. Under the conditions of Theorem 6.1.2, there exists $\rho_1,\rho_2>0$ and $C \geq 1$ such that the following hold. Let $\xi \in \mathbb{R}^*$. Then, let $\chi : T^1 \mathbb{H}^2 \rightarrow \mathbb{R}$ be a Hölder map supported on some compact $K$, and let a $C^2$ function $\varphi : T^1 \mathbb{H}^2 \rightarrow \mathbb{R}$ such that $\|\varphi\|_{C^2} + (\inf_{K} |\partial_{v^+} \varphi|)^{-1} + \|\chi\|_{C^\alpha} \leq |\xi|^{\rho_1} $, we have:
$$ \left|\int_{T^1 \mathbb{H}^2} e^{i \xi \varphi(v)} \chi(v) d\tilde{m}_F(v) \right| \leq \frac{C}{ |\xi|^{\rho_2}} .$$
\end{lemma}

\begin{proof}
Denotes $\tilde{\varphi}$ and $\tilde{\chi}$ the functions $\varphi,\chi$ seen in the Hopf coordinates. We get, for some large $a>0$ depending only on the support of $\chi$:
$$ c_0 \int_{T^1 \mathbb{H}^2} e^{i \xi \varphi(v)} \chi(v) d\tilde{m}_F(v)   = \int_{-a}^a \int_{\Lambda_\Gamma}  \left( \int_{\Lambda_\Gamma} e^{i \xi \tilde{\varphi}(v^+,v^-,t)} \frac{\tilde{\chi}(v^+,v^-,t)}{D_{F,o}(v^+,v^-)^2} d\mu_o^F(v^+) \right) d\mu_o^{F \circ \iota}(v^-) dt .$$
Now, since $\tilde{\chi}$ is supported in a compact subset of $\left((\partial_\infty \mathbb{H}^2 \times \partial_\infty \mathbb{H}^2)\setminus \mathcal{D}\right) \times \mathbb{R}$, and since $D_F$ is uniformly Hölder (and doesn't vanish) on a compact subset of $\Lambda_\gamma \times \Lambda_\Gamma \setminus \mathcal{D}$ (see \cite{PPS15}, Lemma 3.6 and Proposition 3.5), and finally since $|\partial_{v^+} \tilde{\varphi}| \gtrsim |\xi|^{-\rho_2}$ on the compact support of $\tilde{\chi}$, we see that Corollary 6.1.3 applies to the inner integral. (Notice that we can always extend $D_F$ outside of $\Lambda_\Gamma \times \Lambda_\Gamma \setminus \mathcal{D}$ so that it becomes Hölder on all $(\partial_\infty \mathbb{H}^2 \times \partial_\infty \mathbb{H}^2) \setminus \mathcal{D}$, see \cite{Mc34}.) This gives the desired bound.
\end{proof}

\begin{lemma}
Let $\alpha \in (0,1)$. Under the conditions of Theorem 6.1.2, there exists $\rho_1,\rho_2>0$ and $C \geq 1$ such that the following hold. Let $\xi \in \mathbb{R}^*$. Then, let $\chi : T^1 \mathbb{H}^2 \rightarrow \mathbb{R}$ be a Hölder map supported on some compact $K$, and let a $C^2$ function $\varphi : T^1 \mathbb{H}^2 \rightarrow \mathbb{R}$ such that  $\|\varphi\|_{C^2} + (\inf_{K} |\partial_{v^-} \varphi|)^{-1} + \|\chi\|_{C^\alpha} \leq |\xi|^{\rho_1}$, we have:
$$ \forall \xi \in \mathbb{R}^*, \ \left|\int_{T^1 \mathbb{H}^2} e^{i \xi \varphi(v)} \chi(v) d\tilde{m}_F(v) \right| \leq \frac{C}{ |\xi|^{\rho_2}} .$$
\end{lemma}

\begin{proof}
We need to check that when $F$ satisfies the regularity assumptions (R), then $F \circ \iota$ satisfies them too. This is easy, since $\sup_{\Omega \Gamma} F \circ \iota = \sup_{\Omega \Gamma} F < \delta_{\Gamma,F} = \delta_{\Gamma, F \circ \iota} $ by Lemma 3.3 in \cite{PPS15}. Moreover, $F \circ \iota$ is still Hölder regular. Hence, one can apply our previous lemma with $F$ replaced by $F \circ \iota$, and conclude.
\end{proof}

\begin{lemma}
Under the conditions of Theorem 6.1.2, there exists $C \geq 1$ such that, for all $\xi \in \mathbb{R}^*$, for any $\alpha$-Hölder map $\chi : T^1 \mathbb{H}^2 \rightarrow \mathbb{R}$ supported on some compact $K$, and for any $C^2$ function $\varphi : T^1 \mathbb{H}^2 \rightarrow \mathbb{R}$ satisfying $\|\varphi\|_{C^2} + (\inf_K |\partial_{t} \varphi|)^{-1} + \|\chi\|_{C^\alpha} \leq |\xi|^{\alpha/8}$, we have:

$$ \left|\int_{T^1 \mathbb{H}^2} e^{i \xi \varphi(v)} \chi(v) d\tilde{m}_F(v) \right| \leq \frac{C}{ |\xi|^{\alpha/2}} $$
\end{lemma}

\begin{proof}
The proof is classic. We have, for some compact $\tilde{K} \subset  \left(\partial_\infty \mathbb{H}^2 \times \partial_\infty \mathbb{H}^2 \right) \setminus \mathcal{D}$ and for some large enough $a>0$ depending only on the support of $\chi$:
$$ c_0 \int_{T^1 \mathbb{H}^2} e^{i \xi \varphi(v)} \chi(v) d\tilde{m}_F(v)   = \iint_{\tilde{K}}  \left( \int_{-a}^a e^{i \xi \tilde{\varphi}(v^+,v^-,t)} {\tilde{\chi}(v^+,v^-,t)} dt \right) D_{F,o}(v^+,v^-)^{-2} d(\mu_o^F \otimes \mu_o^{F \circ \iota})(v^+,v^-) .$$
We then work on the inner integral. When $\tilde{\chi}$ is $C^1$, we can conclude by an integration by parts. So a way to conclude is to approximate $\tilde{\chi}$ by a $C^1$ map. Fix some smooth bump function $\rho : \mathbb{R} \rightarrow \mathbb{R}^+$ such that $\rho$ is zero outside $[-2,2]$, one inside $[-1,1]$, increasing on $[-2,-1]$ and decreasing on $[1,2]$. For any $\varepsilon>0$, set $$\tilde{\chi}_\varepsilon(\cdot,\cdot,t) := \int_\mathbb{R}\tilde{\chi}(\cdot,\cdot,t-x) \rho(x/\varepsilon) dx/\varepsilon.$$
This function is smooth in the $t$-variable. Moreover, if we denote by $\alpha$ a Hölder exponent for $\chi$, then a direct computation yields:

$$ \|\tilde{\chi}_\varepsilon - \tilde{\chi}\|_\infty \lesssim \xi^{\alpha/8} \varepsilon^\alpha, \quad \|\partial_t \tilde{\chi}_\varepsilon \|_\infty \lesssim \xi^{\alpha/8} \varepsilon^{-(1-\alpha)} .$$
Hence:

$$ \left|\int_{-a}^a e^{i \xi \tilde{\varphi}} {\tilde{\chi}} dt\right| \leq 2 a \| \tilde{\chi} - \tilde{\chi}_\varepsilon \|_\infty + \left|\int_{-a}^a e^{i \xi \tilde{\varphi}} {\tilde{\chi}_\varepsilon} dt\right| . $$
To control the integral on the right, we do our aforementioned integration by parts:

$$ \int_{-a}^a e^{i \xi \tilde{\varphi}} {\tilde{\chi}_\varepsilon} dt = \int_{-a}^a \frac{i \xi \partial_t \tilde{\varphi}}{i \xi \partial_t \tilde{\varphi}} e^{i \xi \tilde{\varphi}} {\tilde{\chi}_\varepsilon} dt $$
$$ = \left[ \frac{\tilde{\chi}_\varepsilon}{i \xi \partial_t \tilde{\varphi}} e^{i \xi \tilde{\varphi}} \right]_{t=-a}^{t=a} - \frac{i}{\xi} \int_{-a}^a \partial_t\left( \frac{\tilde{\chi}_\varepsilon}{\partial_t\tilde{\varphi}} \right) e^{i \xi \tilde{\varphi}} dt ,$$
so that $$ \left| \int_{-a}^a e^{i \xi \tilde{\varphi}} {\tilde{\chi}_\varepsilon} dt \right| \lesssim |\xi|^{-1+\alpha/2} {\varepsilon^{-(1-\alpha)}} .$$
Finally, choosing $\varepsilon = 1/|\xi|$ yields
$$ \left| \int_{T^1 \mathbb{H}^2} e^{i \xi \varphi(v)} \chi(v) d\tilde{m}_F(v) \right| \lesssim \varepsilon^\alpha + \varepsilon^{-(1-\alpha)} |\xi|^{-1+\alpha/2} \lesssim |\xi|^{-\alpha/2} ,$$ which is the desired bound. \end{proof}

\begin{theorem}
Let $\alpha \in (0,1)$. Under the conditions of Theorem 6.1.2, there exists $\rho_1,\rho_2>0$ and $C \geq 1$ such that the following holds. For any $\xi \in \mathbb{R}^*$, for any   $\alpha$-Hölder map $\chi : T^1 M \rightarrow \mathbb{R}$ supported on some compact $K$, for all $C^2$ function $\varphi : T^1 M \rightarrow \mathbb{R}$ satisfying $\|\varphi\|_{C^2}+(\inf_K \|d\varphi\|)^{-1} + \|\chi\|_{C^\alpha} \leq |\xi|^{\rho_1}$, we have:
$$ \left|\int_{T^1 M} e^{i \xi \varphi} \chi d m_F \right| \leq \frac{C}{ |\xi|^{\rho_2}}. $$
\end{theorem}

\begin{proof}
Let $\alpha>0$, and let $\rho_1,\rho_2>0$ be such that Lemma 6.4.2, Lemma 6.4.3 and Lemma 6.4.4 applies. Let $\xi \in \mathbb{R}$, $|\xi| \geq 1$, let $\varphi:T^1 M \rightarrow \mathbb{R}$ and let $\chi : T^1 M \rightarrow \mathbb{R}$ be a $\alpha$-Hölder map such that
$$\|\varphi\|_{C^2}+(\inf_K \|d\varphi\|)^{-1} + \|\chi\|_{C^\alpha} \leq |\xi|^{\rho_1 \rho_2/100}.$$
We then cover $NW(\phi) \subset T^1 M$ by some balls of radius $\xi^{\rho_1 \rho_2/30}$. Since this is a compact subset of a $3$-dimensionnal manifold, this can be done using $\lesssim \xi^{- \rho_1 \rho_2 /10}$ balls. Let us then denote $\text{NW}(\phi) \subset \bigcup_{j \in J} B_j$, with $|I| \lesssim \xi^{-\rho_1 \rho_2/10}$ and $B_j=B(x_j,\xi^{-\rho_1 \rho_2/30})$ for some $x_j \in T^1 M$. We can suppose that each $B_j$ intersects $NW(\phi)$. 
We can further construct this cover so that there is an associated partition of unity $(\chi_j)_{j \in J}$ satisfying $\sum_j \chi_j = 1$ on $\text{NW}(\phi)$, $\text{supp}(\chi_j) \subset B_j$, and $\|\chi_j\|_{C^1} \lesssim \xi^{-\rho_1 \rho_2/30}$. (This is easy to do in local charts: in $\mathbb{R}^3$, construct a partition of unity adapted to a $\mathbb{Z}^3$-periodic partition of $\mathbb{R}^3$ with cubes of side $1$, and then rescale everything.) Then, let us write:
$$ \int_{T^1 M} e^{i \xi \varphi} \chi d m_F = \sum_{j \in I} \int_{B_j \cap \text{NW}(\varphi)} e^{i \xi \phi} \chi \chi_j dm_F. $$
Select some fundamental domain $D_0$ of $\Gamma$ in $T^1 \mathbb{H}^2$, and then consider a large enough compact neighborhood $D$ of $D_0 \cap \tilde{p}^{-1}(\text{Hull}(\Lambda_\Gamma))$. For each $j$, we choose a Hölder lift $\tilde{\chi}_j : T^1 \mathbb{H}^2 \rightarrow \mathbb{R}$ of $\chi \chi_j$ with compact support $\tilde{B}_j \subset D$, with diameter $\lesssim \xi^{-\rho_1 \rho_2/30}$. Finally, let $\tilde{\phi} : T^1 \mathbb{H}^2 \rightarrow \mathbb{R}$ be a $\Gamma$-periodic lift of $\varphi$. We have:
$$ \int_{B_j \cap NW(\phi)} e^{i \xi \phi} \chi \chi_j dm_F = \int_{\tilde{B}_j} e^{i \xi \tilde{\phi}} \tilde{\chi_j} d\tilde{m}_F. $$
Now consider the map $\mathcal{B}:T^1(T^{1} \mathbb{H}^2) \rightarrow \mathbb{R}_+^*$ defined by $$\mathcal{B}( \alpha_{v^+} \partial_{v^+} + \alpha_{v^-} \partial_{v^-} + \alpha_{t} \partial_t ) := \min(|\alpha_{v^+}|,|\alpha_{v^-}|,|\alpha_t|).$$
Since this map is continuous on the compact set $D$, there exists $c_0>0$ such that $\mathcal{B}( D ) \subset [c_0,\infty[$. We can finally go on and control all of our integrals. \\
Let $j \in I$. Let $y_j \in \tilde{B}_j \cap \tilde{p}^{-1}(\text{Hull}(\Lambda_\Gamma))$. Notice that since, for all $x \in \tilde{p}^{-1}(\text{Hull}(\Lambda_\Gamma))$, $\|d\tilde{\phi}_x\| \gtrsim \xi^{-\rho_1 \rho_2 /100}$, there exists some unit vector $h_j \in T^1(T^1M)$ such that $|d\phi_{y_j}(h_j)|>\xi^{-\rho_1 \rho_2 /100}$. It follows that there exists $e \in \{v^+,v^-,t\}$ such that $ |\partial_e \phi(y_j) | > c_0 \xi^{-\rho_1 \rho_2/100} $. Notice then that for each $y \in \tilde{B}_j$, we have:
$$ |\partial_e \phi(y)| \geq |\partial_e \phi(y_j) | - |\partial_e \phi(y) - \partial_e \phi(y_j) | $$ $$\geq c_0 \xi^{-\rho_1 \rho_2/100} - \| \phi \|_{C^2} \text{diam}(\tilde{B}_j) \geq c_0 \xi^{-\rho_1 \rho_2/100}(1 - \xi^{-\rho_1 \rho_2/30}) \geq   \xi^{-\rho_1}  $$
if $\xi$ is large enough. Moreover:
$$ \| \tilde{\chi}_j \|_{C^\alpha} = \| \chi \chi_j \|_{C^\alpha} \lesssim \|\chi\|_{C^\alpha} \xi^{\rho_1 \rho_2 /100} \lesssim \xi^{\rho_1 \rho_2 /50} \leq \xi^{\rho_1}.$$
We can then apply our previous lemmas to control our integral: there exists $C \geq 1$ (independant of $j$) such that:
$$ \Big| \int_{\tilde{B}_j} e^{i \xi \tilde{\phi}} \tilde{\chi_j} d\tilde{m}_F \Big| \leq \frac{C}{|\xi|^{\rho_2}}. $$
Summing those estimates and using the bound $|I| \lesssim \xi^{-\rho_1 \rho_2/10}$ yields:
$$ \Big| \int_{T^1 M} e^{i \xi \varphi} \chi d m_F \Big| \leq \sum_{j \in I} \Big| \int_{B_j \cap \text{NW}(\varphi)} e^{i \xi \phi} \chi \chi_j dm_F \Big| \leq \frac{C |I|}{|\xi|^{\rho_2}} \leq \frac{\tilde{C}}{|\xi|^{\rho_2/2}} .$$
\end{proof}

\begin{remark}
We recover our main Theorem 6.1.4 as a particular case of Theorem 6.4.5. Indeed, if $\varphi : K \subset T^1 M \rightarrow \mathbb{R}^3$ is a $C^2$ local chart, then for any $\zeta \in \mathbb{R}^3 \setminus \{0\}$, one may write:

$$ \int_{T^1 M} e^{i \zeta \cdot \varphi(v)} \chi(v) d m_F(v) = \int_{T^1 M} e^{i |\zeta| (\zeta/|\zeta|) \cdot \varphi(v)} \chi(v) d m_F(v) \leq C |\zeta|^{- \rho_1} ,$$
since the map $(u,v) \in \mathbb{S}^2 \times K \mapsto u \cdot (d\varphi)_v \in T_v^*(T^1 M)$ doesn't vanish (because the range of $(d\varphi)_v$ isn't contained in a plane). In fact, we can write $\|  u \cdot (d\varphi)_v \| \gtrsim \| (^t d \varphi^{-1})_v \|^{-1}$.  Notice that we used the uniformity of the constants $C \geq 1$ given by the phases $u \cdot (d\varphi)$.
\end{remark}

\cleardoublepage
\phantomsection


\addcontentsline{toc}{chapter}{Bibliography}


\begin{thebibliography}{9}


\bibitem[ABV14]{ABV14}
    {V. Araújo, O. Butterley, P. Varandas}
    \emph{Open sets of Axiom A flows with Exponentially Mixing Attractors (with Erratum)},
    Proceedings of the American Mathematical Society, 144(7), 2971-2984 (2016) \href{https://arxiv.org/abs/1402.6654v3}{arXiv:1402.6654v3}


\bibitem[AF91]{AF91}
    {R. Adler, L. Flatto}
    \emph{Geodesic flows, interval maps, and symbolic dynamics}
    Bull. Amer. Math. Soc. (N.S.) 25(2): 229-334 (October 1991).


\bibitem[AGY06]{AGY06}
    \emph{A. Avila, S. Gouëzel, J-C. Yoccoz}
    \emph{Exponential mixing for the Teichmüller flow.}
    Publications Mathématiques de l'IHÉS, Tome 104 (2006), pp. 143-211. \href{http://www.numdam.org/articles/10.1007/s10240-006-0001-5/}{doi:10.1007/s10240-006-0001-5}. 

\bibitem[ARW20]{ARW20}
    {A. Algom, F. Rodriguez, Z. Wang},
    \emph{Pointwise normality and Fourier decay for self-conformal measures}
    Volume 393, 108096. \href{https://doi.org/10.1016/j.aim.2021.108096}{doi:10.1016/j.aim.2021.108096}, \href{https://arxiv.org/abs/2012.06529}{arXiv:2012.06529}


\bibitem[Ba90]{Ba90}
    {A. Baker},
    \emph{Transcendental number theory},
    Cambridge Mathematical Library (2nd ed.), Cambridge University Press, \href{https://www.lehmanns.de/shop/mathematik-informatik/748721-9780521397919-transcendental-number-theory}{ISBN:978-0-521-39791-9}, MR 0422171



\bibitem[Ba18]{Ba18}
    {V. Baladi},
    \emph{Dynamical zeta functions and dynamical determinants for hyperbolic maps - A Functional Approach},
    Springer, 68, 2018, Ergebnisse der Mathematik und ihrer Grenzgebiete. 3. Folge / A Series of Modern Surveys in Mathematics.
    \href{https://link.springer.com/book/10.1007/978-3-319-77661-3}{ISBN:978-3-319-77661-3}


\bibitem[BD17]{BD17}
     {J. Bourgain, S. Dyatlov}, 
     \emph{Fourier dimension and spectral gaps for hyperbolic surfaces},
     Geom. Funct. Anal. 27, 744–771 (2017),  \href{https://arxiv.org/abs/1704.02909}{arXiv:1704.02909} 



\bibitem[Be83]{Be83}
     {A.F. Beardon}, 
     \emph{The geometry of discrete groups},
     vol. 91, Springer-Verlag, Berlin and New York, 1983, xii -f
    337 pp . \href{https://link.springer.com/book/10.1007/978-1-4612-1146-4}{ISBN:0-387-90788-2}


\bibitem[Be91]{Be91}
    {A.F. Beardon},
    \emph{Iteration of Rational Functions: Complex Analytic Dynamical Systems},
    Graduate Texts in Mathematics, 1991, Springer New York


\bibitem[Be16]{Be16}
    {P. Berger},
    \emph{Properties of the maximal entropy measure and geometry of Hénon attractors}
    Preprint, 2016. \href{https://doi.org/10.48550/arXiv.1202.2822}{arXiv:1202.2822}

\bibitem[BGS85]{BGS85}
    {W. Ballmann, M. Gromov, V. Schroeder}
    \emph{Manifolds of nonpositive curvature}
    Progress in mathematics, Vol 61, 1985. \href{https://link.springer.com/book/10.1007/978-1-4684-9159-3}{ISBN:978-1-4684-9159-3}


\bibitem[Ch98]{Ch98}
    {N. I. Chernov}
    \emph{Markov Approximations and Decay of Correlations for Anosov Flows.}
    Annals of Mathematics, vol. 147, no. 2, 1998, pp. 269–324. JSTOR, \href{https://doi.org/10.2307/121010}{doi:10.2307/121010}

  
\bibitem[Bl96]{Bl96}
    {C. Bluhm},
    \emph{Random recursive construction of Salem sets}
    Ark. Mat. 34 (1996), 51-63.

\bibitem[Bo75]{Bo75}
    {R. Bowen},
    \emph{Equilibrium States and the Ergodic Theory of Anosov Diffeomorphisms}
    (New edition of) Lect. Notes in Math. 470, Springer, 1975. ISBN : 978-3-540-77605-5

\bibitem[Bo78]{Bo78}
     {R. Bowen},
    \emph{On Axiom A diffeomorphisms}
    CBMS Regional Conference Series in Mathematics
Volume: 35; 1978; \ \href{https://bookstore.ams.org/cbms-35}{ISBN:978-0-8218-3438-1}


\bibitem[Bo10]{Bo10}
    {J. Bourgain},
    \emph{The discretized sum-product and projection theorems} 
    JAMA 112, 193-236 (2010)  \href{https://doi.org/10.1007/s11854-010-0028-x}{doi:10.1007/s11854-010-0028-x}

\bibitem[BP92]{BP92}
    {R. Benedetti and C. Petronio},
    \emph{Lectures on hyperbolic geometry}. 
    Universitext, Springer, 1992. \href{https://link.springer.com/book/10.1007/978-3-642-58158-8}{ISBN:978-3-642-58158-8}


\bibitem[BR75]{BR75}
    {R. Bowen, D. Ruelle},
    \emph{The ergodic theory of Axiom A flows.}
     Invent Math 29, 181–202 (1975). \href{https://link.springer.com/article/10.1007/BF01389848#citeas}{doi:10.1007/BF01389848}
 

\bibitem[Br19]{Br19}
    {J. Brémont},
    \emph{Self-similar measures and the Rajchman property}
    Annales Henri Lebesgue, Tome 4 (2021), pp. 973-1004, \href{https://arxiv.org/abs/1910.03463}{arXiv:1910.03463 }


\bibitem[BS79]{BS79}
    {R. Bowen, C. Series}
    \emph{Markov maps associated with fuchsian groups.} Publications Mathématiques de L’Institut des Hautes Scientifiques 50, 153–170 (1979). \href{https://doi.org/10.1007/BF02684772}{doi:10.1007/BF02684772}.


\bibitem[BS02]{BS02}
    {M. Brin, G. Stuck},
    \emph{Introduction to Dynamical Systems.} 
     2002. Cambridge: Cambridge University Press. \href{https://doi.org/10.1017/CBO9780511755316}{doi:10.1017/CBO9780511755316}.

\bibitem[BS23]{BS23}
    {S. Baker, T. Sahlsten,}
    \emph{Spectral gaps and Fourier dimension for self-conformal sets with overlaps}
    \href{https://arxiv.org/abs/2306.01389}{arXiv:2306.01389 }

\bibitem[BV05]{BV05}
    {V. Baladi , B. Vallée}
    \emph{Exponential Decay of Correlations for Surface Semi-Flows without Finite Markov Partitions}. Proceedings of the American Mathematical Society 133, no. 3 (2005): 865–74. \href{http://www.jstor.org/stable/4097763}{jstor.org/stable/4097763}.


\bibitem[Cl20]{Cl20}
    {V. Climenhaga}
    \emph{SRB and equilibrium measures via dimension theory}
    to appear in A Vision for Dynamics in the 21st Century: The Legacy of Anatole Katok (2023), Cambridge. 
    \href{https://arxiv.org/abs/2009.09260}{arXiv:2009.09260}


\bibitem[CP15]{CP15}
    {S. Crovisier, R. Potrie}
    \emph{Introduction to partially hyperbolic dynamics}
    Lecture notes for School on Dynamical Systems, ICTP, Trieste, 2015.


\bibitem[CS17]{CS17}
    {Chen X, Seeger A.} 
    \emph{Convolution Powers of Salem Measures With Applications.} 
    Canadian Journal of Mathematics. 2017;69(2):284 320. 
    \href{https://www.cambridge.org/core/journals/canadian-journal-of-mathematics/article/convolution-powers-of-salem-measures-with-applications/D7FA122F47A74691636206C8B7A06FEB}{doi:10.4153/CJM-2016-019-6}
 
\bibitem[Do98]{Do98}
    {D. Dolgopyat},
    \emph{On decay of correlations in Anosov flows}, \\
    Ann. of Math. (2) 147 (2) (1998) 357–390.


\bibitem[Do00]{Do00}
    {D. Dolgopyat},
    \emph{Prevalence of rapid-mixing-II: topological prevalence.}
    Ergodic Theory and Dynamical Systems. 2000;20(4):1045-1059. \href{https://doi.org/10.1017/S0143385700000572}{doi:10.1017/S0143385700000572}


\bibitem[Do02]{Do02}
    {D. Dolgopyat},
    \emph{On mixing properties of compact group extensions of hyperbolic systems.}
    Isr. J. Math. 130, 157–205 (2002). \href{https://doi.org/10.1007/BF02764076}{doi:10.1007/BF02764076}.

\bibitem[DV21]{DV21}
    {D. Daltro, P. Varandas}
    \emph{Exponential decay of correlations for gibbs measures on attractors of axiom A flows}  \href{https://arxiv.org/abs/2104.11839}{arXiv:2104.11839}

\bibitem[DZ16]{DZ16}
    {S. Dyatlov, J. Zahl},
    \emph{Spectral gaps, additive energy, and a fractal uncertainty principle.} 
    Geom. Funct. Anal. 26, 1011–1094 (2016). \href{https://doi.org/10.1007/s00039-016-0378-3}{doi:10.1007/s00039-016-0378-3}, \ \href{https://arxiv.org/abs/1504.06589}{arXiv:1504.06589}


 \bibitem[Ek16]{Ek16}
    {F. Ekström},
    \emph{Fourier dimension of random images} Ark Mat 54, 455–471 (2016). \href{https://doi.org/10.1007/s11512-016-0237-3}{doi:10.1007/s11512-016-0237-3}.


\bibitem[Er40]{Er40}
    {P. Erdös}
    \emph{On the Smoothness Properties of a Family of Bernoulli Convolutions} 
    American Journal of Mathematics 62, no. 1 (1940): 180–86. \href{https://doi.org/10.2307/2371446}{doi:10.2307/2371446}

\bibitem [ES11]{ES11}
    {A. Eremenko and S. Van Strien.}
    \emph{Rational maps with real multipliers},
    Trans. AMS, 363, (2011) p. 6453-6463. \href{https://doi.org/10.1090/S0002-9947-2011-05308-0}{doi:10.1090/S0002-9947-2011-05308-0}


\bibitem[Fe91]{Fe91}
    {R. Feres},
    \emph{Geodesic flows on manifolds of negative curvature with smooth horospheric foliations}
    (1991). Ergodic Theory and Dynamical Systems, 11(4), 653-686. \href{https://www.cambridge.org/core/journals/ergodic-theory-and-dynamical-systems/article/abs/geodesic-flows-on-manifolds-of-negative-curvature-with-smooth-horospheric-foliations/A84BEBA3AFFB2A7C92AAD0F5F23D514D}{doi:10.1017/S0143385700006416}


\bibitem[FH03]{FH03}
    {P. Foulon, B. Hasselblatt} 
    \emph{Zygmund
 strong foliations.} Isr. J. Math. 138, 157–169 (2003). \href{https://doi.org/10.1007/BF02783424}{doi:10.1007/BF02783424}




\bibitem[FH23]{FH23}
    {R. Fraser, K. Hambrook },
    \emph{Explicit Salem sets in $\mathbb{R}^d$},
    Advances in Mathematics, Volume 416, (2023)
    \href{https://doi.org/10.1016/j.aim.2023.108901}{doi:10.1016/j.aim.2023.108901}, \href{https://arxiv.org/abs/1909.04581}{arXiv:1909.04581}.

\bibitem[Fr22]{Fr22}
    {J. M. Fraser}
    \emph{The Fourier dimension spectrum and sumset type problems},
    preprint, \href{https://arxiv.org/abs/2210.07019}{arXiv:2210.07019 }


\bibitem[Fro35]{Fro35}
    {O. Frostman}
    \emph{Potentiel d'équilibre et capacité des ensembles avec quelques applications à la théorie des functions},
    PhD thesis.


\bibitem[FS58]{FS58}
    {R. Finn, J. Serrin},
    \emph{On the Hölder Continuity of Quasi Conformal and Elliptic Mappings} \\
    Transactions of the American Mathematical Society
Vol. 89, No. 1 (Sep., 1958), pp. 1-15 
\href{https://doi.org/10.2307/1993129}{doi:10.2307/1993129}


\bibitem[FS18]{FS18}
    {J.M. Fraser, T. Sahlsten},
    \emph{ On the Fourier analytic structure of the Brownian graph} \\
    Analysis and Partial Differential Equations 11(1):115-132  (2018)
    \href{https://msp.org/apde/2018/11-1/p03.xhtml}{doi:10.2140/apde.2018.11.115}



\bibitem[Gr09]{Gr09}
    {B. Green}, 
    \emph{Sum-product phenomena in $\mathbb{F}_p$: a brief introduction} \\
    Notes written from a Cambridge course on Additive Combinatorics, 2009.  \href{https://arxiv.org/pdf/0904.2075.pdf}{arXiv:0904.2075}


\bibitem[GRH21]{GRH21}
    { A. Gogolev,  F. Rodriguez Hertz}
    \emph{Smooth rigidity for very non-algebraic Anosov diffeomorphisms of codimension one}
    preprint, \href{https://arxiv.org/abs/2105.10539}{arXiv:2105.10539}


\bibitem[GS14]{GS14}
    {K. Gelfert, B. Schapira}
    \emph{Pressures for geodesic flows of rank one manifolds}
    2014 Nonlinearity 27 1575 \href{https://iopscience.iop.org/article/10.1088/0951-7715/27/7/1575}{doi:10.1088/0951-7715/27/7/1575}, \ \href{https://arxiv.org/abs/1310.4088}{arXiv:1310.4088}


\bibitem[Ha89]{Ha89}
    {B. Hasselblatt},
    \emph{Regularity of the Anosov splitting and a new description of the Margulis measure},
    PhD thesis.

\bibitem[Ha97]{Ha97}
    {B. Hasselblatt},
    \emph{Regularity of the Anosov splitting II}
    Ergodic Theory and Dynamical Systems, 17 (1997), no. 1, 169–172
    \href{https://doi.org/10.1017/S0143385797069757}{doi:10.1017/S0143385797069757}

\bibitem[Ha17]{Ha17}
    {K. Hambrook},
    \emph{Explicit Salem sets in $\mathbb{R}^2$}
    (2017) Advances in Mathematics 311(3):634-648
    \href{https://doi.org/10.1016/j.aim.2017.03.009}{doi:10.1016/j.aim.2017.03.009}


\bibitem[HdS22]{HdS22}
    {W. He, N. de Saxcé},
    \emph{Linear random walks on the torus}
    Duke Math. J. 171 (5) 1061 - 1133, 1 April 2022. \href{https://doi.org/10.1215/00127094-2021-0045}{doi:10.1215/00127094-2021-0045}

\bibitem[HK90]{HK90}
    {S. Hurder, A. Katok},
    \emph{Differentiability, rigidity and Godbillon-Vey classes for Anosov flows}
    Publications Mathématiques de l’Institut des Hautes Scientifiques 72, 5–61 (1990). \href{https://doi.org/10.1007/BF02699130}{doi: 10.1007/BF02699130}


\bibitem[HP69]{HP69}
    {M. W. Hirsch, C. C. Pugh},
    \emph{Stable manifolds for hyperbolic sets}
    Bulletin of the American Mathematical Society, 75(1), 149-152, (January 1969)

\bibitem[IPR10]{IPR10}
    {J. Iglesias, A. Portela, A. Rovella},
    \emph{C1 stable maps: Examples without saddles.}
    Fundamenta Mathematicae, 2010 \href{https://www.impan.pl/pl/wydawnictwa/czasopisma-i-serie-wydawnicze/fundamenta-mathematicae/all/208/1/89280/c-1-stable-maps-examples-without-saddles}{doi:208.10.4064/fm208-1-2}

\bibitem[JS16]{JS16}
    {T. Jordan, T. Sahlsten}
    \emph{ Fourier transforms of Gibbs measures for the Gauss map } \\
    Math. Ann., 364(3-4), 983-1023, 2016. 
    \href{ https://arxiv.org/abs/1312.3619 }{arXiv:1312.3619}


\bibitem[Ka66]{Ka66}
    {J. P. Kahane},
    \emph{Images d’ensembles parfaits par des séries de Fourier Gaussiennes}
    Acad. Sci. Paris 263 (1966), 678-681.

\bibitem[Ka80]{Ka80}
    {R. Kaufman},
    \emph{ Continued fractions and Fourier transforms.}
    (1980). Mathematika, 27(2), 262-267. \href{https://www.cambridge.org/core/journals/mathematika/article/abs/continued-fractions-and-fourier-transforms/2BF3B8A3312D7A5BD4CF7B3CC12C7C37}{doi:10.1112/S0025579300010147}


\bibitem[Ka81]{Ka81}
    {R. Kaufman},
    \emph{On the theorem of Jarník and Besicovitch.}
    Acta Arithmetica 39.3 (1981): 265-267. \href{http://eudml.org/doc/205767}{eudml.org/doc/205767}.



\bibitem[Ka66a]{Ka66a}
    {J.-P. Kahane},
    \emph{Images browniennes des ensembles parfaits}
    C. R. Acad. Sci. Paris Sér. A-B, 263:A613–A615, 1966.


\bibitem[Ka66b]{Ka66b}
    {J.-P. Kahane},
    \emph{Images d’ensembles parfaits par des séries de Fourier gaussiennes.}
    C. R. Acad. Sci. Paris Sér. A-B, 263:A678–A681, 1966.

\bibitem[Ka85]{Ka85}
    {J.-P. Kahane},
    \emph{Some Random series of functions}, 2nd edition (Cambridge University Press,1985).

\bibitem[KH95]{KH95}
    {A. Katok, B. Hasselblatt},
    \emph{ Introduction to the Modern Theory of Dynamical Systems.}
    Encyclopedia of Mathematics and its Applications. 1995. Cambridge: Cambridge University Press. \href{https://doi.org/10.1017/CBO9780511809187}{doi:10.1017/CBO9780511809187}

\bibitem[Kh23]{Kh23}
    {O. Khalil},
    \emph{Exponential Mixing Via Additive Combinatorics}
    Preprint, \href{https://arxiv.org/abs/2305.00527}{arXiv:2305.00527}


\bibitem[Ki90]{Ki90}
    {Y. Kiffer},
    \emph{Large deviations in dynamical systems and stochastic processes},
    Transactions of the American Mathematical Society, 321(1990),505-524.


\bibitem[KK07]{KK07}
    {B. Kalinin, A. Katok,}
    \emph{Measure rigidity beyond uniform hyperbolicity: invariant measures for cartan actions on tori}, Journal of Modern Dynamics, 2007, 1(1): 123-146.   \href{https://www.aimsciences.org/article/doi/10.3934/jmd.2007.1.123}{doi:10.3934/jmd.2007.1.123},  \href{https://arxiv.org/abs/math/0602176v1}{arXiv:math/0602176}


\bibitem[Kö92]{Kö92}
    {T.W. Körner},
    \emph{Sets of uniqueness} Cahiers du séminaire d'histoire des mathématiques 2 (1992): 51-63 \href{https://eudml.org/doc/91031}{eudml.org/doc/91031}


\bibitem[KS64]{KS64}
    {J.P. Kahane, R. Salem},
    \emph{ Ensembles Parfaits et Séries Trigonométriques}
    Canadian Mathematical Bulletin, Volume 7, Issue 3, September 1964, pp. 492 - 493 \href{https://www.cambridge.org/core/journals/canadian-mathematical-bulletin/article/ensembles-parfaits-et-series-trigonometriques-by-j-p-kahane-and-r-salem-hermann-paris-1963-paperback-188-pages-27-f/CD1EA7BBCCBDC467866D431347877BEA}{doi:10.1017/S0008439500032021}


\bibitem[Le00]{Le00}    
    {R. Leplaideur},
    \emph{Local product structure for equilibrium states},
    Trans. Amer. Math. Soc. 352 (2000), no. 4, 1889–1912. MR 1661262 
    \href{https://doi.org/10.1090/S0002-9947-99-02479-4}{doi:10.1090/S0002-9947-99-02479-4}
    
\bibitem[Le21]{Le21}
    {G. Leclerc},
    \emph{Julia sets of hyperbolic rational maps have positive Fourier dimension}, \\ Comm. Math. Phys. (2022) \href{https://arxiv.org/abs/2112.00701?context=math}{	arXiv:2112.00701},  \href{https://doi.org/10.1007/s00220-022-04496-6}{doi:10.1007/s00220-022-04496-6}


\bibitem[Le22]{Le22}
    {G. Leclerc},
    \emph{On oscillatory integrals with Hölder phases}, \\ 
    (To appear in) Discrete and Continuous Dynamical Systems, \href{https://arxiv.org/abs/2211.08088}{arXiv:2211.08088}


\bibitem[Le23]{Le23}
    {G. Leclerc},
    \emph{Fourier decay of equilibrium states for bunched attractors}, \\ (To appear in) Nonlinearity, \href{https://arxiv.org/abs/2301.10623}{arXiv:2301.10623} 



\bibitem[Le23b]{Le23b}
    {G. Leclerc},
    \emph{Fourier decay of equilibrium states on hyperbolic surfaces}, \\
    Preprint (2023), \href{https://arxiv.org/abs/2307.10755}{arXiv:2307.10755}


\bibitem[Li17]{Li17}
      {J. Li}, 
     \emph{Decrease of Fourier coefficients of stationary measures},
     Math. Ann. 372, 1189–1238 (2018).  \href{https://doi.org/10.1007/s00208-018-1743-3}{doi:10.1007/s00208-018-1743-3}, \href{https://arxiv.org/abs/1706.07184v3}{	arXiv:1706.07184 }.



\bibitem[Li18]{Li18}
     {J. Li}, 
     \emph{Discretized Sum-product and Fourier decay in }$\mathbb{R}^n$,
     JAMA 143, 763–800 (2021). \newline \href{https://doi.org/10.1007/s11854-021-0169-0}{doi:10.1007/s11854-021-0169-0},  \href{https://arxiv.org/abs/1811.06852v2}{arXiv:1811.06852} 
     

\bibitem[Li20]{Li20}
    {J. Li},
    \emph{Fourier decay, Renewal theorem and Spectral gaps for random walks on split semisimple Lie groups}
    Annales Scientifiques de l'ÉNS, Tome 55, Fasc.6, pp 1613-1686, 2022. \href{https://arxiv.org/abs/1811.06484v2}{	arXiv:1811.06484}.


\bibitem[Liv71]{Liv71}
    {A.N. Livshits},
    \emph{Homology properties of Y-systems.}
    Mathematical Notes of the Academy of Sciences of the USSR 10, 758–763 (1971). \href{https://doi.org/10.1007/BF01109040}{doi:10.1007/BF01109040}
    
\bibitem[LNP19]{LNP19}
    {J. Li, F. Naud, W. Pan},
    \emph{Kleinian Schottky groups, Patterson-Sullivan measures and Fourier decay} 
    Duke Math. J. 170 (4) 775 - 825, 15 March 2021. \\ \href{https://doi.org/10.1215/00127094-2020-0058}{doi:10.1215/00127094-2020-0058}, \href{https://arxiv.org/abs/1902.01103}{	arXiv:1902.01103}.  


\bibitem[ŁP09]{ŁP09}
    {I. Łaba, M. Pramanik},
    \emph{Arithmetic progressions in sets of fractional dimension.}
    Geom. Funct. Anal. 19, 429–456 (2009). \href{https://doi.org/10.1007/s00039-009-0003-9}{doi:10.1007/s00039-009-0003-9}

\bibitem[LQZ03]{LQZ03}
    {P. Liu, M. Qian, Y. Zhao},
    \emph{Large deviations in Axiom A endomorphisms}
    Proceedings of the Royal Society of Edinburgh Section A: Mathematics , Volume 133 , Issue 6 , December 2003 , pp. 1379 - 1388
    \href{https://doi.org/10.1017/S0308210500002997}{doi:10.1017/S0308210500002997}

\bibitem[LS19a]{LS19a}
    {J. Li, T. Sahlsten}
    \emph{Fourier transform of self-affine measures}
    Advances in Mathematics, Volume 374, 18 November 2020, 107349 \href{https://doi.org/10.1016/j.aim.2020.107349}{doi:10.1016/j.aim.2020.107349}, \href{https://arxiv.org/abs/1903.09601v2}{arXiv:1903.09601}

\bibitem[LS19b]{LS19b}
    {J. Li, T. Sahlsten}, \emph{Trigonometric series and self-similar sets.} J. Eur. Math. Soc. 24 (2022), no. 1, pp. 341–368. \href{https://ems.press/journals/jems/articles/2315404}{doi:10.4171/JEMS/1102}, \href{https://arxiv.org/abs/1902.00426}{arXiv:1902.00426}


\bibitem[Ly86]{Ly86}    
    {M. Lyubich},
    \emph{The dynamics of rational transforms: the topological picture} \\
    Russian Mathematical Surveys, volume 41, p 43-117 (1986) \\  \href{https://doi.org/10.1070/RM1986v041n04ABEH003376}{doi:10.1070/RM1986V041N04ABEH003376}

    
\bibitem[Ly20]{Ly20}
    {R. Lyons},
    \emph{Seventy Years of Rajchman Measures}, 
    The Journal of Fourier Analysis and Applications (pp.363-377) (2020). \href{https://www.taylorfrancis.com/chapters/edit/10.1201/9780429332838-22/seventy-years-rajchman-measures-russell-lyons}{doi:10.1201/9780429332838-22.} 



\bibitem[ŁW16]{ŁW16}
    {I. Łaba, H. Wang},
    \emph{Decoupling and near-optimal restriction estimates for Cantor sets}
    International Mathematics Research Notices, Volume 2018, Issue 9, May 2018, Pages 2944–2966, \href{https://doi.org/10.1093/imrn/rnw327}{doi:10.1093/imrn/rnw327},  \href{https://arxiv.org/abs/1607.08302}{arXiv:1607.08302}
  

\bibitem[Ma15]{Ma15}
     {P. Mattila},
     \emph{Fourier analysis and Hausdorff dimension},
     Cambridge University Press, 2015. \href{https://www.cambridge.org/core/books/fourier-analysis-and-hausdorff-dimension/78A25BED2C6E4F9B911971305DC928B0}{ISBN:9781316227619}

\bibitem[Mc34]{Mc34}
    {J. McShane},
    \emph{Extension of range of functions},
    Bull. Amer. Math. Soc. 40 (1934), 837–842.


\bibitem[Mi90]{Mi90}
    {J. Milnor},
    \emph{Dynamics in one complex variable}, 
    Stony Brook IMS Preprint 1990/5, \href{https://arxiv.org/abs/math/9201272}{	arXiv:math/9201272}
    
    
\bibitem[MN20]{MN20}
    {M. Magee, F. Naud}
    \emph{Explicit spectral gaps for random covers of Riemann surfaces. }
    Publ.math.IHES 132, 137–179 (2020). \href{https://doi.org/10.1007/s10240-020-00118-w}{doi:10.1007/s10240-020-00118-w}


\bibitem[MR92]{MR92}
    {R. Mane, L.F. Da Rocha},
    \emph{Julia sets are uniformly perfect} \\
    Proc. Amer. Math. Soc. 116 (1992), no. 1, 251–257.


\bibitem[Na05]{Na05}
    {F. Naud},
    \emph{Expanding maps on Cantor sets and analytic continuation of zeta functions}
    Volume 38, Issue 1, 2005, Pages 116-153, ISSN 0012-9593, \href{https://doi.org/10.1016/j.ansens.2004.11.002}{doi:10.1016/j.ansens.2004.11.002}.    

\bibitem[OW17]{OW17}
    {H. Oh, D. Winter},
    \emph{Prime number theorem and holonomies for hyperbolic rational maps} \\
    Inventiones mathematicae, 2017.  208. \href{https://doi.org/10.1007/s00222-016-0693-1}{doi:10.1007/s00222-016-0693-1}, 
    \href{ https://arxiv.org/abs/1603.00107}{arXiv:1603.00107}


\bibitem[Pe98]{Pe98}
    {Y. Pesin}
    \emph{Dimension theory in dynamical systems: contemporary views and applications}
    Chicago Lectures in Mathematics, University of Chicago Press, \href{https://www.cambridge.org/core/journals/ergodic-theory-and-dynamical-systems/article/abs/dimension-theory-in-dynamical-systems-contemporary-views-and-applications-by-yakov-b-pesin-chicago-lectures-in-mathematics-university-of-chicago-press-312-pp-price-hardback-56-paperback-1995-isbn-0-226-66222-5/4A07F771A704DB6E5C1E3D26AB819D99}{ISBN:0226662225}.     


\bibitem[Pe19]{Pe19}
    {Y. Pesin},
    \emph{Sinai’s Work on Markov Partitions and SRB Measures}
    The Abel Prize 2013-2017 (pp.257-285) \href{https://link.springer.com/chapter/10.1007\%2F978-3-319-99028-6\_11}{doi:10.1007/978-3-319-99028-6\_11}


\bibitem[Po85]{Po85}
    \emph{M. Pollicott},
    \emph{On the rate of mixing of Axiom A flows.} Invent Math 81, 413–426 (1985). \href{https://doi.org/10.1007/BF01388579}{doi:10.1007/BF01388579}

\bibitem[PP90]{PP90}
    {W. Parry, M. Pollicott},
    \emph{Zeta Functions and the Periodic Orbit Structure of Hyperbolic Dynamics},
    Asterisque; 187-188, société mathématique de France, 1990.


\bibitem[PPS15]{PPS15}
    {F. Paulin, M. Pollicott, B. Schapira},
    \emph{Equilibrium states in negative curvature}. 
    Asterisque series, volume 373.
    \href{https://smf.emath.fr/publications/etats-dequilibre-en-courbure-negative}{doi:10.24033/ast.975 }


\bibitem[PR02]{PR02}
    {A. Pinto, D. Rand},
    \emph{Smoothness of Holonomies for codimension 1 hyperbolic dynamics},
      Bulletin of the London Mathematical Society , Volume 34 , Issue 3 , May 2002 , pp. 341 - 352 \\ \href{https://doi.org/10.1112/S0024609301008670}{doi:10.1112/S0024609301008670}


\bibitem[PS16]{PS16}
    {V. Petkov, L. Stoyanov},
    \emph{Ruelle transfer operators with two complex parameters and applications},
    Discrete and Continuous Dynamical Systems, 2016, 36 (11) : 6413-6451. \\ \href{https://www.aimsciences.org/article/doi/10.3934/dcds.2016077}{doi:10.3934/dcds.2016077}

\bibitem[PU09]{PU09}
    {F. Przytycki, M. Urbański},
    \emph{Conformal fractals: ergodic theory methods}, 
    Cambridge university, 2009.

\bibitem[PU17]{PU17}
    {M. Pollicott and M. Urbański},
    \emph{Asymptotic Counting in Conformal Dynamical Systems.}
    Memoirs of the American Mathematical Society. 271. (2017). \href{https://www.ams.org/books/memo/1327/}{doi:10.1090/memo/1327}. 

\bibitem[PW97]{PW97}
    {Y. Pesin and H. Weiss},
    \emph{A multifractal analysis of equilibrium measures for conformal expanding maps and Moranlike geometric constructions}, 
    Journal of Statistical Physics 86, 233—275, (1997)

\bibitem[QR03]{QR03}
    {M. Queffélec, O. Ramaré},
    \emph{Analyse de Fourier des fractions continues à quotients restreints}
    (2003) Enseign. Math. (2). 49. \href{https://www.e-periodica.ch/digbib/view?pid=ens-001:2003:49::251#684}{doi:10.5169/seals-66692}

\bibitem[Ra06]{Ra06}
    {J. G. Ratcliffe},
    \emph{Foundations of hyperbolic manifolds}, volume 149 of Graduate Texts in Mathematics. Springer, New York, second edition, 2006. \href{https://link.springer.com/book/10.1007/978-0-387-47322-2}{ISBN:978-0-387-47322-2}


\bibitem[Ra21]{Ra21}
    {A. Rapaport}
    \emph{On the Rajchman property for self-similar measures on $\mathbb{R}^d$}
    Advances in Mathematics Volume 403, 2022, 108375 
    \href{https://doi.org/10.1016/j.aim.2022.108375}{doi:10.1016/j.aim.2022.108375}, \href{https://arxiv.org/abs/2104.03955}{	arXiv:2104.03955}


\bibitem[Ro99]{Ro99}
    {C. Robinson},
    \emph{Dynamical Systems: Stability, Symbolic Dynamics, and Chaos.}
    (1999)  2nd Edition, CRC Press LLC, Boca Raton, Florida. ISBN: 9780849384950


\bibitem[Ru78]{Ru78}
    {D. Ruelle}, 
    \emph{Thermodynamic Formalism}, 
    Second edition, Cambridge University Press 2004 

\bibitem[Ru89]{Ru89}
    {D. Ruelle},
    \emph{The thermodynamic formalism for expanding maps}, \\
    Commun. Math. Phys. 125, 239-262 (1989) 



\bibitem[Sa43]{Sa43}
    {R. Salem}, \emph{Sets of uniqueness and sets of multiplicity}, 
   Transactions of the American Mathematical Society 54 (1943): 218-228. \href{https://www.semanticscholar.org/paper/Sets-of-uniqueness-and-sets-of-multiplicity-Salem/0737c31586a87737659eb2a547427a118122eaba}{doi:10.1090/S0002-9947-1943-0008428-8}


\bibitem[Sa51]{Sa51}
    {R. Salem}, \emph{On singular monotonic functions whose spectrum has a given Hausdorff dimension.} 
Ark. Mat., 1:353–365, 1951.


\bibitem[ShSt20]{ShSt20}
    {R. Sharp, A. Stylianou},
    \emph{Statistics of multipliers for hyperbolic rational maps}, \\
    Preprint, 2020. \href{https://arxiv.org/abs/2010.15646}{arXiv:2010.15646}

\bibitem[So19]{So19}
    {B. Solomyak},
    \emph{Fourier decay for self-similar measures},
    Proc. Amer. Math. Soc. 149 (2021), 3277-3291
    \href{https://doi.org/10.1090/proc/15515}{doi:10.1090/proc/15515}, \href{https://arxiv.org/abs/1906.12164v2}{	arXiv:1906.12164}


\bibitem[SS85]{SS85}
    {M. Shub, D. Sullivan},
    \emph{Expanding endomorphisms of the circle revisited} \\
    Ergod. Th. \& Dynam. Sys. (1985), 5, 285-289.


\bibitem[SS18]{SS18}
    {P. Shmerkin and V. Suomala.} 
    \emph{Spatially independent martingales, intersections, and applications.}
    Mem. Amer. Math. Soc., 251(1195):v+102, 2018.


\bibitem[SS20]{SS20}
    {T. Sahlsten, C. Stevens},
    \emph{Fourier transform and expanding maps on Cantor sets} \\
    To be published in Amer. J. Math. (2022)
    \href{https://arxiv.org/abs/2009.01703v4}{arXiv:2009.01703 }


\bibitem[St01]{St01}
    {L. Stoyanov},
    \emph{Spectrum of the Ruelle Operator and Exponential Decay of Correlations for Open Billiard Flows}
     American Journal of Mathematics 123, no. 4 (2001): 715–59. \\ \href{http://www.jstor.org/stable/25099080}{jstor.org/stable/25099080}.

\bibitem[St11]{St11}
    {L. Stoyanov},
    \emph{Spectra of Ruelle transfer operators for Axiom A flows.}
    Nonlinearity, 24(4):1089, 2011. \href{https://iopscience.iop.org/article/10.1088/0951-7715/24/4/005}{doi:10.1088/0951-7715/24/4/005}



\bibitem[Ts01]{Ts01}
    {M. Tsujii},
    \emph{Fat solenoidal attractor},
    Nonlinearity 14 1011 (2001) \href{https://iopscience.iop.org/article/10.1088/0951-7715/14/5/306/pdf}{doi:10.1088/0951-7715/14/5/306}, \ \href{https://arxiv.org/abs/2006.04293}{arXiv:2006.04293}

\bibitem[TZ20]{TZ20}
    {M. Tsujii, Z. Zhang},
    \emph{Smooth mixing Anosov flows in dimension three are exponentially mixing}
    Ann. of Math. (2) 197 (1) 65 - 158, January 2023. \href{https://doi.org/10.4007/annals.2023.197.1.2}{doi.org/10.4007/annals.2023.197.1.2} \ \href{https://arxiv.org/abs/2006.04293}{arXiv:2006.04293}


 \bibitem[VY20]{VY20}
    {P. P. Varjú, H. Yu},
    \emph{Fourier decay of self-similar measures and self-similar sets of uniqueness} Anal. PDE 15 (3) 843 - 858, 2022. \href{ https://doi.org/10.2140/apde.2022.15.843}{doi:10.2140/apde.2022.15.843}, \href{https://arxiv.org/abs/2004.09358}{arXiv:2004.09358}


\bibitem[WW17]{WW17}  
    {C.P. Walkden, T. Withers},
    \emph{The stability index for dynamically defined Weierstrass functions}
    preprint, \href{https://arxiv.org/abs/1709.02451}{arXiv:1709.02451}
    

\bibitem[Yo90]{Yo90}
    {L.-S. Young.}
    \emph{Large deviations in dynamical systems.}
    Trans. Amer. Math. Soc., 318(2):525–543, 1990

\bibitem[Yo02]{Yo02}
    {L-S. Young},
    \emph{What Are SRB Measures, and Which Dynamical Systems Have Them ?}
     Journal of Statistical Physics. 108. 733-754. 
     \href{http://dx.doi.org/10.1023/A:1019762724717}{doi:10.1023/A:1019762724717}.

 
\bibitem[Zi96]{Zi96}
    {M. Zinsmeister},
    \emph{Thermodynamic formalism and holomorphic dynamical systems} \\
     Translated from the 1996 French original by C. Greg Anderson. American Mathematical Society, Providence, RI; Société Mathématique de France, Paris, 2000.





\end{thebibliography}
\end{document}